\documentclass[english]{smfbookauthors}

\RequirePackage[l2tabu,orthodox]{nag}
\usepackage[T1]{fontenc}
\usepackage[utf8]{inputenc}
\usepackage{lmodern}
\usepackage[french,english]{babel}
\usepackage{color} 
\usepackage[usenames,dvipsnames]{xcolor}
\usepackage[alphabetic]{amsrefs}
\usepackage{amsthm,amssymb,amsmath}
\usepackage[pdftex]{graphicx}
\usepackage{smfenum}
\usepackage[inline]{enumitem}
\usepackage{url}
\usepackage{faktor}
\usepackage{xfrac}
\makeindex
\usepackage{multicol}
\usepackage{mathrsfs}  
\usepackage{stmaryrd}

\DeclareFontFamily{U}{mathb}{\hyphenchar\font45}
\DeclareFontShape{U}{mathb}{m}{n}{
	<5> <6> <7> <8> <9> <10> gen * mathb
	<10.95> mathb10 <12> <14.4> <17.28> <20.74> <24.88> mathb12
}{}
\DeclareSymbolFont{mathb}{U}{mathb}{m}{n}

\DeclareMathSymbol{\precneq}{3}{mathb}{"AC}
\DeclareMathSymbol{\succneq}{3}{mathb}{"AD}

\numberwithin{equation}{section}
\numberwithin{figure}{section}

\newtheorem{thm}{Theorem}[section]
\newtheorem{introthm}{Theorem}

\newtheorem{introthmprime}{Theorem}

\newtheorem{introthmderoin}{Theorem}

\newtheorem{ques}[thm]{Question}

\newtheorem*{claim}{Claim}

\newtheorem{claimnum}{Claim}

\newtheorem{case}{Case}

\newtheorem{lem}[thm]{Lemma}
\newtheorem{prop}[thm]{Proposition}
\newtheorem{cor}[thm]{Corollary}
\newtheorem{intro-cor}[introthm]{Corollary}

\theoremstyle{definition}
\newtheorem{dfn}[thm]{Definition}
\newtheorem{assumption}[thm]{Assumption}
\newtheorem{notation}[thm]{Notation}

\newtheorem*{main question}{Main Question}

\theoremstyle{remark}
\newtheorem{rem}[thm]{Remark}
\newtheorem{ex}[thm]{Example}

\usepackage{geometry}

\usepackage{indentfirst}
\usepackage{microtype}

\usepackage[colorlinks=true,linkcolor=blue,citecolor=magenta]{hyperref}

\usepackage{blindtext}

\usepackage{nameref}
\DeclareMathOperator{\Aut}{\mathsf{Aut}}

\DeclareMathOperator{\homeo}{\mathsf{Homeo}}
\DeclareMathOperator{\Diff}{\mathsf{Diff}}
\DeclareMathOperator{\Bij}{\mathsf{Sym}}

\DeclareMathOperator{\s}{\mathsf{s}}

\DeclareMathOperator{\gap}{\mathsf{Gap}}
\DeclareMathOperator{\core}{\mathsf{Core}}

\DeclareMathOperator{\stab}{\mathsf{Stab}}
\DeclareMathOperator{\Der}{\mathsf{Harm}}
\DeclareMathOperator{\Hom}{\mathsf{Hom}}
\DeclareMathOperator{\PL}{\mathsf{PL}}
\DeclareMathOperator{\PP}{\mathsf{PP}}

\DeclareMathOperator{\PSL}{\mathsf{PSL}}

\newcommand{\PA}{\operatorname{\mathsf{PDiff}}}

\newcommand{\fix}{\operatorname{\mathsf{Fix}}}
\newcommand{\fixphi}{\operatorname{\mathsf{Fix}}^\varphi}
\DeclareMathOperator{\BP}{\mathsf{BP}}
\DeclareMathOperator{\Aff}{\mathsf{Aff}}
\DeclareMathOperator{\rk}{\mathsf{rk}}
\DeclareMathOperator{\ev}{\mathsf{ev}}

\newcommand{\0}{\mathtt{0}}
\newcommand{\1}{\mathtt{1}}
\DeclareMathOperator{\LPO}{\mathsf{LPO}}

\DeclareMathOperator{\LO}{\mathsf{LO}}
\DeclareMathOperator{\Germ}{\mathsf{Germ}}
\newcommand{\Gfrag}{G_\mathsf{frag}}

\newcommand{\id}{\mathsf{id}}
\newcommand{\R}{\mathbb{R}}
\newcommand{\Q}{\mathbb Q}
\newcommand{\N}{\mathbb N}
\newcommand{\Z}{\mathbb Z}
\newcommand{\T}{\mathbb S^1}
\newcommand{\Tbb}{\mathbb T}

\newcommand{\Dcal}{\mathcal{D}}
\newcommand{\Gcal}{\mathcal{G}}

\newcommand{\Ical}{\mathcal{I}}

\newcommand{\Ocal}{\mathcal{O}}

\renewcommand{\setminus}{\smallsetminus}
\newcommand{\BN}{\mathsf{BN}}
\newcommand{\hor}{\mathfrak{h}}
\newcommand{\dist}{\mathsf{dist}}

\newcommand{\Dsf}{G}

\DeclareMathOperator{\BS}{\mathsf{BS}}
\DeclareMathOperator{\Br}{\mathsf{Br}}

\DeclareMathOperator{\Homirr}{\mathsf{Hom}_{\mathrm{irr}}}

\DeclareMathOperator{\Int}{\mathsf{Int}}
\renewcommand{\ker}{\operatorname{\mathsf{ker}}}
\newcommand{\Id}{\operatorname{\mathsf{id}}}
\renewcommand{\emptyset}{\varnothing}
\newcommand{\treeorder}{\triangleleft}

\newcommand{\treeordereq}{\trianglelefteq}
\newcommand{\treeup}{\curlywedgeuparrow}

\DeclareMathOperator{\supp}{\mathsf{Supp}}

\newcommand{\Stphi}{\operatorname{\mathsf{Stab}}^\varphi}

\DeclareMathOperator{\suppphi}{\mathsf{Supp}^\varphi}
\DeclareMathOperator{\Iphi}{I^\varphi}
\DeclareMathOperator{\Bphi}{B^\varphi}

\DeclareMathOperator{\Ipsi}{I^\psi}

\DeclareMathOperator{\Iphiout}{{I}^\varphi_{\mathsf{out}}}
\DeclareMathOperator{\Iphiinn}{{I}^\varphi_{\mathsf{inn}}}

\newcommand{\dfcn}[5]{\setlength{\arraycolsep}{1.5pt}\begin{array}{cccc}#1\colon&#2&\to&#3\\&#4&\mapsto&#5\end{array}}

\title{Locally moving groups and laminar actions on the line}
\alttitle{Groupes localement mobiles et actions laminaires sur la droite réelle}

\author{Joaqu\'in Brum}
\address{IMERL, Facultad de Ingenier\'ia, Universidad de la Rep\'ublica, Uruguay\\
Julio Herrera y Reissig 565, Montevideo, Uruguay}
\email{joaquinbrum@fing.edu.uy}

\author{Nicol\'as Matte Bon}
\address{CNRS \&
	Institut Camille Jordan (UMR CNRS 5208)\\
	Universit\'e de Lyon\\
	43 blvd.\ du 11 novembre 1918,	69622 Villeurbanne,	France}
\email{mattebon@math.univ-lyon1.fr}

\author{Crist\'obal Rivas}
\address{Dpto.\ de Matem\'aticas, Universidad  de Chile, \newline
	Las Palmeras 3425, \~Nu\~noa, Santiago, Chile}
\email{cristobalrivas@u.uchile.cl}

\author{Michele Triestino}
\address{Institut de Math\'ematiques de Bourgogne (UMR CNRS 5584) \& Institut Universitaire de France\\
	Universit\'e de Bourgogne\\
	9 av.~Alain Savary, 21000 Dijon, France}
\email{michele.triestino@u-bourgogne.fr}

\begin{document}
\frontmatter
\begin{abstract}
We prove various results that, given a sufficiently rich subgroup $G$ of the group of homeomorphisms on the real line, describe the structure of the other possible actions of $G$ on the line, and  address under which conditions such actions must be semi-conjugate to the natural defining action of $G$. The main assumption is that  $G$ should be locally moving, meaning that for every open interval the subgroup of elements fixing pointwise its complement, acts  on it without fixed points. One example (among many others) is given by Thompson's group $F$.

In Part I, we show that when $G$ is a locally moving group, every $C^1$ action of $G$ on the real line without fixed points is semi-conjugate to its standard action or to a non-faithful action. It turns out that the situation is much wilder when considering actions by homeomorphisms: for a large class of groups, including Thompson's group $F$, we describe uncountably many conjugacy classes of faithful minimal  actions by homeomorphisms on the real line.

In Part II, we prove structure theorems describing the dynamics of exotic $C^0$ actions, based on the study of laminar actions, which are actions on the line preserving a lamination. When $G$ is a group of homeomorphisms of the line acting minimally, and with a non-trivial compactly supported element, then any faithful minimal  action of $G$ on the line is either laminar or conjugate to its standard action.
Moreover, when $G$ is a locally moving group satisfying a suitable finite generation condition, we prove that for any faithful minimal laminar action on the line, there is a map from the lamination to the line, called a horograding, which is equivariant with respect to the action on the lamination and the standard action, and satisfies some extra suitable conditions.
%
%Under a suitable finite generation condition on a locally moving group $G$, we prove that every faithful minimal  action of $G$ on the line is either laminar or conjugate to the standard action; moreover, when it is laminar, the action on the associated real tree factors via a horofunction onto the standard action of $G$ on the line.
This establishes a tight relation between all minimal actions on the line of such  groups,  and their standard actions.

Among the various applications of this result,  we show in Part III that for a large class of locally moving groups, the standard action is locally rigid, in the sense that every sufficiently small perturbation in the compact-open topology gives a semi-conjugate action. This is based on an analysis of the space of harmonic actions on the line for such groups. 

Along the way we introduce and study several concrete examples.
\end{abstract}

\begin{altabstract}
	Nous prouvons plusieurs résultats qui, étant donné un sous-groupe $G$ suffisamment riche du groupe des homéomorphismes de la droite réelle, décrivent la structure des autres actions possibles de $G$ sur la droite, et indiquent sous quelles conditions ces actions doivent être semi-conjuguées à l'action naturelle de $G$. L'hypothèse principale est que $G$ doit être localement mobile, ce qui signifie que pour chaque intervalle ouvert, le sous-groupe des éléments qui fixent le complémentaire point par point, y agit sans points fixes. Un exemple (parmi beaucoup d'autres) est donné par le groupe $F$ de Thompson.
	
	Dans la partie I, nous montrons que si $G$ est un groupe localement mobile, toute action $C^1$ de $G$ sur la droite sans points fixes est soit semi-conjuguée à son action standard soit à une action non-fidèle. Il s'avère que la situation est beaucoup plus sauvage lorsqu'on considère les actions par homéomorphismes : pour une grande classe de groupes, y compris le groupe de Thompson $F$, nous décrivons une quantité non-dénombrable de classes de conjugaison d'actions minimales fidèles par homéomorphismes sur la droite réelle. 
	
	Dans la partie II, nous obtenons des théorèmes de structure qui décrivent la dynamique des actions $C^0$ exotiques, par l'étude des actions laminaires, qui sont des actions sur la droite qui préservent une lamination. Lorsque $G$ est un groupe d'homéomorphismes de la droite dont l'action est minimale et possédant un élément non trivial à support compact, alors toute action minimale fidèle de $G$ sur la droite est soit laminaire, soit conjuguée à son action standard. De plus, lorsque $G$ est un groupe localement mobile satisfaisant une condition de génération finie convenable, nous prouvons que pour toute action laminaire minimale fidèle sur la droite, il existe une application de la lamination vers la droite, appelée une horograduation, qui est équivariante par rapport à l'action sur la lamination et l'action standard, et satisfait d'autres conditions convenables.
	Ceci établit une relation étroite entre toutes les actions minimales sur la droite de tels groupes, et leurs actions standard.
%	Sous une condition de génération finie appropriée pour un groupe localement mobile $G$, nous prouvons que toute action  minimale fidèle de $G$ sur la droite est soit laminaire, soit conjuguée à l'action standard ; de plus, lorsqu'elle est laminaire, l'action sur l'arbre réel associé admet comme quotient, par une horofonction, l'action standard de $G$ sur la droite. Ceci établit une relation étroite entre toutes les actions minimales de $G$ sur la droite et son action standard.
	Parmi les diverses applications de ce résultat, nous montrons dans la partie III que pour une grande classe de groupes localement mobiles, l'action standard est localement rigide, dans le sens où toute perturbation suffisamment petite dans la topologie compacte-ouverte donne une action semi-conjuguée. Ceci est basé sur une analyse de l'espace des actions harmoniques sur la ligne pour de tels groupes.  
	
	En cours de route, nous introduisons et étudions plusieurs exemples concrets.
\end{altabstract}



\altkeywords{Actions de groupes sur la droite, groupes localement mobiles, actions sur arbres réels, rigidité locale, groupes ordonnables à gauche, groupes d'homéomorphismes affines par morceaux}

\thanks{The authors have been partially supported by the project ANR Gromeov (ANR-19-CE40-0007). J.B. was partially supported by CONICYT via FONDECYT post-doctorate fellowship 3190719, and with a ``poste rouge'' at ICJ (UMR CNRS 5208) by INSMI for IRL IFUMI members. Part of this work was inspired by discussions during a visit of N.M.B. to the University of Santiago de Chile funded by FONDECYT 1181548. N.M.B was also partially supported by the LABEX MILYON (ANR-10-LABX-0070) of Université de Lyon, within the program ``Investissements d'Avenir'' (ANR11-IDEX-0007) operated by the French National Research Agency.  C.R. was partially supported by FONDECYT 1181548 and FONDECYT 1210155. 
	 M.T. was also partially supported by
	the project ANER Agroupes (AAP 2019 R\'egion Bourgogne--Franche--Comt\'e), by CNRS with a semester of ``délégation'' at IMJ-PRG (UMR CNRS 7586), and by the EIPHI Graduate School (ANR-17-EURE-0002). The authors acknowledge support of the Institut Henri Poincar\'e (UAR 839 CNRS-Sorbonne Universit\'e), and LabEx CARMIN (ANR-10-LABX-59-01).}

\maketitle

\tableofcontents

%----------------------
\mainmatter

\chapter{Introduction}

In this work we prove structure and rigidity results in the setting of groups acting by homeomorphisms, or diffeomorphisms, on the real line. The study of group actions on one-manifolds is a classical topic at the interface of dynamical systems, topology, and group theory, which is still under intense development. An account of the theory can be found in several recent monographs, such as Ghys \cite{Ghys}, Navas \cite{Navas-book}, Clay and Rolfsen \cite{ClayRolfsen}, Deroin, Navas, and the third named author \cite{GOD}, Kim, Koberda, and Mj \cite{KKMj}; see also the surveys by Mann \cite{MannHandbook} and Navas \cite{NavasICM}.
 A central problem in this theory is to describe all possible actions of a given group $G$ on a connected one-manifold $M$ (that is, either the line or the circle), or more precisely its homomorphisms to the group $\homeo_0(M)$ of orientation-preserving homeomorphisms, or to the groups $\Diff_0^r(M)$ of orientation-preserving diffeomorphisms of class $C^r$. There  is essentially no loss of generality in restricting to actions that  have no (global) fixed points in the interior of $M$, that will be called \emph{irreducible}. It is customary to consider actions up to (topological) conjugacy, but  in the one-dimensional setting, it is also convenient to consider actions on the line up to a weaker equivalence relation, called \emph{semi-conjugacy}, which captures their behavior on minimal closed invariant subsets. The definition of semi-conjugacy is recalled in Chapter \ref{s-preliminaries}; here let us simply remind that every action $\varphi \colon G \to \homeo_0(\R)$ of a \emph{finitely generated} group on the real line either is  semi-conjugate  to a  \emph{minimal} action (i.e.\ an action all whose orbits are dense) or $\varphi(G)$ has a discrete orbit, in which case $\varphi$ is semi-conjugate to a \emph{cyclic} action (i.e.\ an action by integer translations); moreover this  minimal or cyclic action is unique up to conjugacy (see e.g.\ Navas \cite{Navas-book}). Hence, studying actions of finitely generated groups up to semi-conjugacy is essentially the same as to study their \emph{minimal} actions up to conjugacy.

 The point of view that we take in this work is the following. Let $G\subseteq \homeo_0(\R)$ be a subgroup of the group of orientation-preserving homeomorphisms of the real line (which will typically be countable or finitely generated, although not always). By definition, such a $G$ has a preferred action on $\R$, that we shall refer to as its \emph{standard} action. We are interested in the following general question: what can be said about other actions $\varphi\colon G\to \homeo_0(\R)$ on the real line? In particular, are there natural conditions on $G$ and $\varphi$, which imply that  $\varphi$ must be reminiscent (for instance, semi-conjugate) to the standard action of $G$? 
 
 To expect some rigidity, it is quite natural to focus on groups $G\subseteq \homeo_0(\R)$ whose defining action is sufficiently rich, so as to be reflected in the intrinsic structure of $G$. A natural condition in this direction is the requirement that $G$ should admit non-trivial elements supported in any non-empty open interval $I\subset \R$; a subgroup of $\homeo_0(\R)$ with this property will be called \emph{micro-supported}.  An important strengthening of this property is the condition that $G$ be \emph{locally moving}. Given an open interval $I\subset \R$, we denote by $G_I\subset G$ the subgroup consisting of elements that fix $\R\setminus I$ pointwise.   
 \begin{dfn}
We say that  a subgroup $G\subseteq \homeo_0(\R)$   is \emph{locally moving} if for every open interval $I\subset \R$, the subgroup $G_I$ acts on $I$ without fixed points. 
\end{dfn}

There are many finitely generated (and even finitely presented) locally moving subgroups of $\homeo_0(\R)$. A well-known example  is \emph{Thompson's group} $F$, customarily defined by an action on the interval $(0,1)$ by piecewise linear (PL) homeomorphisms, which is locally moving and plays the role of its standard action (modulo identifying $(0, 1)$ with $\R$ by a homeomorphism).
The group $F$ is among the most studied examples of subgroups of $\homeo_0(\R)$,  yet little  was known on the problem of describing its actions on the line; the reader can keep it in mind as a motivating special case of our main results. 

This paper is divided into three parts and has four main results (Theorems \ref{t-intro-C1}, \ref{t-intro-laminar-alternative}, \ref{t-intro-horograding} and \ref{t-intro-local-rigidity} below),  complemented by various applications, examples, and additional related results. Together, these results provide a satisfactory picture of actions on the line of micro-supported and locally moving subgroups of $\homeo_0(\R)$.  The remaining part of this introduction provides an outline of the main results of each part.

\section[Part I. Rigidity results for locally moving groups:  $C^1$ actions]{Part \ref{partI}.  Rigidity results for locally moving groups:  $C^1$ actions}

The main result of Part \ref{partI} is a rigidity result for actions of locally moving subgroups of $\homeo_0(\R)$ by $C^1$ diffeomorphisms on the line. For $G\subseteq \homeo_0(\R)$, we denote by $G_c$ the subgroup consisting of elements with relatively compact support. If $G$ is locally moving, a standard simplicity argument (Proposition \ref{p-micro-normal}) shows that the commutator subgroup $[G_c, G_c]$ of $G_c$ is simple and contained in every non-trivial normal subgroup of $G$. In  particular, a locally moving subgroup  $G\subseteq \homeo_0(\R)$ admits a largest proper quotient $G/[G_c, G_c]$.  The following is our first main result.

\begin{introthm}[Rigidity of $C^1$ actions of locally moving groups] \label{t-intro-C1}
Any irreducible action $\varphi\colon G\to \Diff^1_0(\R)$ of a locally moving subgroup $G\subseteq \homeo_0(\R)$ is  semi-conjugate either to
\begin{enumerate}[label=(\roman*)]
\item \label{i-C1-standard} the standard action of $G$ on $\R$, or
\item \label{i-C1-non-faithful} to a non-faithful action (which factors through the largest quotient $G/[G_c, G_c]$). 
\end{enumerate}
\end{introthm}

Note that in the setting of Theorem \ref{t-intro-C1}, the standard action of $G$  may, or may not, be semi-conjugate to some differentiable action (when it is not, case \ref{i-C1-standard} never occurs). For example, the standard PL action of Thompson's group $F$ is well known to be conjugate to an action by $C^\infty$ diffeomorphisms, by a construction of Ghys and Sergiescu \cite{GhysSergiescu} (a $C^1$ realization was previously known to Thurston; see the discussion by Cannon, Floyd, and Parry \cite[\S 7]{CFP}). In \S \ref{ssc:Stein}, we describe some results in the opposite direction for some relatives of Thompson's group $F$, inspired by the work of Bonatti, Lodha, and the last author \cite{BLT}.

On the other hand, locally moving subgroups of $\homeo_0(\R)$ often admit non-faithful actions on the line. In fact, such actions exist whenever $G$ is finitely generated,  as in this case the \emph{groups of germs} of $G$ at $\pm\infty$ are non-trivial proper quotients of $G$, which can be faithfully represented inside $\homeo_0(\R)$ (see Mann \cite{Mann}, where this result is attributed to Navas). Actions of the largest quotient $G/[G_c, G_c]$  have to be studied separately, and can be understood in some relevant cases. For example, the largest quotient of Thompson's group $F$  coincides with its abelianization $F=F/[F, F]\cong \Z^2$, and its actions on the real line are all semi-conjugate to an action by translations (i.e.\ a homomorphism to $(\R, +)$).

It should also be noted that case \ref{i-C1-non-faithful} can arise even if the original action  $\varphi$ is faithful, but there is a closed invariant subset on which $\varphi$ is not faithful (see Theorem \ref{t-lm-C1} for a more precise conclusion in this case). However, we will show in Corollary \ref{c-lm-C1-interval} that, under the additional assumption that $G$ is \emph{fragmentable} (Definition \ref{d-intro-fragmentable-subgroup}), this behavior can be ruled out for $C^1$ actions on compact intervals. This implies, for example, that every faithful irreducible   $C^1$ action  of $F$ on $[0,1]$ is semi-conjugate (on the interior) to its standard action. This conclusion  cannot be improved to obtain a topological conjugacy, as showed by the construction of smooth actions with an exceptional minimal set from  \cite{GhysSergiescu}.

Theorem \ref{t-intro-C1} can be compared to some previously known results, which are true even in $C^0$ regularity. It has long been known that any two locally moving subgroups of $\homeo_0(\R)$ are abstractly isomorphic if and only if their standard actions are topologically conjugate. This is customarily deduced from much more general reconstruction theorems of Rubin, holding true for groups of homeomorphisms of  locally compact spaces \cite{Rubin,Rubin2}. Thus, for a locally moving subgroup $G\subseteq \homeo_0(\R)$, its standard action is characterized as the unique faithful locally moving action, up to conjugacy. However, this does not allow to draw many conclusions on the problem of understanding the possible actions of $G$ on the line in general (this can be seen as a special case of a problem raised in \cite[p.\ 493]{Rubin}). A result announced by Ghys \cite{Ghys} shows that for Thompson's group $T$ acting on the circle, every non-trivial $C^0$ action on the circle is semi-conjugate to its standard action. Ghys' proof was based on bounded cohomology, a tool that is not available for actions on the line; some different proofs are available (see e.g.\  the work of Le Boudec and the second name author \cite[Theorem 4.17]{LBMB-sub}, or \S \ref{s-circle} here), and rely essentially on compactness of the circle. On the line, some results describing the structure of all actions (also in $C^0$ regularity) where known for some much larger (uncountable) subgroups of $\homeo_0(\R)$. In particular the group $\homeo_c(\R)$ of compactly supported homeomorphisms admits a unique $C^0$ action on the real line up to conjugacy, by a result of Militon \cite{Militon}  (relying on results of Matsumoto \cite{Matsumoto}). Recently, Chen and Mann \cite{ChenMann} obtained the same result for the groups $\Diff^r_c(\R)$ of compactly supported diffeomorphisms (with $r\neq 2$). 

The methods behind these results cannot be applied to countable locally moving subgroups of $\homeo_0(\R)$ to prove Theorem \ref{t-intro-C1}. One fundamental reason is that Theorem \ref{t-intro-C1} is merely not true for $C^0$ actions. For example, we have the following. 
\begin{prop}[Abundance of exotic $C^0$ actions] \label{p-intro-existence-exotic}
Thompson's group $F$ admits uncountably many semi-conjugacy classes of faithful minimal actions $\varphi\colon F\to \homeo_0(\R)$.\end{prop}
Note that a faithful minimal action of a group $G$ cannot be semi-conjugate to a non-faithful action of $G$.
In what follows, given a locally moving subgroup $G\subset \homeo_0(\R)$, we call an action $\varphi \colon G\to \homeo_0(\R)$ \emph{exotic}  if it is not semi-conjugate neither to the standard action of $G$, nor to any non-faithful action. Many different constructions of exotic actions will be provided throughout the paper, both for Thompson's group $F$ and for more general classes of locally moving subgroups of $\homeo_0(\R)$. The existence of such exotic actions is a difficulty that needs to be solved to prove Theorem \ref{t-intro-C1}: we need to show that they cannot be semi-conjugate to any $C^1$ action. This task naturally splits in two problems: the first is to understand the topological dynamics of exotic actions,  and the second is to  identify an appropriate obstruction to their $C^1$ smoothability.  The first (topological) problem is in some sense the central object of this work: in Part \ref{partI} we include only the results needed to prove Theorem \ref{t-intro-C1} (more precise structure theorems in the purely topological setting are the object of Parts \ref{partII} and \ref{partIII}). With these results in hand, the solution to the second (differentiable) problem would be substantially easier for actions  of class $C^r$ with $r>1$,  for which many general restrictions are known; actions by $C^1$ diffeomorphisms are notoriously much more flexible, making the argument more delicate.  Among the ingredients of the differentiable part of our proof, it is worth mentioning a result of Deroin, Kleptsyn, and Navas \cite{DKN-acta} (implicit in the work by Katok and Mezhirov \cite{KatokMezhirov}), which guarantees the existence of  hyperbolic fixed points in $C^1$ actions without invariant Radon measures, and the fact that every locally moving subgroup of $\homeo_0(\R)$ contains many copies of Thompson's group $F$ (Proposition \ref{p-chain}), which follows from the ``2-chain lemma'' of  Kim, Koberda, and Lodha \cite{KKL} (whose argument, based on a presentation of $F$, goes back to  Brin \cite{Ubiquity}).

\subsection*{Further results in Part \ref{partI}} As an offspring of the topological part of the proof of Theorem \ref{t-intro-C1}, we also provide  various more elementary rigidity results of similar flavor, which generalize with a unified approach some results already present in the literature. For example, we show that Theorem \ref{t-intro-C1} also holds true for actions by piecewise analytic homeomorphisms (Remark \ref{c-lm-PA}). In \S \ref{s-uncountable},  we provide  ``soft'' sufficient conditions on a huge (necessarily uncountable) locally moving subgroup $G\subseteq \homeo_c(\R)$, that imply uniqueness of the standard action up to conjugacy (recovering the previously mentioned results by Militon \cite{Militon}, and by Chen and Mann \cite{ChenMann}), and show a general rigidity result for locally moving subgroups of $\homeo_0(\mathbb{S}^1)$, based on elementary methods, that generalizes Ghys' results for Thompson's group $T$ \cite{Ghys}.

\section[Part II. The dynamics of $C^0$ actions: laminations and horogradings]{Part \ref{partII}. The dynamics of $C^0$ actions: laminations and horogradings}
In Part \ref{partII} we push our study further, by looking at actions of locally moving subgroups of $\homeo_0(\R)$ in $C^0$ regularity. Given the abundance of actions in this case, showcased by Proposition \ref{p-intro-existence-exotic}, the next natural step is to obtain structure theorems describing them. An important role in our answer to this problem will be played by the study of \emph{invariant laminations} for actions on the line, and the related notion of \emph{horograding}, that we introduce and develop. For simplicity, we shall discuss here only the case of \emph{minimal} actions, keeping in mind that for finitely generated groups,  this is no substantial loss of generality upon to passing to a semi-conjugate action.  
\begin{dfn}[Lamination]
A \emph{lamination} of the real line is a non-empty collection $\mathcal{L}$ of non-empty bounded open intervals, which is closed (for the topology of convergence of endpoints), and \emph{cross-free}, namely any two $I, J\in \mathcal{L}$ are either nested or disjoint.  A lamination is \emph{covering} if it defines a cover of $\R$. An action $\varphi\colon G\to \homeo_0(\R)$ will be called \emph{laminar} if it preserves a covering lamination.
\end{dfn}

Laminar actions are classically studied in the setting of group actions on the circle, where they appear naturally in connection with actions of Fuchsian groups and 3-manifolds groups, see e.g.\ the book of Calegari \cite[\S 2.1]{MR2327361}.
For actions on the real line, they have not been studied systematically yet, perhaps because of an apparent lack of examples. One of the main findings of this paper is that laminar actions appear naturally, although in contexts which are very different in nature from the classical examples on the circle.

 We begin Part \ref{partII} by studying general properties of laminar actions (Chapter \ref{sec.focalgeneral}). A particularly useful fact is that, in a laminar action $\varphi\colon G\to \homeo_0(\R)$, elements $g\in G$ can be classified according to the dynamics of $\varphi(g)$ into two types, that we call \emph{totally bounded} and \emph{pseudo-homothetic};  more general subgroups $H\subseteq G$ can be classified into four types: \emph{totally bounded},  \emph{pseudo-homothetic}, \emph{horocyclic}, and \emph{focal}.  Without getting into the details of this classification here  (see Proposition \ref{p-dyn-class-subgroups}), let us mention that it is modeled on the classical context of group actions on trees which fix an end (see Gromov \cite[\S 3.1]{Gromov-hyp-gps}, or Caprace, Cornulier, Monod, and Tessera \cite{amen-hyp-gps} for a more modern presentation). A reason for this analogy is that, given an invariant lamination for a minimal action $\varphi\colon G\to \homeo_0(\R)$, one can construct an action of $G$ on (topological) real tree, with a fixed end. This construction is not required for the statements or proofs of the main results, but we provide details in Chapter \ref{sec_focal_trees}, as it helps building intuition and it is used in some examples. It also plays an important role in our subsequent work \cite{BMRT-solv}.

Part \ref{partII} contains two main results. The first highlights a tight connection between  laminar and locally moving actions. These two types of actions are in strong contrast (for instance, a locally moving subgroup $G\subseteq \homeo_0(\R)$ acts minimally on the space of ordered $n$-tuples of distinct points of $\R$ for every $n\ge 1$, while the existence of an invariant lamination contradicts this property for $n=2$). It turns out that for a vast class of subgroups of $\homeo_0(\R)$, all faithful minimal actions satisfy the following dichotomy.
\begin{introthm}[Laminar/locally moving alternative] \label{t-intro-laminar-alternative}
Let $G\subseteq \homeo_0(\R)$ be a subgroup acting minimally on $\R$ and containing a compactly supported element.  Then every faithful minimal action $\varphi\colon G\to \homeo_0(\R)$ is either laminar or locally moving. The second possibility holds if and only if the standard action of $G$ on $\R$ is locally moving, and $\varphi$ is conjugate to it.  
\end{introthm}
For a subgroup  $G$ acting minimally on $\R$, the existence of a compactly supported element is equivalent to $G$ being micro-supported. In particular, Theorem \ref{t-intro-laminar-alternative} implies that all exotic actions of a locally moving subgroup $G\subseteq \homeo_0(\R)$ are laminar. When $G$ is not locally moving, the theorem implies that all faithful minimal actions of $G$, including the standard one, are laminar.  Under an additional hypothesis of finite generation, we obtain a much stronger result (Theorem \ref{t-intro-horograding}) that shows that all exotic actions are still tightly related to the standard action, although not via a semi-conjugacy. This is based on the following notion.
\begin{dfn}[Horograding]
A positive (respectively, negative) \emph{horograding} of a laminar action $\varphi\colon G\to \homeo_0(\R)$ by an irreducible action $\rho\colon G\to \homeo_0(\R)$, is a pair $(\mathcal{L}, \hor)$ consisting of a $\varphi$-invariant covering lamination $\mathcal{L}$, and a map $\hor \colon \mathcal{L}\to \R$, such that:
\begin{itemize}
\item for every intervals $I, J\in \mathcal{L}$ with $I\subseteq J$, we have $\hor(I)\leq \hor(J)$ (respectively, $\hor(I)\geq \hor(J)$);
\item for every $I\in \mathcal{L}$ and $g\in G$, we have $\hor(\varphi(g)(I))=\rho(g)(\hor(I))$. 
\end{itemize}
\end{dfn}

We shall see how the existence of a horograding of an action $\varphi$ by an action $\rho$ implies that the large-scale dynamics of $\varphi$ is controlled by $\rho$. For example, the type of each element $g\in G$ under $\varphi$ (according to the classification of elements in laminar actions) is determined by $\rho(g)$ (see Proposition \ref{p-dyn-class-elements-horograded}).  For this reason, we believe that this notion is a useful concept in the study of structure theorems for group actions on the line,  as it provides a possible replacement of a semi-conjugacy to a well-understood ``model'' action when the latter does not exist.

To state the second main result of Part \ref{partII}, we also need the following terminology.
\begin{dfn} \label{d-intro-fragmentable-subgroup} For a subgroup $G\subseteq \homeo_0(\R)$, the \emph{fragmentable subgroup} of $G$ is the (normal) subgroup $\Gfrag$ generated by elements with support contained in a half-line. When $G=\Gfrag$, we say that $G$ is \emph{fragmentable}. \footnote{When $G$ is locally moving, it is not difficult to show that $G$ is fragmentable in this sense if and only if for every finite cover $\R=I_1\cup \cdots \cup I_k$, the group $G$ is generated by the subgroups $G_{I_i}$. This property if often called the \emph{fragmentation property} in the literature on groups of homeomorphisms of manifolds, justifying the terminology.}
\end{dfn}

\begin{introthm}[Horograding exotic actions] \label{t-intro-horograding}
Let $G\subset \homeo_0(\R)$ be a subgroup acting minimally on $\R$, whose fragmentable subgroup $\Gfrag$ is non-trivial and finitely generated. Then $G$ is locally moving, and every  faithful minimal action $\varphi\colon G\to \homeo_0(\R)$ is either topologically conjugate to the standard action, or
 it is laminar and horograded by the standard action. 
\end{introthm}

Theorem \ref{t-intro-horograding} can be equivalently formulated as a classification of actions into three types, as follows. 
\begin{introthmprime} \label{t-intro-horograding-prime}
Let $G\subset \homeo_0(\R)$ be as in Theorem \ref{t-intro-horograding}. Then, every irreducible action $\varphi\colon G\to \homeo_0(\R)$ is semi-conjugate to an action in one of the following families.
\begin{itemize}
\item \emph{(Non-faithful)} A non-faithful action, factoring through the largest quotient $G/[G_c, G_c]$
\item \emph{(Standard)} The standard action of $G$ on $\R$.
\item \emph{(Horograded)} A faithful minimal laminar action, horograded by the standard action. 
\end{itemize}
\end{introthmprime}

\subsection*{Further results in Part \ref{partII}} In addition to these main results, a substantial portion of Part \ref{partII} is devoted to applications, discussion of concrete cases, examples and constructions of laminar actions, and investigation of their finer properties. 

A natural question is which locally moving subgroups $G\subset \homeo_0(\R)$ actually do admit minimal exotic (and hence laminar) actions. We do not have a complete answer to this question, and obtaining a full characterization appears to be difficult. We provide various sufficient condition on $G$ which imply  the existence of minimal exotic actions, and are satisfied by many familiar classes of locally moving subgroups of $\homeo_0(\R)$; these include all locally moving groups of piecewise linear or piecewise projective homeomorphisms of intervals, see \S \ref{ssec.germtype}. On the other hand, we manage to construct  a finitely generated, fragmentable, locally moving subgroup $G\subset \homeo_0(\R)$ without any exotic action (Theorem \ref{t-doubling-actions}). 

For groups within the scope of Theorem \ref{t-intro-horograding}, although the theorem provides a satisfactory understanding of the qualitative dynamics of all exotic actions,  the amount and finer properties of such actions depend subtly on $G$. An interesting family of examples is  provided by the \emph{Bieri--Strebel groups} $G(X; A,\Lambda)$ from \cite{BieriStrebel}. These are groups of PL homeomorphism of an open interval $X\subseteq \R$, naturally associated with the choice of a countable multiplicative subgroup $\Lambda \subset \R_{>0}$ and a countable $\Lambda$-submodule $A\subset \R$ (the definition is recalled in \S \ref{sc.BieriStrebel}).  The groups $G(X; A, \Lambda)$ are all locally moving, and Thompson's group $F$ belongs to this family. We describe a natural construction of exotic actions that applies to all groups $G(X; A, \Lambda)$, that we refer to as the \emph{jump cocycle construction} (explained in \S\S \ref{s.BSjump} and \ref{ss.Bieri-Strebelfocal}). This construction provides a family  of actions $\varphi_{\pm, \leq_\Lambda}$, which are all laminar and horograded by the standard action, indexed by the choice of an  orientation of the interval $X$, and of an invariant (pre)order $\leq_\Lambda$ on the abelian group $\Lambda$; depending on $\Lambda$.

 In Chapter \ref{s-few-actions}, we study actions of $G(X; A, \Lambda)$ in the case $X=\R$.  From the results of \cite{BieriStrebel} we can easily characterize when the group $G(\R; A, \Lambda)_{\mathsf{frag}}$  is finitely generated.  When this is the case, using Theorem \ref{t-intro-horograding}, we show that all exotic actions  of $G(\R; A, \Lambda)$ arise from the jump cocycle construction (Theorem \ref{t-BBS}). For suitable choices of $(A, \Lambda)$, this provides examples with exactly two minimal exotic actions up to conjugacy. In contrast we give an example where $G(\R; A, \Lambda)_{\mathsf{frag}}$ is not finitely generated,  and the group $G(\R; A, \Lambda)$ has additional minimal laminar actions that cannot be horograded by the standard action (Proposition \ref{p-Glambda-transcendental}), showing that the assumption  of finite generation of $\Gfrag$ in Theorem \ref{t-intro-horograding} cannot be dropped.
 
The situation for intervals $X\subsetneq \R$ is completely different, as showed by the case of Thompson's group $F$. It is well-known that $F$ is finitely generated and fragmentable, so Theorem \ref{t-intro-horograding} also applies to it. In this case there is a much richer amount of horograded laminar actions.  Indeed the jump cocycle construction provides only two non-conjugate actions, yet (as stated in Proposition \ref{p-intro-existence-exotic}), the group $F$ has uncountably minimal faithful actions. In Chapter \ref{s-F}, we provide a variety of constructions of minimal laminar actions of $F$, which all share a somewhat similar flavour and many qualitative properties (forced by Theorem \ref{t-intro-horograding}), but are pairwise non-conjugate.  These constructions admit an arborescent set of variations, which keep yielding new non-conjugate actions. For this reason, we are not able to conjecture any reasonably explicit classification of the minimal laminar actions of $F$ up to conjugacy.

\section[Part III. Local rigidity and the space of harmonic actions]{Part \ref{partIII}.  Local rigidity and the space of harmonic actions}

In Part \ref{partIII}, we use the results from Part \ref{partII} (notably Theorem \ref{t-intro-horograding-prime}) to study the topology of the space of actions on the line of locally moving groups. The main application, Theorem \ref{t-intro-local-rigidity} below,  is a \emph{local rigidity} result for a class of locally moving groups (including Thompson's group $F$). 

We write $\Homirr(G,\homeo_{0}(\R))$ for the space of irreducible actions of $G$ on $\R$. Recall that the space $\Homirr(G,\homeo_{0}(\R))$ can be endowed with the natural {compact-open topology}, which means that a neighborhood basis of a given action $\varphi\in \Homirr(G,\homeo_{0}(\R))$ 
is defined by considering, for every $\varepsilon>0$, finite subset $S\subset G$, and compact subset $K\subset \R$, the subset of actions
\[
\left\{
\psi\in \Homirr(G,\homeo_{0}(\R)): \max_{s\in S}\max_{x\in K} |\varphi(s)(x)-\psi(s)(x)|<\varepsilon
\right\}.
\]
\begin{dfn} 
An action $\varphi\in  \Homirr(G,\homeo_{0}(\R))$ is \emph{locally rigid} if there exists a neighborhood $\mathcal U\subset \Homirr(G,\homeo_{0}(\R))$ of $\varphi$ such that every $\psi\in \mathcal U$ is  semi-conjugate to $\varphi$ in an orientation-preserving way (we will simply say \emph{positively} semi-conjugate).
\end{dfn}

Recall that for $G\subset \homeo_0(\R)$ and an open interval $I\subset \R$, we denote by $G_I$ the subgroup of $G$ of elements that are supported in $I$. Our local rigidity criterion relies on a finite generation assumption on the subgroups $G_I$ (that we state here in a slightly non-optimal form for simplicity).

\begin{introthm}[Local rigidity]\label{t-intro-local-rigidity}
Let $G\subset \homeo_0(\R)$ be a locally moving subgroup. Assume that $G$ is finitely generated, and that there exist $y, z \in \R$ such that $G_{(-\infty, z)}$ and $G_{(y, +\infty)}$ are finitely generated. Then, the positive semi-conjugacy class of the standard action is open in $\Homirr(G,\homeo_{0}(\R))$. In particular, the standard action is locally rigid.
\end{introthm}

It is well known that the assumptions of Theorem \ref{t-intro-local-rigidity} are satisfied by Thompson's group $F$, so that in particular its standard action is locally rigid. They are also satisfied by various other examples. The assumption of finite generation of subgroups of the form $G_{(-\infty, y)}$, $G_{(z, +\infty)}$ is an actual restriction on $G$. It implies finite generation of the fragmentable subgroup $\Gfrag$ (which allows to use Theorem \ref{t-intro-horograding}), but it is stronger in general. With Proposition \ref{p-propononrigid}, we give an example of a finitely generated, fragmentable, locally moving subgroup $G\subset \homeo_0(\R)$, whose standard action is \emph{not} locally rigid, showing that this assumption in Theorem \ref{t-intro-local-rigidity}  is substantial.

A well-developed approach to local rigidity of group actions on the line is through the space  $\LO(G)$ of left-invariant orders on $G$. Any such order defines an irreducible action, by the so-called \emph{dynamical realization}, and  isolated points in $\LO(G)$ produce locally rigid actions (see Mann and the third author \cite[Theorem 3.11]{MannRivas}). This can be used to show that some groups, for instance braid groups (see Dubrovin and Dubrovina \cite{DubrovinDubrovina}, or the monograph by Dehornoy, Dynnikov, Rolfsen, and Wiest \cite{OrderingBraids}), do have locally rigid actions. However, the converse to this criterion does not hold,  and this approach has been more fruitful in the opposite direction, namely for showing that a group has no isolated orders from flexibility of the dynamical realization (see for instance the works by Navas \cite{Navas2010}, or by Alonso, and the first and third named authors \cite{ABR,AlonsoBrum}, as well as by Malicet, Mann, and the last two authors \cite{MMRT}). One difficulty underlying this approach  is that  it is usually not easy to determine when two orders in $\LO(G)$ give rise to semi-conjugate actions.  Also the dynamical realization of an isolated order never gives rise to a \emph{minimal} locally rigid action (as it is the case in Theorem \ref{t-intro-local-rigidity}).

To prove Theorem \ref{t-intro-local-rigidity}, we introduce a new approach to local rigidity and spaces of actions on the line, based on a compact space suggested by Deroin \cite{Deroin} as a dynamical substitute of the space of left-invariant orders. One way to construct this space is based on work by Deroin, Kleptsyn, Navas, and Parwani \cite{DKNP} on symmetric random walks on $\homeo_{0}(\R)$. Given a probability measure $\mu$ on $G$ whose support is finite, symmetric, and generates $G$, one defines the space $\Der_\mu(G;\R)$ as the subspace of $\Homirr(G,\homeo_{0}(\R))$ of \emph{normalized $\mu$-harmonic actions}\footnote{In the first version of this work, we proposed to name it the \emph{Deroin space} of $G$, acknowledging the intuition from \cite{Deroin} and the efforts of Bertrand Deroin to advertise this object.}, that is, actions of $G$ for which the Lebesgue measure is $\mu$-stationary  (see \S \ref{ssubsec.Deroin} for details, and our subsequent work \cite{BMRT-realisation} for a more abstract approach).  The space $\Der_\mu(G;\R)$ is compact and metrizable; it is endowed with a natural topological flow
\[\Phi\colon \R\times \Der_\mu(G;\R)\to \Der_\mu(G;\R),\]
defined on it by the conjugation action of the group of translations, and has the property that two actions in $\Der_\mu(G;\R)$ are positively (semi-)conjugate if and only if the are on the same $\Phi$-orbit.  It was shown in \cite{DKNP} that every irreducible action is semi-conjugate to a normalized $\mu$-harmonic action. As a starting point of our approach, we show that the harmonic representative can be chosen to depend continuously on the original action. More precisely,  Theorem \ref{t.retraction_Deroin} states that there exists a continuous retraction
\begin{equation} \label{e-intro-retraction}
	r\colon \Homirr(G, \homeo_0(\R)) \to \Der_\mu(G;\R),
\end{equation}
called the \emph{harmonic retraction}, which has the property that two irreducible actions $\varphi$ and $\psi$ are positively semi-conjugate, if and only if $r(\varphi)$ and $r(\psi)$ belong to the same $\Phi$-orbit. Our proof of continuity of the harmonic retraction is based on an alternative description of $\Der_\mu(G;\R)$ as a quotient of the space of left-invariant \emph{preorders} on $G$ (Theorem \ref{thm.homeoderoin}), which also implies that $\Der_\mu(G;\R)$ does not depend on the choice of the probability measure $\mu$, up to homeomorphism.

Continuity of the harmonic retraction implies that the structure of orbit closures of the flow $(\Der_\mu(G;\R), \Phi)$ carries topological information on how the positive semi-conjugacy classes sit inside $\Homirr(G, \homeo_0(\R))$. Notably, local rigidity of irreducible actions can be deduced by studying the dynamics of the translation flow on $\Der_\mu(G;\R)$ (Corollary \ref{c-rigidity-Deroin}). For groups as is Theorem \ref{t-intro-local-rigidity}, this can be done starting from the classification in Theorem \ref{t-intro-horograding-prime}: the space $\Der_\mu(G;\R)$ can be naturally decomposed into three $\Phi$-invariant subsets according to that result.
Showing local rigidity of the standard actions amounts to show that the exotic actions cannot accumulate on the standard one in $\Der_\mu(G; \R)$, and this requires understanding the dynamics of the flow $\Phi$ on them. Using Theorem \ref{t-intro-horograding-prime}, and studying properties of horograded actions in the space $\Der_\mu(G; \R)$, we manage to show that for a group $G$ as in Theorem \ref{t-intro-local-rigidity}, the flow $(\Der_\mu(G;\R), \Phi)$ has extremely simple dynamics in restriction to the faithful actions. Namely we show the following, which implies Theorem \ref{t-intro-local-rigidity}.

\begin{introthmderoin}[Dynamics of the translation flow on harmonic actions]\label{t-intro-deroin}
Let $G\subset \homeo_0(\R)$ be as in Theorem \ref{t-intro-local-rigidity}. Fix any symmetric probability measure $\mu$ supported on a finite generating set of $G$, and consider the associated flow $\Phi$ on the space of normalized harmonic actions $\Der_\mu(G;\R)$. Then, the subset $\mathcal{U}\subset \Der_\mu(G;\R)$ of faithful actions is an open $\Phi$-invariant subset, and the following hold:
\begin{enumerate}[label=(\roman*)]
\item \label{i-intro-deroin-proper} the restriction of $\Phi$ to $\mathcal{U}$ is proper;
\item \label{i-intro-deroin-rigid} if $\iota\in \mathcal{U}$ is a representative of the standard action, then its $\Phi$-orbit is an open subset of $\mathcal{U}$. 
 \end{enumerate}
\end{introthmderoin}

Figure \ref{fig-intro-F} represents the dynamics of the flow $(\Der_\mu(G;\R), \Phi)$ for Thompson's group $F$.
 
We conclude by pointing out another consequence of \ref{i-intro-deroin-proper}: properness of the flow $\Phi$ implies that the quotient space $\mathcal{U}/\Phi$ is a locally compact Polish space. 
Together with the continuity of the harmonic retraction \eqref{e-intro-retraction}, this has the following consequence for a group $G$ as in the statement: \emph{on the space of faithful minimal  actions of $G$ on the line, there exists a continuous function, with values in a locally compact Polish space, whose value is a complete invariant of the positive conjugacy class of an action}. The existence of such an invariant is not at all a general fact for group actions on the line: for general groups $G$ (such as non-abelian free groups), it is not even possible to find a \emph{measurable} invariant taking values in a standard Borel space (see the discussion in \S \ref{ssc.complexity}, and in particular Remark \ref{rem non smooth}). This is an important conceptual difference between actions on the line the case of group actions  the circle, which can be classified up to positive semi-conjugacy by the \emph{bounded Euler class} (see Ghys \cite{Ghys-bounded, Ghys}). This has been successfully used to understand the space of actions on the circle for various discrete subgroups of Lie groups (see for instance Burger and Monod \cite{BurgerMonod}, Matsumoto \cite{MatsumotoSurface}, Mann \cite{MannThesis},
or Mann and Wolff \cite{mann2018rigidity}), mapping class groups (Mann and Wolff \cite{MannWolff2018mcg}) or of Thompson's group $T$ (Ghys and Sergiescu \cite{GhysSergiescu,Ghys}). In contrast,  a similar global understanding of actions on the line was available only in more limited situations (mostly concerning rather small groups, or groups admitting none or very few actions).

For groups satisfying the assumptions of Theorem \ref{t-intro-deroin}, it is natural to study the topology of the quotient space $\mathcal{U}/\Phi$, which can be interpreted as a ``moduli space'' of faithful minimal actions on the line. While we can completely describe it in some cases (such as for the groups considered in Chapter \ref{s-few-actions}, in general we do not know much about it 
(for instance, for Thompson's group $F$, we do not know whether the standard action is its only isolated point; see Question \ref{q-rigid-actions-F}).  Obtaining a finer understanding of the topology of this space is the main problem left open by this work.

\begin{figure}[ht]
\centering
\includegraphics[scale=1]{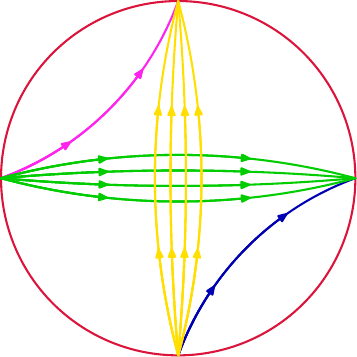}
\caption{The space of normalized harmonic actions of Thompson's group $F$.  The outer red circle parameterizes the actions of $F^{ab}\cong \Z^2$, and it is pointwise fixed by the flow $\Phi$. The remaining $\Phi$-orbits are the faithful actions. The purple and blue orbits correspond to the standard action and to its conjugate by the reflection, and are transversely isolated (this gives local rigidity). The  pencils of green and yellow orbits are the laminar actions: both pencils contain uncountably many orbits, admit a compact transversal to the flow, and the shown convergence to  limit points is uniform. The four limit points correspond to four special cyclic actions, given by the epimorphisms $F\to \Z$ associated with the derivatives at the endpoints in the standard PL action on $(0,1)$. See Figure \ref{fig-F-Der} and  \S \ref{s-F-Deroin} for details.}\label{fig-intro-F}
\end{figure} 

\section{How to read this paper} The next chapter (Chapter \ref{s-preliminaries}) contains the common background results and basic terminology for the three parts, about general group actions on the line and on ordered sets. The remaining chapters are divided into three parts, where each part has little dependence from the others, as we explain.  Part \ref{partI} is clearly totally independent. The main results of Part \ref{partII} require the discussions from Chapters \ref{sec.locallymgeneral} and \ref{s-lm-2}, while several examples in Part \ref{partII} are generalizations of the constructions in Chapter \ref{ss.exoticactions}. Finally, the results of Part \ref{partIII} are mainly based on the results in Chapters \ref{sec.focalgeneral} and \ref{sec.locandfoc} from Part \ref{partII}.

Chapters \ref{ch-additional} and \ref{sec_focal_trees}--\ref{s-F} can be skipped at first reading, as they do not affect the understanding of the main results.

\subsection*{Acknowledgments} The first version of this work was completed in April 2021. Since then, the work has been completely revisited, thanks also to the very careful work of the anonymous referee, and to the opportunities that have been given to us to talk about this project. For this, we thank the welcoming atmosphere of the trimester \emph{Geometric Group Theory} at CRM, Montreal, of the \emph{thematic program on geometric group theory and low dimensional geometry and topology}, at ICMAT, Madrid, of the group of Sang-hyun Kim at KIAS, Seoul, and of the trimester \emph{Group actions and rigidity: around the Zimmer program}, at IHP, Paris.

A special thanks goes to Andr\'es Navas for his many suggestions. We are grateful to Sebastián Hurtado, Bertrand Deroin, and Maxime
Wolff for fruitful discussions about the space of harmonic actions and to Maxime Wolff also
for discussions regarding the Plante-like product. We also acknowledge Todor Tsankov for useful
clarifications about Borel reducibility of equivalence relations which motivated the discussion
in \S \ref{ssc.complexity}, and Yash Lodha for many discussions related to various locally moving groups of
homeomorphisms of the line.

\chapter{Notation and general preliminaries}\label{s-preliminaries}

\section{Actions on the line} 
\subsection{Some notation for actions}
In this work we are mainly concerned with orientation-preserving actions on the real line, that is, homomorphisms $\varphi\colon G\to \homeo_0(\R)$.
We will almost always be interested in actions without (global) fixed points, which will often be called \emph{irreducible}\index{action!irreducible} for short. Note that every action $\varphi\colon G\to \homeo_0(\R)$ can be described in terms of irreducible actions, just considering all restrictions
\[
\dfcn{\varphi_J}{G}{\homeo_{0}(J)\cong \homeo_{0}(\R)}{g}{\varphi(g)\restriction_J}
\]
of the action $\varphi$ to minimal $\varphi$-invariant open intervals $J\subset \R$. We write $\Hom(G,\homeo_{0}(\R))$ for the space of order-preserving actions of the group $G$, endowed with the compact-open topology, and $\Homirr(G,\homeo_{0}(\R))\subset \Hom(G,\homeo_{0}(\R))$ for the subspace of irreducible actions.

Given  $f\in \homeo(\R)$, we write $\fix(f)=\{x\in \R : f(x)=x\}$ for the set of fixed points\index{element!fixed points}, and $\supp(f)=\R\setminus \fix(f)$ for its \emph{support}\index{element!support}. Note that by definition $\supp(f)$ is an open set (we do \emph{not} take its closure unless specified). For a subgroup $G\subseteq \homeo_0(\R)$\footnote{As we will be working with several orders throughout the paper, we prefer to reserve the symbol $<$ for order relations, and use the inclusion $\subset$ when referring to subgroups.} and a point $x\in \R$, we write $\stab_G(x)$ for the \emph{stabilizer} of $x$.  We denote by $\homeo_c(\R)$ the group of homeomorphisms whose support is relatively compact, and we occasionally write $A\Subset B$ when $A$ is relatively compact in $B$. Given an action $\varphi\colon G\to \homeo(\R)$ and $g\in G$, we set $\fixphi(g)=\fix(\varphi(g))$ and $\suppphi(g)=\supp(\varphi(g))$, and write $\stab^\varphi_G(x)$ for the stabilizer. When there is no risk of confusion, we write $g.x$ instead of $\varphi(g)(x)$. The notation $g(x)$ will be reserved to the case when $G$ is naturally given as a subgroup $G\subset \homeo_0(\R)$ to refer to its standard action (but never to another action $\varphi\colon G\to \homeo_0(\R)$).

For $x\in \R\cup\{\pm \infty\}$, we denote by $\Germ(x)$\index{group!group of germs} the \emph{group of germs} of homeomorphisms that fix $x$. Recall that this is defined as the group of equivalence classes of homeomorphisms $f\in \homeo_0(\R)$ that fix $x$, where two such homeomorphisms are identified if they coincide on a neighborhood of $x$. By considering only one-sided neighborhoods, one gets two groups $\Germ_-(x)$ and $\Germ_+(x)$, the groups of \emph{right} germs and the group of \emph{left} germs, respectively, such that $\Germ(x)=\Germ_-(x)\times \Germ_+(x)$. If $G$ is a group of homeomorphisms that fixes $x$, we denote by $\Germ(G, x)$ the group of germs induced by elements of $G$, and by $\mathcal{G}_x\colon G\to \Germ(G, x)$ the associated homomorphism.

\subsection{Commuting actions}
Every homeomorphism $g\in \homeo_{0}(\R)$ is basically determined by its set of fixed points $\fix(g)$ and by how it acts on every connected component of its support $\supp(g)=\R\setminus \fix(g)$. Therefore, it is fundamental to understand the set of fixed points of a given element, or at least to be able to say whether it is empty or not. For this, a very useful relation is that for an element $g\in \homeo_{0}(\R)$ and a subgroup $H\subseteq \homeo_0(\R)$, one has $g(\fix(H))=\fix(gHg^{-1})$. In particular, when $G\subseteq \homeo_0(\R)$ and $H\trianglelefteq G$ is a normal subgroup, for every $g\in G$ one has $g(\fix(H))=\fix(H)$. This holds in particular for commuting subgroups, for which we have the following observation, easily obtained from the fact that the set of fixed point is a closed subset.

\begin{lem} \label{l.fixed_commuting}
	Consider two commuting subgroups $H_1$ and $H_2$ of $G\subseteq \homeo_{0}(\R)$ (that is, $[h_1,h_2]=\id$ for every $h_1\in H_1$ and $h_2\in H_2$). Suppose that both $\fix(H_1)$ and $\fix(H_2)$ are non-empty. Then $H_1$ and $H_2$ admit a common fixed point.
\end{lem}

This lemma will be used in the text without explicit reference.

\subsection{Semi-conjugacy}
It is customary in the field to consider actions up to \emph{semi-conjugacy}\index{semi-conjugacy}. This means that not only we do not really take care of the choice of coordinate on $\R$ (which corresponds to the classical notion of \emph{conjugacy}), but we want to consider only the interesting part of the dynamics of the action. This was first formalized  by Ghys in his work on bounded Euler class \cite{Ghys-bounded}, but the definition has been  unanimously fixed only recently. We follow here the presentation by Kim, Koberda, and Mj  \cite[Definition 2.1]{KKMj}, although we allow order-reversing semi-conjugacies. For the statement we will say that a map $h\colon \R\to \R$ is \emph{proper} if its image $h(\R)$ is unbounded in both directions of the line.

\begin{dfn}
	Let $\varphi,\psi\colon G\to \homeo(\R)$ be two actions of a group $G$ on the real line. We say that $\varphi$ and $\psi$ are \emph{semi-conjugate} if there exists a proper monotone map $h\colon \R\to\R$ such that
	\begin{equation}\label{eq:semiconj}
		h\varphi(g)=\psi(g)h\quad \text{for every }g\in G.
	\end{equation}
	Such a map $h$ is called a \emph{semi-conjugacy} between $\varphi$ and $\psi$. Note that if $\varphi$ and $\psi$ are irreducible, the requirement that the map $h$ be proper follows automatically from the equivariance \eqref{eq:semiconj} and thus can be omitted.
	When $h\in\homeo(\R)$, we say that $h$ is a \emph{conjugacy} between $\varphi$ and $\psi$, in which case we say that $h$ is a \emph{conjugacy}.
	When $h$ is non-decreasing (respectively, non-increasing), we say that $\varphi$ and $\psi$ are \emph{positively} (respectively, \emph{negatively}) semi-conjugate.
\end{dfn}

\begin{rem}
	{The previous definition can be naturally extended to group actions on arbitrary non-empty open intervals $X,Y\subseteq \R$, using homeorphisms $X\cong Y\cong \R$.}
\end{rem}

\begin{rem} Both conjugacy and semi-conjugacy are equivalence relations (for conjugacies this is obvious, for semi-conjugacies the reader can check \cite[Lemma 2.2]{KKMj}). Notice that in the semi-conjugacy case, we do not require that $h$ be continuous; indeed, being continuously semi-conjugate is not even a symmetric relation. \end{rem}

In a few places (essentially only in \S \ref{s-circle}), we will  need the analogous notion of  semi-conjugacy for actions on the circle, which is defined as follows. 

\begin{dfn} Let $\varphi, \psi\colon G\to \homeo_0(\mathbb{S}^1)$ be two actions on the circle.  They are semi-conjugate if there exist a group $\widetilde{G}$ which is a central extension of $G$ of the form
	\[1\to C\to \widetilde{G}\to G\to 1,\]
	with $C\cong \Z$, and two semi-conjugate actions $\widetilde{\varphi}, \widetilde{\psi}\colon \widetilde{G}\to \homeo_0(\R)$ which both map $C$ to the group $\Z$ of integer translations, and  which descend to the quotient, respectively, to the actions $\varphi$ and $\psi$ of $G=\widetilde{G}/C$  on $\mathbb{S}^1=\R/\Z$. 
\end{dfn}

Given an action $\varphi\colon G\to \homeo(\R)$, one can consider the \emph{reversed action}\index{action!reversed} $\widehat\varphi\colon G\to \homeo(\R)$, defined by conjugating $\varphi$ by the order-reversing isometry $x\mapsto -x$. After our definition, the actions $\varphi$ and $\widehat{\varphi}$ are conjugate.

Given a monotone map $h\colon \R \to \R$, we denote by $\gap(h)\subset \R$ the open subset of points at which $h$ is locally constant; we also write $\core(h)=\R\setminus \gap(h)$. Note that when $h\colon \R\to \R$ is a semi-conjugacy between two irreducible actions $\varphi,\psi\colon G\to \homeo(\R)$ (in the sense that the equivariance \eqref{eq:semiconj} holds), then $\core(h)$ is a closed $\varphi$-invariant subset and $h(\core(h))=h(\R)$.
The following folklore result is a sort of converse to this observation, and describes the inverse operation of blow up of orbits (compare for instance \cite[Theorem 2.2]{KKMj}). 

\begin{thm}\label{prop.basica}
	Let $\varphi\colon G\to\homeo_0(\R)$ be an irreducible action, and  $F\subseteq \R$ a non-empty closed $\varphi$-invariant subset. Then, there exist an action $\varphi_F\colon G\to \homeo_0(\R)$, and a continuous positive semi-conjugacy $h\colon \R\to \R$ between $\varphi$ and $\varphi_F$ such that $\core(h)=F$.
\end{thm}

We follow with an easy application of Theorem \ref{prop.basica} that will be repeatedly used in the article when discussing actions coming from quotients. 

\begin{cor}\label{cor.normalsemicon} Let $\varphi\colon G\to \homeo_0(\R)$ be an irreducible action, and let $N\trianglelefteq G$ be a normal subgroup. Then $\varphi$ is semi-conjugate to an action of the quotient $G/N$ if and only if $\fixphi(N)\neq\varnothing$.\end{cor}

\begin{proof}  Notice that, since $N$ is normal, the subset $F=\fixphi(N)$ is closed and $\varphi(G)$-invariant. Assume that $F\neq\varnothing$; we consider the action $\varphi_F$ given by Theorem \ref{prop.basica}. Clearly we have $N\subseteq \ker \varphi_F$.	The other implication is trivial. 
\end{proof}  

The following lemma basically states that the semi-conjugacy class is determined by the action on one orbit.

\begin{lem}\label{lem.semiconjugacy} Let  $\varphi,\psi\colon G\to\homeo_0(\R)$ be  irreducible actions. Consider a non-empty $\varphi$-invariant subset $\Omega\subseteq \R$, and  a monotone map $j_0\colon \Omega\to\R$ which is equivariant in the sense that  it satisfies
	\[
	\psi(g) j_0=j_0\varphi(g)\restriction_\Omega\quad\text{for every } g\in G.\]
	Then $j_0$ can be extended to a semi-conjugacy $j\colon \R\to \R$ between $\varphi$ and $\psi$.
\end{lem}
\begin{proof} Consider the map $j\colon \R\to \R$ defined by \[j(x)=\sup\{j_0(y):y\in \Omega\text{ and }y\leq x\}.\]
	Since $j_0$ is monotone, we get that $j$ is an extension of $j_0$. It is direct from the definition of $j$ that it is monotone and equivariant; therefore $j$ defines a semi-conjugacy, as desired. 
\end{proof}

\subsection{Minimal invariant sets}\label{s:minimal-set}

Next, we briefly recall basic facts about minimal invariant sets for group actions on the line; see for instance the book of Navas \cite[\S 2.1.2]{Navas-book} for a more detailed discussion.

\begin{dfn}
	A non-empty subset $\Lambda\subset \R$ is a \emph{minimal invariant set}\index{minimal invariant set} for an action $\varphi\colon G\to\homeo_0(\R)$, if it is closed, $\varphi$-invariant, and contains no proper, non-empty, closed $\varphi$-invariant subsets. When $\Lambda=\R$, we say that the action $\varphi$ is \emph{minimal}\index{action!minimal}. 
	When $\Lambda$ is a proper perfect subset we say that $\Lambda$ is an \emph{exceptional minimal invariant set}.
	On the other hand, when the image $\varphi(G)$ is generated by a homeomorphism without fixed points, we say that the action $\varphi$ is \emph{cyclic}\index{action!cyclic}.
\end{dfn}

\begin{rem}\label{r.maynotexist}
	A global fixed point is always a minimal invariant set. When an irreducible action is semi-conjugate to a cyclic action, minimal invariant sets are given by  discrete orbits.
	On the other hand,
	for irreducible actions $\varphi\colon G\to\homeo_0(\R)$ which are non-semi-conjugate to any cyclic action, when a minimal invariant set $\Lambda$ exists, it is unique, and  it is contained in the closure of any $\varphi$-orbit. 
	
	Note however that minimal invariant sets may not exist for general group actions. An archetypical example is given by an action of the group $\Z^\infty =\bigoplus_\Z \Z$ in which each canonical generator has fixed points but there is no fixed point for the action. This happens, for instance, in the dynamical realization (see Lemma \ref{lem.dynreal}) of the lexicographic ordering of $\Z^\infty$. 
\end{rem}

When the acting group is finitely generated, there is  always a minimal invariant set for the action; see e.g.\ Navas \cite[Proposition 2.1.12]{Navas-book}, whose proof actually gives the following more general criterion for the existence of a minimal invariant set.

\begin{lem}\label{l.zorn_minimal}
	An irreducible action $\varphi\colon G\to \homeo_0(\R)$ admits a minimal invariant set if and only if there exists a compact interval $I\subset \R$ intersecting every $\varphi$-orbit. 
\end{lem}
\begin{proof}
	The ``if'' direction is proven by a standard Zorn argument, in the same way as \cite[Proposition 2.1.12]{Navas-book}. Conversely, assume that $\Lambda$ is a minimal invariant set. Suppose first that $\Lambda$ is a discrete orbit. Choose $x\in \Lambda$ and $g\in G$ such that $g.x$ is the successor of $x$ in $\Lambda$. Then $I=[x, g.x]$ intersects every $\varphi(g)$-orbit, and thus every $\varphi$-orbit. 
	Otherwise, $\Lambda$ is the unique minimal set for $\varphi$, and the orbit closure of every $x\in \R$ contains $\Lambda$ (see Remark \ref{r.maynotexist}). In this case, we can choose any $I$ whose interior intersects $\Lambda$ non-trivially. 
\end{proof}
We will occasionally use the following terminology.
\begin{dfn}\label{d-wandering-interval}
	We say that an interval $I\subset \R$ is \emph{wandering} for an action $\varphi:G\to\homeo_{0}(\R)$ if for every $g\in G$, either $g.I=I$, or $g.I\cap I=\varnothing$.
\end{dfn}

\begin{rem}\label{r.semi_is_continuous}  
	Let $h\colon \R\to \R$ be a semi-conjugacy between two actions $\varphi,\psi\colon G\to \homeo_0(\R)$, in the sense that \eqref{eq:semiconj} holds.
	It is immediate to show that when $\psi$ is minimal, the semi-conjugacy $h$ is continuous, and that when $\varphi$ is also minimal, then $h$ is a conjugacy. Indeed, the subsets $\core(h)$ and $\overline{h(\R)}$ are closed subsets, invariant under $\varphi$ and $\psi$, respectively. 
\end{rem}

\subsection{Canonical model}
We next introduce a family of canonical representatives for the semi-conjugacy relation when $G$ is finitely generated. Such representatives are well defined up to conjugacy. Later, in \S \ref{sec.harmder}, we will define a (much less redundant) family of representatives for the semi-conjugacy relation consisting of normalized $\mu$-harmonic actions.

The following notion corresponds to the \emph{minimalization} considered by Kim, Koberda, and Mj in \cite[Definition 2.3]{KKMj}.

\begin{dfn} An action $\varphi\colon G\to\homeo_0(\R)$ is a \emph{canonical model}\index{action!canonical model} if it is either minimal or cyclic. 
\end{dfn}

We have the following consequence of the discussion in \S \ref{s:minimal-set} and Theorem \ref{prop.basica}.

\begin{cor}\label{cor.basica} Let  $\varphi\colon G\to\homeo_0(\R)$ be an irreducible action that admits a minimal invariant set (this is automatic when $G$ is finitely generated). Then $\varphi$ is semi-conjugate to a canonical model, unique up to conjugacy.
\end{cor}

A similar discussion holds for actions on the circle. In this case we  say that an action $\varphi\colon G\to \homeo_0(\T)$ is a canonical model if it is either minimal or conjugate to an action whose image is a non-trivial finite cyclic group of rotations (in which case we say that it is \emph{cyclic}). However a crucial simplifying difference  is that in this case, by compactness, \emph{every} group action on the circle admits a non-empty minimal invariant set (regardless on whether $G$ is finitely generated or not). This yields the following well-known fact (see e.g.\ Ghys \cite{Ghys}).
\begin{prop} \label{p-circle-canonical} Every group action $\varphi\colon G\to \homeo_0(\T)$ on the circle without fixed points is semi-conjugate to a canonical model.
\end{prop}

After Corollary \ref{cor.basica}, semi-conjugacy classes of finitely generated group actions can be divided into two families: cyclic and minimal actions. However, for practical purposes, it is preferable to split further the case of minimal actions and the following classical notion will be important.

\begin{dfn} For $M\in \{\R,\T\}$, we say that a minimal action $\varphi\colon G\to\homeo_0(M)$ is \emph{proximal}\index{action!proximal} if for every non-empty open intervals $I,J\subsetneq M$ with $I$ bounded and $\overline I\neq M$, there exists an element $g\in G$ such that $g.I\subset J$.
\end{dfn}

For further reference in \S \ref{s-circle}, let us point out the following.

\begin{rem}\label{rem.obst_proximal}
	When a subgroup $G\subseteq \homeo_0(M)$ commutes with a non-trivial element $h\in \homeo_0(M)$, then its action cannot be proximal. In fact, this is a well-known obstruction for minimal group actions on arbitrary topological spaces (see \cite[Lemma 3.3]{Glasner}).
\end{rem}

A crucial feature of minimal  actions on the circle and the line is that a sort of a converse also holds. In the case of the circle, this goes back to Antonov \cite{Antonov} and was rediscovered by Margulis \cite{Margulis} and Ghys \cite[Theorem 5.14]{Ghys}, and an analogous result for the real line is  provided by   Malyutin \cite{Malyutin}, although McCleary had a similar statement for ordered groups  (this appears in the book of Glass \cite[Theorem 7.E]{Glass}; see also the discussion in the monograph by Deroin, Navas, and the third named author \cite[\S 3.5.2]{GOD}). Given a subgroup $G\subseteq\homeo_0(M)$ with $M\in \{\R,\T\}$, its \emph{centralizer} (in $\homeo_0(M)$) is the subgroup
\[C(G):=\left \{h\in\homeo_0(M):gh=hg\text{ for every }g\in G\right \}.\]
We will write $C^\varphi:=C({\varphi(G)})$ for the centralizer of a given action $\varphi\colon G\to\homeo_0(M)$.  
We have the following two results.

\begin{thm}\label{t-Margulis}
	We have the following alternative for minimal actions $\varphi\colon  G\to \homeo_{0}(\T)$:
	\begin{itemize}
		\item either $C^\varphi$ is trivial, in which case $\varphi$ is proximal, or
		\item $C^\varphi$ is a non-trivial finite cyclic subgroup acting freely, and the action on the topological circle $\T/C^\varphi$ is proximal, or
		\item $C^\varphi$ is conjugate to the group of rotations $\R/\Z$, and $\varphi(G)$ is conjugate to a subgroup of it. 
	\end{itemize}
\end{thm}

\begin{thm}\label{t-centralizer}
	We have the following alternative for minimal actions $\varphi\colon  G\to \homeo_{0}(\R)$:\begin{itemize}
		\item either $C^\varphi$ is trivial, in which case $\varphi$ is proximal, or
		\item $C^\varphi$ is an infinite cyclic subgroup acting freely, and the action on the topological circle $\R/C^\varphi$ is proximal, or
		\item $C^\varphi$ is conjugate to the group of translations $(\R, +)$, and $\varphi(G)$ is conjugate to a subgroup of it. 
	\end{itemize}
\end{thm}

\section{Preorders and group actions}\label{sc.preorders}

\subsection{Preordered sets}
In this work, by a \emph{preorder}\index{preorder} on a set $\Omega$, we mean a transitive binary relation $\leq$ which is \emph{total}, in the sense that for every $x,y\in \Omega$ we have $x\leq y$ or $y\leq x$ (possibly both relations can hold).

We write $x\lneq y$ whenever $x\leq y$ but  it does not hold that $y\leq x$.  On the other hand, when $x\leq y$ and $y\leq x$, we say that $x$ and $y$ are \emph{equivalent} and denote by $[x]_\leq$ the equivalence class of $x$ (we will simply write $[x]$ when there is no risk of confusion).

\begin{rem}\label{r.equivalence_po}
	When $[x]=\{x\}$ for every $x\in \Omega$, the preorder $\le$ is a \emph{total} order; if so, we prefer to denote it by $<$, and say that $(\Omega,<)$ is a totally ordered set to stress this property.
	
	A preorder $\leq$ on $\Omega$ induces a total order on the set of equivalence classes $\{[x]_\leq:x\in \Omega\}$, by declaring  $[x]<[y]$ whenever $x\lneq y$. 
\end{rem}

\begin{dfn}
	We say that a map  between preordered sets $f\colon (\Omega_1,\leq_1)\to (\Omega_2,\leq_2)$ is \emph{(pre)order-preserving} if $x\leq_1 y$ implies $f(x)\leq_2f(y)$.
	
	On the other hand, given a map $f\colon \Omega_1\to(\Omega_2,\leq)$, we define the \emph{pull-back} of $\leq$ by $f$ as the preorder $f^\ast(\leq)$ on $\Omega_1$, denoted by $\preceq$ here, so that $x_1\preceq x_2$ if and only if $f(x_1)\leq f(x_2)$.
\end{dfn}

\begin{dfn}
	An \emph{automorphism} of a preordered set $(\Omega,\le)$ is an order-preserving bijection $f\colon (\Omega,\le)\to (\Omega,\le)$, whose inverse is also order-preserving. We denote by $\Aut(\Omega,\le)$ the group of all automorphisms of $(\Omega,\le)$.
	
	An \emph{order-preserving action} of a group $G$ on a preordered set $(\Omega,\le)$ is a homomorphism $\psi\colon G\to \Aut(\Omega,\le)$.
\end{dfn}

\begin{rem}\label{r.pullback_po}
	Given two actions $\psi_1\colon G\to \Bij(\Omega_1)$ and $\psi_2\colon G\to \Aut(\Omega_2,\le)$, and an equivariant map $f\colon \Omega_1\to (\Omega_2,\le)$, we have that the pull-back preorder $f^*(\le)$ is preserved by $\psi_1(G)$.
\end{rem}

\begin{dfn}
	A subset $A$ of a preordered set $(\Omega,\leq)$ is \emph{($\leq$-)convex} if, whenever $x\leq y\leq z$ and $x,z\in A$, one has $y\in A$.
\end{dfn}

\begin{rem}It is a direct consequence of the definitions that if $f\colon (\Omega_1,\le_1)\to (\Omega_2,\le_2)$ is order-preserving and $A\subset (\Omega_2,\le_2)$ is $\le_2$-convex, then the preimage $f^{-1}(A)$ is $\le_1$-convex. This fact will be used several times without direct reference.
\end{rem}

\subsection{Preorders on groups} \label{s-preorders} A preorder on a group $G$ is \emph{left invariant} if  $h\leq k$ implies $gh\leq gk$ for all $g,h,k\in G$. In other words, the left-multiplication gives an action by automorphisms $G\to \Aut(G,\le)$.
Recall also that a preorder on $G$ is \emph{bi-invariant} if it is preserved by left and right multiplications.

By invariance and Remark \ref{r.equivalence_po}, given a left-invariant preorder $\leq$ on $G$, the equivalence class $[1]_\leq$ is a subgroup of $G$, and $\leq$ is a left-invariant order on $G$ if and only if $[1]_\leq=\{1\}$.
The subgroup $[1]_{\leq}$ is called the \emph{residue}. We say that a left-invariant preorder $\leq$ on $G$ is \emph{trivial} whenever $[1]_\leq=G$, and \emph{non-trivial} otherwise. 
We denote by $\LPO(G)$ the set of all non-trivial left-invariant preorders on $G$.

{From now on, by a preorder on a group, we \emph{always} mean a non-trivial left-invariant preorder}.  We endow $\LPO(G)$ with the product topology induced by the realization of preorders as subsets of $\{\leq,\gneq\}^{G\times G}$. It turns out that with this topology, $\LPO(G)$ is a metrizable and totally disconnected topological space, which is moreover compact whenever $G$ is finitely generated (see Antolín and the third author  \cite{AntolinRivas}, where preorders are called relative orders). Clearly $\LPO(G)$ contains the more classical compact space $\LO(G)$ of \emph{left-invariant orders} of a group $G$.
For a modern treatment of left-invariant orders and preorders see the monograph by Deroin, Navas, and the third author \cite{GOD}, Antol\'in, Dicks, and \v Sunic \cite{AntolinDicksSunic}, or Decaup and Rond \cite{decaup2019preordered}. 

\begin{dfn} The \emph{positive cone} of a preorder $\leq\in\LPO(G)$ is the subset $P_\leq=\{g\in G:g\gneq 1\}$. Similarly, the subset $N_{\leq}=\{g\in G:g\lneq 1\}$ defines the \emph{negative cone} of $\le$.
\end{dfn}

\begin{rem}\label{r.cones}A preorder $\leq\in\LPO(G)$ induces a partition $G=P_{\leq}\sqcup[1]_\leq\sqcup N_{\leq}$. Note that $P_\leq$ and $N_\leq$ are semigroups and satisfy $P_\leq^{-1}=N_\leq$, where $P_\leq^{-1}:=\{g^{-1}:g\in P_\leq\}$. Reciprocally, given a partition $G=P\sqcup H\sqcup N$ such that
	\begin{itemize}
		\item $P$ is a semigroup,
		\item $N=P^{-1}$,
		\item $H$ is a proper (possibly trivial) subgroup, and
		\item $HPH\subseteq P$,
	\end{itemize}
	there exists a preorder $\leq\in\LPO(G)$ such that $P=P_\leq$, $H=[1]_\leq$, and $N=N_\leq$.  See  Decaup and Rond \cite{decaup2019preordered}  for details.
\end{rem}

When $H\subseteq G$ is a $\leq$-convex subgroup, the preorder $\leq$ naturally descends to a preorder $\le_{G/H}$ on the left-coset space $G/H=\{gH:g\in G\}$, defined by setting $g_1H\leq_{G/H} g_2H$ if $g_1h\le g_2h$ for some $h\in H$. This condition does not depend on the choice of $h$, since left cosets are convex. The preorder $\le_{G/H}$  is invariant under left-multiplication of $G$ on $G/H$. 

\begin{dfn}
	Given a $\leq$-convex subgroup $H\subseteq G$, we define the \emph{quotient preorder} of $\leq\in \LPO(G)$ under $H$, as the preorder $\leq_H\in\LPO(G)$ given by the pull-back $\leq_H:=p^\ast(\leq_{G/H})$, where $p\colon G\to G/H$ is the coset projection.
	Equivalently, we can define $\leq_H$ by setting $P_{\leq_H}=P_\leq\setminus H$. 
\end{dfn}

\subsection{Dynamical realizations of actions on totally ordered sets}\label{sec.dynreal} One general principle that we often use is that for building actions of a group $G$ on the line by homeomorphisms, one may start by finding an order-preserving action $\psi\colon G\to\Aut(\Omega,<)$ on a totally ordered set. If the order topology on $\Omega$ is nice enough, for instance when $\Omega$ is {(infinite\footnote{{Countable ordered sets are assumed to be infinite here, otherwise some of the statements could be wrong for trivial reasons.}})} countable, then the action $\psi$ can be ``extended'' to an action $\varphi\colon G\to \homeo_0(\R)$ in such a way that there exists an equivariant order-preserving map $i\colon \Omega \to \R$. Under some mild extra conditions, we call such a $\varphi$ a \emph{dynamical realization} of $\psi$. See Definition \ref{d-dynamical-realisation}.

\begin{rem}\label{rem.dynrealpreo} It is a classical fact that a countable group is left orderable if and only if it embeds into $\homeo_0(\R)$ (a fact that we trace back to Conrad \cite{Conrad} in the abelian setting, see the already cited \cite{GOD,Ghys} for a proof of the general case). In fact, one direction of the  proof consists in building a dynamical realization of the left-multiplication action of a countable left-ordered group on itself. Analogously, one can show that a countable group admits a left-invariant preorder if and only if it admits a (possibly non-faithful) non-trivial action by order-preserving homeomorphisms of the real line; see the work of Antol\'in and third author \cite{AntolinRivas}. 
\end{rem}

We now proceed to formally define what we mean by dynamical realization.

\begin{dfn}\label{dfn.goodbehaved} Let $(\Omega,<)$ be a countable totally ordered set. We say that an injective order-preserving map $i\colon \Omega\to\R$ is a \emph{good embedding} if its image is unbounded in both directions, and every connected component $I$ of the complement of $\overline{i(\Omega)}$ satisfies $\partial I\subset i(\Omega)$.
\end{dfn}

\begin{rem}\label{rem.goodexistence} Following the classical construction of the dynamical realization of a countable left-ordered group (see \cite{GOD, Ghys}), it follows that any countable and totally ordered set $(\Omega,<)$ has a good embedding and that, given two different good embeddings $i_1,i_2\colon \Omega\to\R$, there exists $h\in\homeo_0(\R)$ so that $i_2=h\circ i_1$.  \end{rem}

\begin{dfn}\label{d-dynamical-realisation}
	Assume that $\psi\colon G\to \Aut(\Omega, <)$ is an order-preserving action. An action $\varphi\colon G\to \homeo_0(\R)$ is said to be a \emph{dynamical realization} of $\psi$ if there exists a good embedding $i\colon \Omega\to \R$ such that the following hold:
	\begin{enumerate}[label=(\roman*)]
		\item $i$ is $(\varphi,\psi)$-equivariant: $\varphi(g)(i(x))=i(\psi(g)(x))$ for all $g\in G$;
		\item \label{i-stab-trivial} for every connected component $I$ of the complement of $\overline{i(\Omega)}$, the stabilizer $\stab^\varphi(I)$  acts trivially on  $I$.
	\end{enumerate} 
\end{dfn}

\begin{lem}\label{lem.dynreal} Every order-preserving action $\psi\colon G\to\Aut(\Omega,<)$ on a countable totally ordered set admits a dynamical realization of $\psi$, unique up to positive conjugacy. 
\end{lem}

\begin{proof}[Sketch of proof] Consider a good embedding $i\colon \Omega\to\R$. By inducing the action $\psi$ through $i$ on $i(\Omega)$, and extending it by continuity to the closure, we obtain an action $\varphi_0\colon G\to\homeo_0(\overline{i(\Omega)})$. For every $g\in G$ denote by $\varphi(g)$ the extension of $\varphi_0(g)$ which is affine on every connected component $I$ of $\overline{i(\Omega)}$. It is direct to check that $g\mapsto \varphi(g)$ is a dynamical realization of $\psi$. 
	
	Now, consider two actions $\varphi_1,\varphi_2\colon G\to\homeo_0(\R)$, both  dynamical realizations of $\psi$, with associated good embeddings $i_1,i_2\colon\Omega\to\R$, respectively. By Remark \ref{rem.goodexistence}, there exists $h\in\homeo_0(\R)$ with $i_2=h\circ i_1$, and after conjugating $\varphi_1$ by $h$, we can suppose that $i_1=i_2=:i$. By equivariance, $\varphi_1$ and $\varphi_2$ must coincide on $\overline{i(\Omega)}$. Let $\mathcal I$ be the set of connected components of $\R\setminus \overline{i(\Omega)}$, and note that the $G$-actions on $\mathcal I$ induced by $\varphi_1$ and $\varphi_2$ coincide, so that the set of orbits $\mathcal I/G$ does not depend on the action. For $J\in \mathcal I$, we denote by $\alpha(J)\in \mathcal I/G$ its $G$-orbit. Pick a system of representatives $\{I_\alpha\}_{\alpha\in {\mathcal I/G}}$ of orbits.  For $J\in \mathcal I$, choose $g_J\in G$ such that $g_J(J)=I_{\alpha(J)}$, and for $i\in\{1, 2\}$, set $f_{i, J}=\varphi_i(g_J)\restriction_J$. After the assumption  \ref{i-stab-trivial} in Definition \ref{d-dynamical-realisation}, each $f_{i,J}$ is a homeomorphism from $J$ to $I_{\alpha(J)}$ that does not depend on the choice of $g_J$. Thus, $f_{2, J}^{-1}f_{1, J}$ is a self-homeomorphism of $J$. Define a map $q\colon \R\to \R$ which is the identity on $\overline{i(\Omega)}$ and satisfies $q\restriction_J=f_{2, J}^{-1}f_{1, J}$ for every $J\in \mathcal I$. Then one readily checks that $q$ conjugates $\varphi_2$ to $\varphi_1$. \qedhere
\end{proof}

We proceed to give a sufficient condition for minimality of dynamical realizations of actions on totally ordered sets.

\begin{lem}[Minimality criterion]\label{lem.minimalitycriteria} Let $\psi\colon G\to\Aut(\Omega,<)$ be an order-preserving action on a countable totally ordered set, and assume that for every four points $x_1,x_2,y_1,y_2$ of $\Omega$ with
	\[x_1< y_1< y_2< x_2,\]
	there exists $g\in G$ such that
	\[g.y_1< x_1< x_2< g.y_2.\]
	Then the dynamical realization of $\psi$ is minimal.
\end{lem}
\begin{proof} Let $\varphi\colon G\to\homeo_0(\R)$ be a dynamical realization of $\psi$ with  associated good embedding $i\colon \Omega\to\R$. We first claim that $i$ has dense image. Suppose by contradiction that it is not the case, and take  a connected component $(\eta_1,\eta_2)$ of the complement of $\overline{i(\Omega)}$. Since $i$ is a good embedding, we have that $\{\eta_1,\eta_2\}\subset i(\Omega)$. Choose two points $\xi_1$ and $\xi_2$ in $i(\Omega)$ such that $(\eta_1,\eta_2)\subsetneq (\xi_1,\xi_2)$;
	by our assumption, we can find an element $g\in G$ such that $g.(\xi_1,\xi_2)\subsetneq (\eta_1,\eta_2)$,
	contradicting $\varphi(g)$-invariance of $i(\Omega)$. This shows that $i$ has dense image. 
	
	Suppose again by contradiction that there exists a proper closed $\varphi$-invariant subset $\Lambda\subset\R$. Take  a connected component $(\eta_1,\eta_2)$ of $\R\setminus \Lambda$. By density of  $i(\Omega)$, we can find four points $\xi_1,\xi_2,\zeta_1,\zeta_2$ in $i(\Omega)$ such that
	\[\zeta_1<\eta_1<\xi_1<\xi_2<\eta_2<\zeta_2.\]
	After our assumption, we can find an element $g\in G$ such that $g.(\zeta_1,\zeta_2)\subsetneq (\xi_1,\xi_2) \subsetneq (\eta_1,\eta_2)$. This contradicts $\varphi(g)$-invariance of $\Lambda$, showing that $\varphi$ is minimal, as desired.
\end{proof}

Let us highlight a situation which allows to apply the previous criterion; this requires the following definition. 
\begin{dfn}\label{dfn.homotype}Let $(\Omega,<)$ be a totally ordered set. We say that a subgroup $H\subseteq\Aut(\Omega,<)$ is of \emph{$<$-homothetic type} if the following conditions are satisfied.
	\begin{itemize}
		\item There exists $x_0\in \mathsf{Fix}(H)$.
		\item For every $x\in \Omega\setminus \{x_0\}$, there exists a sequence of elements $(h_n)\subset H$ such that $h_n(x)\to+\infty$ if $x>x_0$, and $h_n(x)\to-\infty$ if $x<x_0$.
	\end{itemize}
\end{dfn}

\begin{prop}\label{p.minimalitycriteria}
	Let $\psi\colon G\to\Aut(\Omega,<)$ be an order-preserving action on a countable totally ordered set. If for every $x\in \Omega$ there exists a subgroup $H_x\subseteq G$ such that $\psi(H_x)$ is a subgroup of $<$-homothetic type fixing $x$, then the dynamical realization of $\psi$ is minimal. 
\end{prop}

\begin{proof}
	Consider four points $x_1,x_2,y_1,y_2$ of $\Omega$ with $x_1< y_1< y_2< x_2$. By assumption, we can find an element $h_1\in H_{y_1}$ such that $y_1<h_1.y_2<y_2$. Write now $x_\ast=h_1.y_2$, and take $h_\ast\in H_{x_\ast}$ such that $h_\ast.y_1<x_1<x_2<h_\ast.y_2$. Thus Lemma \ref{lem.minimalitycriteria} applies.
\end{proof}

\section{Bieri--Strebel groups of piecewise linear homemorphisms}\label{sc.BieriStrebel}

To illustrate our results we will often consider examples of locally moving groups arising as groups of piecewise linear (PL) homeomorphisms of the line. Let us briefly fix the terminology and recall a large systematic family of such groups, following Bieri and Strebel \cite{BieriStrebel}.

We say that a homeomorphism $f\colon X\to X$ of an interval $X\subseteq \R$ is \emph{piecewise linear}\index{homeomorphism!piecewise linear} (PL, for short) if there exists a discrete subset $\Sigma\subset X$ such that in restriction to $X\setminus \Sigma$, the map $f$ is locally affine, that is, of the form  $x\mapsto \lambda x+a$, with $\lambda>0$ and $a\in \R$. We denote by $\BP(f)$ the minimal subset $\Sigma$ satisfying such condition, and 
points of $\BP(f)$ will be the \emph{breakpoints} of $f$. When $\BP(f)$ is finite, we say that $f$ is \emph{finitary} PL.
Note that with this definition, a PL homeomorphism always preserves the orientation.

\begin{dfn}
	\label{d.BieriStrebel}
	Given a non-trivial multiplicative subgroup $\Lambda\subseteq \R_{>0}$, a non-trivial $\Z[\Lambda]$-submodule $A\subset \R$, and an interval $X\subseteq \R$, the \emph{Bieri--Strebel group $G(X;A,\Lambda)$}\index{group!Bieri--Strebel groups} is the group of finitary PL homeomorphisms $f\colon X\to X$ with the following properties:
	\begin{enumerate}[label=(\roman*)]
		\item breakpoints of $f$ are in $A$,
		\item in restriction to $X\setminus \BP(f)$, the map $f$ is locally of the form $\lambda x+a$, with $\lambda \in \Lambda$ and $a\in A$.
	\end{enumerate}
\end{dfn}

For example, Thompson's $F$ corresponds to $G((0,1);\Z[1/2], \langle 2\rangle_* )$\index{group!Thompson's group}. Other interesting examples are provided by the so-called Thompson--Brown--Stein groups, defined as follows.

\begin{dfn}\label{d-Thompson-Stein}
	Let $1 < n_1 < \dots < n_k$ be natural numbers such that the multiplicative subgroup $\Lambda=\langle n_i\rangle \subseteq \R_{>0}$  is an abelian group of rank $k$. Denote by $A$ the ring $\Z\left [1/m\right ]$, where $m$ is
	the least common multiple of the $n_i$. The group 
	\[
	F_{n_1,\ldots,n_k}:=G((0,1);A,\Lambda)
	\]
	is the corresponding \emph{Thompson--Brown--Stein group}\index{group!Thompson--Brown--Stein groups}.	
\end{dfn}

The group $F_2$ is simply Thompson's group $F$. For every $n\ge 2$, the group $F_n$ is isomorphic to a subgroup of $F$, and 
these groups were first considered by Brown in \cite{Brown}, inspired by the so-called Higman--Thompson groups. Later Stein \cite{Stein} started the investigation of the more general class of groups $F_{n_1,\ldots,n_k}$.
She proved that every Thompson--Brown--Stein group is finitely generated and even finitely presented. Moreover, given any $m$-adic open interval $I\subset (0,1)$ (that is, an interval with endpoints in $\Z[1/m]$), the subgroup $\left (F_{n_1,\ldots,n_k}\right )_I$ is isomorphic to $F_{n_1,\ldots,n_k}$.

We refer to the classical monograph by Bieri and Strebel \cite{BieriStrebel} for an extensive investigation of the groups $G(X;A,\Lambda)$.

\begin{rem}
	It would be interesting to see how the results of this text apply to other groups of piecewise projective homeomorphisms, such as Monod's groups $H(A)$ \cite{Monod} (we will not pursue this task).
\end{rem}

\part{Rigidity results for locally moving groups: $C^1$ actions}\label{partI}

This part contains our first results on locally moving groups, which are of rigidity nature. Namely, we provide various sufficient conditions that ensure that an irreducible action $\varphi\colon G\to\homeo_0(\R)$ of a locally moving subgroup $G\subseteq \homeo_0(\R)$ is semi-conjugate to the standard defining action of $G$. The culminating point of this part is Theorem \ref{t-intro-C1} (appearing below as Theorem \ref{t-lm-C1}), dealing with actions by diffeomorphisms of class $C^1$.

Chapter \ref{sec.locallymgeneral} contains the preliminary definitions and results for micro-supported and locally moving subgroups of $\homeo_0(\R)$: besides relatively standard results about topological dynamics (\S\ref{sc.defi_lm}) and normal subgroup structure (\S \ref{sec.normsgp}), we show a slightly refined version of the so-called ``2-chain lemma'' (first introduced by Brin in the PL setting \cite{Ubiquity}) to prove that any locally moving subgroup  of $\homeo_0(\R)$ contains an isomorphic copy of Thompson's group $F$ acting in the standard way, up to semi-conjugacy (Proposition \ref{p-chain}).

In Chapter \ref{s-lm-2}, given a locally moving subgroup $G\subseteq \homeo_0(\R)$, we begin the study of actions of $G$ on the line. We first show a fundamental criterion for semi-conjugacy to the standard action (Proposition \ref{p-lm-reconstruction}), formulated in terms of fixed points for the action of two families of subgroups dynamically defined from the standard action. Building on Proposition \ref{p-lm-reconstruction} and on a method to exploit commutation within $\homeo_0(\R)$, we  show that any action $\varphi\colon G\to \homeo_0(\R)$ that is not semi-conjugate to the standard one must satisfy various general constraints (Proposition \ref{p-lm-property-exotic}). 

Chapter \ref{s-differentiable} contains the main result of this part, corresponding to Theorem \ref{t-intro-C1} from the introduction: any differentiable irreducible action on the line of a locally moving subgroup $G\subseteq\homeo_{0}(\R)$ either is semi-conjugate to the standard action, or comes from an action of a proper quotient. The proof goes through a reduction to   a simpler problem on differentiable actions of Thompson's $F$, based on the results from the previous chapters. More precisely, Proposition \ref{p-chain} combined with Proposition \ref{p-lm-property-exotic} allow to show that the image of every exotic action $\varphi\colon G\to \homeo_0(\R)$ must contain embedded copies of $F$ in $G$ with a large centralizer.  Assuming by contradiction that $\varphi$ is an action by diffeomorphisms,  such copies of $F$ cannot contain elements having hyperbolic fixed points. Consequently, it is enough to understand differentiable actions of $F$ on the line with no element having hyperbolic fixed points:  we show that they actually have abelian image (Proposition \ref{p-Conrad-C1-global}).

At this point, it is natural to wonder whether the global rigidity highlighted in Theorem \ref{t-intro-C1}  remains true for actions by homeomorphisms of class $C^0$.  In Chapter \ref{ss.exoticactions} we show that this is not the case, by providing some first examples of  exotic actions $\varphi\colon G\to \homeo_0(\R)$  of locally moving subgroups $G\subset \homeo_0(\R)$ (recall that an irreducible action $\varphi\colon G\to \homeo_0(\R)$ is exotic if it is not semi-conjugate to the standard action, nor to any action of any proper quotient). In particular we provide two different classes of examples of faithful minimal  exotic actions of Thompson's group $F$.  These examples anticipate a through analysis of exotic actions in Part \ref{partII}, where structure theorems describing their dynamics in much more details will be proved.

The last chapter contains some additional results,  that complement the main results of this part and are proven using the same methods. In \S \ref{s-uncountable} we provide soft conditions on a large (necessarily uncountable) locally moving subgroup $G\subseteq \homeo_0(\R)$, which imply that $G$ has a unique irreducible action $\varphi\colon G\to \homeo_0(\R)$ up to conjugacy. Particular cases of groups satisfying this criterion are the groups of all compactly supported homeomorphisms and $C^r$ diffeomorphisms of the line (for $r\neq 2$), recovering results of  Militon \cite{Militon} and  Chen and Mann \cite{ChenMann}. 
The results shown in \S \ref{s-circle} basically say that there are no exotic actions on the circle of locally moving groups of homeomorphisms of $\homeo_0(\mathbb{S}^1)$ (and of groups acting on more general locally compact Hausdorff spaces). Finally, in \S \ref{ssc:Stein} we show some non-smoothability results. For some locally moving groups of PL homeomorphisms, we show that the standard action cannot actually be semi-conjugate to any $C^r$ action (for any $r>1$, or even for $r=1$ for some class of PL groups). Combined with  Theorem \ref{t-intro-C1}, this shows that such groups cannot admit any faithful actions on intervals by $C^r$ diffeomorphisms. 

\chapter{Micro-supported and locally moving groups}\label{sec.locallymgeneral} \label{s-lm-1}

\section{Definitions}\label{sc.defi_lm}
Throughout this chapter (and mostly in the rest of the paper), we let $X=(a, b)$ be an open interval, with endpoints $a, b\in \R\cup\{\pm \infty\}$. Recall that for a subgroup $G\subseteq \homeo_0(X)$ and a subinterval $I\subset X$, we write
\[G_I=\{g\in G: g(x)=x\text{ for every }x\notin I\}\]
for the subgroup of $G$ consisting of elements that fix pointwise the complement $X\setminus I$. 

\begin{dfn}\label{d.lm}
A subgroup $G\subseteq \homeo_0(X)$ is \emph{micro-supported}\index{group!micro-supported} if for every non-empty subinterval $I\subset X$, the subgroup $G_I$ is non-trivial. We also say that $G$ is \emph{locally moving}\index{group!locally moving} if for every open subinterval $I\subset X$, the subgroup $G_I$ acts on $I$ without fixed points. 
\end{dfn}

Recall that given subsets $I$ and $J$ in $X$, we write $I\Subset J$ if $I$ is relatively compact in $J$. 
For $G\subseteq \homeo_0(X)$, we denote by $G_c$ the normal subgroup of elements with relatively compact support in $X$, that is, $G_c=\bigcup_{I\Subset X} G_I$. We also let $\Germ(G, a)$ and $\Germ(G, b)$ be the \emph{groups of germs}\index{group!group of germs} of elements of $G$ at the endpoints of $X$. Recall that the germ of an element $g\in G$ at $a$ is  the equivalence class of $g$ under the equivalence relation that identifies two elements $g_1, g_2\in G$ if they coincide on some interval of the form $(a, x)$, with $x\in X$. The germ of $g$ at $b$ is defined similarly.  We denote by $\Gcal_a\colon G\to \Germ(G, a)$ and $\Gcal_b\colon G\to \Germ(G, b)$ the two natural germ homomorphisms\index{germ homomorphism}, and their kernels will be denoted by $G_-$ and $ G_+$, respectively.  Note that we have
\[G_-=\textstyle\bigcup_{x\in X} G_{(x, b)} \quad\text{and} \quad G_+=\textstyle\bigcup_{x\in X} G_{(a, x)}.\]
We also introduce the following normal subgroup, which plays an important role in some of our main results.
{\begin{dfn} \label{d-fragmentable-subgroup}
For $X=(a,b)$ and $G\subset \homeo_0(X)$, the \emph{fragmentable subgroup} of $G$ is  defined as $G_\mathsf{frag}:=G_-G_+$. When $G=\Gfrag$, we say that $G$ is \emph{fragmentable}. 
\end{dfn}}

Note that since $G_\pm$ are normal subgroups, so is $\Gfrag$. In practice, $\Gfrag$ consists of all elements $g$ that can be written as a product $g=g_1g_2$, where each $g_i$ is supported in a strict subinterval $I_i\subsetneq X$ (that can be chosen to share one endpoint with $X$).  This definition coincides (for $X=\R$) with Definition \ref{d-intro-fragmentable-subgroup} from the introduction. 

\begin{rem}\label{r-fragmentable-germs}
	A subgroup  $G\subseteq\homeo_{0}(X)$ is fragmentable if and only if, for every pair of germs $\gamma_1\in \Germ(G, a)$ and $\gamma_2\in \Germ(G, b)$, there exists $g\in G$ such that $\Gcal_a(g)=\gamma_1$ and $\Gcal_b(g)=\gamma_2$. Note that this is equivalent to the condition that the natural injective homomorphism
	\[G/G_c\rightarrow \Germ(G, a)\times \Germ(G, b)\]
	be an isomorphism. 
	
	\end{rem}

The next result says that when $G$ acts minimally, the micro-supported condition is equivalent to the non-triviality of the subgroup $G_c$.
\begin{prop} \label{p-micro-compact}
For $X=(a,b)$, let $G\subseteq \homeo_0(X)$ be a subgroup acting minimally on $X$. Then $G$ is micro-supported if and only if it contains a non-trivial element with relatively compact support. {Moreover, when this holds, the action of $G$ on $X$ is proximal.} \end{prop}
\begin{proof}
The forward implication is obvious. Conversely, assume that  there exists a relatively compact subinterval $I \Subset X$ for which $G_I\neq \{\id\}$. The centralizer of  $G_I$ in $\homeo_0(X)$ must fix the infimum of the support of every non-trivial element of $G_I$, so its action on $X$ is not free. Then by Theorem \ref{t-centralizer} the action of $G$ on $X$ is proximal. Therefore, for every non-empty open subinterval $J\subset X$, there exists $g\in G$ such that $g(I)\subset J$, so the subgroup $G_J$ contains $G_{g(I)}=gG_Ig^{-1}$ and is non-trivial.  
\end{proof}

Let us summarize some basic observations on the locally moving condition in the following lemma.

\begin{lem} \label{l-rigid}
For $X=(a,b)$, let $G\subseteq \homeo_0(X)$ be locally moving. Then the following hold for every non-empty open subinterval $I\subset X$.  
\begin{enumerate}[label=(\roman*)]
 \item \label{i-rigid-minimal} The subgroup $G_I$ acts minimally on $I$. In particular, $G$ acts minimally on $X$. 
 
 \item \label{i-derived} The derived subgroup $[G_I,G_I]$ also acts without fixed points on $I$. 
 
 \end{enumerate}
\end{lem}

\begin{proof}
Write $I=(c,d)\subseteq X$ for a non-empty open subinterval. Fix $x, y\in I$ and assume, say,  that $x<y$. Since the subgroup $G_{(c, y)} \subseteq G_I$ has no fixed point in $(c, y)$, there exist elements $g\in G_{(c, y)}$ such that $g(x)$ is arbitrarily close to $y$. Thus the $G_I$-orbit of $x$ accumulates at $y$. By a symmetric argument, the same holds if $y<x$. Since $x$ and $y$ are arbitrary, this shows that every $G_I$-orbit in $I$ is dense in $I$. Finally if $[G_I,G_I]$  admits fixed points in $I$, its set of fixed points is closed and $G_I$-invariant; by minimality of the action of $G_I$ on $I$, we deduce that $[G_I,G_I]$ is trivial, and so $G_I$ is abelian (conjugate on $I$ to a group of translations). This is not possible, since the action of $G_I$ on $I$ is micro-supported. \qedhere
\end{proof}

\section{Normal subgroup structure}\label{sec.normsgp} The following proposition shows that locally moving groups are nearly simple. This follows from well-known arguments, that we repeat here for completeness. 

\begin{prop}[Structure of normal subgroups] \label{p-micro-normal}
For $X=(a,b)$, let $G\subseteq \homeo_0(X)$ be a micro-supported subgroup whose action is minimal.
Then every non-trivial normal subgroup of $G$ contains $[G_c,G_c]$. Moreover, if $[G_c,G_c]$ acts minimally, then it is simple. 

 In particular, when $G$ is locally moving, the subgroup $[G_c, G_c]$ is simple and contained in every non-trivial normal subgroup of $G$. 
\end{prop}
The proof uses the following classical observation on normal subgroups of homeomorphisms, sometimes known as the ``double-commutator lemma''. With this formulation it appears in Nekrashevych \cite[Lemma 4.1]{Nek-fp}
\begin{lem}\label{l.doublecomm}
Let $H$ be a group of homeomorphisms of a Hausdorff space $Z$, and $N$  a non-trivial group of homeomorphisms of $Z$ normalized by $H$. Then, there exists a non-empty open subset $U\subset Z$ such that $N$ contains $[H_U, H_U]$, where $H_U$ is the subgroup of elements fixing pointwise the complement $Z\setminus U$.
\end{lem}
\begin{proof}[Proof of Proposition \ref{p-micro-normal}]
Suppose that $N\trianglelefteq G$ is a non-trivial normal subgroup. By Lemma \ref{l.doublecomm}, $N$ contains $[G_I, G_I]$ for some non-empty open subinterval $I\subset X$. Given an element $g\in [G_c, G_c]$, we can find an open subinterval $J\Subset X$ such that $g\in [G_J, G_J]$.
By proximality of the action of $G$ (Proposition \ref{p-micro-compact}),  we can find $h\in G$ such that $h(J)\subset I$, and consequently $hgh^{-1}\in [G_I, G_I]\subseteq N$. By normality of $N$, this gives $g\in N$. Since $g$ is arbitrary, we have $[G_c, G_c]\subseteq N$. 
Note that this implies in particular that $[G_c, G_c]$ is perfect, since its commutator subgroup is normal in $G$ and thus must coincide with $[G_c, G_c]$.

Assume now that $[G_c, G_c]$ acts minimally on $X$. Then it is micro-supported (Proposition \ref{p-micro-compact}), and we can apply the first part of the proof:  every non-trivial normal subgroup of $[G_c, G_c]$ must contain the derived subgroup of $[G_c, G_c]$.  Since $[G_c, G_c]$ is perfect, this implies that it is simple.
\end{proof}

It follows that when   $G\subseteq \homeo_0(X)$ is  micro-supported and acts minimally (in particular, when $G$ is locally moving), the quotient  $\overline{G}:=G/[G_c, G_c]$ is the largest proper quotient of $G$. The largest quotient $\overline{G}$ is an extension of $G/G_c$ with abelian kernel:
\[1\to G_c/[G_c, G_c] \longrightarrow \overline{G}\longrightarrow G/G_c\longrightarrow 1.\]
Whereas the (\textit{a priori} smaller) quotient $G/G_c$ has a dynamical interpretation, in the sense that it is naturally a subgroup of the product of groups of germs $\Germ(G, a)\times \Germ(G, b)$,
the abelian group $G_c/[G_c, G_c]$  can be difficult to identify in general. However, for some relevant examples of locally moving groups it is known that the group $G_c$ is perfect, so that $\overline{G}=G/G_c$. 

\begin{ex}\label{e-BN}
Consider the \emph{Brin--Navas group}\index{group!Brin--Navas group} $B$, which was studied independently by Brin \cite{Brin} and Navas \cite{NavasAmenable}, and further studied by Bleak \cite{Bleak} who showed that $B$ is contained in any non-solvable subgroup of $\PL([0,1])$.\footnote{In fact, the group $B$ was first considered by Hector \cite{Hector} as a group of diffeomorphisms of the interval, although the action defined by Hector is cyclic, and thus not semi-conjugate to the actions considered by Brin and Navas.}
	The group $B$ has the following presentation (see Bleak, Brough, and Hermiller \cite{bleak2016determining}, and Proposition \ref{lem.presentation}):
	\begin{multline}B= \langle f,w_n\, (n\in \Z)\, \vert\, fw_nf^{-1}=w_{n+1}\,\forall n\in \Z,\,[w_i,w_n^mw_jw_n^{-m}]=1\\
		\forall n>i,n>j,\,\forall m\in \Z\setminus \{0\} \rangle.
	\end{multline}
That is, the group $B$ is defined as an HNN extension of the group generated by the $w_n$ ($n\in \Z$), and the latter is a bi-infinitely iterated wreath product of $\Z$. Following the notation in \cite{Bleak}, we write $(\,\wr\, \Z\,\wr\,)^\infty$ for the subgroup generated by the $w_n$ in $B$, so that $B=(\,\wr\, \Z\,\wr\,)^\infty\rtimes \Z$.  A minimal micro-supported action on $(0,1)$ of $B$ is realized in the group $\PL([0,1])$ of piecewise linear homeomorphisms, choosing generators (see \cite{Bleak} and Figure~\ref{f-BN})
\[
f(x)=\left\{\begin{array}{lr}
\tfrac14 x & x\in [0,\tfrac14],\\[.5em]
x-\tfrac3{16} & x\in [\tfrac14,\tfrac7{16}],\\[.5em]
4x-\tfrac32 & x\in [\tfrac7{16},\tfrac{9}{16}],\\[.5em]
x+\tfrac3{16} & x\in [\tfrac9{16},\tfrac34],\\[.5em]
\tfrac14 x+\tfrac{3}4 & x\in [\tfrac34,1],
\end{array}\right.
\quad\quad
w_0(x)=\left\{\begin{array}{lr}
x & x\in[0,\tfrac7{16}],\\[.5em]
2x-\tfrac{7}{16} & x\in[\tfrac7{16},\tfrac{15}{32}],\\[.5em]
x+\tfrac1{32} & x\in[\tfrac{15}{32},\tfrac12],\\[.5em]
\tfrac12 x+\tfrac9{32} & x\in[\tfrac12,\tfrac9{16}],\\[.5em]
x & x\in[\tfrac9{16},1].
\end{array}\right.
\]
In this case, the subgroup $B_c$ is the normal subgroup $(\,\wr\, \Z\,\wr\,)^\infty$, and $B_c/[B_c,B_c]\cong \Z^\infty$, so that the largest proper quotient $\overline{B}\cong \Z^\infty\rtimes \Z=\Z\wr\Z$ is the lamplighter group. 
Observe that the bi-infinite wreath product $B_c$ does not act minimally on $(0,1)$. 
\begin{figure}[ht]
	\centering
	\includegraphics[scale=1]{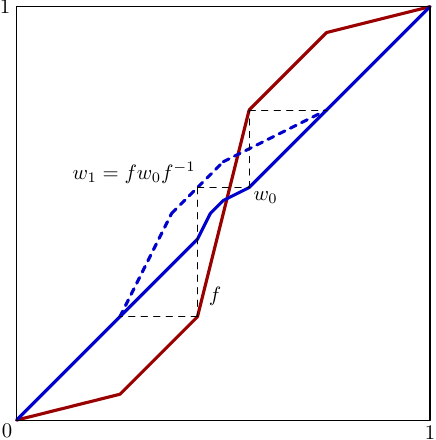}
	\caption{PL realization of the Brin--Navas group with minimal micro-supported action.}\label{f-BN}
\end{figure}
\end{ex}

\begin{ex}\label{ex:Brown}\index{group!Bieri--Strebel groups}
	A rich source of examples of micro-supported (actually locally moving) groups, are the Bieri--Strebel groups introduced in Definition \ref{d.BieriStrebel}. Recall that these are defined as the groups of finitary piecewise linear homeomorphisms of an interval with prescribed arithmetic conditions on breakpoints and slopes. For a quite simple class of examples, fix $n\ge 2$ and consider the group $G=F_n$ of all finitary PL homeomorphisms of $(0,1)$ such that all derivatives are powers of $n$, and the breakpoints are in the ring $A=\Z[1/n]$  (e.g.\ recall that when $n=2$, this is Thompson's group $F$). The subgroup $G_c$ of compactly supported elements consists exactly of elements with derivative $1$ at the endpoints of $[0,1]$, so that $G/G_c\cong \Z^2$. However, the abelianization of $G$ is larger for $n\ge 3$, and more precisely one has $G^{ab}\cong \Z^n$. One can check this directly from the presentation given in \cite[D15.10]{BieriStrebel} (cf.\ also the proof of Lemma \ref{l-BBS-algebraic}, where we discuss the abelianization of different Bieri--Strebel groups).
%	\red{In fact, this is determined by the ``homomorphism into the slope group'' $\nu$ discussed by Bieri and Strebel in \cite[\S C11]{BieriStrebel}. The homomorphism $\nu$ is obtained by gathering together the homomorphism $\nu_1\colon G\to \Z$ given by 
%	\[
%	\nu_1(f)=\log_nD^-f(1),\quad f\in G,
%	\] 
%	and the homomorphisms $\nu_{\Omega}\colon G\to \Z$, where $\Omega$ is a $G$-orbit of a breakpoint of some element of $G$ (there are exactly $n-1$ many of them, as the orbits of the action of $G$ on $A$ can be identified with the cosets $A/(n-1)A$, see \cite[Corollary A5.1 and Illustration A4.3]{BieriStrebel}), which are defined by considering the total jump of derivatives at points in the orbit:
%	\[
%	\nu_{\Omega}(f)=\log_n\prod_{a\in \Omega}\frac{D^+f(a)}{D^-f(a)},\quad f\in G.
%	\]
%	Explicitly, we define $\nu=\nu_1\times \prod_{\Omega\in A/(n-1)A}\nu_\Omega$.}
\end{ex}

\section{Subgroups isomorphic to Thompson's group}
Thompson's group $F$ plays a special role among locally moving groups, due to the following result.

\begin{prop}\label{p-chain}
For $X=(a, b)$, any locally moving subgroup $G\subseteq \homeo_0(X)$  contains a subgroup isomorphic to Thompson's group $F$. 
\end{prop}
This fact  will be crucial  in our proof of the $C^1$ rigidity of locally moving groups (Theorem \ref{t-lm-C1}). Its proof is based on the following variant of the ``2-chain lemma'' of Kim,  Koberda, and Lodha \cite{KKL}. The key idea can be traced back to Brin \cite{Ubiquity}, and has been also largely developed by Bleak \textit{et al.}~in \cite{FastPingPong}.
It is based on the following two properties: on the one hand, every non-trivial quotient of $F$ is abelian; on the other hand, $F$ admits the following finite presentation,
\begin{equation}\label{eq.presF}
F=\left \langle a,b\,\middle\vert\, [a,(ba)b(ba)^{-1}]=[a,(ba)^2b(ba)^{-2}]=1\right \rangle,
\end{equation}
and the two relations have in fact a meaningful dynamical interpretation.

\begin{lem}[Noisy 2-chain lemma] \label{l-2chain}\index{lemma!noisy 2-chain lemma}
	Take two homeomorphisms $f,g\in \homeo_0(\R)$, write $d=\sup \supp(f)$ and $c=\inf \supp(g)$, and assume the following:
	\begin{enumerate}[label=(\roman*)]
		\item\label{i:noisy2c} $c<d$;
		\item\label{ii:noisy2c} $c\notin \fix(f)$ and $d\notin\fix(g)$;
		\item\label{iii:noisy2c} $d$ and $f(c)$ are in the same connected component of $\supp(g)$.
	\end{enumerate}
	Then $\langle f,g\rangle$ contains a subgroup isomorphic to Thompson's $F$
\end{lem}

\begin{figure}[ht]
	\centering
	\includegraphics[scale=1]{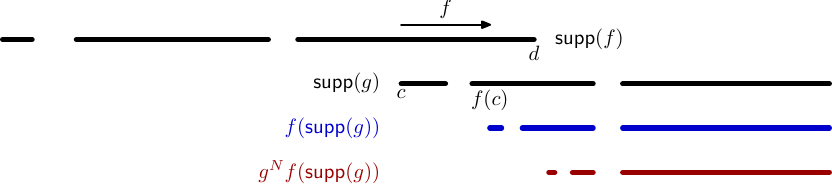}
	\caption{Proof of the noisy 2-chain lemma (Lemma \ref{l-2chain}).}\label{f.noisy2chain}
\end{figure}

\begin{proof}
	After the assumptions, there exists $N\in \Z$ such that $g^Nf(c)>d$. Thus we also have $(g^Nf)^2(c)=g^{2N}f(c)>d$. Hence, for $i\in\{1,2\}$, we have that the subset
	\[(g^Nf)^i\left (\supp(g)\right )=\supp\left ( (g^Nf)^ig(g^Nf)^{-i}\right )=\supp\left ( (g^Nf)^ig^N(g^Nf)^{-i}\right )\]
	is disjoint from $\supp(f)$ (see Figure \ref{f.noisy2chain}). We deduce that the elements $a=f$ and $b=g^N$ satisfy the two relations in the presentation \eqref{eq.presF}, and thus $\langle f,g^N\rangle$ is isomorphic to a quotient of $F$. As the supports $\supp(f)$ and $\supp(g^N)=\supp(g)$ overlap, we deduce that the subgroup $\langle f,g^N\rangle$ is non-abelian and thus isomorphic to $F$.
\end{proof}

\begin{proof}[Proof of Proposition \ref{p-chain}] 
Take $f\in G_c$. Using Lemma \ref{l-rigid}, it is not difficult to find an element $g$ conjugate to $f$, such that conditions \ref{i:noisy2c}--\ref{ii:noisy2c} in Lemma \ref{l-2chain} are satisfied by $f$ and $g$. Write $c=\inf \supp(g)$ and $d=\sup \supp(f)$. Up to replace $f$ by its inverse, we can assume $f(c)>c$. If condition \ref{iii:noisy2c} in Lemma \ref{l-2chain} is not satisfied, we use Lemma \ref{l-rigid} again to find an element $h\in G_{(c, d)}$ such that $h(f(c))$ belongs to the same connected component of $\supp(g)$ as  $d$.
Replace $f$ by the conjugate $hfh^{-1}$ and now property \ref{iii:noisy2c} is also satisfied.
\qedhere

\end{proof}

\chapter{First results for actions of locally moving groups on the line} \label{s-lm-2}

Let $X=(a, b)$ be an open interval, and $G \subseteq \homeo_0(X)$ a locally moving subgroup. In this chapter we begin to study the possible actions of $G$ on the line. 
We derive some general restrictions on such actions, and provide sufficient conditions under which an action  $\varphi\colon G\to \homeo_0(\R)$ must be semi-conjugate to the  standard action on $X$, or to an action induced from the largest quotient $G/[G_c, G_c]$.  In the next chapter, we will use these results in the proof of our main theorem on differentiable actions. General actions by homeomorphisms will be analyzed in much greater detail in Part \ref{partII}. 

In the sequel we will often be dealing with two different actions of the same group $G$, namely its standard action on $X$ and another action $\varphi\colon G\to\homeo_0(\R)$ on the real line. We will use the following terminology.

\begin{dfn} \label{d-exotic-action}
	For $X=(a,b)$, let $G\subseteq \homeo_{0}(X)$ be a subgroup whose action on $X$ is irreducible, and let $\varphi\colon G\to \homeo_{0}(\R)$ be another irreducible action. We say that $\varphi$ is \emph{exotic}\index{action!exotic} if it is not semi-conjugate neither to the standard action on $X$, nor to any action factoring through a proper quotient.
\end{dfn}

Recall also that to avoid confusion, we fix the following notation throughout the paper.

\begin{notation}
	Let $G\subseteq \homeo_0(X)$ be a subgroup and $\varphi\colon G\to \homeo_0(\R)$ an action on the real line. For $g\in G$ and $x\in X$ we use the notation $g(x)$ to refer to the standard action on $X$, while for $\xi\in \R$ we will write $g.\xi:=\varphi(g)(\xi)$. We also write $\fixphi(H)$ for the set of fixed points of a subgroup $H\subseteq G$ with respect to the action $\varphi$, and $\suppphi(H)=\R\setminus \fixphi(H)$.
	
	We will often write $N=[G_c,G_c]$, although this will be systematically recalled.
\end{notation}

\section{Constructing a semi-conjugacy to the standard action}
The following result will be our main criterion to show that an action $\varphi\colon G\to \homeo_0(\R)$ cannot be exotic. For the proof, the locally moving assumption on $G$ is crucial. 

\begin{prop} \label{p-lm-reconstruction}
For $X=(a, b)$, let $G\subseteq \homeo_0(X)$ be locally moving, and write $N=[G_c, G_c]$. 
Let $\varphi\colon G\to \homeo_0(\R)$ be an irreducible action for which  there exist $x,y\in X$ such that both images $\varphi(N_{(a, x)})$ and $\varphi(N_{(y, b)})$  admit fixed points. Then  $\varphi$ is not exotic: it is semi-conjugate either to the standard action on $X$, or to an action factoring through the quotient $G/N$.
\end{prop}

\begin{proof}
First of all, observe that if $\varphi(N)$ admits fixed points, then $\fix^\varphi(N)$ is a closed $\varphi(G)$-invariant subset of $\R$ which accumulates at both $\pm \infty$, and the action on it factors through $G/N$. In this case we deduce that $\varphi$ is semi-conjugate to an action of $G/N$.  Thus we will assume that $\varphi(N)$ has no fixed point, and 
show that in this case the action must be semi-conjugate to the standard action of $G$ on $X$.

Note that since there exists $x$ such that $\varphi(N_{(a, x)})$ has  fixed points, this is actually true for every $x\in X$, for a subgroup $N_{(a, x)}$ is always conjugate into every other $N_{(a, y)}$, where $y\in X$. The same holds true for the images of the form $\varphi(N_{(x, b)})$.  In particular, for every $x\in X$, both images $\varphi(N_{(a,  x)})$ and $\varphi(N_{(x, b)})$ admit fixed points, and since they commute, they admit common fixed points. Thus for every $x\in X$, the $\varphi$-image of the subgroup $H_x:=\langle N_{(a, x)}, N_{(x, b)}\rangle$ admits fixed points. The idea is to construct a semi-conjugacy $q\colon X\to \R$ by setting $q(x)=\inf \fix^\varphi(H_x)$. We need to check that such a map $q$ is well-defined (i.e.\ that the subsets $\fix^\varphi(H_x)$ are bounded), and monotone. Let us first prove the following claim. 

\begin{claimnum} \label{claim-bounded}
Let $x$ and $y$ be two distinct points of $X$. Then we either have 
\[\sup \fix^\varphi(H_x)< \inf \fix^\varphi(H_y),\]
or viceversa. 
\end{claimnum}
\begin{proof}[Proof of claim]
If the conclusion does not hold, then upon exchanging the roles of $x$ and $y$ if needed,  we can find two distinct points $\xi, \xi' \in \fix^\varphi(H_x)$, and $\eta\in \fix^\varphi(H_y)$, such that $\xi \le \eta \le \xi'$. Assume first that $x<y$, and take $g\in N$. We are going to show that in this case $g.\xi\le\xi'$, contradicting our assumption that $\varphi(N)$ has no fixed points (as $g\in N$ is arbitrary). Let $z>x$ be a point of $X$ such that $g\in N_{(a, z)}$. By Lemma \ref{l-rigid}, the action of $N_{(x, b)}$ on $(x, b)$ has no fixed point, so we can find $h\in N_{(x, b)}$ such that $h(z)\in (x, y)$. Note that $h\in H_x$, so that $\varphi(h)$ fixes both $\xi$ and $\xi'$. On the other hand, the element $k=hgh^{-1}$ belongs to $N_{(a, y)}\subset H_y$, and therefore $\varphi(k)$ fixes  $\eta$. Thus, writing $g=h^{-1}kh$, we have
\[g.\xi=h^{-1}k.\xi \le h^{-1}k.\eta=h^{-1}.\eta \le h^{-1}.\xi'=\xi',\]
proving the claim in the case $x<y$. The case $y<x$ is treated analogously.
\end{proof}

After Claim \ref{claim-bounded}, the map $q\colon x\mapsto\inf\fixphi(H_x)$ is well-defined and injective. We next have to verify that it is monotone. 
\begin{claimnum}
Given points $x_1< x_2 < x_3$ of $X$, we either have $q(x_1) <q(x_2)< q(x_3)$, or $q(x_1)>q(x_2)>q(x_3)$.  
\end{claimnum}  
\begin{proof}[Proof of claim]
The arguments are similar to those given for the proof of the previous claim.
For $i\in\{1,2,3\}$, set $\xi_i=q(x_i)$ and note that, by the previous claim, these three images are pairwise distinct. 
We divide the proof into cases according to their relative position.  We will detail the case $\xi_1<\xi_2$ (the case $\xi_1>\xi_2$ being totally analogous); for this, we will assume for contradiction that $\xi_3<\xi_2$ and
we further split into two subcases depending on the relative position of $\xi_3$. 
\begin{case} We have $\xi_1 <\xi_3 <\xi_2$. \end{case}
In this case, we choose an element $g\in N$ such that $g.\xi_1 > \xi_2$. Let $y>x_2$ be a point of $X$ such that $g\in N_{(a, y)}$. Let $h\in N_{(x_2, b)}$ be such that $h(y)\in (x_2, x_3)$. Note that $h\in H_{x_2}$, and since $(x_2, b)\subset (x_1, b)$, we also have $h\in H_{x_1}$, so that $\varphi(h)$ fixes $\xi_2$ and $\xi_1$.   On the other hand, the element $k=hgh^{-1}$ belongs to $N_{(a, h(y))} $; since $(a, h(y))\subset (a, x_3)$, we have $k\in H_{x_3}$, and thus $\varphi(k)$ fixes $\xi_3$. Writing $g=h^{-1}kh$, we have
\[g.\xi_1=h^{-1}k.\xi_1 < h^{-1}k.\xi_3=h^{-1}.\xi_3 < h^{-1}.\xi_2=\xi_2,\]
contradicting the choice of $g$ such that $g.\xi_1>\xi_2$.

\begin{case}   We have $\xi_3< \xi_1<\xi_2$.   \end{case}
 In this case, choose an element $g\in N$ such that $g.\xi_2<\xi_3$. Let $y<x_2$ be a point of $X$ such that $g\in N_{(y, b)}$. Let $h\in N_{(a, x_2)}$ be such that $h(y)\in (x_1, x_2)$. 
Note that $h\in H_{x_2}$, and since $(a, x_2)\subset (a,x_3)$, we also have $h\in H_{x_3}$, so that $\varphi(h)$ fixes $\xi_2$ and $\xi_3$.   On the other hand, the element $k=hgh^{-1}$ belongs to $N_{(h(y),b)} $; since $(h(y),b)\subset (x_1,b)$, we have $k\in H_{x_1}$, and thus $\varphi(k)$ fixes $\xi_1$. Writing $g=h^{-1}kh$, we have
\[g.\xi_2=h^{-1}k.\xi_2 > h^{-1}k.\xi_1= h^{-1}.\xi_1 > h^{-1}.\xi_3=\xi_3,\]
contradicting the choice of $g$ such that  $g.\xi_2<\xi_3$.

Thus, the unique possibility is that $\xi_1<\xi_2<\xi_3$, as desired.
\end{proof}

The claim implies that the map $q\colon X\to \R$ is monotone (increasing or decreasing). Moreover, it is clearly equivariant by construction:
\[
g.q(x)=g.\inf \fix^\varphi(H_x) =\inf g. \fix^\varphi(H_x)=\inf \fix^\varphi(H_{g(x)})=q(g(x)) .
\]
Hence, the map $q$ establishes a semi-conjugacy between $\varphi$ and the standard action of $G$ on $X$.
\end{proof}
Recall from the introduction that any two locally moving actions on the line of the same group $G$ are conjugate (this is customarily deduced from the much more general results of Rubin \cite{Rubin,Rubin2}). From Proposition \ref{p-lm-reconstruction}, we recover this fact in the following  more general form. 
\begin{cor} \label{cor.unique}\index{group!locally moving!Rubin's theorem}
	For $X=(a,b)$, let $G\subseteq \homeo_{0}(\R)$ be locally moving. Let  $\varphi\colon G\to \homeo_{0}(\R)$ be a faithful irreducible action, and suppose that there exists a non-trivial element $g\in G$ such that $\suppphi(g)$ is contained in a half-line. Then $\varphi$ is semi-conjugate to the standard action of $G$ on $X$. In particular, if in addition $\varphi$ is minimal, then it is conjugate to the standard action on $X$. 
\end{cor}
\begin{proof}
Set again $N=[G_c, G_c]$. Assume that the element $g\in G$ from the statement is such that $\supp^\varphi(g)$ is bounded above (the case where it is bounded below is analogous). Then,   the  germ homomorphism $\Gcal_{+\infty}\circ \varphi \colon G\to \Germ(\varphi(G), +\infty)$ is not injective, so its kernel contains $N$ (Proposition \ref{p-micro-normal}). Thus, for any $h\in N$, we have that $\supp^\varphi(h)$ is also bounded above, and the point $\xi=\sup \supp^\varphi(h)$ must be fixed by the centralizer of $h$. Since the centralizer of $h$ contains  $N_{(a, x)}$ and $N_{(y, b)}$ for suitable $x, y\in X$, Proposition \ref{p-lm-reconstruction} gives that $\varphi$ must be semi-conjugate to the standard action on $X$, or to an action of $G/N$. 
Suppose by contradiction that the second case holds, i.e.\ that $\fixphi(N)\neq \varnothing$. In this case, since both $\fixphi(N)$ and $\suppphi(N)$ are non-empty $\varphi$-invariant sets, they both accumulate on $\pm\infty$. Then every connected component of $\suppphi(N)$ is bounded, and $N$ acts faithfully on each of them, for it is simple.  This contradicts that $\suppphi(h)$ is bounded above for every $h\in N$. We conclude that $\varphi$ is semi-conjugate to the standard action of $G$ on $X$. \qedhere 
\end{proof}

\section{Abundance of subgroups with fixed points in exotic actions}

Here we use the criterion given by Proposition \ref{p-lm-reconstruction} to get some first restrictions on the dynamics of exotic actions.
We begin with a useful observation  based on  Theorems \ref{t-Margulis} and \ref{t-centralizer}. 
\begin{prop}[Actions of direct products] \label{p-centralizer-fix}
Let $M\in\{\R,\mathbb{S}^1\}$. 
Let $\Gamma_1$ and $\Gamma_2$ be two groups, and if $M=\R$ assume that $\Gamma_1$ is finitely generated. Then for every action $\varphi\colon \Gamma_1\times \Gamma_2\to \homeo_0(M)$, there exists $i\in\{1,2\}$ such that the image of $[\Gamma_i, \Gamma_i]$ has fixed points.\footnote{We identify here $\Gamma_1$ with the subgroup $\Gamma_1\times \{1\}$, and similarly for $\Gamma_2$.}  
\end{prop}

\begin{proof}
Suppose that $\varphi(\Gamma_1)$ admits no fixed point (otherwise, the conclusion is true). If the action of $\Gamma_1$  admits a discrete orbit, then the action on it factors through a cyclic quotient of $\Gamma_1$, and  consequently  $[\Gamma_1, \Gamma_1]$ fixes it pointwise; so in this case we are done. Otherwise, we can take the unique minimal invariant set $\Lambda\subset M$ for $\varphi(\Gamma_1)$, whose existence is ensured by the assumption that $\Gamma_1$ be finitely generated. By uniqueness, $\Lambda$ is preserved by $\varphi(\Gamma_2)$, and the action of $\Gamma_1$ is semi-conjugate to a minimal action, obtained by collapsing the connected components of $M\setminus \Lambda$. In the case $M=\R$, we apply Theorem \ref{t-centralizer} to this minimal action. If this action is conjugate to an action by translations, then again $[\Gamma_1, \Gamma_1]$ acts trivially on $\Lambda$. Otherwise its centralizer is trivial or cyclic, so that $[\Gamma_2, \Gamma_2]$ fixes $\Lambda$ pointwise. When $M=\T$, we argue similarly using Theorem \ref{t-Margulis}, which gives that either the action is conjugate to an action by rotations (in which case $[\Gamma_1,\Gamma_1]$ acts trivially), or the centralizer of $\varphi(\Gamma_1)$ is finite cyclic (in which case $[\Gamma_2,\Gamma_2]$ fixes every point of $\Lambda$). \end{proof}

\begin{rem}
The case $M=\T$ will be used only in \S \ref{s-circle}. In the case $M=\R$, the assumption that $\Gamma_1$ be finitely generated cannot be dropped, as shown by the following example. Let $(H, \prec)$ be any left-ordered non-abelian countable group, and consider the direct sum  $G=\bigoplus_{n\in \N}H$, i.e.\ the group  of all sequences $(h_n)$ in $H$ such that $h_n=1$ for all but finitely many $n$, with pointwise multiplication. Consider the lexicographic order on $G$, given by $(h_n)\prec (h'_n)$ if $h_m \prec h'_m$ for  $m=\max\{n\in \N : h_n\neq h'_n\}$, and let $\varphi\colon G\to \homeo_0(\R)$ be the dynamical realization of this order (Lemma \ref{lem.dynreal}). Let $\Gamma_1\subset G $ be the subgroup consisting of all sequences $(h_n)$ such that $h_n=1$  for $n$ even, and $\Gamma_2\subset G$ the subgroup of sequences such that $h_n=1$ for $n$ odd, so that $G=\Gamma_1\times \Gamma_2$. It is easy to see that neither the image of $[\Gamma_1, \Gamma_1]$ nor of $[\Gamma_2, \Gamma_2]$ have fixed points. 
\end{rem}
\begin{rem}
A special case of Proposition \ref{p-centralizer-fix} appears in Matsumoto \cite[Proposition 3.1]{Matsumoto} (for $M=\T$ and assuming $\Gamma_i$ are simple). 
\end{rem}

Proposition \ref{p-centralizer-fix} has the following consequence in our setting of micro-supported groups.

\begin{cor}\label{c-micro-fixed}
For $X=(a, b)$, let $G \subseteq \homeo_0(X)$ be a subgroup whose action is irreducible, and $\Gamma\subseteq G_c$ a finitely generated subgroup. Then, for every action $\varphi\colon G\to \homeo_0(\R)$, the image $\varphi([\Gamma, \Gamma])$ has fixed points. 
\end{cor}
\begin{proof}
Let $I\Subset X$ be a relatively compact subinterval such that $\Gamma \subseteq G_I$. Since the standard action of $G$ has no fixed points, we can find $g\in G$ such that $g(I)\cap I=\varnothing$. Then $g\Gamma g^{-1} \subseteq G_{g(I)}$ commutes with $\Gamma$, so that Proposition \ref{p-centralizer-fix} applied to the subgroups $\Gamma_1:=\Gamma$ and $\Gamma_2:=g\Gamma g^{-1}$ implies in either case that $\varphi([\Gamma, \Gamma])$ has fixed points. \end{proof}

When $G$ is locally moving, Corollary \ref{c-micro-fixed} can be combined with Proposition \ref{p-lm-reconstruction} to obtain the following stronger conclusion. 

\begin{prop}\label{p-lm-property-exotic} \label{p-lm-trichotomy}
For $X=(a,b)$, let $G\subseteq\homeo_0(X)$ be locally moving, and write $N=[G_c, G_c]$. Then, for any exotic action $\varphi\colon G\to \homeo_0(\R)$ and finitely generated subgroup $\Gamma \subseteq N$, the following hold.
\begin{enumerate}[label=(\roman*)]
\item  \label{i-exotic-totally-bounded} The boundary $\partial \fixphi(\Gamma)$ of its set of fixed points accumulates on both $\pm \infty$.
\item \label{i-boundary-non-discrete} $\partial \fixphi(\Gamma)$ is non-discrete. 
\item \label{i-exotic-domination} For every connected component $I$ of $\suppphi(\Gamma)=\R\setminus \fixphi(\Gamma)$, there exists $f\in N$ centralizing $\Gamma$, and such that $f.I\cap I=\varnothing$. 
\end{enumerate}
\end{prop}

\begin{proof}
Since $\Gamma$ is finitely generated and every element in some finite generating set can be written as a product of commutators of finitely many elements in $G_c$, we have $\Gamma \subseteq [\Delta, \Delta]$ for some finitely generated subgroup $\Delta\subseteq G_c$, and thus $\fixphi(\Gamma)\neq \varnothing$ by Corollary \ref{c-micro-fixed}. Also, since $\varphi$ is faithful, we have $\fixphi(\Gamma)\neq\R$, and so $\partial \fixphi(\Gamma)\neq \varnothing$.  The centralizer of $\Gamma$ contains $N_{(a, x)}$ and $N_{(y, b)}$ for suitable $x, y\in X$, and by Proposition \ref{p-lm-reconstruction} one of these two subgroups, say $N_{(a, x)}$, acts without fixed points under $\varphi$. Since $\partial \fixphi(\Gamma)$ is non-empty and $\varphi(N_{(a, x)})$-invariant, it must accumulate on both $\pm \infty$, showing \ref{i-exotic-totally-bounded}.   In particular, every connected component of $\supp^\varphi(\Gamma)$ is bounded.
Let $I=(\xi_1, \xi_2)$ be such a component;  since $\fixphi(N_{(a, x)})=\varnothing$, we can choose $f\in N_{(a, x)}$ such that $f.\xi_1>\xi_2$, showing \ref{i-exotic-domination}. Now, since $f\in N$, we can apply \ref{i-exotic-totally-bounded} to $\langle f\rangle$, and get that $\fixphi(f)$ also accumulates on $\pm \infty$. Thus $f^n.\xi_1$ must converge, as $n\to +\infty$, towards a fixed point $\zeta\in \fixphi(f)$. Since $\partial \fix^\varphi(\Gamma)$ is  $\varphi(f)$-invariant, the sequence $f^n.\xi_1$ is contained in $\partial \fix^\varphi(\Gamma)$, and the limit $\zeta=\lim_{n\to\infty}f^n.\xi_1$ also belongs to $\partial \fix^\varphi(\Gamma)$, showing \ref{i-boundary-non-discrete}. \qedhere

\end{proof}

\begin{rem}[Rigidity of piecewise analytic actions] \label{c-lm-PA} Parts \ref{i-exotic-totally-bounded}--\ref{i-boundary-non-discrete} of Proposition \ref{p-lm-property-exotic} can be seen as a first rigidity result of combinatorial nature,  saying that irreducible actions with particularly ``nice'' structure of fixed points are not exotic.
 For example, any action $\varphi$ by \emph{piecewise real-analytic homeomorphisms}\index{homeomorphism!piecewise real-analytic} has the property that $\partial \fixphi(g)$ is discrete for every $g\in G$, so Proposition \ref{p-lm-property-exotic} implies that $\varphi$ cannot be exotic. 
When $\varphi$ is a faithful piecewise analytic action on a \emph{compact} interval, without fixed points in the interior, one can actually conclude that $\varphi$ is semi-conjugate (on the interior) to the standard action. Indeed if $\varphi$ were semi-conjugate to a non-faithful action then by Corollary \ref{cor.unique}, the subset $ \partial \fixphi(g)$ would accumulate on the endpoints  for every $g\in [G_c, G_c]$, but this is impossible in piecewise analytic regularity.

Some special cases of this result can be found in the literature, for instance when studying embeddings between certain groups of piecewise linear or projective homeomorphisms (such as  Thompson--Brown--Stein groups $F_{n_1,\ldots, n_k}$). Typically such results are proven by showing that the image of an embedding must be locally moving, and applying Rubin's theorem. See, for instance, the work of Lodha \cite{Coherent}. 
\end{rem}

\chapter{Differentiable actions of locally moving groups}\label{s-differentiable}

In this chapter we are interested in actions on a closed interval, or on the real line, by diffeomorphisms of class $C^1$. 
First, let us observe that Proposition \ref{p-lm-property-exotic} already gives a rigidity result for actions on the real line by $C^2$ diffeomorphisms: irreducible actions $\varphi \colon  G \to \Diff^2_0(\R)$ of a locally moving group cannot be exotic. Indeed, recall that the classical Kopell's lemma states that whenever $f$ and $g$ are non-trivial commuting $C^2$ diffeomorphisms of a compact interval, if $f$ has no fixed point in its interior, then neither does $g$.   Thus, by Proposition \ref{p-lm-property-exotic}, it is easy to see that when $\varphi$ is an exotic action, the image of $G$ contains commuting elements preserving a compact interval, but not satisfying the conclusion of Kopell's lemma. We will get the same rigidity result for $C^1$ actions on the line, but this requires a different strategy, because Kopell's lemma fails for $C^1$ actions (as pointed out by the examples of Tsuboi \cite{PixtonTsuboi}, based on the work of Pixton \cite{pixton}).

\section{Conradian actions and differentiable actions}

Although we cannot use Kopell's lemma in the $C^1$ setting, there are other classical tools to understand how fixed points of different (possibly non-commuting) elements are disposed.

A \emph{pair of successive fixed points} for a subgroup $G\subseteq \homeo_0(\R)$ is a non-empty open interval $I\subset \R$, for which there is an element
$g\in G$ such that $I$ is a connected component of $\supp (g)$. A \emph{linked pair of successive fixed points}\index{linked pair of successive fixed points} for $G$ consists of two pairs of successive fixed points $I=(a, b)$ and $J=(c , d)$, such that either $\{a,b\}\cap (c,d)$ or $(a,b)\cap \{c,d\}$ is a point.
As pointed out by Navas \cite{Navas2010}, the previous notion is related to the dynamical counterpart of Conradian orderings on groups. Following the terminology of Navas and the third named author \cite{NavasRivas-Conrad}, we will say that a group action on an interval is \emph{Conradian}\index{action!Conradian} if it is irreducible (in restriction to the interior of the interval) and has no linked pair of successive fixed points. We have the fundamental fact that any Conradian action of a finitely generated group is semi-conjugate to an action by translations. This goes back to Plante \cite[Theorem 5.5]{Plante1975} (in the case of groups of sub-exponential growth), and then to Solodov \cite{Solodov}, and Beklaryan \cite{Beklarian} (for a concise proof, we refer to Navas \cite[Proposition 3.12]{Navas2010}). Here we adapt the statement to our purposes.

\begin{thm}\label{t-conrad}
	Any Conradian action $\varphi\colon G\to \homeo_0(\R)$ of a finitely generated group is positively semi-conjugate to an action by translations, via a non-trivial morphism $\tau\colon G\to \R$ (the \emph{Conrad homomorphism}\index{Conrad homomorphism}), which is unique up to positive rescaling.
	
	Moreover, if $f\in \ker\tau$ and $J$ is a connected component of $\suppphi(f)$, then $g.J\cap J=\emptyset$ for any $g\in G$ such that $\tau(g)\neq 0$.
\end{thm}

This result can be applied to describe the dynamics of finitely generated subgroups. Indeed, assume that $\varphi\colon G\to \homeo_0(\R)$ is Conradian, let $H\subseteq G$ be a finitely generated subgroup, and $I \subset \R$ a connected component of $\suppphi(H)$. Then, the restriction to $I$ of the $\varphi$-action of $H$ is still Conradian, therefore it is also semi-conjugate to an action by translations.

The following result is a version for $C^1$ pseudogroups of the classical Sacksteder's theorem in foliation theory. We state it here in the way it appears in the work of Deroin, Kleptsyn, and Navas \cite[Théorème E]{DKN-acta}  (see also Bonatti and Farinelli \cite[Theorem 1.3]{BonattiFarinelli} for a simplified proof). This result can be actually deduced from the earlier work of Katok and Mezhirov \cite{KatokMezhirov}.

\begin{thm}
	\label{p.sacksteder}
	Let $G\subseteq \Diff_0^1([0,1])$ be a subgroup acting with a linked pair of successive fixed points. Then there exist a point $x\in (0,1)$ and an element $h\in G$ for which $x$ is a hyperbolic fixed point: $h(x)=x$ and $h'(x)\neq 1$.
\end{thm}

\begin{rem}\label{r.sacksteder}
	In fact, the proof of Theorem \ref{p.sacksteder} gives a more precise statement that we point out, as it will be useful for the sequel.
	
	Write $I=[0,1]$. It is not difficult to see that the existence of a linked pair of successive fixed points for $G$, as in Theorem~\ref{p.sacksteder}, gives a subinterval $J\subset I$ and two elements $f,g\in G$ such that the images $f(J)$ and $g(J)$ are both contained in $J$, and are disjoint: $f(J)\cap g(J)=\varnothing$. (This situation is the analogue of a Smale's horseshoe for one-dimensional actions.)
	It follows that every element $h\in \langle f,g\rangle_+$ in the (free) semigroup generated by $f$ and $g$ satisfies $h(J)\subset J$, and moreover the images $h(J)$, where $h$ runs through the $2^n$ elements of length $n$ in the semigroup $\langle f,g\rangle_+$ (with respect to the generating system $\{f,g\}$), are pairwise disjoint. Clearly the inclusion $h(J)\subset J$ gives that every $h$ admits a fixed point in $h(J)$. Using a probabilistic argument, and uniform continuity of $f'$ and $g'$ on $J$, one proves that, as $n$ goes to infinity, most of the elements $h$ of length $n$ are uniform contractions on $J$. This implies that most elements $h\in \langle f,g\rangle_+$ of length $n$, when $n$ is large enough, have a unique fixed point in $J$, which is hyperbolic.
	
	We deduce that if $\Lambda\subset J$ is an invariant Cantor set for $f$ and $g$ (and thus for $\langle f,g\rangle_+$), then the hyperbolic fixed point for a typical long element $h\in \langle f,g\rangle_+$ will never belong to the closure of a gap $J_0$ of $\Lambda$: otherwise, $h$ would fix the whole gap $J_0$, and therefore there would be a point $y\in J_0$ for which $h'(y)=1$, contradicting the fact that $h$ is a uniform contraction.
\end{rem}

We point out a straightforward consequence of Theorem \ref{p.sacksteder}, attributed to Bonatti, Crovisier, and Wilikinson in \cite[Proposition 4.2.25]{Navas-book} (and largely investigated in \cite{BonattiFarinelli}).

\begin{cor}\label{t-Bo-Fa}
	If $G\subseteq \Diff^1_0([0,1])$ is a subgroup acting with a linked pair of successive fixed points, then there is no non-trivial element $f\in \Diff^1_0([0,1])$ without fixed points in $(0, 1)$, centralizing $G$.
\end{cor}

\section{Conradian differentiable actions of Thompson's group $F$}
Before discussing $C^1$ actions of general locally moving groups, we first prove a preliminary result in the case of Thompson's group $F$\index{group!Thompson's group}, namely we rule out the existence of faithful Conradian $C^1$ actions of $F$. In fact, this will be used when studying general locally moving groups.

The first step is to analyze actions which are sufficiently close to the trivial action.  For the statement, we recall that the $C^1$ topology on $\Diff_0^1([0,1])$ is defined by the $C^1$ distance
\[d_{C^1}(f,g)=\sup_{\xi\in [0,1]} |f(\xi)-g(\xi)|+\sup_{\xi\in [0,1]}|f'(\xi)-g'(\xi)|.\]
When $G$ is a finitely generated group endowed with a finite generating set $S$, we consider the induced topology on $\Hom\left (G,\Diff_0^1([0,1])\right )$, saying that two representations $\varphi$ and $\psi$ are $\delta$-close if  $d_{C^1}(\varphi(g),\psi(g) )\le \delta$ for every $g\in S$.

\begin{lem}\label{l-Conrad-C1-local} There exists a neighborhood $\mathcal V$ of the trivial representation in the space $\Hom\left (F,\Diff_0^1([0,1])\right )$, such that if $\varphi\in\mathcal V$ has no linked pair of successive fixed points, then $\varphi(F)$ is abelian.
\end{lem}

\begin{proof}  
	Let $\varphi \in\Hom\left (F,\Diff_0^1([0,1])\right )$ be an action with no linked pair of successive fixed points.	As any proper quotient of $F$ is abelian, it is enough to prove that $\ker\varphi$ is not trivial.
	Consider two non-empty open subintervals  $I\Subset J\Subset X=(0,1)$ with dyadic rational endpoints (we will say that the intervals are \emph{dyadic}).
	Let $h\in F$ be such that $h(I)=J$, choose an element $f\in F_I\cong F$ without fixed points in $I$, and set $g=hfh^{-1}\in F_J$ for the conjugate element (which acts without fixed points on $J$). Note that both $f$ and $g$ belong to the subgroup $H:=\langle g, [F_J, F_J]\rangle=\langle g,(F_J)_c\rangle$. The group $H$ is finitely generated (it is generated by $g$ and the subgroup $F_L$ for any dyadic subinterval $L\Subset J$ such that $g(L)\cap L\neq \varnothing)$, and since $[F_J, F_J]$ is simple and non-abelian, the abelianization of $H$ is infinite cyclic, generated by the image of $g$. We want to prove that $\varphi(H)$ acts trivially, giving the desired conclusion.\footnote{We thank the anonymous referee for suggesting the following elementary proof. A similar argument appears in Navas \cite[Lemma 2.7]{NavasGAFA}.} For this, assume by contradiction that $\suppphi(H)\neq\emptyset$, and take any connected component $(\xi_1,\xi_2)$  of $\suppphi(H)$. By assumption, the restriction of the $\varphi$-action of $H$ to $(\xi_1,\xi_2)$, is Conradian, so we can apply Theorem \ref{t-conrad}: the corresponding Conrad homomorphism $\tau\colon  H\to \R$ has cyclic image, generated by the image of $g$, and therefore $\varphi(g)$ acts without fixed points on $(\xi_1,\xi_2)$. As this holds for any connected component of $\suppphi(H)$, we have $\suppphi(g)=\suppphi(H)$. We now fix a connected component $(\xi_1,\xi_2)$ of $\suppphi(g)$ of maximal length (denoted as $|(\xi_1,\xi_2)|$). The interval $h^{-1}.(\xi_1,\xi_2)=:(\eta_1,\eta_2)$ is a connected component of $\suppphi(f)=h^{-1}.\suppphi(g)$, so it is contained in $\suppphi(H)=\suppphi(g)$; let $(\xi_1',\xi_2')$ be the connected component of $\suppphi(g)$ containing $(\eta_1,\eta_2)$. Using Theorem \ref{t-conrad} again, we observe that $g.(\eta_1,\eta_2)\cap (\eta_1,\eta_2)=\emptyset$, so that upon replacing $g$ by its inverse, we can assume $\eta_2\le g.\eta_1$.
		\begin{claim}
			There exists a point $\gamma\in [0,1]$ such that
			${|(\eta_1,\eta_2)|}\le |\varphi(g)'(\gamma)-1|\,{|(\xi_1',\xi_2')|}$.	
		\end{claim}
		\begin{proof}[Proof of claim]
			We choose $\gamma\in [\xi_1',\eta_1]$ such that \[\frac{|(\xi_1',g.\eta_1)|}{|(\xi_1',\eta_1)|}=\frac{|g.(\xi_1',\eta_1)|}{|(\xi_1',\eta_1)|}=\varphi(g)'(\gamma).\]
			Then we have
			\[
			\frac{|(\eta_1,\eta_2)|}{|(\xi_1',\xi_2')|}\le \frac{|(\eta_1,g.\eta_1)|}{|(\xi_1',\eta_1)|}=\left\vert \frac{g.\eta_1-\xi_1'}{\eta_1-\xi_1'}-1\right\vert=|\varphi(g)'(\gamma)-1|,
			\]
			as desired.
		\end{proof}
		Now, we assume that $\varphi$ is sufficiently closed to the trivial representation, so that
		\[\sup_{\xi\in [0,1]}|\varphi(g)'(\xi)-1|<\tfrac12\quad\text{and}\quad \sup_{\xi\in [0,1]}\varphi(h)'(\xi)<2\]
		(these inequalities determine the neighborhood $\mathcal V$ in the statement). If so, the claim gives the inequality
		$|(\eta_1,\eta_2)|< \tfrac12\,|(\xi'_1,\xi_2')|$,
		and thus
		\[
		|(\xi_1,\xi_2)|=|h.(\eta_1,\eta_2)|<2\, |(\eta_1,\eta_2)|< |(\xi_1',\xi_2')|,
		\]
		contradicting the choice of $(\xi_1,\xi_2)$ with maximal size. \qedhere
\end{proof}

The previous statement, which is of local nature (perturbations of the trivial actions), is used to obtain a global result.

\begin{lem}\label{l-Conrad-C1-global} Every Conradian $C^1$ action of $F$ on the closed interval $[0,1]$ has abelian image.
\end{lem}

\begin{proof}
Let $\varphi\in \Hom\left (F,\Diff_0^1([0,1])\right )$ be a Conradian action. Again, as any proper quotient of $F$ is abelian, it is enough to show that $\ker\varphi$ is non-trivial.				
After Theorem \ref{t-conrad}, the action $\varphi$ (restricted to the interior $(0,1)$) is semi-conjugate to an action by translations, and this is given by the Conrad homomorphism $\tau\colon F\to \R$. Take a minimal $\varphi$-invariant set  $\Lambda\subset (0,1)$, and note that $\ker \tau$ pointwise fixes $\Lambda$.  We write $\mathscr J$ for the collection of connected components of $(0,1)\setminus \Lambda$. 
Let us fix a dyadic subinterval $I\Subset X=(0,1)$. Since $F_I\subset \ker \tau$, the image $\varphi(F_I)$ preserves every interval $J\in \mathscr{J}$; let us denote by $\varphi_{I,J}\in \Hom\left (F_I,\Diff_0^1(\overline J)\right )$ the action obtained by restriction of $\varphi$. Note that for any element $g\in F$ such that $\supp(g)\cap I=\emptyset$, one has the relation
\begin{equation}\label{eq.conradF}
	\varphi(g)\circ \varphi_{I,J}\circ \varphi(g)^{-1}=\varphi_{I,g.J}.
\end{equation}
\begin{claim}
	There exists $\delta>0$ such that the image of $\varphi_{I,J}$ is abelian for every $J\in \mathscr{J}$  contained in $(0,\delta)$.
\end{claim}
\begin{proof}[Proof of claim]\footnote{We thank again the anonymous referee for this more elementary argument.}
	Note that as $\varphi$ is Conradian, $\varphi_{I,J}$ has no linked pair of successive fixed points. Let $\mathcal V$ be the neighborhood of the trivial representation provided by Lemma  \ref{l-Conrad-C1-local}, and denote by $\mathcal V_{I,J}\subset \Hom\left (F_I,\Diff_0^1(\overline J)\right )$ the corresponding neighborhood obtained after considering an identification $\Hom\left (F_I,\Diff_0^1(\overline J)\right )\cong \Hom\left (F,\Diff_0^1([0,1])\right )$. As $F_I$ fixes $\Lambda$, which accumulates at $0$, we see that $\varphi(g)'(0)=1$ for any $g\in F_I$; as a consequence (since $F_I$ is finitely generated), we can find $\delta>0$ such that if $J\subset (0,\delta)$, then $\varphi_{I,J}\in \mathcal V_{I,J}$, and hence the image of $\varphi_{I,J}$ is abelian.
\end{proof}
Now, the abelianization of $F$, isomorphic to $\Z^2$, is given by the derivatives of elements at $0$ and $1$ (for the standard action of $F$ on $X$), so when considering the Conrad homomorphism $\tau\colon F\to \R$ for $\varphi$, at least one of the two following situations happens:
	\begin{enumerate}[label=(\roman*)]
		\item\label{i:conradF} if $D^+g(0)\neq 1$ and $D^-g(1)=1$, then $\tau(g)\neq 0$,
		\item if $D^+g(0)= 1$ and $D^-g(1)\neq 1$, then $\tau(g)\neq 0$.
	\end{enumerate}
	Let us assume that the first case holds (otherwise we can argue similarly). We take an element $g\in F$ satisfying the conditions in \ref{i:conradF}, and such that $\supp(g)\cap I=\varnothing$. As $\tau(g)\neq 0$, given any $J\in \mathscr{J}$, we can find $n\in \Z$ such that $g^n.J\subset (0,\delta)$. According to the relation \eqref{eq.conradF}, the actions $\varphi_{I,J}$ and $\varphi_{I,g^n.J}$ are conjugate; since by the claim $\varphi_{I,g^n.J}$ is abelian, we deduce that $\varphi_{I,J}$ also is. As $J\in \mathscr{J}$ was arbitrary, and $\varphi(F_I)$ fixes $\Lambda$ pointwise, we deduce that $\varphi(F_I)$ is abelian, and thus $[F_I,F_I]\subset \ker \varphi$, as desired.
	\end{proof}
	
	With a similar proof, we can extend the previous result to $C^1$ actions on the real line.
	
	\begin{prop}\label{p-Conrad-C1-global} Every Conradian $C^1$ action of $F$ on the real line has abelian image.
	\end{prop}
	
	\begin{proof}
We proceed as in the proof of Lemma \ref{l-Conrad-C1-global}, and for this reason we skip some details. We start with a Conradian action $\varphi\in \Hom\left (F,\Diff_0^1(\R)\right )$, and consider the Conradian homomorphism $\tau\colon F\to \R$, and a minimal invariant set $\Lambda\subset \R$, pointwise fixed by $\ker \tau$. We denote by $\mathscr J$ the collection of connected components of $\R\setminus \Lambda$. For a given dyadic subinterval $I\Subset X$,  this gives rise to actions $\varphi_{I,J}\in \Hom\left (F_I,\Diff_0^1(\overline J)\right )$ without linked pairs of successive fixed points ($J\in \mathscr{J}$). After Lemma \ref{l-Conrad-C1-global} (applied to the restriction of $\varphi_{I,J}$ to the closure of every connected component of $\supp^{\varphi_{I,J}}(F_I)$), we deduce that $[F_I,F_I]\subseteq \ker \varphi_{I,J}$ and therefore $[F_I,F_I]\subseteq \ker \varphi$.
\end{proof}

\section{Differentiable actions of general locally moving groups}

Using that any locally moving group contains a copy of $F$ (Proposition \ref{p-chain}), we get the following consequence of Proposition \ref{p-Conrad-C1-global}.

\begin{prop}\label{p-lm-Conrad-C1} Locally moving groups admit no faithful Conradian $C^1$ actions on the real line.
\end{prop}

We actually have the following alternative, which is a more precise formulation of Theorem \ref{t-intro-C1} from the introduction.

\begin{thm}\label{t-lm-C1}
For $X= (a, b)$, let $G\subseteq \homeo_0(X)$ be locally moving, and write $N=[G_c,G_c]$. Then, every irreducible action $\varphi\colon G\to \Diff^1_0(\R)$ satisfies one of the following:
\begin{itemize}
	\item either $\varphi$ is semi-conjugate to the standard action of $G$ on $X$,or
	\item $\varphi$ is semi-conjugate to an action that factors through the quotient $G/N$.
\end{itemize}
Moreover,  the second case occurs if and only if $\varphi(N)$ has fixed points, in which case the $\varphi$-action of $N$ on each connected component of its support is semi-conjugate to its standard action on $X$. 
\end{thm}

\begin{proof}
Assume by contradiction that $\varphi\colon G\to \Diff^1_0(\R)$ is an exotic action of $G$ (that is, not semi-conjugate to the standard action on $X$, nor to any action of the quotient $G/N$). Note that $N$ is itself locally moving, and by Proposition \ref{p-chain} we can find a subgroup $\Gamma\subseteq N$ isomorphic to $F$.  Take a connected component $I\subset \supp^\varphi(\Gamma)$. By Proposition \ref{p-lm-property-exotic}, $I$ is bounded, and we can take an element $f\in N$ centralizing $\Gamma$, such that $f.I\cap I=\varnothing$. Take the connected component  $J$ of $\suppphi(f)$ containing $I$. Again by Proposition \ref{p-lm-property-exotic}, we have that $J$ is bounded. Since $\Gamma$ and $f$ commute, we have that $\Gamma$ preserves $J$ and, applying Corollary \ref{t-Bo-Fa} to the induced action of $\Gamma$ on $\overline J$ obtained by restriction of $\varphi$, we deduce that the restriction of $\varphi(\Gamma)$ to $\overline{J}$ has no linked pair of successive fixed points. Thus, the action of $\Gamma$ on $\overline{I}$ is Conradian, and we deduce from Lemma \ref{l-Conrad-C1-global} that the restriction of $\varphi(\Gamma)$ to $I$ is abelian. As $I\subset \supp^\varphi(\Gamma)$ is arbitrary, we get that $\varphi(\Gamma)$ is abelian, and therefore $\varphi\colon G\to \Diff^1_0(\R)$ is not faithful, which is an absurd.

The last statement is a consequence of the fact that the subgroup $N$ is still locally moving (by Lemma \ref{l-rigid}), thus we can apply the first part of the theorem to its action on the connected components of the support;  since $N$ is simple (Proposition \ref{p-micro-normal}), only the first case can occur.
\end{proof}

For the last result of this section, recall from Definition \ref{d-fragmentable-subgroup} that $G\subset \homeo_0(X)$ is fragmentable if it is generated by elements having trivial germ at one of the endpoints of $X$.

\begin{cor}\label{c-lm-C1-interval}
For $X=(a,b)$, let $G\subseteq \homeo_{0}(X)$ be a fragmentable locally moving subgroup. Then the following hold.
\begin{itemize}
	\item \label{i-C1-interval} Every faithful action $\varphi\colon G\to \Diff_0^1([0,1])$ without fixed points in $(0,1)$ is semi-conjugate (on $(0,1)$) to the standard action on $X$. 
	\item \label{i-C1-line} Every irreducible faithful action $\varphi\colon G\to \Diff_0^1(\R)$ is either semi-conjugate to the standard action on $X$, or to a cyclic action.
	
	In the latter case, if $\tau \colon G\to \Z$ is the homomorphism giving the semi-conjugate action, then the action of $\ker \tau$ on each connected component of its support is semi-conjugate to its standard action on $X$.
\end{itemize}
\end{cor}

To prove Corollary \ref{c-lm-C1-interval} we need the following lemma.
\begin{lem}\label{l-independent-germs}
For $X=(a, b)$, let $G\subseteq \homeo_0(X)$ be a {fragmentable} locally moving subgroup. Then, for every pair of points $c<d$ in $X$, the subgroup $\langle G_{(a, c)},G_{(d, b)}\rangle $ projects onto the largest quotient $G/[G_c, G_c]$.
\end{lem}
\begin{proof}
Given $\gamma\in \Germ(G, a)$, the assumption implies (by Remark \ref{r-fragmentable-germs}) that there exists $g\in G$ with $\Gcal_a(g)=\gamma$ and $\Gcal_b(g)=\Id$. Thus, we have $g\in G_{(a, x)}$ for some $x\in X$. If we choose $h\in G_c$ such that $h(x)<c$, the element $g'=hgh^{-1}$ belongs to $G_{(a, c)}$ and satisfies $\Gcal_a(g')=\gamma$. We conclude that the group $G_{(a, c)}$ projects onto $\Germ(G, a)$, and similarly $G_{(d, b)}$ projects onto $\Germ(G, b)$. Therefore the subgroup $\langle G_{(a, c)}, G_{(d, b)}\rangle$ projects  onto $\Germ(G, a)\times \Germ(G, b)\cong G/G_c$. Consequently, it is enough to show that its image in $G/[G_c, G_c]$ contains $G_c/[G_c, G_c]$. This happens because every $g\in G_c$ is  conjugate  inside $G_c$  to an element of $\langle G_{(a, c)}, G_{(d, b)}\rangle$ with the same image in the abelianization $G_c/[G_c, G_c]$. \qedhere

\end{proof}

\begin{proof}[Proof of Corollary \ref{c-lm-C1-interval}]
Write $N=[G_c, G_c]$ and usual, and $Y=[0, 1]$ or $Y=\R$, according to the case in the statement. By Theorem \ref{t-lm-C1}, the $\varphi$-action of $G$ on the interior of $Y$ is either semi-conjugate to the standard action on $X$, or to an action that factors through $G/N$. Assume that the second condition holds: after Theorem \ref{t-lm-C1}, we know more precisely that if $I$ is a connected component of $\suppphi(N)$, then the induced action of $N$ on $I$ is semi-conjugate to the standard action on $X$. In particular, it admits linked pairs of successive fixed points and by Theorem \ref{p.sacksteder} we can find $h\in N$ with a hyperbolic fixed point $\xi\in I$. We make a slightly more elaborate argument than the one that is needed for Corollary \ref{t-Bo-Fa}. Let $(c, d)\Subset X$ be such that $h\in N_{(c, d)}$. Then the subgroup $H=\langle G_{(a,c)},G_{(d,b)}\rangle$ centralizes $h$, and thus every point in the orbit $H.\xi$ is fixed by $h$, with derivative always equal to $\varphi(h)'(\xi)\neq 1$. As  the derivative $\varphi(h)'$ is continuous, it must be that the orbit $H.\xi$ is discrete in $Y$. In  the case $Y=[0,1]$, the only possibility is  that $\xi$ is fixed by $\varphi(H)$, and after Lemma \ref{l-independent-germs}, the quotient action of $G/N$ has a fixed point as well. Similarly, in the case $Y=\R$, $\varphi(H)$ cannot fix $\xi$, so that the orbit $H.\xi$ is infinite and discrete.   
Using Lemma \ref{l-independent-germs} again, we deduce that the quotient action of $G/N$ has an infinite discrete orbit, which means that it is semi-conjugate to a cyclic action.

For the last statement, apply the case $Y=[0,1]$ to the action of $\ker \tau$.
\qedhere
\end{proof}

{ Notice that, since Thompson's group $F$ is fragmentable, Corollary \ref{c-lm-C1-interval} gives the following.
\begin{thm}[Rigidity of differentiable actions]\label{c-F-C1}
	Thompson's group $F$ satisfies the following. 
	\begin{itemize}
		\item For  every faithful action $\varphi\colon F\to \Diff^1([0,1])$ without  fixed points in $(0,1)$,  the $\varphi$-action of $F$ on $(0,1)$  is semi-conjugate to the standard action of $F$ on $(0, 1)$.
		\item Every faithful action $\varphi\colon F\to  \Diff^1(\R)$ without discrete orbits is semi-conjugate to the standard action on $(0, 1)$.
	\end{itemize}
\end{thm}
\begin{rem}\label{r-F-C1}
	Note that the conclusion is optimal: there exist $C^1$ actions (and even $C^{\infty}$ actions) of $F$  on closed intervals which are semi-conjugate, but not conjugate to its standard action. The existence of such actions was shown by  Ghys and Sergiescu \cite{GhysSergiescu}, or alternatively can be shown  using the ``2-chain lemma'' of Kim, Koberda, and Lodha \cite{KKL} (see Proposition \ref{p-chain}). 
	\end{rem}}

	\begin{rem}
In the setting of Theorem \ref{t-lm-C1} (or Corollary \ref{c-lm-C1-interval}), it may well happen that the standard action of $G$ is not semi-conjugate to any $C^1$ action. In this case, the theorem also provides a tool  to show that certain subgroups of $\homeo_0(\R)$ do not admit any embedding into groups of interval diffeomorphisms of a certain regularity. We give some examples of applications in this direction in \S \ref{ssc:Stein}.
\end{rem}
\begin{rem}
Theorem \ref{t-lm-C1} is not true if the locally moving assumption on $G$ is relaxed to the assumption that $G$ be only micro-supported. In \S \ref{s-micro-C1}, we construct  groups admitting many, pairwise non-semi-conjugate, micro-supported differentiable actions (this construction is described within the framework of laminar actions, developed in Part \ref{partII}).
\end{rem}

\chapter{First examples of exotic actions}\label{ss.exoticactions} 
The goal of this chapter is to give some first examples showing that the rigidity of $C^1$ actions established in Theorem \ref{t-lm-C1} fails for actions by homeomorphisms: namely we provide three classes of examples of locally moving subgroups of $\homeo_0(\R)$ that admit exotic actions on the line (in the sense of Definition \ref{d-exotic-action}). More constructions will appear in Part \ref{partII}, and some of these examples will be also revisited after proving our main structure theorems for $C^0$ actions.

First, in \S \ref{s-germ-type}, we discuss a general construction of left orders on the  group $\homeo_c(\R)$ of \emph{compactly supported} homeomorphisms, and use it to show that every \emph{countable} subgroup of $\homeo_c(\R)$, whose standard action is minimal, admits exotic actions. While this construction is particularly simple,  it cannot be applied to any finitely generated group, and the exotic actions obtained are never minimal (in fact, they do not admit any minimal invariant set).
In \S \ref{s.BSjump}, we provide a different construction, that applies to groups of (finitary) PL homeomorphisms. We shall show that every group in the family $G(X; A, \Lambda)$ of Bieri--Strebel groups admits  \emph{minimal}  exotic actions.
This provides,  in particular, examples of exotic actions of Thompson's group $F$. Finally, in \S \ref{s-cyclic-germ} we describe yet another construction that produces minimal exotic actions, and applies to any locally moving subgroup of $\homeo_0(X)$ having a cyclic group of germs at one endpoint of $X$. This class also contains Thompson's group $F$, as well as many examples outside the realm of groups of PL homeomorphisms.

\section{Orders of germ type} \label{s-germ-type} 

In this section we build exotic actions on the line that are obtained as dynamical realizations of some left-invariant orders on the group of compactly supported homeomorphisms of $\R$. The orders here are inspired by the well-known construction of bi-invariant orders on the group of orientation-preserving piecewise linear homeomorphisms of an interval (see for instance the work of Brin and Squier \cite[\S 2]{BrinSquier}, or even the older work by Dlab \cite{Dlab}).

Let $G\subset\homeo_c(\R)$ be a countable subgroup whose action on $\R$ is minimal (note that $G$ is automatically micro-supported, by Proposition \ref{p-micro-compact}).  For example, one can take $G=[F,F]$, the commutator subgroup of Thompson's group $F$. Note that $G=G_c$, so that $N=[G,G]$ is the minimal non-trivial proper normal subgroup of $G$ (Proposition \ref{p-micro-normal}).  For a point $x\in\R$ {and an element $g\in \mathsf{Stab}_G(x)$, we let $\Gcal_x(g)$ be its germ at $x$}; with abuse of notation, we will denote by $\id$ the trivial germ, without reference to the base point.

It is known that groups of germs of interval homeomorphisms are left orderable (see Mann \cite[Proposition 3]{Mann}, where the proof, based on a compactness argument, is attributed to Navas; see also the monograph \cite[Remark 1.1.13]{GOD}).
Using this fact, for each $x\in\R$ we can take a left-invariant order $<^{(x)}$ on $\Germ\left (G_{(-\infty,x)},x\right )$, and define 
\[P=\left \{g\in G:\Gcal_{p_g}(g)>^{(p_g)}\id\right \},\]
where $p_g:=\sup \{x\in X : g(x) \neq x\}$. It is straightforward to check that $P$ is a semigroup, disjoint from $P^{-1}$, and that  defines a partition $G=P\sqcup \{1\}\sqcup P^{-1}$. Thus, $P$ is the positive cone of a left-invariant order $\prec\in\LO(G)$ (see Remark \ref{r.cones}). Denote by $\varphi\colon G\to\homeo_0(\R)$ its dynamical realization. 

\begin{prop}
Let $G\subset \homeo_c(\R)$ be a {countable  subgroup} whose action is minimal. Then the action $\varphi\colon G\to \homeo_{0}(\R)$ defined above is not semi-conjugate to any action of the largest proper quotient $G/[G_c,G_c]$, nor to the standard action. Furthermore, the action $\varphi$ does not admit any minimal invariant set. 

\end{prop}
\begin{proof}
First, we will show that $\varphi$ is not semi-conjugate to any action induced from a quotient. For this purpose, we claim that the image $\varphi(N)$ acts on $\R$ without fixed points, where we set $N=[G_c,G_c]$ as usual. Indeed to show this, it is enough to show that the orbit of $\id$ under the action of $N$ is unbounded from above and from below, in the ordered et $(G, \prec)$. For this, fix an element $h\in G$ with $h\succ \id$. Since $N$ is normal in $G$, the subset $\{p_g: g\in N\}\subseteq X$ is $G$-invariant, so that by minimality there exists  $g\in N$ with $p_g>p_h$, and thus $p_{gh^{-1}}=p_g$.  Upon replacing $g$ with $g^{-1}$, we can assume $\Gcal_{p_g}(g)>^{(p_g)} \id$, so that
\[\Gcal_{p_{gh^{-1}}}(gh^{-1})=\Gcal_{p_g}(g) >^{(p_g)} \id,\]
meaning that $g\cdot \id=g\succ h$.
Similarly, for every $h\prec \id$, we can find an element $g'\in N$ such that $g'\cdot \id \prec h$. This shows that the $N$-orbit of $\id$ is unbounded in both directions, as desired.

For the remaining part of the statement, note that every action which is semi-conjugate to a minimal action admits a minimal invariant set. So it is enough to show that $\varphi$ has no minimal invariant set.
For this, note that for every $x\in X$, the subgroup $G_{(a, x)}$ is $\prec$-convex, so $G=\bigcup_{x} G_{(a, x)}$ is an increasing union of $\prec$-convex bounded subgroups. It follows that $\R$ can be correspondingly written as an increasing union of bounded intervals $\R=\bigcup_x I_x$,  which are wandering intervals for $\varphi$, in the sense that any two $\varphi$-images of $I_x$ are either equal or disjoint (namely, define $I_x$ as the interior of the convex hull of $\iota(G_{(a, x)})$, where $\iota\colon (G, \prec)\to \R$ is the good embedding associated with the dynamical realization, see Definition \ref{d-dynamical-realisation}). This easily implies that no bounded interval intersects every $\varphi$-orbit, so $\varphi$ has no minimal invariant set (Lemma \ref{l.zorn_minimal}).\qedhere
\end{proof}

\section{A construction of minimal exotic actions for Bieri--Strebel groups}\label{s.BSjump}

Recall that we denote by $G(X;A,\Lambda)$ the Bieri--Strebel group  acting on an interval $X$, associated with a  multiplicative subgroup $\Lambda\subseteq\R_{>0}$, and a non-trivial $\Z[\Lambda]$-submodule $A\subset \R$ (see Definition \ref{d.BieriStrebel}). We shall assume that $A$ and $\Lambda$ are countable, and set $G:=G(X;A,\Lambda)$. We proceed to construct \emph{minimal} exotic actions of such groups, working in particular for Thompson's group $F=G((0,1), \Z[1/2], \langle 2 \rangle_\ast)$.

For $g\in G$, we define the associated \emph{jump cocycle} as the function
\begin{equation}\label{eq.jump-cocycle}\dfcn{j_g}{X}{\Lambda}{x}{\dfrac{D^+g^{-1}}{D^-g^{-1}}(x).}\end{equation}
Note that for every $g\in G$, the support $\{x\in X:j_g(x)\neq 1\}$ of $j_g$ coincides with the set $\BP(g^{-1})$ of breakpoints of $g^{-1}$, which is finite. 
The chain rule for derivatives gives the {cocycle relation}:
\begin{equation}\label{eq.cocycle}
j_{gh}(x)=j_{g}(x)j_{h}(g^{-1}(x))\quad\text{for every }x\in X\text{ and }g,h\in G.
\end{equation}
We denote by  $\mathsf S=\{j_g:g\in G\}$ the collection of all jump cocycles, which we consider as a subset of the set $\bigoplus_X \Lambda$
of all finitely supported functions $f\colon X\to \Lambda$. The cocycle rule \eqref{eq.cocycle} allows to define an action of   $G$ on  $\bigoplus_X \Lambda$ by the formula
\begin{equation} \label{e-jump-action} g\cdot \s(x)=j_{g}(x)\s(g^{-1}(x))\quad\text{for  }\s\in \textstyle \bigoplus_X \Lambda\text{ and }g\in G.
\end{equation}
This action preserves the subset $\mathsf S$, which is exactly the orbit of the function with constant value~$1$: indeed, the cocycle rule yields $g\cdot j_h=j_{gh}$. 

We now use this action to construct a family of left preorders on $G$. For this, consider a (non-trivial, invariant) preorder $\leq_\Lambda\in\LPO(\Lambda)$ on the abelian group $\Lambda$, and denote by $\equiv_\Lambda$ the associated equivalence relation on $\Lambda$ (that is, $\lambda_1\equiv_\Lambda \lambda_2$ if $\lambda_1\leq_\Lambda \lambda_2\leq_\Lambda \lambda_1$). To the preorder $\leq_\Lambda$, we can associate a preorder of \emph{lexicographic type} on $\bigoplus_X \Lambda$, defined by setting $\mathsf s\preceq\mathsf t$ if and only if 
\begin{itemize}
\item either $\mathsf s(x)\equiv_\Lambda \mathsf t(x)$ for every $x\in X$, or
\item $\mathsf s(\overline{x}_{\s,\mathsf t})\lneq_\Lambda\mathsf t(\overline{x}_{\s,\mathsf t})$, where $\overline{x}_{\s,\mathsf t}:=\max\{x\in X:\s(x)\not\equiv_\Lambda \mathsf t(x)\}$.
\end{itemize}

\begin{dfn}[Jump preorders on Bieri--Strebel groups] For a given preorder $\leq_\Lambda\in\LPO(\Lambda)$, the \emph{jump preorder} on the Bieri--Strebel group $G(X;A,\Lambda)$ associated with $\leq_\Lambda$ is the preorder defined by the relation: $g\preceq h$ if and only if $j_g\preceq j_h$ for the preorder of lexicographic type on $\bigoplus_X \Lambda$.\footnote{The notation should not cause any confusion, as the correspondence $g\mapsto j_g$ preserves the preorder.}
\end{dfn}

\begin{lem}\label{lem.invaPhi} Jump preorders on Bieri--Strebel groups are left invariant. \end{lem}

\begin{proof} First notice that preorder of lexicographic type $\preceq$ on $\bigoplus_X \Lambda$  is invariant with respect to the group structure on $\bigoplus_X \Lambda$ given by pointwise multiplication.  It is also clearly invariant under the action of $\homeo_0(X)$ on $\bigoplus_X \Lambda$ by precomposition. By  \eqref{e-jump-action}, it is invariant under the action of  $G$. Finally, left invariance of the jump preorder on $G$ follows from the relation $g\cdot j_{h}=j_{gh}$. \end{proof}

Given a jump preorder $\preceq\in \LPO(G)$, we can consider its dynamical realization to obtain an action on the line $\varphi\colon G\to \homeo_0(\R)$.  More precisely, we take the dynamical realization of the action on the  ordered space $(G/H, \prec)$, where $H:=[1]_\preceq$ is the residue of the preorder (the subgroup consisting of all $g\in G$ such that $g\preceq h\preceq g$); see \S \ref{s-preorders}. We will show the following. 
\begin{prop}\label{prop.ejemplojump}  Let $G=G(X; A, \Lambda)$ be a (countable) Bieri--Strebel group, where $X=(a, b)$ is an interval with $a, b\in A\cup\{\pm \infty\}$. Then for every preorder $\leq_\Lambda\in \LPO(\Lambda)$, the dynamical realization $\varphi\colon G\to \homeo_0(\R)$ of the associated  jump preorder is a faithful minimal action, not semi-conjugate to the standard action of $G$ on $X$. Moreover, when $\leq_\Lambda$ and $\leq'_{\Lambda}$ are different preorders on $\Lambda$, the dynamical realization of their associated jump preorders are not positively \mbox{(semi-)conjugate}. 
\end{prop}	

\begin{rem} \label{r-description-jump-cosets}  When $\preceq$ is a jump preorder and $H=[\id]_{\preceq}\subset G$ its residue, the ordered space  $(G/H, \prec)$ can be concretely described as follows.  Let $\Lambda_0=[1]_{\le_\Lambda}\subset \Lambda$ be the residue of the preorder $\leq_{\Lambda}$. Then the preorder $\leq_{\Lambda}$ descends to an invariant order on the quotient $\Lambda^*=\Lambda/\Lambda_0$ (which is a group, since $\Lambda$ is abelian), and $\equiv_\Lambda$ is the congruence relation modulo $\Lambda_0$.  Therefore,  $H$ consists precisely of elements  $g\in G(X; A, \Lambda)$ such that $j_g(x)\in \Lambda_0$ for every $x\in X$. Considering the formula \eqref{e-jump-action}  modulo $\Lambda_0$, the $G$-action on $\bigoplus_X \Lambda$ descends to an action on  $\bigoplus_X \Lambda^*$, and $H$ coincides with the stabilizer of the trivial function (constant, equal to 1). Thus $G/H$ can be identified  with the image $\mathsf{S}^*$ of $\mathsf S$ in $\bigoplus_X \Lambda^*$.
Under this identification, the invariant order $\prec$ coincides with the natural lexicographic order on $\mathsf{S}^*$ induced by $\leq_\Lambda$. 
\end{rem}

The proof of Proposition \ref{prop.ejemplojump} requires some preliminaries. In what follows, we fix a preorder $\leq_\Lambda$ on $\Lambda$, and use $\preceq$ to denote both the induced lexicographic preorder on $\bigoplus_X \Lambda$, and the jump preorder on $G$.
We denote by $\equiv$ the equivalence relation on $\bigoplus_X \Lambda$ (and on $G$) associated with the  preorder $\preceq$, namely $\mathsf{s}\equiv \mathsf{t}$ if $\mathsf{s}(x)\equiv_\Lambda \mathsf{t}(x)$ for every $x\in X$, and  $g\equiv h$ if $j_g \equiv j_h$. 
For simplicity, given $g, h\in G$ with $j_{g}\not \equiv j_{h}$, we write $\overline{x}_{g, h}:=\overline{x}_{j_{g}, j_{h}}$, and when $g\not\equiv {\id}$ we simply write
\[\overline{x}_g:=\overline{x}_{\id, g}=\max \{x\in X:j_g(x)\not\equiv_\Lambda 1\}.\]
Notice that with this notation, we have $\id\precneq g$ if and only if $1\lneq_\Lambda j_g(\overline{x}_g)$.

\begin{lem}\label{lem.tecnicoBP}With notation as above, assume that $g,h,k\in G$ are such that $g(\overline{x}_h)> \overline{x}_k$ and $\mathsf{BP}(g)\cap [\overline{x}_h,b)=\emptyset$. Then, we have the following:
\begin{itemize} 
	\item if $h\succneq {\id}$, then $gh\succneq k$,
	\item if $h\precneq \id$, then $gh\precneq k$.
\end{itemize}
\end{lem}
\begin{proof}
Let us first show that $\overline{x}_{gh}=g(\overline{x}_h)$. From the assumption $\BP(g)\cap[\overline{x}_h,b)=\emptyset$ and the observation that $g(\BP(g))=\BP(g^{-1})$, we have
\[
\emptyset=g\left (\BP(g)\cap[\overline{x}_h,b)\right )=\BP(g^{-1})\cap[g(\overline{x}_h),b),
\]
and therefore $j_{g}(x)=1$ for every $x\in[g(\overline{x}_h),b)$. On the other hand, we have
\[\max\{x\in X:j_h(g^{-1}(x))\not\equiv_\Lambda 1\}=g(\overline{x}_h).\]
With these two observations in mind when looking at the relation $j_{gh}(x)=j_g(x)j_h(g^{-1}(x))$, we get that $\overline{x}_{gh}=g(\overline{x}_h)$, as desired. By assumption, this gives $\overline{x}_{gh}>\overline{x}_k$, and thus $\overline{x}_{gh ,k}=\overline{x}_{ gh}$. Therefore, the condition $k\precneq gh$ is equivalent to
\[
1\lneq_\Lambda j_{gh}(\overline{x}_{gh})=j_{gh}(g(\overline{x}_h))=j_{g}(g(\overline{x}_h))j_h(\overline{x}_h)=j_h(\overline{x}_h),
\]
giving the first statement. The proof of the second statement is analogous.
\end{proof}

We will also need the following elementary lemma on Bieri--Strebel groups, which uses the assumption that the endpoints of $X$ belong to $A\cup\{\pm \infty\}$.
\begin{lem}\label{l-BS-endpoint-focal}
Let $X=(a, b)$ with $a, b\in A\cup\{\pm\infty\}$, and take any two points $x_1<x_2$ in $A\cap X$. Then  there exists $g\in G(X; A, \Lambda)$ such that $g(x_1)>x_2$ and $\BP(g)\cap (x_1, b)=\varnothing$.
\end{lem}

\begin{proof}
We will use the result of Bieri and Strebel \cite[Theorem A4.1]{BieriStrebel}, which states that given $y, z_1, z_2\in A$ with $y<z_1$ and $y<z_2$, there exists an element $h\in G(\R; A, \Lambda)$ that maps $(y, z_1)$ to $(y, z_2)$ if and only if $z_2-z_1$ belongs to the submodule 
\[I\Lambda \cdot A:= \langle (1-\lambda)a \text{ with } \lambda \in \Lambda, a\in A\rangle.\]
For $\lambda\in \Lambda$ and $w\in A$, we denote by $g(w,\lambda)$ the affine map $x\mapsto \lambda x +(1-\lambda)w$ with slope $\lambda$ and fixed point $w$.
Let us suppose first that $b=+\infty$.  Choose $\lambda\in \Lambda$ large enough so that $g(a_0,\lambda)(x_1)>x_2$. Then we choose $g$ coinciding with $\id$ on $(a, a_0)$, and with $g(a_0,\lambda)$ on $(a_0, +\infty)$. 
Suppose now that $b<+\infty$. Fix $a_0\in A\cap (a, x_1)$; when $\lambda\in \Lambda$ is small enough, we have $g(b,\lambda)(x_1)\in (x_2, b)$. Since
\[
g(b,\lambda)(x_1)-x_1=(\lambda-1)x_1+(1-\lambda)w\in I\Lambda\cdot A,
\]
we can choose $h\in G(\R; A, \Lambda)$ that fixes $a_0$ and maps $x_1$ to $g(b,\lambda)(x_1)$. Let $g\in G(X; A, \Lambda)$ be the element that coincides with the identity on $(a, a_0)$, with $h$ on $(a_0, x_1)$, and with $g(b,\lambda)$ on $[x_1, b)$. Then $g$ satisfies the desired conclusion.  \qedhere
\end{proof}
\begin{lem}\label{lem.sufminimality} For every $g_1,g_2,h_1,h_2\in G$ satisfying $g_1\precneq h_1\precneq h_2\precneq {g_2}$, there exists $g\in G$ such that $gh_1\precneq g_1\precneq g_2\precneq  gh_2$.
\end{lem}
\begin{proof} First we show that there exists $k\in G$ so that ${h_1}\precneq {k}\precneq {h_2}$. For this, take $x\in X$ so that $\BP(h_i)\cap(x,\overline{x}_{{h_1},{h_2}})=\emptyset$ for $i\in \{1,2\}$, and choose $y\in A\cap(x,\overline{x}_{{h_1},{h_2}})$. Since $y\in A$, we can take $k'\in G$ such that $\overline{x}_{k'}=h_1^{-1}(y)$ and $j_{k'}(\overline{x}_{k'})\gneq_\Lambda 1$.
After these choices,
and from the relation $j_{h_1k'}(x)=j_{h_1}(x)j_{k'}(h_1^{-1}(x))$, we get $\overline{x}_{h_1k',h_1}=y$, and
\[
j_{h_1k'}(y)=j_{h_1}(y)j_{k'}(\overline{x}_{k'})\gneq_{\Lambda} j_{h_1}(y).
\]
Hence, $h_1\precneq h_1k'$. On the other hand, we have $j_{h_1k'}(x)\equiv_{\Lambda}j_{h_1}(x)$ for every $x>y$, and in particular for $x=\overline{x}_{h_1,h_2}$. This gives the other inequality $h_1k'\precneq {h_2}$. Then $k:=h_1k'$ satisfies the desired conclusion.

Now, to prove the statement, by transitivity we can assume $h_1\precneq \id\precneq h_2$. Using Lemma \ref{l-BS-endpoint-focal}, choose $g\in G$ so that $\BP(g)\cap (\overline{x}_{h_i},b)=\emptyset$ and $g(\overline{x}_{h_i})>\overline{x}_{g_i}$ for $i\in \{1,2\}$. Then from Lemma \ref{lem.tecnicoBP},  we deduce that $g{h_1}\precneq{g_1}$ and ${g_2}\precneq  g{h_2}$, as desired.
\end{proof}

\begin{proof}[Proof of Proposition \ref{prop.ejemplojump}] Denote by $\varphi\colon G\to\homeo_0(\R)$ the dynamical realization of the jump preorder. Let  $H=[\id]_{\preceq}$ be the residue of the jump preorder (see Remark \ref{r-description-jump-cosets}), and denote by  $i\colon  G\to\R$ the map obtained by composing the projection $G \to G/H$ with the good embedding $(G/H, \prec)\to \R$ associated with the dynamical realization (Definition \ref{d-dynamical-realisation}). Minimality of $\varphi$ follows from Lemma \ref{lem.sufminimality}, which ensures the conditions in the minimality criterion from Lemma \ref{lem.minimalitycriteria}. 

We proceed to check that $\varphi$ is faithful, which is the same as checking that  for every non-trivial $g\in G$, there exists $k\in G$ such that $gk\not\equiv k$, or equivalently that $ j_{gk}\not \equiv j_k$. 
When $ j_{g}\not\equiv j_{\id}$ this is easily verified, otherwise we must have $j_{g}(x)\equiv_\Lambda 1$ for every $x\in X$. If so, then we have
\[j_{gk}(x)=j_{g}(x)j_k(g^{-1}(x))\equiv_\Lambda j_k(g^{-1}(x))\]
for every $k\in G$.
This implies that if $ j_{gk}\equiv j_k$, then $g$ must fix the finite set $\{x: j_k(x)\not\equiv_\Lambda 1\}$.  However, it is clearly possible to choose $k\in G$ such that $j_k(x)\not\equiv_\Lambda 1$ for some $x\in\supp(g)$, so that $ j_{gk}\not\equiv j_k$, as wanted. 

We will now show simultaneously that $\varphi$ is not semi-conjugate to the standard action, and that the preorder $\leq_\Lambda$ can be read from the positive semi-conjugacy class of $\varphi$. Actually, as we have already checked that $\varphi$ is minimal, we have that any semi-conjugacy is automatically a conjugacy. For this, fix $c\in X$ and let $(d_n)\subset (c,b)\cap A$ be a sequence converging to $b$. Given a non-trivial $\lambda\in\Lambda$,  consider a sequence $(g_n)\subseteq G$ satisfying:

\begin{itemize}
	\item $g_n\in G((c, d_n);A,\langle\lambda\rangle_*)$ for every $n$, where $G((c, d_n);A,\langle\lambda\rangle_*)\subseteq G$ is the Bieri--Strebel group associated with the interval $(c, d_n)\Subset X$ and the cyclic multiplicative subgroup $\langle \lambda\rangle_*\subseteq \Lambda$, 
	
	\item {$D^-g_n^{-1}(d_n)=\lambda^{-1}$.}
	
\end{itemize}
The existence of such a sequence follows from \cite[Corollary A5.3.(ii)]{BieriStrebel}. Now, for every $h\in G$ and $n\in \N$, we have
$j_{g_nh}(x)=j_h(x)$ for every $x\in (d_n,b)$, and
$j_{g_nh}(d_n)\equiv_\Lambda \lambda\, j_h(d_n)$.

Assume first that $\lambda\equiv_\Lambda 1$, and notice that in this case $ j_{g_n}\equiv j_\id$ for every $n\in\N$, since  $g_n$ is in the subgroup $G((c, d_n); A, \langle \lambda \rangle_*)$. Therefore, $i({\id})$ is a common fixed point of the family $\{\varphi(g_n):n\in\N\}$. Consider now the case where $\lambda\not\equiv_\Lambda 1$; in this case, for every $h\in G$ and sufficiently large $n\in \N$, the above computation gives $\overline{x}_{g_nh}=d_n$ and $j_{g_nh}(d_n)\equiv_\Lambda \lambda$. Hence, if $\lambda\gneq_\Lambda 1$ and $h\in G$, we have  $g_nh\to+\infty$ as $n\to+\infty$ (with respect to the jump preorder), which implies that $\varphi(g_n)(\xi)\to+\infty$ for every $\xi\in\R$. Analogously, when $\lambda\lneq_\Lambda 1$, we get $\varphi(g_n)(\xi)\to-\infty$ for every $\xi\in\R$. 
As such qualitative properties of the action $\varphi$ are invariant under positive conjugacy, we deduce that the positive conjugacy class of $\varphi$ determines the preorder $\le_\Lambda$. This also shows that $\varphi$ is not conjugate to the standard action, since by construction all elements in the sequence  $(g_n)$ act trivially on the interval $(a, c)$. \end{proof}

\section{Groups with cyclic group of germs at one endpoint} \label{s-cyclic-germ}\label{sec:cyclic_germs1}
Throughout the section we write $X=(a, b)$, and let  $G\subset\homeo_0(X)$ be a locally moving subgroup such that  $\Germ(G,b)$ is infinite cyclic and acts freely near $b$ (that is, for every $g\in G$ whose projection to $\Germ(G,b)$ is non-trivial, there is an interval $(x,b)$ on which $g$ has no fixed point). We will say for short that $G$ has \emph{cyclic germs} at $b$. 

\begin{ex}
	One example of group with cyclic germs is  Thompson's group $F$, and more generally  any Higman--Thompson's group $F_n$ (see \S\ref{sc.BieriStrebel}). A much larger class of examples with this property is given by \emph{(pre)chain groups} in the sense of Kim, Koberda, and Lodha \cite{KKL}.
\end{ex}

For $G\subset \homeo_0(X)$ with cyclic germs, we  present a mechanism to  build a continuum of  pairwise non-conjugate, minimal exotic actions. For this, we identify $\Germ(G, b)$ with $\Z$ in such a way that any germ for which $b$ is an attractive fixed point is sent to a positive integer. We denote by $\tau\colon G\to \Z$ the homomorphism obtained via this  identification, and we fix an element $f_0\in G$ such that $\tau(f_0)=1$ (that is, the germ of $f_0$ generates $\Germ(G, b)$ and we have $f_0(x)>x$ near $b$). Choose next a bi-infinite sequence $\s=(s_n)_{n\in\Z}\subset X$ with 
\begin{equation}\label{e-conditions-sequence} s_{n+1}=f_0(s_n)\text{ for }n\in\Z,\quad \text{and}\quad \lim_{n\to+\infty}s_n=b.\end{equation} 
Consider the action of the group $G$ on the set of sequences $X^\Z$, where the action of $g\in G$ on a sequence $\mathsf{t}=(t_n)_{n\in \Z}$ is given by
\begin{equation}\label{e-action-sequences} g\cdot \mathsf{t}=\left (g(t_{n-\tau(g)} )\right )_{n\in \Z}.\end{equation}
It is straightforward to check that this defines an action of $G$ on $X^\Z$, using that $\tau$ is a homomorphism. We let $\mathsf{S}\subset X^\Z$ be the orbit of $\s$ under this action. Note that $\s$ is  fixed by $f_0$. 

\begin{lem}\label{l-cyclic-eventual}
With notation as above, for every sequence $\mathsf{t}=(t_n)_{n\in \Z}\in \mathsf{S}$, there exists $n_0\in \Z$ such that $t_n=s_n$ for every $n\ge n_0$.

\end{lem}
\begin{proof}
Let $g\in G$ be such that $\mathsf{t}=g\cdot \s$. Then $t_n=g (s_{n-\tau(g)} )$. As we required $\lim_{n\to+\infty} s_n=b$ in \eqref{e-conditions-sequence}, and $g$ coincides with  $f_0^{\tau(g)}$ on a neighborhood of $b$, the conclusion follows. 
\end{proof}

It follows from the lemma that for every two distinct sequences $\mathsf{t}=(t_n)$ and $\mathsf{t}'=(t'_n)$ in $\mathsf{S}$, the integer
\begin{equation}\label{eq.mts}
m(\mathsf{t}, \mathsf{t}')= \max\left \{n\in \Z: t_n\neq t'_n\right \}\end{equation}
is well defined and finite.
Thus we can introduce the total order relation $\prec$ on $\mathsf{S}$, given by $\mathsf{t}\prec \mathsf{t}'$ if and only if $t_m<t_m'$, with $m=m(\mathsf{t}, \mathsf{t}')$

\begin{lem}\label{l-construction-cyclic-order}
With notation as above, the total order $\prec$ on $\mathsf S$ is preserved by the action of $G$ on $\mathsf{S}$ defined by \eqref{e-action-sequences}. Moreover, the element $f_0$ acts as a homothety on $(\mathsf{S}, \prec)$ (in the sense of Definition \ref{dfn.homotype}), with fixed point $\s$.
\end{lem}
\begin{proof}
It is routine verification that the order $\prec$ is $G$-invariant. Let us check that $f_0$ is a homothety. We have already noticed that the sequence $\s$ is a fixed point for $f_0$. Fix sequences $\mathsf{t},\mathsf{t}'\in \mathsf S$ such that 
\begin{equation}\label{eq:choose_sequences}
	\s\prec \mathsf t\prec \mathsf t'.
\end{equation}
We need to show that there exists $n\in \Z$ such that $f_0^n\cdot \mathsf t\succ \mathsf t'$ (in fact, we will find some $n\ge 1$, showing that $f_0$ acts as an \emph{expanding} homothety). 
Write $m_0=m(\mathsf{t}, \s)$ and $m_1=m(\mathsf{t}', \s)$ and note that  condition \eqref{eq:choose_sequences} gives $m_0\le m_1$ and $t_{m_0}>s_{m_0}$. We claim that $n=m_1-m_0+1$ is fine for our purposes. For this, we compute directly:
\begin{align*}
	\left (f_0^n\cdot \mathsf t\right )_{m_1+1}&=f_0^n(t_{m_1+1-n})=f_0^n(t_{m_0})\\
	&>f_0^n(s_{m_0})=s_{m_0+n}=s_{m_1+1}=t'_{m_1+1},
\end{align*} 
while for every $m> m_1+1=n+m_0$ we have
\[
\left (f_0^n\cdot \mathsf t\right )_{m}=f_0^n(t_{m-n})=f_0^n(s_{m-n})=s_m=t'_m.\]
Thus $m(f_0^n\cdot\mathsf{t}, \mathsf{t}')=m_1+1$ and $f_0^n\cdot \mathsf t\succ \mathsf{t}'$, as desired. Similarly one argues for $\mathsf t'\prec \mathsf t\prec \s$.
\end{proof}

Assume now that $G$ is countable,  so that the set $\mathsf{S}$ is countable as well. Then we can consider the dynamical realization $\varphi_{\s}\colon G\to \homeo_0(\R)$  of the action of $G$ on $(\mathsf{S}, \prec)$.
\begin{prop}\label{prop:cyclic_germs_minimal_faithful}
For $X=(a,b)$, let $G\subset \homeo_0(X)$ be a countable locally moving subgroup with cyclic germs at $b$. 
For every sequence  $\s=(s_n)_{n\in \Z}$ as in \eqref{e-conditions-sequence}, the action $\varphi_{\s}\colon G\to \homeo_0(\R)$ constructed above is minimal and faithful.

Moreover if $\s'$ is another such sequence whose image is different from that of $\s$ (that is, if they are not the same after a shift of indices), then $\varphi_{\s}$ and $\varphi_{\s'}$ are not conjugate. In particular $G$ has uncountably many, pairwise non-conjugate, faithful minimal actions on the real line.
\end{prop}
\begin{proof}
The fact that $\varphi_{\s}$ is minimal follows from Lemma \ref{l-construction-cyclic-order} and Proposition \ref{p.minimalitycriteria} (the action on $\mathsf S$ is transitive, so it is enough to describe what happens at $\s$).
Let $\iota\colon (\mathsf{S}, \prec)\to \R$ be an equivariant good embedding associated with $\varphi_{\s}$. Since $f_0$ is a homothety on $(\mathsf{S}, \prec)$, it follows that its $\varphi_{\s}$-image is a homothety of $\R$ whose unique fixed point is $\iota(\s)$. In particular the stabilizer of this point inside $G_+=\ker{\tau}$, which after \eqref{e-action-sequences} coincides with the stabilizer of $\s$ for the natural diagonal action of $G_+$ on $X^\Z$, is a well-defined invariant of the conjugacy class of $\varphi_{\s}$.  Now, note that if $\s$ and $\s'$ are sequences with distinct images, using that $G$ is locally moving, it is not difficult to construct $g\in G_+$ such that $g\cdot \s= \s$ and $g\cdot \s'\neq \s'$, showing that $\varphi_{\s}$ and $\varphi_{\s'}$ are not conjugate.
\end{proof}

\chapter{Additional results}\label{ch-additional}

In this chapter we collect some further results and applications that follow from the methods and results in Part \ref{partI}.

In \S \ref{s-uncountable}, we use the results from Chapter \ref{s-lm-2} to provide a class of large (necessarily uncountable) locally moving subgroups of $\homeo_0(\R)$, which admit a unique irreducible action on $\R$ up to conjugacy (recovering in particular the results for $\homeo_c(\R)$ and $\Diff_c^r(\R)$ of Militon \cite{Militon}, and Chen and Mann \cite{ChenMann}). In \S \ref{s-circle}, we show how the methods from Chapter \ref{s-lm-2} provide a general rigidity result for actions of locally moving groups on the circle. In this case the situation is much simpler than in the case of the line, and exotic actions essentially do not exist. In \S \ref{ssc:Stein}, {we improve arguments from the work of Bonatti, Lodha, and the fourth author \cite{BLT} to obtain some non-smoothability results, that together with Theorem \ref{t-lm-C1}} give that certain groups of PL homeomorphisms of intervals (such as most of Thompson--Brown--Stein groups) do not admit any $C^r$ actions on closed intervals for $r>1$.

\section{Uncountable groups}\label{s-uncountable}

The class of locally moving groups contains several natural ``huge'' groups, such as the group $\homeo_0(\R)$ and $\Diff^r_0(\R)$, or the subgroups $\homeo_c(\R)$ and $\Diff_c^r(\R)$ of compactly supported elements. Actions of such groups on the line are well understood thanks to work of Militon \cite{Militon}, and of the recent work of Chen and Mann \cite{ChenMann}. In fact, such results fit in a program started by Ghys \cite{MR1115743}, asking when the group of all diffeomorphisms (or homeomorphisms) of a manifold may act on another manifold; in the recent years, very satisfactory results have been obtained, and we refer to the survey of Mann \cite{MannSurvey} for an overview. 

In this section we provide a rigidity criterion for locally moving groups whose standard action has uncountable orbits, and satisfies an additional condition.  We then explain how this criterion recovers some of the results in \cite{Militon,ChenMann}, and unifies them within the setting of the other results of this part of the paper. Our result requires the following relative version of a group property first considered by Schreier \cite[Problem 111]{MR3242261}, and studied by Le Roux and Mann \cite{LeRouxMann} in the setting of transformation groups (from whom we borrow the terminology).
\begin{dfn}
	For a group $G$ and a subgroup $H\subset G$,  we say that the pair $(G, H)$ has the \emph{relative Schreier property} if every countable subset of $H$ is contained in a finitely generated subgroup of $G$.
\end{dfn} 
When $G=H$, this is called the \emph{Schreier property} in \cite{LeRouxMann}.

\begin{thm}\label{t-uncountable}
	For $X=(a, b)$, let $G\subseteq \homeo_0(X)$ be a locally moving subgroup such that for every non-empty open subinterval $I\Subset X$, the following hold:
	\begin{itemize}
		\item all $G_I$-orbits in $I$ are uncountable;
		\item the pair $([G_c,G_c], [G_I,G_I])$ has the relative Schreier property.
	\end{itemize}
	Then,  every irreducible action $\varphi\colon G\to \homeo_0(\R)$ is either conjugate to the standard action of $G$ on $X$, or semi-conjugate to an action that factors through $G/[G_c, G_c]$. 
\end{thm}

The following special case is worth being pointed out.

\begin{cor}\label{c-uncountable}
	Let $G\subseteq \homeo_c(\R)$ be a perfect subgroup of compactly supported homeomorphisms. Suppose that for every bounded non-empty open interval $I\subset \R$ the following hold:
	\begin{itemize}
		\item the $G_I$-orbit of every $x\in I$ is uncountable; 
		\item the pair $(G, G_I)$ has relative Schreier property.
	\end{itemize}
	Then every irreducible action $\varphi\colon G\to \homeo_0(\R)$ is conjugate to the standard action of $G$.
\end{cor}

For the proof we need the following lemmas. 
\begin{lem}\label{l.huge_std}
	For $X=(a, b)$, let $G\subseteq \homeo_0(X)$ be a subgroup such that every $x\in X$ has an uncountable $G$-orbit. Then every action $\varphi\colon G\to \homeo_0(\R)$ which is semi-conjugate to the standard action on $X$, is conjugate to it. 
\end{lem}
\begin{proof}
	Note first that if orbits are uncountable, then the defining action on $X$ is minimal. Indeed, if $\Lambda \subset X$ is a closed invariant  subset, then the boundary $\partial \Lambda$ is countable and $G$-invariant, thus $\partial \Lambda=\varnothing$; hence $\Lambda\in \{\varnothing,X\}$, as wanted.
	Assume that  $\varphi\colon G\to \homeo_0(\R)$ is semi-conjugate to the standard action on $X$, by a monotone equivariant map $q\colon \R\to X$. As the action on $X$ is minimal, the semi-conjugacy $q$ is continuous. If it is not injective, there exist points $x\in X$ for which $q^{-1}(x)$ is a non-trivial interval. But the set of such points is $G$-invariant and at most countable, which is a contradiction. \qedhere \end{proof}
\begin{lem}\label{l-fix-countable}
	Let $G$ be a group of homeomorphisms of a second countable Hausdorff space. Then there exists a countable subgroup $H\subseteq G$ such that $\fix(G)=\fix(H)$. 
\end{lem}
\begin{proof}
	The statement is non-empty only when $G$ is uncountable, and we will assume so.
	Let $\mathcal{U}$ be a countable basis of open subsets of the space. For every $z\in \supp(G)$, we can find  a neighborhood  $U\in \mathcal{U}$ of $z$ and an element $g_U\in G$,  such that $g_U(U)\cap U=\varnothing$. Thus, we can cover $\supp(G)$ with countably many subsets with this property, and the subgroup $H$ generated by the corresponding $g_U\in G$ is countable, and satisfies the desired condition. 
\end{proof}

\begin{proof}[Proof of Theorem \ref{t-uncountable}]
	Assume by contradiction that $\varphi \colon G \to \homeo_0(\R)$ is an exotic action: after Lemma \ref{l.huge_std}, this means that $\varphi$ is not conjugate to the action on $X$, and $\varphi(N)$ has no fixed point (where $N=[G_c, G_c]$, as usual). 
	
	\setcounter{claimnum}{0}
	\begin{claimnum}\label{claim-accumulate}
		For every non-empty open subinterval $I\Subset X$, the subsets $\fixphi([G_I,G_I])$ and $\suppphi([G_I,G_I])$ accumulate at $\pm\infty$.
	\end{claimnum}
	\begin{proof}[Proof of claim]
		By Lemma \ref{l-fix-countable},  we can find a countable subgroup $H\subseteq [G_I, G_I]$ with $\fixphi\left ([G_I, G_I]\right )=\fixphi(H)$. By the relative Schreier property of $(N,[G_I,G_I])$, the subgroup $H$ is contained in a finitely generated subgroup $\Gamma\subset N$, to which we can apply Proposition \ref{p-lm-trichotomy}.\ref{i-exotic-totally-bounded} and deduce that $\fixphi(\Gamma)$ accumulates at $\pm\infty$. As $H\subset\Gamma$, the same holds for $\fixphi(H)=\fixphi([G_I,G_I])$ as desired. Finally, as $\varphi$ is exotic, Corollary \ref{cor.unique} implies that $\suppphi([G_I,G_I])$ accumulates at $\pm\infty$.
	\end{proof}
	
	Since $\varphi$ is exotic, Proposition \ref{p-lm-reconstruction} implies that, either $\fixphi(G_{(a,x)})$ or $\fixphi(G_{(x,b)})$ is empty for every $x\in X$. Assume that the latter case holds, the other case being analogous. Fix $I=(c,d)\Subset X$, and let $\mathcal{I}$ be the collection of connected components of $\suppphi\left ([G_I, G_I]\right )$. Notice that by Claim \ref{claim-accumulate}, all components in $\mathcal{I}$ are bounded, and their union accumulates at $\pm\infty$. Fix $L\in \mathcal{I}$. Then for every $x\in (d, b)$,  we have the inclusion $[G_I,G_I]\subseteq [G_{(c,x)},G_{(c,x)}]$, and therefore $L$ is contained in some connected component $L_x$ of $\suppphi\left ([G_{(c, x)}, G_{(c, x)}]\right )$.
	Consider then the function
	\[
	\dfcn{F_L}{(d,b)}{\R}{x}{\sup L_x,}
	\]
	which is well defined after Claim \ref{claim-accumulate}, and is monotone increasing. 
	
	\begin{claimnum}\label{claim-discontinuity} The function $F_L$ is not continuous.
	\end{claimnum}
	\begin{proof}[Proof of claim] Since we assumed that $\fixphi(G_{(c,b)})$ is empty, it must hold that $F_L$ tends to $+\infty$ as $x$ tends to $b$. Note that for every $x\in (d,b)$, the point $F_L(x)$ belongs to $\fixphi\left ([G_{(c, x)}, G_{(c, x)}]\right )$ and therefore also to $\fixphi([G_I, G_I])$. Since $F_L(x)$ tends to $+\infty$ as $x$ tends to $b$, and its image avoids $\suppphi([G_I,G_I])$ (that accumulates at $+\infty$ by Claim \ref{claim-accumulate}), we conclude that the image of $F_L$ is not an interval, so $F_L$ is not continuous.
	\end{proof}
	
	Notice now that since the subgroup $G_{(d, b)}$ centralizes $[G_I, G_I]$, it permutes the intervals in $\mathcal{I}$. Hence, for every $g\in G_{(d, b)}$, the family of functions $\left \{F_L: L\in \mathcal{I}\right \}$ is equivariant in the following sense:
	\[ g.F_L(x)= F_{g.L}(g(x)).\] 
	In particular, a point $x$ is a discontinuity point for $F_L$ if and only if $g(x)$ is a discontinuity point for $F_{g.L}$. Since all functions $F_L$ are monotone, and there are countably many of them, there are at most countably many points $x\in (d, b)$  which are discontinuity points  of some $F_L$, for $L\in \mathcal{I}$. Thus the $G_{(d, b)}$-orbit of every such point must be countable, contradicting our assumption. 
	\setcounter{claimnum}{0}
	\end{proof}

Let us now give some examples of groups that satisfy the hypotheses of Theorem \ref{t-uncountable} (or simply Corollary \ref{c-uncountable}). In several situations the relative Schreier property can be established using an embedding technique for countable groups due Neumann and Neumann \cite{NeumannNeumann}, based on unrestricted permutational wreath product. This method has been exploited by Le Roux and Mann \cite{LeRouxMann} to show that many homeomorphism and diffeomorphism groups of manifolds have the Schreier property. In order to run this method, it is enough that the group $G$ be closed under certain infinitary products, in the following sense.
\begin{dfn}
	For $X=(a, b)$, let $G\subseteq \homeo_0(X)$ be a subgroup. Let $(I_n)_{n\in  \N}$ be a collection of disjoint open subintervals of $X$. For every $n\in \N$, take an element $g_n\in G_{I_n}$. We denote by $\prod g_n$ the homeomorphism of $X$ defined by 
	\[ \prod g_n\colon x\mapsto \left\{\begin{array}{lr} g_n(x) & \text{if }x\in  I_n, \\[.5em] x &\text{if }x\notin \bigcup_n I_n.\end{array}\right.\]
	We say that the group $G$ is closed under \emph{monotone infinitary products} if for every monotone sequence $(I_n)_{n\in \N}$ of disjoint open subintervals (in the sense that the sequence $(\inf I_n)_{n\in \N}$ is monotone), and every choice of $g_n\in G_{I_n}$, we have $\prod g_n\in G$. 
	
\end{dfn} 
The following lemma provides a criterion for the relative Schreier property in our setting.
\begin{lem}\label{l-Schreier-homeo}
	For $X=(a, b)$, let $G\subseteq \homeo_0(X)$ be locally moving. Suppose that $G$ is closed under monotone infinitary products. Then, for every open subinterval $I\Subset X$, the pair $([G_c,G_c],[G_I,G_I])$	has the relative Schreier property. 
\end{lem}
\begin{proof}
	Let us write $N=[G_c,G_c]$ as usual, and fix $I=(x,y)$.	It is enough to show that for every sequence $(f_n)_{n\in \N}\subset G_I$, there exists a finitely generated subgroup $\Gamma \subset N$ such that $[f_n, f_m]\in \Gamma$ for all $n, m\in \N$. So let $(f_n)$ be such a sequence. Choose $t\in N$ such that $t(y)<x$, and	an increasing sequence of positive integers $(k_n)$ which is \emph{parallelogram free}, that is, if $k_{n_1}-k_{n_2}=k_{m_1}-k_{m_2}\neq 0$, then $n_1=m_1$ and $n_2=m_2$ (for example  the sequence $k_n=2^n$ has this property). Set $g_n=t^{k_n} f_n t^{-k_n}$, so  that $g_n\in G_{I_n}$, where we set $I_n:=t^{k_n}(I)$. Note that the choice of $t$ gives that the intervals $I_n$ are pairwise disjoint and form a monotone sequence, so the product $h:=\prod g_n$ is well defined, and belongs to $G$, by assumption. It is not difficult to check that for every $n\in \N$, the element $t^{-k_n}ht^{k_n}$ coincides with $f_n$ on $I$. Also, the fact that  $(k_n)$ is parallelogram free implies that for $n \neq m$, the intersection of the supports of $t^{-k_m}ht^{k_m}$ and $t^{-k_n}ht^{k_n}$ is contained in $I$. Thus, we have $[t^{-k_n} h t^{k_n}, t^{-k_m} ht^{k_m}]=[f
	_n, f_m]$. It follows that the finitely generated subgroup $\Gamma:=\langle h,  t\rangle \subset N$ gives the desired conclusion.
\end{proof}

Combining this with Theorem \ref{t-uncountable}, one can show that various sufficiently ``huge'' locally moving groups do not admit any exotic action at all. For example, we have the following criterion.
\begin{cor}\label{c-uncountable-complete}
	For $X=(a,b)$, let $G\subseteq\homeo_c(X)$ be a locally moving perfect subgroup of compactly supported homeomorphisms, closed under monotone infinitary products. Assume that for every non-empty open subinterval $I\subset X$, all $G_I$-orbits in $I$ are uncountable. Then,  every irreducible action $\varphi\colon G\to \homeo_0(\R)$  is conjugate to the standard action of $G$ on $X$. 
\end{cor}
This criterion is clearly satisfied by the group $G=\homeo_c(\R)$ of all compactly supported homeomorphisms of $\R$ (it is well known that it is perfect, see for instance Ghys \cite[Proposition 5.11]{Ghys} for a very short proof). Thus, Corollary \ref{c-uncountable-complete} recovers the following result due to Militon \cite{Militon}.
\begin{cor}[Militon]
	Every irreducible action $\varphi\colon \homeo_c(\R)\to \homeo(\R)$ is  conjugate to the standard action.
\end{cor}

Let us now consider the group $G=\Diff_c^r(\R)$ of compactly supported diffeomorphisms of $\R$ of class $C^r$, with $r\in [1, \infty]$. This case is more subtle since the group $G$ is not closed under monotone infinitary products: if $g_n\in G_{I_n}$, the element $\prod g_n$  need not be a diffeomorphism on a neighborhood of any accumulation point of the intervals $I_n$. A way to go around this is provided by the arguments of Le Roux and Mann  in \cite[\S 3]{LeRouxMann}, where they show that the group $\Diff^r(M)$ has the Schreier property for every closed manifold $M$, whenever $r\neq \dim (M)+1$. Note that this result does not apply to the group $\Diff_c^r(\R)$, which in fact does not have the Schreier property, since it can be written as a countable strictly increasing union of subgroups (see \cite{LeRouxMann}). However, the same argument of their proof can be adapted to show the following.
\begin{prop} \label{p-Schreier-diff}
	For every $r\in [1, \infty]\setminus \{2\}$ and interval $I\Subset \R$, the pair $(\Diff^r_c(\R), \Diff^r_c(I))$ has the relative Schreier property.
\end{prop}

\begin{rem}\label{rem:mather}
	The reason for the assumption $r \neq 2$ is that in this case the group $\Diff^r_c(\R)$ is known to be simple, by famous results of Thurston \cite{Thurston} (for $r=\infty$) and  Mather \cite{Mather} (for $r\neq 2$ finite). Whether this holds for $r=2$ remains an open question.  Actually, Mather has proved in \cite{MatherIII} that the group $\Diff^{1+bv}_c(\R)$ of compactly supported  diffeomorphisms whose derivative has bounded variation, is \emph{not} perfect, by constructing an explicit surjection to $\R$.
\end{rem}

\begin{proof}[Proof of Proposition \ref{p-Schreier-diff}]
	We outline the steps, and refer to \cite[\S 3]{LeRouxMann} for details. First of all, observe that in order to prove the proposition, it is enough to find a generating set $S$ of $\Diff^r_c(I)$ with the property that every sequence $(b_n)$ of elements of $S$ is contained in a finitely generated subgroup  of $\Diff^r_c(\R)$ (see Lemma 3.6 in \cite{LeRouxMann}). 
	Set $I=(x, y)$. Following the same strategy as in the proof of Proposition \ref{l-Schreier-homeo}, let $t\in \Diff^r_c(\R)$ be such that $t(y)<x$, and  choose a parallelogram-free increasing sequence $(k_n)\subset \Z_+$. The main difference with the proof of Proposition \ref{l-Schreier-homeo} is that if, we choose $g_n\in \Diff^r_c(I_n)$ arbitrarily, the element $\prod g_n$ does not necessarily belong to $\Diff^r_c(\R)$. However, if the elements $(g_n)$ are such that their $C^r$ norms satisfy $\left\lVert g_n \right\rVert_r \le 2^{-n},$ then the sequence of truncated products $\prod_{n=1}^m g_n$ is a Cauchy sequence, so that the infinite product $\prod g_n$ belongs to $\Diff^r_c(\R)$. Since conjugation by $t^{k_n}$ is continuous in the $C^r$ topology, this implies that there exists a sequence $(\varepsilon_n)$ such that whenever the elements $f_n\in \Diff^r_c(I)$ are such that $\left\lVert f_n\right\rVert_r\le \varepsilon_n$, then the product $\prod (t^{k_n} f_n t^{-k_n})$ is indeed in $\Diff^r_c(\R)$ (compare \cite[Lemma 3.7]{LeRouxMann}). 
	
	With these preliminary observations in mind, the key idea of \cite[\S 3]{LeRouxMann} is to consider a well-chosen generating set of $\Diff^r_c(I)$, consisting of elements belonging to suitable  copies of the {affine group}. By an affine group inside $\Diff_c^r(I)$ we mean a subgroup generated by two one-parameter subgroups $\left \{a^t\right \}_{t\in \R}$ and  $\left \{b^t\right \}_{t\in \R}$ of $\Diff^r_c(\R)$, varying continuously in the $C^r$ topology, which satisfy the relations $a^s b^t a^{-s}=b^{e^st}$. Existence of affine subgroups in $\Diff_c^r(I)$  can be obtained by a classical trick (attributed to \cite{Muller} and Tsuboi \cite{Tsuboi}) applied to the two vector fields generating the affine group (see \cite[Lemma 3.3]{LeRouxMann}). Let now $S \subset \Diff_c^r(I)$ be the set of all diffeomorphisms $b$  that can be expressed as time-one maps $b:=b^1$ for some flow  
	$\left\{b^t\right \}_{t\in \R}$ belonging to a pair of flows generating an affine subgroup. Since the set $S$ is non-empty and stable under conjugation in $\Diff_c^r(I)$, it is a generating set, by simplicity of  $\Diff_c^r(I)$. Now, let $(b_n)\subset S$ be a sequence, where each $b_n=b_n^1$ belongs to an affine subgroup $A_n:=\langle a_n^t, b_n^t \rangle$. The relations in the affine subgroups imply that for every $t, s\in \R$ we have $[a_n^s, b_n^t]=b_n^{(e^s-1)t}$. This equality implies that for $\delta>0$ small enough, the flow element $b_n^\delta$ can be written as a commutator of elements with arbitrarily small $C^r$ norm (see \cite[Corollary 3.4]{LeRouxMann}). Thus, for every $n\in \N$ we can choose $\delta_n=1/l_n$ for some sufficiently large positive integer $l_n>0$, such that we have  $c_n:=b_n^{\delta_n}=[f_{2n}, f_{2n+1}]$ for some sequence $(f_n)\subset \Diff^r_c(I)$ such that $\left\lVert f_n\right\rVert_r\le \varepsilon_n$. By the choice of the sequence $(\varepsilon_n)$ made above, the product $h:=\prod t^{k_n} f_n t^{-k_n}$ is in $\Diff^r_c(\R)$. The same argument in the proof of Lemma \ref{l-Schreier-homeo} then implies that $[t^{-k_2n} ht^{k_2n}, t^{-k_{2n+1}} ht^{k_{2n+1}}]=[f_{2n}, f_{2n+1}]=c_n$, so that the subgroup $\Gamma=\langle h, t\rangle$ contains the sequence $(c_n)$, and thus also contains the sequence $(b_n)$, since $b_n=c_n^{l_n}$. By the remark made at the beginning of the proof, this proves the proposition. \qedhere
	
\end{proof}
Combined with Proposition \ref{p-Schreier-diff}, Theorem \ref{t-uncountable} provides an alternative proof of the following recent result of Chen and Mann \cite{ChenMann}.
\begin{cor}[Chen--Mann]
	For $r\in [1,\infty]\setminus \{2\}$, every irreducible action $\varphi\colon \Diff^r_c(\R) \to \homeo_0(\R)$ is conjugate to its standard action. 
\end{cor}

\begin{rem}
	After Mather's result mentioned in Remark \ref{rem:mather}, the result above is false for the group $\Diff_c^{1+bv}(\R)$. Indeed, as $\Diff_c^{1+bv}(\R)$ surjects to $\R$, it admits an action by translations. One can even construct several faithful actions: fix an integer $d\ge 1$, let $N$ be the kernel of an epimorphism $\Diff_c^{1+bv}(\R)\to \mathbb{Q}^d$ (obtained by post-composing with some epimorphism $\R\to \mathbb{Q}^d$), and take a faithful action of $\mathbb{Q}^d$ by translations; then we can blow up the orbit of some point $x$, for this translation action, and insert the standard action of $N$ in the preimage of $x$, and extend it to the other preimages in an equivariant way.
\end{rem}

\section{A result for actions on the circle} \label{s-circle}

In this section we discuss actions of locally moving groups on the circle. This setting turns out to be  much simpler than what studied so far, essentially due to the compactness of $\mathbb{S}^1$.  In fact, we can actually prove a result for actions on $\mathbb{S}^1$ of a group   of homeomorphisms $G\subseteq\homeo(X)$ where $X$ is an arbitrary locally compact space, provided the action of $G$ on $X$ satisfies suitable dynamical conditions. 

Given a group $G$ of homeomorphisms of a space $X$, and an open subset $U\subset X$, we let $G_U$ be the subgroup of elements pointwise fixing the complement $X\setminus U$. 
Similarly to Definition \ref{d.lm}, we say that a subgroup $G\subseteq \homeo(X)$ is \emph{micro-supported} if for every non-empty subset $U\subset X$, the subgroup $G_U$ is non-trivial. The action of $G$ on $X$ is \emph{extremely proximal} if for every compact subset $K\subsetneq X$ there exists $y\in X$ such that for every open neighborhood $V$ of $y$, there exists $g\in G$ with $g(K)\subset V$. 

When $G\subseteq \homeo(X)$ is a micro-supported group, we denote by $M_G\subset G$ the subgroup of $G$ generated by the subgroups $[G_U, G_U]$, where $U$ varies over relatively compact non-dense open subsets of $X$. Note that $M_G$ is non-trivial and normal in $G$. 
In fact, standard arguments similar to the proof of Proposition \ref{p-micro-normal} imply the following (see e.g.\ Le Boudec \cite[Proposition 4.6]{LB-lattice-embeddings}). 

\begin{prop}
	Let $X$ be a locally compact Hausdorff space, and $G\subset\homeo(X)$  a micro-supported group acting minimally and extremely proximally on $X$. Then, every non-trivial normal subgroup of $G$ contains $M_G$.  Thus, $G$ is simple when $G=M_G$, and $G/M_G$ is the largest non-trivial quotient of $G$ otherwise.
\end{prop}

For $x\in X$, we also denote by $G^0_x$ the subgroup of elements that fix pointwise a neighborhood of $x$, and call it the \emph{germ-stabilizer} of $x$.
Moreover, we say that $G$ has the \emph{independence property for pairs of germs} if for every distinct $x_1, x_2\in X$, and elements $g_1, g_2\in G$ such that $g_1(x_1)\neq g_2(x_2)$, there exist $g\in G$ and open neighborhoods $U_i\ni x_i$ such that $g$ coincides with $g_i$ in restriction to $U_i$, for $i\in \{1, 2\}$.

\begin{thm} \label{t-circle} Let $X$ be a locally compact Hausdorff space, and let $G\subseteq \homeo(X)$ be a micro-supported subgroup of homeomorphisms of  $X$ satisfying the following conditions:
	\begin{itemize}
		\item  the action  of $G$ on $X$ is extremely proximal;
		\item   for every $x\in X$, the germ-stabilizer $ G^0_x$ acts minimally on $X\setminus\{x\}$.	\end{itemize}
	Assume that  $\varphi\colon G\to \homeo_0(\T)$ is a faithful minimal action. 
	Then, there exists a continuous surjective map $\pi \colon \T \to X$, which is $G$-equivariant with respect to the $\varphi$-action on $\T$ and the natural action on $X$. Moreover, if $G$ has the independence property for pairs of germs, then  $X$ is homeomorphic to $\T$, and the map $\pi\colon \T \to X$ is a covering map. 
\end{thm}
Before discussing the proof we give some comment on the statement.
\begin{rem}
	The condition that $\varphi$ be minimal is not so restrictive, since every group action on $\T$ either has a finite orbit (and thus is semi-conjugate to an action factoring through a finite cyclic group), or is semi-conjugate to a minimal action. The condition that $\varphi$ is faithful cannot be avoided, since in this generality little can be said about actions of the largest quotient $G/M_G$  (which could even be a non-abelian free group, see Le Boudec \cite[Proposition 6.11]{LB-lattice-embeddings}). However, when $G=M_G$, the theorem implies that every non-trivial action $\varphi\colon G\to \homeo_0(\T)$ factors onto its standard action on $X$. 
\end{rem}
\begin{rem}
	It is likely that the assumptions on $G$ in Theorem \ref{t-circle} may be relaxed or modified, and we did not attempt to identify the optimal ones. In particular, we do not know whether the assumption that the action of $G$ has independent germs is needed  in the last statement. However, note that even with this assumption, the map $\pi$ may be a non-trivial self-cover of $\T$ (thus it is not necessarily a semi-conjugacy to the standard action). Indeed the groups of homeomorphisms of the circle constructed by Hyde, Lodha, and the third named author in \cite[\S 3]{HLR} satisfy all assumptions in Theorem \ref{t-circle}, and their action lifts to an action on the universal cover $\R\to \T$; in particular, it also lifts to an action under all self-coverings of $\T$.
\end{rem}

The proof of Theorem \ref{t-circle} follows an approach similar to the proofs of its special cases that appeared in the work of Le Boudec and the second name author \cite[Theorem 4.17]{LBMB-sub} for Thompson's group $T$, and in that of two of the authors \cite[Theorem D]{MatteBonTriestino} for the groups $\mathsf{T}(\varphi)$ of piecewise linear homeomorphisms of suspension flows defined there. The main difference is that  some arguments there make crucial use on specific properties of those groups, such as the absence of free subgroups in the group of piecewise linear homeomorphisms of an interval, while here we get rid of these arguments using Proposition \ref{p-centralizer-fix}, and this allows for a generalization to a much broader class of groups.

\begin{proof}[Proof of Theorem \ref{t-circle}]
	Assume that $\varphi\colon G\to \homeo_0(\T)$ is a faithful minimal action. First of all, note that there is no loss of generality in supposing that $\varphi$ is proximal.  Indeed, since $G$ is non-abelian,  by Theorem \ref{t-Margulis} every minimal faithful action of $\varphi\colon G\to \homeo_0(\mathbb{S}^1)$ has finite centralizer $C_\varphi$, and the action $\varphi$ descends via the quotient map $\mathbb{S}^1\to \mathbb{S}^1/C_\varphi\cong \mathbb{S}^1$ to a proximal action $\varphi_{p}$. Thus, by replacing $\varphi$ with $\varphi_{p},$ we can assume that $\varphi$ is proximal. {Note that the group $G$ is center-free (for instance because of extreme proximality of its action on $X$, see  Remark \ref{rem.obst_proximal}), so the centralizer $C_\varphi$ has trivial intersection with $\varphi(G)$.}	
	
	Given $x\in X$ we will denote by $K_x$ the subgroup of all elements whose support is a relatively compact subset of $X\setminus \{x\}$. That is, $K_x$ is the union of the subgroups $G_U$, where $U$ varies over all open subsets $U\Subset X\setminus\{x\}$. Note that $K_x$ is a normal subgroup of $G^0_x$.
	
	By extreme proximality and minimality of the action, for every open $U\Subset X$, there exists $g\in G$ such that $g(\overline{U})\cap \overline{U}=\varnothing$. Hence, $G_U$ is conjugate to a subgroup of its centralizer. Thus, by Proposition \ref{p-centralizer-fix},  we have $\fixphi([G_U, G_U])\neq \varnothing$. By compactness of $\T$, we deduce that for every $x\in X$ we have
	\[\fixphi([K_x, K_x])=\textstyle\bigcap_{ U\Subset X\setminus\{x\}}\fixphi([G_U, G_U])\neq \varnothing.\]
	Note also that $\fixphi([K_x, K_x])\neq \T$, otherwise the subgroup $[K_x, K_x]$ would act trivially, contradicting that the action is faithful. Now, fix $x\in X$, and consider the subset $C:=\fixphi([K_x, K_x])$, which is $\varphi(G_x^0)$-invariant (by normality of $[K_x, K_x]$ in $G_x^0$). Since we are assuming that $\varphi$ is proximal, we can find a sequence $(g_n)$ in $G$ such that $g_n.C$ tends to a point $\xi \in \T$ in the Hausdorff topology. Upon extracting a subnet from $(g_n)$, we can suppose that $g_n(x)$ tends to a limit $y$ in the one-point compactification $\hat{X}:=X\cup \{\infty_X\}$.
	\setcounter{claimnum}{0}
		\begin{claimnum}\label{claim:microsupp_circle1}
			We have $y\in X$, and $\varphi(G^0_y)$ fixes $\xi$.
		\end{claimnum}
		\begin{proof}[Proof of claim]
			Suppose that the limit $y$ is $\infty_X$, that is, that $(g_n(x))$ escapes from every compact subset of $X$. Then, every element $g\in G_c$ belongs to $G^0_{g_n(x)}$ for $n$ large enough, so that $\varphi(g)$ preserves $g_n.C$ for every $n$ large enough, and thus fixes the point $\xi$. Since $g\in G_c$ is arbitrary, we deduce that $\varphi(G_c)$ fixes $\xi$, and by normality of $G_c$ in $G$ and minimality of $\varphi$, we deduce that $\varphi(G_c)=\{\id\}$, contradicting that the action is faithful. Thus, the limit of $(g_n(x))$ is a point $y\in X$. 
			Similarly we prove that $\varphi(G^{0}_y)$ fixes $\xi$. Fix $g\in G^0_y$. Since $g_n(x)\to y$, then $g$ also belongs to $G^0_{g_n(x)}$, and thus preserves $g_n.C$, for $n$ large enough. Since $g_n.C$ tends to $\xi$, it follows that $\varphi(g)$ fixes $\xi$. Hence $\varphi(G^0_y)$ fixes $\xi$, since $g  \in G^0_y$ is arbitrary.
		\end{proof}
	
	Once the existence of such points $y\in X$ and  $\xi\in\T$ has been proven,  the rest of the proof is  essentially  the same as in \cite{LBMB-sub} or \cite{MatteBonTriestino}, but we outline a self-contained argument for completeness. We claim that for every $\zeta\in \T$, there exists a unique point $\pi(\zeta)\in X$ such that $\varphi(G^{0}_{\pi(\zeta)})$ fixes $\zeta$, and that the map $\pi\colon \T\to X$ defined in this way is continuous.
	
	Let us first show that if such a point exists, it must be unique. Namely, assume that $x_1, x_2\in X$ are distinct points such that $\varphi(G^0_{x_i})$ fixes $\zeta$ for $i\in \{1, 2\}$, so that the $\varphi$-image of $H:=\langle G_{x_1}^0, G_{x_2}^0\rangle$ fixes $\xi$. Let $U\subset X$ be any non-dense open subset of $X$. After the minimality assumption of the action of $G_{x_1}^0$ on $X\setminus \{x_1\}$, we can find $g\in G_{x_1}^0$ such that $g(x_2)\notin \overline{U}$, so that $g^{-1}G_Ug=G_{g^{-1}(U)} \subset G_{x_2}^0\subset H$. Since we also have $g\in H$, we obtain that $G_U\subset H$. Thus, $H$ contains the non-trivial normal subgroup $N$ of $G$ generated by the $G_U$, where $U$ varies over all non-dense open subsets of $X$. Then $\varphi(N)$ fixes $\zeta$, which gives a contradiction using again minimality and normality of $N$.
	
	To show existence and continuity, one first checks the following fact.
		\begin{claimnum}\label{claim:microsupp_circle2}
			If $(z_i)$  is a net of points in $X$ converging to a limit $z\in \hat{X}=X\cup \{\infty_X\}$, and if $(\zeta_i)$ is a net of points in $\T$ converging to some limit $\zeta\in \T$ such that $\varphi(G_{z_i}^0)$ fixes $\zeta_i$  for every  $i$, then $z\in X$ and $\varphi(G^0_z)$ fixes $\zeta$.
		\end{claimnum}
		\begin{proof}[Proof of claim]
			The proof is similar to the one for Claim \ref{claim:microsupp_circle1}.
	\end{proof}
	
	Now, let $\zeta\in \T$ be arbitrary, and  choose $y\in X$ and $\xi\in \T$ such that $\varphi(G^0_y)$ fixes $\xi$ (whose existence has already been proven). By minimality of $\varphi$, we can find a sequence $(g_i)$ in $G$ such that  $\zeta_i:=g_i.\zeta$ converges to $\zeta$, and upon extracting a subnet, we can suppose that $y_i:=g_i(y)$ converges to some $z\in \hat{X}$. Then, by Claim \ref{claim:microsupp_circle2}, we have that the limit $z$ is actually in $X$, and $\varphi(G^0_z)$ fixes $\zeta$. Hence, we can define the map $\pi:\T\to X$, by setting $\pi(\zeta):=z$. Claim \ref{claim:microsupp_circle2} also gives continuity of $\pi$. The map is clearly $G$-equivariant, so its image $\pi(\T)$ is a compact $G$-invariant subset of $X$. Therefore, by minimality of the standard action of $G$ on $X$, the map $\pi$ must be surjective.
	
	Suppose now that $G$ has the independence property for pairs of germs, and let us show that the map $\pi$ must be injective (and thus a homeomorphism). As a preliminary observation, note that for every $x\in X$ the fiber $\pi^{-1}(x)$ must have empty interior, since otherwise by $G$-equivariance the interior of $\pi^{-1}(x)$ would be a wandering interval in $\T$, contradicting minimality. 
	Assume by contradiction that there exist  $\xi_1\neq \xi_2$ in $\T$ such that $\pi(\xi_1)=\pi(\xi_2)=:x$, and let $I:=(\xi_1, \xi_2)$ be the arc  between them (with respect to the clockwise orientation of $\T$). Since $I\not\subset \pi^{-1}(x)$, we can choose $\zeta\in I$ with  $z:=\pi(\zeta)\neq x$. By minimality of $\varphi$, there exists $g\in G$ such that $g^{-1}.\zeta\notin I$.
	
	Assume first that $g(x)\neq z$. Using the independence property for pairs of germs, we can find $h\in G$ which coincides with the identity on some neighborhood of $z$, and with $g$ on some neighborhood of $x$. On the one hand we have $h\in G^0_z$, so $h.\zeta=\zeta$. On the other hand, we have $g^{-1}h\in G^0_x$, so that $\varphi(g^{-1}h)$ fixes $\xi_1$ and $\xi_2$, and preserves the arc $I$.This gives $g^{-1}.\xi=g^{-1}h.\xi\in I$, contradicting the choice of $g$.
	
	Assume now that $g(x)=z$. In this case, choose an open subarc $J\subset I$ containing $\zeta$, and such that $g.J\cap I=\varnothing$. Since $\pi^{-1}(z)$ has empty interior, we can find a point $\zeta'\in J$ such that the point $z':=\pi(\zeta')$ is different from $z$. Then we have $g(y)\neq z'$, so we can repeat the previous reasoning using the points $\zeta', z'$ instead of $\zeta, z$. This provides the desired contradiction and shows that the map $\pi$ is injective. \qedhere
\end{proof}
We can deduce a rigidity result for groups that are given by a locally moving action on the circle (in the sense that for {every interval} $I\subset \T$, the subgroup $G_I$ acts on $I$ without fixed points).
\begin{cor}\label{c-circle-circle}
	For $X=\T$, let  $G\subseteq \homeo_0(X)$ be locally moving. Then, for every faithful minimal action $\varphi\colon G\to \homeo_0(\T)$, there exists a continuous surjective equivariant  map $\pi\colon \T\to X$. Moreover, if $G$ has the independence property for pairs of germs, then $\pi$ is a covering map. 
\end{cor}

\begin{ex}
	Let us explain how Corollary \ref{c-circle-circle} recovers the result of Matsumoto that every non-trivial action of $G=\homeo_0(\T)$ on $\T$ is conjugate to its standard action. Note that $G$ clearly satisfies all assumptions of Corollary \ref{c-circle-circle}, and has the independence property for pairs of germs. Since $G$ has elements of finite order, it cannot act non-trivially on the circle with a fixed point, and since it is simple, it cannot act with a finite orbit; hence, every non-trivial action $\varphi\colon G\to \homeo_0(\T)$ must be semi-conjugate to a minimal action $\varphi_{min}$ (which is automatically faithful). Corollary \ref{c-circle-circle} then shows that $\varphi_{min}$ is the lift of its natural action via a self-cover, and again it is not difficult to see (using elements conjugate to rotations) that this is possible only if $\varphi_{min}$ is conjugate to the standard action, and in particular all its orbits are uncountable. If $\varphi$ had an exceptional minimal set $\Lambda \subsetneq \T$, then every connected component of the complement would be mapped to a point with a countable orbit, and this is not possible. We deduce that $\varphi$ is minimal, thus conjugate to $\varphi_{min}$, and thus to the standard action. 
\end{ex}

\begin{ex}
	In a similar fashion, Corollary \ref{c-circle-circle} recovers the result of Ghys \cite{Ghys} that every action of Thompson's group $T$ on the circle is semi-conjugate to its standard action. Indeed $T$ satisfies all assumptions in Corollary \ref{c-circle-circle}. As above, one uses its simplicity to show that every action $\varphi\colon T\to \homeo_0(\T)$ is semi-conjugate to a faithful minimal action, and thus by the corollary, to the lift of the standard action through a self-cover $\pi \colon \T\to \T$. One then argues that this is possible only if $\pi$ is a homeomorphism. 
	
\end{ex}
By considering other kind of spaces, Theorem \ref{t-circle} can be used to construct groups that cannot admit interesting actions on the circle. 
\begin{cor} \label{c-circle-fixed}
	Let $X$ be a locally compact Hausdorff space which is not a continuous image of $\T$ (e.g.\ if $X$ is not compact, or not path-connected), and let $G\subseteq \homeo(X)$ be a subgroup as in Theorem \ref{t-circle}. Then $G$ has no faithful minimal action on $\T$. In particular, if $G$ is simple, every action $\varphi\colon G\to \homeo_0(\T)$ has a fixed point. 
\end{cor}
\begin{ex}Examples of groups to which Corollary \ref{c-circle-fixed} applies are the groups of piecewise linear homeomorphisms of flows $\mathsf{T}(\varphi)$ from \cite{MatteBonTriestino}. For every homeomorphism $\varphi$ of the Cantor set $X$, the group $\mathsf{T}(\varphi)$ is a group of homeomorphisms of the mapping torus $Y^\varphi$ of $(X, \varphi)$ defined analogously to Thompson's group $T$ (see \cite{MatteBonTriestino} for details). When $\varphi$ is a minimal homeomorphism, the group $\mathsf{T}(\varphi)$ is simple \cite[Theorem B]{MatteBonTriestino} and its action  on $Y^\varphi$ satisfies all assumptions in Theorem \ref{t-circle}. Since the space $Y^\varphi$ is not path-connected, Corollary  \ref{c-circle-fixed} recovers the fact that every action of  the group $\mathsf{T}(\varphi)$ on the circle has a fixed point (see  \cite[Theorem D]{MatteBonTriestino}). Moreover, it allows to extend the conclusion to many groups defined similarly but not by PL homeomorphisms, for instance any simple overgroup of $\mathsf{T}(\varphi)$ in $\homeo(Y^\varphi)$ (see Darbinyan and Steenbock \cite{darbinyan2020embeddings} for a vast family of such groups).
\end{ex}

\section{An application to non-smoothability}\label{ssc:Stein}

In Theorem \ref{t-lm-C1} about rigidity of $C^1$ actions, it may happen that the standard action of $G$ is not semi-conjugate to any action of a given regularity, so that the first possibility is not realizable for actions in that regularity. Here we discuss two applications of this to certain groups of piecewise linear homeomorphisms, which improve results on non-smoothability of such groups of Bonatti, Lodha and the fourth author \cite{BLT}. 

\subsection{Thompson--Brown--Stein groups} Here we study differentiable actions of the Thompson--Brown--Stein groups $F_{n_1,\ldots,n_k}$\index{groups!Thompson--Brown--Stein groups} introduced in Definition \ref{d-Thompson-Stein}, which are natural generalizations of Thompson's group $F$.  Such groups are clearly locally moving. It was shown in \cite{BLT} that when $k\geq 2$, the standard action of $F_{n_1,\ldots,n_k}$ cannot be conjugate to any $C^2$ action. Here we show the following. 

\begin{thm}\label{t-Thompson-Stein}
	Let $r>1$. For any $k\ge 2$ and choice of $n_1,\ldots,n_k$ as in Definition \ref{d-Thompson-Stein}, the Thompson--Brown--Stein  group $F_{n_1,\ldots,n_k}$  admits no faithful $C^{r}$ action on the real line.
\end{thm}

We first need a lemma for the usual Higman--Thompson groups $F_n$. Corollary \ref{c-lm-C1-interval} applies to them, so every faithful action of $F_n$ of class $C^1$ is semi-conjugate to its standard action on the interval $(0,1)$. However, as pointed out by Ghys and Sergiescu \cite{GhysSergiescu}, the group $F$ (and every $F_n$) actually admits $C^1$ (even $C^\infty$) actions which are not conjugate to the standard action (see Remark \ref{r-F-C1}). In these examples, the action is obtained by blowing up the orbit of dyadic rationals.
The next lemma, which is a consequence of Theorem \ref{p.sacksteder} and the structure of the group, gives a restriction on possible semi-conjugate but not conjugate $C^1$ actions.

\begin{lem}\label{l-semic-F-C1}
	For $n\ge 2$, let $\varphi\colon F_n\to \Diff_0^1(\R)$ be a faithful action which is semi-conjugate to the standard action $\varphi_0\colon F_n\to \PL((0,1))$, but not conjugate. Let $h\colon \R\to (0,1)$ be the corresponding continuous monotone map such that $h\varphi=\varphi_0 h$. Then there exist a rational point $p\in (0,1)$ which is not $n$-adic (i.e.\ $p\in (\mathbb Q\setminus \Z[1/n])\cap (0,1)$), such that the preimage $\xi=h^{-1}(p)$ is a singleton, and an element $g\in F_n$ for which $\xi$ is a hyperbolic fixed point.
\end{lem}

\begin{proof}
	The proof is a tricky refinement of arguments in \cite[\S 5.1]{BLT}.
	In the following, we write $G=F_n$. Given an action $\varphi$ as in the statement, denote by $\Lambda\subset \R$ the corresponding minimal invariant Cantor set. Fix a non-empty open subinterval $I\Subset (0,1)$ with $n$-adic endpoints. Then $\varphi(G_I)$ preserves the interval $h^{-1}(I)$ and $\Lambda_I:=\Lambda\cap h^{-1}(I)$ is the minimal invariant subset for the restriction $\varphi_I\colon G_I\to \Diff_0^1\left (h^{-1}(\overline I)\right )$ induced by $\varphi$. 
	
	After Theorem \ref{p.sacksteder} and subsequent Remark \ref{r.sacksteder} applied to $\varphi_I$, there exist an element $g\in G_I$ and a point $\xi\in \Lambda_I$, such that $g.\xi=\xi$ and $\varphi(g)'(\xi)<1$. Moreover, such a point $\xi$ cannot belong to the closure of any gap of $\Lambda_I$ (see Remark \ref{r.sacksteder}), or in other terms the semi-conjugacy $h$ must be injective at $\xi$.	It is well known that if a point $p\in I$ is an isolated fixed point for some element of $G_I$ in the standard action, then $p$ is rational (the point $p$ must satisfy a rational equation $n^kp+\frac{a}{n^b}=p$; see Lemma \ref{l:flt} below). From this we deduce that the point $p=h(\xi)$ is rational. Moreover, the point $p$ cannot be $n$-adic: take any element $f\in G_I$ such that $\varphi_0(f)$ coincides with $\varphi_0(g)$ in restriction to $[p,1)$ and is the identity in restriction to $(0,p]$. Then the right derivative of $\varphi(f)$ at $\xi$ must be equal to $\varphi(g)'(\xi)<1$, and the left derivative of $\varphi(f)$ at $\xi$ must be equal to $1$, contradicting the fact that $\varphi$ is a $C^1$ action. This concludes the proof.
\end{proof}

As a consequence of Lemma \ref{l-semic-F-C1}, we get a strong improvement of \cite[Theorem 3.4]{BLT}, on regularity of actions of Thompson--Brown--Stein groups. 
The idea is to replace the use of the Szekeres vector field (which requires $C^2$ regularity), with Sternberg's linearization theorem, which works in $C^{r}$ regularity ($r>1$), but requires hyperbolicity (granted from Lemma \ref{l-semic-F-C1}).\footnote{Note however that Sternberg's linearization theorem cannot be extended in general to $C^1$ regularity, even under hyperbolicity assumptions (and thus our proof cannot be extended to $C^1$ regularity); see the recent work of Eynard-Bontemps and Navas \cite{EN}.}
In this form, these results can be found in Yoccoz \cite[Appendice 4]{Yoccoz} or Navas \cite[Theorems 3.6.2 and 4.1.11]{Navas-book} (a detailed proof when $r<2$ appears in the work of the fourth author \cite[\S 6.2.1]{Triestino}). A similar approach, although less technical, appears in Mann and Wolff \cite{MannWolff} to exhibit examples of groups at ``critical regularity''.

\begin{thm}\label{t-sternberg0}
	Fix $r>1$, and let $f$ be a $C^r$ diffeomorphism of the half-open interval $[0,1)$ with no fixed point in $(0,1)$, and such that $f'(0)\neq 1$. Then there exists a diffeomorphism $h\colon [0,1)\to [0,+\infty)$ of class $C^r$ such that
	\begin{itemize}
		\item $h'(0)=1$,
		\item the conjugate map $hfh^{-1}$ is the scalar multiplication by $f'(0)$.
	\end{itemize}
\end{thm}

\begin{thm}\label{t-sternberg}
	Fix $r>1$, and let $f$ be a $C^r$ diffeomorphism of the half-open interval $[0,1)$ with no fixed point in $(0,1)$, and such that $f'(0)\neq 1$. Then there exists a unique $C^{r-1}$ vector field $\mathcal X$ on $[0,1)$, with no singularities on $(0,1)$, such that
	\begin{itemize}
		\item $f$ is the time-1 map of the flow $\left \{\phi_{\mathcal X}^s\right \}$ generated by $\mathcal X$,
		\item the flow $\left \{\phi_{\mathcal X}^s\right \}$ coincides with the $C^1$ centralizer of $f$ in $\Diff_0^1([0,1))$.
	\end{itemize}
\end{thm}

The vector field from Theorem \ref{t-sternberg} is called the \emph{Szekeres vector field}.
The following statement is the analogue of \cite[Proposition 7.2]{BLT}.

\begin{prop}\label{p-local-flow}
	Fix $a\in (0,1)$ and $r>1$. Assume that two homeomorphisms $f,g\in \homeo_{0}([0,1))$ satisfy the following properties.
	\begin{itemize}
		\item The restrictions of $f$ and $g$ to $[0,a]$ are $C^2$ contractions, namely the restrictions are $C^2$ diffeomorphisms onto their images such that
		\[
		f(x)<x\quad\text{and}\quad g(x)<x\quad\text{for every }x\in (0,a].
		\]
		\item $f$ and $g$ commute in restriction to $[0,a]$, that is,
		\[
		fg(x)=gf(x)\quad\text{for every }x\in [0,a].
		\]
		\item The $C^2$ germs of $f$ and $g$ at $0$ generate an abelian free group of rank 2.
	\end{itemize}
	Then, for every homeomorphism $\psi \in \homeo_{0}([0,1))$ such that
	\begin{itemize}
		\item $\psi f\psi^{-1}$ and $\psi g \psi^{-1}$ are $C^r$ in restriction to $[0,\psi(a)]$,
		\item $(\psi f\psi^{-1})'(0)<1$,
	\end{itemize}
	one has that the restriction of $\psi$ to $(0,a]$ is $C^r$.
\end{prop}

\begin{proof}[Sketch of proof]
	The proof is basically the same as in \cite{BLT}, and we only give a sketch. As $f$ and $g$ are $C^2$ contractions near $0$ and commute, they can be simultaneously linearized by considering the Szekeres vector field $\mathcal X$ for $f$ (given by Theorem \ref{t-sternberg}). As their germs  generate a rank 2 abelian group, we can find a dense subset of times $A\subset \R$, such that for every $\alpha\in A$, there exists an element $h_\alpha\in \langle f,g\rangle$ such that the restriction of $h_\alpha$ to $[0,a]$ coincides with the time-$\alpha$ map of the flow $\phi_{\mathcal X}^\alpha$. Given a map $\psi$ as in the statement, we can also simultaneously linearize  $\psi f\psi^{-1}$ and $\psi g \psi^{-1}$, using Theorem \ref{t-sternberg0}. If $\mathcal Y$ is the corresponding vector field from Theorem \ref{t-sternberg}, we deduce that the restriction of $\psi$ to $(0,a]$ is $C^1$ and sends one vector field to the other: $\left (\psi_{[0,p]}\right )_*\mathcal X=\mathcal Y$. Writing this relation more explicitly, we get
	\[
	\psi'(x)=\frac{\mathcal Y(\psi(x))}{\mathcal X(x)}\quad\text{for every }x\in (0,a],
	\]
	whence we deduce that $\psi$ is $C^r$ in restriction to $(0,a]$.
\end{proof}

The next technical result is an adaptation of classical arguments in one-dimensional dynamics, which can be traced  back to Hector and Ilyashenko (see specifically \cite[Proposition 3.5]{Navas-QS} and \cite[Lemma 3]{Ilyashenko}).

\begin{prop}\label{prop.stabilizer_dense}
	For $r>1$, let $f,g\in \Diff_0^r([0,1))$ be two diffeomorphisms with the following properties:
	\begin{enumerate}[label=(\roman*)]
		\item $f'(0)=\lambda<1$ and $g'(0)=\mu\ge 1$,
		\item\label{i.ilyashenko} for every $(l,m)\in \N^2\setminus \{(0,0)\}$, there exists $\varepsilon>0$ such that $g^mf^l(x)\neq x$ for every $x\in (0,\varepsilon)$.
	\end{enumerate}
	Then, there exists $\delta>0$ such that the $\langle f,g\rangle$-orbit of every point $x\in (0,\delta)$ is dense in $(0,\delta)$.
\end{prop}

\begin{proof}
	By Theorem \ref{t-sternberg0}, we can take a $C^r$ coordinate $h\colon U\to [0,+\infty)$ on a open neighborhood $U\subset [0,1)$ of $0$, so that the map $f$ becomes the scalar multiplication by $\lambda$ on $[0,+\infty)$ (more precisely, we take as $U\subset [0,1)$ the maximal interval containing no fixed points for $f$, except $0$). Write $V=h(U\cap g^{-1}(U))$,  $\overline f=h f h^{-1}$ and $\overline g = h g h^{-1}\restriction_{V}$. Note that $\overline g'(0)=g'(0)=\mu$.
	
	We first rule out the case where $\log \lambda$ and $\log \mu$ are rationally dependent. So take $(l,m)\in \N^2\setminus \{(0,0)\}$ such that $\lambda^l\mu^m=1$, and consider the composition $\gamma:=g^m\circ f^l$, which satisfies $\gamma'(0)=1$, and  write $\overline \gamma= \overline g^m\circ \overline{f}^l$ for the corresponding map defined on an appropriate open subinterval $V'\subset V$ containing $0$.
	Then, for every $x\in V'$ and $n\in \N$, we have
	\[\overline f^{-n}\overline \gamma \overline f^{n}(x)=\frac{\overline \gamma(\lambda^nx)}{\lambda^n},\]
	from which we deduce that $\overline f^{-n}\overline \gamma \overline f^{n}\to \id$ uniformly on compact subsets of $V'$, as $n\to \infty$. Going back to the original coordinate, we get that there exists $\delta>0$ such that $f^{-n}\gamma f^{n}\to \id$ uniformly on $[0,\delta]$ as $n\to \infty$.
	Take now $x\in (0,\delta)$, and let $K$ be the closure of the $\langle f,g\rangle$-orbit of $x$, which clearly contains the point $0$. Assume by contradiction that $K\cap [0,\delta]\neq [0,\delta]$, and let $I$ be a connected component of $[0,\delta]\setminus K$. By $\langle f,g\rangle$-invariance of $K$, and the established uniform convergence to the identity, there exists $n_0\in \N$ such that $f^{-n}\gamma f^{n}(I)=I$ for every $n\ge n_0$. This gives that the element $\gamma=g^m\circ f^l$ preserves all the intervals of the form $f^n(I)$, for $n\ge n_0$, and in particular it admits infinitely many fixed points accumulating at $0$, which contradicts the assumption \ref{i.ilyashenko}.
	
	We assume next that $\log \lambda$ and $\log \mu$ are rationally independent, and in particular that $\mu>1$. Notice that, up to possibly redefining $(f,g):=(g^{-1},f^{-1})$, we can assume that either $g(U)=U$, or $g$ has a fixed point in $U\setminus \{0\}$. Applying Theorem \ref{t-sternberg0} again, take a $C^r$ coordinate $k\colon W\to [0,+\infty)$ on an open neighborhood $W$ of $0$ so that the map $\overline g$ becomes the scalar multiplication by $\mu$ on $[0,+\infty)$; to be precise, we can take as $W$ the maximal interval containing no fixed point for $\overline g$, except $0$, which is contained in $V$. Note that after Theorem \ref{t-sternberg0}, we can take $k$ such that $k'(0)=1$, so that $k(x)=x+O(x^r)$ as $x\to 0$.
	
	Given $\nu >0$, there exist two increasing sequences $(l_n)_{n\in \N},(m_n)_{n\in \N}\subset \N$, such that $\lambda^{l_n}\mu^{m_n}\to \nu$ as $n\to \infty$. For $n\in \N$, the composition $g_n=\overline g^{m_n} \circ \overline f^{l_n}$ is defined on $W$, as $\overline f$ contracts $W$ and $\overline g$ preserves it.
	Fix $x\in W$, so that
	\begin{align*}
		g_n(x)&=k^{-1}\left (\mu^{m_n}k(\lambda^{l_n}x)\right )
		=
		k^{-1}\left (\mu^{m_n} (\lambda^{l_n}x+O(\lambda^{rl_n}x^r))\right )\\
		&=k^{-1}\left (\mu^{m_n} \lambda^{l_n}x+O(\lambda^{(r-1)l_n}x^r)\right )\quad\text{as }n\to \infty.\end{align*}
	We deduce the convergence $g_n(x)\to k^{-1}(\nu x)$ as $n\to \infty$. As $\nu>0$ was arbitrary, this gives that the orbit of every $x\in W\cap (0,+\infty)$ is dense in $W\cap (0,+\infty)$, as desired.
\end{proof}

Finally, we also need a basic fact.

\begin{lem}\label{l:flt}
	Let $n \in \N$ be an integer, and $p\in \mathbb Q$ any rational. Then there exists a non-trivial $n$-adic affine map $g\in\Aff(\Z[1/n],\langle n\rangle_*)\subset \Aff(\R)$ such that $g(p)=p$. 
\end{lem}

\begin{proof}
	Write $g(x)=n^kx+\frac{a}{n^b}$, with $a\in\Z$ and $b,k\in\N$, for a generic $\ell$-adic affine map.
	Note that the condition $g(p)=n^kp+\frac{a}{n^b}=p$ gives $p=\frac{-a}{n^b(n^k-1)}$, which can be any rational number (choosing appropriate $a\in\Z$ and $b,k\in\N$).	
\end{proof}

We can now prove the main result of this section.

\begin{proof}[Proof of Theorem \ref{t-Thompson-Stein}]
	We argue by way of contradiction. Write $G=F_{n_1,\ldots,n_k}$ and let $\varphi\colon G\to \Diff_0^r(\R)$ be a faithful action ($r>1$), that we can assume irreducible. After Corollary \ref{c-lm-C1-interval}, $\varphi$ is either semi-conjugate to the standard action on $X$, or to a cyclic action. Assume first that the former occurs, and write $h\colon \R\to X=(0,1)$ for the semi-conjugacy. Using Lemma \ref{l-semic-F-C1} applied to the action of $F_{n_1}\subseteq G$, we find a rational point $p\in X$ and an element $f\in F_{n_1}$, such that $\xi=h^{-1}(p)$ is a hyperbolic fixed point for $\varphi(f)$. Using Lemma \ref{l:flt}, we can find an element $g\in F_{n_2}$ for which $p\in X$ is an isolated fixed point. In particular, $f$ and $g$ commute in restriction to a right neighborhood $[p,q]$ of $p$, and their right germs at $p$ generate an abelian free group of rank 2 (they are scalar multiplications by powers of $n_1$ and $n_2$, respectively). Thus, up to considering inverse powers, the assumptions of Proposition \ref{prop.stabilizer_dense} are satisfied by the maps $\varphi(f)$ and $\varphi(g)$, from which we deduce that the action of $\langle \varphi(f),\varphi(g)\rangle$ is minimal in restriction to an interval of the form $(\xi,\xi+\delta)$, with $\delta>0$. Hence, the semi-conjugacy $h\colon \R\to X$ considered above is a conjugacy (that is, $h$ is a homeomorphism).
	On the other hand, up to considering inverse powers, the elements $f$ and $g$ satisfy the assumptions of Proposition \ref{p-local-flow}, and we deduce that $h$ is $C^r$ in restriction to $[p,q]$.
	
	We conclude as in \cite[proof of Theorem 7.3]{BLT}. Take an element $\gamma \in G$ with a discontinuity point $r\in [p,q]$ for its derivative; then also the derivative $\varphi(\gamma)'$ has a discontinuity point at $h^{-1}(r)$. This gives the desired contradiction.
	
	In the case of cyclic action, considering the corresponding homomorphism $\tau\colon G\to \Z$, we have that the $\varphi$-action of $\ker \tau$ on every connected component of its support is semi-conjugate to the standard action (Corollary \ref{c-lm-C1-interval}). Thus one can reproduce the previous argument, adapted to $\ker \tau$.
	This is a little tricky, as the abelianization $F_{n_i}/[F_{n_i},F_{n_i}]\cong \Z^{n_i}$ is larger than the quotient $F_{n_i}/(F_{n_i})_c\cong \Z^2$. Start with an element $f\in F_{n_1}\cap \ker \tau$ with a hyperbolic fixed point $\xi$, as in the previous case, and then choose $g_1\in (F_{n_2})_c$ fixing $p=h^{-1}(\xi)$ playing the role of $g$ in the previous case.
	However, it could be that $g_1\notin \ker \tau$; if this happens, take an element $g_2\in G$ such that $g_2\left (\supp(g_1)\right )\cap \supp(g_1)=\varnothing$. Then the commutator $g=[g_1,g_2]$ coincides with $g_1$ on $\supp(g_1)$, and belongs to $\ker \tau$.
\end{proof}

As explained before, the method presented here cannot be applied directly to exclude $C^1$ smoothability.
We believe however that this can be achieved with a different approach. Let us point out that all known examples of $C^1$ smoothable groups of PL homeomorphisms embed in Thompson's group $F$. It would be tempting to conjecture that this is also a necessary condition. However, we estimate that little is known about other groups of PL homeomorphisms, such as those defined by irrational slopes (see  however the very recent work \cite{hyde2021subgroups}, where it is proved that several such groups do not embed in $F$). So, let us highlight the following concrete problem.
\begin{ques}
	Fix an irrational $\tau\in \R_{>0}\setminus \Q$, and write $\Lambda=\langle \tau\rangle_*$ and $A=\Z[\Lambda]$ (as an explicit case, one can take the golden ratio $\tau=\frac{\sqrt{5}-1}{2}$). Consider the irrational slope Thompson's group $F_\tau=G((0,1);A,\Lambda)$. Is the action of $F_\tau$ on the interval $C^1$ smoothable?
\end{ques}

\subsection{An application to Bieri--Strebel groups on the line}
Given a real number $\lambda>1$,  we consider the Bieri--Strebel groups\index{groups!Bieri--Strebel groups} acting on the line $G(\lambda)=G(\R;\Z[\lambda,\lambda^{-1}],\langle \lambda\rangle_*)$. It was remarked in \cite{BLT} that the standard action of $G(\lambda)$ cannot be conjugate to any $ C^1$ action.  One of the main results of \cite{BLT} states that for certain choices of $\lambda$, the group $G(\lambda)$ does not admit any faithful $C^1$ action on the line.  Here we generalize this result by removing all restrictions on $\lambda$.  	 
\begin{cor}
	For $\lambda>1$, there is no faithful $C^1$ action of the Bieri--Strebel group $G(\lambda)$ on the closed interval.
\end{cor}

\begin{proof}
	Indeed, it is proved in \cite[Theorem 6.10]{BLT} that the standard action of $G(\lambda)$ on the line cannot be conjugate to any $C^1$ action on the closed interval, but a closer look at the proof (notably using the results from \cite[\S 4.2]{BMNR}) shows that even a semi-conjugacy is impossible, as the action of the affine subgroup of $G(\lambda)$ must be minimal. Hence the result follows from Theorem \ref{t-lm-C1}.
\end{proof}
In Chapter \ref{s-few-actions}, we will  classify  $C^0$ actions of $G(\lambda)$ up to semi-conjugacy, whenever $\lambda$ is algebraic.

\part{Structure theorems for $C^0$ actions: laminations and horogradings}\label{partII}

In this part we obtain structure results for exotic actions on the line $\varphi\colon G\to\homeo_0(\R)$ of locally moving subgroups $G\subset\homeo_0(\R)$. In Chapter \ref{sec.focalgeneral} we begin the study of laminar actions on the line, namely actions preserving a lamination. In particular we introduce the crucial notion of horograding, which provides a way to understand a laminar action by relating it to a different action of the group on the line, which retains information on its large-scale dynamics (see \S \ref{ssec.class_elts}). Finally, we give a general criterion to find invariant laminations for actions on the line (Proposition \ref{p-commuting-subgroup-lamination}), which will be the starting point of the main results in this part. 

In Chapter \ref{sec.locandfoc}, we obtain the main structure theorems for exotic actions on the line of micro-supported and locally moving subgroups $G\subset \homeo_{0}(\R)$, corresponding to Theorems \ref{t-intro-laminar-alternative} and \ref{t-intro-horograding} from the introduction. More precisely, for micro-supported subgroups $G\subset\homeo_0(\R)$ whose standard action is minimal, we prove that any faithful minimal action $\varphi\colon G\to \homeo_0(\R)$ is either laminar or locally moving, and in the second case it is conjugate to the standard action (Theorem \ref{t-laminations-microsupported}). Then we show that under the assumption that $G$ has finitely generated fragmentable subgroup $\Gfrag$ (Definition \ref{d-fragmentable-subgroup}), any exotic action of $G$ is laminar and horograded by the standard action (Theorem \ref{t-lm-horograding}). We conclude the chapter with a general existence result of minimal laminar actions for a large class of finitely generated micro-supported groups (\S \ref{ssec.germtype}), elaborating on the construction from \S \ref{s-germ-type}.

In Chapter \ref{s-few-actions} we see how the main structure theorems can be applied in an interesting family of examples, the Bieri--Strebel groups of PL homeomorphisms of the line $G:=G(\R; A, \Lambda)$. When the finite generation condition on $\Gfrag$ is satisfied (for Bieri--Strebel groups this is a very explicit condition, see Lemma \ref{p.BieriStrebel_fg_germs}), we use Theorem \ref{t-lm-horograding} to provide an explicit classification of exotic actions of $G$: they  are all obtained through the jump preorder construction introduced in \S \ref{s.BSjump} (Theorem \ref{t-BBS}). Interestingly, by a twist of this construction, we provide an example of a faithful minimal laminar action of a Bieri--Strebel group which cannot be horograded by the standard action (Proposition \ref{p-Glambda-transcendental}), thus showing that the finite generation condition for $\Gfrag$ is crucial in Theorem \ref{t-lm-horograding}. We also provide an example of a finitely generated locally moving group (obtained as an overgroup of a Bieri--Strebel group) that admits no exotic actions at all, see \S \ref{s-no-actions}.

The remaining Chapters \ref{sec_focal_trees}--\ref{s-F} are complementary to the main results of this part. They are devoted to the analysis of further examples and properties of laminar actions.

In Chapter \ref{sec_focal_trees} we develop an alternative language to discuss laminar actions and their horogradings,  by considering actions on planar directed (real) trees, which are the natural dual objects of invariant laminations. This point of view is particularly useful in analyzing and parsing different examples. The correspondence is formalized in the results in \S\ref{s-planar-trees}. The case when these actions are isometric (for some appropriate choice of a metric on the tree) deserves a particular attention; this is the content of \S\ref{ssc.simplicial_actions}. At the end of the chapter, we revisit the construction of minimal exotic actions for locally moving groups with cyclic groups of germs at $+\infty$ from \S \ref{sec:cyclic_germs1}, and show that they can be easily described in terms of actions on simplicial trees. 

In Chapter \ref{sec.examplesRfocal},  we provide a general construction of micro-supported subgroups of $\homeo_0(\R)$ whose standard action is laminar,  generalizing an example introduced by Brin \cite{Brin} and Navas \cite{NavasAmenable}. This construction produces groups admitting uncountably many faithful minimal actions on the line which are all micro-supported,  in some cases even by diffeomorphisms, and are pairwise non-semi-conjugate. This shows that the main rigidity results from Part \ref{partI} for locally moving groups (in particular, Theorem \ref{t-intro-C1}) cannot be extended to general micro-supported groups. The point of view of trees is particularly useful here. Indeed these groups are naturally defined as groups of automorphisms of a directed simplicial tree.  The key observation is that, in some cases, the action on the tree admits many distinct invariant planar orders, and each such order provides a micro-supported action on the line.

The last chapter is focused on probably the most emblematic locally moving group of homeomorphisms of the line, namely Thompson's group $F$. In this case we have $F_{\mathsf{frag}}=F$, so the main structure result (Theorem \ref{t-lm-horograding}) gives that any exotic action is laminar, horograded by the standard action. Most of the previous constructions give several examples of exotic actions for $F$ (more precisely, those in  \S\S \ref{s.BSjump}, \ref{ssec.germtype}, and \ref{ssec.cyclicgerm}). Here we use the specific algebraic and dynamical structure of $F$ to find way more examples, with sensibly different dynamical behavior. Contrary to the case of Bieri--Strebel groups on the line discussed in Chapter \ref{s-few-actions}, we are not able to provide a precise classification of laminar actions of $F$, and our examples here are meant to give a glimpse of the variety of the possible constructions. Notably, it turns out that many examples of laminar actions of $F$ actually arise from actions of $F$ on \emph{simplicial} trees. However, in \S\ref{s-F-non-simplicial}, we provide an example of a minimal laminar action of $F$ for which the dual tree cannot be chosen to be simplicial.

\chapter{Laminations and horogradings}\label{sec.focalgeneral}
In this chapter we study group actions on the line with invariant laminations, and introduce and study the associated notion of horograding. Recall that, given an action $\varphi\colon G\to \homeo_{0}(\R)$, we use Greek letters to denote points of $\R$, and we often use shortcut notation like  $g.\xi$ for $\varphi(g)(\xi)$, $\fixphi(g)$ for $\fix(\varphi(g))$, and so on. 
\section{Laminar actions}\label{ss.focal}
\subsection{First definitions}
\begin{dfn} \label{d-CF-cover} Let $\Omega$ be a set. We say that two subsets $I,J\subset \Omega$ \emph{do not cross} if they are either nested or disjoint: either $I\subset J$, or $J\subset I$, or $I\cap J=\varnothing$. 
A \emph{prelamination} of a totally ordered set $(\Omega,\prec)$ is a non-empty collection $\mathcal L$ of convex bounded open subsets of $\Omega$ (with respect to the order topology), which pairwise do not cross (we also say that $\mathcal L$ is \emph{cross free}). Elements of a prelamination are also called the \emph{leaves}.

In the case $(\Omega,\prec)$ is the real line $\R$ with its standard order, elements of a prelamination are bounded open intervals.
A \emph{lamination} of the real line is a prelamination, which is a closed subset of the set $\R^{(2)}=\{(x, y)\in \R^2: x<y\}$, with respect to the usual topology (given by convergence of endpoints). When $X=(a,b)$ is a real interval, we will keep saying ``bounded intervals'' for ``relatively compact intervals''. 
\end{dfn}
  
\begin{rem}
	If a group action $\varphi \colon G \to \homeo_0(\R)$ has an invariant prelamination $\mathcal{L}$, then its closure in $\R^{(2)}$ is an invariant lamination. The requirement that laminations be closed is coherent with the well-established terminology for group actions on the circle, inspired by hyperbolic geometry (see e.g.\ the book of Calegari \cite[\S 2.1]{MR2327361}). 
It is convenient to visualize a prelamination of the line as a collection of pairwise disjoint, geodesic semi-circles in the upper half-plane, whose diameters are the elements of the prelamination.
\end{rem}

\begin{rem}
	When $(\Omega,\prec)$ is an ordered set, and $\mathcal L$ a collection of convex bounded subsets of $\Omega$ with non-empty interior and without crossings, then we can automatically get a prelamination by removing the endpoints of each subset. Therefore, in the examples, we will never check that leaves are open.
\end{rem}

A prelamination is naturally a partially ordered set  $(\mathcal{L}, \subset)$, or  \emph{poset}, with respect to the inclusion order between intervals. We say that a subset $\mathcal{L}_0\subseteq \mathcal{L}$ is \emph{cofinal} if every $I\in \mathcal{L}$ is contained in some $J\in \mathcal{L}_0$. We will frequently use the fact that every totally ordered subset of a poset is contained in a maximal one, by Zorn's lemma.

We will also use the following terminology.
 
\begin{dfn} A prelamination $\mathcal{L}$ of a set $(\Omega, \prec)$  is \emph{covering} if it contains an increasing exhaustion of $\Omega$; it is \emph{thin} if for every maximal totally ordered family of leaves $\mathcal{F}\subset \mathcal{L}$, the intersection $\bigcap_{I\in \mathcal{F}} I$ contains at most one point.
\end{dfn}

\begin{rem}\label{rem.covering.prelamination}
	A prelamination $\mathcal{L}$ of the real line is  covering in the previous sense if and only if  it is covers  $\R$. The forward direction is clear. For the converse, observe that since intervals in $\mathcal{L}$ do not cross, the unions of any two maximal totally ordered subsets of $\mathcal{L}$  are either equal or disjoint; by connectedness, if $\mathcal{L}$ covers $\R$, then we can find a single maximal totally ordered subset whose union covers $\R$, which must contain an increasing exhaustion of $\R$.
\end{rem}

\begin{dfn} \label{d-laminar}
	An action $\varphi\colon G\to \homeo_0(\R)$ is \emph{laminar} if it preserves a covering lamination (equivalently, a covering prelamination).
 \end{dfn}

In several examples, laminar actions are obtained by considering order-preserving actions on sets admitting invariant prelaminations, and then taking their dynamical realization, as discussed in \S \ref{sec.dynreal}. Although this sounds intuitive, let us detail how the construction works.

\begin{rem}\label{sc:inv_prelam_ordered_sets}
	Assume that $\psi\colon G\to \Aut(\Omega,\prec)$ is an action on a totally ordered set, admitting an invariant prelamination $\mathcal L_0$. When $\Omega$ is countable, we can then consider the dynamical realization $\varphi\colon G\to \homeo_{0}(\R)$, with associated good embedding $i\colon\Omega\to \R$. Then the collection $\mathcal L$ obtained by considering interiors of the convex hulls of the images $i(I)$, for $I\in \mathcal L$, defines a $\varphi$-invariant prelamination of $\R$, as $i$ is order preserving and equivariant.
	With abuse of notation, we keep denoting by $i\colon\mathcal L_0\to \mathcal L$ this correspondence, which is actually bijective, as $i$ is injective. It is clear that if $\mathcal L_0$ is covering, then also $\mathcal L$ is. Note that when the image $i(\Omega)$ is dense (this is the case when the dynamical realization $\varphi$ is minimal), the elements of $\mathcal L$ are simply obtained by taking the interior of the closures of the images of the leaves of $\mathcal L_0$. Let us illustrate this with a concrete example, which will be our running example for getting familiar with laminar actions.
\end{rem}

\begin{ex}[Plante-like actions of permutational wreath products]	\label{subsec.Plantefocal} 
	One of the first examples of focal laminar actions on the line appearing in the literature is an action of the lamplighter group $\Z\wr \Z$ defined by Plante \cite{Plante}. In fact, this action can be recovered as the dynamical realization of a natural affine action defined on the ordered field of Hahn--Neumann series (see e.g.\ Neumann \cite{ordered-div-rings}), but this correspondence has never been pointed out in the literature (as far as we know). In similar spirit, we describe a construction that works for permutational wreath products of countable left-orderable groups. Recall that for general groups $G$ and $H$, with $G$ acting on a set $X$, the \emph{permutational (restricted) wreath product} $H\wr_X G$ is defined as the semidirect product $\left (\bigoplus_XH\right )\rtimes G$, where $G$ acts on the direct sum by permutation of coordinates. More explicitly, considering the direct sum  $\bigoplus_XH$ as the set of functions $\s\colon X\to H$ which are trivial at all but finitely points of $X$, the action of $g\in G$ is given by $\sigma(g)(\s)(x)=\s(g^{-1}.x)$. When considering the action of $G$ by left multiplication on $X=G$, one simply refers to the \emph{wreath product} $H\wr G$ of $H$ and $G$.
	
	Given a left-invariant order $<_H\in \LO(H)$, and a $G$-invariant order $<_X$ on $X$, we can consider an order $\prec$ of lexicographic type on $\bigoplus_XH$, as follows.
	We denote by $\mathsf e$ the trivial element of $\bigoplus_XH$, that is the function satisfying $\mathsf e(x)=1_H$ for every $x\in X$, and we define
	\[
	P=\left \{\s\in \textstyle\bigoplus_XH: \s\neq \mathsf e,\, \s(x_{\s})>_H1_H\right \},
	\]
	where for $\s\neq \mathsf e$ we set $x_{\s}=\max_{<_X}\{x\in X: \s(x)\neq 1_H\}$.
	It is not difficult to check that $P$ defines a positive cone, and thus a left-invariant order $\prec$ on the direct sum $\bigoplus_XH$, which is also invariant under the permutation action $\sigma$ of $G$. This gives an order-preserving action
	$\Psi\colon H\wr_X G\to\Aut\left (\bigoplus_XH,\prec\right )$, that we call the \emph{Plante-like product} of $<_G$ and $<_X$.  Assume now that $G$, $H$, and $X$ are countable; in particular, so is $\bigoplus_XH$. We can then consider the dynamical realization of $\Psi$, which we call the \emph{Plante-like action} associated with $<_X$ and $<_H$. In this situation we  let $\iota\colon (\bigoplus_XH, \prec)\to \R$ be the associated good embedding and  $\varphi\colon H\wr_X G\to \homeo_0(\R)$ the dynamical realization. When $G=H=X=\Z$, and $<_X$ and $<_H$ are the standard left orders of $\Z$, this construction yields the  action of $\Z\wr\Z$ considered by Plante (see \cite[\S 3.3.2]{GOD}), an illustration of which appears in Figure \ref{fig-Plante-action}.
	
	Now, assume that the $G$-action on $X$ is cofinal, in the sense that for any two distinct $x,y\in X$ there exists $g\in G$ such that $g.x>_Xy$. If so, we claim that the Plante-like action $\varphi$ is minimal and laminar.
	To prove minimality, note that the stabilizer of the trivial element $\mathsf e\in \bigoplus_XH$ coincides with $G$, and its orbit is the whole subgroup $\bigoplus_XH$. Thus, after Proposition \ref{p.minimalitycriteria}, it is enough to check that given four elements $\s_1,\s_2,\mathsf t_1,\mathsf t_2\in \bigoplus_XH$ such that $\mathsf t_1\prec \mathsf s_1\prec \mathsf e\prec \s_2\prec \mathsf t_2$, one can find $g\in G$ such that 
	\begin{equation}\sigma(g)(\s_1)\prec \mathsf t_1\prec  \mathsf t_2\prec \sigma(g)(\s_2). \label{e-plante-verification}\end{equation}
	For this, we set $y_*=\max_{<_X}\{x_{\mathsf t_1},x_{\mathsf t_2}\}$ and $x_*=\min_{<_X}\{x_{\s_1},x_{\s_2}\}$; by cofinality, we can take an element $g\in G$ such that $g^{-1}x_*>_Xy_*$. Then it is immediate to check that \eqref{e-plante-verification} is satisfied.
	
	To check that $\varphi$ is laminar, after our discussion right before this example, it is enough to find a $\Psi$-invariant covering prelamination of $(\bigoplus_XH,\prec)$. For this, for $\s\in \bigoplus_XH$ and $y\in X$, we set 
	\begin{equation}\label{e-prelamination-plante}
		C_{\s,y}=\left \{\mathsf t\in\textstyle \bigoplus_XH :\mathsf t(x)=\s(x)\text{ for every }x >_X y\right \}.
	\end{equation}
	Clearly, every $C_{\s,y}$ is a convex bounded subset of $(\bigoplus_XH, \prec)$ with non-empty interior, and the family
	\[\mathcal{L}_0=\left \{C_{\s,y}:\s\in\textstyle \bigoplus_XH,y\in X\right \}\]
	is $\Psi$-invariant.  It only remains to check that $\mathcal L_0$ is a cross-free collection.  For this, take two elements $C_{\s,y}$ and $C_{\s',y'}$  in $\mathcal L_0$ with $y\le_X y'$, and assume there is some element $\mathsf t$ in their intersection. It follows that $\s$, $\mathsf t$, and $\s'$ all agree on $\{x\in  X :  x>_Xy' \}$, and so $C_{\s,y}\subseteq C_{\s',y'}$. This shows that $\mathcal L_0$ is a prelamination, as desired.
	\begin{figure}[ht]
		\centering
		\includegraphics[scale=1]{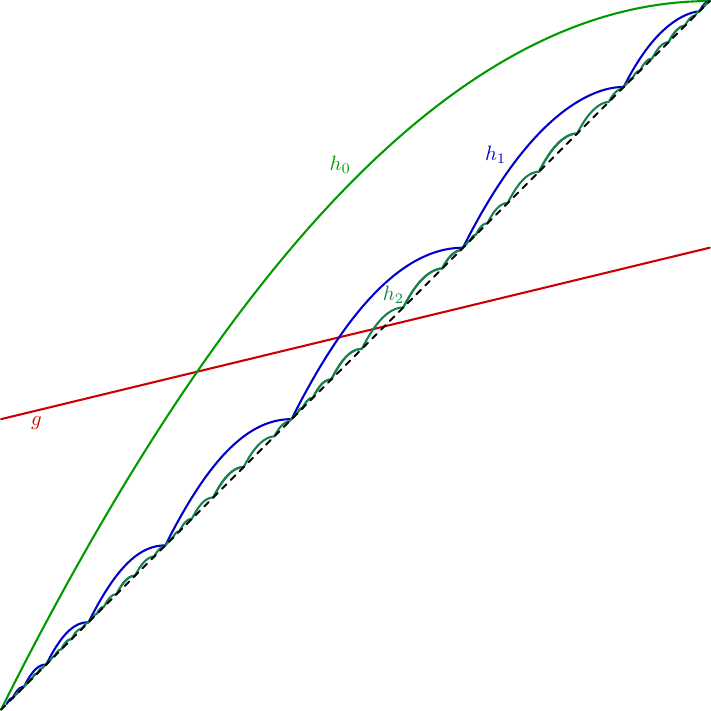}
		\caption{Plante-like action of $\Z\wr\Z$ on the line. One factor is generated by $g$ which acts as a homothety. The generator of the other factor is $h_0$, and we have $h_n=g^nh_0g^{-n}$ for every $n\in \Z$, where the $h_n$ commute and are a basis of the lamp group $\bigoplus_\Z \Z$. }\label{fig-Plante-action}
	\end{figure}
\end{ex}

\subsection{Dynamical classification of subgroups and elements} \label{s-dynclasselements}

Here we want to describe some restrictions for the dynamics of laminar actions. A first straightforward consequence of the definitions is the following.

\begin{lem} \label{rem.fixedepoints}
	If $\varphi\colon G\to \homeo_{0}(\R)$ is a laminar action, then $\fixphi(g)\neq \varnothing$ for every element $g\in G$.
\end{lem}
\begin{proof}
	Let $\mathcal{L}$ be an invariant covering prelamination and fix $g\in G$. For every $\xi \in\R$, we can find $J\in  \mathcal{L}$ which contains $\xi$ and $g.\xi$, so that $g.J\cap J\neq \varnothing$. Then either $g$ or $g^{-1}$ must map $J$ into itself. By the intermediate value theorem, $\varphi(g)$ has a fixed point in $J$.
\end{proof}

Laminations of the line can be considered as objects analogous to trees, with a marked point at infinity (this analogy will be formalized in Chapter  \ref{sec_focal_trees}). Groups acting on trees (more generally, on Gromov-hyperbolic spaces) admit a general classification into five types, only four of which can arise for actions that fix a point in the boundary (see Gromov \cite[\S 3.1]{Gromov-hyp-gps}, or Caprace, Cornulier, Monod, and Tessera \cite{amen-hyp-gps} for a more modern presentation). In a tightly analogous way, laminar group actions on the line can be divided into the following four types. Recall that we say that an interval $I\subset \R$ is \emph{wandering} (for $\varphi$) if for every $g\in G$, either $g.I=I$, or $g.I\cap I=\varnothing$ (Definition \ref{d-wandering-interval}).

\begin{prop}[Dynamical classification of laminar actions] \label{p-dyn-class-subgroups}
Let $\varphi\colon G\to \homeo_0(\R)$ be an action with an invariant covering lamination $\mathcal{L}$, and denote by ${\mathcal L^G}\subset \mathcal L$ the subset of $\varphi$-invariant leaves. Then exactly one of the following holds. 
 \begin{enumerate}[label=($\mathcal{L}$\arabic*)]
\item\label{i-totally-bdd}\emph{(Totally bounded)}. The set $\mathcal L^G$ is cofinal in $\mathcal{L}$.
In this case, the subset $\fixphi(G)$ accumulates on both $\pm \infty$.

\item\emph{(Pseudo-homothetic)}.  \label{i-pseudo-homothetic} We have $\fixphi(G)\neq \varnothing$, and $\mathcal L^G$ is not cofinal in $\mathcal{L}$. In this case, the set of fixed points $\fixphi(G)$ is non-empty and compact, and the $\varphi$-actions of $G$ on the two unbounded connected components $J^{-}:=(-\infty,\min \fixphi(G))$ and $J^{+}:=(\max \fixphi(G),+\infty)$ of $\R \setminus \fixphi(G)$ are negatively  semi-conjugate (that is, semi-conjugate by a non-increasing map $J^{-}\to J^{+}$).

\item[($\mathcal{L}$2')]\emph{(Homothetic)} The previous case holds, and furthermore $\fixphi(G)$ contains a single point; this happens if and only if $\fixphi(G)\neq \varnothing$  and $\mathcal{L}^G=\varnothing$.

\item  \label{i-pseudo-horocyclic} \emph{(Horocyclic)}. The action $\varphi$ is irreducible, and $\mathcal{L}$ has a cofinal subset of wandering intervals. In this case, the action $\varphi$ does not admit any minimal invariant set, and every finitely generated subgroup $H\subset G$ is totally bounded (in the sense that it satisfies \ref{i-totally-bdd}).

\item \emph{(Focal)}. The action $\varphi$ is irreducible, and there exists an interval $I\in \mathcal{L}$ whose $\varphi$-orbit is cofinal.
In this case, the action $\varphi$ has a unique minimal invariant set $\Lambda\subseteq \R$, which is not discrete. \end{enumerate}

In particular, irreducible laminar actions of finitely generated groups are focal.
\end{prop}

\begin{proof}
If the subset  $\mathcal{L}^G\subset \mathcal{L}$ of $\varphi$-invariant leaves is cofinal, then the collection of endpoints of elements of $\mathcal{L}^G$ is contained in $\fixphi(G)$ and accumulates on both $\pm \infty$. In this case, $G$ is totally bounded. 

Suppose that $\fixphi(G)\neq \varnothing$, but  \ref{i-totally-bdd} does not hold: the set of $\varphi$-invariant leaves is not cofinal in $\mathcal{L}$. Take  $\xi\in \fixphi(G)$, and write $\mathcal{L}_\xi=\{I\in \mathcal{L}: \xi\in I\}$. Then $\mathcal{L}_\xi$  is totally ordered, because any two leaves in $\mathcal{L}_\xi$ intersect non-trivially, and $\varphi$-invariant, as $\xi$ is fixed by $G$. Since we assume that the set of fixed leaves is not cofinal in $\mathcal{L}$, the $\varphi$-orbit of any sufficiently large $I\in \mathcal{L}_\xi$ must be cofinal, i.e.\ there exists a sequence $(h_n)\subset G$ such that $h_n.I$ is an increasing exhaustion of $\R$.  It follows that $\fixphi(G)\subset I$, thus $\fixphi(G)$ is compact. Similarly, the collection $\mathcal{L}_{\fixphi(G)}=\{I\in \mathcal{L} \colon \fixphi(G)\subset I\}$ is totally ordered and $\varphi$-invariant.
Let $J^{-}$ and $J^{+}$ be the unbounded connected components of $\suppphi(G)$, adjacent respectively to $-\infty$ and $+\infty$. Considering the two endpoints of any $I\in \mathcal{L}_{\fixphi(G)}$, we define  two equivariant maps  $m_{\pm}\colon (\mathcal{L}_{\fixphi(G)}, \subset)\to J^{\pm}$, where $m_-$ is non-increasing and $m_+$ is non-decreasing. This implies that the actions of $G$ on $J^{-}$ and on $J^{+}$ are negatively semi-conjugate. Explicitly, a semi-conjugacy is provided by the map $J^{-}\to J^{+}, \xi\mapsto m_+(I_\xi)$, where $I_\xi$ is the largest element of $\mathcal{L}_{\fixphi(G)}$ such that $\xi\le m_-(I_\xi)$. This describes the pseudo-homothetic case. Finally, note that if $|\fixphi(G)|\geq 2$, then $\mathcal{L}_{\fixphi(G)}$ has an infimum $I\in \mathcal{L}$, which is $\varphi$-invariant, thus $\mathcal{L}^G\neq \varnothing$. Conversely if $\mathcal{L}^G\neq \varnothing$, the endpoints of any $I\in \mathcal{L}^G$ belong to $\fixphi(G)$, so $|\fixphi(G)|\geq 2$. This describes the homothetic subcase.

Suppose now that $\varphi$ is irreducible, but not focal: no $I\in \mathcal{L}$ has a cofinal orbit. For $I\in \mathcal{L}$, set $\hat{I}=\bigcup_{h\in G, I\subseteq h.I}h.I$.  Note that $\hat{I}$ is an open interval, which cannot be the whole real line. It also cannot be a half-line for, if this was the case, then any $J\in \mathcal{L}$ containing its endpoint would have to cross $h.I$, for some $h\in G$. Hence $\hat{I}$ is a bounded interval, which belongs to  $\mathcal{L}$, since $\mathcal{L}$ is closed. It is clear that no element of $G$ can send $\hat{I}$ to a leaf that strictly contains it, so that $\hat{I}$ is a wandering interval. Since $I$ was arbitrary and $I\subset \hat{I}$, it follows that the subset  $\mathcal W\subset \mathcal L$ of wandering intervals is cofinal. For every  $J\in \mathcal W$, the union of the $\varphi$-images of $J$ is a proper open $\varphi$-invariant subset of $\R$. Now, suppose by contradiction that $\varphi$ admits a minimal invariant set $\Lambda\subset \R$. First note that $\Lambda$ cannot be a  discrete orbit, as this would imply the existence of elements without fixed points, contradicting Lemma \ref{rem.fixedepoints}. Thus, $\Lambda$ is unique and contained in every $\varphi$-invariant closed subset of $\R$; equivalently, it does not intersect any  $\varphi$-invariant proper open subset. However, cofinality of $\mathcal W$ gives an interval $J\in \mathcal W$ such that $J\cap \Lambda\neq \varnothing$, and we have remarked that its $\varphi$-orbit gives a proper $\varphi$-invariant open subset and the desired contradiction. Finally, let $S\subset G$ be any finite symmetric subset, and $(J_n)\subset \mathcal W$ an increasing exhaustion  of $\R$ by wandering intervals. For every $s\in S$ and every sufficiently large $n$, we must have $s.J_n\cap J_n\neq \varnothing$, and thus $s.J_n=J_n$. It follows that the subgroup $\langle S \rangle$ fixes the endpoints of $J_n$ for every sufficiently large $n$, and thus it is totally bounded (in the sense that it fits the conditions in \ref{i-totally-bdd}). This describes the horocyclic case.

In the remaining case, $\varphi$ is irreducible and there exists $I\in \mathcal{L}$ with a cofinal orbit. As before, the action $\varphi$ cannot have a discrete orbit, so $\varphi$ has at most one minimal invariant set; it is enough to show that such a set exists.  Since $\mathcal{L}$ is covering, every non-empty closed $\varphi$-invariant subset of $\R$ intersects an element of $\mathcal{L}$ and so, by cofinality of the orbit of $I$, it intersects $I$ as well.  By compactness of $\overline{I}$, Lemma \ref{l.zorn_minimal} provides a minimal invariant set.
\end{proof}

\begin{dfn}
	In the sequel, we adopt the terminology introduced in Proposition \ref{p-dyn-class-subgroups}. If  $\varphi\colon G\to \homeo_0(\R)$ is a laminar action, we also use the same terminology to distinguish different subgroups of $G$. For instance, we shall say that a subgroup $H\subseteq G$ is horocyclic if $\varphi|_H$ satisfies \ref{i-pseudo-horocyclic}. Note that the homothetic case is a subcase of the pseudo-homothetic case, but it is convenient to single it out.
\end{dfn}

A particularly relevant special case of Proposition \ref{p-dyn-class-subgroups} is when $H$ is a cyclic subgroup. In this case, we have a classification of individual elements into two types, which are analogous to the classification of isometries of trees into loxodromic and elliptic elements. This leads to the introduction of the following terminology.

\begin{dfn}
We say that a homeomorphism $g\in \homeo_0(\R)$ is an expanding (respectively, contracting) \emph{pseudo-homothety} if the set $\fix(g)$ of fixed points is (non-empty and) compact, and for every open interval $I\supset \fix(g)$,  we have $\bigcup_{n=0}^\infty g^n(I)=\R$ (respectively, $\bigcup_{n=0}^\infty g^{-n}(I)=\R$). If in addition $\fix(g)$ is reduced to a single point, we say that $g$ is a \emph{homothety}. 
\end{dfn}

\begin{prop}[Dynamical classification of elements]\label{p.dynclasselements}
Let $\varphi\colon G \to \homeo_0(\R)$ be a laminar action, with invariant covering lamination $\mathcal{L}$. For $g\in G$, set $\mathcal{L}^g=\{I\in \mathcal{L}: g.I=I\}$. Then the following alternative holds:
\begin{itemize}
\item if $\mathcal{L}^g$ is cofinal in $\mathcal{L}$, then $\varphi(g)$ is totally bounded; otherwise,
\item $\varphi(g)$ is a pseudo-homothety. Moreover, $\varphi(g)$ is a homothety if and only if $\mathcal{L}^g=\varnothing$. 
\end{itemize}
\end{prop}

\begin{proof}
Most of the statement is Proposition \ref{p-dyn-class-subgroups}, specialized to the cyclic subgroup generated by $g$, which must satisfy one of the first two cases by Lemma \ref{rem.fixedepoints}. We only need to justify the part of the statement characterizing when $\varphi(g)$ is a homothety. Observe that if $\mathcal{L}^g\neq \varnothing$, then $\varphi(g)$ has at least two fixed points (the endpoints of any $I\in \mathcal{L}^g$). Conversely, suppose $\varphi(g)$ is a pseudo-homothety with $|\fixphi(g)|\geq 2$; then the set $\{I\in \mathcal{L}: \fixphi(g)\subset I\}$ has a smallest non-trivial element, which belongs to $\mathcal{L}^g$. \qedhere
\end{proof}

We now justify that all notions discussed above are stable under semi-conjugacy. 

\begin{prop} \label{p-focal-semiconj} Let $\varphi, \psi\colon G\to \homeo_0(\R)$ be two  irreducible semi-conjugate actions. If $\varphi$ is laminar, then so is $\psi$, and the type of each subgroup $H\subseteq G$ is the same for $\varphi$ and for $\psi$.  
In particular, every focal laminar action is semi-conjugate  to a minimal focal laminar action. 
\end{prop}
\begin{proof}

Let $\mathcal{L}$ be an invariant covering prelamination for $\varphi$, and 
$h\colon \R\to \R$ a semi-conjugacy from $\varphi$ to $\psi$ (in the sense that \eqref{eq:semiconj} hold n $h$ is proper, as a consequence of irreducibility of $\varphi$ and $\psi$. Thus for every sufficiently large interval  $I\in \mathcal{L}$, $h(I)$ is not a singleton. Hence, the collection of the interiors of the convex hulls of the images $h(I)$, where $I$ runs through all the intervals $I\in \mathcal L$, gives a $\psi$-invariant covering prelamination. The statement on the type of subgroups follows easily from Proposition \ref{p-dyn-class-subgroups} since the properties that a subgroup be totally bounded, have compact set of fixed point, admit or not a minimal set, are all stable under semi-conjugacy. Finally, note that if $\varphi$ is focal, then the existence of a non-discrete minimal invariant set implies that it is semi-conjugate to a minimal action (Corollary \ref{cor.basica}), which must be focal by what we have just proved.
\end{proof}

Next, we observe that conversely, any \emph{minimal} laminar action, is automatically focal.  More precisely, we have the following result, which is an improvement of Remark \ref{rem.covering.prelamination}.

\begin{prop}\label{prop.minimalimpliesfocal}
Let $\varphi\colon G\to \homeo_0(\R)$ be a minimal action which preserves a prelamination $\mathcal{L}$. Then $\mathcal{L}$ is covering, thin, and the $\varphi$-orbit of every $I\in \mathcal{L}$ is cofinal.
In particular, $\varphi$ is focal and proximal. 
\end{prop}
\begin{proof}
Fix $I\in \mathcal{L}$. The family $\mathcal{L}_0=\{g.I: g\in G\}$ is also a $\varphi$-invariant prelamination. The union of elements of $\mathcal{L}_0$ is a non-empty $\varphi$-invariant open subset, so that by minimality of the action, it has to coincide with the whole real line. Remark \ref{rem.covering.prelamination} ensures that $\mathcal{L}_0$ is a covering prelamination. Hence so is $\mathcal{L}$, and it also follows that the orbit of $I$ is cofinal in $\mathcal{L}$. Since $I$ is arbitrary, this shows that the orbit of every $I\in \mathcal{L}$ is cofinal. Now, suppose that   $\mathcal{F}\subset \mathcal{L}$ is a maximal totally ordered family, and consider  the (possibly empty) open interval $J$ defined as the interior of the intersection $\bigcap_{I\in\mathcal{F}} I$. The $\varphi$-images of $J$ are pairwise equal or disjoint, so their union is a proper $\varphi$-invariant open subset, hence by minimality we must have $J=\varnothing$.  This shows that $\mathcal{L}$ is thin. Finally, this fact combined with cofinality of the action of $G$ on $\mathcal{L}$ implies that arbitrarily large intervals can be mapped into arbitrarily small neighborhoods of a given point, so that $\varphi$ is proximal.
\end{proof}

\begin{ex}
	The Plante-like actions in Example \ref{subsec.Plantefocal} are minimal and laminar, and so they are focal.
\end{ex}

\section{Horogradings}
We now introduce the notion of horograding, which will play a crucial role in our main theorem for actions of locally moving groups.

\subsection{Definitions}

\begin{dfn}
	Let $\mathcal{L}$ be a prelamination of an ordered set $(\Omega,\prec)$. A \emph{horograding} of $\mathcal{L}$ is a monotone map $\hor\colon (\mathcal{L},\subset) \to (\R,<)$. We say that a horograding $\hor\colon \mathcal L\to \R$ is \emph{positive} (respectively, \emph{negative}) if it is non-decreasing (respectively, non-increasing): if $I\subseteq J$, then $\hor(I)\le \hor(J)$ (respectively, $\hor(I)\ge \hor(J)$).
\end{dfn}

\begin{dfn}[Horogradings of actions]\label{d-horograding-actions}
Let $\psi\colon G\to \Aut(\Omega,\prec)$ be an order-preserving action admitting an invariant covering prelamination.	
A (positive or negative) \emph{prehorograding} of $\psi$ by an action $\rho\colon G\to \homeo_0(X)$, where $X\subset \R$ is an open interval,  is a pair $(\mathcal{L}, \hor)$ consisting of:
\begin{itemize}
\item a $\psi$-invariant covering prelamination $\mathcal{L}$; 
\item a (positive or negative) horograding $\hor\colon \mathcal{L} \to X$, such that $\hor(\psi(g)(I))=\rho(g)(\hor(I))$ for every $I\in \mathcal{L}$ and every $g\in G$.
\end{itemize}

In the case that $\Omega=\R$ and $\mathcal{L}$ is a lamination, we say that $(\mathcal{L},\mathfrak{h})$ is a \emph{horograding} of $\psi$ by $\rho$.
\end{dfn}

Note that although it is convenient to state this definition without any assumptions on $\varphi$  and $\rho$ (except laminarity of $\varphi)$, in practice we will often assume that $\varphi$ is \emph{focal} and that the horograding action  $\rho$ is irreducible. 
In this section, in order to discuss general properties of horogradings, we restrict for simplicity to the case $\Omega=\R$

When there is no ambiguity, we will use the shortcut notation $\hor(g.I)=g.\hor(I)$ for $\hor(\psi(g)(I))=\rho(g)(\hor(I))$ (and similar variations), assuming that it is clear that the action $\psi$ is on the source of $\hor$ (where points are denoted by Greek letters) and the action $\rho$ is on the image of $\hor$ (where points are denoted by Latin letters).

\begin{rem}\label{rem.prehor_dyn_real}
If an action $\psi\colon G\to \Aut(\Omega,\prec)$ has a prehorograding $(\mathcal L_0,\hor)$ by an action $\rho\colon G\to \homeo_0(\R)$, then the dynamical realization $\varphi\colon G\to \homeo_0(\R)$ of $\psi$ also has a prehorograding by $\rho$. Indeed, we have seen in Remark \ref{sc:inv_prelam_ordered_sets} that $\varphi$ preserves a covering prelamination $\mathcal L$, obtained by considering interiors of convex hulls of the images of leaves of $\mathcal L_0$ under the associated good embedding $i\colon (\Omega,\prec)\to (\R,<)$. Then $(\mathcal L,\hor\circ i^{-1})$ gives a prehorograding of $\varphi$ by $\rho$.
\end{rem}

\begin{rem} \label{r-extending-horograding}
A {positive (respectively, negative)} prehorograding $(\mathcal{L}, \hor)$ of a laminar action $\varphi$ by $\rho$ can be extended to a {positive (respectively, negative)} horograding defined on the closure  $\overline{\mathcal{L}}$, provided $\mathcal{L}$ is thin (this is always true if $\varphi$ is minimal, by Proposition \ref{prop.minimalimpliesfocal}). Indeed, in this case, every $J\in \overline{\mathcal{L}}$ contains some element of $\mathcal{L}$. It follows that the map $\hor$ can be extended, {in the positive case}, by setting $\hor(J)=\sup \{\hor(I): I\in \mathcal{L}, I\subset J\}$. 
\end{rem}

\begin{ex}[Self-horograding] \label{e-trivial-horograding}
Every laminar action $\varphi$ can be (positively) horograded by itself. Indeed, choose an invariant covering lamination $\mathcal{L}$, and define a horograding by $\hor(I)=\sup I$. 
\end{ex}

\begin{ex}[Horograding of Plante-like actions]\label{e-Plante-horograding}
Consider our running example of a Plante-like action of a permutational wreath product $H\wr_X G$, described in Example \ref{subsec.Plantefocal}. Fix a $G$-invariant order $<_X$ on $X$ and $<_H\in \LO(H)$ and, as in Example \ref{subsec.Plantefocal}, let $\Psi\colon H\wr_X G\to \Aut(\bigoplus_X H, \prec)$ be the associated Plante-like product (where $\prec$ is the order of lexicographic type on  $\bigoplus_X H$ associated with $<_X$ and $<_H$).
As in Example \ref{subsec.Plantefocal}, we assume that $G$, $H$, and $X$ are countable, and that the action of $G$ on $(X,<_X)$ is cofinal, so that we can take the dynamical realization $\varphi\colon H\wr_X G \to \homeo_0(\R)$, with associated good embedding $\iota\colon (\bigoplus_X H, \prec)\to \R$. Let also $\rho\colon G\to \homeo_0(\R)$ be the dynamical realization of the $G$-action on $(X,<_X$), with associated embedding $\iota_G\colon(X, <_X) \to \R$. We claim that the Plante-like action $\varphi$ is horograded by $\rho$ (seen as an action of $H \wr_X G$ which factors through the projection to $G$). To see this, recall that an invariant prelamination for $\Psi$ is given by $\mathcal{L}_0=\{C_{\s,y} : \s\in \bigoplus_XH, y\in X\}$,  where $C_{\s,y}\subset \bigoplus_XH$ is defined as in \eqref{e-prelamination-plante}. From \eqref{e-prelamination-plante} it is evident that the correspondence
\[
\dfcn{\hor_0}{\mathcal L_0}{(X,<_X)}{C_{\s,x}}{x}
\]
is monotone and $G$-equivariant. It follows that $(\mathcal L_0,\iota_G\circ \hor_0)$ defines a prehorograding of $\Psi$ by $\rho$. After Remark \ref{rem.prehor_dyn_real}, we obtain a prehorograding $(\mathcal L,\hor)$ of the dynamical realization $\varphi$ of $\Psi$ by the dynamical realization $\rho$; after Remark \ref{r-extending-horograding}, since $\varphi$ is minimal (as shown in Example \ref{subsec.Plantefocal}), the prehorograding can be extended to a horograding of $\varphi$ by $\rho$.
\end{ex}

Let us point out the following simple consequence of Definition \ref{d-horograding-actions}, that will be often used implicitly. 

\begin{lem} \label{l-horograding-infinity}
Let $\varphi, \rho\colon G\to \homeo_0(\R)$ be irreducible actions, and $(\mathcal{L}, \hor)$ a positive prehorograding of $\varphi$ by an irreducible action $\rho$. If $(I_n)\subset \mathcal{L}$ is any increasing exhaustion of $\R$, then  $\hor(I_n)\to +\infty$ as $n\to +\infty$.  \end{lem}
\begin{proof}
Since $\hor(\mathcal{L})$ is $\rho$-invariant and $\rho$ is irreducible, we must have $\sup \hor(\mathcal{L})=+\infty$. Since $\hor$ is non-decreasing and $(I_n)$ is a cofinal sequence in $\mathcal{L}$, the conclusion follows.
\end{proof}

In contrast, it is \emph{not} true that $\hor(I_n)$ tends to $-\infty$ whenever $(I_n)$ is a decreasing sequence in $\mathcal{L}$ with diameter tending to $0$. Consider for example the self-horograding  $(\mathcal{L}, \hor)$ of a focal laminar action by itself (Example \ref{e-trivial-horograding}). Then for every decreasing sequence $(I_n)$ in $\mathcal{L}$, the sequence $\hor(I_n)=\sup I_n$ cannot converge to $-\infty$, since $\hor(I_n)\ge  \inf I_0$ for every $n$. It is therefore convenient to introduce the following terminology.
\begin{dfn}[$\hor$-complete points] \label{d-pi-complete-point-line}
Let $(\mathcal{L}, \hor)$ be a positive horograding of  $\varphi\colon G\to \homeo_0(\R)$ by $\rho\colon G\to \homeo_0(\R)$. A point $\xi\in \R$ is $\hor$-\emph{complete} if there exists a decreasing sequence $(I_n)\subset \mathcal{L}$ such that $\bigcap I_n=\{\xi\}$ and $\lim \hor(I_n)= -\infty$. The analogous definition can be given for negative horogradings.

(When the horograding action $\rho$ is on an open interval $X\subsetneq \R$, this definition should be modified by requiring convergence to the appropriate endpoint of $X$.) 
\end{dfn}

\begin{rem} \label{pi-complete-comeager} The set $\Xi_\hor$ of $\hor$-complete points is always a $G_\delta$-subset of $\R$, i.e.\ a countable intersection of open sets. For instance, in the case of a positive horograding, we have 
	\[\Xi_\hor= \textstyle\bigcap_{n\ge 0} \bigcup_{I\in \mathcal L,\,\hor(I)\le -n} I.\]
It is also clearly $\varphi$-invariant. It follows that when the action $\varphi$ is minimal, the subset $\Xi_\hor$ is either empty or {residual}, i.e.\ generic in the sense of Baire.
\end{rem}

\subsection{Classification of elements in the horograded case}\label{ssec.class_elts}
A particularly useful consequence of the existence of a horograding of a laminar action $\varphi$ by an irreducible action $\rho$, is that the type of each element for $\varphi$ is determined by $\rho$.

\begin{prop}[Dynamical classification of elements in the horograded case]\label{p-dyn-class-elements-horograded}
Assume that a laminar action $\varphi\colon G\to \homeo_0(\R)$ is positively horograded by an irreducible action $\rho:G\to\homeo_{0}(\R)$. Then for every element $g\in G$, the following hold.
\begin{itemize}
\item $\varphi(g)$ is totally bounded if and only if $\rho(g)$ has fixed points accumulating on $+\infty$.
\item Otherwise, $\varphi(g)$ is a pseudo-homothety, which is expanding if $\rho(g)(x)>x$ for every sufficiently large $x\in \R$, and contracting otherwise. Moreover if $\rho(g)$ has no fixed point, then $\varphi(g)$ is a homothety, and its unique fixed point is $\hor$-complete.
\end{itemize}
\end{prop}

Instead of proving Proposition \ref{p-dyn-class-elements-horograded} directly, we shall deduce it from the following two lemmas, that give similar conclusions for more general subgroups.

\begin{lem} \label{l-totally-bounded-subgroups-horograding}
Assume that a laminar action $\varphi\colon G\to \homeo_0(\R)$ is positively horograded by an irreducible action $\rho:G\to\homeo_{0}(\R)$. Then, a finitely generated subgroup $H\subseteq G$ is totally bounded for $\varphi$, if and only if $\fix^{\rho}(H)$ accumulates on $+\infty$. 
\end{lem}
\begin{proof}
Let $(\mathcal{L}, \hor)$ be a positive horograding of $\varphi$ by $\rho$. By Proposition \ref{p-dyn-class-subgroups},  $H$ is totally bounded for $\varphi$ if and only if the subset $\mathcal{L}^H$ of $H$ invariant leaves is cofinal. If this is the case, then $\hor(\mathcal{L}^H)$ consists of $\rho(H)$-fixed points, which by Lemma \ref{l-horograding-infinity}, accumulate on $+\infty$.\footnote{Note that this implication does not use that $H$ is finitely generated.} For the converse, assume that $H$ is not totally bounded. By finite generation, Proposition \ref{p-dyn-class-subgroups} gives that it is either pseudo-homothetic or focal. In either case, we can find $I\in \mathcal{L}$ and a sequence $(h_n)\subset H$ such that $h_n.I$ is an increasing exhaustion of $\R$. Then by Lemma \ref{l-horograding-infinity}, the images $h_n.\hor(I)=\hor(h_n.I)$ must converge to $+\infty$,  and so $\rho(H)$ cannot have fixed points accumulating to $+\infty$.
\end{proof}

\begin{lem} \label{l-pseudo-homothetic-horograding}
Assume that a laminar action $\varphi\colon G\to \homeo_0(\R)$ is positively horograded by an irreducible action $\rho:G\to\homeo_{0}(\R)$. Let $H\subseteq G$ be a subgroup which is pseudo-homothetic for $\varphi$, and denote by $J^{-}$ and $J^{+}$ the two unbounded connected components of $\suppphi(H)$ (see Proposition \ref{p-dyn-class-subgroups}). Then the following hold.
\begin{itemize}
\item $\supp^\rho(H)$ has a connected component adjacent to $+\infty$, denoted as $U$.
\item  the $\varphi$-action of $H$ on  $J^{+}$ (respectively, $J^{-})$ is positively (respectively, negatively) semi-conjugate to its $\rho$-action on $U$. \end{itemize} 

Moreover, if $\supp^\rho(H)=\R$, then $\varphi(H)$ is homothetic, and its unique fixed point is $\hor$-complete. 
\end{lem}
\begin{proof}
The proof is similar to part \ref{i-pseudo-homothetic} of the proof of Proposition \ref{p-dyn-class-subgroups}. Let $(\mathcal{L}, \hor)$ be a positive horograding. Set $K=\fixphi(H)$, and consider the collection $\mathcal{L}_K=\{I\in \mathcal{L} : K\subset I\}$, which is totally ordered and $\varphi(H)$-invariant. Set $c=\inf \hor(\mathcal{L}_K)$. It follows from  Proposition \ref{p-dyn-class-subgroups} that every $I\in\mathcal{L}_K$ has a cofinal orbit for $\varphi(H)$, and can also be mapped arbitrarily close to $K$. Hence the $\rho(H)$-orbit of $\hor(I)$ accumulates on $c$ and $+\infty$. It follows that $U=(c, +\infty)$ is a connected component of $\supp^\rho(H)$. For $\xi\in J^{\pm}$, define $I_\xi=\inf \{I\in \mathcal{L}_K : \xi\in \overline{I}\}$. Then the maps $J^{\pm}\to U$, $\xi\mapsto \hor(I_\xi)$, are the desired semi-conjugacies. Finally if $U=\R$, i.e.\ if $c=-\infty$, it follows that $\mathcal{L}_K$ has no lower bound in $\mathcal{L}$, and this is possible only if $K$ is reduced to a single point, which is $\hor$-complete.
\end{proof}

\begin{proof}[Proof of Proposition \ref{p-dyn-class-elements-horograded}]
Since every cyclic subgroup of $G$ is either totally bounded or pseudo-homothetic (Proposition \ref{p.dynclasselements}), the proposition follows from Lemmas \ref{l-totally-bounded-subgroups-horograding} and \ref{l-pseudo-homothetic-horograding}.
\end{proof}

\begin{rem}
In general, the type of a subgroup $H\subseteq G$ with respect to a laminar action $\varphi$, according to Proposition \ref{p-dyn-class-subgroups}, is not determined uniquely by its image in the horograding action $\rho$. Elaborating on the previous arguments, one can see the following.
\begin{itemize}
\item If $\fix^\rho(H)$ accumulates on $+\infty$, then $H$ is either totally bounded or horocyclic. Both cases are possible if $H$ is not finitely generated. 
For instance in the case of a Plante-like action of $H\wr_X G$ (see Example \ref{subsec.Plantefocal}), the subgroup $\bigoplus_X H$ is horocyclic, but $\rho(\bigoplus_X H)=\{\id\}=\rho(\{\mathsf e\})$. 
\item  If $\fix^\rho(H)$ does not accumulate on $+\infty$, then $H$ can be either focal or pseudo-homothetic. Again, both cases are possible. In the case of a Plante-like action, we have $\rho(H\wr_X G)=\rho(G)$; the first is focal and the second is pseudo-homothetic.
\end{itemize}
\end{rem}

In \S \ref{sec.largescale}, we will  describe a more global relation between the large-scale dynamics of a horograded laminar action and the horograding action, which will be important for the results in Part \ref{partIII}.

\section{Finding invariant laminations}\label{ss.focal_condition}
In this final section we describe a mechanism to get an invariant lamination for an action $\varphi\colon G\to \homeo_0(\R)$, that will be used in the proof of the main results of this part. 

\subsection{General criterion}
The central notion of our method is the following. 
\begin{dfn} \label{d-phi-convex} Let $\varphi\colon G\to\homeo_0(\R)$ be an action and $H\subseteq G$ a subgroup. We say that a non-empty interval $I\subset \R$ is an \emph{irreducible wandering interval} for $\varphi(H)$ if the following hold:\begin{enumerate}[label=($\mathcal{W}$\arabic*)]
\item\label{cond0} $I$ is bounded and open,
\item\label{cond1} $I$ is wandering for $\varphi\vert_H$ (that is, for every $h\in H$ either $h.I=I$ or $h.I\cap I=\emptyset$),
\item\label{cond2} the action of $\Stphi_H(I):=\stab^{\varphi\vert_H}_{H}(I)$ on $I$ has no fixed points. 
\end{enumerate}
We denote by $\mathcal{W}_\varphi(H)$ the family of irreducible wandering intervals for $\varphi(H)$.
\end{dfn}

If $\mathcal{W}_\varphi(G)\neq \varnothing$, then $\varphi$ cannot be minimal, as the union of all images of any $I\in \mathcal{W}_\varphi(G)$ is a $\varphi$-invariant proper open subset of $\R$. Conversely, we have the following criterion for finding irreducible wandering intervals. 

\begin{lem}\label{lem.nonempty}
	Let $\varphi\colon G\to\homeo_0(\R)$ be an action with no minimal invariant set. Then every $\xi\in \R$ belongs to an irreducible wandering interval for $\varphi(G)$.
\end{lem} 
\begin{proof}
	Note that $\varphi$ is irreducible, because any global fixed point defines a minimal invariant set. Then for any $\xi\in\R$, we can find an element $g\in G$ such that $g.\xi> \xi$. Since $\varphi$ has no minimal invariant set, there exists a closed non-empty invariant subset $D$ that does not intersect $[\xi,g.\xi]$ (otherwise we could use Lemma \ref{l.zorn_minimal} to find a minimal invariant set). Denote by $J$ the component of $\R\setminus D$ containing $[\xi,g.\xi]$, and set $H:=\Stphi(J)$. Note that $H$ is non-trivial, as it contains the element $g$. Then, we can take the component $J_0\subseteq J$ of $\suppphi(H)$ that contains $\xi$. We claim that $J_0\in\mathcal{W}_\varphi(G)$. We first check that condition \ref{cond1} is satisfied for $J_0$. To see this, take $h\in G$. Since $D$ is $\varphi$-invariant and $J$ is a component of $\R\setminus D$, we have that either $h.J=J$ or $h.J\cap J=\emptyset$. In the second case we have $h.J_0\cap J_0=\emptyset$. In the first case, we have that $h\in H$, and therefore $h$ preserves $J_0$, since it is a component of the support of the action of $H$. Finally, condition \ref{cond2} holds because $J_0$ is a component of the support of $H$. 
\end{proof}

\begin{lem}\label{l-wandering-dont-cross}
	Let $\varphi\colon G\to\homeo_0(\R)$ be an action, and consider two subgroups $H_1$ and $H_2$ of $G$, and distinct wandering intervals $J_i\in\mathcal{W}_\varphi(H_i)$, for $i\in \{1,2\}$. Then, the following hold.
\begin{enumerate}[label=(\roman*)]
\item \label{i-wandering-lemma-1} if $H_1\subseteq H_2$, then $J_1$ and $J_2$ do not cross. 
\item \label{i-wandering-lemma-2} If $H_1$ and $H_2$ commute, then $J_1$ and $J_2$ do not cross. 
\end{enumerate}
In particular, any two distinct wandering intervals in $\mathcal{W}_\varphi(G)$ do not cross.
\end{lem}

\begin{proof}
To prove \ref{i-wandering-lemma-1}, set $J_1=(a,b)$ and $J_2=(c,d)$. Suppose by contradiction that they are crossed, and assume without loss of generality that $a<c<b<d$. Applying condition \ref{cond2} to $J_2$, we find $g\in H_1$ such that $g.J_1=J_1$ and $g.c\neq c$. Since $g.b=b$ and $g\in H_1\subseteq H_2$, condition \ref{cond1} applied to $J_2$ implies that $g.J_2=J_2$. This contradicts the fact that $g.c\neq c$. Thus, $J_1$ and $J_2$ do not cross. To show \ref{i-wandering-lemma-2}, suppose by contradiction that $J_1$ and $J_2$ are crossed. Let $\xi$ be the endpoint of $J_1$ contained in $J_2$; by condition \ref{cond2}, there exists $g\in \Stphi_{H_2}(J_2)$ such that $g.\xi\in J_1$. However, $\xi$ is fixed by $\Stphi_{H_1}(J_1)$, while $g.\xi$ is not (again by \ref{cond2}), contradicting the assumption that $H_1$ and $H_2$ commute.
\end{proof}

The combination of Lemma \ref{l-wandering-dont-cross} and  Proposition \ref{prop.minimalimpliesfocal} immediately implies the following.

\begin{prop}[Criterion for laminarity] \label{p-commuting-subgroup-lamination}
Let $H\subseteq G$ be a subgroup with the property that any two conjugates of $H$ in $G$ commute or are related by inclusion, and suppose that $\varphi\colon G\to \homeo_0(\R)$ is an action such that $\mathcal{W}_\varphi(H)\neq \varnothing$. Then $\mathcal{L}=\bigcup_{g\in G} \mathcal{W}_\varphi({gHg^{-1}})$ is a $\varphi$-invariant prelamination. In particular if $\varphi$ is minimal, then it is focal.
\end{prop}

\subsection{The case of normal subgroups}

It is worth pointing out the following special case of Proposition \ref{p-commuting-subgroup-lamination}.

\begin{prop}\label{prop.maximal}
	Consider a faithful minimal action $\varphi\colon G\to\homeo_{0}(\R)$, and let $N\triangleleft G$ be a normal subgroup which is not a cyclic subgroup of the center of $G$. Assume  that the  $\varphi$-action of $N$ is not minimal. Then $\mathcal{W}_\varphi(N)$ is a $\varphi$-invariant prelamination; in particular, $\varphi$ is focal. 
\end{prop}
First we recall the following lemma. Its proof is well known but we include it for completeness. 
\begin{lem}\label{l-normal-minimal}
	Consider a faithful minimal  action $\varphi\colon G\to\homeo_{0}(\R)$, and let $N\triangleleft G$ be a normal subgroup which is not a cyclic subgroup of the center of $G$. Then, either the image of $N$ acts minimally, or it admits no  minimal invariant set. 
\end{lem}
\begin{proof}
	Assume that there exists a minimal $\varphi(N)$-invariant set  $\Lambda\subseteq \R$, and let us show that $\Lambda=\R$. Note that $\Lambda$ cannot be a fixed point otherwise, as $N$ is normal and $\varphi$ is minimal, this would imply that $N$ acts trivially, which is an absurd because we are assuming that $\varphi$ is faithful (see\ Corollary \ref{cor.normalsemicon} for details).
	Then $\Lambda$  is either the unique minimal $\varphi(N)$-invariant set, or a discrete orbit. In the first case, we have that $\Lambda$ is preserved by the whole group $G$ and thus $\Lambda=\R$ by minimality. In the second case we have that the action of $N$ must be semi-conjugate to a cyclic action coming from a homomorphism $\tau\colon N\to \Z$, and that $\ker \tau$ acts trivially on $\Lambda$. Since $\ker \tau$ is precisely the subset of $N$ acting with fixed points, the subgroup $\ker \tau$ is necessarily normal in  the whole group $G$, and so, as before, we must have $\ker \tau=\{1\}$. Thus $N$ is infinite cyclic and acts as a group of translations. Since $N$ is normal in $G$, this implies that $N$ is central in $G$, contradicting the assumption.
\end{proof}

\begin{proof}[Proof of Proposition \ref{prop.maximal}] 
	By Lemma \ref{l-normal-minimal}, we have that $\varphi|_N$ does not admit any minimal invariant set. Hence, by Lemma \ref{lem.nonempty}, the set $\mathcal{W}_\varphi(N)$ is non-empty, and Proposition \ref{p-commuting-subgroup-lamination} gives that $\mathcal{W}_\varphi(N)$ defines a $\varphi$-invariant prelamination, and that $\varphi$ is focal.
\end{proof}

\begin{rem}
	The converse to Proposition \ref{prop.maximal} is not true, as we will see with an example described in \S \ref{s-F-hyperexotic} that gives a minimal laminar action for the simple group $[F,F]$ (where $F$ is Thompson's group). See however Proposition \ref{p-focal-simplicial} for a partial converse to Proposition \ref{prop.maximal}.
\end{rem}

\subsection{An application to solvable groups} \label{ssc.solvable}
We conclude this section with the following application of Proposition \ref{prop.maximal}, which shows that  laminar  actions appear naturally in the context of solvable groups. Its derivation from Proposition \ref{prop.maximal} uses ideas of Rivas and Tessera in \cite{Rivas-Tessera}. By an affine action  of $G$ on $\R$ we mean an action by affine transformations, i.e.\ transformations of the form $x\mapsto ax+b$ with $a>0$ and $b\in \R$. 

\begin{thm}[Alternative for solvable groups]\label{thm.notaffine} Let $\varphi\colon G\to\homeo_0(\R)$ be an irreducible action of a finitely generated solvable group. Then, either $\varphi$ is semi-conjugate to an affine action, or it is laminar and focal.
\end{thm}
\begin{proof} If $\varphi$ is semi-conjugate to a cyclic action, then the first case holds. So up to semi-conjugacy, we can assume that $\varphi$ is minimal; moreover upon replacing $G$ by a quotient we can suppose that it is faithful. Let $G^{(n)}$ denote the derived series of $G$, and consider $k\in\N$ so that $G^{(k)}\neq\{1\}$ and $G^{(k+1)}=\{1\}$. Then $H:=G^{(k)}$ is an abelian normal subgroup of $G$. As $H$ is normal, we deduce that $\varphi(H)$ cannot have fixed points (otherwise, as $\fixphi(H)$ is a closed invariant subset,  we would get that $H$ acts trivially).
	
	Assume first that $H$ is a cyclic subgroup of the center of $G$. Then $\varphi(H)$ is conjugate to the cyclic group generated by a translation and the action $\varphi$ induces a minimal action on the circle $\R/\varphi(H)$. As $G$ is solvable and thus amenable, it preserves a Borel probability measure on the circle, which must be of total support for the action is minimal; as it is well known (see for instance Navas \cite[Proposition 1.1.1]{Navas-book}), this gives that the action is conjugate to a minimal action by rotations (basically, this is the action defined by the rotation number homomorphism $\mathsf{rot}\colon G\to \T$). Therefore, $\varphi(G)$ is the lift of a group of rotations of the circle, thus conjugate to a group of translations (and so of affine transformations).
	
	Assume next that  $H$ acts minimally. Take a non-trivial $h\in H$ and note as before that $\varphi(h)$ has no  fixed points (otherwise we would get that $H$ acts trivially). Thus, we obtain a minimal action of $H$ on the circle $\R/\langle \varphi(h)\rangle$, and the argument for the previous case leads to the conclusion that $\varphi(H)$ is conjugate to a minimal group of translations. Then, the set of invariant Radon measures for $\varphi(H)$ corresponds to the one-parameter family of positive multiples of the Lebesgue measure, and this family must be preserved by $\varphi(G)$, for $H$ is normal in $G$. By a standard argument, we deduce that $\varphi(G)$  is conjugate to a group of affine transformations (see for instance Plante \cite{Plante}).
	
	If neither of the previous cases hold, Proposition \ref{prop.maximal} gives that $\varphi$ is laminar and focal, as desired.
\end{proof}

Since in a minimal laminar action every element has fixed points (Lemma \ref{rem.fixedepoints}), we deduce the following result first obtained by Guelman and Rivas in \cite{Guelman-Rivas}. 

\begin{cor} Let $G$ be a finitely generated solvable group, and let $\varphi\colon G\to \homeo_0(\R)$ be an action such that some element acts without fixed points. Then, the action $\varphi$ is semi-conjugate to an affine action. 
\end{cor}

\chapter{Laminations and micro-supported groups} 
\label{sec.locandfoc}

This chapter contains the main results of Part \ref{partII}, which describe the dynamics of exotic actions of micro-supported and locally moving groups on the line, using invariant laminations and horogradings. We refer to Chapters \ref{sec.locallymgeneral} and \ref{s-lm-2} for basic definitions and results used throughout this chapter.
In a first result (Theorem \ref{t-laminations-microsupported}), we shall show that if $G\subseteq \homeo_0(X)$ is a micro-supported group acting minimally on an interval $X$, then all faithful minimal actions of $G$ on the line must be laminar, with at most one exception.  Our main result (Theorem \ref{t-lm-horograding}) implies that if $G\subset \homeo_0(X)$ is in addition finitely generated and fragmentable, then any  faithful minimal action of $G$ on the line is either conjugate to the standard action on $X$, or laminar and horograded by it.  

\section{Most actions of micro-supported groups preserve laminations} 

Let $G\subseteq \homeo_0(X)$ be a micro-supported subgroup acting minimally on an open interval $X$ (this is equivalent to the existence of one non-trivial element of relatively compact support, see Proposition \ref{p-micro-compact}). When $G$ is locally moving, the standard action of $G$ cannot preserve any lamination. Indeed, for any bounded\footnote{Recall our convention that ``bounded'' stands for ``relatively compact'' here.} open interval $I=(a_1, a_2)$, one can choose two disjoint intervals $J_1$ and  $J_2$, such that $a_i\in J_i$ for $i\in \{1, 2\}$. As $G$ is locally moving, for each $i\in \{1,2\}$ there exists $g_i\in G_{J_i}$ such that $g_i(a_i)>a_i$, so that $g_1g_2(I)$ crosses $I$. Hence, no bounded open interval $I$ can belong to an invariant lamination.  We will show the following. 

\begin{thm}[Laminar/locally moving alternative]\label{t-laminations-microsupported}
	Let $G\subseteq \homeo_0(X)$ be a micro-supported subgroup acting minimally on an open interval $X$.  Then the following hold.
	\begin{enumerate}[label=(\roman*)]
		\item \label{i-micro-dicotomia} The standard action of $G$ on $X$ is either laminar or locally moving.
		\item \label{i-micro-exotic} Every faithful minimal action $\varphi\colon G\to \homeo_0(\R)$ which is not conjugate to the standard action, is laminar.
	\end{enumerate}
\end{thm}
We first need to discuss a couple of lemmas in preparation for the proof. The following result is essentially Proposition \ref{p-centralizer-fix}, which we restate here in a more explicit form, for the convenience of the reader.

\begin{lem}\label{l-direct-product-explicit}
	Let $\varphi\colon \Gamma_1\times \Gamma_2\to \homeo_0(\R)$ be an action of a direct product, and suppose that $\varphi|_{\Gamma_1}$ is irreducible and admits a minimal invariant set $\Lambda$. Then, either
	\begin{enumerate}[label=(\roman*)] 
		\item \label{i-product-explicit-1} the action of $\Gamma_1$ on $\Lambda$ factors through a free action of an abelian quotient of $\Gamma_1$, or
		\item \label{i-product-explicit-2}  $\Gamma_2$ preserves $\Lambda$, its action on $\Lambda$ is non-trivial and factors through a free action of an abelian quotient of $\Gamma_2$, or
		\item \label{i-product-explicit-3}  $\Gamma_2$ fixes $\Lambda$ pointwise. 
	\end{enumerate}
	In particular, in either case there exists $i\in \{1,2\}$ such that $[\Gamma_i, \Gamma_i]$ fixes $\Lambda$ pointwise. 
\end{lem}
\begin{proof}
	This is a consequence of Theorem \ref{t-centralizer}, the proof is identical to the proof of Proposition \ref{p-centralizer-fix} (note that the finite generation of $\Gamma_1$ was used there only to ensure the existence of a minimal invariant set).
\end{proof}

\begin{lem}\label{lem.dicotomicro}
	For $X=(a, b)$, let $G\subseteq\homeo_0(X)$ be a micro-supported subgroup whose action on $X$ is minimal. Suppose that for any $x\in X$, the subgroups $G_{(a,x)}$ and $G_{(x,b)}$ act without fixed points on $(a,x)$ and $(x,b)$, respectively. Then $G$ is locally moving. 
\end{lem}
\begin{proof}
	Take a bounded open interval $I=(u,v)\Subset X$ and $x\in I$. We need to show that $g(x)\neq x$ for some $g\in G_I$.	Since $G$ is micro-supported, the subgroup $G_c$ of compactly supported elements is non-trivial. Since the action of $G$ is minimal and $G_c$ is normal, we have that $G_c$ has no  fixed points, and thus there exists $h\in G_c$ such that {$h(u)> v$}. Consider a bounded interval $(u_0,v_0)$ containing $\supp(h)\Supset I$. From the assumption, we can find elements $k_1\in G_{(a,x)}$ and $k_2\in G_{(x,b)}$ such that $k_1(u_0)\in (u,x)$ and $k_2(v_0)\in (x,v)$. Writing $k=k_1k_2$, we have  $g=khk^{-1}\in G_I$ and $g(x)\neq x$, as desired.
\end{proof}

\begin{proof}[Proof of Theorem \ref{t-laminations-microsupported}]
	We first prove \ref{i-micro-dicotomia}. Let us denote by $\iota \colon G\to \homeo_0(X)$ the standard action, and suppose that $G$ is not locally moving. After Lemma \ref{lem.dicotomicro}, we can assume that there exists $x\in X$ such that $G_{(a,x)}$ has fixed points in $(a, x)$ (the symmetric case, in which this holds for some {$G_{(x, b)}$}, can be treated similarly). This implies that $\supp(G_{(a,x)})$ has a bounded connected component  $J$, which  in particular belongs to the set $\mathcal{W}_\iota(G_{(a, x)})$ of irreducible wandering intervals. Since all conjugates of $G_{(a, x)}$ are related by inclusion, we conclude by applying Proposition \ref{p-commuting-subgroup-lamination}.
	
	We next prove \ref{i-micro-exotic} under the assumption that the standard action of $G$ on $X$ is locally moving. We want to use Proposition \ref{p-commuting-subgroup-lamination} again. Note that the subgroups of the form $G_{(x, b)}$ or $G_{(a, x)}$ satisfy the condition of Proposition  \ref{p-commuting-subgroup-lamination} (as their conjugates are related by inclusion), and so do their commutator subgroups. So it is enough to find one subgroup $H$ of this type such that $\mathcal{W}_\varphi(H)\neq \varnothing$. Fix $x\in X$. By Proposition \ref{p-lm-reconstruction}, we can assume without loss of generality that $\fixphi(G_{(x,b)})=\varnothing$.  If $G_{(x,b)}$ does not admit any minimal invariant set,  we conclude by Lemma \ref{lem.nonempty}. Else, if $\Lambda\subset \R$ is a minimal invariant set for $G_{(x,b)}$, we deduce from Lemma \ref{l-direct-product-explicit} that $\Lambda$ is fixed  by $H=[G_{(x, b)}, G_{(x, b)}]$ or $H=[G_{(a, x)}, G_{(a, x)}]$. In either case, any connected component $J$ of $\suppphi(H)$ belongs to $\mathcal{W}_\varphi(H)$.
	
	Finally, we prove  \ref{i-micro-exotic} under the assumption that the standard action on $X$ is not locally moving. By part \ref{i-micro-dicotomia}, we can find a covering lamination  $\mathcal{L}$, invariant under the standard action. For every $I\in \mathcal{L}$, the subgroup $G_I$ satisfies the condition of Proposition \ref{p-commuting-subgroup-lamination}, and so does $[G_I, G_I]$. Hence it is enough to find $I\in \mathcal{L}$ such that  $\mathcal{W}_\varphi(H)\neq \varnothing$ for $H=G_I$ or $H=[G_I, G_I]$.  Fix $I\in \mathcal{L}$. If $\suppphi(G_{I})$ has a bounded connected component, say $J$, then $J\in \mathcal{W}_\varphi({G_{I}})$. Else, assume that every connected component of $\suppphi(G_I)$ is unbounded above or below (and thus there are at most two of them). If the action of $G_I$ on one such component $J$ has no minimal invariant set, we can still conclude that $\mathcal{W}_\varphi({G_I})\neq \varnothing$ by Lemma \ref{lem.nonempty}. Finally, suppose that the action of $G_I$ on an unbounded component $J$ has a minimal invariant set $\Lambda$. Choose $g\in G$ such that $g(I)\cap I=\varnothing$. Then $g.J$ is a component of $\suppphi(G_{g(I)})$, which intersects $J$  non-trivially (as it is unbounded). Since $[G_{g(I)},G_I]=\{\id\}$, we deduce that $g.J=J$, so that we can apply Lemma \ref{l-direct-product-explicit} to the action of $G_I\times G_{g(I)}$ on $J$. We deduce that $H=[G_J, G_J]$ or $H=[G_{g(I)}, G_{g(I)}]$ fixes $\Lambda$ pointwise. In either case $\suppphi(H)$ has bounded connected components, so that $\mathcal{W}_\varphi(H)\neq \varnothing$.	
\end{proof}

\section{Horograding by the standard action}\label{sec-horobystandard}

\subsection{Statement of the main result}
Recall (from \S \ref{sc.defi_lm}) that for a subgroup $G\subset \homeo_0(X)$, where $X=(a, b)$, we denote by $G_+$ (respectively, $G_-$) the normal subgroup of elements acting trivially on some neighborhood of $b$ (respectively, $a$), and $\Gfrag=G_-G_+$ is the fragmentable subgroup of $G$. This subgroup plays an important role in the theorem below. 

\begin{thm}[Horograding theorem]\label{t-lm-horograding}
	For $X=(a, b)$, let $G\subset \homeo_0(X)$ be a  subgroup acting minimally on $X$, with $\Gfrag$ non-trivial and finitely generated. Then $G$  is locally moving, and every faithful minimal  action  $\varphi\colon G\to \homeo_0(\R)$ is either conjugate to the standard action on $X$, or  laminar, horograded by the standard action on $X$.
\end{thm}

The finite generation of $\Gfrag$ allows to prove the following corollary, providing a classification for irreducible actions of $G$, up to semi-conjugacy.
%In view of the normal subgroup structure of locally moving groups (Proposition \ref{p-micro-normal}), we get the following classification for irreducible actions of $G$, up to semi-conjugacy. 

\begin{cor} \label{c-lm-semiconj}
	For $X=(a, b)$, let $G\subset \homeo_0(X)$ be a  subgroup acting minimally on $X$, with $\Gfrag$ non-trivial and finitely generated. Then, every irreducible action $\varphi\colon G\to \homeo_0(\R)$ is semi-conjugate to an action in one of the following families.
	\begin{itemize}
		\item \emph{(Induced from a quotient)}. An action induced from the largest quotient $G/[G_c, G_c]$.
		\item \emph{(Standard)}. The standard action on $X$.
		
		\item \emph{(Horograded by standard)}.\label{t-C-i-focal}  A minimal laminar action, horograded by the standard action on $X$.
	\end{itemize}
\end{cor}

\begin{proof}[Proof of Corollary \ref{c-lm-semiconj} assuming Theorem \ref{t-lm-horograding}]
	Some argument is needed because $G$ itself is not supposed to be finitely generated, so that its actions need not be all semi-conjugate to a minimal or cyclic action. Let $\varphi\colon G\to \homeo_0(\R)$ be an irreducible action. Suppose at first that $\varphi|_{\Gfrag}$ is irreducible. Since $\Gfrag$ is finitely generated, there exists a compact interval $I\subset \R$ intersecting every $\varphi(\Gfrag)$-orbit, and thus every $\varphi(G)$-closed invariant subset. Hence $\varphi$ has a minimal invariant set (by Lemma \ref{l.zorn_minimal}), and thus it is semi-conjugate to a minimal or cyclic action; we conclude using Theorem \ref{t-lm-horograding}. If $\varphi(\Gfrag)$ has a fixed point, then $\varphi$ is semi-conjugate to a non-faithful action (Corollary \ref{cor.normalsemicon}). Finally, any non-faithful action must factor through $G/[G_c, G_c]$ by the normal subgroup structure of locally moving groups (Proposition \ref{p-micro-normal}).
\end{proof}

			\begin{rem}\label{r.type_elt_horograding} A concrete consequence of Theorem \ref{t-lm-horograding} is that for all  faithful minimal actions of $G$, the dynamics of individual elements is determined by the way they act in the standard action, via Proposition \ref{p-dyn-class-elements-horograded}. For this, consider the case where $\varphi$ is laminar, positively horograded by the standard action of $G$ on $X$ (the other case is analogous). Then, we have the following: \begin{itemize}
					\item if $\fix(g)$ accumulates on $b$, then $\varphi(g)$ is totally bounded;
					\item if $g(x)>x$ (respectively, $g(x)<x$) on a neighborhood of $b$, then $\varphi$ is an expanding (respectively, contracting) pseudo-homothety;
					\item if $g(x)>x$ (respectively, $g(x)<x$) for every $x\in X$, then $\varphi(g)$ is an expanding (respectively, contracting) homothety. 
				\end{itemize}
			\end{rem}
			
			\begin{rem}
				Let us comment on the assumption of finite generation of $\Gfrag$, which is crucial for Theorem \ref{t-lm-horograding}.  A common reason why this assumption is satisfied is if $G$ is itself finitely generated and fragmentable (that is, $G=\Gfrag$), which boils down to say that $G$ can be generated by finitely many elements supported in strict subintervals of $X$. However, if $G$ is a locally moving group such that $G\neq\Gfrag$, the assumption of finite generation of $\Gfrag$ cannot be replaced by finite generation of $G$: in \S \ref{Sec_BS_transcendal}, we will provide an example of a minimal laminar action of a finitely generated locally moving group $G$ (with $\Gfrag$ not finitely generated), which cannot be horograded by the standard action of $G$. 
			\end{rem}

			\subsection{Overall strategy}
			The rest of this section is devoted to the proof of Theorem \ref{t-lm-horograding}. We will divide the proof into steps, which individually yield some additional information. Before getting into details, let us provide an overview of the strategy.  An application of Theorem \ref{t-laminations-microsupported} shows that a group $G$ as in Theorem \ref{t-lm-horograding} is locally moving, and every  faithful minimal action $\varphi\colon G\to \homeo_0(\R)$ which is not conjugate to the standard action must be laminar. To find a horograding, we will analyze the type of the subgroups $G_{(a, x)}$ and $G_{(x, b)}$, according to the classification of laminar actions in Proposition \ref{p-dyn-class-subgroups} (this does not depend on $x\in X$). We shall use the assumption that $\Gfrag$ is finitely generated to show that one of these two subgroups, say $G_{(a, x)}$, must be totally bounded. This allows to introduce a natural prelamination consisting of the union $\mathcal{L}=\bigcup_x\mathcal{L}_x$, where $\mathcal{L}_x$ is the collection of components of $\suppphi(G_{(a, x)})$. The horograding is then constructed by mapping every $I\in \mathcal{L}_x$ to the point $x$. Let us start. 
			
			\subsection{Type of subgroups}
			
			We begin with the following general result on laminar actions of locally moving groups (note that we do not need here the assumption that $\Gfrag$ is finitely generated).
			\begin{prop}\label{p-lm-domination-balance}
				For $X=(a, b)$, let $G\subseteq \homeo_0(X)$ be locally moving, and $\varphi\colon G\to \homeo_0(\R)$ a laminar action which is not semi-conjugate to any action of $G/[G_c,G_c]$. Then the type of the subgroups $G_{(a, x)}$ and $G_{(x, b)}$ (according to  the classification of laminar actions in Proposition \ref{p-dyn-class-subgroups}) does not depend on $x\in X$. Moreover, one of the following holds.
				
				\begin{enumerate}[label=(\roman*)]
					\item \emph{(One-sided domination)}. \label{i-domination} One of the subgroups $G_{(a, x)}$, $G_{(x, b)}$ acts without fixed points on $\R$ (and thus it is either focal or horocyclic), and the other is totally bounded.
					\item \emph{(Balance between sides)}. \label{i-balance} Both subgroups act without fixed points; in this case they are both horocyclic. 
				\end{enumerate}
			\end{prop}

			\begin{proof}
				Note that since all subgroups in the family $\{G_{(a, x)}: x\in X\}$ can be conjugated into each other, the property that one of them is totally bounded does not depend on $x$ (for a subgroup of a totally bounded subgroup is still totally bounded), and so is the property that one of them is focal (for an overgroup of a focal subgroup is focal). Because of this,  it is enough to show that either  case \ref{i-domination}  or \ref{i-balance} hold for any given $x\in X$, and the independence of the type on $x$ follows.  
				
				Fix $x \in X$. Since $\varphi$ is laminar, it is not semi-conjugate to the standard action (as the latter is locally moving). By assumption, it cannot be conjugate to an action of a proper quotient.  Hence, by Proposition \ref{p-lm-reconstruction}, at least one of the two subgroups $G_{(a, x)}, G_{(x, b)}$ must act without fixed points on $\R$.  
				
				Suppose first that one of them has fixed points and the other does not.  Let us assume that $\fixphi(G_{(a, x)})\neq\varnothing$. Then Proposition \ref{p-dyn-class-subgroups} implies that $G_{(a, x)}$ is either totally bounded, or pseudo-homothetic. In the second case, $\varphi(G_{(x, b)})$ would preserve the compact set $\fixphi(G_{(a, x)})$ and thus fix its largest point, contradicting {the assumption $\fixphi(G_{(x, b)})= \varnothing$}. Thus $G_{(a, x)}$ can only be totally bounded, and we are in case \ref{i-domination}.
				
				Suppose now that both $G_{(a, x)}$ and $G_{(x, b)}$ act without fixed points, and thus are either horocyclic or focal. Assume by contradiction that one of them, say $G_{(a, x)}$, is focal. Then it admits a non-discrete minimal invariant set $\Lambda$. We can therefore apply Lemma \ref{l-direct-product-explicit} to $G_{(a, x)}\times G_{(x, b)}$. Note that cases \ref{i-product-explicit-1} and \ref{i-product-explicit-2} in Lemma \ref{l-direct-product-explicit} imply the existence of elements without fixed points, which do not exist in a laminar action (Lemma \ref{rem.fixedepoints}). It follows that $G_{(x, b)}$ pointwise fixes $\Lambda$, contradicting that it has no fixed points. Thus both subgroups are necessarily horocyclic. 
			\end{proof}
			
			\begin{rem}\label{rem.more_precise_trichotomy}
				Recall  from Proposition \ref{p-dyn-class-subgroups} that all finitely generated subgroups of a horocyclic group are totally bounded. Thus Proposition \ref{p-lm-domination-balance} gives in either case that one of the two subgroups $G_+=\bigcup_{x} G_{(a, x)}$ and $G_-=\bigcup_{x} G_{(x, b)}$ must have the property that all of its finitely generated subgroups are totally bounded. The reader might compare this with Proposition \ref{p-lm-property-exotic}, which provides the same conclusion for the smallest subgroup $[G_c, G_c]$, and played an important role in Part \ref{partI}. Proposition \ref{p-lm-domination-balance} provides a strengthening and a more precise explanation of that statement.
			\end{rem}
			
			In view of Proposition \ref{p-lm-domination-balance}, for the rest of this section we shall  say that a laminar action $\varphi\colon G\to \homeo(X)$ has one-sided domination if case \ref{i-domination} holds. When needed, we will specify the dominating side by saying, for instance, that $\varphi$ has \emph{right-hand side domination} if $G_{(x, b)}$ acts without fixed points.
						
			The next lemma is where finite generation of $\Gfrag$ is crucially used (in both conclusions). 
			
			\begin{lem}[Key lemma] \label{l-lm-Gstar-fix} 
				Let $G\subset \homeo_0(X)$ be a subgroup acting minimally on an open interval $X=(a,b)$, with $\Gfrag$ non-trivial and finitely generated. Then the following hold.
				\begin{enumerate}[label=(\roman*)]
					
					\item \label{i-Gstar-locally-moving} The standard action of $G$ on $X$ is locally moving.
					\item \label{i-Gstar-domination} Every faithful minimal laminar action $\varphi\colon G\to \homeo_0(\R)$ has one-sided domination  (in the sense of Proposition \ref{p-lm-domination-balance}). 
				\end{enumerate}
			\end{lem}
			\begin{proof}
				Recalling that  $\Gfrag=G_-G_+$, let us fix a finite symmetric generating set $S$ of $\Gfrag$ of the form $S=S_-\cup S_+$, with $S_\pm \subset G_\pm$, and then choose $x_+, x_-\in X$ such that $S_+\subset G_{(a, x_+)}$ and $S_-\subset G_{(x_-, b)}$.
				
				To prove \ref{i-Gstar-locally-moving}, note first that as $\Gfrag$ is non-trivial, at least one of the two subgroups $G_-$ and $G_+$ is non-trivial, say $G_-$. Irreducibility of the action of $G$ easily implies that $G_-$ is not finitely generated, so finite generation of $\Gfrag$ gives that $\Gfrag\neq G_-$. So also $G_+$ is non-trivial, and not finitely generated. This easily implies that we can choose $g_1\in G_+$ and $g_2\in G_-$ that do not commute, so that $[g_1, g_2]$ is a non-trivial element of $G_-\cap G_+=G_c$. By Proposition \ref{p-micro-compact}, $G$ is micro-supported. By Theorem \ref{t-laminations-microsupported}, if $G$ is not locally moving, its standard action must be laminar. Now, as the standard action of $G$ is minimal and faithful, the normal subgroup $\Gfrag$ acts without fixed points. As $\Gfrag=\langle S_-, S_+\rangle$, we deduce that $\langle S_+\rangle$ must act without fixed points on a neighborhood of $a$, and trivially on a neighborhood of $b$ (and \textit{viceversa} for $\langle S_-\rangle$). After Proposition \ref{p-dyn-class-subgroups}, this behavior is not possible in a laminar action. Hence $G$ is locally moving.
				
				We next show \ref{i-Gstar-domination}. Fix a $\varphi$-invariant covering lamination $\mathcal{L}$.  By Proposition \ref{p-lm-domination-balance}, we can suppose, without loss of generality, that the subgroups of the form $G_{(x, b)}$ act without fixed points.  If so, Proposition \ref{p-lm-domination-balance} already tells us that the subgroups $G_{(x, b)}$ are either focal or horocyclic. Our goal is to show that we are in the first case. 
								
				\setcounter{claimnum}{0}
				
				\begin{claimnum} \label{claim-gast-1}
					The $\varphi$-action  of ${\Gfrag}$ is focal.
				\end{claimnum}
				
				\begin{proof}[Proof of claim]
					Since  $\varphi$ is faithful and minimal,  and $\Gfrag$ is normal, the restriction $\varphi|_{\Gfrag}$ must be irreducible. Hence, finite generation of $\Gfrag$ implies that  it must be focal (see Proposition \ref{p-dyn-class-subgroups}). \qedhere
				\end{proof} 
				\begin{claimnum} \label{claim-gast-2}
					The $\varphi$-action of ${G_-}$ is focal. 
				\end{claimnum}
				\begin{proof}[Proof of claim]
					Suppose by contradiction that $G_-$ is horocyclic. By Proposition \ref{p-dyn-class-subgroups}, the subset $\mathcal{L}_0\subset \mathcal{L}$ of wandering intervals for $G_-$ is cofinal and closed in $\mathcal{L}$; moreover, $\mathcal{L}_0$ is $\varphi$-invariant, by normality of $G_-$.  By Claim \ref{claim-gast-1} and Proposition \ref{prop.minimalimpliesfocal}, the $\varphi(\Gfrag)$-orbit of every $I\in \mathcal{L}_0$ is cofinal. Choose $I$ large enough so that  the image of $I$ under every element of $S$ intersects $I$ (and hence is related to $I$ by inclusion). Then we can find a sequence $(s_i)\subset S=S_+\cup S_-$ such that $s_n\cdots s_1.I\subsetneq s_{n+1}\cdots s_1.I$ and $\bigcup_n s_n\cdots s_1.I=\R$. But since all intervals ${s_n\cdots s_1. I}$ belong to $\mathcal{L}_0$,  no element of  $S_-$ can send any of them to a strictly larger interval, so the sequence $(s_i)$ must actually be contained in $S_+$. It follows that the $\varphi(G_{(a, x_+)})$-orbit of $I$ is cofinal in $\mathcal{L}$. This is absurd, since $G_{(a, x_+)}$ is either totally bounded or horocyclic (by Proposition \ref{p-lm-domination-balance}).
				\end{proof}
				\begin{claimnum} \label{claim-gast-3}
					The $\varphi$-action of  ${G_{(x, b)}}$ is focal. 
				\end{claimnum}
				
				\begin{proof}[Proof of claim]
					For $I\in \mathcal L$ and $x\in X$, let us set
					\begin{equation} \label{e-alpha-I-H}\alpha(I, x)=\bigcup_{h\in G_{(x, b)},\, I\subset h.I} h.I.\end{equation}
					Observe that $\alpha( I, x)\subset \alpha(I, y)$ for $y<x$, and that the function $\alpha(\cdot, \cdot)$ is $G$-equivariant:
					\[g.\alpha(I, x)=\alpha(g.I, g(x)).\]
					Note also that the definition of $\alpha(I, x)$ and the cross-free property of $\mathcal L$ easily give that if $I,J\in \mathcal L$ and $x\in X$ are such that $J\subset \alpha(I,x)$, then $\alpha(J,x)\subset \alpha(I,x)$; in particular we have that
					\begin{equation}\label{e-stability-alpha}
						\text{if $\alpha(I, x)\supset J$ and $\alpha(J, x)\supset I$, then $\alpha(I, x)=\alpha(J, x)$.}
					\end{equation}
										
					Suppose by contradiction that $G_{(x, b)}$ is not focal, and so horocyclic. Then $\alpha(I, x)\neq \R$ for every $I\in \mathcal{L}$ and $x\in X$. In contrast, since $G_-$ is focal by Claim \ref{claim-gast-2}, we can choose  $I\in \mathcal{L}$ whose $\varphi(G_-)$-orbit is cofinal in $\mathcal{L}$.
					Since $G_-=\bigcup_{x} G_{(x, b)}$, this means exactly that $\bigcup_x \alpha(I, x)=\R$. It follows that we can choose $x_0<x_-$ such that for every $s\in S$, we have $\alpha(I, x_0)\supset s.I$. Since the standard action of $\Gfrag$ on $X$ has no global fixed point, we can choose a sequence $(s_n)_{n\ge 1}\subset S$ such that the sequence $x_n=s_n\cdots s_1(x_0)$ is strictly decreasing and converges to the endpoint $a$. Note that such a sequence $(s_n)$ will be automatically contained in $S_+$, since we chose $x_0<x_-$. For every $n\ge 1$, we have 
					\[\alpha(I, x_{n})\supset \alpha(I, x_0)\supset s_{n}.I,\]
					and (recall that we have chosen $S$ symmetric)
					\[\alpha(s_{n}.I, x_{n})=s_{n}.\alpha(I, x_{n-1})\supset s_{n}.\alpha(I, x_0)\supset s_ns_n^{-1}.I=I.\]
					Hence, from \eqref{e-stability-alpha} and induction, we conclude  that for every $n\ge 1$ we have
					\[\alpha(I, x_{n})=\alpha(s_{n}.I, x_{n})=s_{n}.\alpha(I, x_{n-1})=s_n\cdots s_1.\alpha(I, x_0).\]
					Since $\alpha(I, x_n)$ is an increasing exhaustion of $\R$ and all elements $s_i$ are contained in $S_+$, this implies  that  the $\varphi(G_{(a, x_+)})$-orbit of $\alpha(I, x_0)$ is cofinal in $\mathcal{L}$; this contradicts the fact from Proposition \ref{p-lm-domination-balance} that $G_{(a, x_+)}$ is either totally bounded or horocyclic.   \qedhere
				\end{proof}
				\setcounter{claimnum}{0}
				
				By Proposition \ref{p-lm-domination-balance}, focality of $G_{(x, b)}$ implies that $\varphi$ has right-hand side domination.
			\end{proof}
			
			\subsection{End of the proof of Theorem \ref{t-lm-horograding}: construction of the horograding} \label{ssc.CF_family}
			
			Strictly speaking, the next lemma could be avoided for the proof of Theorem \ref{t-lm-horograding}, but  it makes the construction more natural, and it shows some good properties of the horograding constructed in the proof.
			
			\begin{lem}  \label{p-C-fix}
				For $X=(a,b)$, let $G\subset \homeo_0(X)$ be a {locally moving subgroup}.  Let  $\varphi\colon G\to \homeo_0(\R)$ be a  faithful minimal laminar action, and assume that $G_{(a, x)}$ is totally bounded. Then $\suppphi(G_{(a, x)})$  is dense for every $x\in X$.
			\end{lem}
			\begin{proof}
				For $x\in X$, set $C_x:=\fixphi(G_{(a, x)})$; note that $C_x\neq \R$ because we are assuming that the action $\varphi$ is faithful. Assume by contradiction that there exists $x\in X$ such that $C_x$ has non-empty interior. Fix also an arbitrary $y\in X$ and write $H=G_{(a, y)}$.  Let $I$ be a connected component of $\suppphi(H)=\R\setminus C_y\neq \emptyset$. Since $\varphi$ is proximal (Proposition \ref{prop.minimalimpliesfocal}),  there exists $g\in G$ such that $I\subset g.C_x=C_{g(x)}$. Thus $\varphi(G_{(a, g(x))})$ acts trivially on $I$, and since $G_{(a, g(x))} \cap H \neq \{\id\}$, we deduce that the action of $H$ on $I$ is not faithful. Since $H$ is a {locally moving} subgroup of $\homeo_{0}((a, y))$, from Proposition \ref{p-micro-normal} we have that the kernel of this action must contain $[H_c, H_c]$. Since $I$ was an arbitrary connected component of $\suppphi(H)$, this implies that $[H_c, H_c]$ acts trivially everywhere on  $\suppphi(H)$, and thus the $\varphi$-image of $[H_c, H_c]$ is  trivial, contradicting the assumption that $\varphi$ is faithful.
			\end{proof}
						
			\begin{proof}[Proof of Theorem \ref{t-lm-horograding}]
				By the first part of Lemma \ref{l-lm-Gstar-fix}, we have that the group $G$ is locally moving, so by Theorem \ref{t-laminations-microsupported}, every  faithful minimal action $\varphi\colon G\to \homeo_0(\R)$ which is not conjugate to the standard action must be laminar. 
				After the second part of Lemma \ref{l-lm-Gstar-fix}, every such an action has one-sided domination. We take such a $\varphi\colon G\to \homeo_0(\R)$ with right-hand side domination, and see how to construct a positive horograding by the standard action. The case of left-hand side domination is symmetric, and can be treated by conjugating the standard action by $x\mapsto -x$, thus obtaining a negative horograding. 
				
				For the construction, consider the intersection of supports
				\begin{equation*} \label{e-residual} \Xi:= \textstyle\bigcap_{x\in X} \suppphi(G_{(a, x)}),\end{equation*}
				which defines a $\varphi$-invariant subset. By Lemma \ref{p-C-fix} and Baire's theorem, we have that $\Xi$ is a $G_\delta$-dense subset of $\R$, and in particular it is non-empty.  
				For every $\xi\in \Xi$ and $x\in X$, we denote by $\Iphi(x, \xi)$ the connected component of $\suppphi(G_{(a, x)})$ which contains $\xi$. This is well defined for every $x\in X$, by definition of $\Xi$, and is a bounded interval because $G_{(a,x)}$ is totally bounded.
				
				\begin{claimnum}\label{claim.Iphi_pre}
					The collection $\mathcal L=\{\Iphi(x, \xi)  :  x\in X, \xi \in \Xi\}$ defines a $\varphi$-invariant prelamination, which is covering and thin.
				\end{claimnum}
				\begin{proof}[Proof of claim]
					By definition, we have $g. \Iphi(x, \xi)=\Iphi(g(x), g.\xi)$ for every $g\in G$, $x\in X$, and $\xi\in \Xi$. So $\mathcal L$ is $\varphi$-invariant. Next, if $x,y\in X$ are such that $x<y$, then for any $\xi \in \suppphi(G_{(a,x)})$ we have $\xi\in \suppphi(G_{(a,y)})$ and $\Iphi(x,\xi)\subset \Iphi(y,\xi)$. Moreover, for $x\in X$ and $\xi_1, \xi_2\in \Xi$, we have either $\Iphi(x, \xi_1)=\Iphi(x, \xi_2)$, or $\Iphi(x, \xi_1)\cap\Iphi(x, \xi_2)=\varnothing$. This gives that $\mathcal L$ is a prelamination. Finally, Proposition \ref{prop.minimalimpliesfocal} gives that $\mathcal L$ is covering and thin.
				\end{proof}
				
				\begin{claimnum}\label{claim-inj-x}
					The function
					\[\dfcn{\hor}{\mathcal{L}}{X}{\Iphi(x,\xi)}{x}\]
					is well defined, and $(\mathcal L,\hor)$ is a positive prehorograding of $\varphi$ by the standard action of $G$ on $X$.
				\end{claimnum}
				
				\begin{proof}[Proof of claim]
					In order to check that $\hor$ is well defined, we need to prove that if $\Iphi(x, \xi)=\Iphi(y,\eta)$, then $x=y$. Note first that by Claim \ref{claim.Iphi_pre}, the prelamination $\mathcal L$ is thin, so for any fixed $\xi\in \Xi$, the length of $\Iphi(x, \xi)$ tends to 0 as $x$ tends to $a$.
					Now, assume by contradiction that $\Iphi(x, \xi)=\Iphi(y,\eta)$ for some $x<y$. Since $G_{(a, y)}$ preserves the interval $\Iphi(y, \eta)$,  for every $g\in G_{(a, y)}$ we have 
					\[\Iphi(y,\eta)=g.\Iphi(y,\eta)=g.\Iphi(x, \xi)=\Iphi(g(x), g.\xi)=\Iphi(g(x), \eta).\]
					(The last equality follows from the fact that $\eta$ belongs to the first interval in the chain of equalities.)
					Since $G_{(a, y)}$ acts without fixed points on $(a, y)$, we can choose $g$ such that $g(x)$ is arbitrarily close to $a$, so the size of $\Iphi(g(x), \eta)$ can be made arbitrarily small. This contradicts the previous equality.
					
					The fact that $\hor$ is non-decreasing follows from the fact that $x\le y$ if and only if $\suppphi(G_{(a,x)})\subseteq \suppphi(G_{(a,y)})$. Finally, equivariance follows from equivariance of the intervals $\Iphi(x,\xi)$: for any $g\in G$, $x\in X$, and $\xi\in \Xi$, we have 
					\[
					\hor(g.\Iphi(x,\xi))=\hor(\Iphi(g(x),g.\xi))=g(x)=g(\hor(\Iphi(x,\xi))),
					\]
					as desired.
				\end{proof}
				\setcounter{claimnum}{0}
				
				Finally, the positive prehorograding $(\mathcal L,\hor)$ from Claim \ref{claim-inj-x} can be extended to a positive horograding $(\overline{\mathcal L},\hor)$ as explained in Remark \ref{r-extending-horograding}, since by Claim \ref{claim.Iphi_pre} we have that $\mathcal L$ is thin. This concludes the proof. \qedhere\end{proof}
			
			\begin{rem}\label{l-positive-horograding}
				For further reference, let us write explicitly what we have proved in this subsection.
				
				\emph{For $X=(a,b)$, consider a locally moving subgroup $G\subset \homeo_{0}(X)$ and a faithful minimal laminar action $\varphi\colon G\to \homeo_0(\R)$ for which the subgroups $G_{(a,x)}$ are totally bounded. Set 
					\[\Xi=\textstyle\bigcap_{x\in X}\suppphi(G_{(a,x)})\]
					which is a $G_\delta$-dense  subset. Consider the following.
					\begin{itemize}
						\item For $x\in X$ and $\xi\in \Xi$, let $\Iphi(x,\xi)$ be the connected component of $\suppphi(G_{(a,x)})$ containing $\xi$.
						\item Define $\mathcal L=\{\Iphi(x,\xi):x\in X,\xi\in \Xi\}$.
						\item Define
						\[\dfcn{\hor}{\mathcal{L}}{X}{\Iphi(x,\xi)}{x.}\]
					\end{itemize}
					Then, $\hor$ is well defined, and $(\mathcal L,\hor)$ is a positive prehorograding of $\varphi$ by the standard action of $G$ on $X$, which extends to a positive horograding $(\overline{\mathcal L},\hor)$.}
				
				Although unnecessary for the proof of Theorem \ref{t-lm-horograding}, let us mention that one can check that the set $\Xi$ coincides with the set of $\hor$-complete points (see Definition \ref{d-pi-complete-point-line}).
			\end{rem}

			\section{Examples of minimal laminar actions: orders of germ type revisited}\label{ssec.germtype}
			We have seen in Theorems \ref{t-laminations-microsupported} and \ref{t-lm-horograding} that laminations and horogradings can be used to understand exotic actions of many micro-supported groups. In this section, we show that many such groups (including most groups of piecewise linear or projective homeomorphisms) actually do admit minimal laminar actions.  More constructions of laminar actions will be found in the subsequent chapters.
			
			Set $X=(a, b)$ and let $G\subseteq \homeo_0(X)$ be a (countable) micro-supported group acting minimally on $X$. In order to run our construction, we require $G$ to satisfy the following condition.  \begin{enumerate}[label=($\mathcal{G}$\arabic*)]				
				\item \label{i-germ-section} \emph{The group $\Germ(G, b)$ admits a \emph{section} inside  $ \homeo_0(X)$, namely a subgroup $\Gamma \subseteq \homeo_0(X)$ such that $\Germ(\Gamma, b)=\Germ(G, b)$ and which projects bijectively to $\Germ(G, b)$.}\footnote{The problem of when a group of germs admits a section as a group of homeomorphisms is very interesting. We refer the reader to Mann \cite{Mann} for an example of a  finitely generated group of germs  which does not admit any such section.}	
			\end{enumerate} 
			
			Under assumption \ref{i-germ-section}, let $\Gamma\subseteq \homeo_0(X)$ be a section of $\Germ(G, b)$,  and consider the overgroup $\widehat{G}=\langle G , \Gamma\rangle$. We will proceed by describing an action of $\widehat{G}$ and then restricting it to $G$. Note that $\Germ(\widehat{G}, b)=\Germ(G, b)$, namely $\widehat{G}$ induces the same group of germs at $b$ as $G$. The advantage of passing to $\widehat{G}$ is that it splits as a semidirect product 
			\[\widehat{G}=\widehat{G}_+\rtimes \Gamma,\] 
			where as usual $\widehat{G}_+\subseteq \widehat G$ is the subgroup of elements whose germ at $b$ is trivial. 
			Using this splitting, we can let the group $\widehat{G}$ act ``affinely'' on $\widehat{G}_+$:  the subgroup $\widehat{G}_+$ acts on itself by left-multiplication, and $\Gamma$ acts on it by conjugation.  Explicitly, if $g\in \widehat{G}$ and $h\in \widehat{G}_+$, writing $g=g_+\gamma_g$ with $g_+\in \widehat{G}_+$ and $\gamma_g\in \Gamma$, we set
			\begin{equation} \label{e-affine-action} g\cdot h= g h\gamma_g^{-1}.\end{equation}
			Note that we actually have $g\cdot h=g_+(\gamma_g h\gamma_g^{-1})$, from which it is straightforward to check that this defines indeed an action on $\widehat G_+$.
			We want to find an order $\prec$ on $\widehat{G}_+$ which is invariant under the action of $\widehat{G}$, and then consider the dynamical realization of the action of $\widehat{G}$ on $(\widehat{G}_+, \prec)$. For this we look for a left-invariant order $\prec$ on $\widehat{G}_+$ which is also invariant under the conjugation action of  $\Gamma$. We will say for short that such an order is \emph{$\Gamma$-invariant}.
			
			Good candidates are the orders of germ type on $\widehat{G}_+$ described in \S \ref{s-germ-type}. Recall that an order of germ type is determined by a family  of left-invariant orders  $\{<^{(x)}: x\in X\}$, where for every $x\in X$, $<^{(x)}$ is a left-invariant order on the group of germs $\Germ\left (\widehat{G}_{(a, x)}, x\right )$: the associated order of germ type on $\widehat{G}_+$ is the order $\prec$ whose positive cone is the subset
			\[P=\left \{g\in \widehat{G}_+ : \Gcal_{p_g}(g)\succ^{(p_g)} \id\right \},\]
			where $p_g=\sup\{x\in X : g(x)\neq x\}$. However, not every order of germ type is $\Gamma$-invariant, and this is because for every $x\in X$ the stabilizer $\stab_{\widehat{G}}(x)$ of $x$ acts on $\widehat{G}_{(a, x)}$ by conjugation, and this action descends to {an action on $\Germ\left (\widehat{G}_{(a, x)}, x\right )$}.  In light of this, we are able to produce a $\Gamma$-invariant order of germ type on $\widehat G_+$  if and only if the following condition is satisfied. 
			\begin{enumerate}[label=($\mathcal{G}$\arabic*)]
				\setcounter{enumi}{1}
				\item \label{i-germ-invariant} \emph{For every $x\in X$, the group $\Germ\left (\widehat{G}_{(a, x)}, x\right )$ admits a left-invariant order $<^{(x)}$ which is invariant under {the induced conjugation action of $\stab_\Gamma(x)$}.}
			\end{enumerate}
			Indeed, suppose that $\{<^{(x)}: x\in X\}$ is a family such that the associated order of germ type is $\Gamma$-invariant, then each $\prec^{(x)}$ is as in \ref{i-germ-invariant}. Conversely  if  we choose such an order $<^{(x)}$ as in \ref{i-germ-invariant}, for $x$ in a system of representatives of the $\Gamma$-orbits in $X$, then we can extend it uniquely by $\Gamma$-equivariance to a family $\{<^{(x)}: x\in X\}$ which defines a $\Gamma$-invariant order of germ type.
			
			\begin{rem}
				Here are two simple sufficient conditions for \ref{i-germ-invariant}.
				\begin{enumerate}[label=($\mathcal{G}$2\alph*)]
					\item\label{i-germ-3'}\emph{The group $\Gamma$ acts freely on $X$.} 
					\item \label{i-germ-3''}\emph{For every $x\in X$, every non-trivial germ in $\Germ\left (\widehat{G}_{(a, x)}, x\right )$ has no sequence of fixed points accumulating on $x$ from the left (this does not depend on the choice of the element representing  the germ).}
				\end{enumerate}
				The fact that \ref{i-germ-3'} implies \ref{i-germ-invariant} is clear because in this case $\stab_\Gamma(x)$ is trivial. In contrast, when \ref{i-germ-3''} holds, we can define an order $<^{(x)}$ on $\Germ\left (\widehat{G}_{(a, x)}, x\right )$ by setting $\Gcal_x(g)>^{(x)} \id$  if and only if $g(y)>y$ for every $y\neq x$ in some left-neighborhood of $x$. Then this is a left-order on $\Germ\left (\widehat{G}_{(a, x)},x\right )$, invariant under conjugation by the whole stabilizer of $x$ in $\homeo_0(X)$.  \end{rem}
			
			Summing up, under conditions \ref{i-germ-section} and \ref{i-germ-invariant}, we can consider a $\Gamma$-invariant order of germ type $\prec$ on $\widehat{G}_+$ and let $\widehat{G}$ act on $(\widehat{G}_+, \prec)$ by \eqref{e-affine-action}. Passing to the dynamical realization, we obtain an action of $\widehat{G}$, and thus of $G$, on the real line. 
			This construction yields the following criterion for the existence of exotic actions. 
			
			\begin{prop}\label{p-exotic-germ}
				For $X=(a, b)$, let $G\subset \homeo_0(X)$ be a finitely generated, micro-supported group acting minimally on $X$. Assume that $\Germ(G,b)$  admits a section $\Gamma\subset \homeo_0(X)$ (that is, condition \ref{i-germ-section} holds) such that $\widehat{G}=\langle G, \Gamma\rangle$ satisfies \ref{i-germ-invariant}.

				Then there  exists a faithful minimal laminar action $\varphi\colon G\to \homeo_0(\R)$ which is  not conjugate to the standard action of $G$ on $X$. 
			\end{prop}
			\begin{proof}
				Choose a $\Gamma$-invariant order of germ type $\prec$ on $\widehat{G}_+$, defined from a family of orders $\{<^{(x)}: x\in X\}$. We let $\psi\colon G\to \homeo_0(\R)$ be the dynamical realization of the action of $G$ on $(\widehat{G}_+, \prec)$. Set $N=[G_c, G_c]$, which by Proposition \ref{p-micro-normal} is the smallest non-trivial normal subgroup of $G$.
				
				\begin{claim}
					For every $x\in X$ the group $\psi(N_{(x, b)})$ acts on $\R$ without  fixed points. 
				\end{claim}
				\begin{proof}[Proof of claim]
					Let $\iota\colon (\widehat{G}_+, \prec)\to (\R, <)$ be an equivariant good embedding associated with $\psi$ (Definition \ref{dfn.goodbehaved}). Observe that the subgroups $\widehat{G}_{(a, y)}$, for $y\in X$, are bounded convex subgroups of $(\widehat{G}_+, \prec)$ which form an increasing exhaustion of $\widehat{G}_+$, thus  the interior of the convex hull of every $\iota(\widehat{G}_{(a, y)})$ is a bounded open interval $I_y\subset \R$, giving an increasing exhaustion of $\R$. Now given any $x,y\in X$, every $g\in N_{(x, b)}$ with $p_g>y$   satisfies $g\widehat{G}_{(a, y)}\neq \widehat{G}_{(a, y)}$, which in turn implies that $\psi(g)(I_y)\cap I_y=\varnothing$. Since $y$ is arbitrary, and the intervals $I_y$ exhaust $\R$, this implies that $\fix^{\psi}(N_{(x, b)})=\varnothing$.
				\end{proof}
								
				As $G$ is finitely generated, we can consider a canonical model $\varphi\colon G\to \homeo_0(\R)$  of $\psi$, which is thus either minimal or cyclic.
				Since $\varphi$ is semi-conjugate to $\psi$, the claim gives $\fixphi(N_{(x, b)})=\varnothing$.  Using Proposition \ref{p-micro-normal}, we deduce that $\varphi$ is faithful, and thus minimal. Moreover, we see that it cannot be conjugate to the standard action of $X$, since $N_{(x, b)}$ does have fixed points on $X$.  
								
				Finally,  arguing as in the proof of the claim, it is not difficult to see that the collection
				\[\mathcal L:=\{\psi(g)(I_y): g\in G,y\in X\}\]
				is an invariant covering prelamination for $\psi$ (thus $\psi$ is laminar), see also Remark \ref{sc:inv_prelam_ordered_sets}. Since $\varphi$ is semi-conjugate to $\psi$, Proposition \ref{p-focal-semiconj} implies that $\varphi$ is also laminar.
			\end{proof}
			\begin{rem}
				Working as in \S \ref{ssc.CF_family}, one can show that $\varphi$ can be positively horograded by the natural action of $G$ on $X$. In particular, the dynamics of elements in the image of $\varphi$ is determined by their dynamics on $X$ (see Remark \ref{r.type_elt_horograding}).
			\end{rem}
			
			The criterion given by Proposition \ref{p-exotic-germ} applies to several classes of examples of micro-supported groups. Let us give a first illustration.
			We say that the group of germs $\Germ(G, b)$ acts \emph{freely near} $b$ if every non-trivial germ has no sequence of fixed points accumulating on $b$ (this condition does not depend on the choice of the representative; cf.\ \ref{i-germ-3''}).
			
			\begin{cor}\label{cor.germabelian}
				For $X=(a,  b)$, let $G\subset \homeo_0(X)$ be a finitely generated, micro-supported group acting minimally on $X$. Assume that $\Germ(G, b)$ is abelian and acts freely near $b$. Then there exists a faithful minimal laminar action $\varphi\colon G\to \homeo_0(X)$, which is not conjugate to the standard action of $G$ on $X$.
			\end{cor}
			
			\begin{proof}
				We first check that if $\Germ(G, b)$ is abelian and acts freely near $b$, then it admits a section $\Gamma\subset \homeo_0(X)$ which acts freely on $X$. To see this, using that $G$ is finitely generated, we can take finitely many elements $g_1, \ldots, g_r$ in $G$ whose germs at $b$ are non-trivial and generate $\Germ(G, b)$. Up to taking inverses, the assumption that $\Germ(G, b)$ is abelian and acts freely near $b$ allows to find $z\in X$ such that $g_i(x)>x$ and $g_ig_j(x)=g_jg_i(x)$ for every $x\in (z, b)$ and $i,j\in \{1,\ldots,r\}$. Choose an element $\gamma_1\in \homeo_0(X)$ which coincides with $g_1$ on $(z, b)$ and has no fixed points in $X$. Take $x_0\in (z, b)$ and consider the fundamental domain $I=[x_0,\gamma_1(x_0))$. As the elements $g_2,\ldots, g_r$  commute with $g_1=\gamma_1$ on $I$, they induce an action of $\Z^{r-1}$ on the circle $X/\langle \gamma_1\rangle= [x_0,\gamma_1(x_0)]/_{x_0\sim\gamma_1(x_0)}$, which can be lifted to an action of $\Z^{r-1}$ on $X$ commuting with $\gamma_1$. In simpler terms, writing $I_n=\gamma_1^n(I)$, so that $X=\bigsqcup_{n\in \Z} I_n$, for every $i\in\{2,\ldots,r\}$ we can consider the homemorphism $\gamma_i\in \homeo_{0}(X)$ defined by
				\begin{equation}\label{eq:section_germ}
					\gamma_i(x)= \gamma_1^n g_i \gamma_1^{-n}(x)\quad\text{for }x\in I_n\text{ and }n\in \Z.
				\end{equation}
				The elements $\gamma_2,\ldots,\gamma_r$ define exactly the action of $\Z^{r-1}$ on $X$ which commutes with $\gamma_1$, as discussed above.
				From the definition \eqref{eq:section_germ}, we see that every $\gamma_i$ coincides with $g_i$ on $[x_0,b)$,  and in particular we have $\Gcal_b(\gamma_i)=\Gcal_b(g_i)$. This gives that $\Gamma=\langle \gamma_1,\ldots, \gamma_r\rangle$ is a section of $\Germ(G, b)$ acting freely on $X$. Thus conditions \ref{i-germ-section} and \ref{i-germ-3'} are satisfied, so that the conclusion follows from Proposition \ref{p-exotic-germ}.
			\end{proof}
			
			A case in which the previous criterion applies is when $b<\infty$ and the group of germs $\Germ(G, b)$ coincides with a group of germs of linear homeomorphisms $x\mapsto \lambda(x-b)+b$. This is for instance the case whenever $G$ is a subgroup of the group $\PL(X)$ of finitary piecewise linear homeomorphisms of $X$. This case can be generalized as follows. 
			
			Given  an interval $X=(a, b)\subset \R$, we denote by $\PA_{0}^\omega(X)$ the group of all locally piecewise analytic, orientation-preserving homeomorphisms of $X$, with a finite set of breakpoints in $X$. We also let  $\PP(X)$ be the subgroup of $\PA_{0}^\omega(X)$ of piecewise projective homeomorphisms of $X$ with finitely many breakpoints, namely those that are locally of the form $x\mapsto \frac{px+q}{rx+s}$, with $ps-qr=1$.
			
			\begin{cor}\label{c-exotic-pp}
				For $X=(a,b)$, let $G\subset \PA_{0}^\omega(X)$ be a finitely generated micro-supported group acting minimally on $X$. Assume that one of the following conditions is satisfied:
				\begin{enumerate}[label=(\roman*)]
					\item  \label{i-exotic-pp} $G$ is contained in the group $\PP(X)$ of piecewise projective homeomorphisms;
					\item  \label{i-exotic-pa} the group of germs $\Germ(G, b)$ admits a  section $\Gamma$ contained in $\PA_{0}^\omega(X)$.
				\end{enumerate}
				Then there exists a faithful minimal laminar action $\varphi\colon G\to \homeo_0(\R)$, which is not conjugate to the action of $G$ on $X$.
			\end{cor}
			
			\begin{proof}
				First of all observe that \ref{i-exotic-pp} implies \ref{i-exotic-pa}. To see this, assume first that $X=\R$; then $\Germ(G, +\infty)$ is a subgroup of the group of germs of the affine group $\Aff(\R)=\{x\mapsto ax+b\}$, and thus admits a section inside $\Aff(\R)\subseteq \PP(\R)$. For general $X$, observe that if we fix $x_0\in X$, we can find $A, B\in \PSL(2, \R)$ which fix $x_0$ and such that $A$ maps the interval $(a, x_0)$ to $(-\infty, x_0)$ and $B$ maps $(x_0, b)$ to $(x_0, +\infty)$. Then the map $H\colon X\to \R$, given by 
				\[
				H(x)=\left\{\begin{array}{lr}
					A(x)&\text{if }x\le x_0,\\
					B(x)&\text{if }x>x_0,
				\end{array}\right.
				\]
				conjugates $\PP(X)$ to $\PP(\R)$, so that the we can conclude from the previous case. 
				
				Now assume that \ref{i-exotic-pa} holds, and choose a section $\Gamma\subset \PA_0^\omega(X)$  of $\Germ(G, b)$. Then since non-trivial analytic maps have isolated fixed points, we see that $\widehat{G}$ satisfies \ref{i-germ-3''}, thus \ref{i-germ-invariant}, and we can apply Proposition \ref{p-exotic-germ}. 
			\end{proof}
			
			\begin{rem}
				The conditions in Proposition \ref{p-exotic-germ} cannot be dropped: in \S\ref{s-no-actions} we will construct an example of a finitely generated locally moving group $G\subset \homeo_0(\R)$ which admits no exotic action. Moreover, this example satisfies \ref{i-germ-section} (but not \ref{i-germ-invariant}) and its standard action is by piecewise linear homeomorphisms with a countable set of singularities with finitely many accumulation points inside $X$. This shows the sensitivity of Corollary \ref{c-exotic-pp} to its assumptions.	
			\end{rem}
			
\chapter{A test family of examples: Bieri--Strebel groups on the line}\label{s-few-actions}

In this chapter we discuss how the previous results apply to a concrete family of examples:  Bieri--Strebel groups of the form $G(\R, A, \Lambda)$.  
Recall from \S \ref{sc.BieriStrebel} that we denote by $G(X;A,\Lambda)$ the Bieri--Strebel group, consisting of all finitary PL homeomorphisms of an interval $X$, with slopes in the multiplicative subgroup  $\Lambda\subseteq\R_{>0}$, and breakpoints and constant terms in a non-trivial $\Z[\Lambda]$-submodule $A\subset\R$ (see Definition \ref{d.BieriStrebel}). Our focus here is on the case $X=\R$.  The groups $G(\R; A, \Lambda)$ turn out to be an illuminating test case for both the assumptions and the conclusion of our results, in particular of Theorem \ref{t-lm-horograding}.

Let us summarize the results of this chapter in a special case. For $\lambda>1$, consider the group
\[G(\lambda):=G\left(\R; \Z[\lambda, \lambda^{-1}], \langle \lambda\rangle_\ast\right),\]
where $\langle \lambda \rangle_\ast$ denotes the cyclic multiplicative subgroup of $\R_{>0}$  generated by $\lambda$. 
We will begin by revisiting, in \S \ref{ss.Bieri-Strebelfocal}, the concrete construction of faithful minimal exotic actions of Bieri--Strebel groups from \emph{jump preorders} (introduced in \S \ref{s.BSjump}). We shall show that this construction  yields actions that are laminar and horograded by the standard action. 

The group $G(\lambda)$ is finitely generated for every $\lambda$ (this follows from the result of Bieri and Strebel \cite[Theorem B.7.1]{BieriStrebel}), but it is not always fragmentable.  Its fragmentable subgroup $G(\lambda)_{\mathsf{frag}}$ turns out to be finitely generated if and only if $\lambda$ is an algebraic number (see Lemma \ref{p.BieriStrebel_fg_germs}). Thus, the assumption of Theorem \ref{t-lm-horograding} is satisfied by $G(\lambda)$ only in this case. We will show the following.

\begin{itemize}
	\item  When $\lambda$ is transcendental, there exists a faithful minimal laminar action on $\R$ that cannot be horograded by the standard action (see Proposition \ref{p-Glambda-transcendental}).  This shows that the hypothesis of finite generation of the fragmentable subgroup $\Gfrag$ in Theorem \ref{t-lm-horograding} is not redundant, and cannot be replaced by finite generation of the group $G$ itself.
	
	\item If $\lambda$ is algebraic, we will use Theorem \ref{t-lm-horograding}  to show that  $G(\lambda)$ has exactly three faithful minimal actions on $\R$ up to conjugacy: the standard action, and two  laminar actions arising from jump cocycles preorders, which are horograded by the standard action (this is a special case of Theorem \ref{t-BBS}). This provides an application of Theorem \ref{t-lm-horograding}, for which the conclusion is particularly restrictive.
	
	\item Building on the previous result, we will construct an example of a finitely generated, fragmentable locally moving subgroup $G \subset \homeo_0(\R)$, with the property that all its faithful minimal actions on the line are conjugate to its standard action (Theorem \ref{t-doubling-actions}). This example is a suitable modification of the group $G(\lambda)$ for $\lambda=2$, obtained by allowing PL maps with a countable set of breakpoints (with some control on them). This shows that there exist groups which fall in the scope of Theorem \ref{t-lm-horograding}, but do not admit exotic actions at all. This should be compared with the results in Chapter \ref{ss.exoticactions}  (and other constructions in this paper), which imply that exotic actions always exist for many natural classes of micro-supported groups.
\end{itemize}

\section{Jump preorders revisited}\label{ss.Bieri-Strebelfocal}  
Let us briefly recall the definition of jump preorders on Bieri--Strebel groups, discussed in \S \ref{s.BSjump}. Let $G=G(X;A,\Lambda)$ be a Bieri--Strebel group, that we always assume to be countable. The jump cocycle associated with an element $g\in G$ is  the map $j_g$ defined in \eqref{eq.jump-cocycle}, namely
\[\dfcn{j_g}{X}{\Lambda}{x}{\dfrac{D^+g^{-1}(x)}{D^-g^{-1}(x)}.}\]
Recall that each $j_g$ is a finitely supported function, whose support coincides with the set $\BP(g^{-1})$ of breakpoints of $g^{-1}$. Recall also the cocycle relation \eqref{eq.cocycle}:
\[j_{gh}(x)=j_g(x)j_h(g^{-1}(x)).\]
We denote  by $\mathsf S:=\{j_g:g\in G\}$ the set of jump cocycles. The group $G$ acts on $\mathsf S$ by $g\cdot j_h:=j_{gh}$. 
The following is a slight generalization of the definition of jump preorders given in \S \ref{s.BSjump}. 
\begin{dfn}\label{dfn-jumppreorder}Consider a non-trivial $\Lambda$-invariant preorder $\leq_\Lambda\in\LPO(\Lambda)$, with residue \[\Lambda_0=[1]_{\le_\Lambda}=\{\lambda\in \Lambda:1\le_{\Lambda}\lambda\le_{\Lambda} 1\},\]
	write $\Lambda^*=\Lambda/\Lambda_0$, and $\lambda\equiv_\Lambda\mu$ when $\lambda/\mu\in\Lambda_0$.
	The  \emph{positive jump preorder} on $G$ associated with $\leq_\Lambda$ is defined by setting $g\preceq  h$ if and only if either
	\begin{itemize}
		\item $j_g(x)\equiv_{\Lambda} j_h(x)$ for every $x\in X$, or
		\item $j_g(\overline{x}_{g,h})\lneq_\Lambda j_h(\overline{x}_{g,h})$, where $\overline{x}_{g,h}:=\max\{x\in X:j_g(x)\not\equiv_\Lambda j_h(x)\}$.
	\end{itemize} 
	
	The \emph{negative jump preorder} associated with $\leq_\Lambda$ is defined analogously, by replacing $\overline{x}_{g, h}$ with  $\underline{x}_{g,h}:=\min\{x\in X:j_g(x)\not\equiv_\Lambda j_h(x)\}$. 
	
\end{dfn}
Recall also from \S \ref{sec.dynreal} {(more specifically, from what mentioned in Remark \ref{rem.dynrealpreo})} that given a left-invariant preorder $\preceq$ on $G$, its dynamical realization $\varphi\colon G\to \homeo_0(\R)$ is defined as the dynamical realization of the action of $G$ on the ordered space $(G/[\id]_{\preceq}, \prec)$, where $[\id]_{\preceq}=\{g\in G:\id\leq g\leq \id\}$ is the residue of $\preceq$.

\begin{dfn} \label{d-dyn-rel-jump}
	We denote by $\varphi_{+, \leq_\Lambda} \colon G\to \homeo_0(\R)$ (respectively, $\varphi_{-, \leq_\Lambda}$) the dynamical realization of the positive (respectively, negative) jump preorder associated with $\leq_\Lambda$.
\end{dfn}

\begin{rem}
	In \S  \ref{s.BSjump} we only defined and studied positive jump preorders (simply called jump preorders there), but all proofs given there extend to the negative case by symmetry. In particular,  any positive or negative jump preorder is left invariant (Lemma \ref{lem.invaPhi}), and its dynamical realization is a minimal faithful action, not semi-conjugate to the standard action, whose positive conjugacy class is uniquely determined by the preorder $\leq_\Lambda$ (Proposition \ref{prop.ejemplojump}).
\end{rem}

\begin{rem}\label{r.opposite}
	Given a preorder $\leq\in\LPO(G)$ on a group $G$, we can consider its opposite preorder, namely the preorder $\leq^{op}\in\LPO(G)$ with positive cone $P_{\leq^{op}}=P_{\leq}^{-1}$.  Note that when $\leq_1,\leq_2\in\LPO(\Lambda)$ are one opposite of the other, their associated positive (respectively, negative) jump preorders are also opposite to each other. Therefore, their dynamical realizations are conjugate (although not positively conjugate). 
\end{rem}

With the language of horogradings at hand, we can complement the results in \S \ref{s.BSjump} by the following. 

\begin{prop}\label{prop-jumphoro}
	Dynamical realizations of positive (respectively, negative) jump preorders are laminar and positively (respectively, negatively) horograded by the standard action of the Bieri--Strebel group $G$.
\end{prop} 

\begin{proof}
	We keep the same notation as before.
	We will prove the proposition in the case of positive jump preorders, the negative case is analogous. Let us denote by $H:=[\id]_\preceq$ the residue of the jump preorder, which consists of all $g\in G$ such that $j_g(x)\in \Lambda_0$ for every $x$.  Recall from Remark \ref{r-description-jump-cosets} that the coset space $G/H$ can be equivariantly identified with the set $\mathsf{S}^*$ of jump cocycles considered modulo $\Lambda_0$ (and seen as finitely supported functions from $X$ to $\Lambda^*$). With this identification, the group action is given by $g\cdot j_h^\ast=j_{gh}^\ast$, and the order on $\mathsf{S}^*$ defined by $j_g^\ast\prec j_h^\ast$ if and only if $j_g^\ast(\overline{x}_{g,h})<_\Lambda j_h^\ast(\overline{x}_{g,h})$, where we denote by $j^\ast_g\in \mathsf{S}^*$ the projection of $j_g\in \mathsf{S}$.
	
	After Remark \ref{rem.prehor_dyn_real}, it is then enough to show that the action on $(\mathsf{S}^*, \prec)$ preserves a prelamination,  positively horograded by the action on $X$. In fact, this can be deduced from the study of Plante-like actions of permutational wreath products, discussed in Examples \ref{subsec.Plantefocal} and \ref{e-Plante-horograding}. Namely, the jump cocycle relation \eqref{eq.cocycle} implies that the map
	\[\dfcn{\eta}{G}{\Lambda^* \wr_X G}{g}{(j_g^\ast, g),}\]
	is a group embedding.  Recall from Example \ref{subsec.Plantefocal} that we have an order-preserving action
	\[\Psi\colon \Lambda^*\wr_X G\to \Aut\left ( \textstyle \bigoplus_X \Lambda^*, \prec\right ),\]
	where $\prec$ is the order of lexicographic type determined by the usual order on $X$ and the order $<_\Lambda$ on $\Lambda^*$. Then $(\mathsf{S}^*, \prec)$ can be identified with the orbit of the trivial function for $\Psi\circ \eta$. We saw in Examples \ref{subsec.Plantefocal} and \ref{e-Plante-horograding} that the action $\Psi$ has a prehorograding $(\mathcal{L}_0, \hor_0)$ by the action on $X$, where
	\[\mathcal{L}_0=\left \{C_{\mathsf{s},y}:  \mathsf{s}\in \textstyle \bigoplus_X \Lambda^*, y\in X\right \}, \quad C_{\mathsf{s},y}:=\left \{ \mathsf{t}\in \textstyle \bigoplus_X \Lambda^* : \mathsf{t}(x)=\mathsf{s}(x)\text{ for any } x>y\right \},\]
	and $\hor_0(C_{\mathsf{s},y})=y$.
	With this in mind, given $g\in G$ and $x\in X$, define
	\[L_{g,y}:=C_{j_g^\ast, y}\cap \mathsf{S}^*=\{j_h^\ast\in \mathsf{S}^* : j_h^\ast(x)= j_g^\ast(x)\text{ for any } x>y\}.\] 
	It then follows that the set $\mathcal{L}:=\{L_{g,y}:g\in G,y\in X\}$ is a prelamination of $\mathsf{S}^*$ (as $\mathcal{L}_0$ is a prelamination), and the map
	\[
	\dfcn{\hor}{\mathcal L}{X}{I_{g,y}}{y,}
	\]
	well defined by the claim below, defines a positive prehorograding.
	\begin{claim} $L_{g,y}= L_{h,z}$ implies $y=z$ for every $g\in G$ and $y, z\in X$.\end{claim} 
	\begin{proof}[Proof of claim] Suppose by contradiction that $z<y$. Then, choose $k\in G$ so that $\supp(k)\subset (z,y)$. Since $j^\ast_{kg}(x)=j_g^\ast(k^{-1}(x))j_{k}^*(x)$, we get that $j_{kg}^\ast\in L_{g,y}$. On the other hand, since $k$ can be  arbitrarily chosen inside $G((z,y);A,\Lambda)$, we can choose $k$ so that $j^\ast_{kg}(x)\neq j_h^\ast(x)$ for some $x\in (z,y)$. Therefore we get $j^\ast_{kg}\notin L_{h,z}$.
	\end{proof} 
	This concludes the proof.
\end{proof}

\section{On the fragmentable subgroup of Bieri--Strebel groups}\label{ssc:BieriStrebel}

In this section and the next we restrict the attention to Bieri--Strebel groups of the form $G=G(\R; A,\Lambda)$. Here we clarify the applicability of Theorem \ref{t-lm-horograding} in this case.

\subsection{Characterising the fragmentable subgroup} The fragmentable subgroup of $G(\R; A, \Lambda)$ can be characterized from the results of Bieri and Strebel discussed in \cite[\S\S A--B]{BieriStrebel}. As for many properties of the groups $G(X; A, \Lambda)$, this involves  the $\Z[\Lambda]$-submodule 
\[I\Lambda\cdot A:=\langle (\lambda-1)a:\lambda\in \Lambda,a\in A\rangle.\]

We denote by $\Aff(A, \Lambda)\cong A\rtimes \Lambda$ the group of all affine maps $x\mapsto \lambda x+a$, with $\lambda \in \Lambda$, and $a\in A$, which is a subgroup of $G(\R; A, \Lambda)$. Note that the groups of germs of $G(\R; A, \Lambda)$ at both $\pm \infty$ are naturally isomorphic to $\Aff(A, \Lambda)$. We denote by $\Gcal_{\pm \infty}\colon G(\R; A, \Lambda)\to \Aff(A, \Lambda)$ the associated germ homomorophisms.

\begin{lem}\label{p.BieriStrebel_fg_germs}
	Write $G=G(\R; A, \Lambda)$. Then, the following hold.
	\begin{enumerate}[label=(\roman*)]
		\item \label{i-BS-Gfrag-car}The fragmentable subgroup $\Gfrag$ coincides with the set of elements $g\in G$  whose germ $\Gcal_{+ \infty}(g)$ (equivalently, $\Gcal_{-\infty}(g)$) belongs to $\Aff(I\Lambda\cdot A,\Lambda)$, and we have
		\begin{equation}\label{eq:G/Gfrag-BS}
			G/\Gfrag\cong \faktor{\Aff(A, \Lambda)}{\Aff(I\Lambda\cdot A, \Lambda)}\cong {A}/{I\Lambda\cdot A}.
		\end{equation}
		\item \label{i-BS-Gfrag-fg} $\Gfrag$ is finitely generated if and only if the following conditions are satisfied:
		\begin{enumerate}[label=(BS\arabic*)]
			\item\label{i:LambdaBS} $\Lambda$ is finitely generated as a group,
			\item\label{i:ABS} $A$ is finitely generated as $\Z[\Lambda]$-module,
			\item\label{i:quotientBS} the quotient $A/I\Lambda\cdot A$ is finite;
		\end{enumerate}
		this is also equivalent to the fact that  $G$ is finitely generated and $\Gfrag$ has finite index in $G$.
	\end{enumerate}
\end{lem}
\begin{proof}
	Recall that $\Gfrag=G_-G_+$, where $G_{\pm}=\ker \Gcal_{\pm \infty}$. It follows from \cite[Corollary A5.3]{BieriStrebel} that \[\Gcal_{+\infty}(G_-)=\Gcal_{-\infty}(G_+)=\Aff(I\Lambda\cdot A,\Lambda),\] and hence $\Gcal_{\pm \infty}(\Gfrag)= \Aff(I\Lambda\cdot A,\Lambda)$. Conversely if $g\in G$ is such that $\Gcal_{+\infty}(g)\in\Aff(I\Lambda\cdot A,\Lambda)$, we can find $h\in G_-$ with $\Gcal_{+\infty}(h)=\Gcal_{+\infty}(g)$, so that $gh^{-1}\in G_+$, showing that $g\in \Gfrag$. This shows \ref{i-BS-Gfrag-car}. To show \ref{i-BS-Gfrag-fg}, note that by \eqref{eq:G/Gfrag-BS}, condition \ref{i:quotientBS} is equivalent to the fact that $\Gfrag$ has finite index in $G$.
	On the other hand, \ref{i:LambdaBS} and \ref{i:ABS} are equivalent to finite generation of $G$, by \cite[Theorem B7.1]{BieriStrebel}. Hence \ref{i:LambdaBS}--\ref{i:quotientBS} imply finite generation of $\Gfrag$. Conversely if $\Gfrag$ is finitely generated, then so is its homomorphic image $\mathsf{Aff}(I\Lambda\cdot A,\Lambda)$. This implies that $\Lambda$ must be finitely generated as a group and $I\Lambda\cdot A$ must be finitely generated as a $\Z[\Lambda]$-module. Also, two points $p,q\in A$ are in the same orbit under $\Gfrag$ if and only if $p-q\in I\Lambda\cdot A$ \cite[Corollary A.5.1]{BieriStrebel}. Since all points of $A$ occur as breakpoints for elements in $\Gfrag$, if $\Gfrag$  is finitely generated,  then $A$ must be covered by the orbits of the breakpoints of elements in any finite generating set, which are finitely many (see \cite[\S B6]{BieriStrebel}). Thus condition \ref{i:quotientBS} is also necessary. Finally notice that condition \ref{i:quotientBS} together with the fact that $I\Lambda\cdot A$ is finitely generated as a $\Z[\Lambda]$-module, implies that $A$ is finitely generated as a $\Z[\Lambda]$-module.
\end{proof}

Recall our tacit choice to denote by $\langle S \rangle_*$ the multiplicative group generated by a subset $S\subset \R_{>0}$. As an important special case, the reader can have in mind the following class of Bieri--Strebel groups.

\begin{ex} \label{e-G-lambda}
	For $\lambda>1$, we denote by $G(\lambda)$ the Bieri--Strebel group $G(\lambda):=G(\R;A,\Lambda)$ corresponding to the cyclic group $\Lambda=\langle \lambda\rangle_*$ and to  $A:=\Z[\lambda,\lambda^{-1}]$. One has
	\[I\Lambda\cdot A=(\lambda-1)\, \Z[\lambda,\lambda^{-1}].\]
	It is not difficult to see that the quotient $A/I\Lambda\cdot A$ is finite if and only if $\lambda$ is algebraic (see \cite[Illustration A4.3]{BieriStrebel}). For instance, for a rational $\lambda=p/q$  (with $p$ and $q$ coprime), one has $|A/I\Lambda\cdot A|=p-q$. Therefore, by Lemma \ref{p.BieriStrebel_fg_germs}, the group $G(\lambda)_{\mathsf{frag}}$ is finitely generated exactly for algebraic $\lambda$.
\end{ex}

\subsection{An exotic action not horograded by the standard action}\label{Sec_BS_transcendal}
Consider the group $G=G(\lambda)=G(\R; \Z[\lambda, \lambda^{-1}], \langle \lambda\rangle_\ast)$, with $\lambda$ transcendental. In this case the group $G$ is finitely generated but, by Lemma \ref{p.BieriStrebel_fg_germs}  (and Example \ref{e-G-lambda}), its fragmentable subgroup $\Gfrag$ is not. The next result shows that the assumption of finite generation of $\Gfrag$ in Theorem \ref{t-lm-horograding} cannot be replaced by finite generation of the group $G$ (even assuming \emph{a priori} that $G$ is locally moving).
\begin{prop} \label{p-Glambda-transcendental}
	Let $\lambda>1$ be a transcendent real number, and write $G=G(\lambda)$. Then, there exists a faithful minimal laminar action $\varphi\colon G\to\homeo_0(\R)$ which does not admit any horograding by the standard action of $G$. 
\end{prop}

The proof is based on a modification of the jump preorder construction. 
Note that Definition \ref{dfn-jumppreorder} involves the usual order on $X$ (actually on the subset $A\cap X$), which is used to define the point
\[\overline{x}_{g, h}=\max\{x\in A\cap X:j_g(x)\not\equiv_\Lambda j_h(x)\}.\]
The key observation is that the same construction will provide a left preorder whenever the maximum $\overline{x}_{g,h}$ is considered with respect to a $G$-invariant total order on $A$.  It turns out that such orders exist for the group $G(\lambda)$, when $\lambda$ is transcendental.

In what follows, we set $A:=\Z[\lambda, \lambda^{-1}]$, $\Lambda=\langle \lambda\rangle_\ast$, $G:=G(\lambda)$, and assume that $\lambda$ is transcendental.  Note that we can see $A$ as a ring  of Laurent polynomials in the variable $\lambda$, and thus we have an epimorphism (usually called the \emph{augmentation map})
\begin{equation} \label{e-augmentation-map}
	\dfcn{\varepsilon}{A}{\Z}{\sum a_i\lambda^i}{\sum a_i.}
\end{equation}
Since $I\Lambda\cdot A=(\lambda-1)A=\ker \varepsilon$, we deduce that
$A/I\Lambda\cdot A$ is infinite cyclic, generated by the image of  $1\in A$. Denote by $t_a\in \Aff(A, \Lambda)$ the translation $x\mapsto x+a$. 
By \eqref{eq:G/Gfrag-BS}, we have 
\[G/\Gfrag\cong \faktor{\Aff(A, \Lambda)}{\Aff(I\Lambda\cdot A, \Lambda)}\cong {A}/{I\Lambda\cdot A},\]
which implies that $G/\Gfrag$ is infinite cyclic, generated by the image of $t_1$. It follows that the group $G$ splits as a semidirect product:
\[G=\Gfrag\rtimes \langle t_1\rangle.\]
After this preliminary discussion, we can prove Proposition \ref{p-Glambda-transcendental}.

\begin{proof}[Proof of Proposition \ref{p-Glambda-transcendental}]
	Define an order $<_\varepsilon$ on $A$, by setting $a<_\varepsilon b$ if
	\begin{itemize}
		\item $\varepsilon(a)<\varepsilon(b)$ (where $\varepsilon$ is the map in \eqref{e-augmentation-map}), or
		\item  $\varepsilon(a)=\varepsilon(b)$ and $a< b$ with respect to the standard order on $A$ induced from the inclusion $A\subset \R$.
	\end{itemize}
	We claim that the $G$-action on $A$ preserves $<_\varepsilon$. It is clear that $t_1$ preserves it. Next, we observe that by \cite[Theorem A4.1]{BieriStrebel}, all elements of $G$ having a fixed point in $A$ preserve the $I\Lambda\cdot A$-cosets. Since $\Gfrag$ is generated by such elements and preserves the standard order on $A\subset \R$, it follows that $\Gfrag$ preserves $<_\varepsilon$, and thus so does $G=\Gfrag\rtimes \langle t_1 \rangle$.
	
	Now, recall that we denote by $\mathsf{S}=\{j_g: g\in G\}$ the set of jump cocycles, on which $G$ acts by $g\cdot j_h=j_{gh}$. We see here  jump cocycles as finitely supported functions $j_g\colon A\to \Lambda$ (since $j_g$ is supported on the set of breakpoints $\BP(g^{-1})\subset A$).
	
	We choose as $\leq_\Lambda$ the standard order $<$ on $\Lambda=\{\lambda^n : n\in \Z\}$, and define an order $\prec_\varepsilon$ on $\mathsf{S}$ by setting $j_g\prec_\varepsilon j_h$ whenever $j_g\neq j_h$ (equivalently, $g\neq h$) and $j_g(\overline{x}_{g, h}^\ast)<j_h(\overline{x}_{g, h}^\ast)$, where 
	\[\overline{x}_{g,h}^\ast=\mathrm{max}_{<_\varepsilon}\{x\in A: j_g(x)\neq j_h(x)\}.\]
	The same proof as for Lemma \ref{lem.invaPhi} shows that $\prec_\varepsilon$ is $G$-invariant. Let $\varphi\colon G\to \homeo_0(\R)$ be the dynamical realization of the $G$-action on $(\mathsf{S}, \prec_\varepsilon)$. 
	
	Note that since for every $a\in A$ the translation $t_a$ has trivial jump cocycle, the cocycle rule \eqref{eq.cocycle}) gives:
	\begin{equation}\label{ecuecuecu}(t_a\cdot j_h)(x)=j_h(t_{a}^{-1}.x).\end{equation} 
	Note also that $t_1$-action on $(A, <_\varepsilon)$ is cofinal, namely for every $x\in A$, we have $t_1^n(x)\to \pm \infty$ in $(A, <_\varepsilon)$ as $n\to\pm \infty$. This, together with \eqref{ecuecuecu}, implies that $t_1$ acts on $(\mathsf{S}, \prec_\varepsilon)$ as an expanding homothety with fixed point $j_{\id}$, namely for every elements satisfying
	\[j_{h_1}\prec_\varepsilon j_{k_1}\prec_\varepsilon j_{\id} \prec_\varepsilon j_{k_2}\prec_\varepsilon j_{h_2},\]
	there exists $n$ such that
	\[t_{1}^n\cdot j_{k_1}\prec_\varepsilon j_{h_1}\prec_\varepsilon j_{\id} \prec_\varepsilon j_{h_2}\prec_\varepsilon t_1^n\cdot j_{k_2}.\]
	Therefore, by Lemma \ref{p.minimalitycriteria}, we have that the action $\varphi$ is minimal. Next, we have that $\varphi$ is faithful since the action of $G$ on $\mathsf{S}$ is so (see the proof of Proposition \ref{prop.ejemplojump}). Finally, $\varphi$ is not conjugate to the standard action, since $\varphi(t_1)$ is a homothety. Laminarity follows from Theorem \ref{t-laminations-microsupported}, or can be easily checked directly (as in Proposition \ref{prop-jumphoro}).
	
	It remains to show that $\varphi$ cannot be horograded by the standard action. For this, consider a translation $t_a$ with $a\in I\Lambda\cdot A\setminus \{0\}$. Note that the action of $t_a$ on $A$ preserves all $I\Lambda\cdot A$ cosets, which are bounded convex subsets of $(A,  <_\varepsilon)$. From this and from \eqref{ecuecuecu}, it easily follows that every $t_a$-orbit in $(\mathsf{S}, \prec_\varepsilon)$ is also bounded above and below. Hence $\varphi(t_a)$ is totally bounded.  Now, if $\varphi$ was horograded by the standard action, after Proposition \ref{p-dyn-class-elements-horograded} we would have that $\varphi(t_a)$ is a homothety, contradicting what we have just proved.
\end{proof}

\section{A classification result}
We now turn the attention to the case where $G(\R; A, \Lambda)$ satisfies conditions \ref{i:LambdaBS}--\ref{i:quotientBS}, so that its fragmentable subgroup is finitely generated (Lemma \ref{p.BieriStrebel_fg_germs}). In this case, we will use Theorem \ref{t-lm-horograding} to show the following. 

\begin{thm} \label{t-BBS} 
	Let $G=G(\R;A,\Lambda)$ be a Bieri--Strebel group satisfying conditions \ref{i:LambdaBS}--\ref{i:quotientBS}. Then every faithful minimal action $\varphi\colon G\to \homeo_0(\R)$ is topologically conjugate either to the standard action on $\R$, or to the dynamical realization $\varphi_{\pm, \leq_\Lambda}$ of a jump preorder (Definition \ref{d-dyn-rel-jump}).
\end{thm}

\begin{ex}
	\label{ex:focal_G_lambda}
	Let us consider the special case $G=G(\lambda)$ as in Example \ref{e-G-lambda}, with algebraic $\lambda>1$. Since in this case the group $\Lambda$ is infinite cyclic, it admits only two non-trivial preorders, namely the usual order $<_\Lambda$ and its opposite. Thus, the jump preorder construction gives exactly two actions $\varphi_+:=\varphi_{+, <_\Lambda}$ and $\varphi_-:=\varphi_{-, <_\Lambda}$, which are, respectively, the dynamical realizations of the positive and negative jump preorders associated with $<_\Lambda\in\LPO(\Lambda)$. Indeed, note that after Remark \ref{r.opposite}, the dynamical realization of the jump preorder corresponding to $<_\Lambda^{op}$ is conjugate to that for $<_\Lambda$.  
	Thus, in this case, the group $G$ admits finitely many (more precisely, three) faithful minimal actions, {up to conjugacy}.
\end{ex}

The proof of Theorem \ref{t-BBS} is given in \S \ref{s-proof-BBS} below.

\subsection{Classification up to semi-conjugacy} 

Before proving Theorem \ref{t-BBS}, let us explain how we can complement it with a good description of actions of the largest  quotient $G/[G_c, G_c]$  to obtain a full description of actions of $G$ up to semi-conjugacy (see Corollary \ref{c-lm-semiconj}). 
\begin{cor} \label{c-BBS-semiconjugacy}
	Under the assumptions of Theorem \ref{t-BBS}, every irreducible action $\varphi\colon G\to \homeo_0(\R)$ is semi-conjugate to an action in one of the following families. 
	\begin{itemize}
		\item \emph{(Non-faithful)} A non-faithful action; these are all semi-conjugate to either
		\begin{itemize}[label=$\bullet$]
			\item an action by translations, obtained by composing the homomorphism $G\to \Lambda \times \Lambda$ determined by the eventual slopes near $\pm \infty$ with a homomorphism $\rho\colon \Lambda\times \Lambda \to (\R, +)$, or 
			\item   an action obtained by composing one of the two germ homomorphisms \[\mathcal{G}_{\pm \infty}\colon G\to \Aff(A, \Lambda)\]
			with a minimal action $\rho\colon \Aff(A, \Lambda)\to \homeo_0(\R)$ with non-abelian image.
		\end{itemize}
		\item \emph{(Standard)} The standard piecewise linear action of $G$ on $\R$.
		\item \emph{(Jump preorder)} The dynamical realization  of a jump preorder. 	\end{itemize}
\end{cor}

In order to prove Corollary \ref{c-BBS-semiconjugacy}, {recall from \S\ref{sec.normsgp}} that the largest quotient $:=G/[G_c, G_c]$ can be written as an extension
\[1\to G_c^{ab}\to \overline{G}\to G/G_c\to 1,\]
where $G_c^{ab}=G_c/[G_c, G_c]$.
The group $G/G_c$ embeds in the product  $\Aff(A,\Lambda)\times \Aff(A,\Lambda)$ via the product {$(\mathcal G_{-\infty},\mathcal G_{+\infty})$} of the germ homomorphisms. In particular, $\overline{G}$ is solvable of derived length at most 3.  In contrast, a description of $G_c^{ab}$ is not known in general.  However, we will show that all minimal actions of $\overline{G}$ on the line actually factor through $G/G_c$. For this we need the following algebraic properties of Bieri--Strebel groups.

\begin{lem}\label{l-BBS-algebraic}
	Let $G=G(\R; A, \Lambda)$ be a Bieri--Strebel group. For $g\in G$, write $\alpha_{\pm}(g)\in \Lambda$ and $\beta{\pm}(g)\in A$ for the slope and translation parts, respectively, of the germ $\mathcal{G}_{\pm\infty}(g)\in \Aff(A, \Lambda)$. Then the following hold.
	\begin{enumerate}[label=(\roman*)]
		\item \label{i-BBS-GmodGc}
		The image ${(\mathcal G_{-\infty},\mathcal G_{+\infty})}(G)$ in $\Aff(A,\Lambda)\times \Aff(A,\Lambda)$ is the subgroup 
		\[
		P:=\left\{(f,g)\in \Aff(A,\Lambda)\times \Aff(A, \Lambda): \beta_+(g)-\beta_-(f)\in I\Lambda\cdot A \right\}.
		\]
		In particular, $G/G_c$ is isomorphic to $P$.
		
		\item\label{i-BBS-abelianisation} The map 
		\[
		\setlength{\arraycolsep}{1.5pt}
		\begin{array}{ccc}G&\to&\Lambda \times \Lambda \times \left(A/I\Lambda \cdot A\right)\\g&\mapsto&(\alpha_{-}(g), \alpha_+(g), \beta_{-}(g)+I\Lambda \cdot A)\end{array}\]
		descends to an isomorphism $G^{ab}\to \Lambda \times \Lambda \times \left(A/I\Lambda \cdot A\right) $.
		\item  \label{i-BBS-central} The quotient $G_c^{ab}:=G_c/[G_c, G_c]$ is central in $\Gfrag/[G_c, G_c]$.
	\end{enumerate}
\end{lem}
\begin{proof}
	The statement \ref{i-BBS-GmodGc} is \cite[Corollary A5.5]{BieriStrebel}.  
	To show \ref{i-BBS-abelianisation}, we first claim that $G_c$ is contained in $[G, G]$. For this, note first that since $G_c\subset G_+$, the projection of $G_c$ to $G^{ab}$ transits through the natural homomorphisms
	\[G_c\to G_c^{ab}\to G_+^{ab}\to G^{ab}.\]
	An explicit description of $G_+^{ab}$ is provided by \cite[Proposition C12.1]{BieriStrebel}, as we explain now.  Let $C\subset A$ be a coset of $I \Lambda\cdot A$, and recall that any such a coset is also an orbit for the action of $G_+$ on $A$ \cite[Corollary A5.1]{BieriStrebel}. For $g\in G_+$, set 
	\[\nu_C(g):=\prod_{a\in C} \frac{D^+g(a)}{D^-g(a)}\in \Lambda.\]
	The previous product is well defined, since all but finitely many terms are equal to $1$, and it is easily seen to be a homomorphism, using the chain rule for the derivative.
	It follows from \cite[Proposition C12.1]{BieriStrebel}  that the map
	\[\dfcn{\nu}{G_+}{\Lambda^{A/I \Lambda\cdot A}}{g}{(\nu_C(g))_{C\in A/I \Lambda\cdot A}}\]
	descends to an isomorphism $G_+^{ab}\cong \Lambda^{A/I \Lambda\cdot A}$. (More precisely, \cite[Proposition C12.1]{BieriStrebel} states that the restriction of $\nu$ to any subgroup of the form $G_{(-\infty, x)}$, with $x\in A$, induces an isomorphism $G_{(-\infty, x)}^{ab}\cong\Lambda^{A/I \Lambda\cdot A}$; but since $G_+=\bigcup_{x} G_{(-\infty, x)}$, the conclusion follows.) Note also that for every $g\in G_+$, we have $\prod_C \nu_C(g)=\alpha_-(g).$
	In particular, 
	\begin{equation} \label{e-Gc-nu}
		\prod_{C\in A/I\Lambda\cdot A} \nu_C(g)=1\quad \text{for every }g\in G_c.
	\end{equation}
	Next, observe that if $t_a\in G$ is the translation $x\mapsto x+ a$, then the chain rule gives the relation
	\begin{equation} \label{e-nu-translation}
		\nu_C(t_ag t_a^{-1})=\nu_{a+C}(g).
	\end{equation}
	From \eqref{e-Gc-nu} and \eqref{e-nu-translation} we deduce that $G_c$ is contained in $[G, G]$. To do so, fix $g\in G_c$, and let us show that it must have trivial projection to $G^{ab}$. Choose an enumeration $C_1,\ldots, C_n$ of $A/I\Lambda\cdot A$, and for every $i\in \{1,\ldots,n\}$, choose $h_i\in G_+$ such that $\nu_{C_i}(h_i)=\nu_{C_i}(g)$, and $\nu_{C_j}(h_i)=1$ for $j\neq i$. Then $\nu(g)=\nu\left (\prod_{i=1}^nh_i\right )$, and so $g$ and $\prod_{i=1}^nh_i$ project to the same element of $G^{ab}$. It follows that, if we choose representatives $a_i\in C_i$, then $g$ and $\prod_{i=1}^n t_{a_i} h_{i} t_{a_i}^{-1}$ also have the same projection  to $G^{ab}$ (since this is true for each $h_i$ and $t_{a_i} h_i t_{a_i}^{-1}$). But by \eqref{e-nu-translation}, we see that
	\[\nu_{I\Lambda\cdot A}(t_{a_i} h_{i} t_{a_i}^{-1})=\nu_{C_i}(h_{i})=\nu_{C_i}(g)\quad \text{for any }i\in \{1,\ldots,n\},\]
	and $\nu_C\left (t_{a_i} h_{i} t_{a_i}^{-1}\right )=1$ for $C\neq I\Lambda \cdot A$. 
	Hence, if  {$C\neq I\Lambda\cdot A$}, we have
	\[\nu_{C}\left (\textstyle\prod_{i=1}^n  t_{a_i} h_{i} t_{a_i}^{-1}\right )=\prod_{i=1}^n\nu_{C}\left (t_{a_i} h_{i} t_{a_i}^{-1}\right )=1,\]
	whereas \eqref{e-Gc-nu} gives that 
	\[\nu_{I\Lambda\cdot A}\left (\textstyle\prod_{i=1}^n  t_{a_i} h_{i} t_{a_i}^{-1}\right )=\prod_{i=1}^n\nu_{I\Lambda\cdot A}\left (t_{a_i} h_{i} t_{a_i}^{-1}\right )=\prod_{i=1}^n\nu_{C_i}(g)=1.\]
	Consequently, $\prod_{i=1}^n t_{a_i} h_{i} t_{a_i}^{-1}$ projects trivially already to $G_+^{ab}$. This implies that it projects trivially to $G^{ab}$, and hence so does $g$.
	
	Having shown that $G_c\subset [G, G]$, we have that  $(\mathcal{G}_{-\infty}, \mathcal{G}_{+\infty})$ induces an isomorphism $G^{ab} \to P^{ab}$. To compute $P^{ab}$, observe that $[P, P]=\{(t_{a}, t_{b}): a, b\in I\Lambda \cdot A\}$. Indeed, it is easy to check that every commutator of elements in $P$ is of this form, hence $[P, P]\subseteq \{(t_{a}, t_{b}): a, b\in I\Lambda \cdot A\}$. Conversely, for $\lambda\in \Lambda$ and $a\in A$, denoting by $g_\lambda$ the homothety $x\mapsto \lambda x$, the commutator of $(t_a, t_a)$ and $(g_\lambda, \id)$ is $(t_{(1-\lambda) a}, \id)$, and similarly one shows that $(\id, t_{(1-\lambda) a})\in [P, P]$, which shows  that $\{(t_{a}, t_{b}): a, b\in I\Lambda \cdot A\}\subseteq [P, P]$. This  implies \ref{i-BBS-abelianisation}.
	
	We now show \ref{i-BBS-central}.  Write $\pi\colon G\to \overline{G}=G/[G_c, G_c]$ for the quotient projection, so that $G_c^{ab}=\pi(G_c)$.  Note that $G_c^{ab}$ is a subgroup of  $\pi(\Gfrag)$. We claim that it is contained in its center. For this, recalling that $\Gfrag=G_-G_+$, fix $g\in G_-$ and $k\in G_c$, and let us show that $\pi(g)$ commutes with $\pi(k)$ (the case of $g\in G_+$ is similar). Choose $x\in \R$ such that $g\in G_{(x, +\infty)}$. Pick $h\in G_c$ such that $h(x)>\sup \supp (k)$. Then $g_1:=hgh^{-1}$ and $k$ commute. However, since $\pi(h)$ and $\pi(k)$ commute, we have
	\[
	\pi(h)[\pi(g),\pi(k)]\pi(h)^{-1}=[\pi(hgh^{-1}),\pi(k)]=[\pi(g_1),\pi(k)]=1
	\]
	showing that $\pi(g)$ and $\pi(k)$ commute.
\end{proof}

\begin{proof}[Proof of Corollary \ref{c-BBS-semiconjugacy}]
	By Theorem \ref{t-BBS} and Corollary \ref{c-lm-semiconj}, it is enough to prove the claim about non-faithful actions. Let $\varphi\colon G\to \homeo_0(\R)$ be a non-faithful minimal action (so $[G_c, G_c]\subset \ker \varphi$). As in the previous proof, write $\pi\colon G\to \overline{G}=G/[G_c, G_c]$ for the quotient projection. If $\varphi$ is conjugate to an action by translations, then the claim follows directly from \ref{i-BBS-abelianisation} in Lemma \ref{l-BBS-algebraic}: indeed, as $A/I\Lambda\cdot A$ is finite (by condition \ref{i:quotientBS}), $\varphi$ must factor through $\Lambda\times \Lambda$.
	
	Suppose now that $\varphi$ is not conjugate to any action by translations. Then the same remains true for the restriction $\varphi|_{\Gfrag}$, since $\Gfrag$ has finite index by Lemma \ref{p.BieriStrebel_fg_germs}. After \ref{i-BBS-central} in Lemma \ref{l-BBS-algebraic}, we know that $\varphi(G_c)$ is contained in the center of $\varphi(\Gfrag)$; therefore, if $\varphi(G_c)\neq\{\id\}$, then by Theorem \ref{t-centralizer}  we deduce that $\pi(\Gfrag)$ admits a minimal proximal action on the circle. But the latter possibility is excluded by the fact that $\pi(\Gfrag)$ is solvable, and thus any action on the circle preserves a probability measure. Hence $\varphi(G_c)=\{\id\}$, i.e.\ $\varphi|_{\Gfrag}$ factors through $\Gfrag/G_c$. By Lemma \ref{p.BieriStrebel_fg_germs} and \ref{i-BBS-GmodGc} in Lemma \ref{l-BBS-algebraic}, we have that
	\[\Gfrag/G_c\cong \Aff(I \Lambda \cdot A, \Lambda)\times \Aff(I\Lambda\cdot A, \Lambda)\]
	is a direct product of finitely generated groups. Using Lemma \ref{l-direct-product-explicit}, {it is not difficult to deduce} that one of the factors acts trivially. Equivalently, the kernel of $\varphi$ contains one of the two subgroups $G_{-}$, $G_+$. Since $G_\pm =\ker\mathcal{G}_{\pm \infty}$, this shows that $\varphi$ factors through one of the two germ homomorphisms $\mathcal{G}_{\pm\infty}\colon G\to \Aff(A, \Lambda)$.
\end{proof}

\subsection{Proof of Theorem \ref{t-BBS}} \label{s-proof-BBS}
We now start working towards the proof of Theorem \ref{t-BBS}. Given $a\in A$ and $\lambda\in \Lambda$, we will denote by $g(a,\lambda)$ the affine transformation $x\mapsto \lambda x+(1-\lambda)a$, which is the unique element of $\Aff(A,\Lambda)$ which fixes $a$ and has slope $\lambda$. We will also consider the elements
\begin{equation}\label{eq BBS}
	g_+(a,\lambda)\colon x\mapsto \left\{\begin{array}{lr}x & \text{if }x\in (-\infty,a],\\[.5em]
		g(a, \lambda)(x) &\text{if } x\in [a,+\infty), \end{array} \right.
\end{equation}
and $g_-(a,\lambda):=g(a,\lambda)\,g_+(a,\lambda)^{-1}$. Note that $g_+(a,\lambda)\in G_{(a,+\infty)}$ and $g_-(a,\lambda)\in G_{(-\infty,a)}$. For $a\in A$, we keep denoting by $t_a$ the translation $x\mapsto x+a$.

For every $a\in A$ and $\lambda\in \Lambda$, if $h\in G(\R;A,\Lambda)$ is an element with no breakpoint on $(a,+\infty)$ (respectively on $(-\infty, a)$), we have \[hg_+(a,\lambda)h^{-1}=g_+(h(a),\lambda)\quad\text{(respectively, }hg_-(a,\lambda)h^{-1}=g_-(h(a),\lambda)).\]
In particular, we have the following relations for such elements (see \cite[\S B7]{BieriStrebel}):
\begin{equation}
	\label{e-bx-equivariance1} h \,g_\pm(a,\lambda)\, h^{-1}=g_\pm(h(a),\lambda) \quad\text{for every } h \in \Aff(A,\Lambda),
\end{equation}
as well as
\begin{equation*}
	g_+(a,\lambda)\,g_+(b,\mu)\,g_+(a,\lambda)^{-1}=g_+(g(a,\lambda)(b),\mu)  \quad\text{for every } a>b,
\end{equation*}
\begin{equation}
	\label{e-bx-equivariance2}
	g_-(a,\lambda)\,g_-(b,\mu)\,g_-(a,\lambda)^{-1}=g_-(g(a,\lambda)(b),\mu)  \quad\text{for every } a<b.
\end{equation} 
We also remark  that the subset
\[
\left\{t_a\right\}_{a\in A}\cup \left \{g(0,\lambda),g_+(0,\lambda)\right \}_{\lambda\in \Lambda}
\] 
is generating for  $G(\R; A, \Lambda)$ (see \cite[Theorem B7.1]{BieriStrebel}).

For what follows, the reader can keep in mind the following example.

\begin{ex}
	For $\lambda>1$, the Bieri--Strebel group $G(\lambda)$   is generated by the finite subset $\{ g(0,\lambda),g_+(0,\lambda), t_1\}$.
	\begin{rem}
		The group $\left \langle g(0,\lambda),g_+(0,\lambda), t_1\right \rangle$ appears  in the work of Bonatti, Lodha, and the last author \cite{BLT} (denoted as $G_\lambda$), where it was shown that, for certain algebraic numbers $\lambda>1$ (called Galois hyperbolic \textit{ibid.}) it admits no faithful $C^1$ action on the closed interval. The fact that this group coincides with the Bieri--Strebel group $ G(\lambda)$  was unnoticed in \cite{BLT}.
	\end{rem}
\end{ex}

To work towards the proof of Theorem \ref{t-BBS}, we need some technical results, stated in the following setting.

\begin{assumption}\label{ass.BieriStrebel}
	Fix a non-trivial multiplicative subgroup $\Delta\subseteq \R_{>0}$,  a $\Delta$-submodule $A\subset \R$, and let $H=G(\R; A, \Delta)$ be the corresponding Bieri--Strebel group. Moreover, we let $G \subseteq\homeo_0(\R)$ be a subgroup such that $\Gfrag$ is finitely generated, which contains $H$ as a subgroup. Finally, we assume that $\{B_x\}_{x\in \R}$ is a family of subgroups of $G$ with the following properties.
	\begin{enumerate}[label=(C\arabic*)]
		\item \label{i-L-nested} For each $x\in\R$ we have $\bigcup_{y<x} G_{(-\infty, y)} \subseteq B_x \subseteq G_{(-\infty, x)}$. 
		\item \label{i-L-equivariant} For every $x\in \R$ and every $g\in G$ we have $gB_xg^{-1}=B_{g(x)}$. 
		\item \label{i-L-contains-broken} For every $x\in A$ and every $\delta\in \Delta$ we have $g_-(x, \delta)\in B_x$. 
	\end{enumerate}
	Note that  $G_+=\bigcup_{x\in\R} B_x$,  by \ref{i-L-nested}. Finally, we assume that  $\varphi\colon G\to \homeo_0(\R)$ is a faithful minimal laminar action of $G$. Note that $\varphi$  is horograded by the standard action by Theorem \ref{t-lm-horograding}, and we shall assume that it is positively horograded.
\end{assumption}

\begin{rem}
	For the proof of Theorem \ref{t-BBS}, the reader can have in mind  the case where  $G=G(\R; A, \Lambda)$ is itself a Bieri--Strebel group, $H$ is a subgroup corresponding to some $\Delta\subseteq \Lambda$, and the $B_x$ are subgroups of $G_{(-\infty, x)}$ consisting of elements whose left-derivative at $x$ belongs to  some intermediate subgroup $\Delta\subseteq \Lambda_1\subseteq \Lambda$. However, a different choice of $G$ will be used later in \S \ref{s-no-actions}.
\end{rem}

To avoid confusion, \emph{we will write $X=\R$ for the real line on which the standard action of $G$ is defined}.    The proof of Theorem \ref{t-lm-horograding} gives us an explicit horograding of the action $\varphi\colon G\to \homeo_0(\R)$ by the standard action, as explained in Remark \ref{l-positive-horograding}.
Namely, recall that the groups $G_{(a, x)}$ are totally bounded for $\varphi$, and the set 
\[\Xi:=\textstyle\bigcap_{x\in X} \suppphi\left (G_{(-\infty, x)}\right )\]
is a $G_\delta$-dense  subset of $\R$. 
For $x\in X$ and $\xi\in \Xi$, we write $\Iphi(x,\xi)$ for the connected component of $\suppphi\left (G_{(-\infty, x)}\right )$ containing $\xi$. Then the collection
\[\mathcal{L}=\{\Iphi(x,\xi): x\in X, \xi\in \Xi\} \]
is an invariant prelamination. A prehorograding for $\varphi$ is given by $(\mathcal{L}, \hor)$, where
\[
\dfcn{\hor}{\mathcal L}{X}{\Iphi(x,\xi)}{x.}
\]
Note that the closure $\overline{\mathcal{L}}$ is obtained by adding to $\mathcal L$ the open intervals
\begin{equation*} \label{e-iout}
	\Iphiout(x, \xi):=\Int\left(\textstyle \bigcap_{y>x} \Iphi(y, \xi)\right) \quad \text{and}\quad \Iphiinn(x, \xi)=\textstyle \bigcup_{y<x} \Iphi(y, \xi).
\end{equation*}
The horograding $\hor$ naturally extends to $\mathcal{L}$ by mapping intervals of the form $\Iphiout(x, \xi)$ and $\Iphiinn(x, \xi)$ to $x$. Note also that $\Xi$ is the set of $\hor$-complete points (Definition \ref{d-pi-complete-point-line}).

\begin{lem} \label{l-bs-fix}
	Under Assumption \ref{ass.BieriStrebel}, the group $\Aff(A,\Delta)$ is homothetic for the laminar action $\varphi$. Its unique fixed point, that we denote by $\eta$, belongs to $\Xi$.
	
	In particular, the standard affine action of $\Aff(A,\Delta)$ on $X=\R$ is positively (respectively, negatively) semi-conjugate to its action $\varphi$ on $(\eta, +\infty)$ (respectively, $(-\infty, \eta)$). Explicitly, the map
	\[
	\dfcn{q_+}{\R}{(\eta,+\infty)}{x}{\sup \Iphi(x,\eta)}\] is  strictly increasing and $\Aff(A,\Delta)$-equivariant. 
	Similarly, the map 
	\[
	\dfcn{q_-}{\R}{(-\infty,\eta)}{x}{\inf \Iphi(x,\eta)}\] is strictly decreasing and $\Aff(A,\Delta)$-equivariant.
\end{lem}

\begin{proof}
	Every translation $t_a$, with $a\in A\setminus \{0\}$, acts without fixed points on $X$, therefore by Proposition \ref{p-dyn-class-elements-horograded}, its $\varphi$-image is a homothety, with a unique fixed point $\eta_a\in \Xi$. As the subgroup of translations is abelian, the point $\eta_a=:\eta$ does not depend on $a\in A\setminus \{0\}$. Moreover, as this subgroup is normal in $\Aff(A,\Delta)$, the point $\eta$ is fixed by $\varphi\left (\Aff(A,\Delta)\right )$.
	
	The second statement is a special case of Lemma \ref{l-pseudo-homothetic-horograding}. Monotonicity and equivariance of the  maps $q_\pm$ follow from the fact that the family of $\{\Iphi(x,\eta)\}_{x\in X}$ is increasing with respect to $x\in X$, and moreover one has the equivariance relation $g.\Iphi(x,\eta)=\Iphi(g(x),g.\eta)$ for every $x\in X$ and $g\in G$ (see \S \ref{ssc.CF_family}). Minimality of the action of $\Aff(A, \Lambda)$ on $X$ then implies that both maps are strictly monotone (since the union of the intervals on which they are locally constant would give an invariant open set).
\end{proof}

In what follows, we will always denote by $\eta$ the unique fixed point of $\varphi\left (\Aff(A,\Lambda)\right )$ provided by Lemma \ref{l-bs-fix}. For $x\in X$ and $\xi\in \suppphi(B_x)$, we will denote by $\Bphi(x, \xi)$ the connected component of $\suppphi(B_x)$ containing $\xi$. Note that condition \ref{i-L-nested} implies that $\Bphi(x, \xi)$ is increasing with respect to $x\in\R$, and moreover
\[\Iphiinn(x, \xi)\subseteq \Bphi(x, \xi)\subseteq \Iphi(x, \xi).\] 
The key point is to establish the following strict inclusion when $x\in A$. 

\begin{lem}\label{l-bbs-lex}
	Under Assumption \ref{ass.BieriStrebel}, assume that there exists $g\in H$ such that $g.\eta\neq \eta$.  Then, for every  $x\in A$ we have a strict inclusion $\Iphiinn(x, \eta)\subsetneq \Bphi(x, \eta)$.  
\end{lem}
\begin{proof}
	Assume by contradiction that $\Iphiinn(x, \eta)=\Bphi(x, \eta)$, for some $x\in A$. Note that then this is automatically true for every $x\in A$, since the group $\Aff(A, \Delta)$ acts transitively on $A$ (it contains all translations by elements in $A$) and fixes $\eta$, so that for $h\in \Aff(A, \Delta)$ we have $h.\Iphiinn(x, \eta)=\Iphiinn(h(x), \eta)$ and 
	\begin{equation}\label{eq:L-equivariance}
		h.\Bphi(x, \eta)=\Bphi(h(x), \eta)
	\end{equation}
	(after condition \ref{i-L-equivariant}). 
	
	Fix $x\in A$, and choose $\delta\in \Delta$ with $\delta>1$ and such that $g_-(x, \delta).\eta\neq \eta$. Such a $\delta$ exists because the elements $g_-(x, \delta)$ together with $\Aff(A, \Delta)$ generate $H$, and we assume that $\varphi(H)$ does not fix $\eta$. Note also that once such an element $g_-(x, \delta)$ is found, it follows that $g_-(y, \delta).\eta\neq \eta$ for every $y\in A$, since these elements are all conjugate to each other by elements  of $\Aff(A, \Lambda)$. Note that by condition \ref{i-L-contains-broken} the image  $\varphi(g_-(x, \delta))$ must preserve $\Bphi(x, \eta)$. Now, consider the collection
	\[\mathcal{L}_0=\{I\in \overline{\mathcal{L}}  : I\subsetneq\Bphi(x, \eta)\}.\]
	The assumption that $\Iphiinn(x, \eta)=\Bphi(x, \eta)$ implies that $\mathcal{L}_0$ is a covering lamination of $\Bphi( x, \eta)$. Indeed,  Lemma \ref{l-bs-fix} implies that $\Iphi(\eta, y)\Subset\Iphi(\eta, z)$ for $y<z<x$, so that $\Bphi(x, \eta)=\bigcup_{y<x} \Iphi(\eta, y)$ is exhausted by relatively compact subintervals in $\mathcal{L}_0$. Combined with the fact that elements of $\mathcal{L}_0$ do not cross (as $\mathcal L_0\subset \overline{\mathcal{L}}$), this also implies that no $I\in \mathcal{L}_0$ can share an endpoint with $\Bphi(x, \eta)$, so that all elements of $\mathcal{L}_0$ are relatively compact subintervals of $\Bphi(x, \eta)$. 
	It follows that the restriction $\hor_0$ of $\hor$ to $\mathcal{L}_0$ defines a horograding of the $\varphi$-action of $B_x$ on $\Bphi(x, \eta)$ by its standard action on $(-\infty, x)$.
	From this and from Proposition \ref{p-dyn-class-elements-horograded}, we deduce that $\varphi(g_-(x, \delta))$ acts on $\Bphi( x, \eta)$ as an expanding homothety. Let us denote by $\xi_x\in \Bphi( x, \eta)$ the unique fixed point of $\varphi(g_-(x, \delta))\restriction_{\Bphi(x, \eta)}$. Note that $\xi_x\neq \eta$, by the choice of $\delta$. Without loss of generality, we assume that $\xi_{z}>\eta$ for some $z\in A$. Then, we have the following.
	
	\begin{claim}
		We have $\xi_y>\eta$ for every $y\in A$, and the map  $x\mapsto \xi_y$ is monotone increasing.
	\end{claim}
	\begin{proof}[Proof of claim]
		The relations  \eqref{e-bx-equivariance1} and \eqref{eq:L-equivariance} give that the map $x\mapsto \xi_x$ is $\Aff(A,\Delta)$-equivariant. The conclusion follows using that $\Aff(A, \Delta)$ acts transitively on $A$ and that, after Lemma \ref{l-bs-fix}, we know that the action of $\Aff(A,\Delta)$ on $(\eta, +\infty)$ is positively semi-conjugate to the standard affine action.
	\end{proof}

	Now, by the assumption that $\Iphiinn(x, \eta)=\Bphi(x, \eta)$, and by \ref{i-L-nested}, we can find $y\in A$ with $y<x$ such that $\Bphi( y,\eta)$ contains $\xi_x$. After the claim, we have  $\eta<\xi_y<\xi_x$. Since $\xi_x$ is a repelling fixed point for $\varphi(g_-(x,\delta))$, we have the inclusion $g_-(x,\delta).\Bphi(y,\eta)\supset \Bphi(y,\eta)$. Since the latter contains $\eta$, we have 
	\[g_-(x,\delta).\Bphi(y,\eta)=\Bphi\left ( g_-(x,\delta)(y),\eta\right ).\]
	Then \eqref{e-bx-equivariance2} implies that $g_-(x,\delta).\xi_y=\xi_{g_-(x,\delta)(y)}$. However,on the one hand  $\xi_y<\xi_x$ gives the inequality $g_-(x,\delta).\xi_y<\xi_y$. On the other hand, $\xi_{g_-(x,\lambda)(y)}>\xi_y$ since $g_-(x,\lambda)(y)>y$, and the map $y\mapsto \xi_y$ is increasing (after the claim). This gives the desired contradiction.
\end{proof}

\begin{lem}\label{c-bs-L-semiconjugate}
	Under the same assumptions as in Lemma \ref{l-bbs-lex}, for every $x\in A$, the $\varphi$-action of $B_x$ on $\Bphi(x,\eta)$ is semi-conjugate to a non-faithful action induced from an action of the group of left germs  $\Germ\left (B_{x}, x\right )$.
\end{lem}

\begin{proof}
	Lemma \ref{l-bbs-lex} implies that the normal subgroup $\left (G_{(-\infty, x)}\right )_+=\bigcup_{y<x} G_{(-\infty, y)}$ of $B_x$ has fixed points in $\Bphi(x, \eta)$ (namely the endpoints of $\Iphiinn(x, \eta)$). Thus, the action of $\varphi(B_x)$ on $\Bphi(x, \eta)$ is semi-conjugate to an action induced from the quotient $B_x/\left (G_{(-\infty, x)}\right )_+\cong \Germ\left (B_{x}, x\right )$. \qedhere
\end{proof}

The next result is the only place where a particular choice of the family $\{B_x\}$ is needed. From now on, we withdraw the notation $X=\R$ for more explicit statements.

\begin{prop}\label{cor.intmov}
	Let $G=G(\R;A,\Lambda)$ be a Bieri--Strebel group satisfying conditions \ref{i:LambdaBS}--\ref{i:quotientBS}. Let $\varphi\colon G\to \homeo_0(\R)$ be a  faithful minimal laminar action, positively horograded by the standard action on $\R$. Assume there exists an element $g=g_{-}(x,\lambda)\in G$ such that $g.\eta\neq \eta$.  Then $g.\Iphiinn(x, \eta)\cap\Iphiinn(x, \eta)=\emptyset$.
\end{prop}
\begin{proof} Take $g=g_{-}(x,\lambda)$ not fixing $\eta$, and consider the Bieri--Strebel group $H=G(\R;A,\langle\lambda\rangle_\ast)$, which is a subgroup of $G$. We consider the family $\{B_x\}_{x\in \R}$ defined by
	\[
	B_x=\left \{ h\in G_{(-\infty,x)}: D^-h(x)\in \langle \lambda\rangle_\ast\right \}.
	\]
	It is straightforward to verify that Assumption \ref{ass.BieriStrebel} is fulfilled by such choices. Therefore, by Lemma \ref{c-bs-L-semiconjugate}, the action of $B_x$ on $\Bphi(x,\eta)$ is semi-conjugate to an action that factors through the germ homomorphism $\Gcal_x\colon B_x\to\Germ\left (B_{x}, x\right )$. Since in this case $\Germ\left (B_{x}, x\right )$ is generated by $\Gcal_x(g)$, we conclude that $\fixphi(g)\cap \Bphi(x,\eta)=\emptyset$. 
	On the other hand, by Lemma \ref{l-bbs-lex} we get that $\Bphi(x, \eta)$ strictly contains $\Iphiinn(x, \eta)$. Then, since $\{\Iphiinn(x, \xi):x,\xi\in\R\}$ is a prelamination preserved by the action, we must have $g.\Iphiinn(x, \eta)\cap\Iphiinn(x, \eta)=\emptyset$, as desired.
\end{proof}

The next two lemmas analyze properties of the jump preorders. The first one gives decompositions for elements in $G_+= \bigcup_{x<\infty} G_{(-\infty, x)} $, which are well suited for our purposes. The second one allows to identify dynamical realizations of jump preorders. 

\begin{lem}\label{lem.decomp} Let $G=G(\R;A,\Lambda)$ be a Bieri--Strebel group,  and  let $\leq_\Lambda$ be a preorder on $\Lambda$. Let $\preceq$ be the corresponding positive jump preorder, and take $g\in G_+$ with $\mathsf{id}\precneq g$. Then, there exist elements $h,k,g_{-}(y,\lambda)\in G_+$ satisfying the following conditions: \begin{enumerate}[label=(\roman*)]
		\item\label{i.decomp1} $g=kg_{-}(y,\lambda)h$,
		\item\label{i.decomp2} $h\in [\id]_{\preceq}$,
		\item\label{i.decomp3} $1\lneq_\Lambda \lambda$, and
		\item\label{i.decomp4} $k\in G_{(-\infty,z)}$ for some $z<y$.
	\end{enumerate} 
	The analogous result holds for the negative jump preorder. 
\end{lem}
\begin{proof} Write $\Lambda_0=[1]_{\leq_\Lambda}$, as usual. Since we are assuming $\mathsf{id}\precneq g$, we can consider the point $y:=\overline{x}_{g}=\max\{x\in\R:j_g(x)\not\equiv_\Lambda 1\}$. Since we assume  $g\in G_+$, the restriction of $g^{-1}$ to $(y, +\infty)$ has all slopes in $\Lambda_0$, and thus so does the restriction of $g$ to $(z, +\infty)$, where $z:=g^{-1}(y)$. It follows that we can choose $h\in G(\R;A,\Lambda_0)_+$ coinciding with $g$ on $(z, +\infty)$. Consider the product $f=gh^{-1}$.  The rightmost point of $\supp(f)$ is $y=g(z)=h(z)$, and by the chain rule we have 
	\[D^{-}f(y)=D^{-}g(z)D^-h^{-1}(y)=\,D^-h(z)^{-1}/D^{-}g^{-1}(y).\]
	Note now that $\id\precneq g$ implies that 
	\[1\lneq_{\Lambda} j_g(y)=D^+ g^{-1}(y)/D^{-}g^{-1}(y),\]
	and since $D^+ g^{-1}(y), D^-h(z)\in \Lambda_0$,  we get $1\lneq_\Lambda D^{-}f(y)=:\lambda$. As before, we have that the rightmost point of the support of $fg_{-}(y,\lambda)^{-1}$  coincides with $g_{-}(y,\lambda)(z)$, where $z$ is the second largest breakpoint of $f$ (the one before $y$). Then, write $k=fg_{-}(y,\lambda)^{-1}$. Then the decomposition $g=kg_{-}(y,\lambda)h$ satisfies conditions \ref{i.decomp1}--\ref{i.decomp4} in the statement. 
\end{proof}

\begin{lem}\label{lem.equalpreorder} Consider a Bieri--Strebel group $G=G(\R;A,\Lambda)$ and  a preorder $\preceq\in\LPO(G)$ containing $\Aff(A,\Lambda)$ in its residue. Assume further that $\preceq'\in\LPO(G)$ is a positive (respectively, negative) jump preorder coinciding with $\preceq$ over $G_+$ (respectively, $G_-$). Then $\preceq$ and $\preceq'$ are the same preorder. 
\end{lem}
\begin{proof} Assume that $\preceq'$ is the positive jump preorder associated with the preorder $\leq_\Lambda\in\LPO(\Lambda)$, the case where $\preceq'$ is a negative jump preorder is analogous. Denote by $\Lambda_0$ the residue of $\leq_\Lambda$ and notice that in this case the residue of $\preceq'$ is the subgroup
	\[H:=\{g\in G:j_g(x)\in\Lambda_0\ \forall x\in\R\}\]
	(see Definition \ref{dfn-jumppreorder}). Since elements of $\Aff(A,\Lambda)$ have constant derivative, it holds that $j_g(x)=1$ for every $g\in\Aff(A,\Lambda)$ and $x\in\R$. In particular we have $\Aff(A,\Lambda)\subseteq H$. 
	
	Note that $G$ decomposes as $G=G_+\rtimes\Aff(A,\Lambda)$. Then for every $g\in G$ we can write $g=g_+a_g$, with $g_+\in G_+$ and $a_g\in\Aff(A,\Lambda)$. Denote by $P$ and $P'$ the positive cones of $\preceq$ and $\preceq'$, respectively. Since $\Aff(A,\Lambda)$ is contained in the residue of both $\preceq$ and $\preceq'$, it holds that $g\in P$ if and only if $g_+\in G_+\cap P$, and also that $g\in P'$ if and only if $g_+\in G_+\cap P'$. Finally, since by assumption the equality $G_+\cap P=G_+\cap P'$ holds, the lemma follows. 
\end{proof}

\begin{proof}[Proof of Theorem \ref{t-BBS}]
	The assumptions on $G=G(\R;A,\Lambda)$ ensure that $\Gfrag$ is finitely generated (see Lemma \ref{p.BieriStrebel_fg_germs}). After Theorem \ref{t-lm-horograding}, we only need to show that a  faithful minimal laminar action $\varphi\colon G\to \homeo_0(\R)$, positively (respectively, negatively) horograded by its standard action on $\R$, is conjugate to an action of the form $\varphi_{+, \leq_\Lambda}$ (respectively, $\varphi_{-, \leq_\Lambda}$) for some preorder $\leq_\Lambda$.  We will only discuss the case of positive horograding, the negative case being totally analogous.
	
	We write $\eta\in \R$ for the unique fixed point of $\varphi\left (\Aff(A, \Lambda)\right )$ given by Lemma \ref{l-bs-fix} (applied to the case $\Delta=\Lambda$). In what follows, let $\preceq$ be the preorder on $G$ induced by $\eta$, namely by declaring $g\precneq h$ if and only if $g.\eta<h.\eta$. We will show that this preorder coincides with a positive jump preorder associated with some $\leq_\Lambda\in\LPO(\Lambda)$. Let us first give a candidate for the preorder $\leq_\Lambda$. For $x\in A$, the set of elements $T_x:=\left \{g_-(x, \lambda): \lambda \in \Lambda\right \}$ is a subgroup of $G$ isomorphic to $\Lambda$, which is a section inside $G_{(-\infty, x)}$ of the group of germs $\Germ\left (G_{(-\infty, x)}, x\right )$.  We put on $\Lambda$ the preorder $\leq_\Lambda$ given by restricting $\preceq$ to this subgroup, namely by setting $\lambda\lneq_\Lambda \mu$ if $g_-(x, \lambda).\eta<g_-(x, \mu).\eta$. Note that this preorder does not depend on the choice of $x\in A$, as for $x, y\in A$ the groups $T_{x}$ and $T_{y}$ are conjugate by an element of $\Aff(A, \Lambda)$ (see \eqref{e-bx-equivariance1}), which fixes $\eta$. Denote by $\preceq'$ the positive jump preorder in $G$ associated with $\leq_\Lambda$. We proceed to show that $\preceq'$ and $\preceq$ coincide on $G_+$. By Lemma \ref{lem.equalpreorder}, this will conclude the proof.
	
	Denote by $\Lambda_0$ the residue of the preorder $\leq_\Lambda$, and notice that in this case the residue of $\preceq'$ equals
	\[H=\{g\in G:j_g(x)\in\Lambda_0\ \forall x\in \R\}.\]
	Also notice that if $g\in G_+$ and $j_g(x)\in\Lambda_0$ for every $x\in\R$, then $D^{\pm}g(x)\in\Lambda_0$ for every $x\in\R$. Thus, we have the equality $H\cap G_+=G(\R; A, \Lambda_0)_+$. Note that $H\cap G_+$ fixes $\eta$, since it is generated by $\left \{g_-(x, \lambda): x\in A, \lambda\in \Lambda_0\right \}$ (this can be easily checked from \cite[\S A8.1]{BieriStrebel}). Thus, we have \begin{equation}\label{ecuasion} G_{+}\cap [\id]_{\preceq'}\subseteq G_+\cap[\id]_{\preceq}. \end{equation}
	Assume now that $g\in G_+$ satisfies $\id\precneq' g$.  We proceed to show that in this case $\eta<g.\eta$, which implies that $\id\precneq g$. For this, consider the decomposition $g=kg_-(y,\lambda)h$ given by Lemma \ref{lem.decomp}. Then, by \ref{i.decomp2} in Lemma \ref{lem.decomp} and  \eqref{ecuasion}, we get $h.\eta=\eta$ and therefore $g.\eta=kg_-(y,\lambda).\eta$. On the other hand, \ref{i.decomp3} in Lemma \ref{lem.decomp} together with the definition of $\leq_\Lambda$ imply that $g_-(y,\lambda).\eta>\eta$. Then by Proposition \ref{cor.intmov}, we get 
	\[g_-(y,\lambda).\Iphiinn(y, \eta)\cap\Iphiinn(y, \eta)=\emptyset.\]
	Finally, by \ref{i.decomp4}  in Lemma \ref{lem.decomp} we have that $k.\Iphiinn(y, \eta)=\Iphiinn(y, \eta)$, and this, in light of what we have already done, shows that $g.\eta=kg_-(y,\lambda).\eta>\eta$ as desired. Analogously, one shows that if $g\precneq' \id$, then $g\precneq \id$. This shows that the preorders $\preceq'$ and $\preceq$ coincide on $G_+$ and concludes the proof. 
\end{proof}

\section{A relative of Bieri--Strebel groups with no exotic actions} \label{s-no-actions}
Here we build on the previous results on Bieri--Strebel groups to construct an example of a finitely generated, fragmentable locally moving group which admits no faithful minimal laminar action at all. In  particular every faithful minimal action on $\R$ is conjugate to the standard action.

The starting point of the construction is the Bieri--Strebel group $G(2)=G(\R;\Z[1/2],\langle 2\rangle_* )$  of all finitary dyadic PL  homeomorphisms of $\R$, which we already proved to admit only two faithful minimal laminar actions (Theorem \ref{t-BBS} and Example \ref{ex:focal_G_lambda}). 

As a consequence of Corollary \ref{c-exotic-pp}, if we want to avoid the existence of faithful minimal laminar actions, we must leave the setting of \emph{finitary} PL transformations. We will consider groups whose elements are PL with countably many breakpoints that accumulate on some finite subset of ``higher order'' singularities (with some control on these).

Given an open interval $X\subseteq \R$, we say that a homeomorphism $f\in \homeo_0(X)$ is  \emph{locally PL} if there is a finite subset $\Sigma\subset X$ such that $f$ is (finitary) PL in $X\setminus \Sigma$. For such an $f$, we denote by $\BP^2(f)\subset X$ the smallest subset such that $f$  is PL on $X\setminus \BP^2(f)$. The set $\BP^2(f)$ is the set of \emph{second-order}  breakpoints of $f$. 
Points $x\in X\setminus \BP^2(f)$ where $f$ has discontinuous derivative are called \emph{first-order} breakpoints, and we denote them by $\BP^1(f)$. Also, we write $\BP(f)=\BP^1(f)\cup \BP^2(f)$ for the set of  breakpoints of $f$. Clearly, when $\BP^2(f)=\varnothing$ we  have that $f$ is PL. We will silently use a couple of times the observation that for $f$ and $g$ locally PL, we have that $\BP^2(fg)\subset \BP^2(g)\cup g^{-1}\BP^2(f)$.

\begin{dfn}
	Let $X\subseteq \R$ be an open interval. We write $G(X)=G(X;\Z[1/2],\langle 2\rangle_*)$ for the Bieri--Strebel group (see Definition \ref{d.BieriStrebel}). We also denote by $G_\omega(X)$ the group of all locally PL homeomorphisms of $X$ with the following properties:
	\begin{itemize}
		\item   $f$ is locally \emph{dyadic} PL, that is at each $x\in X\setminus \BP(f)$, the map  $f$ is locally an affine map of the form $x\mapsto 2^nx+b$ for $n\in \Z$ and $b\in \Z[1/2]$;
		\item  breakpoints of $f$ are contained in a compact subset of $X$: $\BP(f)\Subset X$;
		\item breakpoints of $f$ and their images are dyadic rationals: $\BP(f)\cup f(\BP(f))\subset \Z[1/2]$.
	\end{itemize}
\end{dfn}

The group $G_\omega(X)$ is uncountable, so too big for our purposes. We will instead consider some subgroups defined in terms of the local behavior at the second-order breakpoints. Here we keep the notation from the previous section, such as $g(a,\lambda)$, $g_\pm(a,\lambda)$ (see \eqref{eq BBS}), and $t_a\colon x\mapsto x+a$, which denote elements in $\PL(\R)$. For $r\in \R$, we also write $h_r=g(r,\tfrac12)$, which corresponds to the homothety of ratio $1/2$ centered at $r$, and similarly we write $h_{r\pm}=g_{\pm}(r,\tfrac12)$ for shorthand notation.

\begin{dfn}
	Let $g\colon I\to J$ be a homeomorphism between two open intervals. We say that $g$ has a  \emph{$2^n$-scaling germ} at $r\in I$, if there exists a neighborhood $U$  of $r$ such that  $g h_r^n\restriction_U=h_{g(r)}^n g\restriction_U$.
\end{dfn}

\begin{rem}
	Note that when $g(r)=r$ this simply means that the germ of $g$ at $r$ commutes with the germ of $h_r^n$. More generally, if $g(r)\neq r$ and  $h$ is any PL map such that $hg(r)=r$,  then $g$ has a $2^n$-scaling germ at $r$ if and only if the germs of $hg$ and $h_r$ at $r$ commute. This does not depend on the choice of $h$, since every PL map has $2^n$-scaling germ (and more generally $k$-scaling germ for any $k>0$, with the obvious extension of the definition) at every point, including breakpoints. 
\end{rem}

\begin{dfn}
	Given  an open interval $X\subseteq \R$ and $n\ge 1$, we let $G_\omega^{(n)}(X)$ be the subgroup of $\Dsf_\omega(X)$ consisting of elements that have $2^n$-scaling germs at every second-order breakpoint (and thus at all points $x\in X$).
\end{dfn}

For every dyadic point $x\in X$, let  $\Dcal^{(n)}_{x}$ be the group of germs at $x$ of elements in $\Dsf_\omega^{(n)}(X)$ which fix $x$. That is, $\Dcal^{(n)}_{x}$ is the group of germs of homeomorphisms that are locally dyadic PL away from $\{x\}$, and that commute with $h_x^n$. We denote by $\Dcal^{(n)}_{x-}$ and $\Dcal^{(n)}_{x+}$ the corresponding groups of left and right germs, respectively, so that $\Dcal^{(n)}_{x}\cong \Dcal^{(n)}_{x-}\times \Dcal^{(n)}_{x+}$. The groups  $\Dcal^{(n)}_{x-}$ and  $\Dcal^{(n)}_{x+}$ are isomorphic to a well-known group, namely the lift  $\widetilde{T}\subset \homeo_0(\R)$ of Thompson's group $T$ acting on the circle. Explicitly, $\widetilde{T}$ is the group of all dyadic PL homeomorphisms of $\R$ which commute with the unit translation $t_1\colon x\mapsto x+1$. The point is that for every $n\ge 1$ and $x\ge 1$ dyadic, the map $h_{x}^n\restriction_{(-\infty,x)}$  can be conjugated to the translation $t_1$ by a  dyadic PL homeomorphism $f\colon (-\infty, x)\to \R$. This establishes an isomorphism of $\Dcal^{(n)}_{r-}$ with the group of germs of $\widetilde{T}$ at $+\infty$, which is isomorphic to $\widetilde{T}$ itself. Similarly one argues for $\Dcal^{(n)}_{r+}$. This fact will be constantly used in what follows. 
A first consequence is that the groups $\Dcal^{(n)}_{x-}$ and $\Dcal^{(n)}_{x+}$ are finitely generated, since $\widetilde{T}$ is so. 
This leads to the following.
\begin{prop}\label{p-doubling-fg}
	For every dyadic open interval $X=(a, b)\subseteq \R$ and every $n\ge 1$,  the group $G:=\Dsf_\omega^{(n)}(X)$ is finitely generated {and fragmentable.}
\end{prop}
Note that, as breakpoints of every element in $\Dsf_\omega^{(n)}(X)$ are contained in a compact subset of $X$, the group $\Germ(G, a)$ is infinite cyclic if $a>-\infty$, and isomorphic to $\BS(1,2)$ if $a=-\infty$. The corresponding statement holds for $\Germ(G, b)$. 

\begin{proof}[Proof of Proposition \ref{p-doubling-fg}]
	Fix a dyadic point $x \in (a, b)$. Since the group of germs $\Dcal^{(n)}_x=\Dcal^{(n)}_{r-}\times \Dcal^{(n)}_{r+}$  is finitely generated, we can find a finite subset of elements $S\subset \Dsf_\omega^{(n)}(X)$ fixing $x$, whose germs generate $\Dcal^{(n)}_x$ and that have no second-order breakpoints apart from $x$.
	
	\begin{claim}
		We have $\Dsf_\omega^{(n)}(X)=\langle \Dsf(X), S\rangle$.
	\end{claim}
	
	\begin{proof}[Proof of claim]
		Let $g\in \Dsf_\omega^{(n)}$, and let us show that $g\in \langle \Dsf(X), S\rangle$ by induction on the number $k=|\BP^2(g)|$ of second-order breakpoints of $g$. If $k=0$, then $g\in \Dsf(X)$. Assume that $k\ge 1$, and let $y\in \BP^2(g)$ be a second-order breakpoint of $g$. Since $G(X)$ acts transitively on dyadic rationals, we can choose $h_1\in \Dsf(X)$ such that $h_1(g(x))=x$. As $\BP^2(h_1)=\varnothing$, we have that the element $g'=h_1g$ satisfies $|\BP^2(g')|=k$, and moreover $x$ belongs to $\BP^2(g')$ and is fixed by $g'$. Choose $h_2\in \langle S\rangle$ whose germ at $x$ is equal to the germ of $g$. By the choice of $S$, we have $\BP^2(g)=\{x\}$, so that for the element $g''=h_2^{-1}g'$ we have
		$\BP^2(g'')=\BP^2(g')\setminus \{x\}$, and thus $\BP^2(g'')=k-1<k$. By induction, we have $g''\in \langle \Dsf(X), S\rangle$, and it follows that $g=h_1^{-1}h_2g''\in \langle \Dsf(X), S\rangle$.
	\end{proof}
	Since the Bieri--Strebel group $\Dsf(X)=G(X;\Z[1/2],\langle 2\rangle_*)$ is finitely generated for dyadic $X$, from the claim we get that $G$ is finitely generated as well. The fact $\Dsf^{(n)}_\omega(X)$ is fragmentable (i.e.\ generated by elements with trivial germ at one endpoint) follows from the facts that the group $G(X)$ is itself fragmentable, and that the set $S$ in the claim can be chosen to be supported in a relatively compact subinterval $I\Subset X$.
\end{proof}

Here is the main result of this section, whose proof will need some preliminary lemmas.

\begin{thm} \label{t-doubling-actions}
	For $n\ge 2$, every faithful minimal  action   $\varphi\colon \Dsf_\omega^{(n)}(\R) \to \homeo_0(\R)$ is conjugate to its standard action.
\end{thm}

Until the end of the section, for fixed $n\ge1$, we write $G=\Dsf_\omega^{(n)}(\R)$ and $H=\Dsf(\R)=G(2)$, so that $H\subseteq G$. 

\begin{lem}[Upgrading fixed points] \label{l-fix-upgrade}
	With notation as above, let $\varphi\colon G\to \homeo_0(\R)$ be an action on the real line. Then every  fixed point of $\varphi(H)$ must be fixed by $\varphi(G)$. In other words, $\fixphi(H)=\fixphi(G)$.
\end{lem}

\begin{proof}
	Consider the subgroups
	\[K_l=\left \{g\in G_{(-\infty,0)}: \BP^2(g)\subset\{0\}\right \}\quad\text{and}\quad 
	K_r=\left \{g\in G_{(0,+\infty)}: \BP^2(g)\subset\{0\}\right \},
	\]
	and set $K=\langle K_l, K_r\rangle \cong K_l\times K_r$. Note that group $K$ realizes the group of germs $\Dcal^{(n)}_0$; after the assumption on second-order breakpoints, the claim in the proof of Proposition \ref{p-doubling-fg} gives that $G=\langle H, K\rangle=\langle H,K_l, K_r\rangle$. 
	
	Consider the subgroup $H_l\subseteq K_l$ consisting of all elements whose germ at $0$ is given by a power of $h_{0-}^n$. In particular, every $g\in H_l$ satisfies $\BP^2(g)=\varnothing$, hence $H_l$ is a subgroup of $H$.
	Since the germ of $h_{0-}^n$ is central in $\Dcal^{(n)}_{0-}$, we have that $H_l$ is normal in $K_l$, with quotient $K_l/H_l\cong \Dcal^{(n)}_{0-}/\langle h_{0-}^n\rangle$ which is isomorphic to Thompson's group $T$ acting on the circle. The same considerations hold for the subgroup $H_r\subseteq K_r$ defined analogously. 
	
	Assume now that $\varphi\colon G\to \homeo_0(\R)$ is an action such that $\fixphi(H)\neq \varnothing$, so that $\fixphi(H_l)$ is non-empty and contains $\fixphi(H)$. Then $\varphi(K_l)$ preserves $\fixphi(H_l)$, and the $\varphi$-action of $K_l$ on $\fixphi(H_l)$ factors through the quotient $K_l/H_l\cong T$. Since $T$ is a simple group and contains elements of finite order, every order-preserving action on a totally ordered set is trivial. Thus the action of $K_l$ on $\fixphi(H_l)$ is actually trivial, and in particular it fixes $\fixphi(H)$. Similarly, so does $K_r$. Since $G=\langle H, K_l, K_r\rangle$, this implies that every point in $\fixphi(H)$ is fixed by $\varphi(G)$.
\end{proof}

The next lemma makes use of the assumption that $n\ge 2$ in Theorem \ref{t-doubling-actions}, and leverages the fact that the group $\widetilde{T}$ admits only one action on the real line up to semi-conjugacy. In the statement, with abuse of notation, we identify $h_{x-}$ with its germ in  $\Dcal^{(n)}_{x-}$

\begin{lem} \label{l-Ttilde-semiorder}
	For every $n\ge 2$, the group $\Dcal^{(n)}_{x-}$ admits no (non-trivial, left-invariant)  preorder which is invariant under conjugation by the element $h_{x-}$.
\end{lem}
\begin{proof}
	The natural isomorphism $\Dcal^{(n)}_{x-}\cong \widetilde{T}$ maps $h_{x-}$ to an element $h\in \widetilde{T}$ which is an $n$th root of the translation $t_1$, i.e.\ $h^n=t_1$. So it is enough to show that $\widetilde{T}$ admits no preorder invariant under conjugation by such an $h$. Assume by contradiction that $\prec$ is such a preorder. By a result of two of the authors \cite[Theorem 8.7]{MatteBonTriestino}, the dynamical realization of $\prec$ is semi-conjugate to the standard action of $\widetilde{T}$ on the real line, so that the maximal $\prec$-convex subgroup $K$ must be equal to the stabilizer $\widetilde{T}_y$ of some point $y\in \R$ for the standard action. On the other hand, $K$ must be normalized by $h$, so that we must have $\widetilde{T}_y=\widetilde{T}_{h(y)}$. However, since $h^n=t_1$, we have $h(y)\neq y$ and $|h(y)-y|<1$, so that $y$ and $h(y)$ have different projections to the circle $\R/\Z$. But any two distinct points in the circle have different stabilizers in Thompson's group $T$, and thus $y$ and $h(y)$  have different stabilizers in $\widetilde{T}$, which is a contradiction.
\end{proof}

\begin{proof}[Proof of Theorem \ref{t-doubling-actions}]
	Let $\varphi\colon G\to \homeo_0(\R)$ be an irreducible action.
	{Since $G$ is finitely generated and fragmentable}, we can apply Theorem \ref{t-lm-horograding}. We then assume that $\varphi$ is faithful and minimal. By symmetry, it is enough to exclude that $\varphi$ is laminar, positively horograded by the standard action of $h$ on $\R$. {Note that by Lemma \ref{l-fix-upgrade} we know that $\varphi(H)$ has no  fixed point.}
	In order to fulfill Assumption \ref{ass.BieriStrebel}, we will consider the family of subgroups $\left \{B_x\right \}_{x\in \R}$, where $B_x=G_{(-\infty,x)}$, and with this choice we will simply have $\Bphi(x,\xi)=\Iphi(x,\xi)$.
	We apply Lemma \ref{l-bs-fix}: let $\eta$ be the unique fixed point of $\varphi\left (\Aff(A,\Lambda)\right )$. Fix a dyadic rational  $x\in \R$, and consider the preorder $\prec_{\eta}$ on $G_{(-\infty, x)}$ associated with the action of $G_{(-\infty, x)}$ on $\Iphi(x,\eta)$. By Lemma \ref{c-bs-L-semiconjugate}, this preorder descends to a non-trivial preorder $\bar{\prec}_{\eta}$ on $\Germ\left (G_{(-\infty, x)}, x\right )=\Dcal_{x-}$. 
	Consider now the element $h_x\in \Aff(A,\Lambda)$. Since $h_x$ fixes $x$, it normalizes $G_{(-\infty, x)}$; moreover, it fixes $\eta$ and thus $\varphi(h_x)$  preserves $\Iphi(x,\eta)$. We also see that the preorder $\prec_{\eta}$ is invariant under the automorphism induced by $h_x$  on $G_{(-\infty, x)}$. But this automorphism coincides with the inner automorphism defined by conjugation by $h_{x-}$, so that the preorder $\prec_{\eta}$, and thus $\bar{\prec}_{\eta}$, must be invariant under conjugation by $h_{x-}$. This is in contradiction with Lemma \ref{l-Ttilde-semiorder}.
\end{proof}

\chapter{The point of view of trees}
\label{sec_focal_trees}

This chapter is a complement to Chapter \ref{sec.focalgeneral}. We describe how laminar actions and horogradings can be reinterpreted in terms of actions on certain planar (real) trees naturally constructed from the lamination. Beyond helping the intuition, this point of view will be useful in the next chapters in building and understanding various examples of exotic actions on the line  (for which in a certain sense the associated tree comes first, and the problem of finding an invariant planar on it comes later).
The content of this chapter is rather elementary, but unfortunately it becomes a bit technical when coming to proofs. For this reason, we have tried to keep a good balance between providing all details and keeping proofs concise.

\section{Terminology for trees}\label{ss.trees_preliminaries}  
\subsection{Directed trees} Roughly speaking, a real tree is a space $\Tbb$ obtained by gluing copies of the real line in such a way that no closed loops appear.  We are specifically interested in \emph{directed} real trees, which are trees together with a preferred direction to infinity (equivalently, a preferred end $\omega\in \partial \Tbb)$. 
Usually real trees are defined as metric spaces, but we adopt  the point of view of Favre and Jonsson \cite{Fav-Jon}, and introduce the following definition which does not refer to any metric or topology. (The equivalence with the more familiar metric notion will be discussed in \S \ref{subsec metric R-tree} below.)

\begin{dfn}\label{def d-trees}
	A  \emph{directed tree} is a poset $(\Tbb, \treeorder)$ with the following properties.
	\begin{enumerate}[label=($\Tbb$\arabic*)]
		\item \label{i-tree} For every $v\in \Tbb$, the subset $\{u\in \Tbb : v \treeordereq u\}$  is totally ordered and order-isomorphic to  a half-line $[0, +\infty)$. 
		\item \label{i-directed} Every two points  $u, v\in \Tbb$ have a smallest common upper bound, denoted as $u\curlywedgeuparrow v$.
		\item \label{i-chains-R} There exists a countable subset $\Sigma \subset \Tbb$ such that for every  distinct $u, v\in \Tbb$ with $u \treeordereq v$, there exists $z\in \Sigma$ such that $u\treeordereq z\treeordereq v$.
	\end{enumerate}
	We say that a point $v$ is \emph{below} $u$ (or that $u$ is \emph{above} $v$) if $v\neq u$ and $v\treeordereq u$, and write $v\treeorder u$.
\end{dfn}

Condition \ref{i-chains-R} is a separability assumption that we include in the definition for convenience. It ensures, in particular, that all totally ordered subsets of $\Tbb$ are isomorphic to subsets of the real line. Indeed, Zorn's lemma gives that every totally ordered subset of a poset is contained in a maximal one, and for maximal ones in $(\Tbb,\treeorder)$ we have the following.

\begin{lem}
	Every maximal totally ordered subset $\ell\subset \Tbb$ is order-isomorphic to either $[0, +\infty)$ or $\R$.
	If $\ell_1, \ell_2$ are two distinct maximal totally ordered subsets, then $\ell_1\cap \ell_2$ is  order-isomorphic to $[0, +\infty)$.
\end{lem} 
\begin{proof}
	We have $\{u: v\treeordereq u\} \subset \ell$ for every $v\in \ell$,, because otherwise $ \ell\cup \{u: v\treeordereq u\}$ would be a strictly larger totally ordered set, contradicting maximality of $\ell$. If $\ell$  has a minimum $v$, condition \ref{i-tree} gives $\ell=\{u : v\treeordereq u\}\cong [0, +\infty)$. 
	If not, condition \ref{i-chains-R} implies that every maximal totally ordered subset $\ell\subset \Tbb$ contains a strictly decreasing sequence $(v_n)$ such that $\ell=\bigcup_n\{u: v_n\treeordereq u\}$.  As by \ref{i-tree} every  element in the union is order-isomorphic to $[0, +\infty)$, it follows that $\ell$ is isomorphic to $\R$. 
	If $\ell_1$ and $\ell_2$ are distinct maximal totally ordered subsets, choose  $v_1\in \ell_1$ and $v_2\in \ell_2$ which are  not comparable; then  $\ell_1\cap \ell_2=\{w: v_1\curlywedgeuparrow v_2\treeordereq w\}\cong [0, +\infty)$.
\end{proof}

\subsection{Boundary, focus}
\begin{dfn}[End completion] \label{d-tree-boundary}
	Given a directed tree $(\Tbb,\treeorder)$,
	we define its \emph{end-completion} $(\overline {\Tbb}, \treeorder)$  as the poset obtained by adding points to $\Tbb$ as follows.   First we add a  point $\omega$, called the \emph{focus}, which is the unique maximal point of $(\overline{\Tbb}, \treeorder)$. Next, for each   maximal totally ordered subset $\ell\subset \Tbb$ without minimum, we add a point $\xi\in \overline{\Tbb}$ which satisfies $\xi\treeorder v$ for every $v\in \ell\cup \{\omega\}$. The subset $\partial \Tbb=\overline{\Tbb}\setminus \Tbb$ is called the \emph{boundary} of $\Tbb$. We will also write $\partial^*\Tbb=\partial \Tbb \setminus \{\omega\}$.
\end{dfn} 
Note that any two distinct points $x, y\in \overline{\Tbb}$ still admit a unique smallest upper bound, which we continue to denote by $x\curlywedgeuparrow y$, and we have $x\curlywedgeuparrow y\in \Tbb$ unless $\omega\in \{x,y\}$. 

\subsection{Paths, branching points, directions, shadows} 
Given distinct points $u,v\in \overline{\Tbb}$ we define the \emph{arc} between them as the subset 
\[[u, v]=\left \{z\in\overline{\Tbb}: u\treeordereq z\treeordereq (v\curlywedgeuparrow u) \text{ or } v\treeordereq z \treeordereq (v\curlywedgeuparrow u)\right \}.\]
and set $]u, v[=[u, v]\setminus \{u, v\}$. The subsets $[u, v[$ and $]u, v]$ are defined similarly.
A subset $Y\subseteq \overline{\Tbb}$ is \emph{path-connected} if $[u,v]\subseteq Y$ for every $u, v\in Y$. Every subset of $\overline{\Tbb}$ is a disjoint union of maximal path-connected subsets, called its \emph{path-components}.
Given $v\in {\Tbb}$, we define the set $E_v$ of \emph{directions} at $v$ as  the set of path-components of $\overline{\Tbb}\setminus\{v\}$.  We say that $v$ is a \emph{branching point} if $|E_v|\ge 3$; the set of branching points is denoted as $\Br(\Tbb)$.  We shall say that $v$ is \emph{regular} if $|E_v|=2$,  and a \emph{leaf} if $|E_v|=1$. Leaves are precisely the minimal elements in $(\Tbb, \treeorder)$. For most of our purposes, we will consider only trees without leaves (see Remark \ref{r.focal_noleaves}), and we soon always assume so to avoid technical discussions.

For $z\in \overline{\Tbb}\setminus \{v\}$, we let $e_v(z)\in E_v$ denote the direction containing $z$. If $\omega$ is the focus, we let  $E^-_v=E_v\setminus \{e_v(\omega)\}$ denote the set of \emph{directions below} $v$. 
Condition \ref{i-chains-R} implies that each set $E^{-}_v$ is countable, and that $\Br(\Tbb)$ is also countable (since every branching point can be written as $v=z_1\curlywedgeuparrow z_2$ for some $z_1, z_2\in \Sigma$).  
Finally, we denote by $U_v\subset \Tbb$ the subset  of points below $v$. The corresponding subset of the boundary $\partial U_v\subset \partial^*\Tbb$ is called the \emph{shadow} of $v$. See Figure  \ref{fig.directed_tree} for an illustration of these definitions.

\begin{figure}
	\centering
	\includegraphics[scale=1]{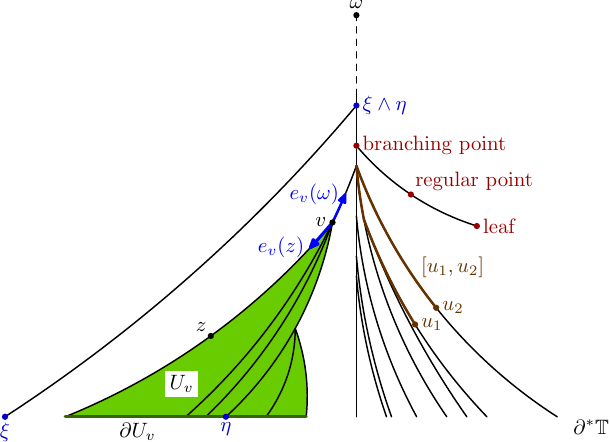}
	\caption{A directed tree $(\Tbb,\treeorder)$ with focus $\omega$, and the corresponding defined objects: arc between two points (brown), directions (blue), smallest upper bound (blue), shadow (green).}\label{fig.directed_tree}
\end{figure}

\subsection{Planar directed trees}

\begin{dfn}
	Let $(\Tbb, \treeorder)$ be a directed tree (without leaves) with focus $\omega\in \partial \Tbb$. 
	A \emph{planar order} on $(\Tbb, \treeorder)$  is a collection $\prec=\{<^v : v\in \Br(\Tbb) \}$   of total orders $<^v$ defined on the set  $E_v^{-}$ of direction below $v$, for each $v\in \Br(\Tbb)$. The triple $(\Tbb, \treeorder,  \prec)$ will be called a \emph{planar directed tree}.
\end{dfn}

A planar order induces a total order on $\partial^\ast \Tbb=\partial \Tbb\setminus\{\omega\}$, that by slight abuse of notation we will still denote by $\prec$, defined as follows. Given two distinct ends $\xi_1,\xi_2\in \partial^* \Tbb $, write $v=\xi_1 \curlywedgeuparrow \xi_2$, and note that $e_v(\xi_1)\neq e_v(\xi_2)$ in $E^-_v$. We set $\xi_1\prec \xi_2$ if  $e_v(\xi_1)<^v e_v(\xi_2)$. It is straightforward to check that this defines an order on $\partial^\ast \Tbb$. 
Moreover, we say that the planar order $\prec$ is \emph{proper} if for every $v\in \Tbb$, the shadow $\partial U_v$ is bounded from above and below in $(\partial^*\Tbb, \prec)$. 

\begin{rem}\label{r.shadow_convex}
	For every planar order $\prec$ on $(\Tbb,\treeorder)$ and $v\in \Tbb$, the shadow $\partial U_v$ is a convex subset of $(\partial^*\Tbb,\prec)$.
\end{rem}

\subsection{Group actions on (planar) directed trees}
Given a directed tree $(\Tbb, \treeorder)$, we denote by $\Aut(\Tbb, \treeorder)$ the group of its order-preserving bijections.  
We will be particularly interested in the following class of actions. 
\begin{dfn} \label{d-focal-action-tree}
	A group action $\Phi\colon G\to  \Aut(\Tbb, \treeorder)$ is \emph{focal} if every orbit is cofinal, namely for every $u, v\in \Tbb$ there exists an element $g\in G$ such that $v \treeorder g.u$.
\end{dfn}
As an equivalent definition, a group action is focal if for every $v\in \Tbb$ there exists a sequence $(g_n)\subset G$ such that $g_n.v$ tends to $\omega$ along the ray $[v, \omega[$.  
\begin{rem}\label{r.focal_noleaves}
	Note that if a directed tree $(\Tbb,\treeorder)$ admits a focal action, then it has no leaves (as the orbit of a leaf cannot be cofinal).
\end{rem}

Given a planar order $\prec$ on $(\Tbb,\treeorder)$, we denote by $\Aut(\Tbb, \treeorder, \prec)$ the subgroup of $\Aut(\Tbb,\treeorder)$ which preserves the planar order $\prec$, meaning that for every element $g\in G$, the corresponding bijections between $E_v^-$ and $E_{g(v)}^-$ induce isomorphisms between $\prec_v$ and $\prec_{g(v)}$. 
Note that the induced action of $\Aut(\Tbb, \treeorder, \prec)$ on $\partial^*\Tbb$ preserves the total order $\prec$.

\begin{rem}\label{rem.planarexistence}
	An action $\Phi\colon G\to\Aut(\Tbb, \treeorder)$ by automorphisms on a directed tree admits an invariant planar order if and only if, for every $v\in \Br(\Tbb)$,  there is an order on $E^-_v$ which is invariant under $\stab_G(v)$. Indeed, it is enough to choose such an order for $v$ in a system of representatives   of $\Phi$-orbits $B\subseteq\Br(\Tbb)$, and then uniquely extend this collection to all $v\in \Br(\Tbb)$, in a  $\Phi$-invariant way.
\end{rem}

We keep saying that an action $\Phi\colon G\to \Aut(\Tbb, \treeorder, \prec)$ is focal if the induced action on $(\Tbb, \treeorder)$ is so.
Given a group action $\Phi\colon G\to \Aut(\Tbb, \treeorder, \prec)$ and $\xi\in \partial^*\Tbb$, we shall denote by $\mathcal{O}_\xi\subset \partial^*\Tbb$ the $\Phi$-orbit of $\xi$. 
\begin{lem}\label{l-focal-implies-proper}
	Let $\Phi\colon G \to \Aut(\Tbb, \treeorder, \prec)$ be a focal action on a planar directed tree with $|\partial^*\Tbb|\ge 2$. Then the following hold.
	\begin{enumerate}[label=(\roman*)]
		\item \label{i-focal-implies-proper} The planar order $\prec$ is  proper.
		\item \label{i-focal-implies-denselyordererd} For every $\xi \in \partial^*\Tbb$, its orbit $\mathcal{O}_\xi$ is cofinal in $(\partial^\ast \Tbb, \prec)$ (that is, unbounded in both directions) and densely ordered (that is, for every $\xi_1, \xi_2 \in \mathcal{O}_\xi$ with $\xi_1\prec \xi_2$, there exists $\eta\in \mathcal{O}_\xi$ with $\xi_1\prec \eta\prec \xi_2$).
	\end{enumerate}
\end{lem}

\begin{proof}
	We first prove \ref{i-focal-implies-proper}. Since $|\partial^*\Tbb|\ge 2$, we can find two points $u, v\in \Tbb$ such that none is below the other, so that the shadows $\partial U_u$ and $ \partial U_v$ are disjoint and $\prec$-convex (Remark \ref{r.shadow_convex}). Therefore at least one of them is bounded below, and at least one of them is bounded above. As the action is focal, for every $z\in \mathbb T$, there exist elements $g,h\in G$ such that $g.\partial U_u=\partial U_{g.u}\supset \partial U_z$ and $h.\partial U_v= \partial U_{h.v}\supset \partial U_z$. It follows that the shadow $\partial U_z$ is bounded above and below, so that $\prec$ is proper. 
	
	We next prove \ref{i-focal-implies-denselyordererd}. First of all, observe that by focality we have $\mathcal{O}_\xi\cap \partial U_v\neq \varnothing$ for every $v\in \Tbb$. Indeed, it is enough to choose $u$ above $\xi$, and $g\in G$ such that $u\treeorder g.v$, so that $\xi\in \partial U_{g.v}$, or equivalently $g^{-1}.\xi \in \partial U_v$. This immediately gives that the orbit $\mathcal{O}_\xi$ is  cofinal in $(\partial^*\Tbb,\prec)$. Let us show that it is densely ordered. Assume by contradiction that there are two points $\xi_1\prec\xi_2$ in $\mathcal O_\xi$ with no elements of $\mathcal{O}_\xi$ between them, and let $g\in G$ be such that $\xi_2=g.\xi_1$. Since the action is order preserving, applying $g^{-1}$ we deduce that the point $\xi_0=g^{-1}.\xi_1$ satisfies $\xi_0\prec \xi_1\prec \xi_2$, and there is no $\eta\in \mathcal{O}_\xi$ satisfying $\xi_0\prec \eta \prec \xi_1$ nor $\xi_1\prec \eta \prec \xi_2$. Now choose $v\in \Tbb$ such that $\xi_1\in \partial U_v$ and $\xi_0, \xi_2\notin \partial U_v$. Since $\partial U_v$ is a $\prec$-convex subset, we deduce that $\partial U_v\cap \mathcal{O}_\xi=\{\xi_1\}$. But clearly we can find $u\in \Tbb$ such that $|\partial U_u\cap \mathcal{O}_\xi|\ge 2$ (just take $u=\xi_1\curlywedgeuparrow \xi_2$). This gives a contradiction since, by focality, there exists $h\in G$ so that $h.\partial U_u\subset \partial U_v$, and therefore $\{h.\xi_1,h.\xi_2\}\subset \partial U_v$. See Figure \ref{fig.focal_proper}.
\end{proof}

\begin{figure}[ht]
	\centering
	\includegraphics[scale=1]{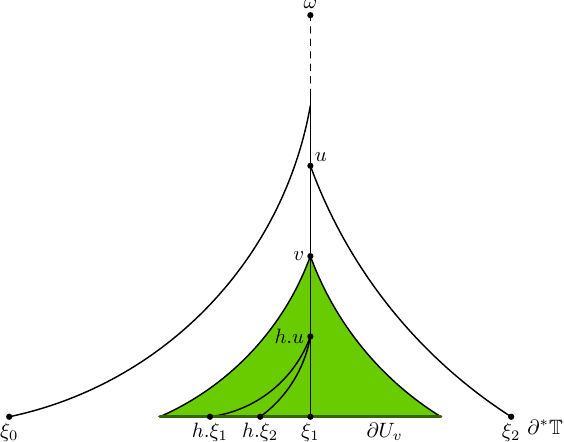}
	\caption{Proof of \ref{i-focal-implies-denselyordererd} in Lemma \ref{l-focal-implies-proper}.}\label{fig.focal_proper}
\end{figure}

\subsection{Horogradings of trees}

\begin{dfn} \label{def horograding}
	Let $(\Tbb, \treeorder)$ be a directed tree. A positive (respectively, negative) \emph{horograding} of $\Tbb$  is an order-preserving (respectively, order-reversing) map  $\hor\colon (\Tbb,\treeorder)\to(\R,<)$ such that for every $u\in \Tbb$, the restriction of $\hor$ to the arc $[u, \omega[\subset \Tbb$ is an order-preserving isomorphism with the interval $[\hor(u), +\infty)$ (respectively, an order-reversing isomorphism with $(-\infty,\hor(u)]$).
\end{dfn}
In other words a horograding is the choice of a way to parametrize vertical totally ordered segments of $\Tbb$. Depending on the circumstances, we may want to consider horogradings $\hor\colon \Tbb\to X$, where $X=(a,b)$ is a non-empty open interval.

If $\hor$ is a positive horograding, the restriction of $\hor$ to a maximal totally ordered subset $\ell\subset \Tbb$ is not always a bijection onto $\R$; namely, it is not necessarily true that $\inf_{v\in \ell} \hor(v)=-\infty$. In general this infimum might be either a finite value $a\in \R$ or $-\infty$,  so that $\hor(\ell)$ is  an open interval of the form $(a, +\infty)$, with $a\in \R\cup\{-\infty\}$. 

\begin{dfn}\label{dfn.pi_complete}
	Assume that $\hor\colon \mathbb{T}\to \R$ is a positive horograding of a directed tree $(\Tbb,\treeorder)$. We define the \emph{$\hor$-complete boundary} as the subset
	\[\partial^\ast_\hor\mathbb{T}:=\left \{\xi\in\partial^\ast\mathbb{T}: \textstyle \inf_{\xi\treeorder v}{\hor}(v)=-\infty\right \}.\] 
	The analogous definition can be given for a negative horograding.
\end{dfn}

The following observation will not be needed, but we include it for completeness. \begin{prop}\label{prop existe horograding}
	Every directed tree admits a horograding.
\end{prop}
\begin{proof}[Sketch of proof]
	We rely on \ref{i-chains-R}. Enumerate the points of $\Sigma$ as $(v_n)$, and proceed inductively by defining $\hor$ on each ray $[v_n, \omega[$. As first step, define $\hor$ on $[v_0, \omega[$ as an arbitrary order-preserving bijection onto an interval $[\hor(v_0), +\infty)$. At step $n$, $\hor$ has already been defined on the subray $[w_n, \omega[\subset [v_n, \omega[$, where $w_n=\min_{\treeorder}\{v_n\curlywedgeuparrow v_j: j< n\}$, and can be extended arbitrarily to $[v_n, w_n]$. This will define $\hor$ on all points of $\Tbb$, except perhaps on leaves; if $\Tbb$ has leaves, take care in the inductive process to keep the image of $\hor$ bounded below, and then extend $\hor$ to leaves by continuity.
\end{proof}

Finally we define a horograding of actions in this context.

\begin{dfn}
	Let $\Phi\colon G\to \Aut(\Tbb, \treeorder)$ be a focal action, and $\rho\colon G\to \homeo_0(\R)$ be an action. A (positive or negative) \emph{horograding} of $\Phi$ by $\rho$ is a (positive or negative) $G$-equivariant horograding $\hor\colon \Tbb \to \R$. When such a horograding exists, we say that $\Phi$ can be horograded by $\rho$. 
\end{dfn}

This terminology extends naturally to the case where $\rho$ is an action on an open interval $X=(a,b)$.
Also, a horograding of an action $\Phi \colon G\to \Aut(\Tbb, \treeorder, \prec)$ on a planar directed tree is a horograding of the action on the underlying directed tree $(\Tbb, \treeorder)$. 

\section{Planar directed trees and laminations}\label{s-planar-trees}

We now clarify the connection between actions on directed trees and actions on $\R$. This is done in Propositions \ref{prop.focalisminimal} and \ref{p-from-focal-to-trees} below, which allow to make a transition from one to the other. We assume here that $G$ is a {countable} group (this assumption could be avoided, but it substantially simplifies  the discussion). Let  $\Phi\colon G\to \Aut(\Tbb, \treeorder, \prec)$ be an action on a planar directed tree. For every end $\xi\in \partial^*\Tbb$, the $\Phi$-orbit $\mathcal{O}_\xi\subset \partial^*\Tbb$ is a countable totally ordered set on which $G$ acts by order-preserving bijections.  We can therefore consider the dynamical realization of the $\Phi$-action on $(\mathcal{O}_\xi, \prec)$ (as explained in \S \ref{sec.dynreal}), thus obtaining an action $\varphi_\xi\colon G\to \homeo_0(\R)$. We have the following. 

\begin{prop}[From planar directed trees to actions on the line] \label{prop.focalisminimal}
	Let $\Phi\colon G \to \Aut(\Tbb, \treeorder, \prec)$ be a focal action of a countable group $G$ on a planar directed tree $(\Tbb, \treeorder, \prec)$  with $|\partial^*\Tbb| \ge 2$. For $\xi\in \partial^*\Tbb$,  let $\varphi_\xi\colon G\to \homeo_0(\R)$ be the dynamical realization of the $\Phi$-action on the orbit of $\xi$. Then $\varphi_\xi$ is a minimal laminar action, and it does not depend on the choice of $\xi$, up to positive conjugacy. Furthermore, if $\Phi$ admits a horograding by an action $\rho\colon G\to \homeo_0(\R)$, then so does $\varphi_\xi$.
\end{prop}

\begin{proof}
	We denote by $\iota_\xi \colon \mathcal{O}_\xi\to \R$ the good embedding associated with the dynamical realization $\varphi_\xi$.  Let us first show that $\varphi_\xi$ is minimal. Recall from Lemma \ref{lem.minimalitycriteria} that it is sufficient to show that  for any points $\zeta_1\prec \xi_1\prec \xi_2 \prec \zeta_2$ of $\mathcal{O}_\xi$, there exists $g$ such that $\xi_1\prec g.\zeta_1 \prec g.\zeta_2\prec \xi_2$. By Lemma \ref{l-focal-implies-proper},  we can choose $\eta\in \mathcal O_\xi$ such  that $\xi_1\prec \eta \prec \xi_2$. Let $v,w\in \Tbb$ be such that $\partial U_v$ contains $\{\zeta_1, \xi_1, \eta, \xi_2, \zeta_2\}$, and $\partial U_w$ separates $\eta$ from $\{\zeta_1, \xi_1, \xi_2, \zeta_2\}$ (one can take for instance $v=\zeta_1\curlywedgeuparrow\zeta_2$ and $w\in ]\eta,\xi_1\curlywedgeuparrow\xi_2[$). By focality,  there exists $g\in G$ such that $g.\partial U_v\subset \partial U_w$. Since $\partial U_w$ is convex, this implies that $\xi_1\prec g.\zeta_1 \prec g.\zeta_2\prec \xi_2$ as desired. 
	
	To prove that the action $\varphi_\xi$ is laminar, we observe  that the collection $\mathcal S=\{\partial U_v\cap \mathcal{O}_\xi : v\in \Tbb\}$ is a $\Phi$-invariant prelamination of $(\mathcal{O}_\xi, \prec)$.
	Indeed, Lemma \ref{l-focal-implies-proper} gives that any element of $\mathcal S$ is a non-empty  $\prec$-convex bounded subset, and from Remark \ref{sc:inv_prelam_ordered_sets} we get a $\varphi_\xi$-invariant prelamination by considering interiors of convex hulls of $\iota_\xi$-images of elements in $\mathcal S$. By minimality of $\varphi_\xi$, Proposition \ref{prop.minimalimpliesfocal} gives that $\varphi_\xi$ is focal. Moreover, if $\hor$ is a horograding of $\Phi$ by an action $\rho:G\to \homeo_0(\R)$, it defines a prehorograding $(\mathcal S,\hor_\xi)$ by $\rho$ by the expression $\hor_\xi(\partial U_v\cap \mathcal{O}_\xi)=\hor(v)$, and consequently by Remarks \ref{rem.prehor_dyn_real} and \ref{r-extending-horograding}, we get a horograding of $\varphi_\xi$ by $\rho$.
	
	Finally, let us check that if $\xi_1, \xi_2\in \partial^\ast \Tbb$, then the actions $\varphi_{\xi_1}$ and $\varphi_{\xi_2}$ are positively conjugate. For clarity, let us denote by $\R_i$ the line on which $\varphi_{\xi_i}$ is defined. 
	First, define a map $h\colon \iota_{\xi_1}(\mathcal{O}_{\xi_1})\to \R_2$, by setting $h(\iota_{\xi_1}(\xi))=\sup\{\iota_{\xi_2}(\eta): \eta\in \mathcal{O}_{\xi_2}, \eta\preceq \xi\}$, for $\xi\in \mathcal{O}_{\xi_1}$. Then, $h$ is non-decreasing,  $G$-equivariant,  so by Lemma \ref{lem.semiconjugacy} it extends to semi-conjugacy $h\colon \R_1\to \R_2$.
	Since both actions $\varphi_{\xi_i}$ are minimal, the map $h$ must be a conjugacy.
\end{proof}

\begin{dfn}
	After Proposition \ref{prop.focalisminimal}, we say that the action $\varphi_\xi\colon G\to \homeo_{0}(\R)$ is the \emph{dynamical realization} of the focal action $\Phi\colon G \to \Aut(\Tbb, \treeorder, \prec)$. Given an  action $\varphi\colon G\to \homeo_0(\R)$, we will say that $\varphi$ is \emph{represented} by a focal action $\Phi\colon G\to \Aut(\Tbb, \treeorder, \prec)$ if $\varphi$ is positively conjugate to the dynamical realization of $\Phi$.
\end{dfn}

Conversely, we have the following.

\begin{prop}[From minimal laminar actions to planar directed trees] \label{p-from-focal-to-trees}
	Let $\varphi\colon G\to \homeo_0(\R)$ be a minimal laminar action of a countable group. Then there exists a focal action $\Phi \colon G\to \Aut(\Tbb, \treeorder, \prec)$ on a planar directed tree, representing  $\varphi$. If moreover $\varphi$ can be horograded by an action $\rho$, then one can choose $\Phi$ which can be horograded by $\rho$.
\end{prop}

\begin{proof}
	It is enough to prove the proposition in the horograded case, since every focal laminar action can be horograded by itself (Example \ref{e-trivial-horograding}). Let therefore $(\mathcal{L}, \hor)$ be a horograding of $\varphi$ by an action $\rho\colon G\to \homeo_0(\R)$.  Then for every  $I, J\in \mathcal{L}$, there is a smallest leaf containing $I\cup J$, that we denote by $I\curlywedgeuparrow J$. The proof proceeds as follows:  for every $I\in \mathcal{L}$, take a copy of the ray $[\hor(I), +\infty)$, and consider the quotient space $\Tbb$ obtained by gluing pairwise the rays corresponding to $I, J$  along the subray $[\hor(I\curlywedgeuparrow J), +\infty)$. Then $\Tbb$ is naturally a planar directed tree, and inherits a group action and a  horograding. Formalizing this idea and performing all the necessary verification is routine; we include below some further detail for the interested reader.
	
	First note the map $\hor$ need not be (strictly) increasing along totally ordered subsets of $\mathcal{L}$: this is not really a problem, but it is a source of more case-by-case analysis. We avoid this possibility with a trick. First, replace $\hor$ with the map $\hor_1$, defined by $\hor_1(I)=\sup_{J\subsetneq I} \hor_1(J)$ if $I$ is non-trivially accumulated by leaves $J\subsetneq I$, and $\hor_1(I)=I$ otherwise. The map $\hor_1$ is still non-decreasing and equivariant, moreover it is now continuous from below. Next, consider the subset
	\[\mathcal{L}_1=\{I\in \mathcal{L}: \hor_1(J)>\hor_1(I)\text{ for every }J\supsetneq I\}.\]
	Then $\mathcal{L}_1\neq \varnothing$, as for every $I_0\in \mathcal L$ the subset $\{I\in\mathcal L: I\supseteq I_0,\hor_1(I)=\hor_1(I_0)\}$ has a maximum, by continuity from below of $\hor_1$, which belongs to $\mathcal{L}_1$. Further,  $\mathcal{L}_1$ is easily seen to be $\varphi$-invariant and, although not necessarily closed, it is still closed from above, namely the limit of a decreasing sequence in $\mathcal{L}_1$ belongs to $\mathcal{L}_1$. As a consequence, every $I, J\in \mathcal{L}_1$ still  have a smallest upper bound in $\mathcal{L}_1$, which we still denote by $I\curlywedgeuparrow J$. Note also that the restriction of $\hor_1$ to $\mathcal{L}_1$ is increasing. We shall work with the prehorograding $(\mathcal{L}_1, \hor_1)$, that we rename $(\mathcal{L}, \hor)$ from now on.
	
	Let $\Tbb$ be the quotient of $\{(I, t)\in \mathcal{L}\times \R : t\ge \hor(I)\}$ by the equivalence relation that identifies $(I_1, t_1)$ and $(I_2, t_2)$ if $t_1=t_2\ge \hor(I\curlywedgeuparrow J)$ (transitivity of this relation follows easily from the cross-free property of $\mathcal{L}$). We denote by $[I, t]$ the equivalence class of $(I, t)$. The order $\treeorder$ on $\Tbb$ is given by $v_1\treeordereq v_2$ if for some (equivalently, every)   choice of representatives $v_i=[I_i, t_i]$, we have $t_1\leq t_2$ and $t_2\leq \hor(I_1\curlywedgeuparrow I_2)$. It is not difficult to verify, using again the cross-free property of $\mathcal{L}$, that $(\Tbb, \treeorder)$ satisfies \ref{i-tree}, and    that \ref{i-directed} is satisfied with 
	\[[I_1, t_1]\curlywedgeuparrow [I_2, t_2]=[I_1\curlywedgeuparrow I_2, \max\{t_1, t_2, \hor(I_1\curlywedgeuparrow I_2)\}].\]
	Note that the map $(\mathcal{L}, \subseteq)\to (\Tbb, \treeordereq), I\mapsto [I, \hor(I)]$ is order preserving and injective (its injectivity uses that $h$ is increasing).  Since $\mathcal{L}$ is separable (for its natural topology), the image $\Sigma$ of any countable dense subset satisfies \ref{i-chains-R}. The action $\Phi\colon G \to \Aut(\Tbb, \treeorder)$ is given by $g.[I, t]=[\varphi(g).I, \rho(g).t]$, and the map $ [I, t]\mapsto t$ is a $G$-equivariant horograding, that we still denote by $\hor\colon \Tbb \to \R$. Note that since every $\varphi$-orbit in $\mathcal{L}$ is cofinal (by Proposition \ref{prop.minimalimpliesfocal}), and since every point in $\Tbb$  has  lower and upper bounds in the image of $\mathcal{L}$, it follows that $\Phi$ is focal.
	To define a planar order on $(\Tbb, \treeorder)$, choose $v\in \Br(\Tbb)$. Let $e_1, e_2\in E_v^{-}$ be distinct directions, and choose points of the form $w_i=[J_i, \hor(J_i)]\in e_i$. Then the $J_i$ must be disjoint (since the points $w_i$ are necessarily not comparable, as they belong to distinct directions $e_i$). We declare that $e_1<^{(v)}e_2$ if $\sup J_1<\inf J_2$.  For a different choice $w_i'=[J_i', \hor(J_i')]\in e_i$, we have $w_i\curlywedgeuparrow w_i'=[J_1\curlywedgeuparrow J_1', \hor(J_1\curlywedgeuparrow J_i')]$,  and again $J_1\curlywedgeuparrow J_1'$ and $J_2\curlywedgeuparrow J_2'$ are disjoint. This implies that $\sup J_1<\inf J_2$ if and only if $\sup J_1'<\inf J_2'$, and hence $<^{(v)}$ is a well-defined total order. The planar order $\prec$ defined by the orders $<^{(v)}$ is clearly $\Phi$-invariant. 
	
	Now, let $(I_n)\subset \mathcal{L}$ be a  decreasing sequence such that $\bigcap_n \overline{I_n}=\{x\}$ for some $x\in \R$ (such a sequence exists by Proposition \ref{prop.minimalimpliesfocal}). Then $[I_n, \hor(I_n)]$ is a decreasing sequence in $\Tbb$, and we claim that it converges to a point $\xi\in \partial^{\ast} \Tbb$. Suppose by contradiction that $[J, t]$ is a lower bound of $[I_n, \hor(I_n)]$ for every $n$. By definition of $\treeorder$, this implies that $\hor(I_n)\ge \hor(I_n\curlywedgeuparrow J)$, and thus it is equal, since $I_n\subset I_n\curlywedgeuparrow J$. But since $\hor$ is increasing, this is possible only if $I_n= I_n\curlywedgeuparrow J\supseteq J$ for every $n$, contradicting that $\overline{I_n}$ is decreasing to a point. Now,  from the construction of the planar order $\prec$, it follows that the $G$-orbits of $\xi$ and $x$, with the respective induced orders, can be $G$-equivariantly identified. Since $\varphi$ is minimal, it follows that it is conjugate to the dynamical realization of $\Phi\colon G\to \Aut(\Tbb, \treeorder)$.
\end{proof}

\begin{ex}[Planar directed trees for Plante-like actions]\label{ex.tree_Plante}
	Let us illustrate how to get a focal action on a planar directed tree in our running example of a Plante-like action of a permutational wreath product. As in Example \ref{subsec.Plantefocal}, we consider two countable groups $G$ and $H$, with an action of $G$ on a countable set $X$, and we fix a left-invariant order $ <_H\in\LO(H)$, and a $G$-invariant order $<_X$ on $X$ such that the $G$-action on $(X,<_X)$ is cofinal. We also take a good embedding $\iota\colon X\to \R$. For every $x\in \R$, consider the collection of functions
	\[
	\mathsf S_x=\left\{s\colon  (x,+\infty)\to H: s(y)=1_H\text{ for all $y>x$ but finitely many exceptions in $\iota(X)$}\right\}.
	\]
	As a set, we define $\Tbb$ as the disjoint union $\bigsqcup_{x\in \R}\mathsf S_x$. We declare that $s\treeordereq t$ if $s\in \mathsf S_x$, $t\in \mathsf S_y$ for some $x\le y$, and $s\restriction_{(y,+\infty)}=t$. With this definition, it is immediate to see that for every $s\in \mathsf S_x$, the subset $\{u : s\treeordereq u\}=\{s\restriction_{(y,+\infty)}:y\ge x\}$ is order-isomorphic to $[x,+\infty)$, giving \ref{i-tree}. To verify \ref{i-directed}, given $s\in \mathsf S_x$ and $t\in \mathsf S_y$, the restriction $s\restriction_{(x_*,+\infty)}$, where $x_*=\max\{z:s(z)\neq t(z)\}$, gives the desired smallest common upper bound. Finally, we have that the collection $\Sigma=\bigsqcup_{x\in \iota(X)}\mathsf S_x$ is countable, as $X$ and $H$ are countable, and this gives \ref{i-chains-R}. Note that $\Sigma$ coincides with the collection of branching points $\Br(\Tbb)$. For $x\in \iota(X)$ and $s\in \mathsf S_x\subset\Br(\Tbb)$, we note that $E^-_s$ is in one-to-one correspondence with $H$, the identification being given by the value of $s$ at $x$; this allows to put on $E^-_s$ the total order coming from $<_H$, defining a planar order $\prec$ on $(\Tbb,\treeorder)$.	
	A horograding $\hor\colon \Tbb\to \R$ is simply given by $\hor(s)=x$, where $x\in \R$ is such that $s\in \mathsf S_x$. 
	
	In order to describe the boundaries $\partial^*\Tbb$ and $\partial^\ast_\hor\Tbb$, set \[\mathsf{S}=\{s\colon  (x,+\infty)\to H: x\in \R\cup\{-\infty\} \text{ and } s\upharpoonright_{(y,+\infty)}\in\mathsf{S}_y\text{ for every }y>x\},\] and notice that the partial order $\treeorder$ naturally extends to $\mathsf{S}$. Denote by $\mathsf{S}^\ast\subseteq\mathsf{S}$ the subset of the functions which are minimal for the relation $\treeorder$. That is, an element $s\in\mathsf{S}$ is in $\mathsf{S}^\ast$ if and only if, either $s$ is defined over $\R$ or the support of $s$ is infinite. There is a correspondence between $\mathsf{S}^\ast$ and $\partial^\ast\mathbb{T}$ so that the function $s\colon (x,+\infty)\to H$ corresponds to the equivalence class of the ray $y\mapsto s\upharpoonright_{(y,+\infty)}$. Under this correspondence, the $\hor$-complete boundary $\partial^\ast_\hor\mathbb{T}$ is identified with those functions in $\mathsf{S}^\ast$ whose domain is $\R$. In particular, we can see	
	$\bigoplus_XH$ as a subset of $\partial^*_\hor\Tbb$ (a function $\s\colon X\to H$ can be seen as a function $\s\colon \R\to H$ with support in $\iota(X)$). Notice that the order induced on $\bigoplus_XH$ by $\prec$ under this identification coincides with that introduced in Example \ref{subsec.Plantefocal}.
	
	The permutational wreath product $H\wr_X G$ acts on $(\Tbb,\treeorder,\prec)$. Indeed, every $g\in G$ sends $\mathsf S_x$ to $\mathsf S_{g.x}$ by pre-composition with $g^{-1}$, where the action $j\colon G\to \homeo_0(\R)$ is the dynamical realization of the $G$-action on $(X,<_X)$ corresponding to the good embedding $\iota$. If $\s\in \bigoplus_XH$, then we can see it as a function defined on $\R$ as before, and thus $\s$ acts on $\Tbb$ by pointwise multiplication of functions. It is direct to check that this action preserves the poset structure $\treeorder$ and the planar order $\prec$.
	
	Note that this action naturally extends to $\mathsf{S}^\ast$, giving the action of $H\wr_X G$ on $\partial^\ast\Tbb$ through the correspondence between $\mathsf{S}^\ast$ and $\partial^\ast\mathbb{T}$. Moreover, the $\hor$-complete boundary $\partial^\ast_\hor\Tbb$ and (the copy of) $\bigoplus_XH$ are invariant subsets for this action. Indeed, the restriction to $\bigoplus_XH$ coincides with the action $\Psi$ (discussed in Example \ref{subsec.Plantefocal}). In other terms, we constructed an order-preserving and equivariant ``embedding'' of the Plante-like product $\Psi\colon H\wr_X G\to\Aut\left (\bigoplus_XH,\prec\right)$ defined in Example \ref{subsec.Plantefocal}, into the restriction to the boundary of an action on a planar directed tree. 	
	
	Finally, note that the horograding $\hor$ satisfies $\hor(g.s)=g.\hor(s)$ for every $g\in G$ and $s\in \Tbb$, while it is constant on $\bigoplus_XH$-orbits.
	See Figure \ref{fig.Plante-2}.
	
	\begin{figure}[ht]
		\centering
		\includegraphics[scale=1]{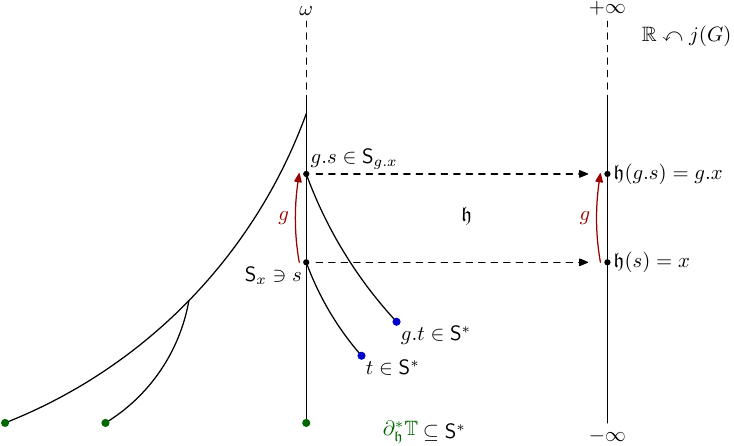}
		\caption{The directed tree for the Plante-like action (Example \ref{ex.tree_Plante}).}\label{fig.Plante-2}
	\end{figure}
\end{ex}

\section{Simplicial trees and metric $\R$-trees}\label{ssc.simplicial_actions}

In this section we clarify how all the notions discussed above relate to the more familiar setting of isometric group actions on simplicial or  real trees. 

\subsection{Horogradings and metrics on $\R$-trees} \label{subsec metric R-tree} A (metric) \emph{real tree}, or \emph{$\R$-tree}, is a metric space $(\Tbb, d)$ such that every two points $u, v\in \Tbb$ can be connected by a unique injective continuous arc $\gamma\colon [0, 1]\to \Tbb$, and this arc can be chosen to be geodesic (i.e.\ an isometric embedding). This notion is well studied; see for instance the survey of Bestvina \cite{Bestvina}. Let us clarify the connection between $\R$-trees and directed trees, in the sense of Definition \ref{def d-trees}.

Assume that $(\Tbb, d)$ is a separable metric $\R$-tree.  Let $\partial_\infty {\Tbb}$ be the \emph{Gromov} (or \emph{visual}) \emph{boundary} of $\Tbb$, namely  $\partial_\infty\Tbb$ is the set of equivalence classes of geodesic rays in $\Tbb$, where a geodesic ray is a subset $r\subset \Tbb$ isometric to $[0, +\infty)$, and two rays $r_1, r_2$ are equivalent if $r_1\cap r_2$ is a geodesic ray. Assume that $\partial_\infty \Tbb$ is non-empty, and choose $\omega\in \partial_\infty \Tbb$. The choice of $\omega$ naturally defines a partial order $\treeorder$ on $\Tbb$, by the condition that $v\treeorder u$ when $u\in [v, \omega[$, where $[v, \omega[$ is the ray from $v$ to $\omega$. 
Moreover, the choice of a point   $z_0\in \Tbb$ allows to  define a \emph{horofunction} $\hor_{\omega, z_0}\colon \Tbb \to \R$ centered at $\omega$ (normalized to vanish on $z_0$), by the formula
\begin{equation} \hor_{\omega, z_0}(v)= d(z_0, v\curlywedgeuparrow z_0)-d(v, v\curlywedgeuparrow z_0). \label{e-tree-metric}\end{equation}
This function  is a horograding in the sense of Definition \ref{def horograding}, and commonly the level sets $\{v\in \Tbb :  \hor_{\omega,z_0}(v)=t\}$ are called \emph{horospheres}, see for instance the book of Bridson and Haefliger \cite[Chapter II.8]{BridsonHaefliger}. 
Conversely,  a directed tree $(\Tbb, \treeorder)$ as in Definition \ref{def d-trees} can always be endowed with a metric that makes it a separable $\R$-tree. Indeed, choose a horograding $\hor\colon \Tbb\to \R$ (see Proposition \ref{prop existe horograding}). We can then define a distance on $\Tbb$ by the formula
\begin{equation}\label{eq.d_pi}d_\hor(u, v)=|\hor(u\curlywedgeuparrow v)-\hor(u)|+|\hor(u\curlywedgeuparrow v)-\hor(v)|.\end{equation}
This distance turns $\Tbb$ into an $\R$-tree; see Favre and Jonsson \cite[Proposition 3.10]{Fav-Jon}. We call a distance of this form a \emph{compatible $\R$-tree metric} on $(\Tbb, \treeorder)$, associated with the horograding $\hor$. The focus $\omega\in \Tbb$ defines a point in the visual boundary $\partial_\infty \Tbb$, and $\hor$ is a horofunction with respect to $d_\hor$. Thus, these two constructions are inverse to each other.

(Note however that the visual boundary $\partial_\infty \Tbb$ does not correspond to the whole boundary $\partial\Tbb$ in the sense of Definition \ref{d-tree-boundary}, but rather to  the subset $\{\omega\}\cup \partial_\hor^\ast \Tbb$, where  $\partial_\hor^\ast \Tbb$ is the set of $\hor$-complete boundary points (Definition \ref{dfn.pi_complete}). To recover the whole of $\partial \Tbb$,  consider the metric completion $\hat{\Tbb}$ of $\Tbb$. Then $ \partial \Tbb$ identifies with  $\partial_\infty \hat{\Tbb} \cup (\hat{\Tbb}\setminus \Tbb$).)

Suppose now that $\Phi\colon G\to \Aut(\Tbb, \treeorder)$ is an action horograded by an action $\rho\colon G\to \homeo_0(\R)$, with associated horograding $\hor\colon \Tbb \to \R$, and let $d_\hor$ be the associated $\R$-tree metric as in \eqref{e-tree-metric}. Equivariance of $\hor$ simply means that $\Phi$ must preserve the partition of $\Tbb$ into horospheres associated with the focus $\omega$. It is evident from \eqref{e-tree-metric} that the action on $\Tbb$ preserves the metric if and only if the horograding action $\rho$ takes values in the group of translations $(\R, +)$.   More explicitly, if the action is isometric, then  every $g\in G$ must eventually coincide with a translation along any ray $[v, \omega)$, and $\rho(g)$ is equal to its signed translation length (with the convention that $\rho(g)>0$ if $g$ pushes points towards $\omega$). 
This, together with Propositions \ref{prop.focalisminimal} and \ref{p-from-focal-to-trees}, implies the following. 

\begin{prop}[Horogradings by translations and actions on $\R$-trees] \label{p-isometric-horograding} A minimal laminar action $\varphi\colon G\to \homeo_0(\R)$ can be horograded by an action by translations $\rho\colon G\to (\R, +)$, if and only if it is represented by a focal action $\Phi\colon G\to \Aut(\Tbb, \treeorder, \prec)$ preserving a compatible $\R$-tree metric on $\Tbb$.
\end{prop}

\begin{rem}\label{rem.homotheties_isometric}
	Observe that when a laminar action can be horograded by an action by translations, Proposition \ref{p-dyn-class-elements-horograded} gives that any pseudo-homothety is is actually a homothety. This fact can also be recovered from Proposition \ref{p-isometric-horograding}, using the classification of isometries of $\R$-trees.
\end{rem}

\subsection{Case of simplicial trees} \label{ssc.simplicial}
The simplest type of real trees are {simplicial} trees. A simplicial tree is a connected graph $\Tbb$ without cycles. We identify $\Tbb$ with its geometric realization, so that each edge is a subset of $\Tbb$ isometric to $[0, 1]$.  We will say that a directed tree $(\Tbb, \treeorder)$ is \emph{simplicial} if $\Tbb$ is a simplicial tree of countable degree (that is, the degree at every point is at most countable), and $\treeorder$ is the partial order associated with the choice of an end $\omega\in \partial \Tbb$, where $u\treeorder v$ if $v$ lies on the ray from $u$ to $\omega$.  Note that in this case the set of directions $E_v^-$ below a vertex $v$ can be identified with the set of edges at $v$ which are in the opposite side of $v$ with respect to the focus. Thus, a planar order corresponds to the choice of a linear order of such edges, for every vertex $v$.

As it turns out, there is a particularly explicit characterization of the  laminar actions that can be represented by an action on a {simplicial} planar directed tree.  This characterization is given by item \ref{i-coZ} in the following result (which can be seen as a partial converse to Proposition \ref{prop.maximal}). We will elaborate on this result in \S \ref{s-F-simplicial}, when studying examples of laminar actions of Thompson's group $F$. 

\begin{prop}[Simplicial laminar actions] \label{p-focal-simplicial}
	Let $\varphi\colon G\to \homeo_0(\R)$ be a minimal action whose image is not isomorphic to $\Z^2$. Then the following are equivalent. 
	\begin{enumerate}[label=(\roman*)]
		
		\item \label{i-discrete-lamination} The action $\varphi$ preserves a {discrete} lamination. 
		
		\item \label{i-focal-simplicial} The action $\varphi$ is represented by a focal action on a planar directed   tree $(\Tbb, \treeorder, \prec)$, such that $\Tbb$ is a simplicial tree of countable degree, and the action of $G$ on $\Tbb$ is by simplicial automorphisms. 
		
		\item \label{i-horograded-cyclic} The action $\varphi$ is laminar, horograded by a cyclic action $\rho\colon G\to \Z$.
		
		\item \label{i-coZ} There exists a non-trivial normal subgroup $N\unlhd G$ such that $G/N\cong \Z$, and $\varphi(N)$ admits no minimal invariant set.
	\end{enumerate}
\end{prop}
\begin{dfn}\label{dfn.simplicial_laminar}
	We shall say that a laminar action $\varphi\colon G\to \homeo_0(\R)$ is \emph{simplicial} if it satisfies one of the equivalent conditions in Proposition \ref{p-focal-simplicial}.
\end{dfn}
\begin{proof}[Proof of Proposition \ref{p-focal-simplicial}]
	Too see that \ref{i-discrete-lamination}$\Rightarrow$\ref{i-focal-simplicial}, note that a discrete lamination $\mathcal{L}$ is naturally the vertex set of a simplicial tree. Namely, for every $I\in \mathcal{L}$, the set  $\{J\in \mathcal{L} : I\subsetneq J\}$ does not accumulate on $I$, and thus has a smallest element $\hat{I}$.  The tree $\Tbb$ is obtained by connecting each $I$ to $\hat{I}$ by an edge. It can be naturally endowed with a $G$-invariant order $\treeorder$ and a planar order  $\prec=\{\prec^J : J\in \mathcal{L}\}$ on $\Tbb$, where two edges $(I_1, J)$ and $(I_2, J)$ with $I_1,I_2\subsetneq J$ are ordered according to the order in which $I_1$ and $I_2$ appear in $J$. By construction, $\varphi$ is conjugate to the dynamical realization of the action on $(\Tbb, \treeorder, \prec)$.
	
	The implication \ref{i-focal-simplicial}$\Rightarrow$\ref{i-horograded-cyclic} is a particular case of the discussion before Proposition \ref{p-isometric-horograding}. Indeed, the action of $G$ on $(\Tbb, \treeorder)$ is horograded by the homomorphism $\rho\colon G\to (\R, +)$ given by the signed translation length in the direction of $\omega$, which has integer values when $\Tbb$ is simplicial.
	
	To see that  \ref{i-horograded-cyclic}$\Rightarrow$\ref{i-coZ}, assume that $(\mathcal{L}, \hor)$ is a horograding of $\varphi$ by a cyclic action $\rho\colon G\to \Z$, and let $N=\ker\rho$.  Then $\fixphi(N)$ is a closed $\varphi$-invariant subset and thus, by minimality,  $\fixphi(N)=\varnothing$ (as otherwise the $\varphi$-action of $N$ would be trivial,  which is impossible since $G/N\cong \Z$ cannot act minimally on $\R$).  Also, no $I\in \mathcal{L}$ can have a cofinal orbit under $\varphi(N)$, since $\rho(N)$ is trivial. It follows from Proposition \ref{p-dyn-class-subgroups} that $\varphi(N)$ has no minimal invariant set. 
	
	We now show that  \ref{i-coZ}$\Rightarrow$\ref{i-discrete-lamination}. Assume that \ref{i-coZ} holds, and choose $f\in G$ which projects to a generator of $G/N\cong \Z$, so that $G=N\rtimes \langle f \rangle$. By Proposition \ref{prop.maximal}, the set $\mathcal{W}_\varphi(N)$ of irreducible wandering intervals for $\varphi|_N$ is a $\varphi$-invariant covering prelamination. Choose $J\in \mathcal{W}_\varphi(N)$ such that $f.J\cap J\neq \varnothing$ (e.g.\ containing $\xi$ and $f.\xi$ for some arbitrary $\xi\in \R$.) First note that it is not possible that $f.J=J$. Indeed, this would imply that for every $g\in G$, writing $g=hf^n$ with $h\in N$, we would have $g.J=h.J$ for some $h\in N$, and thus $g.J=J$ or $g.J\cap J=\varnothing$. But this contradicts the fact that the $\varphi$-orbit of $J$ is cofinal in $\mathcal{L}$ (Proposition \ref{prop.minimalimpliesfocal}). Therefore $f.J\neq J$, and upon replacing $f$ by $f^{-1}$, we may assume that $J\subsetneq f.J$.  For $n\in \Z$, write $\mathcal{L}_n=\{hf^n.J : h\in N\}$, so that $\mathcal{L}=\bigcup_{n\in \Z}\mathcal{L}_n$ is the $\varphi$-orbit of $J$. Since $f^n.J\in \mathcal{W}_\varphi(N)$, any two elements of $\mathcal{L}_n$ are either equal or disjoint. Moreover, every $I\in \mathcal{L}_n$ is contained in a unique element of $ \mathcal{L}_{n+1}$: indeed,  $I=hf^n.J\Subset hf^{n+1}.J$. It follows that $\mathcal{L}$ is a discrete $\varphi$-invariant lamination.
\end{proof}

\subsection{Examples of simplicial trees: groups with cyclic germs at infinity}\label{ssec.cyclicgerm}
\label{subsubsec.simplicialfocal}
Let us revisit the construction of exotic actions in \S \ref{sec:cyclic_germs1}, and show that it actually gives rise to a simplicial laminar action. As in in \S \ref{sec:cyclic_germs1}, we consider a locally moving subgroup $G\subset \homeo_{0}(X)$, where $X=(a,b)$, with cyclic germs at $b$. We fix the homomorphism $\tau\colon G\to \Z\cong \Germ(G,b)$ and the preferred element $f_0\in G$ such that $\tau(f_0)=1$, in such a way that $f_0(x)>x$ on a neighborhood of $b$.
We then fix a bi-infinite sequence $\s$ as in \eqref{e-conditions-sequence}: we fix $s_0$ sufficiently close to $b$, such that $s_n=f_0^n(s_0)$ tends to $b$ as $n\to \infty$, and set $\s=(s_n)_{n\in \Z}\in X^\Z$. Recall that in \eqref{e-action-sequences} we have defined  an action of $G$ on $X^\Z$, by the expression
\[
g\cdot \mathsf{t}=\left (g(t_{n-\tau(g)})\right )_{n\in \Z},\]
and considered the $G$-orbit of $\s$, denoted as $\mathsf{S}$. The relation $\mathsf{t}\prec \mathsf{t}'$ if and only if $t_m<t'_m$, with $m=m(\mathsf t,\mathsf t')=\max\{n\in \Z:t_n\neq t_n'\}$, defines a $G$-invariant order on $\mathsf S$ (Lemma \ref{l-construction-cyclic-order}), and we observed that the dynamical realization $\varphi_{\s}\colon G\to \homeo_{0}(\R)$ of the action of $G$ on $(\mathsf S,\prec)$ is faithful and minimal (Proposition \ref{prop:cyclic_germs_minimal_faithful}).
Here we want to see that $\varphi_{\s}$ is in fact a simplicial laminar action.

\begin{rem} 
	After Remark \ref{l-positive-horograding}, we have that the dynamical realization $\varphi_{\s}$ can also be positively horograded by the action of $G$ on $X$. This is an illustration of the fact that a planar directed tree encoding a laminar action is not unique, and identifying a tree with good properties may be important for some purposes.
\end{rem}

For $n\in \Z$, denote by  $\Z_{\ge n}$ the set of integers $j\ge n$.  We let $\mathsf{S}_{\ge n}\subset X^{\Z_{\ge n}}$ be the subset of sequences indexed by $\Z_{\ge n}$ obtained by restricting  sequences in $\mathsf{S}$ to $\Z_{\ge n}$:
\[\mathsf{S}_{n}=\{(t_j)_{j\ge n}: (t_j)_{j\in \Z}\in \mathsf{S}\}.\]
We will call \emph{truncation} this operation. Given a sequence  $(t_j)_{j\ge n} \in \mathsf{S}_n$, we say that $(t_j)_{j\ge n+1}\in \mathsf{S}_{n+1}$ is its \emph{successor}.
The disjoint union $\bigsqcup_{n\in \Z} \mathsf{S}_n$ is naturally the vertex set of a  simplicial tree $\Tbb$, obtained by connecting each element to its successor. Indeed, it is clear that the graph obtained in this way has no cycles; moreover it is connected, because all elements of $\mathsf{S}$ eventually coincide with the sequence $\s$ (Lemma \ref{l-cyclic-eventual}).

If we endow all edges of $\Tbb$ with the orientation from a point to its successor, then all edges point to a common end $\omega\in \partial \Tbb$, hence we get a directed tree $(\Tbb,\treeorder)$ with focus $\omega$.
Note that every $\mathsf{t}=(t_j)_{j\in \Z}\in \mathsf{S}$ defines a bi-infinite ray of $\Tbb$, whose vertices are the successive truncations of $\mathsf{t}$. For every $\mathsf{t}\in \mathsf{S}$, this  sequence converges to $\omega$ as $n\to +\infty$. As $n\to -\infty$, it converges to some end $\alpha_{\mathsf{t}} \in \partial^*\Tbb=\partial \Tbb\setminus \{\omega\}$.  The map $\mathsf{t}\mapsto \alpha_{\mathsf{t}}$ is clearly injective, and thus allows to identify $\mathsf{S}$ with a subset of $\partial^* \Tbb$. 

The group $G$ has a natural action on $(\Tbb,\treeorder)$: for every vertex $v=(t_j)_{j\ge n}$ of $\Tbb$ and $g\in G$, we set
\[g\cdot (t_j)_{j\ge n}=\left (g(t_{j-\tau(g)})\right )_{j\ge n+\tau(g)}.\]
Note in particular that if $v\in \mathsf{S}_n$, then $g\cdot v\in \mathsf{S}_{n+\tau(g)}$. This action is by simplicial automorphism and fixes the end $\omega$. Observe that if $g\in G$ is such that $\tau(g)\neq 0$, then $g$ has no fixed point on  $\Tbb$, and thus acts as a hyperbolic isometry. On the other hand, if $\tau(g)=0$, then $g$ preserves each of the sets $\mathsf{S_n}$ and acts  as an elliptic isometry (indeed, since $g$ acts trivially on some neighborhood of $b$ in $x$, it must fix all vertices  $(s_j)_{j\ge n}$ for $n$ large enough). 

Let us now define a planar order on $(\Tbb,\treeorder)$. In this case, this just means an order $<^v$ for every $v=(t_j)_{j\ge n}\in \Tbb$ on the set of edges $E^-_v$ which lie below $v$ (i.e.\ opposite to $\omega$). Fix $v=(t_j)_{j\ge n}$, and consider two distinct edges $e_1, e_2\in E_v^-$. Then for $i\in\{1,2\}$, we have  $e_i=(w, v)$ for some $w= (t^{(i)}_j )_{j\ge {n-1}}$ with $t^{(1)}_j=t^{(2)}_j=t_j$ for $j\ge n$, and $t^{(1)}_{n-1}\neq t^{(2)}_{n-1}$. Thus we set $e_1<^v e_2$ if and only if $t^{(1)}_{n-1}<t^{(2)}_{n-1}$. The collection $\{<^v: v\in \Tbb\}$ defines a planar order $\prec$ on $(\Tbb, \omega)$ which is invariant under the action of $G$. 

The map $\mathsf{t}\mapsto \alpha_{\mathsf{t}}$ is $G$-equivariant and order preserving, with respect to the order on $\mathsf{S}$ and the order on $\partial^* \Tbb$ induced from the planar order. 
Thus the $G$-action on $(\mathsf{S}, \prec)$ can be identified with an action on an orbit in $\partial^* \Tbb$. It follows that $\varphi_{\s}$ is represented by the action on the planar directed tree $(\Tbb, \prec, \omega)$.

\chapter{Groups with many micro-supported actions}
\label{sec.examplesRfocal}

The goal of this chapter is to present a generalization of the Brin--Navas group defined in Example~\ref{e-BN}. The new examples, called \emph{generalized Brin--Navas groups}, serve to show that some results in this article concerning locally moving groups do not extend to the general context of micro-supported groups (notably Corollary \ref{cor.unique} and Theorem \ref{t-lm-C1} fail).

Recall from Theorem \ref{t-laminations-microsupported} that if $G\subset \homeo_0(\R)$ is a micro-supported subgroup	 whose action on $\R$ is minimal, then either $G$ is locally moving or its standard action is laminar.
In this subsection we give a general construction of micro-supported subgroups of $\homeo_0(\R)$ whose action is minimal and laminar.  We will use this construction to illustrate that the class of micro-supported subgroups of $\homeo_0(\R)$ is much more flexible than the class  of locally moving groups: we can find groups that admit uncountably many, pairwise non-semi-conjugate, faithful micro-supported actions (in contrast with Rubin's theorem for locally moving groups, see Corollary \ref{cor.unique}). Moreover, many of these examples can even be chosen to be of class $C^1$ (in  contrast with Theorem \ref{t-lm-C1}). After Theorem \ref{t-laminations-microsupported}, these actions are all laminar. In fact, these groups are directly described as groups of automorphisms of directed (simplicial) trees, by adapting the classical construction of Burger and Mozes \cite{BuMo}, and the related groups defined by Le Boudec \cite{LB-ae}. The crucial point is that the directed trees admit many planar orders which are preserved by our groups.

\section{Construction of laminar actions}
We say that a pair $(A,a_0)$ is a \emph{marked alphabet} if $A$ is a set and $a_0\in A$. Then, we denote by $\mathsf{S}\subseteq A^\Z$ the set of sequences with values in $A$ which take the constant value $a_0$ in all but finitely many terms. Following the notation as in \S\ref{subsubsec.simplicialfocal}, we denote by $\mathsf{S}_n$ the truncations to $\Z_{\geq n}$ of the elements in $\mathsf{S}$. Also, given a sequence $(t_j)_{j\geq n}$, we say that $(t_j)_{j\geq n+1}$ is its successor.  This defines a simplicial directed tree $(\mathbb{T}_A,\treeorder)$ whose focus $\omega$ is defined by the successive truncations of the constant path $\s$ with value $a_0$. 
By abuse of notation, we will identify $\Tbb_A$ with its set of vertices $\bigsqcup_{n\in \Z}\mathsf{S}_n$.
Recall that for a vertex $v\in \Tbb_A$, we denote by $E^-_v$ the set of edges below $v$ (opposite  to $\omega$). For $v=(t_j)_{j\ge n}$, the set $E^-_v$ is naturally identified with the alphabet $A$, considering the labeling \[\dfcn{j_v}{E^-_v}{A}{[(t_j)_{j\geq {n-1}},(t_j)_{j\geq n}]}{t_{n-1}.}\]
Note also that every $g\in \Aut(\Tbb_A, \treeorder)$ induces a bijection between $E^-_v$ and $E^-_{g.v}$. We write $\sigma_{g, v}:= j_{g(v)}\circ g\circ  j_v^{-1}$ for the induced permutation of $A$. Observe that we have the cocycle relation 
\begin{equation}\label{eq:cocycle-BN}
	\sigma_{g,h(v)}\,\sigma_{h,v}=\sigma_{gh,v}.
\end{equation}

\begin{dfn}
	Let $(A, a_0)$ be a marked alphabet,   $G\subseteq \Bij(A)$ a group of permutations of $A$, and $(\Tbb_A, \treeorder)$ the directed tree defined above together with the labelings $j_v$, $v\in \Tbb_A$. We define the \emph{generalized Brin--Navas group}  $\BN(A;G)$ as the group of all elements $g\in \Aut(\Tbb_A, \treeorder)$ such that $\sigma_{g, v}\in G$ for all $v\in \Tbb_A$ and $\sigma_{g, v}=\id$ for all but finitely many $v\in \Tbb_A$.
\end{dfn}
\begin{rem}
	The name comes from the fact that the Brin--Navas group considered in Example~\ref{e-BN} is isomorphic to the group $\BN(\Z;G)$ where $a_0=0\in \Z$, and $G\subset\Bij(\Z)$ is the group of translations of the integers. 
\end{rem}

\begin{rem}
	Note that we omit the marked element $a_0$ in the notation $\BN(A;G)$. This is because we will only consider transitive subgroups $G\subseteq \Bij(A)$, and in this case the definition of $\BN(A;G)$ does not depend on the choice of the marked element, up to a group isomorphism induced by an isomorphism between the corresponding trees.
\end{rem}

We need to define a suitable generating set. For this, given a vertex $v\in \Tbb_A$, denote by $G_v\subset \BN(A;G)$ the subgroup of all $g\in \BN(A;G)$ which fix $v$ and such that $\sigma_{g, w}=\id$ for $w\neq v$ (clearly $G_v\cong G$). We then choose an extra generator defined as follows. Note that the shift map $\sigma\colon \mathsf S\to \mathsf S$, which sends a bi-infinite sequence $(t_j)_{j\in \Z}$ to $(t_{j-1})_{j\in \Z}$ naturally acts on the set of truncated sequences $\bigsqcup_{n\in \Z}\mathsf S_n$ preserving the successor relation. Thus it defines an automorphism
$f_0\in\Aut(\Tbb_A,\treeorder)$. It is direct to check that $f_0$ is a hyperbolic element in $\BN(A;G)$, whose axis consists of the geodesic $(w_n)_{n\in \Z}$ with $w_n=(a_0)_{j\geq n}$, and that moreover $\sigma_{f_0,v}=\id$ for every vertex $v\in \Tbb_A$. 
With this notation set, we have the following.
\begin{lem}\label{l-BN-generators}
	Let $(A,a_0)$ be a marked alphabet, and assume that $G\subset \Bij(A)$ acts transitively on $A$. Then the group $\BN(A;G)$ is generated by $G_{w_0}$ and $f_0$. In particular, it is finitely generated as soon as $G$ is so. 
\end{lem} 
\begin{proof}
	Write $\Gamma=\langle G_{w_0},f_0\rangle$ for the subgroup of $\Aut(\Tbb_A,\treeorder)$ generated by $G_{w_0}$ and $f_0$. We first observe the following.
	
	\begin{claim}
		For every vertices $v_1,v_2\in\Tbb_A$, there exists $g\in\Gamma$ such that $g.v_1=v_2$ and $\sigma_{g,v}=\id$ for every vertex $v\treeorder v_1$.
	\end{claim} 
	\begin{proof}[Proof of claim]
		First notice that, by composing with powers of $f_0$ and using the cocycle relation \eqref{eq:cocycle-BN}, we can assume that both $v_1$ and $v_2$ belong to $\mathsf{S}_0$. 
		Similarly, we can also assume that $v_1=w_0=(a_0)_{n\ge 0}$; write $v_2=(t_n)_{n\geq 0}$. Since $G$ acts transitively on $A$,  there exists a sequence $(h_n)_{n\geq 0}$ in $G$ such that $h_n(a_0)=t_n$. Moreover, since $t_n=a_0$ for $n$ large enough, we can take $h_n$ to be the identity for $n$ large enough. With abuse of notation, denote by $h_n\in G_{w_0}$ the element with $\sigma_{h_n,w_0}=h_n$. Then, the product $g:=\prod_{n\geq 0}(f_0^nh_nf_0^{-n})$ is actually a finite product and thus defines an element of $\Gamma$, which moreover satisfies $\sigma_{g,v}=\id$ for every $v\in \bigsqcup_{n<0}\mathsf S_n$. It follows directly from the choices that $g.w_0=v_2$. This  proves the claim. 
	\end{proof}
	
	Take a vertex $v_0\in\Tbb_A$. By the previous claim, we can take $g\in \Gamma$ so that $g.v_0=w_0$ and $\sigma_{g,v}=\id$ for every $v\treeorder v_0$. Then it is direct to check that $g^{-1}G_{w_0}g=G_{v_0}$. This shows that $G_{v_0}\subseteq \Gamma$ for every vertex $v_0\in\Tbb_A$. Finally, given $g\in \BN(A;G)$, write
	\[C(g):=|\{v\in \Tbb_A:\sigma_{g,v}\neq \id\}|.\]
	Notice that if $C(g)=0$, then $g$ is a power of $f_0$, hence it belongs to $\Gamma$. Take an element $g\in \BN(G)$ with $C(g)> 0$, and a vertex $v\in \Tbb_A$ so that $\sigma_{g,v}\neq \id$. Then we can find $h\in G_{v}$ satisfying $\sigma_{h,v}=\sigma_{g,v}$, and we have $C(gh^{-1})=C(g)-1$. By repeating this procedure finitely many times, we can find $h'\in\Gamma$ so that $C(gh')=0$. We deduce that $g$ belongs to the group $\Gamma$.
\end{proof}

Notice that for every total order $<$ on the alphabet $A$,  there exists a unique planar order $\prec$ on $(\Tbb_A,\treeorder)$ for which the maps $j_v$ are order preserving. If in addition the subgroup $G\subset\Bij(A)$ preserves $<$, the group $\BN(A;G)$ preserves the associated planar order $\prec$. We call this action the induced \emph{planar directed tree representation} associated with $G$ and $<$.

\begin{prop}\label{prop.BNfocal}
	Let $(A,a_0)$ be a marked alphabet. Assume that $G\subset\Bij(A)$ acts transitively on $A$, and let $<$ be a $G$-invariant total order on $A$. Then, the dynamical realization of the planar directed tree representation associated with $G$ and $<$, is a faithful minimal laminar action of $\BN(A;G)$, which is moreover micro-supported.
	
	Moreover, for two different $G$-invariant orders on $A$, the dynamical realizations of their corresponding planar directed tree representations are not positively semi-conjugate.
\end{prop}
\begin{proof}
	Let $\prec$ be the planar order associated with $<$, and let $\Phi\colon \BN(A;G)\to\Aut(\Tbb_A,\treeorder,\prec)$ be the corresponding action. Since $\Phi$ acts transitively on the vertices of $\Tbb_A$ and $\BN(A;G)$ contains a hyperbolic element, focality of $\Phi$ follows.
	
	Let $\xi\in\partial^\ast\Tbb_A$ be the end defined by the sequence $(w_n)_{n\leq 0}$, and let $\mathcal{O}_\xi$ be its orbit under the induced $\Phi$-action on $\partial^\ast\Tbb_A$. Since $f_0$ is a hyperbolic element with axis $(w_n)_{n\in\Z}$, it acts on $(\mathcal{O}_\xi,\prec)$ as a homothety fixing $\xi$. This allows us to apply Lemma \ref{lem.minimalitycriteria} to deduce that the dynamical realization of $\BN(A;G)\to\Aut(\mathcal{O}_{\xi},\prec)$ is minimal. In other words, the dynamical realization of $\Phi$, that we denote by $\varphi$, is minimal. On the other hand, since $G_v$ is supported on the set of points below $v$, its associated action on $\mathcal{O}_\xi$ is supported on the shadow of $v$. This implies that the support of $\varphi(G_v)$ is relatively compact, and thus by Proposition \ref{p-micro-compact} we deduce that $\varphi$ is micro-supported. Finally, faithfulness of $\varphi$ follows from that of $\Phi$.
	
	Given an element $h\in G$ and a vertex $v\in \Tbb_A$ we denote by $h_v$ the element of $G_v$ satisfying $\sigma_{h_v,v}=h$. Also, let $p\in \R$ be the fixed point of $\varphi(f_0)$. Then, given $h\in G$, we have
	\[\varphi(h_{w_0})(p)>p\quad\text{if and only if}\quad h(a_0)>a_0.\]
	Since the action of $G$ on $A$ is transitive, this implies that we can read the total order $<$ from the action $\varphi$. In particular, dynamical realizations corresponding to different $G$-invariant orders on $A$ give rise to non-positively-conjugate actions. Finally, since these dynamical realizations are minimal and not positively conjugate, they are not positively semi-conjugate.
\end{proof}

Note that by the transitivity assumption in Proposition \ref{prop.BNfocal}, the planar directed tree representation is in fact determined by a choice of a left-invariant preorder on $G$, as an abstract group: write the identification $A\cong G/H$, with $H=\stab_G(a_0)$, so that a $G$-invariant order on $A$ corresponds to a preorder $\le$ on $G$ whose residue (that is, the subgroup of elements which are $\le$-equivalent to $1_G$; see \S \ref{s-preorders}) is $H$. Hence,
the second part in the statement says that two different left-invariant preorders on $G$ with same residue yield non-positively-semi-conjugate actions. As a particular case, one can consider a left-invariant order on $G$, in which case we can identify the marked alphabet with $(G,1_G)$, and the subgroup $G\subset \Bij(G)$ is the group of left translations. In such a case, we will simply write $\BN(G)$ for the group $\BN(G;G)$.

Recall  that if $G\subseteq \homeo_0(\R)$ is a locally moving group, then every faithful locally moving action of $G$ on $\R$ is conjugate to its standard action (by Rubin's theorem, or by Corollary \ref{cor.unique}). The groups $\BN(G)$ show that the this is far from being true for micro-supported subgroups of $\homeo_0(\R)$. 

\begin{cor}\label{cor.brin-navas} Let $G$ be a finitely generated group admitting uncountably many left-invariant orders. Then $\BN(G)$ has uncountably many, pairwise non-conjugate, faithful minimal micro-supported actions on the line. 
\end{cor}
\begin{proof} After the previous discussion, we can apply Proposition \ref{prop.BNfocal} to each left-invariant order in $\LO(G)$, which gives the desired result.
\end{proof}

\section{Groups with many differentiable micro-supported actions} \label{s-micro-C1} Here we extend the result given by Corollary \ref{cor.brin-navas} by showing that for some finitely generated groups $G$, we can actually get many faithful micro-supported actions of class $C^1$ of the group $\BN(G)$ (compare with Theorem \ref{t-lm-C1}).

\begin{thm}\label{t.C1nonrigid} There exists a finitely generated group  admitting uncountably many, faithful minimal micro-supported actions, which are pairwise not conjugate (and not conjugate to any non-faithful action), and each of which is semi-conjugate to a $C^1$ action on $\R$.More precisely, the group $\BN(\Z^2)$ satisfies such properties.
\end{thm}
\begin{rem}
	In fact, it seems also possible to prove that the group $\BN(\Z^2)$ admits  faithful minimal micro-supported actions which are \emph{conjugate} to a $C^1$ action and pairwise not conjugate, but since this is more technical, we will content ourselves of the previous statement (which suffices to disprove the analogue of Theorem \ref{t-lm-C1} for micro-supported groups).
\end{rem}

Let us also remark that from the point of view of regularity, our construction for $\BN(\Z^2)$ relies on the Pixton--Tsuboi examples (see Tsuboi \cite{PixtonTsuboi}) and naturally gives actions by diffeomorphisms of class $C^{\alpha}$, for any $\alpha<3/2$. On the other hand, a result of Deroin, Kleptsyn, and Navas (see \cite[Théorème C]{DKN-acta}) tells that our construction cannot give more regular actions. We are not aware of any method that could give uncountably many, pairwise non-semi-conjugate, micro-supported actions by $C^2$ diffeomorphisms on $\R$, morally because actions by $C^2$ diffeomorphisms on $[0,1]$ with an exceptional minimal set in $(0,1)$ are expected to satisfy many constraints. However, an elementary construction gives the following.

\begin{prop}
	For any $n\ge 1$, there exists a finitely generated group admitting (at least) $2^n$ pairwise non-positively-conjugate, faithful minimal micro-supported actions by $C^\infty$ diffeomorphisms on $\R$.
\end{prop}

\begin{proof}[Sketch of proof]
	For any integer $n\ge 1$, consider the subgroup $H_n$ of $\BN(\Z)$ generated by the power $f_0^n$ and the subgroups $G_{w_k}=f_0^kG_{w_0}f_0^{-k}$, for $k\in \{0,\ldots, n-1\}$. The action of $H_n$ on the directed tree $(\Tbb_\Z,\treeorder)$ has $n$ distinct orbits of vertices (the orbits of $w_k$, for $k\in \{0,\ldots,n-1\}$), so it admits at least $2^n$ distinct planar orders, obtained by taking all possible switching of signs of the standard orders on the sets of edges $E_{w_k}^-\cong \Z$, for $k\in \{0,\ldots,n-1\}$. It is not difficult to see (elaborating on the arguments described in this section) that the dynamical realizations of the actions of $H_n$ on these planar ordered trees are conjugate to $C^\infty$ actions on $\R$, and that different choices of orders on the $E_{w_k}^-$ lead to pairwise non-positively-conjugate actions.
\end{proof}

We now discuss Theorem \ref{t.C1nonrigid}. It will be obtained as a consequence of Proposition \ref{prop.larguisima}, which is 
a criterion to recognize faithful actions of $\BN(G)$ on the line. As a preliminary result, which may also help to follow the rather technical proof of Proposition \ref{prop.larguisima}, we work out  a presentation of $\BN(G)$, which is analogous to the one for $\BN(\Z)=B$ appearing in Example \ref{e-BN}. 
So let $G$ be a finitely generated group, and fix a presentation \[G=\langle g_1,\ldots,g_n\mid r_\nu(g_1,\ldots,g_n)\, (\nu\in\N)\rangle.\]
We will assume for simplicity that the generating set is symmetric and that for every distinct $i,j\in \{1,\ldots,n\}$, the generators $g_i$ and $g_j$ represent distinct non-trivial elements in $G$. The free product $G\ast \Z$ admits the presentation
\[G\ast\Z=\langle f,g_1,\ldots,g_n\mid r_\nu(g_1,\ldots,g_n)\, (\nu\in\N)\rangle.\]
It is convenient to introduce the following more redundant presentation of $G\ast \Z$, which corresponds to applying  a (multiple) Tietze transformation to the previous one.  For $m\in \Z$, write  $\mathcal{G}_m=\{g_{i,m}\}_{i\in\{1,\ldots,n\}}$, and consider the union $\mathcal G=\bigcup_{m\in \Z}\mathcal G_m$.
Then, we get the presentation
\begin{equation}\label{eq.presentation-gamma}
	G\ast\Z=\left\langle f_0,\mathcal{G}\,\middle\vert\, fg_{i,m}f^{-1}=g_{i,m+1} \text{ }(i\in \{1,\ldots,n\},m\in \Z),\text{  }r_\nu(g_{1,0},\ldots,g_{n,0})\, (\nu\in\N)\right\rangle.
\end{equation}
Write $\mathcal{R}_0$ for the set of relators in the previous presentation. As in the proof of Lemma \ref{prop.BNfocal}, given $g\in G\subset\Bij(G)$ and $v\in\Tbb_G$, we denote by $g_v$ the element in $\BN(G)$ fixing $v$ such that $\sigma_{g,v}=g$, and $\sigma_{g,w}=\id$ for every $w\neq v$. Then we denote by $\Psi_0\colon G\ast\Z\to \BN(G)$ the morphism such that $\Psi_0(f)=f_0$ and $\Psi_0(g_{i,m})=f_0^m(g_i)_{w_0}f_0^{-m}$ for every $i\in\{1,\ldots,n\}$ and $m\in \Z$ (as before, we write $w_0=(1_G)_{n\ge 0}$).

In order to complete the set of relations for the desired presentation of $\BN(G)$, we need to study the support of some elements.

\begin{lem}\label{l.relations1}
	With notation as above, fix $m\in \Z$. Then, for every $g\in \langle\mathcal G_m\rangle \setminus \{1\}$ and $g_1,g_2\in \bigcup_{q<m} \mathcal G_q$ we have that the commutator $[g_1,gg_2g^{-1}]$ is in the kernel of $\Psi_0$.
\end{lem}

\begin{proof}
	We will show that for any such choices of elements, $g_1$ and $gg_2g^{-1}$ have disjoint support in $\Tbb_G$ and hence commute. On the one hand, for every $k\in \Z$, and $h\in \mathcal G_k$, the support of $\Psi_0(h)$ is contained in $\Tbb_G^{w_k}$, where $w_k=(1_G)_{n\ge k}$, and $\Tbb_G^{w_k}=\{u\in \Tbb_G:u\treeordereq w_k\}$ denotes the subtree of vertices below $w_k$. On the other hand,  for any  $h\in\bigcup_{q<m}\mathcal{G}_q$, we have that the support of $\Psi_0(ghg^{-1})$ is contained in the subtree $\Tbb_G^{v_{g,m-1}}$, where $v_{g,m-1}\in\Tbb_G$ is the vertex corresponding to the sequence $(t_n)_{n\geq m-1}$ such that $t_0=1_G$ for $n\geq m$ and $t_{m-1}=g$. In particular, these two remarks apply respectively to the elements $g_1$ and $g_2$ from the statement, so that the images $\Psi_0(g_1)$ and $\Psi_0(gg_2g^{-1})$ have disjoint supports.
\end{proof}

Write $\mathcal{R}_1$ for the set of all the commutation relators from Lemma \ref{l.relations1} and set \begin{equation}\label{eq.presentation-gamma2}\Gamma=\langle f,\mathcal{G}\mid\mathcal{R}_0,\mathcal{R}_1\rangle.\end{equation}We are ready for the following statement.

\begin{prop}\label{lem.presentation}
	The map $\Psi\colon \Gamma\to \BN(G)$ induced by $\Psi_0$ is an isomorphism. 
\end{prop}

Before proving Proposition \ref{lem.presentation} we need to fix some notation and state a technical lemma that will also be used later for Proposition \ref{prop.larguisima}, and whose proof is postponed. As in the proof of Lemma \ref{l.relations1}, given $g\in G$ and $m\in\Z$ denote by $v_{g,m}\in\Tbb_G$ the vertex $(t_n)_{n\geq m}$ such that $t_n=1_G$ for $n>m$ and $t_m=g$; in particular $w_m=v_{1,m}$. 
Note that after the conjugation relations in $\mathcal R_0$, the subgroup $H:=\langle \mathcal G\rangle$ is the normal closure of $\mathcal G_0$ in $\Gamma$, and the quotient $\Gamma/H$ is generated by the image of $f$. Given $\gamma\in H$, we  denote by $\|\gamma\|_{\mathcal{G}}$ its word-length with respect to the generating system $\mathcal{G}$.

\begin{lem}\label{sublem.technicalalternativo}\label{sublem.technical}
	Take a non-trivial element $\gamma\in H$, written as $\gamma=\gamma_1\cdots\gamma_k$ with $\gamma_j\in\mathcal{G}_{m_j}$ for $j\in\{1,\ldots,k\}$, and write $M=\max\{m_j:j\in\{1,\ldots,k\}\}$. Assume that $\Psi(\gamma)(w_m)=w_m$ for some $m<M$. Then, there exist $h_1,\ldots,h_l\in\left \langle\bigcup_{m<M}\mathcal{G}_{m}\right \rangle$ and pairwise distinct elements $f_1,\ldots,f_l\in\langle\mathcal{G}_M\rangle$,  such that 
	\begin{equation}\label{eq.sublem}
		\gamma=(f_1h_1f_1^{-1})\cdots (f_lh_lf_l^{-1})
	\end{equation}
	and $\|h_i\|_{\mathcal{G}}<k$ for every $i\in \{1,\ldots,l\}$.
\end{lem}

\begin{proof}[Proof of Proposition \ref{lem.presentation}]
	First notice that, by Lemma \ref{l-BN-generators}, $\Psi$ is surjective. To prove injectivity of $\Psi$, suppose by contradiction that $\ker \Psi$ is non-trivial. As the tree $\Tbb_G$ is simplicial, we have a morphism $\rho\colon \BN(G)\to \Z$ given by the translation length in the direction of $\omega$ (see Proposition \ref{p-focal-simplicial}). Notice that $\rho$ vanishes at $G_{w_0}=\Psi(\mathcal G_0)$ and that $\rho(f_0)$ is non-trivial. As the quotient of $\Gamma$ by the normal closure $H$ of $\mathcal G_0$ is generated by the image of $f$, we deduce that $\ker \Psi\subseteq H$.
	In particular,  the word-length $\|\cdot\|_{\mathcal G}$ is defined on the kernel of $\Psi$, so that we can consider a non-trivial element $\gamma\in \ker \Psi$ of minimal word-length. 
	We write $\gamma=\gamma_1\cdots\gamma_k$ with $k=\|\gamma\|_{\mathcal{G}}$ and $\gamma_j\in \mathcal G_{m_j}$ for every $j\in \{1,\ldots,k\}$. As $\Phi(\gamma)$ acts trivially, in particular it fixes every vertex of the form $w_m$, we can apply Lemma \ref{sublem.technical} to the factorization 
	$\gamma=\gamma_1\cdots\gamma_k$, and obtain a decomposition as in \eqref{eq.sublem} with $\|h_i\|_{\mathcal{G}}<\|\gamma\|_{\mathcal{G}}$ for $i\in\{1,\ldots,l\}$. Keeping the same notation as in Lemma \ref{sublem.technical}, we observe that the support of $\Psi(f_ih_if_i^{-1})$ is contained in the subtree $\Tbb_G^{v_{f_i,M-1}}$, for $i\in\{1,\ldots,l\}$.  Thus, since $f_i\neq f_j$ for $i\neq j$, we get that different factors in the factorization of $\gamma$ above have disjoint support. We deduce that every such factor is in the kernel of $\Psi$, and therefore also every element $h_i$ is. However, as the word-length of every such element is less than $\|\gamma\|_{\mathcal{G}}$, the minimality assumption on $\gamma$ implies that every $h_i$ is trivial, contradicting the choice of $\gamma$.
\end{proof}

\begin{proof}[Proof of Lemma \ref{sublem.technicalalternativo}] 
	In order to simplify notation, from here and until the end of the proof, we will write $\prod_{j\in E}\alpha(j):=\alpha(i_1)\cdots \alpha(i_k)$, for any function $\alpha\colon E\to\Gamma$, with $E=\{i_1,\ldots,i_k\}\subseteq\N$ with $i_1<\cdots<i_k$. We will also write	
	$\sigma_{g}$ instead of $\sigma_{\Psi(g),w_M}$, for every $g\in \Gamma$. Note that $\sigma_{g}=\id$ for every $g\in\langle\bigcup_{m<M}\mathcal{G}_m\rangle$.
	For given $j\in \{1,\ldots,k\}$, we let $i_j$ be the index such that $\gamma_j=g_{i_j,m_j}$.
	Notice that $\Psi(\gamma_j)(w_M)=w_M$ for every $j\in \{1,\ldots,k\}$, and	
	in particular $\Psi(\gamma)(w_M)=w_M$. Using the cocycle relation \eqref{eq:cocycle-BN}, we have
	\begin{equation}\label{eq:sigma-prod}
		\sigma_\gamma=\textstyle\prod_{j=1}^k\sigma_{\gamma_j}.
	\end{equation}
	On the other hand, since $\Psi(\gamma)(w_s)=w_s$ for some $s<M$ and the action of $G\subset\Bij(G)$ on $G$ is free, we conclude that $\sigma_{\gamma}=\id$. Since only factors in $\mathcal G_M$ give non-trivial factors in the product \eqref{eq:sigma-prod}, writing $P=\{j\in\{1,\ldots,k\}:\gamma_j\in \mathcal G_M\}$, we get the equality
	\[
	\id=\textstyle\prod_{j\in P}\sigma_{\gamma_j}.
	\]
	Note that for $j\in P$, the permutation $\sigma_j$ corresponds to the left translation by $g_{i_j}$, whence we get $\prod_{j\in P}g_{i_j}=1_G$. This gives	
	\begin{equation}\label{eq.identity}
		\textstyle\prod_{j\in P}\gamma_j=1_{\Gamma}.
	\end{equation}
	Next, for every $j\in \{1,\ldots,k\}$, consider the product in $\langle \mathcal G_M\rangle$ given by
	\[f_j:=\textstyle\prod_{l\in P,\,l\le j}\gamma_l.\]
	Set  $Q:=\{1,\ldots,k\}\setminus P$, and notice that after the relation \eqref{eq.identity} we can write
	\begin{equation}\label{eq.telescopic}
		\gamma=\textstyle \prod_{j\in Q}\left(f_j\gamma_jf_j^{-1}\right).
	\end{equation}
	Note that as $f_j\in \langle \Gcal_M\rangle$, and $\gamma_j\in\langle\mathcal{G}_{m_j}\rangle$ with $m_j<M$ for every $j\in Q$, we deduce from Lemma \ref{l.relations1} that we can make any two factors $f_i\gamma_if_i^{-1}$ and $f_j\gamma_jf_j^{-1}$ in  the product in \eqref{eq.telescopic} commute, provided $f_i\neq f_j$. Rearranging factors in this way, we can write \begin{equation}\label{eq.thefinal} \gamma=\textstyle \prod_{j\in Q_0}\left(f_j\left (\prod_{l\in Q_0^j}\gamma_l\right )f_{j}^{-1}\right),\end{equation}
	where $Q_0\subseteq Q$ is a section of the map $Q\to \Gamma$ given by $j\mapsto f_j$, and for $j\in Q_0$ we set $Q_0^j:=\{l\in Q_0:f_l=f_j\}$. We claim that the decomposition in \eqref{eq.thefinal} is the one that we are looking for. First notice that, by the choice of $Q_0$, $f_i\neq f_j$ whenever $i, j$ are different indices in $Q_0$. Secondly  for every $j\in Q_0$, by definition of $Q_0$ and $Q_0^j$, the element $h_j:=\prod_{l\in Q_0^j}\gamma_l$ belongs to $\left \langle\bigcup_{m< M}\mathcal{G}_m\right \rangle$ and clearly satisfies $\|h_j\|_{\mathcal{G}}\le |Q_0^j|<k$.
\end{proof}

Consider now, for a left-invariant order $<$ on a group $G$, the directed tree representation $\Phi\colon \BN(G)\to\Aut(\Tbb_G,\treeorder,\prec)$ associated with $G$ and $<$; let $\varphi\colon \BN(G)\to\homeo_0(\R)$ be its corresponding dynamical realization. Pursuing the discussion from the proof of Proposition \ref{prop.BNfocal}, we extract some properties of $\varphi$, which characterize it up to (positive) semi-conjugacy.

First recall that $G_{w_0}\subset\BN(G)$ is a subgroup supported on the subset of points below the vertex $w_0$. This implies that $J_0:=\suppphi(G_{w_0})$ is a relatively compact interval. Since $\Phi(f_0)$ is a hyperbolic element, the homeomorphism $\varphi(f_0)$ is a homothety (see for instance Remark \ref{rem.homotheties_isometric}). 
For $n\in \Z$, write $J_n:=f_0^{n}.J_0$; since $f_0^{-1}.w_0\treeorder w_0$, we have the inclusion $J_{-1}=f_0^{-1}.J_0\Subset J_0$.  Finally, since the action of $G$ on itself by left translations is free, we get that $G_{w_0}$ acts freely on $E_{w_0}^-$, which implies $g.J_{-1}\cap J_{-1}=\emptyset$ for every $g\in G_{w_0}\setminus\{1\}$. Moreover, the total order $<$ on $G$ (which coincides with the planar order $\prec^{w_0}$ on $E_{w_0}^-$) can be read from the action: \begin{center}$g>1$ $\Leftrightarrow$ $g.x>x$  for some $x\in J_{-1}$ (equivalently, for any $x\in J_{-1}$).\end{center} Summarizing we have the following:
\begin{enumerate}[label=(\alph*)]
	\item\label{i-BrinNavas} $\suppphi(G_{w_0})=:J_0$ consists of a relatively compact interval;
	\item\label{ii-BrinNavas} $\varphi(f_0)$ is a homotethy satisfying $J_{-1}\Subset J_0$, with $J_{-1}:=f_0^{-1}.J_0$;
	\item\label{iii-BrinNavas} $g.J_{-1}\cap J_{-1}=\emptyset$ for every $g\in G_{w_0}\setminus \{1\}$;
	\item\label{iv-BrinNavas} given $x\in J_{-1}$ and $g\in G_{w_0}$, it holds that $g.x>x$ if and only if $g>1$.
\end{enumerate}
For the next statement, recall that $\BN(G)$ is isomorphic to the group $\Gamma$ which, in turn, can be written as the quotient $\Gamma=G\ast\Z/\langle\langle\mathcal{R}_1\rangle\rangle$ (following the notation in  \eqref{eq.presentation-gamma} and \eqref{eq.presentation-gamma2}, see Proposition \ref{lem.presentation}). We denote by $\pi\colon G\ast\Z\to\Gamma$ the corresponding projection. We also define the \emph{height} of an element $\gamma\in H=\langle\mathcal{G}\rangle$ as 
\begin{equation}\label{eq.height}\mathsf{ht}(\gamma):=\inf\left \{n\in\Z:\gamma\in\left \langle\textstyle \bigcup_{m\leq n}\mathcal{G}_m\right \rangle\right \}.\end{equation} Notice that $\mathsf{ht}(\gamma)>-\infty$ for every $\gamma\in H\setminus\{1\}$. Indeed, if it were not the case, its image under the isomorphism $\Psi$ in Proposition \ref{lem.presentation} would be trivial, which is not the case.

\begin{prop}\label{prop.larguisima} Let $(G,<)$ be a finitely generated left-ordered group. Consider the free product $G\ast\Z$, and denote by $f$ a generator of its cyclic factor. Consider also an action $\varphi_0\colon G\ast\Z\to\homeo_0(\R)$ and assume that it satisfies conditions \ref{i-BrinNavas} to \ref{iv-BrinNavas} above, with $\langle \mathcal G_0\rangle$ and $f$ instead of $G_{w_0}$ and $f_0$, respectively. Then, $\varphi_0$ factors through the quotient $\pi\colon G\ast\Z\to\Gamma$, inducing an action $\varphi_1\colon \Gamma\to \homeo_0(\R)$ which is (positively) semi-conjugate to the dynamical realization of the directed tree representation associated with $G\subset\Bij(G)$ and $<$. 
\end{prop}

\begin{proof} After Proposition \ref{lem.presentation}, in order to show that $\varphi_0$ factors through the projection $\pi$ we need to check that the elements in $\mathcal{R}_1$ belong to the kernel of $\varphi_0$. That is, we need to check that the elements of the form $[g_{1},gg_{2}g^{-1}]$ are in the kernel of $\varphi_0$, whenever $g\in\langle\mathcal{G}_m\rangle\setminus\{1\}$ and $g_1,g_2\in \bigcup_{q<m}\mathcal{G}_q$. We closely follow the proof of Lemma \ref{l.relations1}.	
	On the one hand,  for every $q<m$ and $h\in \mathcal G_q$, there exists $h'\in \mathcal G_0$ such that $h=f^qh'f^{-q}$, so that by condition \ref{i-BrinNavas}, we have that the support of $\varphi_0(h)$ is contained in $J_q:=f_0^q.J_0$, and thus in $J_m$ after condition \ref{ii-BrinNavas} (and the same argument for $q=m$).
	On the other hand, for any $g\in \langle\mathcal{G}_m\rangle\setminus\{1\}$,  condition \ref{iii-BrinNavas} gives that $g.J_{m-1}\cap J_{m-1}=\emptyset$. Thus, the support of $\varphi_0(gg_{2}g^{-1})$ is disjoint from $J_{m}$. Putting this all together, we get that the support of $\varphi_0(g_{1})$ and that of $\varphi_0(gg_{2}g^{-1})$ are disjoint, as desired. 
	
	As in the statement, denote by $\varphi_1\colon \Gamma\to\homeo_0(\R)$ the action induced by $\varphi_0$,  and let $\mathcal{I}$ be the orbit of $J_0$ under $\varphi_1(\Gamma)$. 
	Let also $\Phi\colon \Gamma\to\Aut(\Tbb_G,\treeorder,\prec)$ be the directed tree representation associated with $G$ and $<$, and denote by $\varphi$ the dynamical realization of $\Phi$. Our goal is to show that $\varphi$ and $\varphi_1$ are semi-conjugate. This will be a direct consequence of the following statement.
	\begin{claim}
		The family of intervals $\mathcal I$ is a covering prelamination, which determines a simplicial planar directed tree, order-isomorphic to $(\Tbb_G,\treeorder,\prec)$ via a $\Gamma$-equivariant isomorphism.
	\end{claim}
	\begin{proof}[Proof of claim]
		We consider the map
		\[\dfcn{F}{\Tbb_G}{\mathcal{I}}{g.w_0}{g.J_0}\]
		and then prove that it gives the desired $\Gamma$-equivariant order-isomorphism.
		To simplify notation, given $I_1,I_2\in\mathcal{I}$, we write $I_1<I_2$ if $\sup I_1\le \inf I_2$. Also, given two vertices $v_1,v_2\in\Tbb_G$, we write $v_1\prec v_2$ if for every points $\xi_1\in\partial U_{v_1}$ and $\xi_2\in\partial U_{v_2}$ in the shadows, we have $\xi_1\prec\xi_2$.
		We want to show that the map $F$ is well defined, $\Gamma$-equivariant, and satisfies the following conditions:
		\begin{enumerate}[label=(\roman*)]
			\item\label{condFi} $v_1\treeorder v_2$ implies $F(v_1)\subset F(v_2)$,
			\item\label{condFii} $v_1\prec v_2$ implies $F(v_1)<F(v_2)$
		\end{enumerate}	
		
		To see that $F$ is well defined, we need to check that $\mathsf{Stab}^{\Phi}(w_0)\subseteq \mathsf{Stab}^{\varphi_1}(J_0)$.
		Given $\gamma\in\langle\mathcal{G}\rangle$ consider its height $\mathsf{ht}(\gamma)$ defined as in \eqref{eq.height}. Notice that if $\mathsf{ht}(\gamma)\leq 0$, then the support of $\gamma$ is contained in $\bigcup_{r\leq 0}J_r$, and therefore $\gamma\in\mathsf{Stab}^{\varphi_1}(J_0)$. In order to show the inclusion between the stabilizers, we claim that every $\gamma\in\mathsf{Stab}^{\Phi}(w_0)$ can be written as 
		\begin{equation}\label{eq:decomp_stab}
			\gamma=\gamma_1\gamma_2\quad\text{with }\gamma_1\in \mathsf{Stab}^{\varphi_1}(J_0)\text{ and }\mathsf{ht}({\gamma_2})<\mathsf{ht}({\gamma}).
		\end{equation}
		To obtain a decomposition as in \eqref{eq:decomp_stab}, first note that in the case $\mathsf{ht}(\gamma)\leq 0$ we are done. Suppose it is not the case, and write $\gamma=\gamma_1\cdots\gamma_k$ with 
		\[\max\{\mathsf{ht}(\gamma_j):j\in \{1,\ldots,k\}\}=\mathsf{ht}(\gamma)>0.\]
		Then, since $\gamma.w_0=w_0$ and $0<\mathsf{ht}(\gamma)$, we are in condition to apply Lemma \ref{sublem.technical} to the decomposition $\gamma=\gamma_1\cdots\gamma_k$. Thus, we can write $\gamma=(f_1h_1f_1^{-1})\cdots(f_lh_lf_l^{-1})$ with $f_1,\ldots,f_l\in\langle\mathcal{G}_{\mathsf{ht}(\gamma)}\rangle$ such that $f_i\neq f_j$ for $i\neq j$, and $\mathsf{ht}({h_i})<\mathsf{ht}({\gamma})$ for $i\in \{1,\ldots,l\}$. As discussed above, for every $i\in \{1,\ldots,l\}$ such that $f_i\neq 1_\Gamma$, we have that the support of $\varphi_1(f_ih_if_i^{-1})$ is disjoint from $J_{\mathsf{ht}(\gamma)-1}$ and, as a consequence, disjoint from $J_0$. Thus, if $f_i\neq 1_\Gamma$ for every $i\in\{1,\ldots,l\}$ we are done. Suppose it is not the case, and that for some $i$ we have $f_i=1_\Gamma$. In this case, by applying the commutation relations in $\mathcal{R}_1$ from Lemma \ref{l.relations1}, we can assume  $f_l=1_\Gamma$ and we set $\gamma_1=(f_1h_1f_1^{-1})\cdots(f_{l-1}h_{l-1}f_{l-1}^{-1})$ and $\gamma_2=h_{l}$. 
		Notice that, as we argued in the previous case, the element $\gamma_1$ fixes $J_0$. Finally notice that, by the choice from Lemma \ref{sublem.technical}, we have $\mathsf{ht}({h_l})<\mathsf{ht}(\gamma)$. This gives the desired decomposition as in \eqref{eq:decomp_stab}.
		
		By applying the decomposition as in \eqref{eq:decomp_stab} finitely many times, we get a factorization $\gamma=\delta_1\cdots\delta_r$ where $\delta_i$ is in the stabilizer of $J_0$ for $i\in\{1,\ldots,r-1\}$, and $\mathsf{ht}({\delta_r})\leq 0$. Finally, since $\mathsf{ht}({\delta_r})\leq 0$, we also have that $\delta_r$ is in the stabilizer of $J_0$, and therefore the inclusion between the stabilizers follows. This gives that the map $F$ is well defined, as wanted. Moreover, by definition of $F$, we also have that it is $\Gamma$-equivariant. 
		
		In order to prove condition \ref{condFi}, first recall that $\BN(G)$ acts transitively on the vertices of $\Tbb_G$ (see the claim in the proof of Lemma \ref{l-BN-generators}). Thus, for every vertices $v_1\treeorder v_2$ in $\Tbb_G$, there exists $\gamma\in \Gamma$ such that $\gamma.v_1=w_s$ for some $s\in \Z$, and thus $\gamma.v_2=w_r$ for some $r>s$. Since the partial order $\treeorder$ is preserved by the action $\Phi$, and the inclusion relation is preserved by the action induced by $\varphi_1$ on $\mathcal{I}$, using $\Gamma$-equivariance of $F$ we only need to check that condition \ref{condFi} holds when we take $v_1,v_2$ in the subset $\{w_n:n\in\Z\}$. For this, consider $w_i\treeorder w_j$; then $F(w_i)=J_i\subset J_j=F(w_j)$, as desired. To prove condition \ref{condFii}, first notice that, since condition \ref{condFi} holds, it is enough to check the condition taking $v_1\prec v_2$ adjacent to $v_1\treeup v_2$. Notice that the relation $\prec$ on $\Tbb_G$ is invariant under $\Phi$, and the relation $<$ on $\mathcal{I}$ is invariant under the action induced by $\varphi_1$. Following the same reasoning as for condition \ref{condFi}, it is enough to check condition \ref{condFii} taking $v_1,v_2$ adjacent to $w_0$. In that case, condition \ref{condFi} follows from condition \ref{iv-BrinNavas} in the statement. Summarizing, conditions \ref{condFi} and \ref{condFii} are satisfied by the map $F$, as wanted.
	\end{proof}
	
	After the claim, we have that the directed tree representation $\Phi\colon \Gamma\to \Aut(\Tbb_G,\treeorder,\prec)$ is conjugate to a focal action representing $\varphi_1$ (technically speaking, its positive semi-conjugacy class, as $\varphi_1$ need not be minimal and we may be forced to consider a minimal action semi-conjugate to $\varphi_1$). This implies that $\varphi$ and $\varphi_1$ are positively semi-conjugate.
\end{proof}

\begin{proof}[Proof of Theorem \ref{t.C1nonrigid}] Given any irrational $\alpha\in\R\setminus\Q$, denote by $\tau^\alpha\colon \Z^2\to\homeo_0(\R)$ the action by translations so that $\tau^\alpha((1,0))(x)=x+1$ and $\tau^\alpha((0,1))(x)=x+\alpha$. By the Pixton--Tsuboi examples (see Tsuboi \cite{PixtonTsuboi}), for each $\alpha\in\R\setminus\Q$, there exists an action $\varphi^\alpha\colon \Z^2\to\Diff^1_0(\R)$ such that:
	\begin{itemize}
		\item $\varphi^\alpha$ is supported on $(0,1)$,
		\item the restriction $\varphi^\alpha\restriction_{(0,1)}$ is semi-conjugate to $\tau^\alpha$, and
		\item $\varphi^\alpha\restriction_{(0,1)}$ has an exceptional minimal set $\Lambda_\alpha\subset (0,1)$.
	\end{itemize}
	Then, for any irrational $\alpha\in\R\setminus\Q$, consider an affine expanding homothety $f_\alpha\colon \R\to\R$ with fixed point in $(0, 1)\setminus \Lambda_\alpha$ so that $f_\alpha^{-1}((0,1))\cap \Lambda_\alpha=\emptyset$. Consider the free product $\Z^2\ast\Z$, and denote by $f_0$ a generator of the cyclic factor. Then, we define the action $\varphi^\alpha_0\colon \Z^2\ast\Z\to\Diff^1_0(\R)$ so that $\varphi^\alpha_0$ coincides with $\varphi^\alpha$ on the $\Z^2$-factor and $\varphi^\alpha_0(f_0)=f_\alpha$. It is direct to check that the action $\varphi^\alpha_0$ satisfies conditions \ref{i-BrinNavas} to \ref{iv-BrinNavas} in the statement of Proposition \ref{prop.larguisima}, with $G=\Z^2$ and with $<$ being  the left-invariant order $<_\alpha$ induced by $\tau^\alpha$. Then, applying Proposition \ref{prop.larguisima} we conclude that $\varphi^\alpha_0$ induces an action $\Psi^\alpha\colon \BN(\Z^2)\to \Diff^1_0(\R)$, semi-conjugate to the dynamical realization of the planar directed tree representation associated with $\Z^2\subset\Bij(\Z^2)$ and $<_\alpha$, which is faithful, minimal,  and micro-supported (in particular, any action semi-conjugate to $\Psi^\alpha$ must be faithful). On the other hand, by Proposition \ref{prop.BNfocal}, different orders $<_{\alpha_1}$ and $<_{\alpha_2}$ give rise to planar directed tree representations with non-conjugate dynamical realizations. Therefore, $\Psi^{\alpha_1}$ and $\Psi^{\alpha_2}$ are not semi-conjugate.  
\end{proof}

\chapter{A plethora of laminar actions of Thompson's group $F$}\label{s-F}

Perhaps the most basic example of a finitely generated locally moving group is Thompson's group $F$ (cf.\ Proposition \ref{p-chain}).
Recall that we have defined $F$ in \S \ref{sc.BieriStrebel} as the Bieri--Strebel group $G((0,1);\Z[1/2],\langle 2\rangle_*)$. 
Some examples of laminar actions of $F$ can be obtained by the constructions in \S \ref{s.BSjump} (jump cocycles), \S\ref{ssec.germtype} (orders of germ type revisited), and \S\ref{ssec.cyclicgerm} (groups with cyclic germs at infinity), which already show that $F$ admits an uncountable family of non-conjugate such actions. In this chapter, we illustrate the abundance of laminar  actions  of the group $F$ by providing various other constructions, and describing some subtle differences in their dynamical behavior.  While we will focus more on  some significant examples, the reader will notice that these constructions  admit various variations involving several choices, which depend on the data and are not always compatible between them. Trying to take them all into account simultaneously would result in an obscure treatment.  
This abundance of laminar actions of $F$ may appear surprising when compared with the results in Chapter \ref{s-few-actions}, where we have seen that the Bieri--Strebel group $G(2)$, whose definition is very close to that of $F$, admits exactly \emph{two}  faithful minimal exotic actions up to conjugacy.

As we shall see, an interesting feature shared by many of the constructions that we will provide is that they yield laminar actions of $F$ that are \emph{simplicial} in the sense of Definition \ref{dfn.simplicial_laminar} (that is, preserve a discrete lamination). We will see in \S \ref{s-F-simplicial} various equivalent characterizations of this property, one of them being that the commutator subgroup $[F, F]$ does not admit any minimal invariant set for the action. However, we will also construct in \S \ref{s-F-non-simplicial} a family of actions which remain minimal in restriction of $[F, F]$ (and thus are not simplicial). 

\section{Structure of laminar actions of $F$}

We begin by restating Corollary \ref{c-lm-semiconj} in the special case of $F$. Note that $F$ is well known to be fragmentable (for instance, the generators $a$ and $b$ in the presentation \eqref{eq.presF} can be taken in the subgroups $F_+$ and $F_-$, respectively). Furthermore, the largest quotient of $F$ is its abelianization  $F^{ab}\cong \Z^2$. Thus Corollary \ref{c-lm-semiconj} applies and gives the following structure theorem for actions of $F$.
\begin{thm}\label{t-F-trichotomy}
	Every irreducible action $\varphi\colon F\to \homeo_0(\R)$ is semi-conjugate to one of the following. 
	\begin{itemize}
		\item \emph{(Non-faithful)} An action by translations of  $F^{ab}\cong \Z^2$.
		\item \emph{(Standard)} The standard piecewise linear action of $F$ on $(0,1)$. 
		\item \emph{(Exotic)} A minimal laminar action, horograded by the standard action of $F$ on $(0, 1)$.
	\end{itemize}
\end{thm}
\begin{rem}\label{r.F-laminar-is-faithful}
	Note that minimal laminar actions of $F$ are always faithful, as any proper quotient of $F$ is abelian, and minimal laminar actions cannot be conjugate to any action by translations (for instance, because every element has fixed points, see Lemma \ref{rem.fixedepoints}).
\end{rem}
Theorem \ref{t-F-trichotomy} gives many constraints on the structure of actions of $F$ on $\R$, in terms of the standard action.  In particular, it implies that for all exotic actions, the type of all individual elements of $F$ satisfy a dynamical classification that can be read from the standard action on $(0, 1)$ (see Proposition \ref{p-dyn-class-elements-horograded}). For ease of reference, let us restate this dynamical classification in this special case. For this, recall that given $g\in F$, we denote by $D^{\pm}g(x)$ the (right or left) derivative of $g$ at a point $x\in [0, 1]$, with respect to the standard piecewise linear action.

\begin{prop}[Dynamical classification of elements] \label{p-F-dynclasselements}
	Let $\varphi\colon F\to \homeo_0(\R)$ be a minimal laminar action, positively horograded by the standard action on $(0, 1)$. Then the following hold.
	\begin{itemize}
		\item For every $x\in (0, 1)$, the image $\varphi(F_{(0, x)})$ is totally bounded. In particular, the $\varphi$-image of every element $g\in F$ with $D^-g(1)=1$ is totally bounded.
		\item For every $g\in F$ such that $D^-g(1)\neq 1$ the $\varphi$-image of $g$ is a pseudo-homothety, which is expanding if $D^-g(1)<1$ and contracting otherwise. If moreover $g\in F$ has no fixed points in $(0, 1)$, then its image is a homothety.
	\end{itemize}	
\end{prop}

Let us also fix some notation that will be used throughout the chapter. Recall that the commutator subgroup $[F,F]$ is simple and coincides with the subgroup $F_c$ of compactly supported elements, so that the largest quotient $F/[F_c,F_c]$ coincides with the abelianization $F^{ab}\cong \Germ(F,0)\times \Germ(F, 1)\cong \Z^2$.
We choose the identification \[(\tau_0,\tau_1)\colon F^{ab}\xrightarrow{\sim} \Germ(F,0)\times \Germ(F, 1)\]
obtained 
by identifying the groups of germs $\Germ(F,0)$ and $\Germ(F, 1)$ with $\Z$, with the convention that $\tau_x(g)>0$ if and only if the corresponding endpoint $x\in\{0, 1\}$ is  an attracting fixed point of $g$. Explicitly,
\begin{equation} \label{e-F-germs0}\tau_0(g)=-\log_2 D^+g(0) \quad \text{and}\quad \tau_1(g)=-\log_2 D^-g(1).\end{equation}
In addition, we will denote by $f$ the element of the generating pair of $F$ given by 
\begin{equation} \label{e-F-big-generator0}f(x)=\left\{\begin{array}{lr}2x & x\in [0, \frac{1}{4}],\\[.5em] x+ \frac{1}{4} & x\in [\frac{1}{4}, \frac{1}{2}],\\[.5em] \frac{1}{2}x & x\in [\frac{1}{2}, 1]. 
	\end{array}\right.\end{equation}
We will also write $1_F$ for the trivial element of $F$ (and we will simply denote it by $1$ when there is no risk of confusion).
Since $\tau_1(f)$ is a generator of $\Germ(F, 1)$, we have a  splitting
\[F=F_+\rtimes\langle f \rangle.\]

\begin{rem}
	A common feature of all our constructions of laminar actions of $F$ is the choice of a closed subset $K\subseteq (0,1)$, which is invariant under the 
	element $f$ defined in \eqref{e-F-big-generator0}.
	These sets appear quite naturally with the point of view of focal actions. To understand this, take  a minimal laminar action $\varphi\colon F\to \homeo_0(\R)$. By Theorem \ref{t-F-trichotomy}, we know that it can be horograded by the standard action on $(0, 1)$, meaning that one can find a focal action $\Phi\colon G\to\Aut(\Tbb, \treeorder, \prec)$ on a directed planar tree, horograded by the standard action, and whose dynamical realization is (conjugate to) $\varphi$.
	By Proposition \ref{p-F-dynclasselements}, the element $f$ fixes a unique end $\xi_0\in \partial^*\Tbb$, so that it preserves the axis $]\xi_0, \omega[\subset \Tbb$, which is naturally identified with the interval $(0,1)$ via the horograding map $\hor\colon \Tbb\to (0,1)$. In particular, the $\hor$-image of the closure $\overline{]\xi_0, \omega[\cap \Br(\Tbb)}$ of the subset of branching points on this axis defines an $\varphi(f)$-invariant closed subset $K\subseteq (0,1)$. Although for some choices of the action $\Phi$, the subset $K$ can be the whole interval $(0,1)$,  it is not the case in most examples with correct choice of $\Phi$.
\end{rem}

\section[A plethora of laminar actions, I]{A plethora of laminar actions, I: restriction preorders} \label{s-F-plethora1}

Starting from now, we will present various constructions of laminar actions of the group $F$ and study some of their properties.  
\subsection{A reinterpretation in terms of preorders}
Recall that we write $\tau_1\colon F\to \Z \cong \Germ (F, 1)$  for the germ homomorphism given by \eqref{e-F-germs0}, and consider the element $f$ of the standard generating pair of $F$ given by \eqref{e-F-big-generator0}.
Recall also the splitting $F=F_+\rtimes\langle f \rangle$.
As in \S \ref{ssec.germtype}, we can make $F$ act on $F_+$ ``affinely'', by letting $F_+$ act on itself by left translations, and $f$ act on $F_+$ by conjugation. As in \eqref{e-affine-action}, for $g=hf^n\in F$, with $h\in F_+$ and $n\in \Z$ and for $r\in F_+$, this action is given by $g\cdot r=hf^nrf^{-n}$.

Assume that $\preceq$ is an \emph{$f$-invariant preorder} on $F_+$, that is, a left-invariant preorder on $F_+$ which is also invariant under conjugation by $f$. In particular its residue $H=[1]_\preceq$ is normalized by $f$, so that the action of $F$ on $F_+$ descends to an order-preserving action on $(F_+/H, \prec)$, where $\prec$ is the total order induced by $\preceq$. Then we can consider the dynamical realization $\varphi\colon F\to \homeo_0(\R)$ of this action. We have the following equivalence.

\begin{prop} \label{p-F-focal-plo}
	Let $\varphi\colon F\to \homeo_+(\R)$ be an action. The following are equivalent.
	\begin{enumerate}[label=(\roman*)]
		\item \label{i-F-plo-focal} $\varphi$ is a minimal laminar action, horograded by the standard action on $(0, 1)$. 
		
		\item  \label{i-F-plo-preorder} There exists an $f$-invariant preorder $\preceq$ on $F_+$ such that, writing $H=[1]_\preceq$, the map $f$ acts as a homothety on $(F_+/H,\prec)$, and $\varphi$ is conjugate to the dynamical realization of the action of $F$  on $(F_+/H, \prec)$.
	\end{enumerate}
	Moreover, two distinct preorders as in \ref{i-F-plo-preorder} give rise to (positively) non-conjugate minimal laminar actions.
\end{prop}
\begin{proof}
	Let us prove that \ref{i-F-plo-preorder} implies \ref{i-F-plo-focal}. Assume that $\preceq$ verifies the conditions, and let $\varphi$ be the dynamical realization of the action of $F$ on $(F_+/H, \prec)$. Since $f$ is a homothety on $(F_+/H, \prec)$,  Proposition \ref{p.minimalitycriteria} implies that $\varphi$ is minimal. Furthermore, the fact that $f$ is a homothety on  $(F_+/H, \prec)$ implies that $\varphi(f)$ is a homothety. As $\varphi$ must be described by one of the cases of Theorem \ref{t-F-trichotomy}, the only possibility is that $\varphi$ is laminar, horograded by the standard action on $(0, 1)$. 
	
	For the converse, let $\varphi$ be as in \ref{i-F-plo-focal}. If so, then Proposition \ref{p-F-dynclasselements} gives that $\varphi(f)$ is a homothety; let $\xi\in \R$ be its unique fixed point, and consider the preorder $\preceq$ on $F_+$ associated with this point: $g\precneq h$ if and only if $g.\xi< h.\xi$. Using that $\xi$ is fixed by $f$, we see that $\preceq$ is invariant under conjugation by $f$, and that the natural action of $F$ on $(F_+/[1]_\preceq, \prec)$ can be identified with the action of $F$ on the orbit of $\xi$, showing the claim. 
	
	Finally note that these two constructions are inverse to each other, and since $\xi$ is the unique fixed point of $f$, the preorder $\prec$ is  uniquely determined by the (positive) conjugacy class of the action. 
\end{proof}

\subsection{Restriction preorders on $F_+$}\label{s-restriction-preorder}
Let us describe a concrete construction of preorders on $F_+$ satisfying \ref{i-F-plo-preorder} in Proposition \ref{p-F-focal-plo}. This yields a family of laminar actions of $F$ which contains as special cases the constructions in \S\ref{ssec.germtype} and \S\ref{ssec.cyclicgerm}.

Let $K\subseteq (0, 1)$ be a closed subset. We consider a preorder $\preceq^{K}$ on $F_+$ which is obtained by looking at the restriction of elements of $F_+$ to $K$, as follows.
We first consider the subgroup $H=\{g\in F_+: g(x)=x\text{ for every }x\in K\}$, and for $g\in F_+$ define
\[
x_g=\left\{
\begin{array}{lr}
	0&\text{if }g\in H,\\
	\sup \{x\in K: g(x)\neq x\}&\text{if }g\in F_+\setminus H.
\end{array}
\right.
\]
We immediately observe that $x_g=x_{g^{-1}}$ for every $g\in F_+$. Moreover, we have the following behavior when considering compositions.
\begin{lem}\label{l.xg-restriction} Let $K\subseteq (0,1)$ be a non-empty closed subset, and take $g,h\in F_+$. Then we have the inequality $x_{gh}\le \max\{x_g,x_h\}$, and when $x_h\neq x_g$ the equality $x_{gh}= \max\{x_g,x_h\}$ holds.
\end{lem}
\begin{proof}
	Note that if  $x\in K$ is such that $x>\max\{x_g,x_h\}$, then $gh(x)=g(x)=x$. This gives the inequality $x_{gh}\le \max\{x_g,x_h\}$.
	
	Assume now $x_h\neq x_g$. Since $x_{gh}=x_{h^{-1}g^{-1}}$, upon replacing the pair $(g, h)$ with $(h^{-1}, g^{-1})$, we can assume that $x_h< x_g$. Assume first we are in the case $g(x_g)\neq x_g$. Then $gh(x_g)=g(x_g)\neq x_g$, proving that $x_g\le x_{gh}$, hence $x_{gh}=x_g$ (using the previous inequality). When
	$g(x_g)=x_g$, then $x_g$ is accumulated from the left by points of $K$ which are moved by $g$; in particular for every such point $x$ with $x_h<x<x_g$, we have $gh(x)=g(x)\neq x$, giving $x\le x_{gh}$. Taking the supremum we obtain the desired equality $x_g=x_{gh}$.
	Note also that the same assumption $x_h<x_g$ (which is equivalent to $x_{h^{-1}}<x_{g^{-1}}$) gives $x_{g^{-1}h^{-1}}=x_{g^{-1}}=x_{g}$.	
	As $x_{hg}=x_{g^{-1}h^{-1}}$, we deduce from the previous case that $x_{hg}=x_g$. This concludes the proof.
\end{proof}

We next introduce the subset
\begin{equation}\label{eq.cone_restriction}
	P_K=\left \{g\in F_+\setminus H: \text{either }g(x_g)>x_g,\text{ or }g(x_g)=x_g\text{ and }D^-g(x_g)>1\right \}
\end{equation}
and observe the following.

\begin{lem}\label{l.restriction_cone}
	For any  non-empty closed subset $K\subseteq (0,1)$, the subset $P_K$ defines a positive cone in $F_+$. 
\end{lem}
\begin{proof}
	We have to verify the conditions in Remark \ref{r.cones}. Let us first prove that $F_+=P_K\sqcup H\sqcup P_K^{-1}$. For this notice that, since $x_g=x_{g^{-1}}$, we have \[P_K^{-1}=\{g\in F_+\setminus H: \text{either }g(x_g)<x_g,\text{ or }g(x_g)=x_g\text{ and }D^-g(x_g)<1\}.\]
	Thus, we automatically get that $H\cap \big(P_K\cup P_K^{-1}\big)=\emptyset$ and $P_K\cap P_K^{-1}=\emptyset$. It only remains to show that $F_+\subseteq P_K\sqcup H\sqcup P_K^{-1}$. For this, take $g\in F_+\setminus H$, so that $x_g>0$. If $x_g\neq g(x_g)$ we are done. In the complementary case, $x_g$ must be accumulated from the left by points that are moved by $g$. Since $g$ is piecewise linear we must have $D^-g(x_g)\neq 1$, showing that $g\in P_K\sqcup P_K^{-1}$.
	Next, let us check that $P_K$ is a semigroup and $HP_KH\subseteq P_K$.
	
	Take $g,h\in P_K$, and assume first $x_h< x_g$. Then Lemma \ref{l.xg-restriction} gives $x_{gh}=x_g$ and
	$gh(x_{gh})= g(x_g)$. If $g(x_g)>x_g$, we deduce immediately $gh\in P_K$; otherwise $x_g$ is accumulated from the left by points of $K$, which must be fixed by $h$, so that $D^-h(x_g)=1$. Then $D^-(gh)(x_{gh})=D^-g(x_g)\,D^-h(x_g)>1$, and we conclude that $gh\in P_K$.
	
	Assume now that $x_g<x_h$, so that $x_{gh}=x_h$ by Lemma \ref{l.xg-restriction}. 
	Consider first the case $h(x_h)=x_h$.  	
	Then $gh(x_{gh}) =g(x_h)=x_h=x_{gh}$, and as in the previous case 
	we see that $D^-g(x_h)=1$, so that
	$D^-(gh)(x_{gh})=D^-h(x_h)>1$. If $h(x_h)>x_h$, then
	$gh(x_{gh})>g(x_h)=x_h$. In both cases we have $gh\in P_K$.
	
	Note that the previous argument works also when one of the two elements is in the residue $H$, proving that $HP_KH\subseteq P_K$.
	
	Finally, consider the case $x_g=x_h$. As $h(x_h)\ge x_h$ and $g(x_g)\ge x_g$, then if any of the two inequalities is strict,	 we deduce $gh(x_{g})>x_{g}$, and thus $x_{gh}=x_g$ (by the inequality of Lemma \ref{l.xg-restriction}) and $gh\in P_K$. Otherwise, assume that both $g$ and $h$ fix $x_g=x_h$. Then we have the relation $D^-(gh)(x_{gh})=D^-g(x_g)\,D^-h(x_h)>1$, showing that $x_{gh}=x_g$ (again by Lemma \ref{l.xg-restriction}) and $gh\in P_K$ also in this case.
\end{proof}

The previous lemma leads to the following definition.

\begin{dfn}
	Given a closed subset $K\subseteq(0, 1)$, the preorder $\preceq^K$ on $F_+$ defined by the positive cone $P_K$ in  \eqref{eq.cone_restriction} will be called the \emph{restriction preorder} associated with $K$. We will always write $H=[1]_{\preceq^K}$ for its residue.
\end{dfn}

Let us describe some elementary properties related to the preorder $\preceq^K$ that will be useful in the sequel.

\begin{lem}\label{lem.Lx-convex}
	Let $K\subseteq (0,1)$ be a non-empty closed subset, and let $\preceq^K$ be the corresponding restriction preorder on $F_+$. Then the following hold.
	\begin{enumerate}[label=(\roman*)]
		\item \label{ii.Lx-convex} For $g,h\in F_+$ with $1\preceq^K g\preceq^K h$, we have $x_g\le x_h$.
		\item\label{iii.Lx-convex} For $x\in (0,1)$, the subset $L_x:=\{g\in F_+: x_g\le x\}$ is a $\preceq^K$-convex subgroup.
	\end{enumerate}
\end{lem}
\begin{proof}
	We first prove \ref{ii.Lx-convex}. We can assume $g\in P_K$, otherwise $x_g=0$ and the result follows.
	Assume for contradiction that $x_g>x_h$. Then from Lemma \ref{l.xg-restriction} we have $x_{g^{-1}h}=x_g$. Consider first the case $g(x_g)>x_g$, then $g^{-1}h(x_{g})=g^{-1}(x_g)<x_g$, so that $g^{-1}h\precneq^K 1_F$, contradicting the assumption $g\preceq^K h$. Consider next the case $g(x_g)=x_g$, so that $D^-g(x_g)>1$ and $D^-h(x_g)=1$ (as in this case, $x_g$ is accumulated from the left by points of $K$). Then $g^{-1}h(x_g)=x_g$ and $D^-(g^{-1}h)(x_g)=D^-g(x_g)^{-1}<1$, giving again the contradiction $g^{-1}h\precneq^K 1_F$.
	
	The inequality $x_{gh}\le \max\{x_g,x_h\}$ from Lemma \ref{l.xg-restriction} shows that the subset $L_x$ in \ref{iii.Lx-convex} is a subgroup, whilst \ref{ii.Lx-convex} proves that $L_x$ is $\preceq^K$-convex.
\end{proof}

Note that  the coset space $F_+/H$ can be identified with the set of restrictions 
$\{g\restriction_K : g\in F_+\}$, 
so that two elements $g, h\in F_+$ are equivalent for $\preceq^K$ if and only if their restrictions to $K$ coincide. 

\begin{lem}
	Suppose that the closed subset $K\subseteq (0, 1)$ is $f$-invariant. Then the restriction preorder $\preceq^K$ on $F_+$ is $f$-invariant, and the conjugacy induces a  homothety on $(F_+/H, \prec^K)$ fixing $H$. 
\end{lem}  

\begin{proof}
	The verification that $\preceq^K$ is $f$-invariant follows easily from $f$-invariance of $K$. Indeed, it is clear that it fixes the point corresponding to $H$.
	We next verify that conjugation by $f$ preserves the positive cone $P_K$. Take $h\in P_K$, write $x_*=x_h$ and note that $f(x_h)=x_{fhf^{-1}}$. When $h(x_h)>x_h$, we have $fhf^{-1}(x_{fhf^{-1}})=fh(x_h)>x_h$; otherwise, we have $h(x_{h})=x_{h}$ and \[D^-(fhf^{-1})(x_{fhf^{-1}})=D^-h(x_h)>1.\] Hence $fhf^{-1}\in P_K$, as wanted.
	
	More generally, for $n\in \N$, consider $h_n=f^nhf^{-n}$  and observe that the point $x_{h_n}=f^n(x_h)$ tends to $1$ as $n\to \infty$.
	Take $r\in P_K$, and let $y\in (0, 1)$ be such that $r$ acts trivially on $(y, 1)$. If $n$ is large enough so that $h_n(x_{h_n})=f^nh(x_h)$ and $x_{h_n}$ are both greater than $y$, we have that $x_{r^{-1}h_n}=x_{h_n}$ and $r^{-1}h_n$ coincides with $h_n$ on a neighborhood of $x_{h_n}$. Since $h_n\in P_K$, and this depends only on the behavior of $h_n$ on some neighborhood of $x_{h_n}$, we must have $r^{-1}h_n\in P_K$ for $n$ large enough, and thus $h_n\succneq^K r$. Since $h$ and $r$ were arbitrary $\preceq^K$-positive elements and we can repeat the same reasoning for arbitrary $h,r\in P_K^{-1}$, this shows that the conjugation by $f$ is a homothety. 
\end{proof}

\begin{prop}
	Given a non-empty $f$-invariant closed subset $K\subseteq (0, 1)$,  denote by $\psi_K\colon F\to \homeo_0(\R)$ the dynamical realization of the action of $F$ on $(F_+/H, \prec^K)$ defined above. Then $\psi_K$ is a minimal laminar action, positively horograded by the standard action on $(0, 1)$. Moreover, for two $f$-invariant closed subsets $K_1\neq K_2$, the actions $\psi_{K_1}$ and $\psi_{K_2}$ are not conjugate. 	
\end{prop}

\begin{proof}
	By Proposition \ref{p-F-focal-plo} $\psi_K$  is laminar, positively horograded by the standard action on $(0, 1)$. Note also that since the residue $H$ is the pointwise stabilizer of $K$, and two distinct closed subsets of $(0, 1)$ have different pointwise stabillizers, $\preceq^K$ determines $K$ completely. In particular, by the last part of Proposition \ref{p-F-focal-plo}, when $K_1\neq K_2$, we have that their associated actions $\psi_{K_1}$ and $\psi_{K_2}$ are not conjugate.
\end{proof}

\subsection{Some properties of the actions arising from restriction preorders}
Given  a non-empty $f$-invariant  closed subset $K\subseteq (0, 1)$, we  keep denoting by $\psi_K\colon F\to \homeo_{0}(\R)$ the action constructed above. We want to point out some dynamical features of this family of actions. 
Recall that a minimal action of a group $G$ on a locally compact space $Y$ is \emph{topologically free} if the set of fixed points $\fix(g)$ has empty interior for every $g\in G$. By Baire's theorem, this is equivalent to the requirement that there is a $G_\delta$-dense set of points in $Y$ with trivial stabilizer in $G$. 

\begin{prop}[Freeness and non-freeness] \label{p-F-restriction-topfree}
	Let $K\subseteq (0, 1)$ be a non-empty  $f$-invariant closed subset. Then, the laminar action $\psi_K\colon F\to \homeo_0(\R)$ defined above is topologically free  if and only if $K=(0, 1)$. In particular, $F$ admits both topologically free and  non-topologically free minimal laminar actions.
\end{prop}

\begin{proof}
	Assume $K=(0, 1)$. We claim that the action $\psi:=\psi_{K}$ is topologically free. Indeed, in this case,   the preorder $\preceq^{K}$ is actually a total order on $F_+$. Thus, there is a dense subset of points in $\R$ with trivial stabilizer for $\psi(F_+)$, which implies that the action of $F_+$ is topologically free. Assume by contradiction that $g\in F$ is such that $\fix^\psi(g)$ has non-empty interior, and let $I$ be a connected component of its interior.  Note that $g\notin F_+$, so that by Propositions \ref{p-F-focal-plo} and \ref{p-F-dynclasselements}, the image $\psi(g)$ must be a pseudo-homothety; in particular $I$ is bounded. As the action $\psi$ is proximal (see for instance Proposition \ref{prop.minimalimpliesfocal}), there exists $h\in F$ such that $\psi(h)(I)\Subset I$. Then it is not difficult to see that the commutator $[g, h]=ghg^{-1}h^{-1}$ is non-trivial, belongs to $F_+$, and fixes $\psi(h)(I)$ pointwise. This is a contradiction since we have already shown that the action of $F_+$ is topologically free. 
	
	Now consider the case $K\neq (0, 1)$.  We can take a connected component $U=(y,z)$ of the complement $(0,1)\setminus K$, and consider a non-trivial element $h\in F_+$ whose support is  contained in $U$. Fix $x<y$ and consider the $\preceq^K$-convex  subgroup $L_x$ from Lemma \ref{lem.Lx-convex}. Take an element $g\in L_x$, and let us prove that the conjugate $g^{-1}hg$ belongs to $H$. For this, note that the condition $x_g<x$ implies $g^{-1}(U)=U$, so that the restriction of $g^{-1}hg$ to the complement $(0,1)\setminus U$ is trivial. This immediately implies that $g^{-1}hg$ fixes every point of $K$, so that $g^{-1}hg$ belongs to the residue $H$. This proves that $hgH=gH$ for any element $g\in L_x$, so that the element $h$ fixes the $\prec^K$-convex subset $L_x/H$ pointwise.  We deduce that $\psi_K(h)$ fixes a non-empty open interval. Moreover, $\psi_K(h)$ is non-trivial because laminar actions of $F$ are always faithful (see Remark \ref{r.F-laminar-is-faithful}).
\end{proof}

\begin{rem}
	Proposition \ref{p-F-restriction-topfree} should be compared with the fact that many groups arising via a micro-supported action by homeomorphisms satisfy  rigidity results for their non-topologically free actions on \emph{compact} spaces, as shown in the works of Le Boudec and the second author \cite{LBMB-sub, bon2018rigidity, boudec2020commutator} using results on uniformly recurrent subgroups and confined subgroups.
	As an example tightly related to this setting, consider  Thompson's group $F$ and  its sibling $T$ acting on the circle. Then, every minimal action of $T$ on any compact space is either topologically free, or factors to its standard action on the circle, while every faithful minimal action of $F$ on a compact space is topologically free  \cite{LBMB-sub}.
	Proposition \ref{p-F-restriction-topfree} shows that actions on the line behave very differently from this perspective, and the notion of topological freeness is much less relevant.
\end{rem}
Another feature of this family of actions is the following. Recall that a minimal laminar action is \emph{simplicial} if it preserves a discrete lamination (see Definition \ref{dfn.simplicial_laminar} and Proposition \ref{p-focal-simplicial}).
\begin{prop}[Simpliciality]  \label{p-F-restriction-not-minimal}Let $K\subset (0,1)$ be a non-empty $f$-invariant closed subset, and consider the corresponding action $\psi_K\colon F\to \homeo_0(\R)$, as constructed above. Then, the image of $F_+$ does not act minimally on $\R$. In particular, every action $\psi_K$ is simplicial. 
\end{prop}

\begin{proof}
	Fix $x\in (0,1)$ and consider the $\preceq^K$-convex subgroup  $L_x=\{g\in F_+: x_g\le x\}$ (Lemma \ref{lem.Lx-convex}). In the dynamical realization $\psi_K$ of the action $F\to\Aut\left (F_+/H, \prec^K\right )$, the cosets of $L_x$ span a $\psi_K(F_+)$-invariant family of disjoint open intervals, showing that the $\psi_K$-action of $F_+$ is not minimal, and thus it does not admit any minimal invariant set (Lemma \ref{l-normal-minimal}). After Proposition \ref{p-focal-simplicial}, this is equivalent to $\psi_K$ being simplicial.
\end{proof}

One way to analyze finer properties of laminar actions of the group $F$ is to apply Theorem \ref{t-F-trichotomy} inductively, by exploiting the self-similarity  of $F$. Namely, assume that $\varphi\colon F\to \homeo_+(\R)$ is a minimal laminar action, positively horograded by the standard action on $(0, 1)$. Recall that for every dyadic $x\in (0,1)$, the group $F_{(0,x)}$ is isomorphic to $F$, and its image under $\varphi$ is totally bounded (that is, $\fixphi(F_{(0, x)})$ accumulates on both $\pm \infty$). Thus, we can apply  Theorem \ref{t-F-trichotomy} to the action of $F_{(0, x)}$ on every connected component $J$ of $\suppphi(F_{(0, x)})$. It follows that this action still falls into one of the three cases up to semi-conjugacy: action by translations, the standard action, and laminar actions. In the third case, this analysis can of course be iterated. We will speak of ``sublevels'' of the action $\varphi$ to refer to the actions of the subgroups $F_{(0, x)}$ obtained in this way. From this point of view, the actions $\psi_{K}$ arising from restriction preorders  are very special: indeed they are not  exotic on any sublevel (in contrast with other laminar actions of $F$; see Proposition \ref{p-F-CB-sublevels} below). The proof of this fact is relatively short when using the notion of domination between preorders that will be introduced later in Definition \ref{d-preorder-dominates}, and in particular the fact that if one preorder dominates another one, then their dynamical realizations are positively semi-conjugate (Lemma \ref{lem.domsemicon}).

\begin{prop}[Absence of exotic sublevels] \label{p-restriction-sublevels}
	Let $K\subseteq (0,1)$ be a non-empty $f$-invariant closed subset, and consider the corresponding action $\psi_K\colon F\to \homeo_0(\R)$, as constructed above. Let $x\in X$ be a dyadic point, and $J$ a connected component of $\supp^{\psi}(F_{(0, x)})$. Then, the $\psi$-action of $F_{(0, x)}$ on $J$   is    semi-conjugate either to its standard action on $(0, x)$,  or to a cyclic action by translations induced from the group of germs $\Germ(F_{(0, x)}, x)\cong \Z$.	
\end{prop}

\begin{proof}
	Let $\xi_0$ be the unique fixed point of $\psi(f)$. Let us first show the claim for the action of $F_{(0, x)}$ on $J=\Ipsi(x, \xi_0)$ (the connected component of $\supp^{\psi}(F_{(0,x)})$ containing $\xi_0$). The semi-conjugacy type of this action is determined by the preorder $\preceq_{\xi_0}\in\LPO(F_{(0,x)})$ induced by the point $\xi_0$ on $F_{(0, x)}$, which coincides with the restriction of $\preceq^{K}$ to $F_{(0, x)}$. Now we distinguish two cases. 
	
	First assume that $K\cap (0, x)$ does not accumulate on $x$. Write $y=\sup \{K\cap (0, x)\} < x$, and let $\preceq_y\in\LPO(F_{(0,x)})$ be its induced preorder on $F_{(0, x)}$. If $g\in F_{(0,x)}$ is such that $g(y)> y$, then $g\in P_K$, and by Lemma \ref{lem.dominequiv}, this is equivalent to the fact that $\preceq_{y}$ dominates $\preceq_{\xi_0}$. Then Lemma \ref{lem.domsemicon} gives that the dynamical realizations of the two preorders are positively semi-conjugate; as the dynamical realization of $\preceq_y$ is the standard action of $F_{(0,x)}$, the conclusion follows in this case.
	
	Assume now that $\sup \{K\cap (0, x)\}= x$. In this case, by definition of $\preceq^K$, we get that $\preceq_{\xi_0}$ is dominated by a preorder  obtained as the pull-back of one of the two non-trivial preorders on $\Germ(F_{(0, x)}, x)\cong \Z$. This shows the conclusion for $\xi=\xi_0$.
	
	If now $\xi\in \supp^\psi(F_{(0, x)})$ is arbitrary, then by minimality we can choose $h\in F$ such that $\psi(h)(\xi_0)\in \Ipsi(x, \xi)$. Then, the conclusion follows from the previous case applied to the action of $F_{(0, h^{-1}(x))}=h^{-1} F_{(0, x)}h$ on $\Ipsi(h^{-1}(x), \xi_0)$. \qedhere
\end{proof}

\subsection{Some variations on the restriction preorder construction}
The restriction preorder construction can be modified in multiple ways to produce new families of minimal laminar actions, which are not conjugate to the actions $\psi_K$ defined above. We indicate some of them, without detailed exploration nor attempt to include them all in a unified family.

\begin{enumerate}[leftmargin=*]
	\item \emph{Twisting with sign choices.} In addition to the $f$-invariant subset $K\subseteq (0, 1)$, consider an $f$-invariant choice of signs $u\colon K\to  \{+1, -1\}$. We proceed to define a preorder $\preceq^{(K, u)}$ on $F_+$. For this, given $g\in F_+$ we say that $g\succneq^{(K,u)}1_F$ if either $u(x_g)=1$ and $g\succneq^K1_F$, or $u(x_g)=-1$ and $g\precneq^K1_F$.  It is direct to check (following the proof of Lemma \ref{l.restriction_cone}) that $\preceq^{(K,u)}$ is an invariant preorder on $F_+$ and that $f$-invariance of $u$ makes $\preceq^{(K,u)}$ invariant under conjugation by $f$. Of course, when $u\equiv 1$ the preorders $\preceq^{(K,u)}$ and $\preceq^K$ coincide. 
	There are some straightforward variations to this twist. For instance, one may  consider two different $f$-invariant functions $u, v\colon K\to \{\pm 1\}$ to determine the sign in the two different cases $g(x_g)\neq x_g$ and $g(x_g)=x_g$.

	\item \emph{Twisting with derivative morphisms.} In this case, in addition to the $f$-invariant set $K\subseteq (0, 1)$, consider a left-invariant order $<_0$ on the abelian group \[A=\{(2^n, 2^m): n, m\in \Z\}\cong \Z^2\]
	(note that $A$ can be though as the set of derivatives that an element of $F$ can take at a dyadic point). As before, we will define an $f$-invariant preorder on $F_+$. For this, consider a  different definition of $x_g$, namely define 
	\[x'_g:=\sup\left \{x\in K:g(x)\neq x, \text{ or }  g(x)=x \text{ and }(D^-g(x),D^+g(x))\neq (1,1)\right \}.\] Then, set $\preceq^{K}_0\in\LPO(F_+)$ so that $g\succneq^{K}_0 1_F$ if either $g(x'_g)>x'_g$, or $g(x'_g)=x'_g$ and $(D^-g(x'_g),D^+g(x'_g))>_0(1,1)$. Again, it is straightforward to check (following the proof of Lemma \ref{l.restriction_cone}) that the preorder $\preceq^K_0$ is $f$-invariant.
	
	To compare these preorders with the preorders of the form $\preceq^{(K,u)}$, consider $p\in(0,1)\cap\Z[1/2]$ and the closed subset $K_p=\{f^n(p):n\in\Z\}$. In this case, all the twists $\preceq^{(K_p,u)}$ given by sign choices coincide with $\preceq^{K_p}$, while the preorder $\preceq^{K_p}_0$ just defined does not.
	
	\item \emph{Twisting with new orderings of $(0,1)$.} In the construction of the preorder $\preceq^K$ one can modify the definition of the point $x_g$ by taking the supremum with respect to an order $\prec_0$ on $K$ which is different from the order induced from the embedding $K\subseteq (0, 1)$. The whole construction will still be well defined, provided $\prec_0$ is $f$-invariant and satisfies suitable assumptions, which are not difficult to figure out, but are rather technical to state. Instead of discussing this in general,  let us give an example. 
	
	Take $0<x_0<p_1<p_2<f(x_0)<1$, and define $K$ as the union of the orbits of $p_1$ and $p_2$. Then, we define the total order $\prec_0$ on $K$ so that $f^m(p_i)\prec_0 f^n(p_j)$ if either $m+i<n+j$, or $m+i=n+j$, $i=1$ and $j=2$. More explicitly, we have
	\[\cdots \prec_0 f^{-2}(p_2)\prec_0 p_1\prec_0 f^{-1}(p_2)\prec_0 f(p_1)\prec_0\cdots.\]
	It is clear that $\prec_0$ is $f$-invariant. We can then define a preorder $\preceq^{K, \prec_0}$ in the same way as the restriction preorder $\preceq^K$, except that we replace the point $x_g$ by the point $x''_g$ consisting of the $\prec_0$-greatest element of the subset $\{x\in K:g(x)\neq x \}$. It is straightforward to check that $\preceq^{K,\prec_0}$ is an $f$-invariant preorder, inducing an order-preserving action $F\to \Aut \left (F_+/[1]_{\preceq^{K,\prec_0}},\prec^{K,\prec_0}\right )$ as above. Denote by $\Psi_0=\Psi_{K,\prec_0}$ the dynamical realization of this action, and assume that its associated good embedding satisfies $\iota([1]_{\preceq^{K,\prec_0}})=0$. It can be shown that different choices of $p_1$ and $p_2$ produce non-conjugate actions. On the other hand, the interested reader can check that the semi-conjugacy classes of the sublevels $F_{(0,x)}\curvearrowright \mathrm I^{\Psi_0}(x,0)$ only depend on the choice of $p_2$, but not  of $p_1$. This shows that exotic actions cannot be reconstructed with the information of the semi-conjugacy classes of its sublevels (as defined in Proposition \ref{p-restriction-sublevels}). 
	
	Again there are some obvious variations of this, such as considering preorders on $A$ instead of orders, and modifying the definition of the point $x'_g$ accordingly. 
	\end{enumerate}
	
Of course one can consider appropriate combinations of the variants defined above.  However, whether such combinations  make sense or not, depends on the choice of the parameters, and a unified treatment would be obscure and pointless. All constructions obtained using these methods yield simplicial actions.

\section[A plethora of laminar actions, II]{A plethora of laminar actions, II: ordering the orbit of a closed subset of $(0, 1)$} \label{s-F-orbit-construction}
We now describe another method to construct laminar actions of $F$. The starting ingredient of this method is again a non-empty closed subset $K\subseteq (0, 1)$ which is invariant under the generator $f$ given by \eqref{e-F-big-generator0}. We assume now $K\neq (0, 1)$,  and consider the $F$-orbit of $K$ among closed subsets of $(0, 1)$, and denote it as \[\mathcal{O}_K:=\{g(K) : g\in F\}.\]
As $K\subset (0,1)$ is a proper subset, we clearly have that the orbit $\Ocal_K$ is infinite.
The natural attempt is to try to define an $F$-invariant order on $\mathcal{O}_K$, and then consider its dynamical realization. While this may seem similar to the construction just discussed in \S \ref{s-restriction-preorder}, it turns out to be quite different, and it produces actions with more exotic dynamical properties. Note that we are not aware of any general receipt to build $F$-invariant orders on $\mathcal{O}_K$ which works \emph{for all} $K$:  the way such orders arise  depend subtly on the properties of the subset $K$. 
However, what is true is that for any $K$ we can consider a natural focal action on a directed tree
\[
\Phi\colon F\to \Aut(\Tbb_K,\treeorder),
\]
horograded by the standard action of $F$ on $(0,1)$, together with an $F$-equivariant injective map $i\colon \mathcal O_K\to \partial^*\Tbb_K$, so that the problem is reduced to find a $\Phi$-invariant planar order on $(\Tbb_K,\treeorder)$, which is in principle easier than finding a general $F$-invariant order on $\mathcal O_K$.
We will first detail this strategy in general, and then illustrate it in practice with a  concrete choice of a subset $K$ (there are examples of subsets $K$ for which this strategy cannot work, see Example \ref{rem.nonplanar}). More examples of actions obtained using this method will appear later in \S \ref{s-F-hyperexotic}. 

\subsection{A strategy to order $\mathcal{O}_K$} \label{s-F-orbit-strategy} Assume that $K\subsetneq (0, 1)$ is an $f$-invariant closed subset.
Since the germ of $f$ at $1$ generates the group of germs $\Germ(F, 1)\cong \Z$ and $K$ is $f$-invariant, it follows that every $K_1=g(K)\in \mathcal{O}_K$ must coincide with $K$ on an interval of the form  $(1-\varepsilon, 1)$, with $\varepsilon>0$.   Thus, it follows that any two distinct subsets $K_1, K_2\in \mathcal{O}_K$ coincide on some interval of the form $(1-\varepsilon, 1)$, so that we can define 
\begin{equation} \label{e-alpha-K1-K2}\alpha(K_1, K_2)=\inf\{x\in (0,1): K_1\cap[x, 1)=K_2\cap[x, 1)\}.\end{equation} 
As $K$ is closed, we have $\alpha(K_1,K_2)\in K_1\cap K_2$; when $K_1=K_2$, we declare $\alpha(K_1,K_2)=0$.
Moreover, in light of the previous discussion, we get that $\alpha(K_1,K_2)<1$ for every $K_1,K_2\in\mathcal{O}_K$.
It is clear from the definition that for every $K_1,K_2,K_3\in \Ocal_K$ with $\alpha(K_1,K_3)\le\alpha(K_2,K_3)$,  we have $\alpha(K_1,K_2)\le\alpha(K_2,K_3)$ (indeed, when $K_2\neq K_3$, the three intersections $K_i\cap [\alpha(K_2,K_3),1)$ for $i\in\{1,2,3\}$ coincide). This gives the ultrametric inequality
\[
\alpha(K_1,K_2)\le\max\{\alpha(K_2,K_3),\alpha(K_1,K_3)\}.
\]
In other terms, we have just verified that the map $\alpha\colon \mathcal{O}_K\times\mathcal{O}_K\to[0,1)$ is an \emph{ultrametric} on $\mathcal O_K$. For $L\in \mathcal O_K$ and $x\in [0,1)$, we will write
\[
B_\alpha(L,x)=\{L'\in \mathcal O_K:\alpha(L,L')\le x\}
\]
for the $\alpha$-ball of radius $x$ centered at $L$. We denote by $\mathcal B_K$ the collection of $\alpha$-balls in $\mathcal O_K$ of radius $x>0$. We remark the following property.

\begin{lem}\label{lem:balls_CF}\label{l-tree-K-focal}
	For every closed $f$-invariant subset $K\subsetneq (0,1)$, let $\mathcal B_K$ be the collection of $\alpha$-balls defined above. Then the following hold.
	\begin{itemize}
		\item\label{i:BK-CF} The collection $\mathcal B_K$ is cross free.
		\item\label{i:BK-inv} For any $g\in F$, $L\in \mathcal O_K$, and $x\in [0,1)$, we have
		\begin{equation}\label{eq.alpha_balls_invariants}
			g.B_\alpha(L,x)=B_\alpha(g(L),g(x)).
		\end{equation}
		In particular, the collection $\mathcal B_K$ is $F$-invariant.
		\item\label{i:BK-cov} For any $\alpha$-ball $B\in \mathcal B_K$, there exists a sequence of elements $(g_n)\subset F$ such that the sequence $g_n.B$ defines an increasing exhaustion of $\Ocal_K$.
	\end{itemize}
\end{lem}
\begin{proof}
	The fact that $\mathcal B$ is cross free is a well-known consequence of the ultrametric inequality. Next, for given $g\in F$ and $K_1,K_2\in \Ocal_K$, we have
	\begin{align*}
		\alpha(g(K_1),g(K_2))&=\inf\{x\in (0,1): K_1\cap[g^{-1}(x), 1)=K_2\cap[g^{-1}(x), 1)\}\\
		&=\inf\{g(y)\in (0,1): K_1\cap[y, 1)=K_2\cap[y, 1)\}=g\left (\alpha(K_1,K_2)\right ),
	\end{align*}
	from which we deduce the second statement. Finally, as the action of $F$ on $\Ocal_K$ is transitive, it is enough to check that there exists a sequence of elements $(g_n)\subset F$ such that the sequence of $\alpha$-balls $g_n.B_\alpha(K,x)$ defines an increasing exhaustion of $\Ocal_K$. For this, note that by $f$-invariance of $K$ we have $f^n.B_\alpha(K,x)=B_\alpha(K,f^n(x))$, and thus $\Ocal_K=\bigcup_{n \ge 0}f^n.B_\alpha(K,x)$, as desired.
\end{proof}

From this, we see that for any $F$-invariant order $<$ on $\mathcal O_K$ for which $\alpha$-balls are convex, the collection $\mathcal B_K$ gives an $F$-invariant prelamination, and the action of $F$ is focal. We will say for short that $<$ is \emph{$\alpha$-convex} if it satisfies this property. Note also that the relation \eqref{eq.alpha_balls_invariants} implies that the function
\begin{equation}\label{eq:horK}
	\dfcn{\hor_K}{\mathcal B_K}{(0,1)}{B_\alpha(L,x)}{x}
\end{equation}
defines a positive prehorograding of the $F$-action on $\mathcal O_K$ by the standard action on $(0,1)$. By taking the dynamical realization, this will give our desired laminar action of $F$. However, it is not clear \textit{a priori} that for a given subset $K$, such an $\alpha$-convex order exists, and this is why what we have just described is simply a \emph{strategy}. In practice, $\alpha$-convex orders on $\Ocal_K$ are such that the order relation between $K_1$ and $K_2$ only depends on how $K_1,K_2$ behave ``right before'' the point $\alpha(K_1, K_2)$, in an $F$-invariant way. 

In fact, a nice way of thinking about possible $\alpha$-convex orders is to go through the well-known correspondence between ultrametric spaces and trees (see for instance Choucroun \cite{choucroun} or Hughes \cite{hughes}). Here we simply have to reproduce the construction in Proposition \ref{p-from-focal-to-trees}, starting with the action of $F$ on $\mathcal O_K$, preserving the collection $\mathcal B_K$, and using the monotone equivariant function $\hor_K$. This gives a focal action of $F$ on a directed tree $(\Tbb_K,\treeorder)$, and $\alpha$-convex orders on $\mathcal O_K$ correspond to $F$-invariant planar orders on $(\Tbb_K,\treeorder)$.

Roughly speaking, the directed tree $(\Tbb_K,\treeorder)$ is obtained by taking a copy of $(0,1)$ for each $K_1\in \mathcal{O}_K$ (so that a pair $(K_1,x)$ corresponds to the $\alpha$-ball $B_\alpha(K_1,x)$), and by gluing the two copies corresponding to $K_1$ and $K_2$ along the interval $[\alpha(K_1, K_2), 1)$. We denote by $p\colon \mathcal{O}_K\times (0,1)\to\mathbb{T}_K$ the quotient projection and $[K_1,x]:=p(K_1,x)$. Then, two points $v,w\in \Tbb_K$ satisfy $v\treeorder w$ (that is, $v$ lies below $w$) if and only there exist $K_1\in\mathcal{O}_K$ and $x,y\in(0,1)$, so that $v=[K_1,x]$, $w=[K_1, y]$, and $x<y$. The diagonal action of $F$ on $\mathcal{O}_K \times (0, 1)$ descends to an action on $(\Tbb_K,\treeorder)$, and the projection to the second coordinate descends to a positive $F$-equivariant horograding $\hor_K\colon \Tbb_K \to (0, 1)$ (extending the function in \eqref{eq:horK}). Finally, the embedding $i\colon \mathcal{O}_K\to\partial^\ast\mathbb{T}_K$ is defined so that each $K_1\in \mathcal{O}_K$ is sent to the infimum of the $\treeorder$-chain $\{[K_1,x]:x\in (0,1)\}$, which naturally belongs to $\partial^\ast\mathbb{T}_K$.

One can actually see from this construction, by using Proposition \ref{prop.focalisminimal}, that if an $F$-invariant order $<$ on $\mathcal O_K$ is $\alpha$-convex, then the dynamical realization gives a minimal laminar action. Here is a summary of this whole discussion.

\begin{prop}\label{aconvex_dyn-real_minimal}
	For every closed $f$-invariant subset $K\subsetneq (0,1)$ and $F$-invariant $\alpha$-convex order on $\mathcal O_K$, the dynamical realization of the action of $F$ on $(\mathcal O_K,<)$ is a minimal laminar action, positively horograded by the standard action of $F$.
\end{prop}

\begin{ex}[Non-planarly-orderable actions]\label{rem.nonplanar}
	This point of view is well suited for understanding in a more conceptual way whether, for a given $K$, we can find an $\alpha$-convex order on $\mathcal O_K$, or equivalently, a $\Phi$-invariant planar order on $(\Tbb_K,\treeorder)$.
	This turns out to depend on the local geometry of $K$ relatively to dilations by 2. We explain this with an example, but first we introduce some terminology to discuss the local geometry.
	We say that two closed subsets $K_1,K_2\subseteq (0,1)$ have equivalent left-germs at $x$ if for some $\varepsilon>0$ it holds $K_1\cap(x-\varepsilon,x]=K_2\cap(x-\varepsilon,x]$. We denote by $K^-_x$ the left-germ class of the subset $K$ at $x$. Notice that the group  $\Germ_-(x)$ of left-germs of homeomorphisms fixing $x$ naturally acts on the set of left-germs of closed subsets at $x$. We denote by $h_x\in\Germ_-(x)$ the germ of the homothety that fixes $x$ and has derivative $2$.
	
	Recall from Remark \ref{rem.planarexistence}, that the existence of such an invariant planar ordering boils down to the existence, for each branching point $v\in \Tbb_K$, of an ordering of the set of connected components $E^-_v$ below $v$, invariant under the action of the stabilizer $\stab^\Phi(v)$. An obstruction for this is clearly given by finite orbits.
	With this in mind, consider an $f$-invariant closed subset $K\subset (0,1)$, containing a dyadic point $x\in K\cap\Z[\frac{1}{2}]$ such that $h_x(K^-_x)\neq K^-_x$ but $h_x^n(K^-_x)= K^-_x$ for some $n>1$. Write $v=[K,x]$, and let $e_v(K)$ be  the component of $E_v^-$ corresponding to the ray $\{[K,x]:x\in(0,1)\}$. Since $h_x$ generates the group of germs of elements in $F$ fixing $x$, it holds that the component $e_v(K)$ has a finite orbit which is not a fixed point, so that there exists no $\mathsf{Stab}^\Phi(v)$-invariant total order on $E_v^-$. 
\end{ex}

\subsection{A concrete example}\label{subsub.concrete}
We now illustrate the flexibility of the method described in \S\ref{s-F-orbit-strategy} with an explicit example of subset $K$. More precisely, we will construct a subset $K\subset (0, 1)$  with the following property:  \emph{there is an explicit (continuous) injective map from the set $\mathsf O(\N)$ of orders on the natural numbers $\N$ to the set of $F$-invariant orders on the orbit $\mathcal{O}_K$}. This will provide a  family of minimal laminar actions of $F$, which are naturally indexed by orders on  $\N$.

We start by choosing an irrational point $x_0\in (0, 1)$, and consider the interval $I=(f^{-1}x_0, x_0]$, which is a fundamental domain for $f$. Next we choose a sequence of open intervals $(J_n)_{n\ge 1}$ with dyadic endpoints such that $ J_n\Subset J_{n+1}\subset I$ for every $n\ge 1$, and such that $\bigcup_{n\ge 1} J_n=(f^{-1}x_0, x_0)$. For every $n\ge 1$, write $y_n=\sup J_n$ and choose an element $h_n\in F_{J_n}$ with the following properties:
\begin{itemize}
	\item   $h_n(x)>x$ for every $x\in J_n$, 
	\item $h_n(J_{n-1})\cap J_{n-1}=\varnothing$ for $n\ge 2$, and
	\item $D^-h_n(y_n)=1/2$ (in other words,  the germ of $h_n$ at $y_n$ generates the group of germs  $\Germ(F_{(0, y_n)}, y_n)$). 
\end{itemize}
Choose now a dyadic point $z_0\in J_1$, and let $\Sigma_0=\{h_1^n(z_0) : n \in\N\}$ be its forward orbit under $h_1$. By construction, we have the inclusion $\Sigma_0\subset J_1$ and the equality $\overline{\Sigma_0}=\Sigma_0\cup \{y_1\}$. Set $\Sigma_1=\bigcup_{n\ge 0} h_2^n(\overline{\Sigma_0})$, so that $\overline{\Sigma_1}=\Sigma_1\cup\{y_2\}$. Continue in this way by defining a subset $\Sigma_i=\bigcup_{n\ge 0} h_{i+1}^n(\overline{\Sigma_{i-1}})$ for every $i\ge 1$. Set $\Sigma_\omega=\bigcup_{i\in \N} \Sigma_i$, and note that $\overline{\Sigma_\omega}=\Sigma_\omega\cup \{x_0\}$. Note also that $\overline{\Sigma_\omega}$ is contained in the fundamental domain $I$ of $f$. Thus, we obtain an $f$-invariant closed subset $K$ as 
\begin{equation}\label{ex.K}
	K=\bigcup_{n\in \Z} f^n(\overline{\Sigma_\omega}).
\end{equation}
By construction, the subset $\overline{\Sigma_\omega}$ is invariant under the  semigroup $S:=\langle h_n: n\ge 1\rangle_+$, in the sense that $s(\overline{\Sigma_\omega})\subset \overline{\Sigma_\omega}$ for every $s\in S$. See Figure \ref{fig.exotic_CB}.

\begin{figure}[ht]
	\centering
	\includegraphics[width=\textwidth]{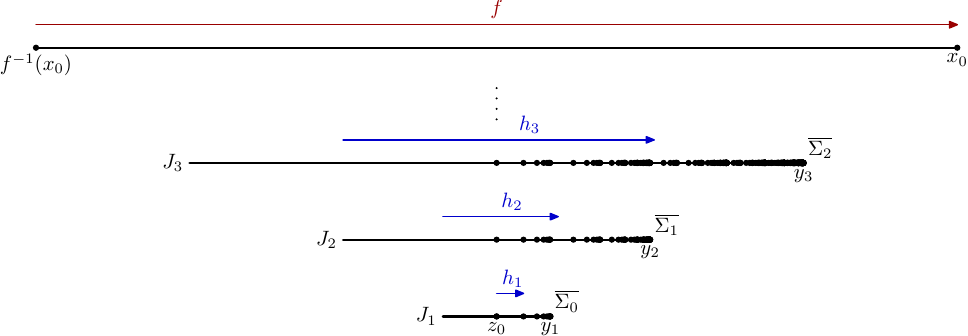}
	\caption{Construction of the compact set $K$ for \S\ref{subsub.concrete}.}\label{fig.exotic_CB}
\end{figure}

The subset $\overline{\Sigma_\omega}$ is countable and compact, and its points can be classified according to their \emph{Cantor--Bendixson rank} (see the book of Kechris \cite[\S6]{Kechris}), as follows. Points of rank 0 are the isolated points: these are exactly points in the orbit of $z_0$ under the semigroup $S$. Points of rank 1 are those that are not isolated, but become isolated after removing the isolated points: these are exactly points in the $S$-orbit of $y_1$. Continuing in this way, points of rank $n$ are precisely points in the $S$-orbit of $y_{n}$. Finally, there is a unique point whose rank is the first countable ordinal $\omega$, namely the point $x_0$. This discussion can be directly extended to the subset $K$.  We write $\rk_K(x)$ for the Cantor--Bendixson rank of a point $x\in K$. Note that we have the relation $\rk_{g(K)}(g(x))=\rk_K(x)$ for every $g\in F$ and $x\in K$.

We next consider the ultrametric
$\alpha\colon \Ocal_K\times\Ocal_K\to [0,1)$ defined  as in \eqref{e-alpha-K1-K2}, and the key observation is that the particular choice of the subset $K$ allows to directly relate $\alpha$ with the Cantor--Bendixson rank.

\begin{lem} \label{l-F-CB}
	Let $K\subset (0,1)$ be the subset defined at \eqref{ex.K}. For every $K_1, K_2\in \mathcal{O}_K$,  the point $x=\alpha(K_1, K_2)$ is such that  $\rk_{K_1}(x)$ and $\rk_{K_2}(x)$ are both finite, and moreover $\rk_{K_1}(x)\neq \rk_{K_2}(x)$ unless $\rk_{K_1}(x)=\rk_{K_2}(x)=0$. 
	
	Conversely, for every distinct $n, m\in \N$,  there exist $K_1, K_2\in \mathcal{O}_K$ such that the point $x=\alpha(K_1, K_2)$ satisfies $\rk_{K_1}(x)=n$ and $\rk_{K_2}(x)=m$.
\end{lem}
\begin{proof}
	We first need some observations.
	
	\setcounter{claimnum}{0}
	
	\begin{claimnum}\label{cl.CB1}
		For every $x\in K$, and every $g\in F$ such that $g(x)=x$, there exists $\varepsilon>0$ such that $g(K)\cap (x-\varepsilon, x]=K\cap (x-\varepsilon, x]$.
	\end{claimnum}
	\begin{proof}[Proof of claim]
		Up to replace $g$ by its inverse, we can assume $D^-g(x)\le 1$.
		Also, upon conjugating by powers of $f$, we can assume $x\in \overline{\Sigma_\omega}$. When $x=x_0$ this follows from the fact that we chose $x_0$ to be irrational, so that every element of $F$ that fixes $x_0$ must actually fix a neighborhood of it. When $x$ is isolated in $K$, the conclusion is obvious. Finally, assume that $n:=\rk_K(x)\notin\{0, \omega\}$.  Then $x$ is in the $S$-orbit of the point $y_{n}$, so that it is fixed by a conjugate $h$ of $h_n$, which has therefore the property that $D^-h(x)=1/2$. Hence, the restriction of $g$ to a left-neighborhood of $x$ must coincide with the restriction of some non-negative power of $h$, so that we can conclude  from the fact that $K$ is forward invariant under $h$.
	\end{proof}
	
	\begin{claimnum}\label{cl.CB2}
		For every pair of points $x, y\in K$ with $\rk_k(x)=\rk_k(y)$, there exist an element $h\in F$ and $\varepsilon>0$ such that  $h(x)=y$ and  $h(K)\cap (y-\varepsilon, y]=K\cap (y-\varepsilon, y]$.
	\end{claimnum}
	\begin{proof}[Proof of claim]Upon replacing $x,y$ with  $f^m(x), f^n(y)$ for suitable $n, m$, we can assume  $x, y\in \overline{\Sigma_\omega}$. Then $x$ and $y$ are in the same $S$-orbit, and so it is enough to observe that elements of $S$ and their inverses have this property.  
	\end{proof}
	
	With this in mind, let us prove the lemma. We can assume without loss of generality that $K_1=K$. Take $g\in F$ such that $K_2=g(K)$, and set  $x=\alpha(K, K_2)$ and  $y=g^{-1}(x)\in K$, so that $\rk_{K_2}(x)=\rk_K(y)$. Assume by contradiction that $\rk_K(x)=\rk_K(y)\ge 1$. After Claim \ref{cl.CB2},  we can choose $h\in F$ such that $h(x)=y$, and $\varepsilon>0$ such that $h(K)\cap (y-\varepsilon, y]=K\cap (y-\varepsilon, y]$. Then the element $g'=hg$ is such that $g(y)=y$, so that upon taking a smaller $\varepsilon$, by Claim \ref{cl.CB1} we also have $g'(K)\cap (y-\varepsilon, y]=K\cap(y-\varepsilon, y]$. Applying $h^{-1}$, we deduce that there is $\varepsilon'>0$ such that $g(K)\cap (x-\varepsilon', x]=K\cap (x-\varepsilon', x]$, and the latter intersection is not reduced to $\{x\}$, since we assume that $\rk_K(x)\ge 1$.   This contradicts the definition of $x=\alpha(g(K),K)$.
	Thus $\rk_K(x)\neq \rk_{g(K)}(x)$, unless both ranks are 0. Finally, this also implies that we cannot have $\rk_K(x)=\omega$. Indeed, since points of rank $\omega$ are the only non-dyadic points in $K$, this would imply that $\rk_{g(K)}(x)=\omega$ as well, contradicting the previous reasoning. \qedhere
	
\end{proof}

Now, let $\mathsf{O}(\N)$ be the set of total orders on the natural numbers.
To every order $\prec$ in $\mathsf{O}(\N)$, we associate an $F$-invariant order $\prec^*$ on $\mathcal{O}_K$, as follows. Given distinct $K_1, K_2\in  \mathcal{O}_K$, set $n_1=\rk_{K_1}(\alpha(K_1, K_2))$ and $n_2= \rk_{K_2}(\alpha(K_1. K_2))$. If $n_1\neq n_2$, then we declare $K_1\prec^* K_2$ if and only if $n_1\prec n_2$. Else, by Lemma \ref{l-F-CB}, we have $n_1=n_2=0$; that is, the point $\alpha(K_1, K_2)$ is isolated in both $K_1$ and $K_2$. In this case, set 
\begin{equation}\label{eq:Kisolated}
	x_i=\max\{x\in K_i: x<\alpha(K_1, K_2)\}\quad\text{for }i\in \{1, 2\}.
\end{equation}
Then  we must have $x_i<\alpha(K_1, K_2)$ for $i\in \{1,2\}$, and $x_1\neq x_2$ by definition of $\alpha(K_1, K_2)$. In this case, we declare $K_1 \prec^* K_2$ if and only if $x_1<x_2$. It is routine to verify that this defines indeed a total order relation, and it is clear from the construction, $F$-equivariance \eqref{eq.alpha_balls_invariants} of the ultrametric $\alpha$, and of the Cantor--Bendixson rank,  that this order is $F$-invariant.

Denote by $\varphi_{\prec}\colon F\to \homeo_{0}(\R)$ the dynamical realization of the action of $F$ on $(\mathcal{O}_K,\prec^\ast)$. We want to prove that $\varphi_{\prec}$ is a laminar action, positively horograded by the standard action of $F$. After Proposition \ref{aconvex_dyn-real_minimal}, this is equivalent to the property that the order $\prec^*$ is $\alpha$-convex.
This is what we verify next.

\begin{lem}\label{lem.Kisconvex} With notation as above, the $\alpha$-ball \[B_\alpha(L,x)=\left \{L'\in\mathcal{O}_K:\alpha(L,L')\leq x\right \}\] is $\prec^\ast$-convex for every $L\in\mathcal{O}_K$ and $x\in (0,1)$. 
\end{lem}

\begin{proof} First notice that the $\prec^\ast$-order relation between $K_1,K_2\in\mathcal{O}_K$ is determined by the intersections $K_1\cap[x,1)$ and $K_2\cap[x,1)$, for any $x\in(0,1)$ such that these intersections do not coincide. 
	
	Now, take elements $K_1,K_2\in B_\alpha(L,x)$ for some $L\in\mathcal{O}_K$ and $x\in (0,1)$. This is equivalent to the condition that
	\begin{equation}\label{eq.ball_intersection}
		K_1\cap [x,1)=K_2\cap [x,1)=L\cap[x,1).
	\end{equation}
	Consider next an element $K_3$ between $K_1$ and $K_2$ (with respect to $\prec^*$), and assume by contradiction that $K_3\notin B_\alpha(L,x)$. This means that $K_3\cap [x,1)\neq L\cap [x,1)$, and therefore, considering the equalities \eqref{eq.ball_intersection}, the $\prec^\ast$-order relation between $K_i$ and $K_3$ is determined by the intersections $K_3\cap[x,1)$ and $L\cap[x,1)$, for every $i\in \{1,2\}$. Hence, we conclude that the $\prec^\ast$-order relation between $K_1$ and $K_3$ coincides with that of $K_2$ and $K_3$. As we are assuming that $K_3$ lies between $K_1$ and $K_2$, we necessarily have $K_1=K_2=K_3$, but this contradicts the assumption $K_3\notin B_\alpha(L,x)$.
\end{proof}

As a conclusion of our discussion, we have the following. 
\begin{prop} \label{p-F-CB}
	With notation as above, for any $\prec\in \mathsf O(\N)$, the dynamical realization $\varphi_{\prec}\colon F\to \homeo_{0}(\R)$ of the action of $F$ on $(\Ocal_K,\prec^*)$ is a minimal laminar action, positively horograded by the standard action of $F$.
	
	Moreover, if $\prec_1$ and $\prec_2$ are distinct orders on $\N$, then the actions $\varphi_{\prec_1}$ and $\varphi_{\prec_2}$ are not conjugate.
\end{prop}
\begin{proof}
	The first statement is a direct consequence of  Lemma \ref{lem.Kisconvex} and Proposition \ref{aconvex_dyn-real_minimal}. Next,
	for a given order $\prec\in \mathsf O(\N)$, observe that by definition of dynamical realization, the $F$-action on $(\mathcal{O}_K, \prec^*)$ can be identified with the $\varphi_\prec$-action on the orbit of the unique fixed point $\xi$ of $\varphi_\prec(f)$, with the order induced by $\R$, so that the order $\prec^*$ can  be reconstructed from  $\varphi_{\prec}$. Finally, the order $\prec$ on $\N$ can be reconstructed from $\prec^*$ by the last statement in Lemma \ref{l-F-CB}.
\end{proof}

We now point out a qualitative difference which distinguishes  the family of laminar actions constructed here, from the one obtained via the restriction preorder construction, as in \S \ref{s-restriction-preorder}. Indeed, in this case the  actions of the subgroups $F_{(0, x)}\cong F$ on the components of their support can remain exotic (compare this with Proposition \ref{p-restriction-sublevels}).

\begin{prop}[Presence of exotic sublevels]\label{p-F-CB-sublevels}
	Fix an order $\prec$ on $\N$, and let $\varphi:=\varphi_\prec\colon F\to \homeo_0(\R)$ be the laminar action constructed above. Then, there exist a dyadic point $x\in (0,1)$ and a connected component $J$ of  $\suppphi(F_{(0,x)})$, such that the action of $F_{(0, x)}$ on $J$ is semi-conjugate to a minimal laminar action.
\end{prop}
\begin{proof}
	Note first that for every $x\in (0,1)$, and every $g\in F$ such that
	$g(K\cap [x,1))=K\cap [g(x), 1)$, one has $g.B_\alpha(K,x)=B_\alpha(K,g(x))$. In particular, 
	the $\alpha$-ball $B_\alpha(K,x)$ is preserved by the subgroup $F_{(0,x)}$.
	Let $\iota\colon (\mathcal{O}_K, \prec^*)\to \R$ be an equivariant good embedding associated with $\varphi$ (in the terminology of Definition \ref{dfn.goodbehaved}), and let  $I_x$ be the open interval spanned by $\iota(B_\alpha(K,x))$. That is, $I_x$ is the interior of the closure of $\iota(B_\alpha(K,x))$ (using minimality of the action and that $\alpha$-balls are $\prec^*$-convex after Lemma \ref{lem.Kisconvex}).
	Consider the element $h_1$ from the construction of $K$, and consider the points $z_0\in J_1$ and $y_1=\sup J_1$, as in the construction; for $n\ge 1$, set $z_n=h_1^n(z_0)$, which by construction is an increasing sequence converging to $y_1$. For every $n\ge 0$, we have $h_1^n(K\cap [z_0, 1))=K\cap [z_n, 1)$, so that $h_1^n.B_\alpha(K,z_0)=B_\alpha(K,z_n)$. The corresponding intervals $I_{z_n}$ satisfy $I_{z_n}\Subset I_{z_{n+1}}$ and $h_1.I_{z_n}=I_{z_{n+1}}$. Set   $B:=\bigcup_{n \ge 0} B_\alpha(K,h_1^n(z_0))$, and let $J=\bigcup_{n \ge 0} I_{z_n}$ be the interval spanned by $\iota(B)$. Then, $B$ is preserved by $F_{(0, y_1)}$, and the action of $F_{(0, y_1)}$ on $B$ is cofinal (with respect to the order $\prec^*$ restricted to $B$). As a consequence, $F_{(0, y_1)}$ preserves $J$, and acts on it without fixed points, so that $J$ is a connected component of $\suppphi(F_{(0, y_1)})$. Since moreover 
	$h_1.I_{z_n}=I_{z_{n+1}}$, we deduce that $\varphi(h_1)$ acts on  $J$ as a pseudo-homothety. 
	This cannot happen if the action of $F_{(0, y_1)}$ on $J$ is semi-conjugate to an action by translations, nor if it is semi-conjugate to the standard action on $(0, y_1)$. Thus, by Theorem \ref{t-F-trichotomy} the action of $F_{(0, y_1)}$ must be semi-conjugate to a minimal laminar action.
\end{proof}

Nonetheless, this family of examples 
still turns out to produce simplicial laminar actions, in the sense of Definition \ref{dfn.simplicial_laminar}.

\begin{prop}[Simpliciality] \label{p-F-orbit-not-minimal}
	For an order $\prec$ on $\N$, let $\varphi_\prec\colon F\to \homeo_0(\R)$ be the laminar action constructed above. Then $\varphi_\prec(F_+)$ does not act minimally on $\R$. In particular each action $\varphi_\prec$ is simplicial.
\end{prop}
\begin{proof}
	We keep the same notation as in the proof of Proposition \ref{p-F-CB-sublevels}. Let $x_0\in K$ be the point as in the construction of $K$. We claim that the $\alpha$-ball $B_\alpha(K,x_0)\subset \mathcal{O}_K$ has the property that for every $g\in F_+$ we have either $g.B_\alpha(K,x_0)=B_\alpha(K,x_0)$, or $g.B_\alpha(K,x_0)\cap B_\alpha(K,x_0)=\varnothing$. It then follows that the interval $I_{x_0}$ has the same property for $\varphi_\prec(F_+)$, so that the union of its translates defines a proper invariant open subset, contradicting minimality. To prove our claim, suppose that $g\in F_+$ and $K_1\in B_\alpha(K,x_0)$ are such that $g(K_1)\in B_\alpha(K,x_0)$, namely we assume
	\[K_1\cap [x_0,1)=g(K_1)\cap[x_0,1)=K \cap [x_0, 1).\]
	The key observation is that this implies that $g$ must actually fix $K\cap [x_0, 1)$. First of all, observe that $g$ must send points of rank $\omega$ in $K_1$ to points of rank $\omega$ in $g(K_1)$, and the set of such points in both $K_1\cap [x_0,1)$ and $g(K_1)\cap [x_0, 1)$ consists precisely of  the sequence $x_n:=f^n(x_0)$ for $n\ge 0$. This is a discrete increasing sequence, and we deduce from the condition $g\in F_+$ that $g(x_n)=x_n$ for every $n\ge 0$. As a consequence, the cyclic subgroup $\langle g\rangle$ must preserve every interval $[x_n, x_{n+1}]$, with $n\ge 0$, and thus every intersection $K\cap [x_n, x_{n+1}]$, for $n \ge 0$. 
	Assume by contradiction that there exist $n\ge 0$ and a point $t\in K\cap [x_n,x_{n+1}]$ which is not fixed by $g$. Consider the orbit $\Omega=\{g^m(t): m\in \Z\}$, which is a subset of $K$. As $K\cap [x_n,x_{n+1}]$ is compact, the point $\inf \Omega$ is in $K\cap [x_n,x_{n+1}]$, and it is accumulated by points of $\Omega$ (and hence of $K$) from the right. This is in contradiction with the choice of $K$, as by construction,  every point of $K$ is isolated from the right-hand side. Hence, $g$ fixes $K\cap [x_0, 1)$,  which implies that $g.B_\alpha(K,x_0)=B_\alpha(K,x_0)$.
\end{proof}

\section[A plethora of laminar actions, III]{A plethora of laminar actions, III: non-simplicial laminar actions} \label{s-F-hyperexotic}
\subsection{Simplicial laminar actions of $F$ and minimality of the action of $[F, F]$} \label{s-F-simplicial}

All the examples of laminar actions of $F$ discussed so far are simplicial in the sense of Definition \ref{dfn.simplicial_laminar}. In this case, Proposition \ref{p-focal-simplicial} can actually be stregnthened to give the following.

\begin{prop}[Simplicial laminar actions of $F$]\label{prop.hyperequiv}  Let $\varphi\colon F\to\homeo_{0}(\R)$ be a minimal laminar action, positively horograded by the standard action on $(0, 1)$. Then, the following are equivalent.

\begin{enumerate}[label=(\roman*)]
	\item \label{i-F+-minimal} The image $\varphi(F_+)$ admits no minimal invariant set.
	\item \label{i-F'-minimal} The image $\varphi([F, F])$ admits no  minimal invariant set.
	
	\item \label{i-F-discrete} The action $\varphi$ preserves a discrete lamination. 
	
	\item  \label{i-simplicial} The action $\varphi$ is represented by a focal action on a planar directed tree $(\Tbb, \treeorder, \prec)$, such that $\Tbb$ is a simplicial tree of countable  degree, and the action of $F$ on $\Tbb$ is by simplicial automorphisms.
	
	\item \label{i-isometric} The action $\varphi$ is represented by a focal action on a planar
	directed tree $(\Tbb, \treeorder, \prec)$ such that $F$ acts  by isometries with respect to a compatible $\R$-tree metric on $\Tbb$. 
	
	\item \label{i-homothety} Every pseudo-homothety in the image of $\varphi$ is a homothety. 
	
	\item \label{i-not-type-III} There exist non-empty bounded open intervals $I,J\subset \R$ such that $J\not\subset g.I$ for any $g\in F_+$.
\end{enumerate}
\end{prop}

For many  constructions of laminar actions of $F$ discussed below, one can easily check  conditions \ref{i-F+-minimal} or \ref{i-F'-minimal}, and thus they turn out to be simplicial  (although a simplicial tree does not always appear naturally in the construction, and it might also be not obvious to check directly condition \ref{i-homothety}).

\begin{proof}[Proof of Proposition \ref{prop.hyperequiv}]
Let us first show that  \ref{i-F+-minimal} and \ref{i-F'-minimal} are equivalent. If $[F, F]$ admits a minimal invariant set, then by Lemma \ref{l-normal-minimal} it acts minimally on $\R$, and thus so does $F_+$. Conversely, assume by contradiction that $[F, F]$ does not act minimally on $\R$, but $F_+$ does. Then, we can apply Proposition \ref{p-focal-simplicial} to $G=F_+$ and $N=[F, F]$: we deduce that $\varphi|_{F_+}$ can be horograded by a cyclic action coming from the quotient $F_+/[F, F]\cong \Z$. Since the quotient $F_+/[F, F]$ is simply the group of germs $\Germ(F_+, 0)$ we deduce that the $\varphi$-image of every  $g\in F_+$ with a non-trivial germ at $0$  must be a pseudo-homothety. However, we were assuming that $\varphi$ is positively horograded by the standard action of $F$ on $(0, 1)$, so that the image of every element $g\in F_+$ must be totally bounded (by Proposition \ref{p-dyn-class-elements-horograded}). This is a contradiction.

The implications \ref{i-F+-minimal}$\Rightarrow$\ref{i-F-discrete}$\Rightarrow$\ref{i-simplicial} follow from Proposition \ref{p-focal-simplicial}, \ref{i-simplicial}$\Rightarrow$\ref{i-isometric} is obvious, and  \ref{i-isometric}$\Rightarrow$\ref{i-homothety} follows from Remark \ref{rem.homotheties_isometric}.
To show that \ref{i-homothety} implies \ref{i-not-type-III}, we  show that  \ref{i-homothety} actually implies the following more explicit condition, which clearly implies \ref{i-not-type-III}. 

\begin{itemize}
	\item[\emph{(vii')}] \emph{For every element $h\in F$ which   in the standard action satisfies $h(x)>x$ for every $x\in (0, 1)$, there exists  a bounded open  interval $I\subset \R $  such that $I\subset h.I$  and  $h.I\not\subset g.I$ for any $g\in F_+$.}
\end{itemize}

Indeed, assume that $h$ is such that $h(x)>x$ for every $x\in X$, so that by \ref{i-homothety} its image $\varphi(h)$ is an expanding homothety; let $\xi\in \R$ be the unique fixed point of the image $\varphi(h)$. Note that the subgroup $\langle F_+,h\rangle$ has finite index in $F$, and for any $x\in X$ we have $\langle F_+,h\rangle=\langle F_{(0,x)},h\rangle$. As $\varphi$ is irreducible, this implies that $\xi\in \suppphi(F_{(0,x)})$ for any $x\in X$, and thus we can consider the connected component $I=\Iphi(x,\xi)$ of $\suppphi(F_{(0,x)})$ containing $\xi$ (keeping the same notation as in \S \ref{ssc.CF_family}).
As $\varphi(h)$ is an expanding homothety, we have $I\subset h.I$. Suppose by contradiction that there exists $g\in F_+$ such that $h.I\subset g.I$.
Note that we have the equalities \[g.I=g.\Iphi(x,\xi)=\Iphi(g(x),g.\xi)=\Iphi(g(x),\xi)\]
(the last equality follows from the assumption that $\xi\in I\subset g.I$). Moreover, we have  $h.I=\Iphi(h(x),\xi)$, and so we conclude that  $\Iphi(h(x),\xi)\subset \Iphi(g(x),\xi)$, implying that $g(x)\ge h(x)$ (see \S \ref{ssc.CF_family}). Since $g\in F_+$, this implies that for some $y\ge x$ it holds that $g(y)=h(y)$. Therefore,  we have
\[g.\Iphi(y,\xi)=\Iphi(g(y),\xi)=\Iphi(h(y),\xi)=h.\Iphi(y,\xi)\]
(the first equality follows from the fact that $\Iphi(y,\xi)\supset I$),
and this implies that $\varphi(g^{-1}h)$ preserves $\Iphi(y,\xi)$. Since $g^{-1}h$ has the same germ  at $1$ as $h$, it acts as a pseudo-homothety, so that by \ref{i-homothety} it is  a homothety, and this is a contradiction with the fact that it preserves $\Iphi(y, \xi)$.

Finally, to show that \ref{i-not-type-III} implies \ref{i-F+-minimal}, assume by contradiction (using Lemma \ref{l-normal-minimal}) that the action of $F_+$ is minimal. As we are assuming \ref{i-not-type-III}, the action of $F_+$ cannot be proximal, so by Theorem \ref{t-centralizer} the centralizer of $\varphi(F_+)$ in $\homeo_0(\R)$ must be infinite cyclic generated by an element without fixed points. Since $F_+$ is normal, we deduce that the whole group $\varphi(F)$ must normalize this cyclic centralizer, and thus centralize it, contradicting that minimal laminar actions are proximal (Proposition \ref{prop.minimalimpliesfocal}). 
\end{proof}

\subsection{A non-simplicial action} \label{s-F-non-simplicial}
By the previous discussion and examples, it would be  tempting  to try to prove that all laminar actions of $F$ have this property. However, we build  an exotic action of Thompson's group $F$ which is non-simplicial. 
With Proposition \ref{prop.hyperequiv}   in mind, we will prove the following.
\begin{thm} \label{t-F-hyperexotic}
There exist laminar actions $\varphi \colon F\to \homeo_0(\R)$ for which $\varphi([F,  F])$ acts minimally (and thus are not simplicial). More precisely, there exist uncountably many such actions, whose restriction to $[F, F]$ yield pairwise non-conjugate actions of $[F, F]$. 

In particular, the group $[F, F]$ admits uncountably many, pairwise non-conjugate, minimal actions $\varphi\colon [F, F]\to  \homeo_0(\R)$.
\end{thm}

\begin{rem}
The last statement in Theorem \ref{t-F-hyperexotic} should be compared with the general constructions of exotic actions from orders of germ type for groups of compactly supported homeomorphisms (including $[F,F]$), described in \S \ref{s-germ-type}, which provide actions without any  minimal invariant set.
\end{rem}

The construction given here relies  on  the classical symbolic coding of the standard action of $F$ by binary sequences, which is specific to Thompson's groups.
For the proof, it will be convenient to see  $F$ as a group of homeomorphisms of $X=\R$ rather than of the interval $(0,1)$. Namely, we realize $F\subset \PL(\R)$  as the group of piecewise linear maps of the line, with dyadic breakpoints, slopes in the group $\langle 2^n:n\in\Z\rangle$, and which coincide with integer translations near $\pm\infty$. It is well known that this action is conjugate to the natural action of $F$ on $(0,1)$ (see e.g.\ \cite[Lemma E18.4]{BieriStrebel}).
\emph{From now and until the end  of this subsection, the term  \emph{standard action}  will refer to the action of $F$ on $\R$ described above. We will denote by $f\in F$ the translation $f(x)=x+1$.} (Note that the element $f$  corresponds to the element given by \eqref{e-F-big-generator0} in the action on $(0, 1)$.) 

The proof of Theorem \ref{t-F-hyperexotic} employs the strategy described in \S \ref{s-F-orbit-strategy}. Namely, we will start with a closed $f$-invariant subset $K\subset \R$, and define an $F$-invariant $\alpha$-convex order on its orbit $\mathcal{O}_K:=\{g(K): g\in F\}$. The main difficulty is that we need to construct a subset $K$ satisfying a somewhat delicate combination of properties. We begin with a definition.

\begin{dfn} We say that a subset $K\subset \R$ has property $(O)$ if it is proper, non-empty, closed, $f$-invariant, and moreover $g(K)\cap K$ is open in $K$ for every $g\in F$.
\end{dfn}

\begin{rem}
Note that the last condition for property $(O)$ is actually  equivalent to  $K_1\cap K_2$ being open in $K_1$ and $K_2$  for every $K_1, K_2\in\mathcal{O}_K$. 
\end{rem}

\begin{ex}
Property $(O)$ is clearly satisfied when $K$ is a non-empty $f$-invariant discrete subset, as for example the $f$-orbit of a point. However, this is not a good example for the construction described in this subsection, as the stabilizer of such $K$ in $[F,F]$ is trivial (cf.\ Proposition \ref{p-F-hyperexotic}).
\end{ex}

Assume that $K$ is a subset with property $(O)$ (many examples are exhibited by Lemma \ref{lem.propertyO}). We consider the function $\alpha\colon \Ocal_K\times \Ocal_K\to \R\cup \{-\infty\}$ similarly as in \S \ref{s-F-orbit-strategy}, namely 
\begin{equation}\label{eq:alpha_O}
\alpha(K_1,K_2):=\inf\left \{x\in\R:[x,+\infty)\cap K_1=[x,+\infty)\cap K_2\right \}.
\end{equation}
Note that the function $\alpha$ defined here is not an ultrametric, but this is only because it takes values in $\R\cup \{-\infty\}$: all the other properties discussed in \S \ref{s-F-orbit-strategy} are satisfied, and one can recover an ultrametric by considering, for instance, the function $2^\alpha$.
Reasoning as in the example of \S \ref{subsub.concrete}, we proceed to construct an $F$-invariant order on $\mathcal{O}_K$, and then prove that the dynamical realization of the action of $F$ on $(\Ocal_K,\prec)$ is minimal and laminar. This is the content of the next result.

\begin{prop}\label{prop O implica orden}
	If a subset $K\subset \R$ has property  $(O)$, the relation $\prec$ on $\mathcal{O}_K$ defined by $K_1\prec K_2$ if and only if
\[ 
\max\left \{x\in K_1: x<\alpha(K_1,K_2)\right \}<\max\left \{x\in K_2: x<\alpha(K_1,K_2)\right \},\]
is an $F$-invariant total order on $\mathcal{O}_K$. Moreover, the dynamical realization $\varphi_K\colon F\to \homeo_{0}(\R)$ of the action of $F$ on $(\Ocal,\prec)$ is a minimal laminar action.
\end{prop} 

\begin{proof} Recall that $\alpha(K_1,K_2)\in K_1\cap K_2$ whenever  $K_1$ and $ K_2$ are different elements of $ \mathcal{O}_K$. As $K$ satisfies property $(O)$, we have that  $K_1\cap K_2$ is open inside both $K_1$ and $K_2$. Hence, $K_1\cap K_2$ is an open neighborhood of $\alpha(K_1,K_2)$ inside  $K_1$ and $K_2$. Therefore, $\alpha(K_1,K_2)$ is isolated from the left-hand side in both  $K_1$ and $K_2$. 
As for \eqref{eq:Kisolated}, we deduce that the points
\[
x_i:=\max\{x\in K_i: x<\alpha(K_1,K_2)\}\quad\text{for }i\in \{1,2\}
\]
are distinct, so we can declare
$K_1\prec K_2$ if and only if $x_1<x_2$.
As for the order $\prec^*$ from \S\ref{subsub.concrete}, it is routine to check that this defines indeed an $F$-invariant total order on $\Ocal_K$.

One proceeds similarly as in \S \ref{subsub.concrete}, to check that $\varphi_K$ is minimal and laminar. Namely, one verifies that the order $\prec$ is $2^\alpha$-convex, and  the proof of Lemma \ref{lem.Kisconvex} can be adapted \textit{verbatim} to this case (just replacing $\prec^\ast$ with $\prec$ and $(0,1)$ with $\R$). Then Proposition \ref{aconvex_dyn-real_minimal} gives the desired conclusion.
\end{proof}

The main difference from the construction in \S \ref{subsub.concrete}  is the way that the commutator subgroup $[F,F]$ acts in the actions $\varphi_K$.

\begin{prop} \label{p-F-hyperexotic}
	Given a subset $K\subset (0,1)$ with property $(O)$, let $\varphi_K\colon F\to \homeo_{0}(\R)$ be the corresponding minimal laminar action from Proposition \ref{prop O implica orden}. Then the following hold.
\begin{enumerate}[label=(\roman*)]
		\item \label{i-F-O-minimal} If $K$ has property $(O)$, then $\varphi_K([F, F])$ acts minimally, provided that  the stabilizer of $K$ in $[F, F]$ (with respect to the standard action) acts on $K$ without  fixed points. Moreover, in this case, the induced action of $[F,F]$ is  minimal and laminar. 
	\item \label{i-F-O-distinct} If two distinct subsets  $K,K'\subset \R$ have property $(O)$, then the restrictions of  $\varphi_K$ and $\varphi_{K'}$ to $[F, F]$ are not conjugate actions of $[F, F]$. In particular, $\varphi_K$ and $\varphi_{K'}$ are not conjugate. 
\end{enumerate}
\end{prop}

\begin{proof}
	To prove \ref{i-F-O-minimal}, assume that the stabilizer of $K$ in $[F, F]$ acts without fixed points on $K$. 
Fix $x\in (0,1)$, and choose a sequence of elements $(g_n)_{n\in \N}$ in $[F, F]$ which preserve $K$, and such that $g_n(x)$ tends to $+\infty$ as $n\to \infty$. Then, $g_n.B_\alpha(K,x)=B_\alpha(K,g_n(x))$ for every $n\in \N$, so that $\Ocal_K=\bigcup_{n \ge 0}g_n.B_\alpha(K,x)$. Then, by Proposition \ref{p-focal-semiconj}, the subgroup $[F,F]$ admits a unique minimal invariant set $\Lambda\subset \R$, which is preserved by $F$ because $[F,F]$ is a normal subgroup. We deduce that the action of $[F,F]$ is also minimal.

To prove \ref{i-F-O-distinct}, take $K\neq K'$ with property $(O)$, and assume (without loss of generality) that $K'\not\subset K$. Write $\alpha\colon \Ocal_K\times \Ocal_K\to \R\cup\{-\infty\}$ and  $\beta\colon \Ocal_{K'}\times \Ocal_{K'}\to \R\cup\{-\infty\}$
	for the corresponding functions defined as in \eqref{eq:alpha_O}.  Fix $x\in K$, and let $D$ be the subgroup of $[F,F]_{(x,+\infty)}$ which fixes $K$ pointwise. Then, for every $g\in D$, we have $g.B_\alpha(K,x)=B_\alpha(K,x)$, and actually $g$ fixes $B_\alpha(K,x)$ pointwise. Consequently, the dynamical realization $\varphi_K$ fixes a non-empty open interval pointwise. On the contrary, for any $y\ge x$ with $y\in K'\setminus K$, we can consider an element $h\in D$ such that $h(y)\notin K'$ and $h(y)>y$. 
	Let us show that for such choices we have  $h.B_\beta(K',y)\cap B_\beta(K',y)=\varnothing$. Indeed, assume there exists $L\in h.B_\beta(K',y)\cap B_\beta(K',y)$; then, as $h.B_\beta(K',y)=B_\beta(h(K'),h(y))$, we have
	\[L\cap [h(y),+\infty)=h(K')\cap [h(y),+\infty),\]
	and in particular $h(y)\in L$. However, if $L\in B(K',y)$, then $L\cap [y,+\infty)=K'\cap [y,+\infty),$
	and thus $h(y)\notin L$, which is an absurd.
	
	By $f$-invariance of $K'\setminus K$, we can find arbitrarily large points $y$, and thus elements $h\in D$, satisfying such properties. As $\Ocal_{K'}=\bigcup_{y}B_\beta(K',y)$, this implies that $D$ acts without fixed points, so that the actions $\varphi_K$ and $\varphi_{K'}$ cannot be conjugate.
\end{proof}

After the previous proposition, in order to prove Theorem \ref{t-F-hyperexotic}, we need to show the existence of subsets $K\subset \R$ with property $(O)$ and the additional property that the stabilizer of $K$ in $[F, F]$ does not have fixed points. For this, we are going to use the symbolic description of real numbers by binary expansions.

To each infinite sequence $(a_n)_{n\geq 1}\in \{\0,\1\}^{\mathbb{N}}$, we associate the real number $\ev((a_n)):=\sum_{n\geq 1}a_n2^{-n}\in[0,1]$. Note that this association is continuous if we endow $\{\0,\1\}^\N$ with the product topology. If $z, w_1,w_2$ are finite binary sequences (\emph{binary words} for short), we consider the \emph{cylinder} over $z$, defined by \[C_z=\{w\in\{\0,\1\}^{\mathbb{N}}:z\text{ is a prefix of }w\},\]
and denote by $\tilde{K}_0(w_1,w_2)\subseteq\{\0,\1\}^{\mathbb{N}}$ the subset of all infinite concatenations of $w_1$'s and $w_2$'s. Clearly, both images $\ev(C_z)$ and  $K_0(w_1, w_2):=\ev(\tilde{K}_0(w_1,w_2))$ are closed subsets of $[0,1]$, and the former is a closed interval with dyadic endpoints (a \emph{dyadic interval} for short). Note that, conversely, any closed dyadic interval  is the union of (the real numbers represented by) finitely many cylinders.

With this in mind, if $z_1$ and $z_2$ are binary words, the \emph{substitution map} $S(z_1,z_2)\colon C_{z_1}\to C_{z_2}$  defined by $S(z_1,z_2)(z_1w)=z_2w$ represents an affine map $\overline{S(z_1,z_2)} $ between dyadic intervals of $[0,1]$. Therefore, in the action of  $F$ on $\R$, every element of $F$ locally coincides (except at breakpoints, which are finitely many dyadic rationals) with transformations of the form  $f^n\circ \overline{S(z_1,z_2)}\circ f^m$, for some powers $n,m\in\Z$ and some finite sequences $z_1,z_2$.

We say that a pair of binary words $w_1,w_2$ has the \emph{cancellation property}, if whenever $zw=w'$ for $w,w'\in \tilde{K}_0(w_1,w_2)$, it holds that $z$ is a finite concatenation of $w_1$'s and $w_2$'s.  As a concrete example of a pair of words with the cancellation property, we may take $w_1=\0$ and $w_2=\1$ but these are \emph{constant} binary words (i.e.\ made of a single repeated bit). As a concrete example of non-constant binary words with the cancellation property we can take $w_1=\mathtt{10001}$ and $w_2=\mathtt{01110}$.

\begin{lem}\label{lem.propertyO}
	Let $w_1$ and $w_2$ be non-constant binary words satisfying the cancellation property, and write $K_0:=K_0(w_1,w_2)$. Then, the subset $K:=\bigcup_{n\in\Z}f^n(K_0) $ has property $(O)$. 
\end{lem}

\begin{proof} Since $K_0$ is a closed subset of $[0,1]$, the subset $K$ is a closed and  $f$-invariant subset of $\R$.  Also, since $w_1$ and $w_2$ are non-constant, the set $\tilde{K}_0(w_1,w_2)$ contains no eventually constant sequences, and so the subset $K_0$ contains no dyadic points. It follows that $\ev\colon \tilde{K}_0(w_1, w_2)\to K_0$ is a homeomorphism onto its image, and that the set of intersections of the form $(p/2^n,(p+1)/2^n)\cap K$, with $p\in \Z$ and $n\in \N$, forms a basis of its topology. The restriction of every element of $F$ to $K$ is locally given by maps of the form $f^n\circ\overline{ S(z_1,z_2)}\circ f^m$,  and since $K$ is $f$-invariant, in order to check property $(O)$, it is enough to check that $\overline{S(z_1,z_2)}(K_0\cap \ev(C_{z_1}))$ is open in $K_0$, for every pair of finite binary  words $z_1,z_2$. 
	
	For this, consider two binary  finite words $z_1,z_2$, and also $w\in \tilde{K}_0(w_1,w_2)\cap C_{z_1}$  so that $S(z_1,z_2)(w)\in \tilde{K}_0(w_1,w_2)$. We need to check that  $S(z_1,z_2)(\tilde{K}_0(w_1,w_2)\cap C_{z_1})$ contains a neighborhood of $S(z_1,z_2)(w)$. Since the pair $w_1,w_2$ has the cancellation property, we can write $w=z_1'w'$ with $z_1'=z_1z_1''$ and $w'\in \tilde{K}_0(w_1,w_2)$. Since $S(z_1,z_2)(w)$ equals $z_2z_1''w'$, and belongs to $\tilde{K}_0(w_1,w_2)$, again we conclude, by the cancellation property,  that $z_2z_1''$ is a finite concatenation of $w_1$'s and $w_2$'s. Therefore,
	\[S(z_1,z_2)\left (C_{z_1'}\cap \tilde{K}_0(w_1,w_2)\right )=C_{z_2z_1''}\cap \tilde{K}_0(w_1,w_2),\] showing that $S(z_1,z_2)(\tilde{K}_0(w_1,w_2))$ contains a neighborhood of $S(z_1,z_2)(w)$ inside $\tilde{K}_0(w_1,w_2)$. This concludes the proof.
\end{proof}

In order to ensure that the stabilizer of $K$ in $[F, F]$ has no fixed points, we need  to impose one last extra condition on $K$. 

Say that a map $h\colon I\to J$ is a \emph{dyadic affine map} between intervals  if $I$ is a dyadic interval, and $h$ is of the form $x\mapsto ax+b$, where $ a\in \{2^n: n\in {\mathbb{Z}}\}$, and $b\in\Z[1/2]$. Consider now a compact subset $K_0\subseteq (0,1)$. We say that $K_0$ admits a \emph{self-similar decomposition} if there exists a pair of dyadic affine maps $h_1,h_2\colon I\to (0,1)$ such that
\[h_1(I)\cap h_2(I)=\emptyset,\quad\text{and}\quad h_1(K_0)\cup h_2(K_0)=K_0.\] 
For example, the subset $K_0(w_1, w_2)$ admits a self-similar decomposition, provided the words $w_1$ and $w_2$ are such that $\ev(C_{w_1})$ and $\ev(C_{w_2})$ are disjoint. Indeed, in this case we have that for $i\in \{1,2\}$, the symbolic maps  $w\mapsto w_iw$ correspond to dyadic affine maps $h_i\colon [0,1]\to [0,1]$ with disjoint images, and such that $K_0=h_1(K_0)\sqcup h_2(K_0)$. 

\begin{lem}\label{lem.horseshoe} 
	Let $K_0\subset (0,1)$ be a closed subset admitting a self-similar decomposition, and set $K=\bigcup_{n\in\Z}f^n(K_0)$. Then, the action of $H=\{g\in [F, F] :  g(K)=K\}$  on $K$  has no  fixed points. 
\end{lem}

To show Lemma \ref{lem.horseshoe} we use its self-similarity to build elements in $H$ moving points of the real line arbitrarily far away. But before giving the formal proof, let us see how this ends the proof of Theorem \ref{t-F-hyperexotic}.

\begin{proof}[Proof of Theorem \ref{t-F-hyperexotic} given Lemma \ref{lem.horseshoe}]
	Let $K\subsetneq \R$ be a subset satisfying property $(O)$ and such that $K\cap (0,1)$ admits a self-similar decomposition. As a concrete example, we may take $K=\bigcup_{n\in \Z} f^n(K_0) $, with $K_0=K_0(w_1, w_2)$ for   $w_1=\mathtt{10001}$ and $w_2=\mathtt{01110}$. By Lemma \ref{lem.propertyO}, we may consider the laminar action $\varphi_K$, and by Lemma \ref{lem.horseshoe} and Proposition \ref{p-F-hyperexotic}, we have that $\varphi_K([F, F])$ acts minimally. To finish the proof, we show that from the existence of one subset $K$ with these properties, we can deduce the existence of uncountably many.
	Let $K$ be such a subset. Clearly $K$ is locally a Cantor set, so $\R\setminus K$ is a countable union of open intervals, that we call the \emph{gaps} of $K$. Pick $\beta\in (0,1)$. For each gap $I$ of $K$, consider the point $p_{I}(\beta)$ where $p_I\colon (0,1)\to I$ is the unique order-preserving affine map. We let $K^\beta$ be the subset resulting from adding to $K$ all the points of the  form $ p_I(\beta)$, where $I$ runs over gaps of $K$. Clearly $K^\beta$ is still closed, and $f$-invariant. Moreover, $K^\beta$ still  admits a self-similar decomposition, since the maps involved in the definition of a self-similar decomposition are affine maps sending gaps of $K$ to gaps of $K$, so in particular they preserve the proportion of the subdivision we have introduced in the gaps. We claim that for uncountably many $\beta\in (0,1)$, the subset $K^\beta$ also satisfies (the last condition of) property $(O)$. That is,  $g.K^\beta \cap K^\beta$ is open in $K^\beta$ for every $g\in F$.  
	
	Fix $g\in F$. The only problem that may arise is that a point of the form $p_I(\beta)$ (which is an isolated point) might land inside $K$ under the action of $g$. But if we fix a gap $I$ of $K$, the set of parameters $\beta$ such that $g(p_I(\beta))$ does not belong to $K$, is open and dense in $(0,1)$ (since $K$ has empty interior). In particular, since there are only countably many gaps, and $F$ is also countable, with a Baire-like argument we obtain that $K^\beta$ satisfies property $(O)$ for a generic choice of $\beta\in (0,1)$.
\end{proof}

We conclude this subsection with the proof of Lemma \ref{lem.horseshoe}. For this, we need the following elementary interpolation lemma. Its  proof follows from transitivity of the action of $[F, F]$ on unordered $n$-tuples of dyadic numbers  (see for instance Bieri and Strebel \cite{BieriStrebel}), and details are left to the reader. To simplify the statement, given  (possibly unbounded) intervals $I,J$, we write $I<J$ whenever $\sup I<\inf J$.
\begin{lem}\label{sublem.gluing} For $k\ge 1$,  consider intervals $I_1<\cdots<I_k$ and $J_1<\cdots<J_k$ in $\R$ with dyadic endpoints, with  $I_1=J_1=(-\infty,p]$, $I_k=J_k=[q,+\infty)$,
	and such that $h_n\colon I_n\to J_n$ are dyadic affine maps (for $n\in \{1,\ldots,k\}$). 
	Assume moreover that $h_1$ and $h_k$ are restrictions of the identity. 
	Then, there exists $g\in [F, F]$ such that $g\restriction_{I_n}=h_n$ for $n\in \{1,\ldots,k\}$. 
\end{lem}
\begin{proof}[Proof of Lemma \ref{lem.horseshoe}]
	Consider the dyadic affine maps $h_1,h_2\colon I_0\to(0,1)$ given by the self-similar decomposition of $K_0$. Since $K_0$ is a closed subset of $(0,1)$, we can assume that $I_0$ is a closed dyadic interval inside $(0,1)$. For $i\in\{1,2\}$, write $K_0^i=h_i(K_0)$ and $I_0^i=h_i(I_0)$. Note that $I_0^1\cap I_0^2=\varnothing$ and $K_0=K_0^1\sqcup K_0^2$.
	
	Now, for $n\in\Z$ and $i\in \{1, 2\}$, write $I_n=f^n(I_0)$, $K_n^i=f^n(K_0^i)$,  and $I_n^i=f^n(I_0^i)$. Let us consider the following locally dyadic affine maps:
	\begin{itemize}
		\item $a\colon I_0^1\sqcup I_0^2\to I_0\sqcup I_1$ defined by
		\[a(x)=\left\{\begin{array}{lc} h_1^{-1}(x)  & \text{if } x\in I_0^1,  \\   f\circ h_2^{-1}(x) & \text{if } x\in I_0^2,   \end{array} \right.  \] 
		\item $b\colon I_{3}\sqcup I_4\to I_4^{1}\sqcup I_4^2$ defined by \[b(x)=\left\{\begin{array}{lc} f^4\circ h_1\circ f^{-3}(x)      & \text{if } x\in I_{3},  \\   
			f^4\circ h_2\circ f^{-4}   & \text{if }x\in I_{4},   \end{array} \right.  \]
		\item $c\colon [1,2]\to[2,3]$ defined by $c(x)=x+1$.
	\end{itemize}
	Then, we can apply Lemma \ref{sublem.gluing} to construct $h\in [F, F]$ which simultaneously extends $a, b, c$ and $\id\restriction_{(-\infty,0]\cup[4,+\infty)}$. By construction, $h$ preserves $K$ and has no fixed points in $[1,2]$. Thus, the subgroup $H=\{h\in [F, F] : h(K)=K\}$ has no  fixed points inside $[1, 2]$. Finally, note that $f$ normalizes $H$, so that it preserves its set of fixed points. Since $\bigcup_{n\in \Z} f^n([1, 2])=\R$, we deduce that $H$ has no fixed points on $\R$, whence on $K$.
\end{proof}

\part{Local rigidity and the space of harmonic actions}\label{partIII}

In this final part, we investigate the structure of the space of irreducible actions $\Homirr(G, \homeo_0(\R))$. In particular, we obtain a criterion for the local rigidity of the standard action of a locally moving subgroup $G\subset \homeo_0(\R)$ (corresponding from Theorem \ref{t-intro-local-rigidity} from the introduction). 

To this end, we develop a method to investigate the global structure of the space $\Homirr(G, \homeo_0(\R))$ for a finitely generated group $G$, based on the study of the subspace of so called normalized harmonic actions $\Der_\mu(G;\R)$, studied by Deroin, Kleptsyn, Navas, and Parwani \cite{DKNP}. The formal definition and fundamental properties are discussed in Chapter \ref{s.Deroin}. 
To define this, one has to fix a symmetric probability measure $\mu$ on $G$ whose support is finite and generates $G$: then a $\mu$-harmonic action of $G$ on the line is an action for which the Lebesgue measure is $\mu$-stationary. By the results in \cite{DKNP}, any irreducible action of $G$ on the line is semi-conjugate to a $\mu$-harmonic action, unique up to affine conjugacy. By a normalizing condition, we can reduce this extra symmetry to the group of isometries. This reduction leads to the definition of the compact space $\Der_\mu(G;\R)$ of normalized $\mu$-harmonic actions, with a natural topological flow defined on it by considering conjugacy of actions by translations. These facts can be deduced from previous works by Deroin and collaborators, and we summarize them in Theorem \ref{t.def_Deroin}. Our main result of Chapter \ref{s.Deroin} (Theorem \ref{t.retraction_Deroin}) states that $\Der_\mu(G;\R)$ is a retract of the space of irreducible actions of the line of $G$, by a retraction that preserves positive semi-conjugacy classes. This provides an explicit correspondence between the topological properties of semi-conjugacy classes in the space $\Homirr(G, \homeo_0(\R))$, and the dynamical properties of the translation flow on $\Der_\mu(G, \R)$. We use this to study how semi-conjugacy classes behave under small perturbations in the compact-open topology, and we are particularly interested in understanding when an action is locally rigid, meaning that sufficiently small perturbations do not change the positive semi-conjugacy class (see Corollary \ref{c-rigidity-Deroin}). The proof of Theorem \ref{t.retraction_Deroin} is discussed in \S \ref{sec.preorder_harmonic}, and passes through an alternative definition of $\Der_\mu(G;\R)$ (considered as an abstract topological space) in terms of left-invariant preorders on $G$. This is the content of Theorem \ref{thm.homeoderoin}, which gives as a by-product that the space $\Der_\mu(G;\R)$ does not depend on the choice of the probability measure $\mu$. In \S \ref{ssc.complexity} we use the properties of the space $\Der_\mu(G;\R)$ to understand the Borel complexity of the semi-conjugacy equivalence relation for irreducible group actions on the line.

In Chapter \ref{ch.laminar-harmonic} we discuss laminar actions on the line of finitely generated groups and their relation to the space of $\mu$-harmonic actions. For this, we first revisit the notion  of horograding, formalizing the idea that horogradings of laminar actions describe the large-scale dynamics of the action (see \S\ref{sec.largescale}). When working with harmonic actions, this description can be used to understand the closure of the subset of laminar actions in $\Der_\mu(G;\R)$ (see \S\ref{sec.laminar_harmonic}).

In the final chapter (Chapter \ref{ch.harm_locmov}) we put  the pieces together to study the space $\Der_\mu(G;\R)$ for locally moving subgroup $G\subset \homeo_0(\R)$, by combining the results from the previous chapters and the main results in Part \ref{partII}. Our main result states that, under a suitable finite generation condition on $G$, the translation flow on $\Der_\mu(G;\R)$ must have very restricted dynamics; see Theorem \ref{thm-decoDeroin}. As a corollary, we get the local rigidity of the standard action, which corresponds to Theorem \ref{t-intro-local-rigidity} from the introduction. In \S \ref{s-examples-class}, we discuss some concrete families of groups satisfying Theorem \ref{thm-decoDeroin},  including Thompson's group $F$, for which we provide a more precise description of the space $\Der_\mu(F;\R)$. Finally,  in \S \ref{s-example-non-lr}, we give an example of a finitely generated, locally moving subgroup of $\homeo_0(\R)$ whose standard action is not locally rigid: this shows that the additional assumptions on $G$ in Theorem \ref{thm-decoDeroin} cannot be removed.

\chapter{Normalized harmonic actions and preorders}
\label{s.Deroin}

Recall from the introduction that for a finitely generated group $G$, we denote by $\Homirr(G, \homeo_0(\R))$ the space of irreducible actions $\varphi\colon G\to \homeo_0(\R)$, endowed with the compact open topology. 

In this preliminary chapter we study the subspace $\Der_\mu(G)$ of (normalized) $\mu$-harmonic actions (associated to a suitable probability measure $\mu$ on $G$), whose construction is based on the work of Deroin, Kleptsyn, Navas, and Parwani \cite{DKNP} on symmetric random walks in $\homeo_0(\R)$. The main result of the chapter shows that the space $\Homirr(G, \homeo_0(\R))$ admits a continuous retraction onto $\Der_\mu(G;\R)$ which preserves semi-conjugacy classes; as a consequence, we derive a characterization of the locally rigid minimal actions of $G$ in terms of properties of $\Der_\mu(G;\R)$. The proof goes through an alternative description of $\Der_\mu(G;\R)$ as a quotient of the space of left preorders on $G$.

\section{Preliminiaries on harmonic actions}\label{sec.harmder}\label{ssubsec.Deroin}

Throughout the chapter, we will work in the following setting.

\begin{assumption}\label{ass.DKNP}
	We let $G$ be a finitely generated group, and fix a probability measure $\mu$ on $G$ whose support is finite and generates $G$, and which is symmetric in the sense that
	\[
	\mu(g)=\mu(g^{-1})\quad\text{for every }g\in G.
	\]
	We will denote by $S\subset G$ the support of $\mu$.
\end{assumption}

	Recall that given a nontrivial action $\varphi\colon G\to \homeo_{0}(\R)$, a Radon measure $\nu$ on $\R$ is \emph{stationary} for the action $\varphi$ (and the probability measure $\mu$) if for every Borel subset $A\subseteq \R$ one has
	\[
	\nu(A)=\sum_{g\in S}\nu(g^{-1}.A)\,\mu(g).
	\]

\begin{dfn}
		An action $\varphi\colon G\to \homeo_0(\R)$ is \emph{$\mu$-harmonic} if the Lebesgue measure on $\R$ is stationary.
\end{dfn}

Properties of $\mu$-harmonic actions are studied in \cite{DKNP}; a treatment can also be found in the monograph by Deroin, Navas, and the third author \cite[\S 4.4]{GOD}. We summarize some of them below. Recall from Chapter \ref{s-preliminaries} that we say that an action $\varphi\in \Homirr(G, \homeo_0(\R))$ is a \emph{canonical model} if it is either minimal, or \emph{cyclic} (the latter means that $\varphi$ is conjugate to an epimorphism to a cyclic group of translations). 

\begin{prop}[After \cite{DKNP}]\label{prop.minimalisharmonic}
	Under Assumption \ref{ass.DKNP}, the following properties hold.
	\begin{enumerate}[label=(\roman*)]
		\item Every non-trivial $\mu$-harmonic action $\varphi\colon G\to \homeo_{0}(\R)$ is a canonical model.
		\item Conversely, every canonical model $\varphi\colon G\to \homeo_{0}(\R)$ is conjugate to a  $\mu$-harmonic action $\psi\colon G\to \homeo_{0}(\R)$.
		
		\item\label{i:DKNP_affine} Moreover, any two semi-conjugate $\mu$-harmonic actions $\psi_1$ and $\psi_2$ of a group $G$ are conjugate by an affine homeomorphism. 
	\end{enumerate}
\end{prop}

After a careful reading of the proof of \cite[Proposition 8.1]{DKNP}, there is a natural way to reduce the symmetries in part \ref{i:DKNP_affine} of the statement above to the group of translations. For this, given a homeomorphism $h\in \homeo_{0}(\R)$ consider the following area function:
\[
\mathsf A^h(\xi)=\left\{\begin{array}{lr}
	\int_{h^{-1}(\xi)}^{\xi}(h(\eta)-\xi)\,d\eta&\text{if }h(\xi)\ge \xi,\\[.5em]
	\int_{h(\xi)}^{\xi}(h^{-1}(\eta)-\xi)\,d\eta&\text{if }h(\xi)\le \xi.\\
\end{array} \right.
\]
Note that $\mathsf A^h(\xi)\ge 0$ for every $\xi \in \R$, and $\mathsf A^h(\xi)=0$ if and only if $\xi\in \fix(h)$.
In the case $h(\xi)\ge \xi$, the quantity $\mathsf A^h(\xi)$ represents indeed the area of the bounded planar region delimited by the segments $[h^{-1}(\xi),\xi]\times \{\xi\}$, $\{\xi\}\times [\xi,h(\xi)]$, and the graph of $h$; when $h(\xi)\le \xi$ one has to switch the endpoints of the segments, and the same interpretation is valid. Being a two-dimensional area, it should be clear that when considering the map $\widetilde h=ah a^{-1}$, obtained by conjugating $h$ by an affine map $a(\xi)=\lambda \xi+b$, one has the following relation:
\begin{equation}\label{eq.change_area}
	\lambda^2 \mathsf A^h(\xi)= \mathsf A^{\widetilde h}(a(\xi))\quad\text{for every }\xi\in \R.
\end{equation}
\begin{lem}[After \cite{DKNP}]\label{l.constant_area}
	Under Assumption \ref{ass.DKNP}, let $\varphi\colon G\to \homeo_{0}(\R)$ be a $\mu$-harmonic action. Then the expected area function $\mathsf{A}^\varphi\colon \R\to \R$ defined by
	\[
	\mathsf A^\varphi(\xi)=\sum_{g\in S} \mathsf{A}^{\varphi(g)}(\xi)\,\mu(g),
	\]
	is constant and positive.
\end{lem}

\begin{dfn}\label{dfn.Deroin}
	Under Assumption \ref{ass.DKNP}, we introduce the space $\Der_\mu(G;\R)$ of \emph{normalized ($\mu$-)harmonic actions} as the subset of $\mu$-harmonic actions $\varphi\in \Hom(G,\homeo_{0}(\R))$ such that $\mathsf A^\varphi=1$, endowed with the induced compact-open topology.
\end{dfn}

Note that as the group $G$ is generated by the finite subset $S$, the compact-open topology of $\Hom(G,\homeo_{0}(\R))$ is simply given by the product topology of $\homeo_0(\R)^S$, where $\homeo_0(\R)$ is considered with the topology of uniform convergence on compact subsets.

\begin{thm}[After \cite{DKNP, GOD, deroin2020non}]\label{t.def_Deroin}
	Under Assumption \ref{ass.DKNP}, the following properties hold.
	\begin{enumerate}[label=(\roman*)]
		\item\label{i.def_Deroin} The space $\Der_\mu(G;\R)$ is compact.
		\item\label{ii.def_Deroin} Every irreducible action $\varphi\colon G\to \homeo_{0}(G)$ is positively semi-conjugate to a normalized $\mu$-harmonic action $\psi\in \Der_\mu(G;\R)$, which is unique up to conjugacy by a translation.
		\item\label{iii.def_Deroin} Conjugacy by translations defines a topological flow $\Phi\colon \R\times \Der_\mu(G;\R)\to \Der_\mu(G;\R)$, called the \emph{translation flow}. Explicitly, this is given by
		\[
		\Phi^t(\varphi)(g)=T_{-t}\varphi(g)T_{t},
		\]
		where $T_t\colon \xi\mapsto \xi+t$.
	\end{enumerate}
\end{thm}

\begin{proof}[Sketch of proof]
	After the identity \eqref{eq.change_area} and Lemma \ref{l.constant_area}, if $\varphi$ is a normalized $\mu$-harmonic action, then also $\Phi^t(\varphi)$ is normalized for every $t\in \R$. So the translation flow is well defined. 
	
	After Corollary \ref{cor.basica} and Proposition \ref{prop.minimalisharmonic}, we know that every irreducible action  $\varphi\colon G\to \homeo_{0}(\R)$ is semi-conjugate to a $\mu$-harmonic action, which is unique up to affine rescaling. Moreover, the identity \eqref{eq.change_area} and Lemma \ref{l.constant_area} give that after an affine rescaling, we can assume that the $\mu$-harmonic action is normalized. This gives \ref{ii.def_Deroin}.
	
The fact that $\Der_\mu(G;\R)$ is compact follows from the fact that in a $\mu$-harmonic action, every element $g\in G$ acts as a bi-Lipschitz homeomorphism with bounded displacement, with constants bounded uniformly over actions in $\Der_\mu(G;\R)$ \cite{DKNP}; further details can be found in the work of Deroin and Hurtado \cite[proof of Theorem 5.4]{deroin2020non}.
\end{proof}

Although this will not be used, note that the group $G$ acts on the space $\Der_\mu(G;\R)$ by the formula	$g.\varphi=\Phi^{\varphi(g)(0)}(\varphi)$ (see \cite[(4.5)]{GOD}). In other words, the action on the parameterized $\Phi$-orbit of $\varphi$ is basically the action $\varphi$.

Let us point out the following consequence of Theorem \ref{t.def_Deroin}.

\begin{cor}\label{rem.fixes0} \label{rem.conjclass} 
	Under Assumption \ref{ass.DKNP}, let $\varphi_1$ and $\varphi_2$ be two actions in $\Der_\mu(G;\R)$.
	\begin{enumerate}[label=(\roman*)]
		\item \label{i-rem.fixes0} If $\varphi_1$ and $\varphi_2$ are positively conjugate by a homeomorphism $h\colon \R\to\R$ that 
		fixes $0$, then $\varphi_1=\varphi_2$.
		\item \label{i-rem.conjclass}  The actions $\varphi_1$ and $\varphi_2$ are positively conjugate if and only if they belong to the same $\Phi$-orbit.
	\end{enumerate}
\end{cor}

\section{Retraction and local rigidity}\label{sec.retraction_Deroin}

Let us now state the main results of this chapter. From the probabilistic construction in \cite{DKNP}, it is not clear that the harmonic representative of an action $\varphi\in\Homirr(G, \homeo_0(\R))$ can be chosen to depend continuously on $\varphi$ (i.e.\ that sufficiently close actions have close harmonic representatives). The following result states that this is the case. 

\begin{thm}\label{t.retraction_Deroin} Let $G$ be a finitely generated group, and consider a symmetric probability measure $\mu$ whose support is finite and generates $G$. There exists a continuous retraction \[r\colon \Homirr(G,\homeo_0(\R))\to \Der_\mu(G;\R)\] which preserves positive semi-conjugacy classes.
\end{thm}

\begin{dfn}\label{d.retraction_Deroin}
	The map $r$ from Theorem \ref{t.retraction_Deroin} will be called the \emph{harmonic retraction} of the representation space $\Homirr(G,\homeo_0(\R))$. 
\end{dfn}

This has an application to the study of local rigidity of representations. 
The following notion is classical in dynamical systems.

\begin{dfn}
	An irreducible action $\varphi\colon G\to\homeo_0(\R)$ is \emph{locally rigid} if there exists a neighborhood  $\mathcal{U}$ of $\varphi$ in $\Homirr(G,\homeo_0(\R))$ consisting of representations all positively semi-conjugate to $\varphi$.
\end{dfn}

In other terms, every sufficiently small perturbation of the action $\varphi$ gives a positively semi-conjugate action. Theorem \ref{t.retraction_Deroin} has the following consequence. 

\begin{cor}[Local rigidity criterion]\label{prop.rigidityderoin}\label{c-rigidity-Deroin}
	Let $G$ be a finitely generated group, and consider a symmetric probability measure $\mu$ whose support is finite and generates $G$. For $\varphi\in \Homirr(G, \homeo_0(\R))$, let $\psi\in \Der_\mu(G;\R)$ be a normalized $\mu$-harmonic action which is positively semi-conjugate to $\varphi$. If the orbit along the translation flow of $\psi$ is open in $\Der_\mu(G;\R)$, then $\varphi$ is locally rigid. The converse holds provided $\varphi$ is minimal.
\end{cor}

\begin{proof}
Let $\mathcal{I}=\{\Phi^t(\psi):t\in\R\}$ be the orbit of $\psi$ in $\Der_\mu(G;\R)$ and let $r$
be the harmonic retraction (Theorem \ref{t.retraction_Deroin}). Then $r(\varphi)$ is an element of $\Der_\mu(G;\R)$ which is positively semi-conjugate to $\psi$ and thus belongs to $\mathcal{I}$, i.e.\ $\varphi\in r^{-1}(\mathcal{I})$. If  $\Ical$ is open, then $r^{-1}(\Ical)$ is an open neighborhood of $\varphi$ consisting only of actions positively semi-conjugate to $\psi$, so the claim follows. 
Conversely, assume that $\varphi$ is minimal. Since local rigidity is preserved under conjugacy,  we can assume that $\varphi=\psi$. Thus if $\mathcal{I}$ is not open, using the translation flow we find that $\varphi$ can be approximated by $\mu$-harmonic actions in different $\Phi$-orbits. Since actions in $\Der_\mu(G;\R)$ belonging to different $\Phi$-orbits are not positively semi-conjugate (\ref{i-rem.conjclass} in Corollary \ref{rem.conjclass}), this implies that $\varphi$ is not locally rigid.
\end{proof}

\begin{rem}\label{rem.stable_non-isolated}
	The requirement that $\varphi$ is minimal in the second part of  Corollary \ref{prop.rigidityderoin} is essential to ensure that the orbit under the translation flow of its semi-conjugate representative $\psi\in\Der_\mu(G;\R)$ is open.
%	The requirement that $\varphi$ is minimal in the second part of  Corollary \ref{prop.rigidityderoin} is essential. 
	To see this, consider the action $\varphi\colon \mathbb{F}_2\times\mathbb{Z}\to\homeo_0(\R)$ obtained as the lift of the ping-pong action $\varphi_0\colon \mathbb{F}_2\to\mathsf{PSL}(2,\R)\subset \homeo_{0}(\mathbb{S}^1)$, given by the action of the fundamental group of a hyperbolic \emph{one-holed} torus on the boundary at infinity of the hyperbolic plane; to make this more concrete, one can consider the subgroup of $\mathsf{PSL}(2,\R)$ generated by the two homographies $\begin{bmatrix}5& 3\\3 & 2 \end{bmatrix}$ and $\begin{bmatrix}1& 1\\1 & 2 \end{bmatrix}$.
The action $\varphi_0$ is then non-minimal and locally rigid, so the same holds for $\varphi$. However, the canonical model of $\varphi_0$ corresponds to the action of the fundamental group of a hyperbolic \emph{one-punctured} torus on the boundary at infinity of the hyperbolic plane, and so this canonical model has a parabolic element (the commutator of the generators). This can be used to show that this minimal action and its corresponding lift to the line are not locally rigid.  Thus the representative of $\varphi$ in $\Der_\mu(\mathbb{F}_2\times\Z;\R)$ has non-isolated orbit along the translation flow. 
\end{rem}

The proof of Theorem \ref{t.retraction_Deroin} is postponed to the end of the next section.

\section{Preorders and harmonic actions}\label{sec.preorder_harmonic}

Recall from \S\ref{s-preorders} that given a finitely generated group $G$, we denote by $\LPO(G)$ the space of non-trivial left-invariant preorders on $G$. To every preorder, we may associate an action on the line given by its dynamical realization (after Lemma \ref{lem.dynreal}, this is well defined up to positive conjugacy). In  \S \ref{sc.equivalence_preo}, we introduce an appropriate equivalence relation on the set of preorders, such that equivalent preorders give rise to positively semi-conjugate actions (through a semi-conjugacy preserving a marked point). In \S \ref{sc.preoders_dynamics}, we make explicit the relations between properties of preorders and their dynamical realizations. In \S \ref{sc.topology_preo}, we investigate  the topology of the corresponding quotient space. Finally, in \S \ref{sc.deroin_preo} we identify such a quotient space with the space of normalized harmonic actions, and provide the proof of Theorem \ref{t.retraction_Deroin}.
As a by-product of this, we will get that the topology of $\Der_\mu(G;\R)$ does not depend on the choice of the probability measure $\mu$.

\subsection{Equivalence of preorders}\label{sc.equivalence_preo} 
Part of our discussion in this subsection is  close to the exposition of Decaup and Rond \cite{decaup2019preordered}. Before going on, let us recall from Remark \ref{r.cones} that a preorder $\le\in \LPO(G)$ corresponds to a partition
$G=P_{\le}\sqcup [1]_{\le}\sqcup P^{-1}_{\le}$,
where $P_\le=\{g\in G:g\gneq 1\}$ is positive cone of $\le$ (which is a semigroup), and $[1]_\le=\{g\in G:g\le 1\le g\}$ is the residue (which is a subgroup of $G$). For further basic terminology, we refer more generally to \S \ref{sc.preorders}.

\begin{dfn} \label{d-preorder-dominates}Let $\leq,\preceq\in\LPO(G)$ be two preorders on a group $G$. We say that $\leq$ \emph{dominates} $\preceq$ if the identity map $\id\colon (G,\preceq)\to (G,\leq)$ is order-preserving.  
\end{dfn}

\begin{rem}The direction in the definition of domination could appear counterintuitive, but in fact it is justified when thinking in terms of preorders (it is even clearer in terms of dynamical realization). Indeed, take an element $g\in G$, and suppose that we want to know whether $g\succneq 1$. Then we first check whether $g\gneq 1$ (or $g\lneq 1$), and only in the case $g\in [1]_{\leq}$ we take the preorder $\preceq$ into consideration.
\end{rem}

Recall that given a preorder $\leq$ on $G$ and a $\leq-$convex subgroup $H$, the preorder $\leq$ descends to a preorder on $G/H$, denoted as $\leq_{G/H}$, and defines the quotient preorder $\le_H$ on $G$ by pull-back via the projection $p\colon G\to G/H$ (or more concretely, defined by the positive cone $P_{\le_H}=P_{\le}\setminus H$).

\begin{rem}\label{r.domination_po}
	After Remark \ref{r.pullback_po}, the identity map $\id\colon (G,\leq)\to(G,\leq_H)$ is order-preserving and equivariant with respect to the actions of $G$ by left multiplication.
\end{rem}

Note that when $\leq$ dominates $\preceq$, the subgroup $H=[1]_{\le}$ is $\preceq$-convex, thus the quotient preorder $\preceq_{H}$ is well defined. 
The following lemma characterizes domination in terms of positive cones. 

\begin{lem}\label{lem.dominequiv}
	Let  $\leq,\preceq\in\LPO(G)$ be two preoders on a group $G$.
	Then the following are equivalent.
	\begin{enumerate}[label=(\roman*)]
		\item $\leq$ dominates $\preceq$.\label{i.dominequiv}
		\item The subgroup $H=[1]_{\leq}$ is $\preceq$-convex, and $P_\le=P_\preceq\setminus H$ (so that $\le$ coincides with the quotient preorder $\preceq_H$).\label{iii.dominequiv}
		\item $P_{\leq}\subseteq P_{\preceq}$.\label{ii.dominequiv}
	\end{enumerate}
\end{lem} 
\begin{proof}
	We prove that \ref{i.dominequiv} implies \ref{iii.dominequiv}. As $\id\colon (G,\preceq)\to (G,\le)$ is order preserving, the subgroup $H=[1]_\le$, which coincides with its preimage, is $\preceq$-convex.
	Moreover, we have the following inclusions:
	\[
	P_{\preceq}\subseteq P_\le \sqcup [1]_\leq,\quad [1]_{\preceq}\subseteq [1]_{\leq},\quad P^{-1}_{\preceq}\subseteq [1]_\leq\sqcup P^{-1}_{\leq}.
	\]
	As $G=P_\preceq\sqcup[1]_\preceq\sqcup P^{-1}_{\preceq}$, we deduce $P_\preceq = P_\leq \setminus [1]_\leq$.	
	Clearly \ref{iii.dominequiv} implies \ref{ii.dominequiv}.
	Finally, by left invariance, we have $g\preceq h$ if and only if $h^{-1}g\in G\setminus P_\preceq$, and clearly the same holds for the preorder $\le$. This gives that \ref{i.dominequiv} and \ref{ii.dominequiv} are equivalent.
\end{proof}

For a given a preorder $\leq\in\LPO(G)$ on a group $G$, 
the collection of proper $\le$-convex subgroups is totally ordered by inclusion. The subgroup $[1]_\le$ is always the least such subgroup, however a maximal proper $\le$-convex subgroup may not exist.
This issue is the order-theoretic analogue of the problem of the existence of a minimal invariant set (Remark \ref{r.maynotexist}). To simplify the discussion, we will systematically assume finite generation in the rest of the section, in which case there exists a maximal proper $\le$-convex subset, that we will denote as $H_\leq$.
This leads to the notion of minimal model of a preorder which is an order-theoretic analogue of canonical models for actions on the line. 

\begin{dfn} Let $G$ be a finitely generated group. The \emph{minimal model} of a preorder $\leq\in\LPO(G)$ is the quotient preorder $\le_H$, where $H=H_{\le}$ is the maximal proper $\le$-convex subgroup. The minimal model will be denoted as $\leq^\ast$. When the preorder $\leq$ coincides with $\leq^\ast$, we say that $\leq\in\LPO(G)$ is a minimal model.
\end{dfn}

\begin{rem}\label{r.domination_minimal}\label{prop.minimalconvex}
	Minimal model preorders are exactly those whose structure of convex subgroups is the simplest possible. 
	Indeed, after Lemma \ref{lem.dominequiv}, the minimal model dominates the original preorder. Therefore, $H_\le$ is the unique proper $\le^*$-convex subgroup of $(G,\le^*)$, and thus $H_\le=[1]_{\leq^*}$.
	Reversely, if $[1]_\le=H_\le$ is the unique proper $\le$-convex subgroup, then $\le$ is a minimal model.
\end{rem}

It should be not surprising that the minimal model is unique, in the following sense.

\begin{lem}\label{prop.caracterdom} 
	Consider preorders $\leq,\preceq\in\LPO(G)$ on a finitely generated group $G$, such that $\leq$ is a minimal model and dominates $\preceq$. Then $\leq$ is the minimal model of $\preceq$.  
\end{lem}
\begin{proof}
	As $\le$ dominates $\preceq$, the subgroup $H=[1]_\leq$ is $\preceq$-convex, so $H\subseteq H_\preceq$. This gives $P_\leq\supseteq P_{\preceq^*}$, so by Lemma \ref{lem.dominequiv}, the map
	\[
	\id\colon (G,\le)\to (G,\preceq^*)
	\]
	is order preserving. Using Remark \ref{prop.minimalconvex}, as $H_\preceq=[1]_{\preceq^*}$, we deduce that $H_\preceq$ is $\le$-convex, and thus that $H_\preceq=H$.
\end{proof}

We next introduce the following equivalence relation on preorders.

\begin{dfn} We say that two preorders on a (finitely generated) group $G$ are \emph{equivalent} if they have the same minimal model. We denote by $[\leq]$ the equivalence class of $\leq\in\LPO(G)$ and we write $[\LPO](G)=\{[\leq]:\leq\in\LPO(G)\}$ for the corresponding quotient.\end{dfn}  

It is immediate to verify that this is indeed an equivalence relation.
We have the following result.

\begin{prop}\label{prop.starequiv}
	Consider preorders $\leq_1,\leq_2\in\LPO(G)$ on a finitely generated group $G$. Then, there exists a preorder $\le$ that dominates $\leq_1$ and $\leq_2$ if and only if they are equivalent.
\end{prop}

\begin{proof}
	We have that the minimal order $\le^*$ dominates $\le$ and thus both $\le_1$ and $\le_2$. From Lemma \ref{prop.caracterdom}, we conclude that $\le^*$ is the minimal model of both $\le_1$ and $\le_2$. The converse statement follows directly from Remark \ref{r.domination_minimal}.	
\end{proof}

\subsection{Relations between preorders and actions on the real line}\label{sc.preoders_dynamics}

We can now make explicit the relation between minimal models and their dynamical counterpart. For this, note that if  $\le\in \LPO(G)$ is a preorder on a finitely generated group $G$, writing $H=[1]_\le$, we can consider a \emph{dynamical realization} $\varphi_{\leq}$ for the action of $G$ on the totally ordered set $(G/H,<_{G/H})$; see Lemma \ref{lem.dynreal}.
Conversely, given an irreducible action $\varphi\colon G\to\homeo_{0}(\R)$, we denote by $\leq_\varphi\in \LPO(G)$ the preorder induced by the $\varphi$-orbit of $0$. 

\begin{prop}\label{lem.minimalisminimal}
	Let $G$ be a finitely generated group and $\varphi\colon G\to\homeo_0(\R)$ a canonical model. Then the preorder $\leq_\varphi$ is a minimal model.
\end{prop}
\begin{proof}
	Suppose by way of contradiction that there exists a proper $\leq_\varphi$-convex subgroup $H\neq [1]_{\leq_\varphi}$. Then, for each $gH\in G/H$, denote by $I_{gH}\subset \R$ the interior of the convex hull of $\varphi(gH)(0)$. Since $H$ is $\leq_\varphi$-convex, we get that $\R\setminus \bigcup_{g\in G}I_{gH}$ is a proper closed $\varphi$-invariant subset. On the other hand, since $H$ strictly contains $[1]_{\leq_\varphi}$, the stabilizer in $\varphi(G)$ of each $I_g$ acts non-trivially on $I_g$, and this implies that the action $\varphi$ is not cyclic, a contradiction.
\end{proof}

\begin{prop}\label{prop.minimalmodel}
	Consider a preorder $\leq\in\LPO(G)$ on a finitely generated group $G$. Then, the dynamical realization $\varphi_\leq$ is a canonical model if and only if $\leq$ is a minimal model.
	
	Moreover, the dynamical realization $\varphi_{\leq^\ast}$ is a canonical model for $\varphi_{\leq}$.
\end{prop}

The rest of the subsection is devoted to the proof of Proposition \ref{prop.minimalmodel}. We will need a preliminary result.
\begin{lem}\label{lem.domsemicon}
Consider preorders $\leq_1,\leq_2\in\LPO(G)$ on a group $G$, with $\leq_2$ dominating $\leq_1$. Then the dynamical realizations $\varphi_{\leq_1}$ and $\varphi_{\leq_2}$ are positively semi-conjugate. 
\end{lem}
\begin{proof}
	For $i\in \{1,2\}$ we write $H_i:=[1]_{\leq_i}$ and $\iota_i\colon G/H_i\to \R$ for the corresponding good embeddings (see Definition \ref{d-dynamical-realisation}). As $\id\colon (G,\leq_1)\to(G,\leq_2)$ is order preserving, we have $H_1\subseteq H_2$ and the quotient map
	\[i_0\colon (G/H_1,<_1)\to(G/H_2,<_2)\]
is order preserving and equivariant (here $<_i$ denotes the total ordered induced by $\leq_i$ for $i\in\{1,2\}$). Write $\Omega=\iota_1(G/H_1)$, and define $j_0\colon \Omega\to\R$ as
\[j_0(\iota_{1}(gH_1))=\iota_{2}(i_0(gH_1))=\iota_{2}(gH_2).\]
It follows directly from the definitions that $\Omega$ is $\varphi_{\leq_1}$-invariant and that $j_0$ is order preserving and equivariant. We conclude using Lemma \ref{lem.semiconjugacy}.
\end{proof}

Let us immediately point out  the following conclusion.

\begin{cor}\label{cor.equivsemicon}
	For finitely generated groups, equivalent preorders have positively semi-conjugate dynamical realizations.
\end{cor}
\begin{proof}
	After Lemma \ref{lem.domsemicon}, the dynamical realizations of equivalent preorders are positively semi-conjugate to the dynamical realization of their minimal model. As positive semi-conjugacy is an equivalence relation, the conclusion follows.
\end{proof}

\begin{proof}[Proof of Proposition \ref{prop.minimalmodel}] Suppose first that $\varphi_\leq$ is not a canonical model. Then $\varphi_\leq$ is neither minimal nor cyclic, and so either $\varphi_\leq$ has an exceptional minimal invariant set $\Lambda\subset \R$,
or $\varphi_\leq$ has a discrete orbit but its image is not cyclic. 
We write $\Omega=\iota(G/[1]_\leq)$ for the image of the  good embedding $\iota\colon G/[1]_\leq\to\R$ associated with $\varphi_\leq$, and remark that $\Omega$ consists of a single $\varphi_{\leq}$-orbit.

\setcounter{case}{0}

\begin{case}\label{c.minimalmodel}$\varphi_\leq$ has an exceptional minimal invariant set.
\end{case}

We first show the following.

\begin{claim}
	$\Omega$ is contained in $\R\setminus\Lambda$.
\end{claim}
\begin{proof}[Proof of claim]
	We assume by contradiction that $\Omega\cap \Lambda\neq\varnothing$. As $\Omega$ consists of a single orbit, by invariance of $\Lambda$ we get that $\Omega\subset \Lambda$, and thus $\Lambda=\overline{\Omega}$ by minimality of $\Lambda$.
	Consider  a connected component $I=(\xi,\eta)$ of $\R\setminus\Lambda$ and note that, by Definition \ref{dfn.goodbehaved} of a good embedding, we  have $\{\xi,\eta\}\subseteq \Omega$ and thus  the points $\xi$ and $\eta$ are in the same orbit. This shows that $\Lambda$ is discrete, which is a contradiction.
\end{proof}

Consider now the connected component $U$ of $\R\setminus\Lambda$ that contains $\xi_0=\iota([1]_\leq)$, and write $H=\stab^{\varphi_{\leq}}_{G}(U)$. Note that $H=\iota^{-1}(U)$ is a $\leq$-convex subgroup. We will show that $H$ strictly contains $[1]_\leq$, which by Remark \ref{prop.minimalconvex} implies that $\leq$ is not a minimal model. Equivalently, we need to show that $U\cap \Omega$ strictly contains $\xi_0$. Assume that $\xi_0$ is an isolated point of $\Omega$, otherwise there is nothing to prove. 
Let $\lambda\in \Lambda$ be the rightmost point of $U$. Note that $\overline \Omega$ is $\varphi_{\leq}$-invariant, so we must have $\overline \Omega\supseteq \Lambda$. Thus  $\lambda\in \overline{\Omega}$. If $\Omega\cap (\xi_0,\lambda)=\varnothing$, then $(\xi_0,\lambda)$ is a connected component of $\R\setminus \overline{\Omega}$. As $\iota$ is a good embedding, this gives $\lambda\in \Omega$, contradicting the claim.

\begin{case} $\varphi_\leq$ has a closed discrete orbit but its image is not cyclic.
\end{case}
\setcounter{case}{0}
	
Take a point $\xi\in\R$ with closed discrete orbit and write $H=\stab^{\varphi_{\leq}}_G(\xi)$ and $F=\fix^{\varphi_{\leq}}(H)$. Note that $H$ is the kernel of the morphism $G\to \Z$ induced from the orbit of $\xi$.
Since $\varphi_\leq$ is assumed to be non-cyclic, we have that $F$ is a proper closed subset, and it is $\varphi_{\leq}$-invariant, for $H$ is normal.
If $\Omega\cap F\neq\varnothing$, then the fact that $\Omega$ is a single $\varphi_{\leq}$-orbit gives $\Omega\subseteq F=\fix^{\varphi_{\leq}}(H)$; as $\varphi_{\leq}$ is a dynamical realization of $\le$, this implies $H=\varphi_{\leq}([1]_{\leq})$, so the image $\varphi_{\leq}(G)$ is cyclic. A contradiction.
Thus $\Omega\subseteq\R\setminus F$. More precisely, we have the following.

\begin{claim}
	$F=\fix(H)$ is contained in $\overline \Omega\setminus \Omega$.
\end{claim}
\begin{proof}[Proof of claim]
	Assume by way of contradiction that there is a point of $F$ in the complement of $\overline \Omega$, and let $I$ be corresponding connected component of the complement. As the action $\varphi_{\leq}$ is a dynamical realization (Definition \ref{d-dynamical-realisation}) and $F=\fix^{\varphi_{\leq}}(H)$, this gives that the closure $\overline I$ is fixed by $H$, so that $\overline I\subseteq F$. However, $\iota$ is a good embedding, so $\partial I\in \Omega$, which contradicts the fact that $\Omega\cap F=\varnothing$.
\end{proof}

We can now argue analogously as in Case \ref{c.minimalmodel} to find a proper $\leq$-convex subgroup strictly containing $[1]_\leq$.

For the converse, assume that $\leq$ is not a minimal model and take  a proper $\leq$-convex subgroup $H\neq[1]_\leq$. Denote by $U$ the interior of the convex hull of $\iota(H)\subset \R$ and notice that the orbit $\varphi_\leq(G)(U)$ is a proper open $\varphi_\leq$-invariant subset. Also notice that the stabilizer $\stab_G^{\varphi_{\leq}}(U)$ acts non-trivially on $U$, which implies that $\varphi_\leq$ is not a cyclic action. The last two facts together imply that $\varphi_\leq$ is not a canonical model. 

It remains to prove that $\varphi_{\leq^\ast}$ is a canonical model for $\varphi_\leq$. After Remark \ref{r.domination_minimal} and Lemma \ref{lem.domsemicon}, the actions $\varphi_{\leq^\ast}$ and $\varphi_\leq$ are positively semi-conjugate. On the other hand, by the first part of this proposition, the action $\varphi_{\leq^\ast}$ is a canonical model. 
\end{proof}

\subsection{Topology of the space of equivalence classes of preorders}
\label{sc.topology_preo}
Recall that we consider $\LPO(G)$ as a subspace of $\{\leq,\gneq\}^{G\times G}$ endowed with the product topology, and that this makes $\LPO(G)$ a  metrizable and totally disconnected topological space, which is compact when $G$ is finitely generated. 

\begin{lem}\label{lem.closedrelation}
	Let $G$ be a finitely generated group. The subset \[\mathcal{E}=\{(\leq_1,\leq_2)\in\LPO(G)\times\LPO(G):[\leq_1]=[\leq_2]\}\]
	 is closed in $\LPO(G)\times\LPO(G)$. Therefore, $[\LPO](G)$ considered with the quotient topology is a Hausdorff topological space.
\end{lem}
\begin{proof}
	Consider a convergent sequence $(\preceq_n,\preceq_n')\to(\preceq_{\infty},\preceq_\infty')$ so that $(\preceq_n,\preceq_n')\in\mathcal{E}$ for every $n\in\N$. We want to show that $[\preceq_\infty]=[\preceq'_{\infty}]$ and this amounts to show, after Proposition \ref{prop.starequiv}, that there exists a preorder that dominates both $\preceq_\infty$ and $\preceq_\infty'$.

For every $n\in \N$, using Proposition \ref{prop.starequiv} we  find a preorder $\leq_n$ that dominates both $\preceq_n$ and $\preceq_n'$. As $\LPO(G)$ is compact, upon passing to a subsequence, we can suppose that $\leq_n$ has a limit $\leq_{\infty}$. We shall prove that $\leq_{\infty}$ dominates both $\preceq_{\infty}$ and $\preceq_{\infty}'$, and  by Lemma \ref{lem.dominequiv} this amounts to show that $P_{\leq_\infty}\subseteq (P_{\preceq_\infty}\cap P_{\preceq_\infty'})$. For this, note that by definition of the product topology, if $g\in P_{\leq_\infty}$ then $g\in P_{\leq_n}$ for $n$ large enough, which by Lemma \ref{lem.dominequiv} implies that $g\in P_{\preceq_n}\cap P_{\preceq_n'}$ for $n$ large enough. This implies that $g\in P_{\preceq_\infty}\cap P_{\preceq_\infty'}$, as desired.
\end{proof}

\begin{lem}\label{lem.convcriteria}
	Consider left-invariant preorders $(\leq_n)_{n\in\N}$ and $\leq$ on a finitely generated group $G$. Assume that for every finite subset $F\subset P_{\leq}$ there exists $n_0\in\N$ such that $F\subset P_{\leq_n}$ for every $n\geq n_0$. Then $[\leq_n]\to[\leq]$ in $[\LPO](G)$ as $n\to \infty$.
	
	In particular, if $\varphi_n,\varphi\colon G\to\homeo_0(\R)$ are representations such that $\varphi_n\to\varphi$ in the compact-open topology, then we get $[\leq_{\varphi_n}]\to[\leq_\varphi]$ in $[\LPO](G)$, as $n\to \infty$. 
\end{lem}
\begin{proof}
	By hypothesis, we have that the limit of every convergent subsequence $\leq_{{k_n}}\to \leq_\infty$ satisfies $P_{\leq}\subseteq P_{\leq_\infty}$. This implies that $\leq$ dominates $\leq_\infty$ and therefore, by Proposition \ref{prop.starequiv}, we get $[\le]=[\le_\infty]$. Since every convergent subsequence of $(\leq_{n})_{n\in\N}$ converges to a preorder equivalent to $\leq$, we conclude that $[<_n]\to[<]$ in $[\LPO](G)$, as desired.
\end{proof}

\subsection{Identification with the space of normalized harmonic actions and harmonic retraction}\label{sc.deroin_preo}

As in Assumption \ref{ass.DKNP},  we fix a finitely generated group $G$ and a symmetric probability measure $\mu$ on $G$ whose support is finite and generates $G$.
We will show that $[\LPO](G)$ with the quotient topology is homeomorphic to $\Der_\mu(G;\R)$. In particular, this will show that the topology of the space of normalized harmonic actions does not depend on the choice of the probability measure $\mu$.

\begin{thm}\label{thm.homeoderoin} Under Assumption \ref{ass.DKNP}, the space $\Der_\mu(G;\R)$ is homeomorphic to $[\LPO](G)$. In particular, the homeomorphism type of $\Der_\mu(G;\R)$ does not depend on $\mu$. 
\end{thm}
\begin{proof} Consider the map
	\[\dfcn{I}{\Der_\mu(G;\R)}{[\LPO](G)}{\varphi}{[\leq_\varphi].}\]
	We will show that $I$ is a homeomorphism. By Proposition \ref{prop.minimalmodel} we know that if $\preceq\in\LPO(G)$ is a minimal model, then $\varphi_\preceq$ is a canonical model. On the other hand, by Proposition \ref{prop.minimalisharmonic}, we can choose $\varphi_\preceq$ to be $\mu$-harmonic.
	We introduce the map
	\begin{equation}\label{eq.J}
		J\colon [\LPO](G)\to \Der_\mu(G;\R),
	\end{equation}
	where $J$ assigns to $[\preceq]$ the $\mu$-harmonic dynamical realization of $\preceq^\ast$, whose associated good embedding satisfies $\iota_{\preceq^\ast}([1]_{\preceq^\ast})=0$.
	We will show that $J$ is the inverse of $I$.
	The equality $I\circ J=\Id$ is given by the fact that 
	$\leq_{\varphi_\leq}$ coincides with $\leq$ for every preorder $\leq\in\LPO(G)$, which is a direct consequence of the definitions. 
	
	To verify that $J\circ I=\Id$, fix $\varphi\in \Der_\mu(G;\R)$. Since $\varphi$ is a canonical model (Proposition \ref{prop.minimalisharmonic}), Proposition \ref{lem.minimalisminimal} implies that $\leq_\varphi$ is a minimal model. Therefore, $J([\leq_\varphi])$ is a dynamical realization of $\leq_\varphi$. Set $\Omega=\varphi(G)(0)$, and notice that there exists an order-preserving map $j_0\colon \Omega\to\R$ which is equivariant for the actions $\varphi$ and $J([\leq_\varphi])$. Since both actions are canonical models, it follows from Lemma \ref{lem.semiconjugacy} and Corollary \ref{cor.basica} that $j_0$ can be extended to a positive conjugacy. Moreover, since $j_0(0)=0$, this conjugacy fixes $0$ and therefore, by \ref{i-rem.fixes0} in Corollary \ref{rem.fixes0}, we conclude that $\varphi=J([\leq_\varphi])$, as desired. 

The continuity of $I$ follows directly from Lemma \ref{lem.convcriteria}. Finally, since $\Der_\mu(G;\R)$ is compact (Theorem \ref{t.def_Deroin}) and $[\LPO](G)$ is a Hausdorff topological space (Lemma \ref{lem.closedrelation}), we conclude that $I$ is a homeomorphism. 
\end{proof}

We can now prove the main result of this section.

\begin{proof}[Proof of Theorem \ref{t.retraction_Deroin}]
We consider the map
\[
\dfcn{r}{\Homirr(G,\homeo_0(\R))}{\Der_\mu(G;\R)}{\varphi}{J([\leq_\varphi]),}
\]
where $J$ is the map introduced in \eqref{eq.J} and $\leq_\varphi$ is the preorder induced by $0$. It follows directly from the definitions that $r$ is the identity in restriction to $\Der_\mu(G;\R)$. Also, by Lemma \ref{lem.convcriteria} and Theorem \ref{thm.homeoderoin} we conclude that $r$ is continuous. Finally, to check that $r$ preserves positive semi-conjugacy classes note that $\varphi$ is positively semi-conjugate to the dynamical realization of $\leq_\varphi$ and that, by Corollary \ref{cor.equivsemicon}, the dynamical realizations of $\leq_\varphi$ and $\leq_\varphi^\ast$ are also positively semi-conjugate.
\end{proof}

\section{Application to the Borel complexity of the semi-conjugacy relation}\label{ssc.complexity}

For a given group $G$ it is natural to try to determine how difficult is to distinguish actions in $\Homirr(G, \homeo_0(\R))$ up to semi-conjugacy, and whether it is possible to  classify them. 
To conclude this chapter, we highlight that Theorem \ref{t.retraction_Deroin}   sheds light on this question, and suggests a line of research for group actions on the line. 
An appropriate framework to formalize the question is  the theory of \emph{Borel reducibility} of equivalence relations.  Most notions of isomorphisms can be naturally interpreted as  equivalence relations on some standard Borel space, so that this theory  offers tools to study various classification problems and compare the difficulty of one with respect to another. 
Let us recall some basic notions from this setting following the exposition of Kechris \cite{KechrisCBE} (to which we refer for more details). Recall that a standard Borel space $\mathcal{Z}$ is a measurable space isomorphic to a complete separable metric space endowed with its Borel $\sigma$-algebra. A subset of such a space is \emph{analytic} if it is the image of a Borel set under a Borel map from a standard Borel space. An analytic equivalence relation on $\mathcal{Z}$ (henceforth just \emph{equivalence relation}) is an equivalence relation $\mathcal{R}$  which is an analytic subset  $\mathcal{R}\subset \mathcal{Z}\times \mathcal{Z}$. This standing assumption is convenient to develop the theory, and general enough to model most natural isomorphism problems.   It includes in particular the class of Borel equivalence relations (those which are Borel subsets of $\mathcal{Z}\times \mathcal{Z}$), which is much better behaved, but too restricted for some purposes (see below). We write $x\mathcal{R}y$ when $(x, y)\in \mathcal{R}$

 For $i\in\{1, 2\}$, let $\mathcal{R}_i$ be an equivalence relation on a standard Borel space $\mathcal{Z}_i$. We say that $\mathcal{R}_1$ is \emph{reducible} to $\mathcal{R}_2$, and write $\mathcal{R}_1\le_B \mathcal{R}_2$, if there exists a Borel map $q\colon \mathcal{Z}_1\to \mathcal{Z}_2$ such that $x\mathcal{R}_1y$ occurs if and only if $q(x)\mathcal{R}_2 q(y)$. This means that distinguishing classes of $\mathcal{R}_1$ is ``easier'' than for $\mathcal{R}_1$.  If $\mathcal{R}_1\le_B\mathcal{R}_2$ and $\mathcal{R}_2\le_B \mathcal{R}_1$, we say that $\mathcal{R}_1$ and  $\mathcal{R}_2$ are \emph{bireducible}.    
 
The  simplest type of equivalence relations from the perspective of reducibility are the \emph{smooth} ones. An equivalence relation $\mathcal{R}$ on $\mathcal{Z}$ is said to be {smooth} if there exist a standard Borel space $\mathcal{W}$ and a Borel map $q\colon \mathcal{Z}\to \mathcal{W}$ which is a complete invariant for the equivalence classes of $\mathcal{R}$, in the sense  that $x\mathcal{R} y$ if and only if $q(x)=q(y)$. 

The next class of  equivalence relations are the \emph{hyperfinite} ones. An equivalence relation $\mathcal{R}$ is {hyperfinite} if it is a countable union $\mathcal{R}=\bigcup\mathcal{R}_n$ of equivalence relations whose classes are finite, and \emph{essentially hyperfinite} if it is bireducible to a hyperfinite one.   An example of a hyperfinite equivalence relation which is not smooth, denoted as $\mathcal{E}_0$, is the equivalence relation on the set of one-sided binary sequence $\{0,1\}^\N$, where $(x_n)$ and $(y_n)$ are equivalent if $x_n=y_n$ for large enough $n$. In fact (after Harrington, Kechris, and Louveau \cite{HKL}), up to bireducibility, the relation $\mathcal{E}_0$ is the unique essentially hyperfinite equivalent relation which is not smooth, and moreover every Borel equivalence relation $\mathcal{R}$ is either smooth, or satisfies $\mathcal{E}_0\le_B \mathcal{R}$, so that $\mathcal{E}_0$ can be thought of as the {simplest} non-smooth Borel equivalence relation. 
 The essentially hyperfinite equivalence relations form a strict subset of the essentially countable ones (those that are bireducible to a relation whose classes are countable). These are precisely those induced by orbits of actions of countable groups, and the poset of bireducibility types of such relations  is quite complicated (see Adams ans Kechris \cite{AdamsKechris}). Essentially countable relations are themselves a strict subset of Borel equivalence relations, after which we find general (analytic) equivalence relations. 

 The bireducibility type of the conjugacy relation of various classes of dynamical systems and group actions has been extensively studied. For many classes of topological or measurable dynamical systems, the conjugacy relations is known or conjectured to be quite complicated from the perspective of bireducibility; see for instance the works of Foreman, Rudolph, and Weiss \cite{MR2800720}, Gao, Jackson, and Seward \cite{GJS}, Sabok and Tsankov \cite{TS}, and Le Roux \cite{LR-Borel}. As an example,  the conjugacy relation of elements of $\homeo_0(\R)$ is bireducible to the isomorphism relation of all countable (but not necessarily locally finite) graphs, which is analytic but {not} Borel (see Hjorth \cite[Theorem 4.9]{Hjorth}), and the conjugacy relation of homeomorphisms of the plane is strictly more complicated by a result of Hjorth \cite[Theorem 4.17]{Hjorth}. As a consequence, one may expect that the semi-conjugacy relation on the space of irreducible action of a given group $G$ should also be complicated and might not even be Borel. In  contrast,  the existence of the space of normalized harmonic actions and  Theorem \ref{t.retraction_Deroin} imply the following.

\begin{cor} \label{c-borel-hyperfinite}
Let $G$ be a finitely generated group. Then the semi-conjugacy relation on the space $\Homirr(G, \homeo_0(\R))$ is essentially hyperfinite (in particular, it is Borel). 
\end{cor}
\begin{proof}
Theorem  \ref{t.retraction_Deroin} shows that this relation is bireducible to the orbit equivalence relation on the space $\Der_\mu(G;\R)$ induced by the translation flow $\Phi$. 
 On the other hand, the orbit equivalence relation of any Borel flow on a standard Borel space is essentially hyperfinite (see Kechris \cite[Theorem 8.33]{KechrisCBE}). \qedhere
\end{proof}

After Corollary \ref{c-borel-hyperfinite} and the previous discussion, we may further distinguish two cases: either the semi-conjugacy relation on $\Homirr(G, \homeo_0(\R))$ is smooth, or it is not, in which case it is bireducible to $\mathcal{E}_0$. As mentioned above, the first case corresponds to groups for which actions up to semi-conjugacy can be completely classified by a Borel invariant, so that it is natural to ask which groups have this property.    This can be characterized in terms of a very restricted behavior for the dynamics of the flow $\Phi$ on $\Der_\mu(G; \R)$. Recall that given a flow $(\Psi, Y)$ on a topological space, a point $y\in Y$ is said to be \emph{recurrent} if for every neighborhood $U$ of $y$, the set of times  $\{t\in \R: \Psi^t(y)\in U\}$ is unbounded (above or below), and \emph{periodic} if $\Psi^t(y)=y$ for some $t\neq 0$. 

\begin{cor}\label{cor.cross-section}
	Under Assumption \ref{ass.DKNP}, the semi-conjugacy relation on the space of irreducible actions $\Homirr(G, \homeo_0(\R))$ of a finitely generated group $G$ is  smooth if, and only if, every recurrent point of the flow $(\Der_\mu(G;\R), \Phi)$ is periodic.
\end{cor}
\begin{proof}
By Theorem \ref{t.retraction_Deroin}, the semi-conjugacy relation on $\Homirr(G, \homeo_0(\R))$ is smooth if and only if so is the orbit equivalence relation of the flow $(\Der_\mu(G;\R), \Phi)$. A  result of Effros implies that the orbit equivalence relation induced by a flow $\Phi$ on a Polish space is smooth if and only if every recurrent point of $\Phi$ is periodic, see the book of Becker and Kechris \cite[Theorem 3.4.2]{BeckerKechris}.
\end{proof}
 Corollary \ref{cor.cross-section}  suggests the following general question on groups acting on the line.
 
 \begin{ques}
 Which finitely generated groups satisfy the properties in Corollary \ref{cor.cross-section}?
 \end{ques} 
 
 We will present some results in this direction in \S \ref{ss-lr}.

\begin{rem}\label{rem non smooth}
There do exist finitely generated groups $G$ for which the semi-conjugacy relation on $\Homirr(G, \homeo_0(\R))$ is not smooth (and in particular, the space of semi-conjugacy classes is not a standard measurable space). For example, it is not difficult to show this for the free group $\mathbb{F}_2$ as follows. Fix a ping-pong pair of homeomorphisms $g, h$ of $\R/\Z$, with  $\fix(g)=\{0, 1/2\}$ and $\fix(h)=\{1/4, 3/4\}$, where both $g, h$ have one attracting and one repelling fixed point, and such that $\langle g, h\rangle$ acts minimally on $\R/\Z$. Let $\tilde{g}, \tilde{h}\in \homeo_0(\R)$ be two lifts, with $\fix(\tilde{g})=\frac{1}{2}\Z$ and $\fix(\tilde{h})=\frac{1}{2}\Z+\frac{1}{4}$.  Given a sequence $\omega=(\omega_n)\in \{+1, -1\}^\Z$, define an element  $\tilde{g}_\omega\in \homeo_0(\R)$  by $\tilde{g}_\omega(x)=\tilde{g}^{\omega_n}(x)$ if $x\in [\frac{1}{2}n, \frac{1}{2}n+1]$, and if $\mathbb{F}_2$ is the free group with free generators $a, b$, define a representation $\varphi_\omega\colon \mathbb{F}_2\to \homeo_0(\R)$ by $\varphi_\omega(a)=\tilde{g}_\omega$ and $\varphi_\omega(b)=\tilde{h}$. It is not difficult to check that the map $\omega\mapsto \varphi_\omega$ is Borel (actually continuous). 
Since $\varphi_\omega(\mathbb F_2)=\langle \tilde{g}_\omega, \tilde{h}\rangle$ acts on $\R$ with the same orbits as $\langle \tilde{g}, \tilde{h}\rangle$, every action $\varphi_\omega$ is minimal, and this implies that for $\omega, \omega'\in \{\pm 1\}^\Z$, the actions $\varphi_\omega$ and $\varphi_{\omega'}$ are positively semi-conjugate if and only if they are positively conjugate, and this happens if and only if $\omega$ and $\omega'$ belong to the same orbit of the bilateral shift. Since the orbit equivalence relation of the shift is not smooth, the conclusion follows. \end{rem}

\begin{rem}
In contrast, for every countable group $G$ the semi-conjugacy relation on the space $\Homirr(G, \homeo_0(\mathbb{S}^1))$ of irreducible actions on the circle  is always smooth. One way to prove this is to repeat the above argument by fixing a generating probability measure $\mu$ on $G$, and note that every  action is semi-conjugate to an action for which the Lebesgue measure is stationary (unlike for the case of the real line, this is a straightforward consequence of compactness). This can be used to construct an analogue of the space of normalized harmonic actions, where the translation flow should be replaced by an action of the group $\mathbb{S}^1$, and since this is a compact group, the action must always admit a Borel cross section. Another route would be to prove that the bounded Euler class is a Borel complete invariant under semi-conjugacy. 
This difference can be seen as a formalization of the observation, mentioned in the introduction, that studying actions on the circle up to semi-conjugacy is easier than for the real line thanks to compactness.
\end{rem}

\chapter{Spaces of harmonic actions and laminar actions}\label{ch.laminar-harmonic}

In this chapter, we develop some technical tools that relate the framework of laminar actions and horograding, developed in Part \ref{partII}, with the space of normalized harmonic actions

\section{Horogradings as partial semi-conjugacies} \label{sec.largescale}

We have already seen, in Chapter \ref{sec.focalgeneral}, that if a laminar action $\varphi\colon G\to \homeo_0(\R)$ admits a horograding by another action $\rho\colon G\to \homeo_+(\R)$, then various properties of $\varphi$ are governed by $\rho$, especially as far as the behavior near $\infty$ is concerned (see for instance Proposition \ref{p-dyn-class-elements-horograded}, or Lemma \ref{l-pseudo-homothetic-horograding}). Here we explain that when $G$ is finitely generated, a horograding can be thought of as a partially-defined semi-conjugacy between $\varphi$ and $\rho$. Namely, the choice of a finite symmetric generating set $S$ determines canonically a decomposition of the line into a bounded interval  $I_S$ (the \emph{central leaf}), and two half-lines $J_S^{\pm}$ (the \emph{outer rays}) on which the horograding defines a semi-conjugacy with the horograding action $\rho$, except that the equivariance is lost when orbits visit $I_S$. We first explain how to canonically select such a leaf $I_S$.

\begin{lem}[Central leaf] \label{l-central-leaf}
	Let $G$ be a group generated by a finite symmetric set $S$, and consider a focal laminar action $\varphi\colon G\to \homeo_0(\R)$ with invariant lamination $\mathcal{L}$. 
	Let  $\mathcal{L}_S$ be the set of leaves $I\in \mathcal{L}$ such that:
	\begin{enumerate}[label=(CL\arabic*)]
		\item \label{i-barrier-cofinal} the $\varphi$-orbit of $I$ is cofinal in $\mathcal{L}$;
		\item \label{i-barrier-comparable} we have $s.I\cap I\neq \varnothing$ for every $s\in S$.
	\end{enumerate}
	Then, there exists a leaf $I_S\in \mathcal{L}_S$ such that $\mathcal{L}_S=\{I\in \mathcal{L}: I_S\subseteq I\}$.
\end{lem}

\begin{dfn}
	With notation and assumptions as in Lemma \ref{l-central-leaf}, the leaf $I_S\in \mathcal{L}$ will be called the \emph{central leaf} associated with the generating set $S$. 
\end{dfn}
\begin{proof}[Proof of Lemma \ref{l-central-leaf}]
	As $\varphi$ is focal, we can find a cofinal orbit in $\mathcal L$, thus the collection of leaves satisfying \ref{i-barrier-cofinal} is non-empty.
	Now, both conditions \ref{i-barrier-cofinal} and \ref{i-barrier-comparable} are stable under passing to a larger leaf, and clearly any sufficiently large element of $(\mathcal{L}, \subseteq)$ satisfies \ref{i-barrier-comparable} (it is enough to choose an element which contains a given point $\xi$ and its $\varphi$-images under elements of $S$). We conclude that $\mathcal{L}_S$ is non-empty.
	
	Observe next that for every $I\in \mathcal{L}_S$, there must exists $s\in S$ such that $s.I\neq I$, since the action is irreducible. Upon replacing $s$ by $s^{-1}$, there exist $s\in S$ with $s.I\supsetneq I$. It follows that we can find  a sequence $(s_n)$ of elements of $S$ such that $I_n:=s_n\cdots s_1.I$ is an increasing exhaustion of $\R$. Such a sequence can be constructed by induction, by choosing at each step $s_n\in S$ such that $s_n.I_{n-1}$ is maximal. If $(I_n)$ does not exhaust $\R$, then it must converge to some $J\in \mathcal{L}$ such that there exists no $s\in S$ with $s.J\supsetneq J$,  contradicting what was argued before. 
	
	We now show that $\mathcal{L}_S$ is a totally ordered subset of $\mathcal{L}$. Suppose by contradiction that $I_0, J\in \mathcal{L}_S$ are disjoint, and let $I_n:=s_n\cdots s_1.I_0$ be an increasing exhaustion as above. There is a smallest $m\ge 1$ such that $I_m\supseteq J$. Then $I_{m-1}\cap J=\varnothing$, and it follows that $s_m.J\cap J=\varnothing$, contradicting that $J\in \mathcal{L}_S$. Thus $\mathcal{L}_S$ is totally ordered. 
	
	It follows that $K=\bigcap_{I\in \mathcal{L}_S}{\overline{I}}$ is non-empty. Note that $K$ cannot be reduced to a point $\{\xi\}$, for if this was the case, by irreducibility of $\varphi$, there would exist $s\in S$ such that $s.\xi\neq \xi$, and thus $s.I\cap I=\varnothing$ for any sufficiently small $I\in \mathcal{L}_S$, violating \ref{i-barrier-comparable}. Hence $K=\overline{I}_S$ for some $I_S\in \mathcal{L}$. To conclude, we need to check that $I_S\in \mathcal{L}_S$, i.e.\ that it satisfies \ref{i-barrier-cofinal} and \ref{i-barrier-comparable}. It is easy to see that \ref{i-barrier-comparable} is a closed condition (or more precisely that its negation, namely that $ s.I\cap I=\varnothing$ for some $s\in S$, is an open condition), and thus is satisfied by $I_S$. As for \ref{i-barrier-comparable}, arguing as above we can choose  $s\in S$ such that $s.I_S\supsetneq I_S$, and thus $s.I_S\in \mathcal{L}_S$. In particular $s.I_S$ has a cofinal orbit (condition \ref{i-barrier-comparable}), and thus so does $I_S$.
\end{proof}

We now explain how the additional data of a horograding determines a partially defined semi-conjugacy on each connected component of $\R\setminus I_S$. 

\begin{prop}[Partial semi-conjugacies]\label{p-horograding-partial-semiconj}
	Let $G$ be a group generated by a finite symmetric set $S$, and consider a focal laminar action $\varphi\colon G\to \homeo_0(\R)$ with a positive horograding $(\mathcal L,\hor)$ by an irreducible action $\rho\colon G\to \homeo_0(\R)$.
	Let $I_S\in \mathcal{L}$ be the central leaf associated with $S$, and  denote by  $J_S^-$ and $J_S^+$ the connected components of $\R\setminus I_S$ adjacent, respectively, to $-\infty$ and $+\infty$. Then there exist maps
	\[
	\hor^\pm_S\colon J^\pm_S\to [\hor(I_S),+\infty),
	\]
	satisfying the following properties:
	\begin{enumerate}[label=(\roman*)]
		\item\label{i1-partial-semi-conj} $\hor^+_S$ is non-decreasing, $\hor^-_S$ is non-increasing, and
		\[
		\lim_{\xi\to+\infty}\hor^+_S(\xi)=\lim_{\xi\to -\infty}\hor^-_S(\xi)=\infty.
		\]
		\item\label{i2-partial-semi-conj} For any $m\ge 1$ and $s_1,\ldots,s_m\in S$, if $\xi\in J_S^+$ is such that
		\[
		s_j\cdots s_1.\hor^+_S(\xi)>\hor(I_S)\quad\text{for every }j\in \{1,\ldots,m\},
		\]
		then $s_m\cdots s_1.\xi\in J_S^+$ and
		\[
		\hor^+_S(s_m\cdots s_1.\xi)=s_m\cdots s_1.\hor^+_S(\xi).
		\]
		The analogous statement holds for $\hor^-_S$.
	\end{enumerate}
\end{prop} 
\begin{rem} The analogous statement for negative horogradings also holds. In this case, the map $\hor^+_S$ is non-increasing while the map $\hor^-_S$ is non-decreasing. 
\end{rem}

Proposition \ref{p-horograding-partial-semiconj} is a useful technical tool to work with horograded actions, so that it is again worth to give a specific name to the objects appearing in the statement.

\begin{dfn}\label{dfn-decompcentralrays}\label{dfn-partialsemicon}
	With notation and assumptions as in Proposition \ref{p-horograding-partial-semiconj}, the two connected components $J^{-}_S$ and $J^{+}_S$ of $\R\setminus {I}_S$ will be called the \emph{outer rays} associated with $S$.
	The maps $\hor^{\pm}_S$ are called the \emph{partial semi-conjugacies} associated with $S$. The point $c_S=\hor(I_S)$ will be called the \emph{central value} of the horograding $\hor$ (associated with $S$).
\end{dfn}

Roughly speaking, Proposition \ref{p-horograding-partial-semiconj} states that, as a point $\xi$ gets further away from the central leaf, the action $\varphi$ around $\xi$ resembles more and more to an action semi-conjugate to the action $\rho$, in the sense that the difference is undetectable on longer and longer elements of $G$ (in the word metric determined by $S$). Moreover one has a control on the ``good'' group elements in terms of the action $\rho$ only (and the point $c_S$). 

\begin{proof}[Proof of Proposition \ref{p-horograding-partial-semiconj}]
	Assume that $(\mathcal{L}, \hor)$ is a positive horograding of $\varphi\colon G \to \homeo_0(\R)$   by an irreducible action $\rho\colon G\to \homeo_0(\R)$. Let $J^{-}_S$ and $J^{+}_S$ be the outer rays. Similarly to what we have done in the proof of Lemma \ref{l-pseudo-homothetic-horograding}, for each $\xi\in J^{\pm}_S$, let $I_\xi\in \mathcal{L}_S$ be the smallest leaf such that $\xi\in \overline{I}_\xi$. 
	Then we consider the maps
	\[\dfcn{\hor^{\pm}_S}{J^{\pm}_S}{\R}{\xi}{\hor(I_\xi).}\]
	By construction, both maps take values in the half-line $[\hor(I_S), +\infty)$. The map $\hor^{+}_S$ is non-decreasing,   $\hor^{-}_S$ is non-increasing, and by Lemma \ref{l-horograding-infinity}, we have  
	\[\lim_{\xi\to+\infty} \hor^{+}_S(\xi)=\lim_{\xi\to-\infty} \hor^{-}_S(\xi)=+\infty,\]
	so that \ref{i1-partial-semi-conj} is verified. We next discuss \ref{i2-partial-semi-conj}.
	By induction, it is enough to establish the property for $m=1$. Fix $\xi\in J^{+}_S$ and $s\in S$ such that ${c}_S<s.\hor^{+}_S(\xi)=\hor\left(s.I_\xi\right)$. We need to justify that this implies that $s.\xi\in J^{+}_S$. 
	It cannot be the case that $s.\xi\in J^{-}_S$, as this would imply that $s.I_S\cap I_S=\varnothing$, contradicting \ref{i-barrier-comparable}. Thus the only possibility is that  $s.\xi\in \overline{I}_S$, i.e.\ $\xi\in s^{-1}.\overline{I}_S$.  Note that $J:=s^{-1}.I_S$ and $I_S$ must be related by inclusion, by \ref{i-barrier-comparable}, and since $\overline{J}$ contains $\xi\notin \overline{I}_S$, we deduce that $I_S\subset J\in \mathcal{L}_S$.  This implies that $I_\xi\subseteq J$, by definition of $I_\xi$,  and $s.I_\xi\subset I_S$. We deduce that $c_s<\hor(s.I_\xi)\leq \hor(I_S)=c_S$, a contradiction; hence $\xi\in J^{+}_S$. It follows immediately that $I_{s.\xi}=s.I_\xi$, and thus $\hor^{+}_S(s.\xi)=s.\hor^{+}_S(\xi)$. \qedhere
\end{proof}

\section{Laminar actions in harmonic coordinates}\label{sec.laminar_harmonic}

Starting from now, we discuss properties of central leaves and partial semi-conjugacies when working with normalized harmonic actions.

\begin{assumption} \label{a-G-S-mu}
	Throughout the section, we fix a finitely generated group $G$ together with a finite symmetric generating set $S$, and  a symmetric probability measure $\mu$ on $G$ whose support is finite and generates $G$. 
\end{assumption}

\subsection{Uniformity of central leaves}

The following lemma says that central leaves of laminar actions in $\Der_\mu(G;\R)$ have uniformly controlled diameter. Given an interval $J=(a,b)$, we denote by $|J|=b-a$ its length, as customary. 

\begin{lem}\label{lem.uniformsize}
	Under Assumption \ref{a-G-S-mu},  there exist constants $0<C_1< C_2$ such that for every laminar action $\varphi\in\Der_\mu(G;\R)$ and $\varphi$-invariant lamination $\mathcal{L}$, the central leaf $I_{S}\in \mathcal L$ satisfies $C_1< |I_{S}|< C_2$. 
\end{lem}

\begin{proof}
	For any irreducible action $\varphi:G\to\homeo_0(\R)$, define 
	\[\delta(\varphi)=\max\{\varphi(s)(0):s\in S\},\]
	which is non-zero (by irreducibility of $\varphi$), and actually positive (since $S$ is symmetric). On the one hand,
	$\delta$ defines a continuous function on $\Homirr(G,\homeo_0(\R))$, so that by compactness of $\Der_\mu(G;\R)$, we can 
	set
	\begin{equation}\label{eq:delta+}
		\delta_+=\min_{\varphi\in\Der_\mu(G;\R)}\delta(\varphi),
	\end{equation}
	which is therefore positive. On the other hand,
	for any $\xi\in \R$ and $\varphi\in \Homirr(G,\homeo_0(\R))$, the relation $\Phi^{\xi}(\varphi)(g)(0)=\varphi(g)(\xi)-\xi$,
	gives
	\[\delta(\Phi^{\xi}(\varphi))=\max\{\varphi(s)(\xi)-\xi:s\in S\}.\]
	By $\Phi$-invariance of $\Der_\mu(G;\R)$, we deduce that 
	\[\min_{\varphi\in\Der_\mu(G;\R),\,\xi\in \R}\max\{\varphi(s)(\xi)-\xi:s\in S\}=\min_{\varphi\in\Der_\mu(G;\R)}\delta(\varphi)=\delta_+.\]
	Similarly, we can introduce the well-defined positive constant
	\begin{align}\label{eq:Delta-laminar}
		\Delta=&\max\{|\varphi(s)(\xi)-\xi|:\varphi\in\Der_\mu(G),s\in S,\xi\in \R\}\\=&\max\{|\varphi(s)(0)|:\varphi\in\Der_\mu(G),s\in S\}>0.\nonumber
	\end{align}
	
	Fix now a laminar action $\varphi\in\Der_\mu(G;\R)$ together with a $\varphi$-invariant lamination $\mathcal{L}$, and let $I_{S}=(I_{S}^-,I_{S}^+)$ be its associated central leaf. After the previous discussion, we can take $s\in S$ with $s.I_{S}^{-}-I_{S}^{-}\ge \delta_+$. By definition of $I_{S}$, we get $s.I_{S}^+\subset I_{S}^+$, and we conclude that $|I_{S}|\ge \delta_+$. 
	For the other inequality, we claim that $|I_{S}|\le 3\Delta$. Indeed, since $S$ is symmetric, we can take $s\in S$ so that $s.I_{S}^+<I_{S}^+$. In this case, by definition of $I_{S}$, we have that $s.I_{S}\subset I_{S}$. On the one hand, this gives $|s.I_{S}|\geq |I_{S}|-2\Delta$. On the other hand, recalling the properties of $I_S$ (see Lemma \ref{l-central-leaf}), this also gives that for some $s'\in S$ we must have $s's.I_{S}\cap s.I_{S}=\emptyset$, and consequently $|s.I_{S}|\le \Delta$. Putting both inequalities together, we get that $|I_{S}|\le 3\Delta$.
\end{proof}

\subsection{Limits of centered laminar actions}

In general, a limit point of laminar actions in the space $\Der_\mu(G; \R)$ need not be laminar.  The next result provides a sufficient condition under which this is the case.

\begin{lem}\label{lem.centerlaminated}
	Under Assumption \ref{a-G-S-mu}, for any $K>0$, the subset of focal laminar actions in $\Der_\mu(G;\R)$ admitting an invariant lamination whose central leaf is contained in the interval $(-K,K)$, is closed.
\end{lem}
\begin{proof} 
	Given an action $\varphi\colon G\to \homeo_0(\R)$ and a  $\varphi$-invariant lamination $\mathcal{L}$, we will write $I_{S,\varphi}$ for the central leaf.
	Consider a sequence of laminar actions $\varphi_n\in\Der_\mu(G;\R)$, with $\varphi_n$-invariant laminations $\mathcal{L}_n$, such that $I_{S,\varphi_n}\subset (-K,K)$ for every $n\in\N$, and assume that it converges, as $n\to\infty$, to an action $\varphi\in\Der_\mu(G;\R)$. By Lemma \ref{lem.uniformsize}, we have that there exists a positive constant $C_1>0$ so that $|I_{S,\varphi_n}|> C_1$ for every $n\in\N$. Since $I_{S,\varphi_n}\subset (-K,K)$ for every $n\in\N$, upon passing to a subsequence, we can assume that $I_{S,\varphi_n}$ converges to an interval $J\subset [-K,K]$ with non-empty interior. Moreover, since crossing is an open condition, we have that the $\varphi$-orbit of the interior of $J$ is cross free and therefore gives a $\varphi$-invariant prelamination. Furthermore, since $I_{S,\varphi_n}$ is related by inclusion with $\varphi_n(s)(I_{S,\varphi_n})$ for every $s\in S$, the same happens with $J$ and $\varphi(s)(J)$. On the one hand, this implies that $\varphi$ is not a cyclic action, and since it belongs to $\Der_\mu(G;\R)$, it is minimal, so that it is a focal laminar action (Proposition \ref{prop.minimalimpliesfocal}). On the other hand, it gives that $I_{S,\varphi}\subset J\subset [-K,K]$, concluding the proof.
\end{proof}

\subsection{Horogradings in harmonic coordinates}

\begin{lem} \label{l-partial-quasi-isometry}
	Under Assumption \ref{a-G-S-mu}, there exists $C>0$ such that if $\varphi\in \Der_\mu(G, \R)$ is laminar, and $(\mathcal{L}, \hor)$ a horograding of $\varphi$ by an action $\rho\in \Der_\mu(G, \R)$, then the associated partial semi-conjugacies $\hor^{\pm}_S \colon J^{\pm}_S\to \R$ satisfy
	\[\frac{1}{C}|\xi-\eta|-C\leq |\hor^{\pm}_S(\xi)-\hor^{\pm}_S(\eta)|\leq C|\xi-\eta|+C,\]
	for every $\xi, \eta\in J^{\pm}_S$.
\end{lem}

\begin{proof}
	Let us prove the statement for $\hor^+_S$. Let $\delta_+$ and $\Delta$ be the constants defined in the proof of Lemma \ref{lem.uniformsize}, and assume that $\xi,\eta\in J_S^+$, with $\xi<\eta$. Choose inductively a sequence of generators $(s_i)$ such that $s_{i+1}\cdots s_1.\xi-s_{i}\cdots s_1.\xi$ is maximal for every $i$ (and hence bounded from below by $\delta_+$, but still bounded from above by $\Delta$). Then Proposition \ref{p-horograding-partial-semiconj} implies that $s_{i+1}\ldots s_1.\hor^+_S(\xi)- s_{i}\ldots s_1.\hor^+_S(\xi)$ is also maximal (and hence also bounded from below by $\delta_+$, and from above by $\Delta$). Let $n$ be such that $s_{n}\cdots s_1.\xi<\eta \leq s_{n+1}\cdots s_1.\xi$.  Then 
	\[s_{n}\cdots s_1.\hor^+_S(\xi)\leq\hor^+_S(\eta) \leq s_{n+1}\cdots s_1.\hor^+_S(\xi),\]
	so that $\eta-\xi$ and $\hor^+_S(\eta)-\hor^+_S(\xi)$ are both bounded from below by $\delta_+n$ and from above by $\Delta n+\Delta$. The conclusion follows easily. \qedhere
\end{proof}

We conclude this chapter with an important technical statement which relates horograded actions and their horograding actions in the space $\Der_\mu(G; \R)$. By Proposition \ref{p-horograding-partial-semiconj}, if a laminar action $\varphi$ has a horograding by $\rho$, then $\varphi$ and $\rho$ look semi-conjugate far away from the central leaf (up to the orientation). In the space  $\Der_\mu(G;\R)$, this implies that the $\Phi$-orbits of $\varphi$ and $\rho$ are eventually close (with a uniform control on constants).

\begin{lem}\label{lem.equivalentconverg} 
	Under Assumption \ref{a-G-S-mu}, let $\dist$ be any distance inducing the topology of $\Der_\mu(G, \R)$. Then for every $\varepsilon>0$, there exists $T>0$ such that the following hold. Suppose that $\varphi\in \Der_\mu(G, \R)$ is laminar, and that $(\mathcal{L}, \hor)$ is a positive horograding of $\varphi$ by an action $\rho\in \Der_\mu(G, \R)$, with central leaf $I_S$ and partial semi-conjugacies $\hor^{\pm} \colon J_S^{\pm} \to \R$. Then, the following hold.
	\begin{enumerate}[label=(\roman*)]
		\item \label{i-positive-t} If $t\in J^+_S$ is such that $t-\sup I_S>T$, then $\dist(\Phi^{t}(\varphi), \Phi^{\hor^+(t)}(\rho))<\varepsilon$.
		\item \label{i-negative-t} If $t\in  J^-_S$ is such that $t-\inf I_S<-T$,  then $\dist(\Phi^{t}(\varphi), \Phi^{\hor^-_S(t)}(\widehat{\rho}))<\varepsilon$, where $\widehat{\rho}$ is the conjugate of $\rho$ by the map $x\mapsto -x$.
	\end{enumerate}	
\end{lem}

\begin{rem}
	Note that any negative horograding $(\mathcal{L}, \hor)$ of $\varphi$ by $\rho$ is a positive horograding of $\varphi$ by $\widehat{\rho}$, so the  same result holds for negative horogradings, upon exchanging the roles of $\rho$ and $\hat{\rho}$.
\end{rem}

\begin{proof}[Proof of Lemma \ref{lem.equivalentconverg}]
	We will prove \ref{i-positive-t}, the proof of \ref{i-negative-t} being analogous. Suppose by contradiction that the conclusion does not hold. Then we can find sequences of laminar actions $(\varphi_n)$ in $\Der_\mu(G, \R)$, with horograding $(\mathcal{L}_n, \hor_n)$ of $\varphi_n$ by an action $\rho_n\in\Der_\mu(G, \R)$, with central leaves $I_n$ and partial semi-conjugacies $\hor_n^\pm \colon J_n^{\pm} \to \R$, and a sequence of real numbers $t_n\in J_n^+$, such that  $t_n -\sup I_n\to \infty$, but $\dist(\Phi^{t_n}(\varphi_n), \Phi^{u_n}(\rho_n))$ is bounded away from 0, where $u_n:=\hor_n^+(t_n)$.   After extracting a subsequence, we can suppose that $\Phi^{t_n}(\varphi_n)$ and $\Phi^{u_n}(\rho_n)$ converge to different limits in $\Der_\mu(G; \R)$.  We will work with the homeomorphism
	\[\dfcn{I}{\Der_\mu(G;\R)}{[\LPO](G)}{\varphi}{[\leq_\varphi],}\]
	defined in the proof of Theorem \ref{thm.homeoderoin} (recall that $\leq_\varphi$ is the preorder induced by the $\varphi$-orbit of $0$).  Denote by $\leq_n$ the preorder induced by the $\varphi_n$-orbit of $t_n$, which equals the preorder induced by the $\Phi^{t_n}(\varphi_n)$-orbit of $0$, and let $\preceq_n$ be the preorder induced by the $\rho_n$-orbit of $u_n$, which equals the preorder induced by the $\Phi^{u_n}(\rho_n)$-orbit of $0$.
	
	Assume now that $\Phi^{u_n}(\rho_n)$ converges to an action $\rho_\infty$; then, this is equivalent to saying that $[\preceq_n]\to[\preceq_\infty]$, where $\preceq_\infty$ is the preorder induced by the $\rho_\infty$-orbit of $0$. To obtain a contradiction, we will prove that in this case $\Phi^{t_n}(\varphi_n)$ also converges to $\rho_\infty$, by proving that $[\leq_n]\to[\preceq_\infty]$. Note that by Lemma \ref{lem.convcriteria}, in order to prove that $[\leq_n]\to[\preceq_\infty]$, it is enough to show that for every $g\in P_{\preceq_\infty}$, there exists $n_0\in\N$ such that $g\in P_{\leq_n}$ for every $n\geq n_0$. So, let us fix an element $g\in P_{\preceq_\infty}$, and let us first show that $g\in P_{\preceq_n}$ for any $n$ big enough. If it were not the case, by compactness of $\mathsf{LPO}(G)$ we could find a convergent subsequence $\preceq_{n_k}\to\preceq_\ast$ with $g\notin P_{\preceq_{*}}$.
	Then, on the one hand, since $[\preceq_n]\to[\preceq_\infty]$, we must have that $\preceq_\ast\in[\preceq_\infty]$. On the other hand, since $\rho_\infty\in\Der_\mu(G;\R)$ is a canonical model, Proposition \ref{lem.minimalisminimal} implies that $\preceq_\infty$ is a minimal model of $[\preceq_\infty]$. These two facts together give that $P_{\preceq_\infty}\subset P_{\preceq_\ast}$, contradicting that $g\notin P_{\preceq_\ast}$. Conversely, we show the following.
	
	\begin{claim}\label{sublem.1731}
		There exists $n_0\in \N$ such that if $n\ge n_0$ and $g\in P_{\preceq_n}$, then $g\in P_{\leq_n}$.
	\end{claim}
	
	\begin{proof}[Proof of claim]
		We want to prove that for any sufficiently large $n$,
		if $\rho(g)(u_n)>u_n$, then $\varphi_n(g)(t_n)>t_n$. Assume that $g$ satisfies so, and let $c_n:=\hor_n(I_n)$ be the central value of the horograding. Since $t_n-\sup I_n$ tends to $+\infty$,  Lemma \ref{l-partial-quasi-isometry} implies that $u_n-c_n$ tends to $+\infty$. Consider the constant $\Delta$ defined as in \eqref{eq:Delta-laminar}, and write $g=s_k\cdots s_1$, for some $k\in \N$ and $s_1,\ldots,s_k\in S$. Then for $n$ large enough, we have   $u_n-c_{n}>k\Delta$, and hence $\rho_n(s_i\cdots s_1)(u_n)\ge u_n-i\Delta>c_{n}$ for any $i\in\{1,\ldots,k\}$.  For such an $n$, Proposition \ref{p-horograding-partial-semiconj} implies that $\varphi_n(g)(t_n)\in J_{n}^{+}$ and
		\begin{equation*}\label{ecusemiconj}
			\hor_{n}^+(\varphi_n(g)(t_n))=\rho_n(g)(\hor^+(t_n))=\rho_n(g)(u_n)>u_n.
		\end{equation*}
		Since $\hor_{n}^+$ is non-decreasing, this gives $\varphi_n(g)(t_n)\ge t_n$,
		as desired.
	\end{proof}
	
	Thus, since $g\in P_{\preceq_n}$ for any $n$ large enough, the claim gives that $g\in P_{\leq_n}$ for any $n$ large enough. This finishes the proof that $\Phi^{t_n}(\varphi_n)$ converges to $\rho_\infty$, and thus the proof of the lemma.
\end{proof}

\chapter{Spaces of harmonic actions for locally moving groups}\label{ch.harm_locmov}

In this chapter we prove our main results in the setting of locally moving groups. In particular,  we prove a criterion for the local rigidity of the standard action.

\section{Preliminary considerations}

Let $X\subseteq \R$ be an open interval, and suppose that $G\subset \homeo_0(X)$ is a finitely generated micro-supported subgroup acting minimally on $X$. Fix a symmetric probability measure $\mu$ on $G$ with finite generating support, and consider the space $\Der_\mu(G;\R)$ of normalized $\mu$-harmonic actions. The latter admits a decomposition 
\begin{equation} \label{e-lr-decomposition}
	\Der_\mu(G;\R)=\mathcal Q\sqcup \Ical \sqcup \widehat{\Ical} \sqcup \mathcal E,
\end{equation}
where the subsets in this partition are defined by what follows:
\begin{itemize}
	\item we let $\mathcal Q\subset \Der_\mu(G;\R)$ be the subset consisting of  actions $\varphi$ that are not faithful, or equivalently factor through $G/ [G_c, G_c]$ (by Proposition \ref{p-micro-normal});
	\item we fix a harmonic representative $\iota\in \Der_\mu(G;\R)$ of the standard action on $X$, and let $\mathcal{I}=\{\Phi^t(\iota): t\in \R\}$ be its $\Phi$-orbit, and $\widehat{\mathcal{I}}$ the $\Phi$-orbit of $\widehat{\iota}$, the conjugate of $\iota$ under the reflection $x\mapsto -x$;
	\item we let $\mathcal{E}$ be the subset of $\Der_\mu(G;\R)$ of exotic actions, that is, all actions that are faithful, minimal,  and not conjugate to the standard action. Recall that these are all laminar by Theorem \ref{t-laminations-microsupported}.
\end{itemize}
The notation introduced above will be used throughout the chapter.
The following proposition states some elementary properties of the decomposition in \eqref{e-lr-decomposition} (which do not rely on any of our main results). Recall that for a micro-supported subgroup $G\subseteq \homeo_0(X)$ acting on $X=(a, b)$, we denote by $G_-$ (respectively, $G_+$) the subgroups of elements with trivial germ at $a$ (respectively, $b$).

\begin{lem} \label{l-deroin-partition-general}
	Let $G\subset\homeo_{0}(X)$, be a finitely generated micro-supported subgroup acting minimally on $X$, and fix a symmetric probability measure with finite generating support $\mu$ on $G$. Consider the decomposition \eqref{e-lr-decomposition} of $\Der_\mu(G;\R)$. Then all sets in the decompositions are invariant under the translation flow $\Phi$, and the following hold:
	\begin{enumerate}[label=(\roman*)]
		\item \label{i-Q-closed} The set $\mathcal{Q}$ is closed.
		\item \label{i-limit-standard} All accumulation points of the $\Phi$-orbits $\Ical, \widehat{\Ical}$ are contained in $\mathcal{Q}$.  More precisely:
		\begin{itemize}
			\item if $\psi$ belongs to the $\omega$-limit set of $\Ical$, or to the $\alpha$-limit set of $\widehat{\Ical}$,  then $\ker \psi\supset G_+$;
			\item if $\psi$ belongs to the $\alpha$-limit set of $\Ical$, or to the $\omega$-limit set of $\widehat{\Ical}$, then $\ker \psi\supset G_-$.
		\end{itemize}
		\item \label{i-exotic-open} The set $\mathcal{E}$ is open.
	\end{enumerate}
\end{lem}

\begin{proof}
	By Proposition \ref{p-micro-normal}, an action $\varphi\in \Der_\mu(G, \R)$ belongs to $\mathcal{Q}$ if and only if  $\varphi([G_c,G_c])=\{\mathsf{id}\}$; the latter is a closed condition, showing \ref{i-Q-closed}.
	To show \ref{i-limit-standard}, consider, for instance, the case where $\psi$ is an accumulation point of $\Phi^t(\iota)$, for $t\to + \infty$ (the other cases being analogous).  For simplicity, we  assume that $\iota$ is the defining action of $G$ on $X=\R$. For every finite subset $\Sigma\subset G_+$, we can find $x\in X$ such that $\Sigma\subset G_{(-\infty, x)}$. Since $\Phi^t(\iota)$ is the conjugate of $\iota$ by the translation $x\mapsto x-t$, it follows that for any arbitrary compact subset $K\subset \R$ and $t>x-\min K$, the image $\Phi^t(\iota)(\Sigma)$ acts trivially on $K$, so that $\psi(\Sigma)=\{\Id\}$. Since $\Sigma$ is an arbitrary finite subset of $G_+$, we deduce that $G_+\subset \ker \psi$ (in particular $\psi$ is not faithful, i.e. it belongs to $\mathcal{Q}$).
	Finally, \ref{i-Q-closed} and \ref{i-limit-standard} together imply that $\mathcal{Q}\cup \Ical\cup \widehat{\Ical}$ is closed, i.e. that $\mathcal{E}$ is open, showing \ref{i-exotic-open}.
\end{proof}

\section{Main result}
\label{s-lr}

\subsection{Statement and main corollaries}
\label{ss-lr}

The main result of Part \ref{partIII} is Theorem \ref{thm-decoDeroin} below. It shows that for a class of locally moving subgroups $G\subset \homeo_0(X)$, the dynamics of the flow  $(\Der_\mu(G; \R), \Phi)$ satisfies some strong restrictions, leading in particular to the local rigidity of the standard action.  This result relies on a technical but important assumption on $G$, related to the generating sets of the subgroups $G_I$ for  $I\subsetneq X$. We will use the following terminology (in the lack of a better one). 

\begin{dfn}\label{def.classF} Let $X=(a, b)$. We say that a subgroup $G\subseteq \homeo_0(X)$ satisfies condition \ref{class} if 
	
	\begin{enumerate}[label=($\mathcal F$)]
		\item\label{lab2}  \label{class} there exists $x,y\in X$ such that the subgroup $G_{(a, x)}$ is contained in a finitely generated subgroup of $G_+$, and $G_{(y, b)}$ is contained in a finitely  generated subgroup of $G_-$.  \end{enumerate}
\end{dfn}
\begin{rem}
	Condition \ref{class} is satisfied provided there exists $x, y$ such that $G_{(a, x)}, G_{(y, b)}$ are themselves finitely generated. For example, this condition is satisfied by Thompson's group $F$, and many relatives.  Various sources of examples of locally moving groups with \ref{class} will be discussed in \S \ref{s-examples-class}.
	In \S \ref{s-example-non-lr}, we will construct an example showing that the main results in this section become false if \ref{class} is dropped.
\end{rem}
Recall that a flow $\Psi$ on a locally compact space $Z$ is \emph{proper} if for every compact subset $K\subseteq Z$, the set $\{t\in \R: \Psi^t(K)\cap K\neq \varnothing\}$ is compact.
\begin{thm}\label{thm-decoDeroin} Let $G\subset\homeo_{0}(X)$ be a finitely generated locally moving subgroup satisfying \ref{class}, and fix  a symmetric probability measure $\mu$ with finite generating support on $G$.  Consider the decomposition of the space $\Der_\mu(G; \R)$ as in \eqref{e-lr-decomposition}. Then the following hold:
	\begin{enumerate}[label=(\roman*)]
		\item \label{i-closure-exotic} The orbits $\mathcal{I}$ and $\widehat{\mathcal{I}}$ are open subsets of $\Der_\mu(G)$. 
		\item \label{i-flow-proper} The translation flow  restricted to the open subset $\Der_\mu(G, \R)\setminus \mathcal{Q}$ of faithful actions is proper.
	\end{enumerate}
\end{thm} 

Theorem \ref{thm-decoDeroin} has the following immediate corollary, which implies Theorem \ref{t-intro-local-rigidity} from the introduction.
\begin{cor}[Local rigidity of the standard action] \label{c-local-rigidity}
	For $X=(a, b)$, let $G\subset \homeo_0(X)$ be a finitely generated locally moving subgroup satisfying \ref{class}. Then, the subset of actions in $\Homirr(G, \homeo_0(\R))$ which are positively semi-conjugate to the standard action of $G$ on $X$ is open. In particular, the standard action of $G$ on $X$ is locally rigid. 
\end{cor}
\begin{proof}
	The retraction 
	\[r\colon\Homirr(G,\homeo_0(\R))\to\Der_\mu(G;\R)\]
	from Theorem \ref{t.retraction_Deroin} preserves the positive semi-conjugacy classes; hence, the subset of actions positively semi-conjugate to the standard action coincides with the preimage $r^{-1}(\mathcal{I})$. As by Theorem \ref{thm-decoDeroin}, we have that the subset $\mathcal{I}$ is open, the result follows.
\end{proof}

\begin{rem}
	Notice that the reversed standard action is also locally rigid. We point out that in the setting of Corollary \ref{c-local-rigidity}, these actions are not always the unique locally rigid actions of $G$: the analysis of actions of Bieri--Strebel groups in Chapter \ref{s-few-actions} provides examples of groups satisfying Corollary \ref{c-local-rigidity} (see \S \ref{s-examples-lr-pl}), that admit finitely many conjugacy classes of faithful minimal laminar actions. These actions correspond to  finitely many $\Phi$-orbits in $\Der_\mu(G; \R)$, and are therefore locally rigid as well.
\end{rem}

Another consequence of Theorem \ref{thm-decoDeroin} is that the structure of the space of semi-conjugacy classes in $\Homirr(G, \homeo_0(\R))$ is particularly well behaved from a topological and Borel perspective. In particular, faithful minimal  actions of $G$, considered up to conjugacy, give rise to a well-behaved topological space. (Recall the discussion in  \S \ref{ssc.complexity} for the notion of smoothness of an equivalence relation in the Borel setting, and its significance to classification problems).

\begin{cor}[Tameness of the space of semi-conjugacy classes]\label{c-class-borel}
	For $X=(a, b)$, let $G\subset \homeo_0(X)$ be a finitely generated locally moving subgroup satisfying \ref{class}. Let $\mathcal{M}$ be the subset of $\Homirr(G, \homeo_0(\R))$ consisting of actions that are faithful and minimal, and denote by $\sim$ the equivalence relation on $\mathcal{M}$ of positive conjugacy. Then the following hold:
	
	\begin{enumerate}[label=(\roman*)]
		\item\label{i1-class-borel} The quotient space $\mathcal{M}/\sim$ is a locally compact metrizable space, in which the conjugacy class of the the standard action is an isolated point.\item\label{i2-class-borel} 
		The equivalence relation $\sim$ on $\mathcal{M}$ is smooth. 
		The semi-conjugacy relation on the space $\Homirr(G, \homeo_0(\R))$ is smooth if and only if the same holds true for the space $\Homirr(G/[G_c, G_c], \homeo_0(\R))$ of irreducible actions of the largest quotient. 
	\end{enumerate} 
\end{cor}

\begin{proof} Let $\sim$ be the conjugacy relation on $\mathcal{M}$, and let
	\[\mathcal{U}:= \mathcal{M}\cap\Der_\mu(G; \R)=\Der_\mu(G;\R)\setminus \mathcal{Q}.\]
	The harmonic retraction \[r\colon\Homirr(G,\homeo_0(\R))\to\Der_\mu(G;\R)\]
	defined in Theorem \ref{t.retraction_Deroin} satisfies that $r(\mathcal{M})=\mathcal{U}$, and two actions $\varphi_1, \varphi_2\in \mathcal{M}$ are conjugate if and only if $r(\varphi_1)$ and $r(\varphi_2)$ are in the same $\Phi$-orbit. From this it follows that the quotient space $\mathcal{M}/\sim$ is naturally homeomorphic to the quotient space $\mathcal{U}/\Phi$.  Note that $\mathcal{U}$ is a locally compact metrizable space (as it is an open subset in a compact metrizable space). Theorem \ref{thm-decoDeroin} implies that $\mathcal{I}/\Phi$ and $\widehat{\Ical}/\Phi$ are isolated points in the quotient, and that the flow $\Phi$ on $\mathcal{U}$ is proper.  This implies that $\mathcal{U}/\Phi$ is Hausdorff, locally compact (Abels \cite[Corollary 1.15]{Abels}),  and regular (Palais \cite[Proposition 1.2.8]{Palais}), hence metrizable by the Urysohn lemma, showing \ref{i1-class-borel}. For \ref{i2-class-borel}, first note that \ref{i1-class-borel} implies in particular that $\mathcal{M}/\sim$ is a standard Borel space, so $\sim$ is smooth on $\mathcal{M}$.
	For the remaining statement, we apply Corollary \ref{cor.cross-section}. Theorem \ref{thm-decoDeroin}, implies $\mathcal{E}$ does not contain any  recurrent point for the flow $\Phi$, and the same conclusion holds for $\Ical, \widehat{\Ical}$ by Lemma \ref{l-deroin-partition-general}. Hence all recurrent points of $(\Der_\mu(G; \R), \Phi)$ are contained in $\mathcal{Q}$.  To conclude, note that $\mathcal{Q}$ naturally identifies with the space of harmonic actions $\Der_{\overline{\mu}}(G/[G_c, G_c]; \R)$, where $\overline{\mu}$ is the push-forward of $\mu$ by the quotient map.
\end{proof}

\subsection{The proof}\label{ssc.limit_laminar_harm} \label{s-deroin-proof}

The main tool for proving Theorem \ref{thm-decoDeroin} is Theorem \ref{t-lm-horograding}. It is applicable thanks to the following lemma.
\begin{lem}\label{lem-classFfgfrag}
	For $X=(a,b)$, let $G\subset\homeo_0(X)$ be a locally moving subgroup satisfying \ref{class}. Then the fragmentable subgroup $\Gfrag$ is finitely generated.
\end{lem}
\begin{proof}
	Since $G$ satisfies condition \ref{class}, there exist  $x_1,x_2\in X$ and finite subsets $S_1\subset G_-$ and $S_2\subset G_+$, such that $G_{(a,x_2)}\subset\langle S_2\rangle$ and $G_{(x_1,b)}\subset\langle S_1\rangle$, and upon conjugating $G_{(a, x_2)}$ by some $g\in G$, we can assume that $x_1<x_2$. We set $H=\langle S_1,S_2\rangle$, and we want to prove that $H=G_{\mathsf{frag}}$. The inclusion $H\subseteq \Gfrag$ is clear. To show the other inclusion, first notice that since  $x_1<x_2$, the subgroup $\langle G_{(a,x_2)},G_{(x_1,b)}\rangle\subseteq H$ acts on $X$ without fixed points. Thus, the union
	\[
	\bigcup_{h\in H} hG_{(a,x_2)}h^{-1}=\bigcup_{h\in H}G_{(a,h(x_2))},
	\]
	contained in $H$, coincides with the subgroup $G_+=\bigcup_{x\in X}G_{(a,x)}$. A similar argument shows that $G_-$ is also contained in $H$. Since $\Gfrag=G_-G_+$, we get the desired conclusion.
\end{proof}

Through the rest of this subsection, let $G$ be as in Theorem \ref{thm-decoDeroin}, and fix a finite symmetric generating subset $S\subset G$.  As before, we choose a representative $\iota\in\Der_\mu(G;\R)$ of the standard action of $G$. For simplicity, we will assume that $X=\R$ and identify the standard action with $\iota$ (although we keep the notation $X$, to avoid confusion with actions $\varphi\in \mathcal{E}$). We consider the decomposition $\Der_\mu(G; \R)=\mathcal Q\sqcup \Ical \sqcup \widehat{\Ical} \sqcup \mathcal E$ given by \eqref{e-lr-decomposition}.
By Theorem \ref{t-lm-horograding}, every $\varphi\in\mathcal{E}$ is  (positively or negatively) horograded by $\iota$. Thus we may further decompose $\mathcal{E}$ as $\mathcal{E}=\mathcal{E}_+\cup \mathcal{E}_-$, where $\mathcal{E}_+$ (respectively, $\mathcal{E}_-$) is the subset of actions that are positively (respectively, negatively) horograded by $\iota$. 

For every action $\varphi\in \mathcal{E}_+$, the proof of Theorem \ref{t-lm-horograding}  provides an explicit prehorograding $(\mathcal{L}_\varphi, \hor_\varphi)$ of  $\varphi$ by $\iota$,  described in Remark \ref{l-positive-horograding}, that we recall now. The subgroups $G_{(a,x)}$ are totally bounded for every $\varphi\in \mathcal E_+$ (recall that this means that each connected component of $\suppphi(G_{(a, x)})$ is bounded).  This allows to define a  $\varphi$-invariant prelamination given by $\mathcal{L}_\varphi=\{\Iphi(x,\xi):x\in X,\xi\in \Xi_\varphi\}$ where 
\begin{itemize}
	\item $\Xi_\varphi=\bigcap_{x\in X}\suppphi(G_{(a,x)})$ (which is a $G_\delta$-dense subset of $\R$), and
	\item $\Iphi(x,\xi)$ is the connected component of $\suppphi(G_{(a,x)})$ containing $\xi$.
\end{itemize}
The prehorograding $\hor_\varphi$ is then given by $\Iphi(x, \xi)\mapsto x$, and it extends to a horograding defined on the closure $\overline{\mathcal{L}}_\varphi$. These objects can be introduced in an analogous way when $\varphi\in \mathcal E_-$ is negatively horograded by the standard action, by replacing the instances of $G_{(a,x)}$ by $G_{(x,b)}$. We will discuss the case of positive horograding below.

As we have fixed the generating subset $S$, we can consider the associated central leaf $I_{S,\varphi}\in \overline{\mathcal L}_\varphi$, the outer rays $J_{S,\varphi}^\pm$, and the partial semi-conjugacies $\hor^{\pm}_{S,\varphi}\colon J^\pm_{S,\varphi}\to [c_{S,\varphi},+\infty)$, where $c_{S,\varphi}=\hor_\varphi(I_{S,\varphi})$ is the central value (see Definition \ref{dfn-partialsemicon}). 

With this notation at disposal, the main idea in the proof of Theorem \ref{thm-decoDeroin} is to control the accumulation points of $\mathcal{E}$ by discussing separately the case where the central leaf is at bounded distance from the origin, that can be settled using Lemma \ref{lem.centerlaminated}, and the case where the central leaf is far from the origin, which can be studied using Lemma \ref{lem.equivalentconverg}. For this, the crucial missing argument is contained in the following lemma. Its proof uses condition \ref{class} in an essential way.

\begin{lem}[Key lemma]\label{lem.conclusion}
	Retain the assumptions of Theorem \ref{thm-decoDeroin}, and the notation above. Then, there exist $c_+, c_-\in X$ such that $c_{S,\varphi}\ge c_+$ for every $\varphi\in\mathcal E_+$, and $c_{S,\varphi}\le c_-$ for every $\varphi\in\mathcal E_-$
\end{lem}

\begin{proof}
	We first see that $\mu$-harmonicity of actions provides a uniform control on the size of the intervals $\Iphi(x,\xi)$.
	
	\setcounter{claimnum}{0}
	
	\begin{claim}\label{lem-sizeIfi}
		For every $x\in X$, there exists a constant $C>0$ such that $|\Iphi(x,\xi)|<C$ for every laminar action $\varphi\in\mathcal{E}$ and $\xi\in \Xi_\varphi$.
	\end{claim}
	\begin{proof}[Proof of claim]
		We only detail the case where $\varphi\in \mathcal E_+$ is positively horograded by the standard action (working with $\mathcal E_-$ is totally analogous).	
		Fix $x\in X$. Since $G$ satisfies condition \ref{class}, we can find a finitely generated subgroup $H=\langle h_1,\ldots,h_n\rangle\subset G$ and a point $y\in [x,b)$ such that $G_{(y,b)}\subseteq H\subseteq G_{(x,b)}$. We will show that the constant 
		\[C=\max\left \{|\varphi(h_i)(\xi)-\xi|:\varphi\in\Der_\mu(G;\R),\xi\in\R,\text{ and }i\in \{1,\ldots,n\}\right \},\]
		which is well defined and positive (by compactness and $\Phi$-invariance of $\Der_\mu(G;\R)$, see the argument in the proof of Lemma \ref{lem.uniformsize}),
		gives the desired uniform control, and we will do this by way of contradiction. For this, assume there exist $\varphi\in\Der_\mu(G;\R)$ and $\xi\in\Xi_{\varphi}$ with $|\Iphi(x,\xi)|>C$.
		
		On the one hand, for any $h\in H$, we have either $h.\Iphi(x,\xi)=\Iphi(x,\xi)$, or $h.\Iphi(x,\xi)\cap \Iphi(x,\xi)=\emptyset$: indeed, since $H\subseteq G_{(x,b)}$, we have that  $h.\Iphi(x,\xi)=\Iphi(x,h.\xi)$ for every $h\in H$; as $\Iphi(x,\xi)$ and $\Iphi(x,h.\xi)$ are components of the support of $G_{(a,x)}$, they are either equal or disjoint.
		On the other hand, by definition of $C$ and the assumption $|\Iphi(x,\xi)|>C$, generators of $H$ cannot move $\Iphi(x,\xi)$ disjoint from itself.
		
		These two remarks together give that  $h_i.\mathrm \Iphi(x,\xi)=\Iphi(x,\xi)$ for every $i\in \{1,\ldots,n\}$, and therefore $H.\Iphi(x,\xi)=\Iphi(x,\xi)$. In particular, we must have $\fixphi(G_{(y,b)})\neq\emptyset$. This gives the desired contradiction since, by Proposition \ref{p-lm-domination-balance}, the subgroups $G_{(y,b)}$ must act without fixed points.
	\end{proof}

	Now to prove the lemma,  consider the positive constant $\delta_+$ as in \eqref{eq:delta+}, and take a laminar action $\varphi\in \mathcal E$. For any leaf $I\in \overline{\mathcal L}_\varphi$ containing the central leaf $I_{S,\varphi}$, the intersection $s.I\cap I$ is non-empty for every $s\in S$ (see Lemma \ref{l-central-leaf}). By the choice of $\delta_+$, we can choose $s_I\in S$ so that $|s_I.I|\geq|I|+\delta_+$. Let us apply this inductively to the sequence of intervals $I_0=I_{S,\varphi}$, $I_{n+1}=s_n.I_n$, where $s_n=s_{I_n}\in S$. Writing $g_n=s_n\cdots s_0$, we thus have $|g_n.I_{S,\varphi}|\geq |I_{S,\varphi}|+n\delta_+>n\delta_+$.
	
	Fix now a point $x\in X$, and the constant $C>0$ from the claim. We can then choose $N\in \N$ such that $N\delta_+>C$. Next, we take a point $\xi\in\Xi_\varphi\cap g_N.I_{S,\varphi}$. On the one hand, we have $|\Iphi(x,\xi)|<C$, while on the other hand we have $|g_N.I_{S,\varphi}|>N\delta_+>C$. As both intervals contain $\xi$, we must have $\Iphi(x,\xi)\subset g_N.I_{S,\varphi}$.
	
	When $\varphi\in \mathcal E_+$, the horograding $\hor_\varphi$ is non-decreasing, and we conclude that $\hor_\varphi(g_N.I_{S,\varphi})\geq x$. By equivariance of $\hor_\varphi$, we obtain that $\hor_\varphi(I_{S,\varphi})=c_{S,\varphi}\geq g_{N}^{-1}(x)$. Since $x$ and $N$ are fixed and $\varphi$ was arbitrary, $c=g_N^{-1}(x)$ gives the desired uniform lower bound. When $\varphi\in \mathcal E_-$, a similar reasoning gives $c_{S,\varphi}\le g_N^{-1}(x)$, leading to the desired conclusion.
\end{proof}

\begin{proof}[Proof of Theorem \ref{thm-decoDeroin}] We first show that $\overline{\mathcal{E}}\subset \mathcal{E}\cup \mathcal{Q}$. For this, consider a convergent sequence $\varphi_n\to\varphi$ with $\varphi_n\in\mathcal{E}$ for every $n\in\N$. Write $\varphi_n=\Phi^{t_n}(\varphi'_n)$ with $\varphi'_n$ such that $0\in I_{S,\varphi'_n}$. Then, up to considering a subsequence we can assume that we are in one of the following cases:
	\begin{enumerate}[label=(\arabic*)]
		\item\label{lab1tn} either $t_n$ is bounded, or
		\item\label{lab2tn} $t_n\to\pm\infty$.
	\end{enumerate}
	Assume first that we are in case \ref{lab1tn}. Consider $K\in\R$ such that $|t_n|<K$ for every $n\in\N$, and take the constant $C_2$ from Lemma \ref{lem.uniformsize}, so that $|I_{S,\varphi_n}|< C_2$ for every $n\in\N$.
	Then we have $I_{S,\varphi_n}\subset (-K-C_2,K+C_2)$ for every $n\in\N$, and Lemma \ref{lem.centerlaminated} gives that $\varphi$ is laminar, and thus $\varphi\in\mathcal{E}\cup\mathcal{Q}$. In case \ref{lab2tn}, Lemma \ref{lem.equivalentconverg} implies that  $\varphi=\lim \Phi^{t_n}(\varphi'_n)= \Phi^{\hor^+_{S, \varphi_n}(t_n)}(\iota)$.   On the other hand, Lemma \ref{l-partial-quasi-isometry} implies that $\hor^+_{S, \varphi_n}(t_n)-c_{S, \varphi_n}\to +\infty$, and thus, by Lemma \ref{lem.conclusion}, we have $\hor^+_{S, \varphi_n}(t_n)\to \infty$. Then, by Lemma \ref{l-deroin-partition-general}, we have $\varphi=\lim \Phi^{\hor^+_{S, \varphi_n}(t_n)}(\iota) \in \mathcal{Q}$. This shows that $\overline{\mathcal{E}}\subseteq \mathcal{Q}\cup \mathcal E$. The fact that $\Ical$ and $\widehat{\Ical}$ are open now follows from this and from Lemma \ref{l-deroin-partition-general}, since 
	\[\mathcal{I}=\Der_\mu(G;\R)\setminus(\widehat{\mathcal{I}}\cup\mathcal{Q}\cup\mathcal{E})\quad\text{and}\quad \widehat{\mathcal{I}}=\Der_\mu(G;\R)\setminus(\mathcal{I}\cup\mathcal{Q}\cup\mathcal{E}).\]	
	This shows \ref{i-closure-exotic}. To show \ref{i-flow-proper}, by \ref{i-closure-exotic}, it is enough to show that the restriction of the translation flow to $\mathcal{E}$ is proper. For this, fix a compact subset $\mathcal K\subset\mathcal{E}$, and a distance $\dist$ on $\Der_\mu(G; \R)$ inducing the topology. Since $\mathcal K\cap \overline{\mathcal{I}}=\varnothing$, the distance between points in $\mathcal K$ and $\mathcal{I}$ is uniformly bounded from below by some  $\varepsilon>0$. By Lemma \ref{lem.equivalentconverg}, we can fix $t_0>0$ such that $\dist(\Phi^{t}(\varphi), \mathcal{I})<\varepsilon$ for every $\varphi\in\mathcal{E}$ with $0\in I_{S,\varphi}$, and every $|t|>t_0$. In particular, every $\varphi\in \mathcal K$ can be written as $\varphi=\Phi^t(\varphi')$, for some $|t|\le t_0$ and $\varphi'\in\Der_\mu(G;\R)$ such that $0\in I_{S,\varphi'}$. Thus, for every $|t|>2t_0$,  $\Phi^t(\mathcal K)$ is contained in the $\varepsilon$-neighborhood of $\mathcal{I}$, so that $\Phi^t(\mathcal K)\cap \mathcal K=\varnothing$. This concludes the proof of the theorem.
\end{proof}

\section{Examples} \label{s-examples-class}
\subsection{Thompson's group $F$}\label{s-F-Deroin}

As already mentioned, Thompson's group $F\subset \homeo_0((0, 1))$ satisfies condition \ref{class}; namely,  the subgroups $F_{(0, x)}$ and $F_{(x, 1)}$ are isomorphic to $F$ (hence finitely generated) whenever $x$ is a dyadic number. Therefore all the discussion in this chapter applies to $F$. In particular Corollary \ref{c-local-rigidity} implies the following.

\begin{thm}[Local rigidity of the standard action] \label{t-F-lr}
	The standard piecewise linear action of $F$ on $(0,1)$ is locally rigid.
\end{thm}

Moreover, Corollary \ref{c-class-borel} implies that actions of $F$ on the line can be distinguished by a Borel complete invariant up to semi-conjugacy.

\begin{thm}[Smoothness of the semi-conjugacy relation]\label{t-F-classification}
	The semi-conjugacy relation on the space $\Homirr(F,\homeo_{0}(\R))$ is smooth. In particular, the conjugacy relation on the subspace of minimal actions of $F$ on $\R$ is smooth.
\end{thm}
\begin{proof}
	After Corollary \ref{c-class-borel}, the semi-conjugacy relation on $\Homirr(F,\homeo_{0}(\R))$ is smooth if and only if the same is true for the quotient $F/[F_c,F_c]\cong \Z^2$.  Now, every irreducible action of $\Z^2$ is semi-conjugate to an action by translations, thus any $\mu$-harmonic action of $\Z^2$ (for any appropriate measure $\mu$ on $\Z^2$) is actually an action by translations. Since the translation flow on $\Der_\mu(\Z^2; \R)$ corresponds to conjugating actions by translations, it is actually trivial. We conclude using Corollary \ref{cor.cross-section}.
\end{proof} 

In fact, the proof of Theorem \ref{thm-decoDeroin} can be specialized to the case of $F$ to obtain a more precise description of the space $\Der_\mu(F;\R)$. A cartoon of this space $\Der_\mu(F;\R)$ appears in Figure \ref{fig-F-Der}; the purpose of the rest of this subsection is to give a detailed description of this picture.

\begin{figure}[ht]
	\centering
	\includegraphics[scale=1]{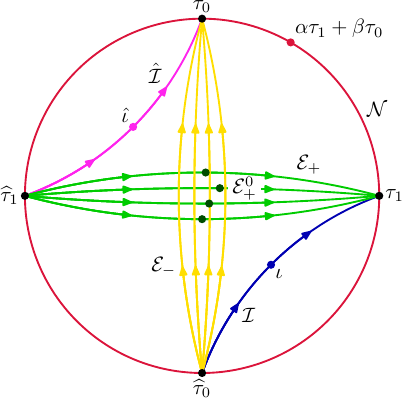}
	\caption{The space of normalized $\mu$-harmonic actions of Thompson's group $F$. See \S \ref{s-F-Deroin} for a detailed explanation.}\label{fig-F-Der}
\end{figure}

From now on, we fix a symmetric probability measure $\mu$ on $F$ supported on a finite generating set $S$ of $F$, and consider the associated space of normalized $\mu$-harmonic actions $\Der_\mu(F;\R)$, with its translation flow $\Phi$.  We will follow the discussion in the previous chapter, and refine it when possible to obtain more precise information in this special case. As in \eqref{e-lr-decomposition}, we have a decomposition of $\Der_\mu(F;\R)$ into $\Phi$-invariant subspaces
\[\Der_\mu(F;\R)=\mathcal{Q}\sqcup \Ical \sqcup \widehat{\Ical} \sqcup \mathcal{E},\]
defined as follows:
\begin{itemize}
	\item we let $\mathcal{Q}\subset \Der_\mu(F;\R)$ is the set of non-faithful actions, induced from actions of $F^{ab}\cong \Z^2$;
	\item we fix a representative $\iota\in \Der_\mu(F;\R)$ of the standard action on $(0, 1)$, and let $\mathcal{I}=\{\Phi^t(\iota): t\in \R\}$ denote its $\Phi$-orbit, whilst $\widehat{\mathcal{I}}$ denotes the $\Phi$-orbit of the reversed action $\widehat{\iota}$;
	\item we let $\mathcal{E}=\mathcal{E}_+\sqcup \mathcal{E}_-$ be the subset of $\Der_\mu(F;\R)$ of laminar actions, where $\mathcal{E}_+$ (respectively, $\mathcal{E}_-$) is the subset of laminar actions which are positively (respectively, negatively) horograded by the standard action on $(0, 1)$.
\end{itemize}

Let  $\tau_0, \tau_1\colon F\to \Z$ be the two homomorphisms defined as in \eqref{e-F-germs0}:
\[\tau_0(g)=-\log_2 D^+g(0) \quad \text{and}\quad \tau_1(g)=-\log_2 D^-g(1).\]
Note that these induce identifications of  the groups of germs $\Germ(F,0)$ and $\Germ(F, 1)$ with $\Z$. We may identify $\tau_0$ and  $\tau_1$ with two points of $\mathcal{Q}$ given by their corresponding cyclic actions. As usual, we denote by $\widehat{\tau}_0, \widehat{\tau}_1\in \mathcal{Q}$ their conjugate by the map $x\mapsto -x$.

We begin by observing that Lemma \ref{l-deroin-partition-general} can be refined in this case to obtain the following description of the sets $\mathcal{Q}$ and $\Ical, \widehat{\Ical}$ inside $ \Der_\mu(F; \R)$.

\begin{lem} \label{l-Deroin-F-general}
	The subset $\mathcal{Q}\subset \Der_\mu(F; \R)$ is homeomorphic to the circle $\T$ and is fixed by the flow $\Phi$. The actions $\iota$ and $\widehat{\iota}$ satisfy the following:
	\[\ \lim_{t\to+\infty} \Phi^t(\iota)=\tau_1, \quad\lim_{t\to -\infty} \Phi^t(\iota)=\widehat{\tau}_0, \quad  \lim_{t\to +\infty} \Phi^{t}(\widehat{\iota})=\tau_0, \quad\lim_{t\to -\infty} \Phi^{t}(\widehat{\iota})=\widehat{\tau}_1.\]	
\end{lem}

The set $\mathcal{Q}$ corresponds to the circle shown in red in Figure \ref{fig-F-Der}. The orbit  $\Ical$ is therefore a copy of $\R$ inside $\Der_\mu(F;\R)$ which connects the points $\widehat{\tau}_0$ and $\tau_1$,  while $\widehat{\Ical}$ connects $\tau_0$ to $\widehat{\tau}_1$ as shown in Figure \ref{fig-F-Der}. 

\begin{proof}[Proof of Lemma \ref{l-Deroin-F-general}]
	The set $\mathcal{Q}$ is homeomorphic to $\Der_{\bar{\mu}}(\Z^2;\R)$, where $\bar{\mu}$ is the projection of $\mu$ to $F^{ab}\cong \Z^2$. Every element of $\Der_{\bar{\mu}}(\Z^2;\R)$ corresponds to a   $\Z^2$-action  by translations given by a non-trivial homomorphism $\varphi\colon \Z^2\to (\R, +)$ up to rescaling  by a positive real, so that $\Der_{\bar{\mu}}(\Z^2;\R)$ is homeomorphic to the circle $\T$, and it consists of points which are fixed by the translation flow $\Phi$. We deduce that  $\mathcal{Q}$ is a closed subset of $\Der_\mu(F;\R)$, homeomorphic to a circle, and it consists of points which are fixed by the translation flow $\Phi$. 
	Let us show, for instance, that $\lim_{t\to+\infty} \Phi^t(\iota)=\tau_1$ (the other cases are analogous). By Lemma \ref{l-deroin-partition-general}, if $\psi$ is a limit point of $ \Phi^t(\iota)$ for $t\to \infty$, then $\psi$ factors through $F/F_+\cong \Z$, so $\psi$ is necessarily a cyclic action, and we only need to  determine its sign. This is done by noting that the homomorphism $\tau_1$ has been defined in such a way that $\tau_1(g)>0$ if and only if $g(x)>x$ for every  $x$ close enough to 1 (in the standard action of $F$ on $X=(0, 1)$). 
\end{proof}

To complete the explanation of Figure \ref{fig-F-Der}, we give a description of the dynamics of the flow $\Phi$ on the set $\mathcal{E}$. For this it is convenient to fix an element $f\in F$ which in the standard action satisfies $f(x)>x$ for every $x\in (0, 1)$. For definiteness, we can choose as $f$ the element of the standard generating pair of $F$ defined in \eqref{e-F-big-generator0}.

By Proposition \ref{p-dyn-class-elements-horograded}, for every $\varphi\in \mathcal{E}$, the element $\varphi(f)$ must be a homothety, which is expanding if $\varphi\in \mathcal{E}_+$, and contracting if $\varphi\in \mathcal{E}_-$. For $\varphi\in \mathcal{E}$, we let $\xi_\varphi$ be the unique fixed point of $\varphi(f)$. We say that $\varphi\in \mathcal{E}$ is $f$-centered if $\xi_\varphi=0$, and let $\mathcal{E}^0\subset \mathcal{E}$ be the subset of $f$-centered laminar actions. Finally we set $\mathcal{E}^0_\pm=\mathcal{E}^0\cap \mathcal{E}_\pm$.  
With this notation, we have the following more precise version of Theorem \ref{thm-decoDeroin}.

\begin{prop}\label{p-F-limits}
	Retain all notation as above. Then both sets $\mathcal{E}_+$ and $\mathcal{E}_-$ are open.  Each subset $\mathcal{E}_\pm^0\subset \mathcal{E}_\pm$ is closed in $\Der_\mu(F;\R)$ (hence compact), and it is a cross section for the flow $\Phi$ to $\mathcal{E}_\pm$ (namely it intersects every orbit in exactly one point).  Moreover, for every $\varphi\in \mathcal{E}_\pm$ the limits $\lim_{t\to \pm \infty} \Phi^t(\varphi)$ exist and the following hold:
	\begin{itemize}
		\item if $\varphi\in \mathcal{E}_+$, then 
		\[\lim_{t\to +\infty} \Phi^t(\varphi)=\tau_1\quad\text{and}\quad\lim_{t\to -\infty }\Phi^t(\varphi)=\widehat{\tau}_1,\] 
		where the convergence is uniform over $\varphi$ in the cross section $\mathcal{E}^0_+$;
		\item if $\varphi\in \mathcal{E}_-$, then 
		\[\lim_{t\to +\infty} \Phi^t(\varphi)=\tau_0\quad\text{and}\quad\lim_{t\to -\infty} \Phi^t(\varphi)=\widehat{\tau}_0,\]
		where the convergence is uniform over  $\varphi$ in the cross section $\mathcal{E}^0_-$.
	\end{itemize}
	As a consequence $\overline{\mathcal{E}_+}=\mathcal{E}_+\cup\{\tau_1, \widehat{\tau}_1\}$ and $\overline{\mathcal{E}_-}=\mathcal{E}_-\cup\{\tau_0, \widehat{\tau}_0\}$.
	
\end{prop}
\begin{proof}
	
	We already know that $\mathcal{E}$ is open, by Lemma \ref{l-deroin-partition-general}. Moreover, for $\varphi\in \mathcal{E}_+$ (respectively, $\varphi\in \mathcal{E}_-$) we have that $\varphi(f)$ is an expanding (respectively, contracting) homothety. This  clearly implies that the sets $\mathcal{E}_\pm$ are both open. For every $\varphi\in \mathcal{E}_\pm$, we have $\Phi^{\xi_\varphi}(\varphi)\in \mathcal{E}_\pm^0$, so each set $\mathcal{E}^0_\pm$ is a cross section of the flow $\Phi$ on $\mathcal{E}_\pm$.
	
	Choose a generating set $S$ of $F$ containing $f$. Then for every $\varphi\in \mathcal{E}$ and $\varphi$-invariant lamination $\mathcal{L}_\varphi$, the central leaf $I_{S,\varphi}\in \mathcal{L}_\varphi$ contains $\xi_\varphi$ (the fixed point of $f$), since otherwise  $\varphi(f)(I_{S,\varphi})$ cannot be related by inclusion with $I_{S,\varphi}$ (see  Lemma \ref{l-central-leaf}). Hence, Lemmas \ref{lem.uniformsize} and \ref{lem.centerlaminated} imply that if $\psi$ is  any action in the closure of $\mathcal{E}^0_\pm$, then it is laminar and thus  belongs to $\mathcal{E}$ (since the complement of $\mathcal{E}$ does not contain any laminar action). Moreover, $\psi(f)$ fixes $0$. This implies that $\overline{\mathcal{E}^0_\pm}\subseteq \mathcal{E}^0_\pm$, and thus the subsets $\mathcal{E}^0_\pm$ are closed.	The statements on convergence now follow from Lemmas \ref{lem.equivalentconverg} and \ref{lem.conclusion}, by the same argument as in the proof of Theorem \ref{thm-decoDeroin}. 
\end{proof}

To conclude this discussion, we observe that it is a tantalizing problem to obtain further results on the topology of the compact cross sections $\mathcal{E}^0_+$ and  $\mathcal{E}^0_-$, which is at the moment quite mysterious. Note that each section $\mathcal{E}^0_\pm$ is homeomorphic to the quotient space $\mathcal{E}_\pm/\Phi$, and thus it is independent, up to homeomorphism, on the choice of the generator $f\in F$ made above, and by symmetry the two sections are homeomorphic one to the other. The constructions of laminar actions in Chapter \ref{s-F} show  that the spaces $\mathcal{E}^0_+$ and $\mathcal{E}^0_-$ are uncountable, and contain homeomorphic copies of a Cantor set. However we were not able to construct any non-trivial connected subset of $\mathcal{E}^0_+$, and we do not know whether they are totally disconnected. We also do not know the answer to the following question. 
\begin{ques}\label{q-rigid-actions-F}
	Do the cross sections $\mathcal{E}^0_+$ and $\mathcal{E}^0_-$ admit isolated points?
\end{ques}
By Corollary \ref{prop.rigidityderoin}, this is equivalent to the question of whether $F$ admits minimal laminar actions which are locally rigid. 

\subsection{Other groups of piecewise linear and projective homeomorphisms} \label{s-examples-lr-pl}
Beyond Thompson's group $F$, some other examples of finitely generated locally moving groups which satisfy \ref{class}\; are provided by other well-studied groups of piecewise linear and projective homomorphisms. Let us mention a few.
\begin{itemize}
	\item The Thompson--Brown--Stein groups $G=F_{n_1,\ldots, n_k}$. Indeed, they satisfy that the subgroups $G_{(0,a)}$ and $G_{(a,1)}$ are finitely generated for any $a\in\Z[1/n_1,\ldots,1/n_k]$. This follows from \cite[Corollary B9.10]{BieriStrebel}.
	\item All Bieri--Strebel groups of the form $G:=G(\R;A,\Lambda)$ with finitely generated fragmentable subgroup. Recall that this is equivalent to the conditions that $A$ is a finitely generated $\Z[\Lambda]$-module, $\Lambda\subset\R_{>0}$ a finitely generated subgroup, and the quotient $A/I\Lambda\cdot A$ is finite (compare with Lemma \ref{p.BieriStrebel_fg_germs}). In this case condition \ref{class} follows from \cite[Theorem B.8.2]{BieriStrebel}, which shows that the subgroups $G_{-(\infty, x)}$ and $G_{(x, +\infty)}$ are finitely generated for $x\in A$.
	\item Similar examples satisfying all assumptions of Theorem \ref{thm-decoDeroin} arise as subgroups of the group of piecewise projective homeomorphisms of the real line. For instance, the group introduced by Lodha and Moore  in \cite{LodhaMoore}  (condition \ref{class} can be shown using  the symbolic description derived in \cite{LodhaMoore}).
	
\end{itemize}

\subsection{Groups with cyclic germs}
The following criterion provides many examples of groups satisfying Theorem \ref{thm-decoDeroin}, outside the realm of piecewise linear and piecewise projective homeomorphisms. For example, it applies to the class of (pre-)chain groups studied by Kim, Koberda, and Lodha \cite{KKL}.
\begin{prop} \label{p-class-cyclic}
	Let $G\subset \homeo(X)$ be a group acting minimally on $X=(a, b)$. Suppose that $G$ has a finite generating set $S=\{s_1,\ldots, s_k\}$ such that  $s_1$ is supported on an interval of the form $(a, x)$ with $x\in X$,  the supports of $s_2,\ldots, s_{k-1}$ are relatively compact in $X$, and $s_k$  is supported on an interval of the form $(y, b)$ with $y\in X$.  Then $G$ is locally moving and satisfies \ref{class}.
\end{prop}
Thus any group $G$ as in Proposition \ref{p-class-cyclic} satisfies Theorem \ref{thm-decoDeroin} and its consequences. In particular, Corollary \ref{c-local-rigidity} implies the following.
\begin{cor}
	For any group $G$ as in Proposition \ref{p-class-cyclic}, the standard action of $G$ on $X$ is locally rigid.
\end{cor}
\begin{proof}[Proof of Proposition \ref{p-class-cyclic}]
	Since elements of $S$ are contained in $G_-\cup G_+$, the group $G$ is fragmentable, and since it is finitely generated and acts minimally on $X$, it is locally moving (by Lemma \ref{l-lm-Gstar-fix}). We need to check condition \ref{class}. 
	We will find $z\in X$ such that $G_{(a, z)}$ is contained in a finitely generated subgroup of $G_+$ (the symmetric case is analogous). For this, choose $z\in X$ such that  $s_1, \cdots, s_{k-1}\in G_{(a, z)}$. Note that necessarily $s_k(z)\neq z$ (else $z$ would be fixed by $G$). Let $y\in X$ be as in the statement, and choose $g\in G_c$ such that $g(y)>\max\{s_k(z), s_k^{-1}(z)\}$. 
	Let $\tilde{s}_k=[s_k, g]=s_k(g s_k g^{-1})^{-1}$. Note  that $gs_kg^{-1}$ coincides with $s_k$ on a neighborhood of $b$ (since $g\in G_c$), and is supported on $G_{(g(y), b)}$. The former remark gives that $\tilde{s}_k\in G_+$, and the latter that $\tilde{s}_k^{\pm 1}\restriction_{(a, z]}=s_k^{\pm 1}\restriction_{(a, z]}$. 
	Set $H=\langle s_1, \cdots, s_{k-1}, \tilde{s}_k\rangle$, and note that $H\subset G_+$. We will show that $H$ contains  $G_{(a, z)}$. Since $H$ is finitely generated,  this will give the desired conclusion. Let $I=(c, d)$ be the component of $\supp (\tilde{s}_k)$ containing $z$. Since $\tilde{s}_k$ and $s_k$ and their inverses agree on $(a, z)$, we have that $J:=(c, b)$ is a component of the support of $s_k$ (indeed $s_k$ cannot have any fixed points in $(z, b)$, as these would be fixed by the whole group $G$). Thus we can find a homeomorphism $f\colon J\to I$, equal to the identity on $(c, z)$, such that $fs_kf^{-1}=\tilde{s}_k\restriction_I$ (it is enough to define $f$ to be the identity on any fundamental domain of $s_k$ contained in $(c, z)$, and extend it by equivariance). We extend $f$ to a homeomorphism $f\colon X\to (a, d)$, equal to the identity on $(a, z)$.
	
	Take now $g\in G_{(a, z)}$, and write $g=t_n\cdots t_1$, with  $t_i\in S\cup  S^{-1}$. Let $\tilde{g}=\tilde{t}_n\cdots \tilde{t}_1\in H$, where $\tilde{t}_i=\tilde{s}_k^{\pm 1}$ if $t_i=s_k^{\pm 1}$, and $\tilde{t}_i=t_i$ otherwise. We claim that $\tilde{g}=g$. First of all note that all the elements $s_1, \ldots, s_{k-1}, \tilde{s}_k$ fix $c$, and only $\tilde{s}_k$ acts non-trivially on $(c, b)$. Since $g\in G_{(a, z)}$ has trivial germ at $b$, and the group of germs of $G$ at $b$ is generated by the image of $s_k$, we have that the total sum of exponents of occurrences of $s_k$ in the word $t_n\cdots t_1$ is 0. Hence the same is true for occurrences of $\tilde{s}_k$ in $\tilde{t}_n\cdots \tilde{t}_1$, and thus 
	\begin{equation} \label{e-c-b} \tilde{g}\restriction_{(c, b)}=\id\restriction_{(c, b)}=g\restriction_{(c, b)}.\end{equation}
	To determine $\tilde{g}\restriction_{(a, c)}$ note that, since $f= \id$ on $(a, z)$, and each $s_i$ is supported on $(a, z)$ for $i\le k-1$,  we have $\tilde{t}_i\restriction_{(a, c)}=ft_if^{-1}$ for every $i$. Thus 
	\begin{equation} \label{e-a-c} \tilde{g}\restriction_{(a, c)}= \textstyle\prod_i (\tilde{t}_i\restriction_{(a, c)})= \prod_i (ft_if^{-1})= f(\prod_i t_i)f^{-1}= f gf^{-1}=g\restriction_{(a, c)}.\end{equation}
	The last equality uses that $g\in G_{(a, z)}$ and that $f$ is the identity on $(a, z)$. Now \eqref{e-c-b} and \eqref{e-a-c} together imply that $g=\tilde{g}\in H$. Since $g\in G_{(a, z)}$ was arbitrary, we have $G_{(a, z)}\subset H$, as desired.  \qedhere
\end{proof}

\subsection{Extending arbitrary actions to locally rigid ones}

The proof of the following proposition is based on an idea similar to that for the proof of Proposition \ref{p-class-cyclic}.
\begin{prop}\label{prop.superF}
	For $X=(a,b)$, let $H\subset \homeo_0(X)$ be a countable subgroup. Then $H\subseteq G$ for some finitely generated locally moving subgroup $G\subset \homeo_0(X)$ satisfying \ref{class}.  
\end{prop}

\begin{proof} We can assume without loss of generality that $H$ is finitely generated, because every countable subgroup of $\homeo_0(X)$ is contained in a finitely generated one, by a result of Le Roux and Mann \cite{LeRouxMann}. Assume also for simplicity that  $X=(0, 1)$. Let $b_\ell\colon (0, 1)\to (0, {3/4})$ and $b_r\colon (0, 1)\to (1/4, 1)$ be the homeomorphisms defined respectively by
	\[b_\ell(x)=\left\{\begin{array}{lr}x & \text{if }x\in (0, 1/2),\\[.5em] \frac{1}{2}x+\frac{1}{4} &\text{if } x\in [1/2, 1),
	\end{array} \right. \quad b_r(x)=\left\{\begin{array}{lr}\frac{1}{2}x +\frac{1}{4}&\text{if } x\in (0, 1/2),\\[.5em] x &\text{if } x\in [1/2, 1).
	\end{array} \right. \]
	Set also $b_0=b_\ell b_r$.
	That is, $b_0$ is the homothety of slope $1/2$ centered at the point $1/2$, while each of $b_\ell$ and  $b_r$ fixes half of the interval $(0, 1)$ and coincides with $b_0$ on the other half. See Figure \ref{fig.maps_classF}.
	
	\begin{figure}[ht]
		\centering
		\includegraphics[scale=1]{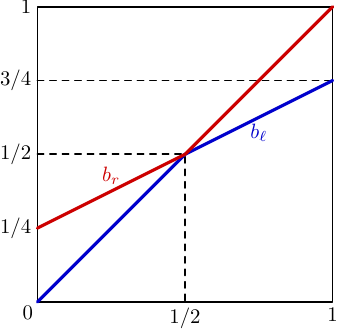}
		\caption{The maps $b_r$ (red) and $b_\ell$ (blue).}\label{fig.maps_classF}
	\end{figure}
	
	Set $H_0=b_0Hb_0^{-1}$, $H_\ell=b_\ell H b_\ell^{-1}$, and $H_r=b_rHb_r^{-1}$. Note that each of these groups is a group of homeomorphisms of a subinterval of $(0, 1)$, and we see them as groups of homeomorphisms of $(0, 1)$ by extending them to the identity outside their support. Set $G=\langle H, H_\ell, H_r, H_0, F\rangle$, where $F$ is the standard copy of Thompson's group $F$ acting on $(0, 1)$.   As $F$ and $H$ are finitely generated, so is $G$; moreover, $G$ is locally moving because $F$ is so. It remains to prove that $G$ satisfies \ref{class}. 
	
	Let us show that $G_{(0, 1/2)}$ is contained in a finitely generated subgroup of $G_+$. To this end, consider the group $\Gamma=b_\ell G b_\ell ^{-1}$. Then $\Gamma$ is a finitely generated subgroup of $\homeo_0((0, 1))$  supported in $(0, 3/4)$ (we again extend it as the identity on $(3/4, 1)$). Moreover, $\Gamma$ contains $G_{(0, 1/2)}$: indeed since $b_\ell$ acts trivially on $(0, 1/2)$,  for $g\in G_{(0, 1/2)}$ we have $g=b_\ell gb_\ell^{-1}\in \Gamma$. Thus the desired conclusion follows if we show that $\Gamma$ is contained in $G$ (hence in $G_+$). 
	
	For this, note that $b_\ell F b_\ell^{-1}=F_{(0, 3/4)} \subset F \subset G$ and $b_\ell Hb_\ell^{-1}=H_\ell \subset G$; we also easily observe
	\[
	b_\ell H_rb_{\ell}^{-1}=b_\ell b_r H (b_\ell b_r)^{-1}=b_0H b_0^{-1}=H_0\subset G.
	\]
	Choose an element $f\in F$ which coincides with $b_\ell$ in restriction to $(0, 3/4)$. Then the conjugation action of $f$ on the subgroups $H_\ell, H_0$ coincides with the conjugation by $b_\ell$, thus $b_\ell H_\ell b_\ell^{-1}= f H_\ell f^{-1}$ and $b_\ell H_0 b_\ell^{-1}=f H_0 f^{-1}$ are also contained in $G$.  Since
	\[\Gamma=b_\ell \langle H, H_0, H_\ell, H_r, F\rangle b_\ell^{-1},\]
	we conclude that  $\Gamma$ is contained in $G$. Thus $\Gamma \subset G_+$ is a finitely generated subgroup that contains $G_{(0, 1/2)}$. Since for every $x\in (0,1)$ the group $G_{(0, x)}$ is conjugate to a subgroup of $G_{(0, 1/2)}$, we have that $G_{(0, x)}$ is contained in a finitely generated subgroup of $G_+$ for every $x$. The case of the subgroups $G_{(x, 1)}$ is analogous. \qedhere
\end{proof}
We point out the following consequence, which shows that it is not possible in general to perturb non-trivially a group action on the line, by looking only at its restriction to a subgroup.

\begin{cor}\label{cor-contain-loc-rigid} Let $H\subset\homeo_0(\R)$ be a countable group. Then, there exists a finitely generated subgroup $G\subset \homeo_0(\R)$ containing $H$ such that the action of $G$ on $\R$ is locally rigid.
\end{cor}
\begin{proof}  By Proposition \ref{prop.superF} there exists a finitely generated group $G\subset\homeo_0(\R)$ that contains $H $ and satisfies all assumptions in Corollary \ref{c-local-rigidity}, so that its standard action is locally rigid. 
\end{proof}

\section{A non-locally rigid example} \label{s-example-non-lr}
The assumption \ref{class} in Theorem \ref{thm-decoDeroin} might seem just a technical requirement. However, this assumption is substantial and cannot be dropped. Here we construct an example of a subgroup of subgroup of $\homeo_0(\R)$ which finitely generated, locally moving, and fragmentable (in particular, it satisfies Theorem \ref{t-lm-horograding}), but does not satisfy the main consequences of Theorem \ref{thm-decoDeroin} (due to the failure of \ref{class}).  

\begin{prop}\label{p-propononrigid} There exists a finitely generated, fragmentable locally moving subgroup $H\subset\homeo_0(\R)$ satisfying the following. 
	\begin{enumerate}[label=(\roman*)]
		\item \label{i-non-lr} The standard action of $H$ is not locally rigid.
		\item \label{i-non-hausdorff} Let  $\mathcal{M}\subset \Homirr(H, \homeo_0(\R))$ be the subspace of faithful minimal actions. Then the quotient of $\mathcal{M}$ by the equivalence relation of positive conjugacy is not Hausdorff. 
	\end{enumerate}
\end{prop}

\begin{proof} The example is based on a variation of the jump preorder construction for groups of piecewise linear homeomorphisms, introduced in \S\ref{s.BSjump}.
	Recall that for a piecewise linear map $f\colon \R\to \R$, with a discrete (but possibly infinite) set $\BP(f)$ of discontinuity points for the derivative, we define the associated jump cocycle as the map
	\[\dfcn{j_g}{\R}{\R}{x}{\dfrac{D^+g^{-1}(x)}{D^-g^{-1}(x)}.}\]
	The desired example is given by the subgroup $H=\langle G,f\rangle$ of $\homeo_{0}(\R)$ generated by the Bieri--Strebel group $G=G(\R;\Z[1/2],\langle 2\rangle_*)$ and one extra element $f$ defined as the piecewise linear map, with an infinite set of breakpoints $\BP(f)$, satisfying the following conditions (see Figure \ref{fig:non-loc-rigid}): \begin{itemize}
		\item $f(x)=x$ for $x\ge 0$,
		\item $\BP(f)=\frac12\Z_{\le 0}$, and $f(x)=x$ for $x\in \BP(f)$.
		\item $j_f(-n)=p_{n}$ and $j_f(-n-\frac{1}{2})=1/p_{n}$ for every integer $n\ge 0$, where $(p_n)_{n\ge 0}\subset \N$ denotes the increasing enumeration of primes (we keep this convention for the rest of the proof).
		\begin{figure}[ht]
			\[
			\includegraphics[scale=1]{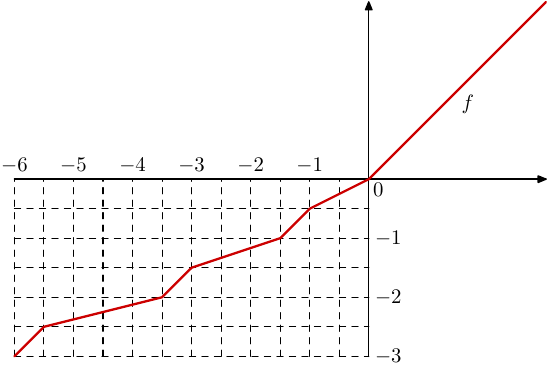}
			\]	
			\caption{The piecewise linear map $f$ for the non-locally rigid example.
			}\label{fig:non-loc-rigid}
		\end{figure}
	\end{itemize}
	
	Let us first check that $H$ is finitely generated and fragmentable.
	On the one hand, finite generation of $G$, and therefore of $H$, follows from \cite[Theorem B7.1]{BieriStrebel}. On the other hand, as $f\in H_+$, to show that $H$ is fragmentable, it is enough to show that $G$ is. To see this, recall from Lemma \ref{p.BieriStrebel_fg_germs} that we have $G/\Gfrag\cong A/I\Lambda\cdot A$, where $\Lambda=\langle 2\rangle_*$ and $A=\Z[\frac{1}{2}]$, and then notice that $I\Lambda\cdot A=A$ (see Example \ref{e-G-lambda}). This gives the desired conclusion.
	
	For the rest of the proof, we need to define a sequence of minimal laminar actions,  that accumulate on the standard action of $H$. 
	To start with, we remark that from the definition of the subgroup $H$, for every $g\in H$, the jump cocycle $j_g$ can take only rational values, and 
	is a function whose support is discrete, and can only accumulate at $-\infty$.
	The multiplicative abelian group $\Q_{>0}$ is free abelian, with basis given  by the primes $p_n$. We denote by $j_g^n$ the jump cocycle $j_g$ post-composed with the projection homomorphism $\Q_{>0}\to \langle p_n\rangle$, for fixed $n\ge 0$. We will denote by $\mathsf{S}_{n}=\{j_g^n:g\in H\}$ the collection of the (reduced) jump cocycles. Note that any element of $\mathsf S_n$ is a function on $\R$ whose support is discrete, and can only accumulate on $-\infty$, so that when $g,h\in H$ are such that $j_g^n\neq j_h^n$, the point 
	\[\overline{x}^n_{g,h}=\max\{x\in\R:j_g(x)\not\equiv_n j_h(x)\}\]
	is well defined, and we can declare
	$j_g^n\prec_n j_h^n$ if $j_g^n(\overline{x}^n_{g,h})<j_h^n(\overline{x}^n_{g,h})$. As in Lemma \ref{lem.invaPhi}, one checks that the action of $H$ on $\mathsf S_n$ defined by $g\cdot j_h^n=j_{gh}^n$ preserves the order $\prec_n$. Let us denote by $\Psi_n\colon H\to \Aut(\mathsf S_n,\prec_n)$ this order-preserving action.
	We proceed to define a $\Psi_n$-invariant prelamination $\mathcal L^n$ on $(\mathsf{S}_n,\prec_n)$, analogous to the one appearing in the proof of Proposition \ref{prop-jumphoro}. For this, given $g\in H$ and $x\in \R$, consider the subset
	\[L^n_{g,x}=\{j_h^n\in\mathsf S_n:j_h^n(y)= j_g^n(x)\text{ for any } y> x\}\]
	and define $\mathcal{L}_0^n=\{L_{g,x}:g\in H,\,x\in\R\}$.
	As in the proof of Proposition \ref{prop-jumphoro}, one can check that $\mathcal{L}_0^n$ is a $\Psi_n$-invariant covering prelamination with action of $H$ given by $\Psi_n(g)(L^n_{h,x})=L^n_{gh,g(x)}$. Moreover, the map $\hor_n\colon \mathcal{L}_0^n\to\R$ so that $\hor_n(L^n_{g,x})=x$, gives a positive prehorograding of $\Psi_n$ by the standard action of $H$.
	We will need the following.
	
	\setcounter{claimnum}{0}
	
	\begin{claimnum}\label{lem.L1g} For every $x\in\R$ and $g\in H$, there exists $n_0\in\N$ such that $\Psi_n(g)(L^n_{\id,x})\cap L^n_{\id,x}\neq \varnothing$
		for every $n\geq n_0$. 
	\end{claimnum}
	\begin{proof}[Proof of claim] On the one hand, we have $\Psi_n(g)(L^n_{\id,x})=L^n_{g,g (x)}$. On the other hand, since the support of $j_g\restriction_{[g(x),+\infty)}$ is finite, there  are only finitely many $n\in \N$ such that $j_g^n\restriction_{[g(x),+\infty)}$ is non-trivial. As the action $\Psi_n$ is given by $g\cdot j_h^n=j_g^nj_h^n\circ g^{-1}$, one can easily deduce that $L^n_{g,g( x)}=L^n_{\id,g (x)}$ for every $n\geq n_0$, as wanted.
	\end{proof}

	With a similar argument, we have that if $n\ge 1$ (equivalently, if $p_n\ge 3$), then $j_g^n(x)=1$ for every $g\in G$ and $x\in \R$. Then, for such choices, we have $\Psi_n(g)(L_{\id,x}^n)=L_{\id,g(x)}^n$. Then, with a routine argument, one checks that the conditions in Proposition \ref{p.minimalitycriteria} are satisfied, and therefore the dynamical realization $\varphi_n$ of $\Psi_n$ is minimal. Thus, after conjugating, we can assume that $\varphi_n\in\Der_\mu(H;\R)$, for a fixed symmetric probability measure $\mu$ on $H$ with finite support, and $n\geq 1$. Also, denote by $\iota$ a representative of the standard action of $H$ in $\Der_\mu(H;\R)$. Consider the $\varphi_n$-invariant prelamination $\mathcal L^n$ defined from $\mathcal L_0^n$, as described in Remark \ref{sc:inv_prelam_ordered_sets}. Namely, considering an equivariant good embedding $i_n\colon (\mathsf S_n,\prec_n)\to (\R,<)$ associated with $\varphi_n$, one considers the collections of intervals $l^n_{g,x}\subset\R$, obtained as the interior of the convex hull of $i_n(L^n_{g,x})$, for  $g\in H$ and $x\in\R$.
	By abuse of notation, denote  its closure by $\mathcal{L}^n$, and  by $\hor_n\colon \mathcal{L}^n\to\R$ the horograding extending the map $l^n_{g,x}\mapsto x$. Then, $(\mathcal{L}^n,\hor_n)$ is a positive horograding of $\varphi_n$ by $\iota$. 
	
	Consider now a finite symmetric generating set  $S_0$ of $G$, and the finite symmetric generating set of $H$ defined as $S=S_0\cup\{f,f^{-1}\}$. Denote by $I_{n}$ the central leaf of $\mathcal{L}^n$ associated with $S$, and $J^{\pm}_{n}$ its associated outer rays (see Definition \ref{dfn-decompcentralrays}). 
	
	\begin{claimnum}\label{claim-non-loc-rigid}
		We have $\hor_n(I_n)\to-\infty$ as $n$ goes to $+\infty$. 
	\end{claimnum}
	\begin{proof}[Proof of claim] To see this, notice that by Claim \ref{lem.L1g}, for every  $x\in\R$, there exists $n_1\in\N$ such that $\varphi_n(s)(l^n_{1,x})\cap l^n_{1,x}\neq\emptyset$ for every $n\geq n_1$ and $s\in S$. This implies that for $n\geq n_1$ we must have $I_n\subset l^n_{1,x}$. Since $\hor_n$ is a positive horograding, we conclude that $\hor_n(I_{n})\le x$ for every $n\geq n_1$.
	\end{proof}
	\setcounter{claimnum}{0}
	
	Denote by ${\hor}^+_n\colon J^+_{n}\to [c_n,+\infty)$ the partial semi-conjugacy (see Definition \ref{dfn-partialsemicon}), where $c_n=\hor_n(I_n)$ is the central value, and set $t_n:=\sup l^n_{1,0}$. Notice that for sufficiently large $n$, we have that $I_n\subset l^n_{1,0}$, so $t_n\in J^+_n$, and moreover $\hor^+_n(t_n)=0$. Then, by Lemma \ref{lem.equivalentconverg} and Claim \ref{claim-non-loc-rigid}, we have that the limit of the sequence $\Phi^{t_n}(\varphi_n)$ is $\Phi^{0}(\iota)=\iota$.  This shows \ref{i-non-lr}.
	On the other hand, setting $u_n:=\inf l^n_{1,0}$, then the same reasoning shows (by Lemma \ref{lem.equivalentconverg}) that $\Phi^{u_n}(\varphi_n)$ converges to $\widehat{\iota}$. It follows that for any neighborhoods $U$ of $\iota$ and $V$ of $\widehat{\iota}$, both $U$ and $V$ intersect the positive conjugacy class of $\varphi_n$ for sufficiently large $n$. This shows \ref{i-non-hausdorff}. 
\end{proof}

\backmatter

\bibliography{biblio.bib}
\bibliographystyle{smfalpha}

%\printindex

\end{document}